\documentclass[a4paper,11pt,leqno]{amsart}

\usepackage[shortlabels]{enumitem}
\usepackage{xr-hyper, zref}
\usepackage[hypertexnames=false]{hyperref}
\usepackage{etoolbox}
\usepackage{bm}
\usepackage{amsthm,amssymb,amsmath}
\usepackage[labelformat=simple]{subcaption}

\usepackage{tikz,graphicx,placeins}
\usetikzlibrary{bending,decorations.pathmorphing}
\usetikzlibrary{shapes.geometric}
\usetikzlibrary{arrows}
\usetikzlibrary{arrows.meta}
\tikzset{>={Latex[width=2mm,length=2mm]}}
\usepackage{psfrag,mathrsfs}
\usepackage{mathtools}
\usepackage[foot]{amsaddr}
\usepackage[british]{babel}
\usepackage[babel]{microtype}
\usepackage[margin=1in]{geometry}
\usepackage[capitalise]{cleveref}
\usepackage[textsize=footnotesize]{todonotes}
\usepackage{environ}
\usepackage{float}
\usepackage{bigfoot}
\usepackage[noadjust]{cite}
\usepackage{dsfont}
\usepackage{moreenum}
\usepackage{eqparbox}
\usepackage[nonumberlist]{glossaries}
\renewcommand{\glsglossarymark}[1]{}

\date{}
\address{Institut f\"{u}r Informatik, Universit\"{a}t Heidelberg, 69120 Heidelberg, Deutschland}

\author[Bertille~Granet]{Bertille Granet}
\email{granet@informatik.uni-heidelberg.de}

\newcommand{\COMMENT}[1]{}
\newcommand{\PRIVATEAPPENDIX}[1]{}

\newcommand{\APPENDIX}[1]{}
\newcommand{\NOAPPENDIX}[1]{#1}
\renewcommand{\APPENDIX}[1]{#1}\renewcommand{\NOAPPENDIX}[1]{}

\newcommand{\onlyinsubfile}[1]{#1}
\newcommand{\notinsubfile}[1]{}

\hypersetup{colorlinks=true,
	citecolor=blue,
	filecolor=blue,
	linkcolor=blue,
}

\SetEnumitemKey{longlabel}{wide=0.5cm, leftmargin=*, align=right}
\setlist[enumerate]{itemsep=5pt, topsep=5pt,leftmargin=1.5cm}
\setlist[itemize]{itemsep=5pt, topsep=5pt, label=--}

\newlist{steps}{enumerate}{1}
\setlist[steps,1]{
	label=\textbf{Step \arabic*:},
	ref=\arabic*, 
	wide,
	parsep=0pt,
	itemsep=10pt,
	topsep=10pt}
\Crefname{stepsi}{Step}{Steps}
\crefname{stepsi}{Step}{Steps}

\newlist{case}{enumerate}{1}
\setlist[case,1]{
	label=\textbf{Case \arabic*:},
	ref=\arabic*, 
	wide,
	parsep=0pt,
	itemsep=10pt,
	topsep=10pt}
\Crefname{casei}{Case}{Cases}
\crefname{casei}{Case}{Cases}

\AtBeginEnvironment{thm}{\setlist[enumerate,1]{label=\upshape(\roman*),ref=(\roman*)}}
\AtBeginEnvironment{lm}{\setlist[enumerate,1]{label=\upshape(\roman*),ref=(\roman*)}}
\AtBeginEnvironment{prop}{\setlist[enumerate,1]{label=\upshape(\roman*),ref=(\roman*)}}
\AtBeginEnvironment{cor}{\setlist[enumerate,1]{label=\upshape(\roman*),ref=(\roman*)}}
\AtBeginEnvironment{claim}{\setlist[enumerate,1]{label=\upshape(\roman*),ref=(\roman*)}}
\AtBeginEnvironment{fact}{\setlist[enumerate,1]{label=\upshape(\roman*),ref=(\roman*)}}

\creflabelformat{equation}{#2(#1)#3}
\crefdefaultlabelformat{#2\textup{#1}#3}

\Crefname{enumi}{}{}
\Crefname{thm}{Theorem}{Theorems}
\Crefname{lm}{Lemma}{Lemmas}
\Crefname{cor}{Corollary}{Corollaries}
\Crefname{prop}{Proposition}{Propositions}
\Crefname{claim}{Claim}{Claims}
\Crefname{equation}{}{}
\Crefname{conjecture}{Conjecture}{Conjectures}
\Crefname{figure}{Figure}{Figures}
\Crefname{fact}{Fact}{Facts}

\theoremstyle{definition}
\newtheorem{definition}{Definition}[section]

\theoremstyle{plain}
\newtheorem{claim}{Claim}
\newtheorem{prop}[definition]{Proposition}
\newtheorem{thm}[definition]{Theorem}
\newtheorem{cor}[definition]{Corollary}
\newtheorem{lm}[definition]{Lemma}
\newtheorem{fact}[definition]{Fact}
\newtheorem{conjecture}[definition]{Conjecture}

\newtheorem*{lm*}{Lemma}

\NewEnviron{property}[1]{\begin{equation}
	\tag{\(#1\)}
	\parbox{0.9\linewidth}{{\itshape\BODY}}
	\end{equation}}

\newenvironment{proofclaim}[1][Proof of Claim]{\begin{proof}[#1]}{\end{proof}}

\newcounter{claimnumber}
\AtBeginEnvironment{proof}{\setcounter{claim}{0}}
\AtBeginEnvironment{proofclaim}{\setcounter{claimnumber}{\value{claim}}}
\AtEndEnvironment{proofclaim}{\setcounter{claim}{\value{claimnumber}}}

\numberwithin{equation}{section}

\newcommand{\eps}{\varepsilon}
\renewcommand{\epsilon}{\varepsilon}

\DeclareMathOperator{\Bin}{Bin}
\DeclareMathOperator{\HGeom}{Hyp}

\DeclareMathOperator{\reg}{reg}

\newcommand{\cA}{\mathcal{A}}
\newcommand{\cB}{\mathcal{B}}
\newcommand{\cC}{\mathcal{C}}

\newcommand{\cF}{\mathcal{F}}

\newcommand{\cH}{\mathcal{H}}
\newcommand{\cI}{\mathcal{I}}

\newcommand{\cM}{\mathcal{M}}

\newcommand{\cP}{\mathcal{P}}
\newcommand{\cQ}{\mathcal{Q}}
\newcommand{\cR}{\mathcal{R}}

\newcommand{\cU}{\mathcal{U}}
\newcommand{\cV}{\mathcal{V}}

\newcommand{\tC}{\widetilde{C}}
\newcommand{\tD}{\widetilde{D}}

\newcommand{\tF}{\widetilde{F}}

\newcommand{\tH}{\widetilde{H}}

\newcommand{\tM}{\widetilde{M}}
\newcommand{\tN}{\widetilde{N}}

\newcommand{\tP}{\widetilde{P}}
\newcommand{\tQ}{\widetilde{Q}}
\newcommand{\tR}{\widetilde{R}}
\newcommand{\tS}{\widetilde{S}}

\newcommand{\tV}{\widetilde{V}}
\newcommand{\tW}{\widetilde{W}}

\newcommand{\tGamma}{\widetilde{\Gamma}}

\newcommand{\hC}{\widehat{C}}

\newcommand{\hH}{\widehat{H}}

\newcommand{\hM}{\widehat{M}}

\newcommand{\hR}{\widehat{R}}

\newcommand{\hU}{\widehat{U}}
\newcommand{\hV}{\widehat{V}}
\newcommand{\hW}{\widehat{W}}

\newcommand{\sC}{\mathscr{C}}

\newcommand{\sM}{\mathscr{M}}

\newcommand{\sP}{\mathscr{P}}

\newcommand{\sU}{\mathscr{U}}

\newcommand{\tk}{\widetilde{k}}

\newcommand{\tm}{\widetilde{m}}

\newcommand{\hk}{\widehat{k}}

\usepackage{subfiles}

\makeglossaries

\newglossaryentry{setup}
{
    name=setup,
    description={\cref{def:ST}: \cref{def:ST-C,def:ST-P',def:ST-P,def:ST-P*,def:ST-R,def:ST-U',def:ST-U'supreg,def:ST-U}}
}
\newglossaryentry{bi-setup}
{
    name=bi-setup,
    description={\cref{def:BST}: \cref{def:BST-P,def:BST-R,def:BST-C,def:BST-U,def:BST-P',def:BST-U',def:BST-U'supreg,def:BST-P*}}
}
\newglossaryentry{cycle-setup}
{
	name=cycle-setup,
	description={\cref{def:CST}: \cref{def:CST-D,def:CST-M,def:CST-ST}}
}
\newglossaryentry{cycle-framework}
{
	name=cycle-framework,
	description={\cref{def:CF}: \cref{def:CF-C,def:CF-M,def:CF-P,def:CF-P*,def:CF-U}}
}
\newglossaryentry{consistent cycle-framework}
{
	name=consistent cycle-framework,
	description={\cref{def:CCF}}
}
\newglossaryentry{consistent bi-system}
{
	name=consistent bi-system,
	description={\cref{def:CBSys}: \cref{def:CBSys-biproboutexp,def:CBSys-C,def:CBSys-C0,def:CBSys-deg,def:CBSys-P,def:CBSys-P0,def:CBSys-R,def:CBSys-R0}}
}

\newglossaryentry{contraction}
{
	name=matching contraction,
	description={\cref{def:contractexpand}\cref{def:contract}}
}
\newglossaryentry{expansion}
{
	name=matching expansion,
	description={\cref{def:contractexpand}\cref{def:expand}}
}

\newglossaryentry{equivalent linear forests}
{
	name=equivalent linear forests,
	description={\cref{def:equivalentP}}
}
\newglossaryentry{feasible system}
{
	name=feasible system,
	description={\cref{def:feasible}: \cref{def:feasible-backward,def:feasible-linforest,def:feasible-exceptional}}
}
\newglossaryentry{pseudo-feasible system}
{
	name=pseudo-feasible system,
	description={\cref{def:pseudo}: \cref{def:feasible-backward} and \cref{def:feasible-cycle,def:feasible-degV',def:feasible-degV*}}
}
\newglossaryentry{placeholder}
{
	name=placeholder,
	description={\cref{def:placeholder}}
}

\newglossaryentry{universal walk}
{
	name=universal walk,
	description={\cref{def:U}: \cref{def:U-degree,def:U-edges,def:U-size}}
}
\newglossaryentry{bi-universal walk}
{
	name=bi-universal walk,
	description={\cref{def:BU}: \cref{def:BU-degree,def:BU-edges,def:BU-size}}
}

\newglossaryentry{canonical interval partition}
{
	name=canonical interval partition,
	description={\cref{def:interval}}
}
\newglossaryentry{special path system}
{
	name=special path system,
	description={\cref{def:SPS}: \cref{def:SPS-V+-,def:SPS-V0}}
}
\newglossaryentry{special factor}
{
	name=special factor,
	description={\cref{def:SF}}
}
\newglossaryentry{friendly extended special path system}
{
	name=friendly extended special path system,
	description={\cref{def:FESPS}: \cref{def:FESPS-M,def:FESPS-SPS}}
}
\newglossaryentry{extended special path system}
{
	name=extended special path system,
	description={\cref{def:ESPS}}
}
\newglossaryentry{extended special factor}
{
	name=extended special factor,
	description={\cref{def:ESF}}
}

\newglossaryentry{special cover}
{
	name=special cover,
	description={\cref{def:SC}}
}
\newglossaryentry{complete special sequence}
{
	name=complete special sequence,
	description={\cref{def:MSC}}
}
\newglossaryentry{localised special cover}
{
	name=localised special cover,
	description={\cref{def:LSC}}
}
\newglossaryentry{balanced special cover}
{
	name=balanced special cover,
	description={\cref{def:BSC}}
}

\newglossaryentry{uniform refinement}
{
	name={uniform refinement},
	description={\cref{def:uniformref}: \cref{def:URef}}
}
\newglossaryentry{optimal partition}
{
	name=optimal partition,
	description={\cref{def:optimal}}
}
\newglossaryentry{exceptional set}
{
	name={\ensuremath{(\varepsilon, \cU)}-exceptional set},
	description={\cref{def:ES}: \cref{def:ES-backward,def:ES-size}},
}
\newglossaryentry{epsilon partition}
{
	name={\ensuremath{(\varepsilon, 4)}-partition},
	description={\cref{def:epspartition}}
}

\title{Hamilton decompositions of regular bipartite tournaments}

\begin{document}

\renewcommand{\onlyinsubfile}[1]{}
\renewcommand{\notinsubfile}[1]{#1}

\begin{abstract}
	
	\onlyinsubfile{
		\setcounter{section}{0}
\section{Abstract}}

A \emph{regular bipartite tournament} is an orientation of a complete balanced bipartite graph $K_{2n,2n}$ where every vertex has its in- and outdegree both equal to $n$.
In 1981, Jackson conjectured that any regular bipartite tournament can be decomposed into Hamilton cycles. We prove this conjecture for all sufficiently large bipartite tournaments.
Along the way, we also prove several further results, including a conjecture of Liebenau and Pehova on Hamilton decompositions of dense bipartite digraphs.

\end{abstract}

\maketitle

\section{Introduction}\label{sec:intro}

	\onlyinsubfile{
\section{Introduction}}

Given a (directed) graph $G$, can we partition the edges of $G$ into Hamilton cycles? This is one of the central questions of graph theory, which dates back to the 19\textsuperscript{th} century and has attracted considerable interest since then.
In 1883, Walecki \cite{lucas1883recreationsII} (see also \cite{alspach2008wonderful}) showed that the edges of a complete graph on $n$ vertices can be decomposed into $\lfloor\frac{n-1}{2}\rfloor$ Hamilton cycles and at most one perfect matching (depending on the parity of $n$). 
It follows that complete digraphs of odd order have a Hamilton decomposition (see e.g.\ \cite{bermond1976decomposition}).
In 1980, Tillson \cite{tillson1980hamiltonian} showed that complete digraphs on $2n\geq 8$ vertices can be decomposed into Hamilton cycles. (Bermond and Faber~\cite{bermond1976decomposition} verified that such decompositions do not exist for $2n\in \{4,6\}$.)
The analogous question for regular tournaments was posed by Kelly in 1968 (see \cite{moon1968topics}) and resolved for large tournaments by K\"{u}hn and Osthus \cite{kuhn2013hamilton} in 2013. Building on this, Csaba, K\"{u}hn, Lo, Osthus, and Treglown \cite{csaba2016proof} extended Walecki's result to show that even-regular graphs whose degree is sufficiently large to ensure the existence of a single Hamilton cycle also have a Hamilton decomposition.

Walecki's construction \cite{lucas1883recreationsII} can also be used to show that a complete bipartite graph on vertex classes of size $2n$ can be decomposed into Hamilton cycles. In 1972, Dirac \cite{dirac1972hamilton} showed that a complete bipartite graph on vertex classes of size $2n+1$ can be decomposed into $n$ Hamilton cycles and one perfect matching.
More generally, Hetyei \cite{hetyei19751}
and Laskar and Auerbach \cite{laskar1976decomposition} independently showed in the 1970's that complete $r$-partite graphs on vertex classes of size $n$ can be decomposed into $\lfloor\frac{n(r-1)}{2}\rfloor$ Hamilton cycles and at most one perfect matching (depending on the parity of $n(r-1)$).
In 1997, Ng \cite{ng1997hamiltonian} showed that complete $r$-partite digraphs on vertex classes of size $n$ have a Hamilton decomposition if and only if $(r,n)\notin\{(4,1),(6,1)\}$.

\subsection{Partite tournaments}

Throughout this paper, a \emph{digraph} $D$ consists of a set $V(D)$ of vertices and a set $E(D)$ of directed edges on $V(D)$. (Note that, if $D$ is a digraph and $u,v\in V(D)$ are distinct, then $D$ may contain a directed edge $uv$ from $u$ to $v$ as well as a directed edge $vu$ from $v$ to $u$.)
Given $r\geq 2$, we say that a digraph $T$ is an \emph{$r$-partite tournament} if $T$ can be obtained by orienting the edges of a complete $r$-partite graph. (Note that, in an $r$-partite tournament, there exists at most one edge between any two distinct vertices.) 
A digraph $D$ is \emph{$r$-regular} if each vertex of $D$ has both its in- and outdegree equal to $r$. A digraph $D$ is \emph{regular} if there exists $r$ such that $D$ is $r$-regular. 

K\"uhn and Osthus \cite{kuhn2014hamilton} showed that if $r\geq 4$, then any sufficiently large regular $r$-partite tournament can be decomposed into Hamilton cycles. They also conjectured that this can be extended to the $r=3$ case (see \cite[Conjecture 1.14]{kuhn2014hamilton}). 
However, we observe that this conjecture is false.

\begin{prop}\label{prop:tripartite}
	For any integer $n\geq 2$, there exists a regular tripartite tournament on vertex classes of size $n$ which does not have a Hamilton decomposition.	
\end{prop}

In this paper, we focus on the last remaining possibility: the $r=2$ case, that is, \emph{bipartite tournaments}. 
In 1981, Jackson \cite{jackson1981long} showed than any regular bipartite tournament is Hamiltonian and conjectured that such digraphs have a Hamilton decomposition.

\begin{conjecture}[Jackson]\label{conjecture:Jackson}
	Any regular bipartite tournament can be decomposed into Hamilton cycles.
\end{conjecture}

Some progress on this conjecture was made by Liebenau and Pehova \cite{liebenau2020approximate}, who showed that any sufficiently large regular bipartite digraph of sufficiently large degree has an approximate Hamilton decomposition. 

\begin{thm}[{\cite{liebenau2020approximate}}]\label{thm:approxJackson}
	For any $\delta>\frac{1}{2}$ and $\varepsilon>0$, there exists $n_0\in \mathbb{N}$ such that any $\delta n$-regular bipartite digraph on $2n\geq n_0$ vertices contains at least $(1-\varepsilon)\delta n$ edge-disjoint Hamilton cycles. 
\end{thm}

Note however that  \cref{thm:approxJackson} does not actually apply to regular bipartite tournaments, i.e.\ even the existence of an approximate Hamilton decomposition of regular bipartite tournaments was not known so far. We prove \cref{conjecture:Jackson} for sufficiently large bipartite tournaments. (Note that the number of vertices in \cref{thm:Jackson} is a multiple of $4$ since any regular bipartite tournament must necessarily have vertex classes of the same even size.)

\begin{thm}\label{thm:Jackson}
    There exists $n_0\in \mathbb{N}$ such that any regular bipartite tournament $T$ on $4n\geq n_0$ vertices has a Hamilton decomposition.
\end{thm}

The arguments of \cref{thm:Jackson} will be split into two cases: $T$ is a ``bipartite robust outexpander'' (defined in \cref{sec:intro-biprobexp}) and $T$ is ``close to the complete blow-up $C_4$" (defined in \cref{sec:intro-blowupC4}). 

\subsection{Bipartite robust outexpanders}\label{sec:intro-biprobexp}

Roughly speaking, a robust outexpander is a digraph $D$ such that for any $S\subseteq V(D)$ which is neither too small nor too large, there are significantly many more than $|S|$ vertices of $D$ which have a linear number of inneighbours in $S$. This was first introduced in \cite{kuhn2010hamiltonian} by K\"uhn, Osthus, and Treglown. One can check that bipartite digraphs are not robust outexpanders (the largest vertex class does not expand). However, we can easily define a bipartite analogue of robust outexpansion as follows. (Note that an undirected version of bipartite robust outexpansion was introduced in \cite{kuhn2015robust} by K\"uhn, Lo, Osthus, and Staden.)

More precisely, given a digraph $D$ on $n$ vertices and $S\subseteq V(D)$, the \emph{$\nu$-robust outneighbourhood of $S$}, denoted by $RN_{\nu, D}^+(S)$, consists of all the vertices of $D$ which have at least $\nu n$ inneighbours in $S$.
A digraph $D$ on $n$ vertices is called a \emph{robust $(\nu, \tau)$-outexpander} if $|RN_{\nu, D}^+(S)|\geq |S|+\nu n$ for every $S\subseteq V(D)$ satisfying $\tau n\leq |S|\leq (1-\tau) n$.
A balanced bipartite digraph $D$ on vertex classes $A$ and $B$ of size $n$
is called a \emph{bipartite robust $(\nu, \tau)$-outexpander with bipartition $(A,B)$} if the following holds. Let $S\subseteq V(D)$ satisfy $\tau n\leq |S|\leq (1-\tau)n$. If $S\subseteq A$ or $S\subseteq B$, then $|RN_{\nu, D}^+(S)|\geq |S|+\nu n$.

The main result of \cite{kuhn2013hamilton} states that any sufficiently large regular robust outexpander of linear degree can be decomposed into Hamilton cycles.

\begin{thm}[{\cite{kuhn2013hamilton}}]\label{thm:Kelly}
	For any $\delta >0$, there exists $\tau >0$ such that, for all $\nu>0$, there exists $n_0\in \mathbb{N}$ for which the following holds. Let $D$ be a robust $(\nu, \tau)$-outexpander on $n\geq n_0$ vertices and suppose that $D$ is $r$-regular for some $r\geq \delta n$. Then, $D$ has a Hamilton decomposition.
\end{thm}

As explained in \cite{kuhn2013hamilton,kuhn2014hamilton}, the proof of Kelly's conjecture in \cite{kuhn2013hamilton} and its $r$-partite analogue for $r\geq 4$ are actually special cases of \cref{thm:Kelly}.
In \cite{osthus2013approximate}, Osthus and Staden proved an approximate version of \cref{thm:Kelly} (which was then extended to a full Hamilton decomposition in \cite{kuhn2013hamilton}).
Together with Gir\~{a}o, K\"{u}hn, Lo, and Osthus \cite{girao2020path}, we provided a simpler proof of this approximate result. We will see that these arguments can easily be adapted to find approximate Hamilton decompositions of regular bipartite robust outexpanders. Then, the leftovers can be covered using tools of \cite{kuhn2013hamilton} to obtain the following bipartite version of \cref{thm:Kelly}.

\begin{thm}\label{thm:biprobexp}
	For any $\delta>0$, there exists $\tau>0$ such that, for all $\nu>0$, there exists $n_0\in \mathbb{N}$ for which the following holds. Let $D$ be a balanced bipartite digraph on vertex classes $A$ and $B$ of size $n\geq n_0$. Suppose that $D$ is a bipartite robust $(\nu, \tau)$-outexpander with bipartition $(A,B)$ and that $D$ is $r$-regular for some $r\geq \delta n$. Then, $D$ has a Hamilton decomposition.
\end{thm}

\Cref{thm:biprobexp} can be used to prove an analogous result for undirected graphs.
Given a graph $G$ on $n$ vertices and $S\subseteq V(G)$, the \emph{$\nu$-robust neighbourhood of $S$}, denoted by $RN_{\nu, G}(S)$, consists of all the vertices of $G$ which have at least $\nu n$ neighbours in $S$.
A balanced bipartite graph $G$ on vertex classes $A$ and $B$ of size $n$
is called a \emph{bipartite robust $(\nu, \tau)$-expander with bipartition $(A,B)$} if, for any $S\subseteq A$ which satisfies $\tau |A|\leq |S|\leq (1-\tau)|A|$, we have $|RN_{\nu, G}(S)|\geq |S|+\nu n$. (Note that the ordering of $A$ and $B$ matters here.)

\begin{cor}\label{cor:undirected}
	For any $\delta>0$, there exists $\tau>0$ such that, for all $\nu>0$, there exists $n_0\in \mathbb{N}$ for which the following holds. Let $G$ be a bipartite graph on vertex classes $A$ and $B$ of size $n\geq n_0$. Suppose that $G$ is a bipartite robust $(\nu, \tau)$-expander with bipartition $(A,B)$, as well as with bipartition $(B,A)$, and that $G$ is $r$-regular for some even $r\geq \delta n$.
	Then, $G$ has a Hamilton decomposition.
\end{cor}

Moreover, it turns out that regular digraphs of sufficiently large degree are bipartite robust outexpanders (see \cref{lm:1/2rob}). Thus, \cref{thm:biprobexp} implies that the following holds, which resolves a conjecture of Liebenau and Pehova \cite[Conjecture 4.1]{liebenau2020approximate}.

\begin{cor}\label{cor:1/2}
	For any $\delta>\frac{1}{2}$, there exists $n_0\in \mathbb{N}$ for which the following holds. Let $D$ be a bipartite digraph on vertex classes of size $n\geq n_0$. If $D$ is $r$-regular for some $r\geq \delta n$, then $D$ has a Hamilton decomposition.
\end{cor}

Note that the non-bipartite versions of \cref{cor:undirected,cor:1/2} were derived from \cref{thm:Kelly} by K\"{u}hn and Osthus \cite{kuhn2014hamilton}, while the exact threshold for the non-bipartite version of \cref{cor:1/2} was obtained by Csaba, K\"{u}hn, Lo, Osthus, and Treglown \cite{csaba2016proof} using \cref{thm:Kelly}. 
For further related problems in the area, see e.g.\ the survey by K\"{u}hn and Osthus \cite{kuhn2014hamiltonICM}.
In particular, much work has been done to find the size of an optimal packing of edge-disjoint Hamilton cycles in, for example, dense graphs \cite{csaba2016proof,ferber2017counting,kuhn2014hamilton}, dense digraphs \cite{kuhn2014hamilton}, dense hypergraphs \cite{joos2021decomposing}, and $\varepsilon$-regular graphs \cite{frieze2005packing}. 
Since random bipartite (di)graphs are robust (out)expanders with high probability, \cref{thm:biprobexp,cor:undirected} can be used to obtain analogous results for such (di)graphs. Moreover, one can show that $\varepsilon$-regular bipartite (di)graphs are in fact robust (out)expanders and so \cref{thm:biprobexp,cor:undirected} can be used to improve the approximate bounds of Frieze and Krivelevich \cite{frieze2005packing} to exact bounds. (See \cref{sec:applications} for details.)
Another related line of research has been to count Hamilton decompositions, see e.g.\ \cite{ferber2018counting,glebov2017number}. As some of these results (e.g.\ \cite{csaba2016proof,kuhn2014hamilton,ferber2018counting,glebov2017number}) use \cref{thm:Kelly}, it may be that \cref{thm:biprobexp,cor:undirected} have further applications too.

\subsection{The complete blow-up \texorpdfstring{$C_4$}{C4} case}\label{sec:intro-blowupC4}

The \emph{complete blow-up $C_4$ with vertex classes of size $n$} is the $n$-fold blow-up of the directed $C_4$. We say that a regular bipartite tournament is \emph{$\varepsilon$-close to the complete blow-up $C_4$ on vertex classes of size $n$} if it can be obtained from the complete blow-up $C_4$ with vertex classes of size $n$ by flipping the direction of at most $4\varepsilon n^2$ edges.

It is well-known that a regular tournament is a robust outexpander, thus \cref{thm:Kelly} directly implies Kelly's conjecture on Hamilton decompositions of regular tournaments. However, regular bipartite tournaments are not necessarily bipartite robust outexpanders: for example, the vertex classes of the complete blow-up $C_4$ do not expand. It is thus much more difficult to prove the existence of a Hamilton decomposition in the bipartite case.
However, from the definition of a bipartite robust outexpander, one can easily verify that a regular bipartite tournament is either a bipartite robust outexpander or close to the complete blow-up $C_4$.

\begin{lm}\label{lm:twocases}
	For any $\tau>0$, there exists $\nu>0$ such that, for all $0<\nu'\leq \nu$, there exists $n_0\in \mathbb{N}$ for which the following holds. Let $T$ be a regular bipartite tournament on vertex classes $A$ and $B$ of size $2n\geq n_0$. Then, one of the following holds.
	\begin{enumerate}
		\item $T$ is a bipartite robust $(\nu',\tau)$-outexpander with bipartition $(A,B)$.
		\item $T$ is $\sqrt{\nu'}$-close to the complete blow-up $C_4$ on vertex classes of size $n$.
	\end{enumerate}
\end{lm}

Thus, \cref{thm:Jackson} follows from \cref{thm:biprobexp}, \cref{lm:twocases}, and the following.

\begin{thm}\label{thm:blowupC4}
	There exist $\varepsilon>0$ and $n_0\in \mathbb{N}$ for which the following holds. Let $T$ be a regular bipartite tournament on vertex classes of size $2n\geq n_0$. Suppose that $T$ is $\varepsilon$-close to the complete blow-up $C_4$ on vertex classes of size $n$.
	Then, $T$ has a Hamilton decomposition.
\end{thm}

\begin{proof}[Proof of \cref{thm:Jackson}]
    Define $\delta\coloneqq \frac{1}{2}$. Let $\tau>0$ be the constant obtained by applying \cref{thm:biprobexp}, let $\nu>0$ be the constant obtained by applying \cref{lm:twocases}, and let $\varepsilon>0$ be the constant obtained by applying \cref{thm:blowupC4}. Define $\nu'\coloneqq \min\{\nu, \varepsilon^2\}$. Let $n_0'$ the largest of the constants obtained by applying \cref{thm:biprobexp,lm:twocases,thm:blowupC4}. Define $n_0\coloneqq 2n_0'$. Let $T$ be a regular bipartite tournament on $4n\geq n_0$ vertices. Denote by $A$ and $B$ the vertex classes of $T$. By definition of a regular bipartite tournament, we have $|A|=|B|=2n\geq n_0'$ and $T$ is $n$-regular. If $T$ is a bipartite robust $(\nu, \tau)$-outexpander with bipartition $(A,B)$, then \cref{thm:biprobexp} (applied with $T$ and $2n$ playing the roles of $D$ and $n$) implies that $T$ has a Hamilton decomposition, as desired. Otherwise, \cref{lm:twocases} implies that $T$ is $\varepsilon$-close to the complete blow-up $C_4$ on vertex classes of size $n$ and so \cref{thm:blowupC4} implies that $T$ also has a Hamilton decomposition.
\end{proof}

Most of this paper will be devoted to the proof of \cref{thm:blowupC4}. The core of the proof will be to decompose and incorporate the few edges with reversed direction. The approximate decomposition will be constructed using the bipartite analogue of the approximate decomposition lemma of \cite{girao2020path}. To decompose the leftovers, we will use the ``robust decomposition lemma" of \cite{kuhn2013hamilton}. Roughly speaking, this tool states that a robust outexpander $D$ contains an absorber $D^{\rm rob}\subseteq D$ which can decompose any sparse regular leftover of $D$. In this paper, we derive an analogue of this lemma for blow-up cycles.

\subsection{Organisation}
This paper is organised as follows. In \cref{sec:sketch}, we provide a proof overview of \cref{thm:biprobexp,thm:blowupC4}. The notation is introduced in \cref{sec:notation}.
In \cref{sec:applications}, we prove \cref{prop:tripartite}, derive \cref{cor:1/2,cor:undirected}, and discuss optimal packings of Hamilton cycles more thoroughly. \Cref{sec:preliminaries} collects useful tools and preliminary results. We then introduce our main tools for constructing approximate decompositions and absorbing leftovers in \cref{sec:maintools}. In \cref{sec:biprobexp}, we prove \cref{thm:biprobexp}, while \cref{sec:blowups,sec:cyclerobustdecomp,sec:backward,sec:blowupC4,sec:constructfeasible,sec:constructpseudofeasible,sec:pseudofeasible,sec:specialpseudofeasible,sec:ESF} are devoted to proving \cref{lm:twocases,thm:blowupC4} (which are derived in \cref{sec:twocases,sec:blowupC4proof}, respectively).

\onlyinsubfile{\bibliographystyle{abbrv}
\bibliography{Bibliography/Bibliography}}

\section{Proof overview}\label{sec:sketch}

	\onlyinsubfile{
		\setcounter{section}{1}
\section{Proof overview}}

First, we give a proof overview of our main theorems, that is, \cref{thm:biprobexp,thm:blowupC4}. Given a bipartite digraph $D$ on vertex classes $A$ and $B$, we denote by $E_D(A,B)$ the set of edges of $D$ which are oriented from $A$ to $B$ and by $D[A,B]$ the bipartite (undirected) graph induced by $E_D(A,B)$.

\subsection{Constructing a Hamilton cycle in a bipartite digraph}\label{sec:sketch-cycle}

Most of our Hamilton cycles will be constructed using the following procedure. (See also \cref{fig:sketch-bipHam}.)
Let $D$ be a balanced bipartite digraph on vertex classes $A$ and $B$. First, we find a perfect matching $M$ of $D$ whose edges are all oriented from $B$ to $A$. (For example, if $D$ is regular, then we can simply obtain $M$ by applying Hall's theorem in $D[B,A]$.) Then, we restrict ourselves to constructing a Hamilton cycle which contains $M$. That is, we need to find a perfect matching $M'$ of $D$ whose edges are all oriented from $A$ to $B$ and such that $M\cup M'$ forms a Hamilton cycle. 
To do so, we construct an auxiliary digraph $D_M$ on $A$ whose edge set is obtained from $E_D(A,B)$ by identifying the vertices which are matched in $M$. (This digraph will be called the \emph{$M$-contraction of $D[A,B]$}, see \cref{def:contractexpand}\cref{def:contract} for a formal definition.)
Then, we find a Hamilton cycle $C$ in $D_M$. Finally, we observe that $C$ corresponds to a perfect matching $M'$ of $D$ whose edges are all oriented from $A$ to $B$ and such that $M\cup M'$ forms a Hamilton cycle, as desired.

\begin{figure}[htb]
	\centering
	\begin{subfigure}[t]{0.28\textwidth}
	\centering
	\begin{tikzpicture}[scale=0.9]
		\draw node[circle, draw=black,fill=black, inner sep=1.5pt](a1) at (-3,0) {};
		\node[above] at (-3,0) {$a_1$};
		\node[circle, draw=black,fill=black, inner sep=1.5pt](a2) at (-1.5,0) {};
		\node[above] at (-1.5,0) {$a_2$};
		\node[circle, draw=black,fill=black, inner sep=1.5pt](a3) at (0,0) {};
		\node[above] at (0,0) {$a_3$};
		\draw node[circle, draw=black,fill=black, inner sep=1.5pt](a4) at (1.5,0) {};
		\node[above] at (1.5,0) {$a_4$};
		\draw node[circle, draw=black,fill=black, inner sep=1.5pt](b1) at (-3,-3) {};
		\node[below] at (-3,-3) {$b_1$};
		\node[circle, draw=black,fill=black, inner sep=1.5pt](b2) at (-1.5,-3) {};
		\node[below] at (-1.5,-3) {$b_2$};
		\node[circle, draw=black,fill=black, inner sep=1.5pt](b3) at (0,-3) {};
		\node[below] at (0,-3) {$b_3$};
		\draw node[circle, draw=black,fill=black, inner sep=1.5pt](b4) at (1.5,-3) {};
		\node[below] at (1.5,-3) {$b_4$};
		\draw[->,thick,dashed] (b1)-- (a1);
		\draw[->,thick,dashed] (b2)-- (a2);
		\draw[->,thick,dashed] (b3)-- (a3);
		\draw[->,thick,dashed] (b4)-- (a4);
		\draw[->,thick] (a1)-- (b2);
		\draw[->,thick] (a2)-- (b3);
		\draw[->,thick] (a3)-- (b4);
		\draw[->,thick] (a4)-- (b1);
		\draw[->,thick] (a1)-- (b3);
		\draw[->,thick] (a2)-- (b4);
		\draw[->,thick] (a2)-- (b1);
	\end{tikzpicture}
	\caption{A bipartite digraph $D$ on vertex classes $A= \{a_1, a_2,a_3,a_4\}$ and $B= \{b_1, b_2,b_3,b_4\}$ which contains a perfect matching $M$ (dashed edges) whose edges are all oriented from $B$ to~$A$.}
	\end{subfigure}
	\hspace{0.05\textwidth}
	\begin{subfigure}[t]{0.28\textwidth}
		\centering
		\begin{tikzpicture}[scale=0.9]
			\draw node[circle, draw=black,fill=black, inner sep=1.5pt](a1) at (0,0) {};
			\node[above] at (0,0) {$a_1$};
			\node[circle, draw=black,fill=black, inner sep=1.5pt](a2) at (1.5,-1.5) {};
			\node[right] at (1.5,-1.5) {$a_2$};
			\node[circle, draw=black,fill=black, inner sep=1.5pt](a3) at (0,-3) {};
			\node[below] at (0,-3) {$a_3$};
			\draw node[circle, draw=black,fill=black, inner sep=1.5pt](a4) at (-1.5,-1.5) {};
			\node[left] at (-1.5,-1.5) {$a_4$};
			\draw[->,thick,dotted] (a1)-- (a2);
			\draw[->,thick,dotted] (a2)-- (a3);
			\draw[->,thick,dotted] (a3)-- (a4);
			\draw[->,thick,dotted] (a4)-- (a1);
			\draw[->,thick] (a1) -- (a3);
			\draw[->,thick] (a2) -- (a4);
			\draw[->,thick] (a2) to [out=80,in=10] (a1);
		\end{tikzpicture}
		\caption{The digraph $D_M$ on $A=\{a_1, a_2,a_3,a_4\}$ whose edge set is obtained from $E_D(A,B)$ by identifying the vertices which are matched in $M$. The dotted edges form a Hamilton cycle $C$ of $D_M$.}
	\end{subfigure}
	\hspace{0.05\textwidth}
	\begin{subfigure}[t]{0.28\textwidth}
		\centering
		\begin{tikzpicture}[scale=0.9]
			\draw node[circle, draw=black,fill=black, inner sep=1.5pt](a1) at (-3,0) {};
			\node[above] at (-3,0) {$a_1$};
			\node[circle, draw=black,fill=black, inner sep=1.5pt](a2) at (-1.5,0) {};
			\node[above] at (-1.5,0) {$a_2$};
			\node[circle, draw=black,fill=black, inner sep=1.5pt](a3) at (0,0) {};
			\node[above] at (0,0) {$a_3$};
			\draw node[circle, draw=black,fill=black, inner sep=1.5pt](a4) at (1.5,0) {};
			\node[above] at (1.5,0) {$a_4$};
			\draw node[circle, draw=black,fill=black, inner sep=1.5pt](b1) at (-3,-3) {};
			\node[below] at (-3,-3) {$b_1$};
			\node[circle, draw=black,fill=black, inner sep=1.5pt](b2) at (-1.5,-3) {};
			\node[below] at (-1.5,-3) {$b_2$};
			\node[circle, draw=black,fill=black, inner sep=1.5pt](b3) at (0,-3) {};
			\node[below] at (0,-3) {$b_3$};
			\draw node[circle, draw=black,fill=black, inner sep=1.5pt](b4) at (1.5,-3) {};
			\node[below] at (1.5,-3) {$b_4$};
			\draw[->,thick,dashed] (b1)-- (a1);
			\draw[->,thick,dashed] (b2)-- (a2);
			\draw[->,thick,dashed] (b3)-- (a3);
			\draw[->,thick,dashed] (b4)-- (a4);
			\draw[->,thick,dotted] (a1)-- (b2);
			\draw[->,thick,dotted] (a2)-- (b3);
			\draw[->,thick,dotted] (a3)-- (b4);
			\draw[->,thick,dotted] (a4)-- (b1);
		\end{tikzpicture}
		\caption{The Hamilton cycle $C$ of $D_M$ induces a perfect matching $M'$ (dotted edges) of $D$ whose edges are all oriented from $A$ to $B$ and such that $M\cup M'$ forms a Hamilton cycle of~$D$.}
	\end{subfigure}
	\caption{Constructing a Hamilton cycle in a bipartite digraph.\label{fig:sketch-bipHam}}	
	\end{figure}
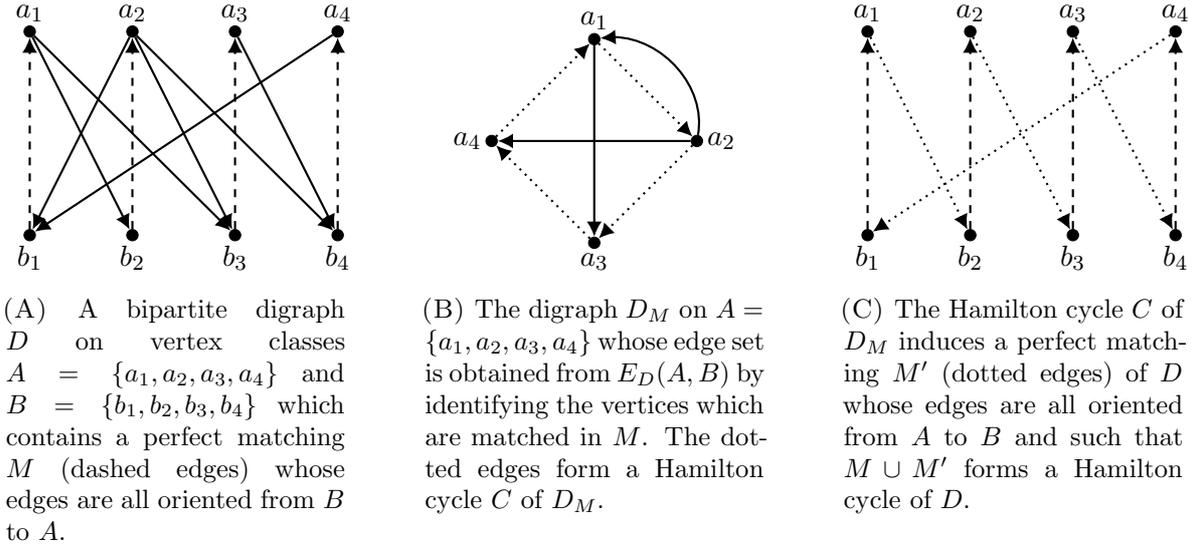

\subsection{The bipartite robust expander case: proof overview of Theorem \ref{thm:biprobexp}}\label{sec:sketch-rob}

Let $T$ be a regular bipartite tournament on vertex classes $A$ and $B$ of size $n$ and suppose that $T$ is a bipartite robust outexpander with bipartition $(A,B)$. (The same arguments hold for a regular bipartite robust outexpander of linear degree.) 

\begin{steps}
	\item \textbf{Constructing an absorber.}\label{step:sketch-rob-A} First, we apply Szemer\'{e}di's regularity lemma to exhibit the structure required to apply (the bipartite version of) the robust decomposition lemma of \cite{kuhn2013hamilton}. This then guarantees a sparse regular absorber $D^{\rm abs}\subseteq T$ which satisfies the following property: for any sparse regular leftover $H\subseteq T\setminus D^{\rm abs}$, the digraph $H\cup D^{\rm abs}$ has a Hamilton decomposition.
	
	\item \textbf{Approximate decomposition.} Denote $D\coloneqq T\setminus D^{\rm abs}$ and note that since $D^{\rm abs}$ is regular and sparse, $D$ is still a very dense regular bipartite robust outexpander. We approximately decompose $D$ into Hamilton cycles using the procedure described in \cref{sec:sketch-cycle}. More precisely, we can construct a Hamilton cycle of $D$ as follows. Since $D$ is regular, we can obtain a perfect matching $M$ of $D$ whose edges are all oriented from $B$ to $A$ simply by applying Hall's theorem in $D[B,A]$. Denote by $D_M$ the auxiliary digraph as defined in \cref{sec:sketch-cycle}. Since $D$ is a regular bipartite robust outexpander, one can verify that $D_M$ is a regular robust outexpander. Then, we use arguments of \cite{girao2020path} to construct a Hamilton cycle $C$ of $D_M$. Let $M'$ be the perfect matching of $D$ induced by $C$. As explained in \cref{sec:sketch-cycle}, $M\cup M'$ is Hamilton cycle of $D$.
	
	Of course, removing the edges of $M\cup M'$ from $D$ affects the bipartite robust outexpansion and, in general, we would not be be able to repeat this argument sufficiently many times to obtain an approximate decomposition. However, the arguments of \cite{girao2020path} allow us to preserve bipartite robust outexpansion in a sufficiently strong way that we can repeat the above arguments to construct many edge-disjoint Hamilton cycles $C_1, \dots, C_{\frac{(1-\varepsilon)n}{2}}$ of~$D$.
	
	\item \textbf{Decomposing the leftovers.} Let $H\coloneqq D\setminus \bigcup_{i\in [\frac{(1-\varepsilon)n}{2}]}E(C_i)$. Note that $H$ is sparse and regular. Thus, the absorbing property described in \cref{step:sketch-rob-A} implies that $H\cup D^{\rm abs}$ can be decomposed into edge-disjoint Hamilton cycles. Together with $C_1, \dots, C_{\frac{(1-\varepsilon)n}{2}}$, this gives us a Hamilton decomposition of $T$, as desired.
\end{steps}

Note that \cref{thm:biprobexp} is proved in \cref{sec:biprobexp}. The tools for constructing the absorber are introduced in \cref{sec:robustdecomp,sec:preprocessing}. The approximate decomposition step is discussed more thoroughly in \cref{sec:approxdecomp}.

\subsection{The complete blow-up \texorpdfstring{$C_4$}{C4} case: proof overview of a special case of Theorem \ref{thm:blowupC4}}\label{sec:sketch-completeblowupC4}

Let $T$ be the complete blow-up $C_4$ with vertex classes of size $n$. That is, there is a partition of $V(T)$ into vertex classes $U_1, \dots, U_4$ of size $n$ such that $E(T)$ consists of all the edges which start in $U_i$ and end in $U_{i+1}$ for some $i\in [4]$ (where $U_5\coloneqq U_1$).

Note that the vertex classes $U_1, \dots, U_4$ do not expand, so $T$ is not a bipartite robust outexpander and we cannot apply the above arguments. (Recall that robust outexpansion was key to construct the approximate decomposition. It is also needed to apply the robust decomposition lemma.) However, we can (roughly) reduce the decomposition of $T$ to the bipartite robust outexpander case as follows.

First, we discuss how to construct a single Hamilton cycle. (See also \cref{fig:sketch-C4Ham}.) For each $i\in [3]$, observe that $T[U_i, U_{i+1}]$ is a complete balanced bipartite graph and so Hall's theorem implies that it contains a perfect matching $M_i$. Then, $M_1\cup M_2\cup M_3$ induces a set $\sP$ of $n$ vertex-disjoint paths of $T$, each starting in $U_1$ and ending in $U_4$. Moreover, $\sP$ covers all of the vertices of $T$. We restrict ourselves to constructing a Hamilton cycle of $T$ which contains $E(\sP)$. Let $M$ be the auxiliary perfect matching from $U_1$ to $U_4$ obtained by replacing each path in $\sP$ by an edge from its starting point to its ending point. Then, it suffices to find a perfect matching $M'\subseteq E_T(U_4,U_1)$ such that $M\cup M'$ forms a Hamilton cycle on $U_4\cup U_1$. This can be done using the arguments of \cref{sec:sketch-cycle} (with $A=U_4$, $B=U_1$, and $E(D)=M\cup E_T(U_4,U_1)$).

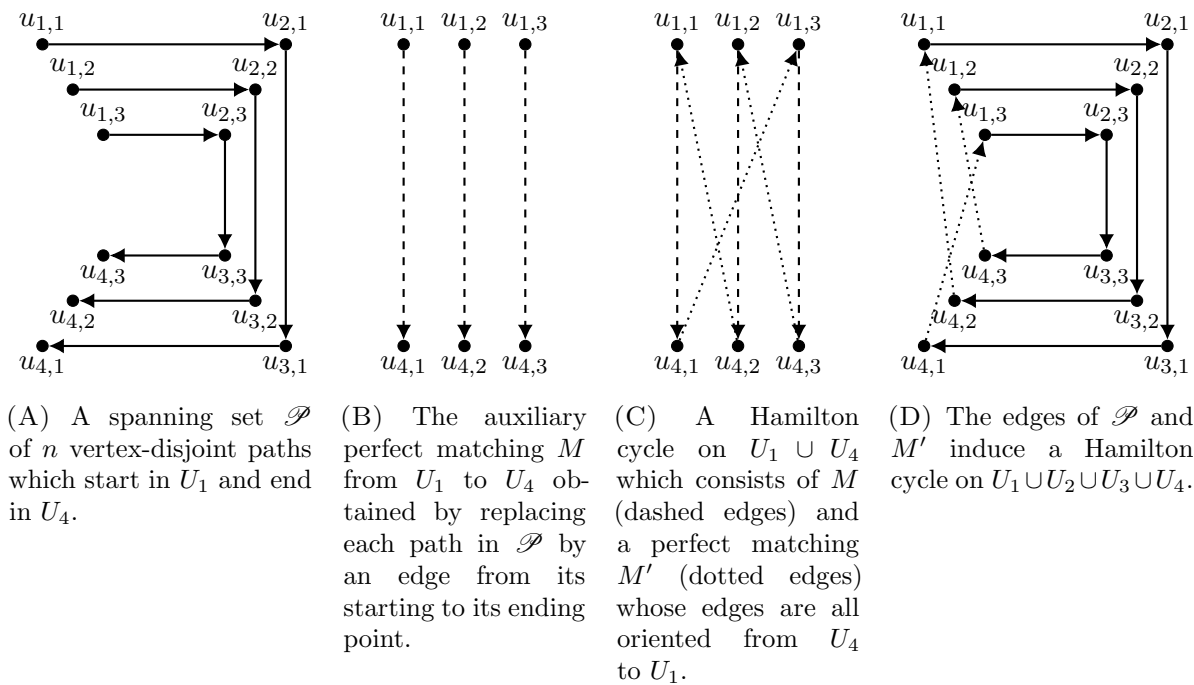
\begin{figure}[htb]
	\centering
	\begin{subfigure}[t]{0.25\textwidth}
	\centering
	\begin{tikzpicture}[scale=0.8]
		\draw node[circle, draw=black,fill=black, inner sep=1.5pt](u11) at (0,0) {};
		\node[above] at (0,0) {$u_{1,1}$};
		\node[circle, draw=black,fill=black, inner sep=1.5pt](u12) at (0.5,-0.75) {};
		\node[above] at (0.5,-0.75) {$u_{1,2}$};
		\node[circle, draw=black,fill=black, inner sep=1.5pt](u13) at (1,-1.5) {};
		\node[above] at (1,-1.5) {$u_{1,3}$};
		\draw node[circle, draw=black,fill=black, inner sep=1.5pt](u21) at (4,0) {};
		\node[above] at (4,0) {$u_{2,1}$};
		\node[circle, draw=black,fill=black, inner sep=1.5pt](u22) at (3.5,-0.75) {};
		\node[above] at (3.5,-0.75) {$u_{2,2}$};
		\node[circle, draw=black,fill=black, inner sep=1.5pt](u23) at (3,-1.5) {};
		\node[above] at (3,-1.5) {$u_{2,3}$};
		\draw node[circle, draw=black,fill=black, inner sep=1.5pt](u31) at (4,-5) {};
		\node[below] at (4,-5) {$u_{3,1}$};
		\node[circle, draw=black,fill=black, inner sep=1.5pt](u32) at (3.5,-4.25) {};
		\node[below] at (3.5,-4.25) {$u_{3,2}$};
		\node[circle, draw=black,fill=black, inner sep=1.5pt](u33) at (3,-3.5) {};
		\node[below] at (3,-3.5) {$u_{3,3}$};
		\draw node[circle, draw=black,fill=black, inner sep=1.5pt](u43) at (1,-3.5) {};
		\node[below] at (1,-3.5) {$u_{4,3}$};
		\node[circle, draw=black,fill=black, inner sep=1.5pt](u42) at (0.5,-4.25) {};
		\node[below] at (0.5,-4.25) {$u_{4,2}$};
		\node[circle, draw=black,fill=black, inner sep=1.5pt](u41) at (0,-5) {};
		\node[below] at (0,-5) {$u_{4,1}$};
		\draw[->,thick] (u11)-- (u21);
		\draw[->,thick] (u21)-- (u31);
		\draw[->,thick] (u31)-- (u41);
		\draw[->,thick] (u12)-- (u22);
		\draw[->,thick] (u22)-- (u32);
		\draw[->,thick] (u32)-- (u42);
		\draw[->,thick] (u13)-- (u23);
		\draw[->,thick] (u23)-- (u33);
		\draw[->,thick] (u33)-- (u43);
	\end{tikzpicture}
	\caption{A spanning set $\sP$ of $n$ vertex-disjoint paths which start in $U_1$ and end in $U_4$.}
	\end{subfigure}
	\hspace{0.01\textwidth}
	\begin{subfigure}[t]{0.2\textwidth}
		\centering
		\begin{tikzpicture}[scale=0.8]
			\draw node[circle, draw=black,fill=black, inner sep=1.5pt](u11) at (0,0) {};
			\node[above] at (0,0) {$u_{1,1}$};
			\node[circle, draw=black,fill=black, inner sep=1.5pt](u12) at (1,0) {};
			\node[above] at (1,0) {$u_{1,2}$};
			\node[circle, draw=black,fill=black, inner sep=1.5pt](u13) at (2,0) {};
			\node[above] at (2,0) {$u_{1,3}$};
			\draw node[circle, draw=black,fill=black, inner sep=1.5pt](u41) at (0,-5) {};
			\node[below] at (0,-5) {$u_{4,1}$};
			\node[circle, draw=black,fill=black, inner sep=1.5pt](u42) at (1,-5) {};
			\node[below] at (1,-5) {$u_{4,2}$};
			\node[circle, draw=black,fill=black, inner sep=1.5pt](u43) at (2,-5) {};
			\node[below] at (2,-5) {$u_{4,3}$};
			\draw[->,thick,dashed] (u11)-- (u41);
			\draw[->,thick,dashed] (u12)-- (u42);
			\draw[->,thick,dashed] (u13)-- (u43);
		\end{tikzpicture}
		\caption{The auxiliary perfect matching $M$ from $U_1$ to $U_4$ obtained by replacing each path in $\sP$ by an edge from its starting to its ending point.}
	\end{subfigure}
	\hspace{0.01\textwidth}
	\begin{subfigure}[t]{0.2\textwidth}
		\centering
		\begin{tikzpicture}[scale=0.8]
			\draw node[circle, draw=black,fill=black, inner sep=1.5pt](u11) at (0,0) {};
			\node[above] at (0,0) {$u_{1,1}$};
			\node[circle, draw=black,fill=black, inner sep=1.5pt](u12) at (1,0) {};
			\node[above] at (1,0) {$u_{1,2}$};
			\node[circle, draw=black,fill=black, inner sep=1.5pt](u13) at (2,0) {};
			\node[above] at (2,0) {$u_{1,3}$};
			\draw node[circle, draw=black,fill=black, inner sep=1.5pt](u41) at (0,-5) {};
			\node[below] at (0,-5) {$u_{4,1}$};
			\node[circle, draw=black,fill=black, inner sep=1.5pt](u42) at (1,-5) {};
			\node[below] at (1,-5) {$u_{4,2}$};
			\node[circle, draw=black,fill=black, inner sep=1.5pt](u43) at (2,-5) {};
			\node[below] at (2,-5) {$u_{4,3}$};
			\draw[->,thick,dashed] (u11)-- (u41);
			\draw[->,thick,dashed] (u12)-- (u42);
			\draw[->,thick,dashed] (u13)-- (u43);
			\draw[->,thick,dotted] (u41)-- (u13);
			\draw[->,thick,dotted] (u42)-- (u11);
			\draw[->,thick,dotted] (u43)-- (u12);
		\end{tikzpicture}
		\caption{A Hamilton cycle on $U_1\cup U_4$ which consists of $M$ (dashed edges) and a perfect matching $M'$ (dotted edges) whose edges are all oriented from $U_4$ to~$U_1$.}
	\end{subfigure}
	\hspace{0.01\textwidth}
	\begin{subfigure}[t]{0.25\textwidth}
		\centering
		\begin{tikzpicture}[scale=0.8]
			\draw node[circle, draw=black,fill=black, inner sep=1.5pt](u11) at (0,0) {};
			\node[above] at (0,0) {$u_{1,1}$};
			\node[circle, draw=black,fill=black, inner sep=1.5pt](u12) at (0.5,-0.75) {};
			\node[above] at (0.5,-0.75) {$u_{1,2}$};
			\node[circle, draw=black,fill=black, inner sep=1.5pt](u13) at (1,-1.5) {};
			\node[above] at (1,-1.5) {$u_{1,3}$};
			\draw node[circle, draw=black,fill=black, inner sep=1.5pt](u21) at (4,0) {};
			\node[above] at (4,0) {$u_{2,1}$};
			\node[circle, draw=black,fill=black, inner sep=1.5pt](u22) at (3.5,-0.75) {};
			\node[above] at (3.5,-0.75) {$u_{2,2}$};
			\node[circle, draw=black,fill=black, inner sep=1.5pt](u23) at (3,-1.5) {};
			\node[above] at (3,-1.5) {$u_{2,3}$};
			\draw node[circle, draw=black,fill=black, inner sep=1.5pt](u31) at (4,-5) {};
			\node[below] at (4,-5) {$u_{3,1}$};
			\node[circle, draw=black,fill=black, inner sep=1.5pt](u32) at (3.5,-4.25) {};
			\node[below] at (3.5,-4.25) {$u_{3,2}$};
			\node[circle, draw=black,fill=black, inner sep=1.5pt](u33) at (3,-3.5) {};
			\node[below] at (3,-3.5) {$u_{3,3}$};
			\draw node[circle, draw=black,fill=black, inner sep=1.5pt](u43) at (1,-3.5) {};
			\node[below] at (1,-3.5) {$u_{4,3}$};
			\node[circle, draw=black,fill=black, inner sep=1.5pt](u42) at (0.5,-4.25) {};
			\node[below] at (0.5,-4.25) {$u_{4,2}$};
			\node[circle, draw=black,fill=black, inner sep=1.5pt](u41) at (0,-5) {};
			\node[below] at (0,-5) {$u_{4,1}$};
			\draw[->,thick] (u11)-- (u21);
			\draw[->,thick] (u21)-- (u31);
			\draw[->,thick] (u31)-- (u41);
			\draw[->,thick] (u12)-- (u22);
			\draw[->,thick] (u22)-- (u32);
			\draw[->,thick] (u32)-- (u42);
			\draw[->,thick] (u13)-- (u23);
			\draw[->,thick] (u23)-- (u33);
			\draw[->,thick] (u33)-- (u43);
			\draw[->,thick,dotted] (u41)-- (u13);
			\draw[->,thick,dotted] (u42)-- (u11);
			\draw[->,thick,dotted] (u43)-- (u12);
		\end{tikzpicture}
		\caption{The edges of $\sP$ and $M'$ induce a Hamilton cycle on $U_1\cup U_2\cup U_3\cup U_4$.}
	\end{subfigure}
	\caption{Constructing a Hamilton cycle in the complete blow-up $C_4$ on vertex classes of size $n=3$, where $U_i= \{u_{i,1},u_{i,2},u_{i,3}\}$ for each $i\in [4]$.\label{fig:sketch-C4Ham}}	
	\end{figure}

In fact, the above argument can be repeated to obtain an approximate decomposition of $T$ into Hamilton cycles. Indeed, for each $i\in [3]$, $T[U_i, U_{i+1}]$ is a complete bipartite graph and so Hall's theorem can be applied repeatedly to (approximately) decompose it into perfect matchings. Moreover, $T[U_4,U_1]$ is also a complete bipartite graph and so one can verify that it is a bipartite robust expander. Thus, it can be approximately decomposed into suitable perfect matchings using the same arguments as in the bipartite robust outexpander case.

To obtain a full Hamilton decomposition of $T$, one needs to find an absorber $A$ which will decompose the edges which are leftover after the approximate decomposition. Unfortunately, we cannot directly apply the tools of \cite{kuhn2013hamilton} in $T$ since it is not a robust outexpander. However, we will derive an analogue of the robust decomposition lemma which can be applied in a blow-up $C_4$. This argument is discussed in \cref{sec:cyclerobustdecomp} (a detailed proof overview is given in \cref{sec:cyclerobustdecomp-sketch}).

\subsection{The \texorpdfstring{$\varepsilon$}{epsilon}-close to the complete blow-up \texorpdfstring{$C_4$}{C4} case: proof overview of Theorem \ref{thm:blowupC4}}

Let $T$ be a regular bipartite tournament and suppose that $T$ is $\varepsilon$-close to the complete blow-up $C_4$ on vertex classes of size $n$. That is, there is a partition of $V(T)$ into vertex classes $U_1, \dots, U_4$ of size $n$ such that $E(T)$ satisfies the following properties (where $U_5\coloneqq U_1$).
\begin{itemize}
	\item For each $i\in [4]$, $u\in U_i$, and $v\in U_{i+1}$, $E(T)$ contains either the edge $uv$ from $u$ to $v$ or the edge $vu$ from $v$ to $u$ (but not both).
	\item $E(T)$ does not contain any other edges.
	\item $\sum_{i\in [4]}|E_T(U_{i+1}, U_i)|\leq 4\varepsilon n^2$.
\end{itemize}
Note that if $E_T(U_{i+1}, U_i)=\emptyset$ for each $i\in [4]$, then $T$ is in fact the complete blow-up $C_4$ on vertex classes of size $n$ and so it can be decomposed using the arguments presented in \cref{sec:sketch-completeblowupC4}. In general, the sets $E_T(U_{i+1}, U_i)$ will be non-empty and the main difficulty will be to incorporate these edges, which we call \emph{backward edges}. 

Our overall strategy is the following. First, we decompose all the backward edges into $n$ small digraphs $\cF_1, \dots, \cF_n$. Then, we restrict ourselves to constructing a Hamilton decomposition of $T$ where each Hamilton cycle contains precisely one of the $\cF_i$'s.

For this to be possible, $\cF_1, \dots, \cF_n$ will need to have a very specific structure.
First, each $\cF_i$ will have to be a linear forest (any proper subdigraph of a Hamilton cycle is a linear forest). Moreover, each $\cF_i$ will need to contain a ``balanced" number of backward edges. To see this, suppose that $C$ is cycle of $T$ such that $C$ contains a backward edge, say from $U_1$ to $U_4$, and all other edges of $C$ are from $U_i$ to $U_{i+1}$ for some $i\in [4]$. Then, one can verify that $C$ covers one more vertex from each of $U_1$ and $U_4$ than from each of $U_2$ and $U_3$. Thus, $C$ cannot be a Hamilton cycle of $T$ (recall that $U_1, \dots, U_4$ are equal sized vertex classes which partition $V(T)$). This example shows no $\cF_i$ can consist of a single backward edge. 

More generally, we will have restrictions on the number of backward edges contained in each $\cF_i$. (Formally, each $\cF_i$ will have to be a \emph{feasible system}, as defined in \cref{sec:feasible}.)
To illustrate this further, consider the simple example where $T$ is obtained from the complete blow-up $C_4$ on vertex classes of size $n$ by flipping the orientation of precisely one $C_4$. Then, $T$ contains precisely four backward edges and since these form a small cycle, they cannot all be included into a common Hamilton cycle. As discussed above, they also cannot be spread across four different Hamilton cycles. Thus, they will be incorporated two by two as follows:
\begin{itemize}
    \item $\cF_1$ will consists of the backward edge from $U_2$ to $U_1$ and the backward edge from $U_4$ to $U_3$,
    \item $\cF_2$ will consists of the backward edge from $U_1$ to $U_4$ and the backward edge from $U_3$ to $U_2$, and
    \item $\cF_3, \dots, \cF_n$ will be empty.
\end{itemize}
We will see that this decomposition of backward edges will ensure that the vertex classes $U_1, \dots, U_4$ can be covered in a ``balanced" way (as opposed to the previous example where $C$ covered more vertices from $U_1$ and $U_4$ than from $U_2$ and $U_3$).

Decomposing the backward edges of $T$ into suitable digraphs $\cF_1, \dots, \cF_n$ will be the core of the proof and the sole focus of \cref{sec:backward,sec:pseudofeasible,sec:constructfeasible,sec:constructpseudofeasible,sec:specialpseudofeasible}. We defer further discussions about how to decompose backward edges to these \lcnamecrefs{sec:backward}, which contain further intuition and motivation.

Once we have constructed suitable digraphs $\cF_1, \dots, \cF_n$, the Hamilton decomposition will be constructed using the arguments presented in \cref{sec:sketch-completeblowupC4}. However, the backward edges will introduce additional problems.
In particular, recall that in \cref{sec:sketch-completeblowupC4} we decomposed $T[U_1,U_2]$, $T[U_2, U_3]$, and $T[U_3, U_4]$ into perfect matchings by applying Hall's theorem. But, this is no longer possible since $T[U_1, U_2]$, $T[U_2, U_3]$, and $T[U_3, U_4]$ may no longer be regular bipartite graphs. Moreover, the matchings will now have to incorporate some of backward edges, as prescribed by $\cF_1, \dots, \cF_n$. Thus, $T[U_1,U_2]$, $T[U_2, U_3]$, and $T[U_3, U_4]$ will now need to be decomposed building on methods from \cite{girao2020path}. (For more detail on how construct an approximate decomposition which incorporates given $\cF_i$'s see \cref{sec:blowupC4approxdecomp}.) 

As mentioned in \cref{sec:sketch-completeblowupC4}, the absorber required to decompose the leftovers will be constructed using an analogue of the robust decomposition lemma for blow-up cycles (see \cref{sec:cyclerobustdecomp} for more detail). The decomposition properties of this absorber will be robust enough to allow us to prescribe the backward edges of the $\cF_i$'s left over by the approximate decomposition.

\onlyinsubfile{\bibliographystyle{abbrv}
	\bibliography{Bibliography/Bibliography}}

\section{Notation}\label{sec:notation}

	\onlyinsubfile{
		\setcounter{section}{2}
\section{Notation and definitions}}

For simplicity, we collect the key notation and concepts that will be used throughout the paper.
The core definitions will be defined when first needed and are indexed in the glossary at the end of the paper.
Given $n\in \mathbb{N}$, we define $[n]\coloneqq \{1, \dots, n\}$. Given $a,b,c\in \mathbb{R}$, we write $a=b\pm c$ to mean that $b-c\leq a\leq b+c$.

\subsection{Graphs and digraphs}\label{sec:notation-graphs}
In this paper, all (directed) graphs are without loops and, unless otherwise specified, without multiple edges.
A \emph{digraph} is a directed graph which contains at most two edges between any pair of distinct vertices and at most one in each direction. An \emph{oriented graph} is a digraph which contains at most one edge between any pair of distinct vertices.
Given a (di)graph $G$, a \emph{sub(di)graph} of $G$ is a (di)graph whose vertex and edge sets are subsets of those of $G$.
Let $G$ be an undirected graph. An \emph{orientation} of $G$ is an oriented graph which can be obtained by orienting the edges of $G$. Given an orientation $D$ of $G$, we say that $G$ is the undirected graph \emph{underlying} $D$.
A directed edge from a vertex $u$ to a vertex $v$ is denoted by $uv$. If $e$ is the directed edge $uv$, we say that $u$ and $v$ are the starting and ending points of $e$, respectively.

A \emph{multigraph} is an undirected graph without loops which may contain multiple edges between the same pair of distinct vertices. Similarly, a \emph{multidigraph} is a directed graph without loops which may contain multiple edges of the same direction between the same pair of distinct vertices. 
Given a multi(di)graph $G$, a \emph{submulti(di)graph} of $G$ is a multi(di)graph whose vertex set is a subset of the vertex set of $G$ and whose edge multiset is a submultiset of the edge multiset of $G$.
The edges of a multi(di)graph are always considered to be distinct. More precisely, given a multi(di)graph $G$ and distinct vertices $u$ and $v$, denote by $\mu_G(uv)$ the \emph{multiplicity} of the edge $uv$ in $G$ (that is, $\mu_G(uv)$ is the number of edges between $u$ and $v$ if $G$ is undirected and the number of edges from $u$ to $v$ if $G$ is directed). Then, given a multi(di)graph $G$ and submulti(di)graphs $G_1$ and $G_2$ of $G$, we say that $G_1$ and $G_2$ are \emph{edge-disjoint} if $\mu_{G_1}(uv)+\mu_{G_2}(uv)\leq \mu_G(uv)$ for any distinct vertices $u$ and $v$ of $G$.

\subsection{Edge sets}
Let $G$ be a (di)graph. We denote by $V(G)$ and $E(G)$ the vertex and edge sets of $G$, respectively. The \emph{order} of $G$ is $|V(G)|$ and we define the \emph{size} of $G$ as $e(G)\coloneqq |E(G)|$. We say that $G$ is \emph{empty} if $E(G)=\emptyset$. 

Let $G$ be a (di)graph and let $A, B\subseteq V(G)$ be disjoint. If $G$ is undirected, we denote by $E_G(A,B)$ the set of undirected edges of $G$ which have an endpoint in $A$ and an endpoint in $B$. If $G$ is directed, we denote by $E_G(A,B)$ the set of directed edges of $G$ which start in $A$ and end in $B$. Define $e_G(A,B)\coloneqq |E_G(A,B)|$. Given any disjoint vertex sets $A'$ and $B'$ which are not necessarily contained in $V(G)$, we sometimes abuse the above notation and write $E_G(A', B')\coloneqq E_G(A'\cap V(G), B'\cap V(G))$ and $e_G(A', B')\coloneqq |E_G(A', B')|$.

Let $G$ be a (di)graph and let $A$ and $B$ be any disjoint vertex sets.
We denote by $G[A,B]$ the undirected bipartite graph on vertex classes $A$ and $B$ induced by $E_G(A,B)$ and, if $G$ is directed, we denote by $G(A,B)$ the directed bipartite graph on vertex classes $A$ and $B$ induced by $E_G(A,B)$. (Thus, if $G$ is directed, then $G[A,B]$ is the undirected graph underlying $G(A,B)$.)

All the above definitions from this subsection extend naturally to multi(di)graphs. That is, if $G$ is a multi(di)graph and $A$ and $B$ are disjoint vertex sets, then $E(G)$ and $E_G(A, B)$ are now multisets of edges, while $G[A,B]$, as well as $G(A,B)$ if $G$ is directed, are now multi(di)graphs. The vertex set $V(G)$ of a multi(di)graph is still a set rather than a multiset.

We sometime abuse notation and consider a set of (directed) edges as a (di)graph. In particular, given a set of edges $E$, we write $V(E)$ for the set of vertices which are incident to an edge in $E$.

\subsection{Subgraphs}
Let $G$ and $H$ be (di)graphs and let $F$ be a sub(di)graph of $G$. We write $F\subseteq G$, and, if $V(F)=V(G)$, we say that $F$ is \emph{spanning}. Given $S\subseteq V(G)$, we write $G[S]$ for the sub(di)graph of $G$ \emph{induced} by $S$ and define $G-S\coloneqq G[V(G)\setminus S]$.
We denote by $G\setminus H$ the (di)graph obtained from $G$ by deleting all the edges in $E(G)\cap E(H)$, we denote by $G\cup H$ the (di)graph with vertex set $V(G)\cup V(H)$ and edge set $E(G)\cup E(H)$, and we denote by $G\cap H$ the (di)graph with vertex set $V(G)\cap V(H)$ and edge set $E(G)\cap E(H)$.
Given a set of (directed) edges $E$, we sometimes abuse the above notation and write $G\setminus E$, $G\cup E$, and $G\cap E$ for the (di)graphs obtained as above when $E$ is viewed as a (di)graph.

All the above definitions from this subsection extend naturally to multi(di)graphs, with unions and differences now considered as multiset unions and differences. More precisely, let $G$ and $H$ be multi(di)graphs. We denote by $G\setminus H$ the multi(di)graph with vertex set $V(G)$ where $\mu_{G\setminus H}(uv)\coloneqq \max\{\mu_G(uv)-\mu_H(uv),0\}$ for any distinct $u,v\in V(G)$, we denote by $G\cup H$ the multi(di)graph with vertex set $V(G)\cup V(H)$ where $\mu_{G\cup H}(uv)\coloneqq \mu_G(uv)+\mu_H(uv)$ for any distinct $u,v\in V(G)\cup V(H)$, and we denote by $G\cap H$ the multi(di)graph with vertex set $V(G)\cap V(H)$ where $\mu_{G\cap H}(uv)\coloneqq \min\{\mu_G(uv),\mu_H(uv)\}$ for any distinct $u,v\in V(G)\cap V(H)$.

\subsection{Neighbourhoods and degrees}
We use standard notation for neighbourhoods and degrees. More precisely, let $G$ be an undirected graph. Given $v\in V(G)$, we denote by $N_G(v)$ the \emph{neighbourhood} of $v$ in $G$ and by $d_G(v)\coloneqq |N_G(v)|$ the \emph{degree} of $v$ in $G$. The \emph{maximum degree} of $G$ is $\Delta(G)\coloneqq \max\{d_G(v)\mid v\in V(G)\}$ and the \emph{minimum degree} of $G$ is $\delta(G)\coloneqq \min\{d_G(v)\mid v\in V(G)\}$.

Similarly, let $D$ be a digraph. Given $v\in V(D)$, we denote by $N_D^+(v)$ and $N_D^-(v)$ the \emph{outneighbourhood} and \emph{inneighbourhood} of $v$ in $D$, respectively, and by $d_D^+(v)\coloneqq |N_D^+(v)|$ and $d_D^-(v)\coloneqq |N_D^-(v)|$ the \emph{outdegree} and \emph{indegree} of $v$ in $D$, respectively. The \emph{neighbourhood} of a vertex $v\in V(D)$ is the set $N_D(v)\coloneqq N_D^+(v)\cup N_D^-(v)$ and the \emph{degree} of a vertex $v\in V(D)$ is $d_D(v)\coloneqq d_D^+(v)+ d_D^-(v)$.
The \emph{maximum} and \emph{minimum outdegree} of $D$ are $\Delta^+(D)\coloneqq \max\{d_D^+(v) \mid v\in V(G)\}$ and $\delta^+(D)\coloneqq \min\{d_D^+(v) \mid v\in V(G)\}$, respectively. The \emph{maximum/minimum indegree} and \emph{maximum/minimum degree} of $D$ are defined analogously and denoted by $\Delta^-(D)$, $\delta^-(D)$, $\Delta(D)$, and $\delta(D)$, respectively. We denote by $\Delta^0(D)\coloneqq \max\{\Delta^+(D), \Delta^-(D)\}$ the \emph{maximum semidegree} of $D$ and by $\delta^0(D)\coloneqq \min\{\delta^+(D), \delta^-(D)\}$ the \emph{minimum semidegree} of $D$.

Let $G$ be a (di)graph and $S\subseteq V(G)$. The \emph{neighbourhood} of $S$ in $G$ is the set $N_G(S)\coloneqq \bigcup_{v\in S}N_G(v)$. If $G$ is directed, the \emph{outneighbourhood} $N_G^+(S)$ and \emph{inneighbourhood} $N_G^-(S)$ of $S$ in $G$ are defined analogously.

\subsection{Regularity}\label{sec:notation-regularity}
An undirected graph $G$ is \emph{$r$-regular} if $d_G(v)=r$ for all $v\in V(G)$ and a digraph $D$ is \emph{$r$-regular} if $d_D^+(v)=r=d_D^-(v)$ for all $v\in V(D)$. A (di)graph is \emph{regular} if it is $r$-regular for some $r\in \mathbb{N}$. An undirected graph $G$ on $n$ vertices is \emph{$(\delta, \varepsilon)$-almost regular} if $d_G(v)= (\delta\pm \varepsilon)n$ for all $v\in V(G)$ and a digraph $D$ is  \emph{$(\delta, \varepsilon)$-almost regular} if both $d_D^+(v), d_D^-(v)= (\delta\pm \varepsilon)n$ for all $v\in V(D)$.

\subsection{Matchings}
A \emph{matching} is a set of pairwise non-adjacent edges. Given a vertex set $V$, a matching $M$ is called \emph{perfect} if $V(M)=V$. 

Let $M$ be a directed matching. We say that $M$ is a \emph{matching from $A$ to $B$} if $V(M)\subseteq A\cup B$ and all the edges of $M$ are directed from $A$ to $B$. We say that $M$ is a \emph{perfect matching from $A$ to $B$} if $M$ is a matching from $A$ to $B$ satisfying $V(M)=A\cup B$.

\subsection{Blow-ups}\label{sec:notation-blowup}
Let $D$ be a digraph and $r\in \mathbb{N}$. The \emph{$r$-fold blow-up of $D$} is the digraph $D'$ defined as follows. The vertex set $V(D')$ consists of $r$ copies of $v$ for each $v\in V(D)$. Let $u,v\in V(D)$ and $u',v'\in V(D')$. Suppose that $u'$ is a copy of $u$ and $v'$ is a copy of $v$. Then, $u'v'\in E(D')$ if and only if $uv\in E(D)$.
For each $v\in V(D)$, the set of $r$ copies of $v$ in $V(D')$ is called a \emph{vertex class} of $D'$.

\subsection{Paths and cycles}
Throughout this paper, all paths and cycles are directed, with consistently oriented edges. The number of edges contained in a path/cycle $P$ is called the \emph{length} of $P$ and denoted by $e(P)$.
A path $P$ is \emph{trivial} if $e(P)=0$. Given a vertex set $V$, a \emph{Hamilton cycle} is a cycle $C$ satisfying $V(C)=V$.

Let $P=v_1\dots v_\ell$ be a path. The \emph{starting point} of $P$ is $v_1$, the \emph{ending point} of $P$ is $v_\ell$, the \emph{endpoints} of $P$ are $v_1$ and $v_\ell$, and the \emph{internal vertices} of $P$ are $v_2, \dots, v_{\ell-1}$. We denote $V^+(P)\coloneqq \{v_1\}$, $V^-(P)\coloneqq \{v_\ell\}$, and $V^0(P)\coloneqq \{v_2, \dots, v_\ell\}$. A \emph{$(u,v)$-path} is a path which starts at $u$ and ends at $v$. Given $1\leq i\leq j\leq \ell$, denote by $v_iPv_j\coloneqq v_iv_{i+1}\dots v_j$ the $(v_i, v_j)$-path induced by $P$.

A set of vertex-disjoint paths is sometimes called a \emph{linear forest}. Given a set $\sP$ of (not necessarily disjoint) paths, we denote by $V^+(\sP)$ the set $\bigcup_{P\in \sP} V^+(P)$ of vertices which are the starting point of a path in $\sP$. Define $V^-(\sP)$ and $V^0(\sP)$ analogously. Note that $V^+(\sP), V^-(\sP)$, and $V^0(\sP)$ are always sets rather than multisets.

Let $\sP$ be a set of (not necessarily disjoint) paths. We sometimes abuse notation and view $\sP$ as a multidigraph. In particular, we denote by $V(\sP)$ the set $\bigcup_{P\in \sP}V(P)$ and by $E(\sP)$ the multiset $\bigcup_{P\in \sP}E(P)$. For any vertex $v\in V(\sP)$, we denote $d_\sP(v)=\sum_{P\in \sP}d_P(v)$, and define the out- and indegrees $d_\sP^+(v)$ and $d_{\sP}^-(v)$ of $v$ in $\sP$ analogously. 
Given a digraph $D$, we write $D\setminus \sP\coloneqq D\setminus E(\sP)$.

\subsection{Decompositions}\label{sec:notation-decomp}
Given a (di)graph $G$, a \emph{decomposition} of $G$ is a set of edge-disjoint sub(di)graphs of $G$ which altogether cover all the edges of $G$. A \emph{Hamilton decomposition} is a decomposition into Hamilton cycles.

Recall that the edges of a multi(di)graph are considered to be distinct. Thus, a \emph{decomposition} of a multi(di)graph $G$ is a set $\{H_1, \dots, H_\ell\}$ of submulti(di)graphs of $G$ such that $\mu_G(uv)=\sum_{i\in [\ell]} \mu_{H_i}(uv)$ for any distinct $u,v\in V(G)$.

\subsection{Hierarchies}
In a statement, the hierarchy $0<a\ll b\ll c\leq 1$ means that there exist non-decreasing functions $f\colon (0,1]\longrightarrow (0,1]$ and $g\colon (0,1]\longrightarrow (0,1]$ for which the statement holds for all $0<a,b,c\leq 1$ satisfying $b\leq f(c)$ and $a\leq g(b)$. Hierarchies with more constants are defined analogously and should always be read from right to left.
Whenever a constant appears in the form $\frac{1}{a}$ in a hierarchy, we implicit assume that $a\in \mathbb{N}$.

\subsection{\texorpdfstring{$\pm$}{Plus/minus}-notation}
To avoid repetitions, we sometime write statements of the form $\cC^\pm$ to mean that the statements $\cC^+$ and $\cC^-$ both hold.
In particular, if $\cC^\pm$ is a statement of the form ``$\cA^\pm$ implies $\cB^\pm$'', then we mean that ``$\cA^+$ implies $\cB^+$'' and ``$\cA^-$ implies $\cB^-$''. Similarly, a statement of the form ``$\cA^\pm$ implies $\cB^\mp$'' means that ``$\cA^+$ implies $\cB^-$'' and ``$\cA^-$ implies $\cB^+$''.

\section{Tripartite tournaments and some applications of Theorem \ref{thm:biprobexp}}\label{sec:applications}

	\onlyinsubfile{
		\setcounter{section}{13}
		\section{Tripartite tournaments and some applications of Theorem \ref{thm:biprobexp}}}
		
In this section, we construct a family of regular tripartite tournaments which cannot be decomposed into Hamilton cycles and derive several consequences of \cref{thm:biprobexp}. In particular, we prove \cref{prop:tripartite} and \cref{cor:1/2,cor:undirected}.
	
\subsection{Tripartite tournaments: proof of Proposition \ref{prop:tripartite}}

Let $n\geq 2$. We show that if $T$ is obtained from the $n$-fold blow-up of the directed $C_3$ by flipping the orientation of precisely one triangle, then $T$ does not have a Hamilton decomposition.

\begin{proof}[Proof of \cref{prop:tripartite}]
	Let $U_1, U_2$, and $U_3$ be disjoint vertex sets of size $n\geq 2$. For each $i\in [3]$, let $u_i\in U_i$. Denote $E\coloneqq \{u_1u_3,u_3u_2,u_2u_1\}$ and let $T$ be the digraph on $U_1\cup U_2\cup U_3$ defined by 
	\[E(T)\coloneqq E\cup (\{uv\mid i\in [3], u\in U_i, v\in U_{i+1}\}\setminus \{u_1u_2,u_2u_3,u_3u_1\})\]
	(where $U_4\coloneqq U_1$).
	Note that $E(T[U_i])=\emptyset$ for each $i\in [3]$.
	
	Suppose for a contradiction that $\sC$ is a Hamilton decomposition of $T$.
	Since $u_1u_3u_2$ is a non-spanning cycle of $T$, the edges $u_1u_3,u_3u_2$, and $u_2u_1$ do not all lie on a common Hamilton cycle in $\sC$. By the pigeon-hole principle, there exists $i\in [3]$ and $C\in \sC$ such that $E(C)\cap E=\{u_iu_{i-1}\}$ (where $u_0\coloneqq u_3$).
	Denote $C=u_iu_{i-1}v_1\dots v_{3n-2}$. Then, for each $j\in [3n-2]$, we have $v_j\in U_{i-1+j}$ (where the index $i-1+j$ is taken modulo $3$). In particular, $v_{3n-2}\in U_i$. But $u_i\in U_i$ and so $v_{3n-2}u_i\notin E(T)$, a contradiction.	
\end{proof}

Note that the above arguments can easily be extended to show that none of the edges in $E$ lie on a Hamilton cycle. 

Moreover, the same arguments can be used to show that, if $T$ is obtained from the complete blow-up $C_4$ on vertex classes of size $n$ by flipping the orientation of a set $E'$ of edges, then no Hamilton cycle of $T$ contains a single edge from $E'$. This illustrates the fact that the edges of $T$ with reversed direction in \cref{thm:blowupC4} will have to be decomposed in a ``balanced" way. 
	
\subsection{Bipartite robust expanders: proof of Corollary \ref{cor:undirected}}

The arguments of \cite[Lemma 3.6]{kuhn2014hamilton} can be easily adapted to the bipartite case to show that the edges of a regular bipartite robust expander can be oriented to form a regular bipartite robust outexpander.

\begin{lm}\label{lm:orientation}
	Let $0<\frac{1}{n}\ll \nu'\ll \nu \leq \tau \ll \delta\leq 1$ and let $r\geq \delta n$ be even. Let $G$ be an $r$-regular bipartite graph on vertex classes $A$ and $B$ of size $n$. Suppose that $G$ is a bipartite robust $(\nu, \tau)$-expander with bipartition $(A,B)$, as well as with bipartition $(B,A)$. Then, there exists an $\frac{r}{2}$-regular orientation $D$ of $G$ such that $D$ is a bipartite robust $(\nu', \tau)$-outexpander with bipartition $(A,B)$.
\end{lm}

This can be proved by considering a random orientation of the edges of $G$ and then adjusting the orientations of a small proportion of edges to ensure that $D$ is $\frac{r}{2}$-regular.

\begin{proof}[Proof of \cref{cor:undirected}]
	Let $\delta>0$ and let $\tau'$ be the constant obtained by applying \cref{thm:biprobexp}. We may assume without loss of generality that $\delta\ll 1$. Fix additional constants such that $0<\frac{1}{n_0}\ll \tau \ll \tau', \delta$ and $\frac{1}{n_0}\ll \nu$. Let $n\geq n_0$ and $r\geq \delta n$. Suppose that $r$ is even. Let $G$ be an $r$-regular balanced bipartite graph on vertex classes $A$ and $B$ of size $n$. Suppose that $G$ is a bipartite robust $(\nu, \tau)$-expander with bipartition $(A, B)$, as well as with bipartition $(B,A)$. By definition of a bipartite robust expander, we have $\nu \leq \tau$. Fix an additional constant such that $\frac{1}{n}\ll \nu'\ll \nu$.
	
	By \cref{lm:orientation}, there exists an $\frac{r}{2}$-regular orientation $D$ of $G$ such that $D$ is a bipartite robust $(\nu', \tau)$-outexpander with bipartition $(A,B)$.
	By definition, $D$ is also a bipartite robust $(\nu', \tau')$-outexpander with bipartition $(A,B)$.
	Apply \cref{thm:biprobexp} (with $\nu'$ and $\tau'$ playing the roles of $\nu$ and $\tau$) to obtain a Hamilton decomposition $\sC$ of $D$. Let $\sC'$ be obtained from $\sC$ by replacing each directed edge $uv\in E(\sC)$ by an undirected edge between $u$ and $v$. By construction, $\sC'$ is a Hamilton decomposition of $G$.
\end{proof}

\subsection{Dense bipartite digraphs: proof of Corollary \ref{cor:1/2}}

We show that any bipartite digraph of sufficiently large minimum semidegree is a bipartite robust outexpander.

\begin{lm}\label{lm:1/2rob}
	Let $0<\nu \leq \tau \ll \varepsilon<1$. Let $D$ be a bipartite digraph on vertex classes $A$ and $B$ of size $n$. Suppose that $\delta^0(D)\geq (\frac{1}{2}+\varepsilon)n$. Then, $D$ is a bipartite robust $(\nu,\tau)$-outexpander with bipartition $(A,B)$.
\end{lm}

\begin{proof}
	Let $S\subseteq A$ satisfy $\tau n\leq |S|\leq (1-\tau)n$ and denote $T\coloneqq RN_{\nu, D}^+(S)$. We show that $|T|\geq |S|+\nu n$.
	If $|S|\geq \frac{n}{2}$, then each $v\in B$ satisfies $|N_D^-(v)\cap S|\geq \varepsilon n$ and so $T=B$. We may therefore assume that $|S|\leq \frac{n}{2}$. Then,
	\[\left(\frac{1}{2}+\varepsilon\right)n|S|\leq e_D(S,B)\leq \nu n^2+|S||T| \leq \frac{\nu n}{\tau}|S|+|S||T|\]
	and so $|T|\geq (\frac{1}{2}+\varepsilon-\frac{\nu}{\tau})n\geq |S|+\nu n$.
	
	Similarly, if $S\subseteq B$ satisfies $\tau n\leq |S|\leq (1-\tau)n$, then $|RN_{\nu,D}^+(S)|\geq |S|+\nu n$. Thus, $D$ is a bipartite robust $(\nu, \tau)$-outexpander with bipartition $(A,B)$.
\end{proof}

\begin{proof}[Proof of \cref{cor:1/2}]
	Let $\delta >\frac{1}{2}$ and let $\tau>0$ be the constant obtained by applying \cref{thm:biprobexp}. Let $0<\nu\ll \tau,\delta$ and let $n_0\in \mathbb{N}$ be the constant obtained by applying \cref{thm:biprobexp}. 
	Let $D$ be a bipartite digraph on vertex classes $A$ and $B$ of size $n\geq n_0$. Suppose that $D$ is $r$-regular for some $r\geq \delta n$%
		\COMMENT{Note that $r\leq n-1$.}. 
	Denote $\varepsilon\coloneqq \frac{r}{n}-\frac{1}{2}$. By assumption, $0<\delta-\frac{1}{2}\leq \varepsilon\leq \frac{n-1}{n}-\frac{1}{2}\leq 1$. Fix an additional constant $\tau'$ such that $0<\nu\leq \tau'\ll \varepsilon, \tau$. By \cref{lm:1/2rob}, $D$ is a bipartite robust $(\nu, \tau')$-outexpander with bipartition $(A,B)$. By definition of a bipartite robust outexpander, $D$ is also a bipartite robust $(\nu, \tau)$-outexpander with bipartition $(A,B)$ and so \cref{thm:biprobexp} implies that $D$ has a Hamilton decomposition.
\end{proof}

\subsection{Optimal packings of Hamilton cycles}\label{sec:packings}

Given a graph $G$, denote by $\reg_{\rm even}(G)$ the maximum degree of an even-regular spanning subgraph of $G$. Given a digraph $D$, denote by $\reg(D)$ the maximum degree of a regular spanning subdigraph of $D$. Clearly, $\frac{\reg_{\rm even}(G)}{2}$ and $\reg(D)$ provide natural upper bounds on the size of an optimal packing of Hamilton cycles in a graph $G$ and a digraph $D$, respectively.
The following corollary of \cref{thm:biprobexp,cor:undirected} states that these are the correct values for $\varepsilon$-regular bipartite (di)graphs, random bipartite (di)graphs, and random bipartite tournaments. (Note that an approximate version of \cref{thm:packings}\cref{thm:packings-epsG} was already obtained by Frieze and Krivelevich \cite{frieze2005packing}.)

More precisely, a bipartite graph $G$ on vertex classes $A$ and $B$ is \emph{$\varepsilon$-regular} if 
\[\left|\frac{e_G(A,B)}{|A||B|}-\frac{e_G(A',B')}{|A'||B'|}\right|<\varepsilon\]
for all $A'\subseteq A$ and $B'\subseteq B$ which satisfy $|A'|\geq \varepsilon |A|$ and $|B'|\geq \varepsilon |B|$. A bipartite digraph $D$ on vertex classes $A$ and $B$ is \emph{$\varepsilon$-regular} if both $D[A,B]$ and $D[B,A]$ are $\varepsilon$-regular.

We denote by $G_{n,n,p}$ the \emph{binomial random bipartite graph}, which is obtained from the complete bipartite graph $K_{n,n}$ by selecting each edge independently with probability $p$. Similarly, $D_{n,n,p}$ denotes the \emph{binomial random bipartite digraph}, which is obtained from the complete bipartite digraph on vertex classes of size $n$ by selecting each edge independently with probability $p$.

\begin{cor}\label{thm:packings}
    For any $0<p\leq 1$, there exist $\varepsilon>0$ and $n_0\in \mathbb{N}$ for which the following hold.
    \begin{enumerate}
        \item Let $G$ be an $\varepsilon$-regular bipartite graph on vertex classes of size $n$ and suppose that $\delta(G)\geq p n$. Then, $G$ contains $\frac{\reg_{\rm even}(G)}{2}$ edge-disjoint Hamilton cycles.\label{thm:packings-epsG}
        \item Let $D$ be an $\varepsilon$-regular bipartite digraph on vertex classes of size $n$ and suppose that $\delta^0(D)\geq p n$. Then, $D$ contains $\reg(D)$ edge-disjoint Hamilton cycles.\label{thm:packings-epsD}
        \item With high probability, $G_{n,n,p}$ contains $\frac{\reg_{\rm even}(G_{n,n,p})}{2}$ edge-disjoint Hamilton cycles.\label{thm:packings-Gnnp}
        \item With high probability, $D_{n,n,p}$ contains $\reg(D_{n,n,p})$ edge-disjoint Hamilton cycles.\label{thm:packings-Dnnp}
        \item Let $T$ be chosen uniformly at random among the bipartite tournaments on vertex classes of size $n$. Then, $T$ contains $\reg(T)$ edge-disjoint Hamilton cycles with high probability.\label{thm:packings-T}
    \end{enumerate}
\end{cor}

The proof of \cref{thm:packings} is standard, so we only give a brief proof overview. \APPENDIX{(Full details can be found in \cref{app:packings}.)}\NOAPPENDIX{(Full details can be found in the appendix of the arXiv version of this paper.)}
Let $G,D$, and $T$ be defined as in \cref{thm:packings}. Observe that by \cref{thm:biprobexp,cor:undirected}, it is enough to show that each of $G$, $D$, $G_{n,n,p}$, $D_{n,n,p}$, and $T$ contain, with high probability, a spanning regular sub(di)graph of degree $\frac{\reg_{\rm even}(G)}{2}$, $\reg(D)$, $\frac{\reg_{\rm even}(G_{n,n,p})}{2}$, $\reg(D_{n,n,p})$, and $\reg(T)$, respectively, which is a bipartite robust (out)expander.

Arguments of \cite{frieze2005packing} imply that $\reg_{\rm even}(G)\geq (p- 2\varepsilon)n$. Thus, one can use basic properties of $\varepsilon$-regular bipartite graphs to show that any $\reg_{\rm even}(G)$-regular spanning subgraph of $G$ is still $\varepsilon$-regular. It is also easy to see that any $\varepsilon$-regular bipartite graph is also a bipartite robust expander. Thus, \cref{thm:packings}\cref{thm:packings-epsG} holds. Similar arguments hold for the directed case and so \cref{thm:packings}\cref{thm:packings-epsD} is satisfied.

A simple Chernoff bound can be used to show that $G_{n,n,p}$ is an $\varepsilon$-regular bipartite graph of minimum degree at least $(p-\varepsilon)n$ with high probability. Thus, \cref{thm:packings}\cref{thm:packings-Gnnp} follows from \cref{thm:packings}\cref{thm:packings-epsG}. Similarly, \cref{thm:packings}\cref{thm:packings-Dnnp,thm:packings-T} follow from \cref{thm:packings}\cref{thm:packings-epsD}.

\onlyinsubfile{\bibliographystyle{abbrv}
	\bibliography{Bibliography/Bibliography}}

\section{Preliminaries}\label{sec:preliminaries}

We now introduce some preliminary tools and results which will be used throughout this paper.

\subsection{(Bipartite) robust (out)expanders}

	\onlyinsubfile{
		\setcounter{section}{5}
\subsection{(Bipartite) robust outexpanders}}

In \cref{sec:intro-biprobexp}, we introduced the concept of (bipartite) robust (out)expansion. We start that by recalling and expanding on these definitions.

\subsubsection{Definitions}

Let $D$ be a digraph on $n$ vertices. Recall that for any $S\subseteq V(D)$, we denote by $RN_{\nu, D}^+(S)$ the set of vertices $v\in V(D)$ which satisfy $|N_D^-(v)\cap S|\geq \nu n$.
Then, we say that $D$ is a \emph{robust $(\nu, \tau)$-outexpander} if, for any $S\subseteq V(D)$ satisfying $\tau n\leq |S|\leq (1-\tau)n$, we have $|RN_{\nu, D}^+(S)|\geq |S|+\nu n$.

Let $G$ be a graph on $n$ vertices. Recall that for any $S\subseteq V(G)$, we denote by $RN_{\nu, G}(S)$ the set of vertices $v\in V(G)$ which satisfy $|N_D(v)\cap S|\geq \nu n$.
Then, we say that $G$ is a \emph{robust $(\nu, \tau)$-expander} if, for any $S\subseteq V(G)$ satisfying $\tau n\leq |S|\leq (1-\tau)n$, we have $|RN_{\nu, G}(S)|\geq |S|+\nu n$.
Let $G$ be a bipartite graph on vertex classes $A$ and $B$ of size $n$.
We say that $G$ is a \emph{bipartite robust $(\nu, \tau)$-expander with bipartition $(A,B)$} if, for any $S\subseteq A$ satisfying $\tau n\leq |S|\leq (1-\tau)n$, we have $|RN_{\nu, G}(S)|\geq |S|+\nu n$.
Note that the order of $A$ and $B$ matters.

In \cref{sec:intro-biprobexp}, we defined an analogue of bipartite robust expanders for digraphs. 
Let $D$ be a bipartite digraph on vertex classes $A$ and $B$ of size $n$.
We say that $D$ is a \emph{bipartite robust $(\nu, \tau)$-outexpander with bipartition $(A,B)$} if
\begin{itemize}
	\item for any $S\subseteq A$ such that
	$\tau n\leq |S|\leq (1-\tau)n$, we have $|RN_{\nu, D}^+(S)|\geq |S|+\nu n$; and
	\item for any $S\subseteq B$ such that
	$\tau n\leq |S|\leq (1-\tau)n$, we have $|RN_{\nu, D}^+(S)|\geq |S|+\nu n$.
\end{itemize}
Note that, here, the order of $A$ and $B$ does not matter.

\subsubsection{Basic properties of (bipartite) robust (out)expanders}

The following \lcnamecrefs{fact:biprobexp} hold by definition.

\begin{fact}\label{fact:biprobexp}
	A digraph $D$ is a bipartite robust $(\nu, \tau)$-outexpander with bipartition $(A,B)$ if and only if $D[A, B]$ is a bipartite robust $(\nu, \tau)$-expander with bipartition $(A,B)$ and $D[B, A]$ is a robust $(\nu, \tau)$-expander with bipartition $(B,A)$.
\end{fact}

\begin{fact}\label{fact:biprobexpparameters}
	Suppose that $G$ is a bipartite robust $(\nu,\tau)$-expander with bipartition $(A,B)$. Then, for any $\nu'\leq \nu$ and $\tau'\geq \tau$, $G$ is a bipartite robust $(\nu',\tau')$-expander with bipartition $(A,B)$.
\end{fact}

By definition, bipartite robust outexpansion is preserved when only a few edges are removed at each vertex.

\begin{lm}\label{cor:verticesedgesremovalbiproboutexp}
	Let $0<\frac{1}{n}\ll\varepsilon\leq \nu\leq 1$%
		\COMMENT{No need to add $\tau$ in hierarchy since the existence of $D$ automatically implies $\tau\geq \nu$.}. 
	Let~$D$ be a bipartite digraph on vertex classes $A$ and $B$ of size $n$. Suppose that $D$ is a bipartite robust~$(\nu, \tau)$-outexpander with bipartition $(A,B)$. 
	If~$D'$ is obtained from~$D$ by removing at most~$\varepsilon n$ inedges and $\varepsilon n$ outedges at each vertex, then~$D'$ is a bipartite robust~$(\nu-\varepsilon, \tau)$-expander with bipartition $(A,B)$.
\end{lm}

\COMMENT{\begin{proof}
	For any $S\subseteq V(D)$, we have $RN_{\nu-\varepsilon,D'}^+(S)\supseteq RN_{\nu, D}^+(S)$.
\end{proof}}

In \cite{kuhn2010hamiltonian,keevash2009exact}, Keevash, K\"uhn, Osthus, and Treglown showed that a robust outexpander of linear minimum degree is Hamiltonian.

\begin{thm}[{\cite[Theorem 16]{kuhn2010hamiltonian}}]\label{lm:rob}
	Let $0<\frac{1}{n}\ll \nu\ll \tau\leq \frac{\delta}{2}\leq 1$. Let~$D$ be a robust~$(\nu, \tau)$-outexpander on~$n$ vertices with $\delta^0(D)\geq \delta n$. Then, $D$ is Hamiltonian.
\end{thm}

The analogue of \cref{lm:rob} holds for bipartite robust outexpanders. This can be derived from \cref{lm:rob} using the procedure presented in \cref{sec:sketch-cycle}. The formal proof is deferred to the end of \cref{sec:contractingM}, where we introduce the required definitions.

\begin{cor}\label{lm:biprobHamcycle}
	Let $0<\frac{1}{n}\ll \nu\ll \tau\leq \delta\leq 1$. Let~$D$ be a balanced bipartite digraph on vertex classes $A$ and $B$ of size~$n$. Suppose that $D$ is a bipartite robust~$(\nu, \tau)$-outexpander with bipartition $(A,B)$ and that $\delta^0(D)\geq \delta n$. Then, $D$ is Hamiltonian.
\end{cor}

Almost complete bipartite graphs are bipartite robust expanders.

\begin{prop}\label{prop:almostcompleterob}
	Let $0<\frac{1}{n}\ll \varepsilon\ll \nu\ll \tau \ll 1$. Let $G$ be a bipartite graph on vertex classes $A$ and $B$ of size $n$. If $\delta(G)\geq (1-\varepsilon)n$, then $G$ is a bipartite robust $(\nu, \tau)$-expander with bipartition $(A,B)$.
\end{prop}

\begin{proof}
	Let $S\subseteq A$ satisfy $\tau n\leq |S|\leq (1-\tau)n$. 
	Each $v\in B$ satisfies $|N_G(v)\cap S|\geq (1-\varepsilon)n-|A\setminus S|\geq (\tau-\varepsilon)n\geq \nu n$. Thus, $|RN_{\nu, G}(S)|=|B|\geq |S|+\tau n\geq |S|+\nu n$.
	Similarly, if $S'\subseteq B$ satisfies $\tau n\leq |S'|\leq (1-\tau)n$, then $|RN_{\nu, G}(S')|\geq |S'|+\nu n$.
\end{proof}

Recall the definition of an $r$-fold blow-up from \cref{sec:notation-blowup}.
The next \lcnamecref{lm:biprobblowup2} states that bipartite robust outexpansion is preserved when taking $r$-fold blow-ups.
The proof is very similar to that of its non-bipartite analogue (see \cite[Lemma 5.3]{kuhn2013hamilton}), so we omit the details.

\begin{lm}\label{lm:biprobblowup2}
	Let $0<3\nu\le \tau<1$ and $r\ge 3$. Let $D$ be a balanced bipartite robust $(\nu,\tau)$-outexpander with bipartition $(A, B)$. Let $D'$ be the $r$-fold blow-up of	$D$. Let $A'$ be the set of vertices in $V(D')$ which are a copy of a vertex in $A$. Let $B'\coloneqq V(D')\setminus A'$. Then $D'$ is a bipartite robust $(\nu^3,2\tau)$-outexpander with bipartition $(A', B')$.
\end{lm}

\onlyinsubfile{\bibliographystyle{abbrv}
\bibliography{Bibliography/Bibliography}}

\subsection{(Super)regularity}\label{sec:regularity}

	\onlyinsubfile{
		\setcounter{section}{5}
		\setcounter{subsection}{1}
\subsection{Regularity}}

In \cref{sec:packings}, we introduced the concept of $\varepsilon$-regular (di)graphs. We start by recalling and expanding on these definitions.

\subsubsection{Definitions}

Suppose that $G$ is an undirected bipartite graph on vertex classes $A$ and $B$. The \emph{density} of $G$ is defined as
\[d_G(A,B)\coloneqq \frac{e_G(A,B)}{|A||B|}.\]
Let $\varepsilon>0$. We say that $G$ is \emph{$\varepsilon$-regular} if, for any $A'\subseteq A$ and $B'\subseteq B$ satisfying $|A'|\geq \varepsilon |A|$ and $|B'|\geq \varepsilon |B|$, we have $|d_G(A,B)-d_G(A',B')|<\varepsilon$.
Let $0\leq d \leq 1$. We say that $G$ is \emph{$(\varepsilon, d)$-regular} if $G$ is $\varepsilon$-regular and has density $d_G(A,B)=d$. We say that $G$ is \emph{$(\varepsilon, \geq d)$-regular} if there exists $d'\geq d$ such that $G$ is $(\varepsilon, d')$-regular.
We say that $G$ is \emph{$[\varepsilon, d]$-superregular} if $G$ is $\varepsilon$-regular, each $a\in A$ satisfies $d_G(a)=(d\pm \varepsilon)|B|$, and each $b\in B$ satisfies $d_G(b)=(d\pm\varepsilon)|A|$. We say that $G$ is \emph{$[\varepsilon, \geq d]$-superregular} if there exists $d'\geq d$ such that $G$ is $[\varepsilon, d']$-superregular.

\subsubsection{Basic properties of (super)regular pairs}

The next \lcnamecref{prop:epsremovingadding} states that (super)regulari\-ty is preserved when few vertices and edges are removed and/or added to a bipartite graph. This follows easily from the definitions and a similar observation was already made (and proved) in \cite[Proposition 4.3]{kuhn2013hamilton}, so we omit its proof here.

\begin{prop}\label{prop:epsremovingadding}
	Let $0<\frac{1}{m}\ll \varepsilon \leq \varepsilon'\leq d\leq 1$. Let $G$ be a bipartite graph on vertex classes $A$ and $B$ of size at least $m$. Let $A'$ and $B'$ be disjoint vertex sets satisfying $|A\triangle A'|\leq \varepsilon |A'|$ and $|B\triangle B'|\leq \varepsilon |B'|$. Let $G'$ be a bipartite graph on vertex classes $A'$ and $B'$ and suppose that $G'[A'\cap A, B'\cap B]$ is obtained from $G[A'\cap A, B'\cap B]$ by removing and adding at most $\varepsilon'|B'|$ edges incident to each vertex in $A'\cap A$ and at most $\varepsilon'|A'|$ edges incident to each vertex in $B'\cap B$.
	\begin{enumerate}
		\item If $G$ is $(\varepsilon, \geq d)$-regular, then $G'$ is $(3\sqrt{\varepsilon'}, \geq d-3\sqrt{\varepsilon'})$-regular.\label{prop:epsremovingadding-reg}
		\item Suppose that $A'\subseteq A$ and $B'\subseteq B$. If $G$ is $[\varepsilon, d]$-superregular, then $G'$ is $[3\sqrt{\varepsilon'},d]$-superregular.\label{prop:epsremovingadding-supreg}
	\end{enumerate}
\end{prop}

\begin{lm}[{\cite[Proposition 4.14]{kuhn2013hamilton}}]\label{lm:epsperfectmatching}
	Let $0<\frac{1}{m}\ll \varepsilon\ll \delta\leq 1$. Let $G$ be a balanced bipartite graph on vertex classes of size $m$. Suppose that $G$ is $\varepsilon$-regular and $\delta(G)\geq \delta m$. Then, $G$ contains a perfect matching.
\end{lm}

One can easily verify from the definition of superregularity that bipartite graphs of very high minimum degree are superregular.

\begin{prop}\label{prop:almostcompleteeps}
	Let $0<\frac{1}{m}\ll \varepsilon\ll \varepsilon'\ll 1$. Let $G$ be a bipartite graph on vertex classes $A$ and $B$ of size at least $m$. Suppose that each $a\in A$ satisfies $d_G(a)\geq (1-\varepsilon) |B|$ and each $b\in B$ satisfies $d_G(b)\geq (1-\varepsilon)|A|$. Then, $G$ is $[\varepsilon', \geq 1-\varepsilon']$-superregular.
\end{prop}

\COMMENT{\begin{proof}
		Note that $d_G(A, B)\geq 1-\varepsilon$.
		Let $A'\subseteq A$ satisfy $|A'|\geq \varepsilon' |A|$ and $B'\subseteq B$ satisfy $|B'|\geq \varepsilon' |B|$.
		Then, 
		\[1\geq d_G(A', B')\geq \frac{|A'|(|B'|-\varepsilon |B|)}{|A'||B'|}\geq 1-\frac{\varepsilon}{\varepsilon'}\geq 1-\varepsilon'\]
		and so $G$ is $\varepsilon'$-regular.
		Thus, $G$ is $[\varepsilon', d]$-superregular for each $1-\varepsilon'\leq d\leq 1$.
\end{proof}}

\begin{lm}[{\cite[Corollary 4.15]{kuhn2013hamilton}}]\label{lm:regularitypaths}
	Let $0<\frac{1}{m}\ll \varepsilon\ll d\leq 1$ and $k\geq 4$. Let $D$ be a digraph and $V_1\cup \dots \cup V_k$ be a partition of $V(D)$ into $k$ clusters of size $m$. Suppose that $D[V_i, V_{i+1}]$ is $[\varepsilon, \geq d]$-superregular for each $i\in [k-1]$. Let $u_1, \dots, u_m$ and $v_1, \dots, v_m$ be enumerations of $V_1$ and $V_k$, respectively. Then, $D$ contains a spanning set $\sP$ of $m$ vertex-disjoint paths, one $(u_i,v_i)$-path for each $i\in [m]$.
\end{lm}

If the pair $D[V_k, V_1]$ is also superregular, one can find a matching in $D(V_k,V_1)$ to tie the paths obtained with \cref{lm:regularitypaths} into a Hamilton path.

\begin{cor}\label{cor:regularityHam}
	Let $0<\frac{1}{m}\ll \varepsilon\ll d\leq 1$ and $k\geq 4$. Let $D$ be a digraph and $V_1\cup \dots \cup V_k$ be a partition of $V(D)$ into $k$ clusters of size $m$. Suppose that $D[V_i, V_{i+1}]$ is $[\varepsilon, \geq d]$-superregular for each $i\in [k]$ (where $V_{k+1}\coloneqq V_1$). Let $u\in V_1$ and $v\in V_k$. Then, $D$ contains a Hamilton $(u,v)$-path.
\end{cor}

\begin{proof}
	By \cref{prop:epsremovingadding}, $D[V_k\setminus \{v\}, V_1\setminus \{u\}]$ is still $[3\sqrt{\varepsilon}, \geq d]$-superregular and so \cref{lm:epsperfectmatching} implies that there exists a perfect matching $M\subseteq E_D(V_k\setminus \{v\}, V_1\setminus \{u\})$. Let $v_1u_1, \dots, v_{m-1}u_{m-1}$ be an enumeration of $M$. Denote $u_0\coloneqq u$ and $v_m\coloneqq v$. Let $\sP$ be the spanning set of vertex-disjoint paths obtained by applying \cref{lm:regularitypaths} with $u,u_1, \dots, u_{m-1}$ and $v_1,\dots, v_{m-1}, v$ playing the roles of $u_1, \dots, u_m$ and $v_1, \dots, v_m$. For each $i\in [m]$, let $P_i$ denote the $(u_{i-1},v_i)$-path contained in $\sP$. Then, $uP_1v_1u_1P_2v_2\dots u_{m-1}P_mv$ is a Hamilton $(u,v)$-path of $D$.
\end{proof}

Let $D$ be a digraph and suppose that $V(D)$ is partitioned into clusters which form superregular pairs. Then, one can adjust this partition in such a way that superregularity is preserved and all the vertices of a small given set $S$ are concentrated into few of the clusters.

\begin{lm}\label{lm:adjustP}
	Let $0<\frac{1}{n}\ll \varepsilon \ll \frac{1}{k}\ll \varepsilon'\ll \varepsilon''\ll 1$. Let $U_1, \dots, U_4$ be disjoint vertex sets of size $n$. Let $D$ be a digraph on $U_1\cup \dots \cup U_4$. For each $i\in [4]$, let $\cP_i$ be a partition of $U_i$ into $k$ clusters of size $\frac{n}{k}$. Suppose that for each $i\in [4]$, $D[V,W]$ is $[\varepsilon', \geq 1-\varepsilon']$-superregular whenever $V\subseteq U_i$ and $W\subseteq U_{i+1}$ are unions of clusters in $\cP_i$ and $\cP_{i+1}$, respectively (where $U_5\coloneqq U_1$ and $\cP_5\coloneqq \cP_1$).
	Let $S\subseteq V(D)$ satisfy $|S|\leq \varepsilon n$. Define $\cP_3'\coloneqq \cP_3$ and $\cP_4'\coloneqq \cP_4$. Then, there exists, for each $i\in [2]$, a partition $\cP_i'$ of $U_i$ into $k$ clusters of size $\frac{n}{k}$ such that the following hold.
	\begin{enumerate}
		\item For each $i\in [4]$, $D[V,W]$ is $[\varepsilon'', \geq 1-\varepsilon'']$-superregular whenever $V\subseteq U_i$ and $W\subseteq U_{i+1}$ are unions of clusters in $\cP_i'$ and $\cP_{i+1}'$, respectively (where $\cP_5'\coloneqq \cP_1'$).\label{lm:adjustP-supereg}
		\item For each $i\in [2]$, there exists a cluster $V\in \cP_i'$ for which $S\cap U_i\subseteq V$.\label{lm:adjustP-S}
	\end{enumerate}
\end{lm}

\begin{proof}
	For each $i\in [2]$, denote by $V_{i,1}, \dots, V_{i,k}$ the clusters in $\cP_i$ and observe that since $|S\cap U_i|\leq |V_{i,k}|$, one can greedily swap each vertex in $S\cap (U_i\setminus V_{i,k})$ with a distinct vertex in $V_{i,k}\setminus S$ to obtain a partition $\cP_i'$ of $U_i$ into $k$ clusters $V_{i,1}', \dots, V_{i,k}'$ such that $S\cap U_i\subseteq V_{i,k}'$ and 
	\begin{equation*}
		|V_{i,j}\triangle V_{i,j}'|\leq |S\cap U_i|\leq \frac{\varepsilon' n}{k}
	\end{equation*} 
	for each $j\in [k]$.
	Then, \cref{lm:adjustP-S} holds.
	Moreover, \cref{lm:adjustP-supereg} follows easily from \cref{prop:epsremovingadding}.%
\end{proof}

\subsubsection{The regularity lemma}

We now state a degree form of Szemer\'{e}di's regularity lemma for balanced bipartite digraphs. In \cite{alon2004testing}, Alon and Shapira proved a regularity lemma for digraphs. A degree form can be derived using similar arguments as the undirected version (see e.g.\ \cite{taylor2013regularity}).
The bipartite version stated below can easily be obtained by adjusting the partition obtained with the degree form regularity lemma for digraphs.

\begin{lm}[Degree form regularity lemma for balanced bipartite digraphs]\label{lm:bipreglm}
	For all $\varepsilon>0$ and $M'\in \mathbb{N}$, there exist $M,n_0\in \mathbb{N}$ such that, if $D$ is a balanced bipartite digraph on vertex classes $A$ and $B$ of size $n\geq n_0$ and $d\in [0,1]$, then there exist a spanning subdigraph $D'\subseteq D$ and a partition of $V(D)$ into an \emph{exceptional set} $V_0$ and $2k$ \emph{clusters} $V_1, \dots, V_{2k}$ such that the following hold.
	\begin{enumerate}
		\item $M'\leq 2k \leq M$.\label{lm:bipreglm-k}
		\item $|V_0\cap A|=|V_0\cap B|\leq \varepsilon n$.\label{lm:bipreglm-V0}
		\item For each $i\in [2k]$, either $V_i\subseteq A$ or $V_i\subseteq B$.\label{lm:bipreglm-AB}
		\item $|V_1|=\dots =|V_{2k}|\eqqcolon m$. In particular, there are precisely $k$ indices $i\in [2k]$ such that $V_i\subseteq A$ and precisely $k$ indices $i\in [2k]$ such that $V_i\subseteq B$. \label{lm:bipreglm-m}
		\item For each $v\in V(D)$, $d_{D'}^\pm(v)>d_D^\pm(v)-(d+\varepsilon)n$.\label{lm:bipreglm-deg}
		\item For each $i\in [2k]$, $D'[V_i]$ is empty.\label{lm:bipreglm-empty}
		\item Let $i,j\in [2k]$ be distinct. Then, $D'[V_i,V_j]$ is either empty or $(\varepsilon, \geq d)$-regular. Moreover, if $D'[V_i,V_j]$ is non-empty, then $D'[V_i,V_j]=D[V_i,V_j]$.\label{lm:bipreglm-reg}
	\end{enumerate}
\end{lm}

Let $\varepsilon>0$, $M'\in \mathbb{N}$, and $d\in [0,1]$.
Let $D$ be a balanced bipartite digraph.
The \emph{bipartite pure digraph of $D$ with parameters $\varepsilon, d$, and $M'$} is the digraph $D'\subseteq D$ obtained by applying \cref{lm:bipreglm} with these parameters.
The \emph{bipartite reduced digraph of $D$ with parameters $\varepsilon, d$, and $M'$} is the digraph $R$ defined as follows. Let $V_0, V_1, \dots, V_{2k}$ be the partition of $V(D)$ obtained by applying \cref{lm:bipreglm} with parameters $\varepsilon, d$, and $M'$. Denote by $D'$ the bipartite pure digraph of $D$ with parameters $\varepsilon, d$, and $M'$.
Then, $V(R)\coloneqq \{V_i\mid i\in [2k]\}$ and, for any distinct $U,V\in V(R)$, $UV\in E(R)$ if and only if $D'[U,V]$ is non-empty. Note that \cref{lm:bipreglm}\cref{lm:bipreglm-reg} implies that $D'[U,V]=D[U,V]$ is $(\varepsilon, \geq d)$-regular for any $UV\in E(R)$ and \cref{lm:bipreglm}\cref{lm:bipreglm-AB} implies that $R$ is a bipartite digraph on vertex classes $\{V\in V(R)\mid V\subseteq A\}$ and $\{V\in V(R)\mid V\subseteq B\}$.

The following lemma states that if a balanced bipartite digraph $D$ is a robust outexpander, then so is its corresponding bipartite reduced digraph.
The proof is very similar to that of its non-bipartite analogue (see \cite[Lemma 14]{kuhn2010hamiltonian}) and is therefore omitted.

\begin{lm}\label{lm:Rrob}
	Let $0<\frac{1}{n}\ll \varepsilon \ll d\ll \nu, \tau, \delta \leq 1$ and $\frac{M'}{n}\ll 1$. Let $D$ be a balanced bipartite digraph on vertex classes $A$ and $B$ of size $n$. Suppose that $D$ is a bipartite robust $(\nu, \tau)$-outexpander and that $\delta^0(D)\geq \delta n$.
	Let $R$ be the bipartite reduced digraph of $D$ with parameters $\varepsilon, d$, and $M'$.
	Then, $\delta^0(R)\geq \frac{\delta |R|}{4}$ and $R$ is a bipartite robust $(\frac{\nu}{2}, 2\tau)$-outexpander with bipartition $(\cA,\cB)$, where $\cA\coloneqq \{V\in V(R)\mid V\subseteq A\}$ and $\cB\coloneqq \{V\in V(R)\mid V\subseteq B\}$. 
\end{lm}

\onlyinsubfile{\bibliographystyle{abbrv}
\bibliography{Bibliography/Bibliography}}

\subsection{Probabilistic estimates}\label{sec:probability}

	\onlyinsubfile{
		\setcounter{section}{5}
		\setcounter{subsection}{2}
\subsection{Probabilistic estimates}}

Let~$X$ be a random variable. We write~$X\sim \Bin(n,p)$ if~$X$ follows  a binomial distribution with parameters~$n$ and~$p$. Let~$N,n,m\in \mathbb{N}$ be such that~$\max\{n, m\} \leq N$. Let~$\Gamma$ be a set of size~$N$ and~$\Gamma'\subseteq \Gamma$ be of size~$m$. Recall that~$X$ has a \textit{hypergeometric distribution with parameters~$N, n$, and $m$} if~$X=|\Gamma_n\cap \Gamma'|$, where~$\Gamma_n$ is a random subset of~$\Gamma$ with~$|\Gamma_n|=n$ (i.e.~$\Gamma_n$ is obtained by drawing~$n$ elements of~$\Gamma$ without replacement). We will denote this by~$X\sim \HGeom(N,n,m)$. \COMMENT{Note that if~$X\sim \HGeom(N,n,m)$ then~$\mathbb{E}[X]=\frac{nm}{N}$.}

\subsubsection{Chernoff's bound}

First, we will need Chernoff's bound.

\begin{lm}[{Chernoff's bound, see e.g.~\cite[Theorems 2.1 and 2.10]{janson2011random}}]\label{lm:Chernoff}
	Assume~$X\sim \Bin(n,p)$ or~$X\sim \HGeom(N,n,m)$. Then, for any~$0< \varepsilon \leq 1$, the following hold.
	\begin{enumerate}
		\item $\mathbb{P}\left[X\leq (1-\varepsilon)\mathbb{E}[X]\right] \leq \exp \left(-\frac{\varepsilon^2}{3}\mathbb{E}[X]\right)$.
		\item $\mathbb{P}\left[X\geq (1+\varepsilon)\mathbb{E}[X]\right] \leq \exp \left(-\frac{\varepsilon^2}{3}\mathbb{E}[X]\right)$.
	\end{enumerate}
\end{lm}

One can use \cref{lm:Chernoff} to show that (super)regularity is preserved with high probability when taking a random edge-slice, i.e.\ when selecting a random spanning subgraph by including each edge independently with some fixed probability $p$. This was already observed in (the proof of) \cite[Lemma 4.10(iv)]{kuhn2013hamilton} and so we omit the details here.

\begin{lm}\label{lm:randomreg}
	Let $0<\frac{1}{n}\ll \varepsilon\ll \varepsilon'\ll d\leq 1$ and let $\varepsilon\ll p\leq 1$.
	Let $G$ be a bipartite graph on vertex classes of size $n$ and let $G'$ be obtained from $G$ by selecting each edge independently with probability $p$.
	\begin{enumerate}
		\item If $G$ is $(\varepsilon, \geq d)$-regular, then $G'$ is $(\varepsilon', \geq pd-\varepsilon)$-regular with high probability.
		\item If $G$ is $[\varepsilon, d]$-superregular, then $G'$ is $[\varepsilon', pd]$-superregular with high probability.
	\end{enumerate}
\end{lm}

\COMMENT{\begin{proof}
	Use the arguments of \cite[Lemma 4.10(iv)]{kuhn2013hamilton} to get $\varepsilon'$-regularity. We have $\mathbb{E}[d_{G'}(A,B)]=p\cdot d_G(A,B)\geq pd$ and so \cref{lm:Chernoff} implies that
	\[\mathbb{P}[d_{G'}(A,B)< pd-\varepsilon]\leq \mathbb{P}[d_{G'}(A,B)< (1-\varepsilon)\mathbb{E}[d_{G'}(A,B)]]\leq \exp\left(-\frac{\varepsilon^2 pd}{3}\right).\] If $d_G(v)=(d\pm \varepsilon)n$, then $\mathbb[d_{G'}(v)]=(pd\pm p\varepsilon)n$ and so \cref{lm:Chernoff} implies that
	\[\mathbb{P}[d_{G'}(v)\neq (pd \pm \varepsilon')n]\leq \mathbb{P}[d_{G'}(v)\neq (1 \pm 2\varepsilon')\mathbb{E}[d_{G'}(v)]]\leq \exp\left(-\frac{4(\varepsilon')^2(pd-\varepsilon')n}{3}\right).\]
\end{proof}}

\begin{cor}\label{lm:edgeslice}
	Let $0<\frac{1}{n}\ll \varepsilon\ll \varepsilon'\ll \frac{1}{k}\ll d\ll 1$.
	Let $U_1, \dots, U_4$ be disjoint vertex sets of size $n$.
	Let $D$ be a digraph on $U_1\cup \dots \cup U_4$. For each $i\in [4]$, let $\cP$ be a partition of $U_i$ into $k$ clusters of size $\frac{n}{k}$. Suppose that for each $i\in [4]$, $D[V, W]$ is $[\varepsilon, \geq 1-\varepsilon]$-superregular whenever $V$ and $W$ are unions of clusters in $\cP_i$ and $\cP_{i+1}$, respectively (where $\cP_5\coloneqq \cP_1$).
	Let $D_1$ be obtained by selecting each edge of $D$ independently with probability $1-2d$. Let $D_2\coloneqq D\setminus D_1$.
	Then, the following holds with high probability. For each $i\in [4]$, $D_1[V, W]$ is $[\varepsilon', \geq 1-3d]$-superregular and $D_2[V,W]$ is $[\varepsilon', \geq d+\varepsilon']$-superregular whenever $V$ and $W$ are unions of clusters in $\cP_i$ and $\cP_{i+1}$, respectively.
\end{cor}

\COMMENT{\begin{proof}
		Let $i\in [4]$ and let $V$ and $W$ be unions of clusters in $\cP_i$ and $\cP_{i+1}$, respectively. Using the arguments of \cite[Lemma 4.10(iv)]{kuhn2013hamilton}, $D_1[V,W]$ is $[\varepsilon', \geq (1-2d)(1-\varepsilon)]$-superregular with high probability and $D_2[V,W]$ is $[\varepsilon', \geq 2d(1-\varepsilon)]$-superregular with high probability. There are
		\[4\left(\sum_{j\in [k]}\binom{k}{j}\right)^2=4(2^k-1)^2\]
		choices for $(i,V,W)$ so a union bound gives the desired result.
\end{proof}}

One can also use \cref{lm:Chernoff} to show that bipartite robust outexpansion is preserved with high probability when taking random edge-slices. The arguments are similar to those used in the proof of \cite[Lemma 3.2(ii)]{kuhn2014hamilton} and are therefore omitted.

\begin{lm}\label{lm:randomrob}
	Let $0<\frac{1}{n}\ll \nu \ll \tau \ll \gamma \ll \delta \leq 1$. Let $D$ be a balanced bipartite digraph on vertex classes $A$ and $B$ of size $n$. Suppose that $D$ is a bipartite robust $(\nu, \tau)$-outexpander with bipartition $(A,B)$. Let $D'$ be obtained from $D$ by taking each edge independently with probability $\frac{1}{2}$. 
	Then, with high probability, both $D'$ and $D\setminus D'$ are bipartite robust $(\frac{\nu}{4}, \tau)$-outexpanders with bipartition $(A,B)$.
\end{lm}

\subsubsection{McDiarmid's inequality}

We will also need McDiarmid's inequality.

\begin{lm}[{McDiarmid's inequality \cite{mcdiarmid1989method}}]\label{lm:McDiarmid}
	Let $X_1, \dots, X_n$ be independent random variables, each taking values in $\{0,1\}$. Let $c_1, \dots, c_n\in \mathbb{R}$ and let $f\colon \{0,1\}^n\longrightarrow \mathbb{R}$ be a measurable function. Suppose that for any $i\in [n]$ and $x_1, \dots, x_n, x_i'\in \{0,1\}$, we have
	\[|f(x_1, \dots, x_n)-f(x_1, \dots, x_{i-1}, x_i', x_{i+1}, \dots, x_n)|\leq c_i.\]
	Then, for any $t>0$,
	\[\mathbb{P}[|f(X_1, \dots, X_n)-\mathbb{E}[f(X_1, \dots, X_n)]|>t]\leq 2\exp\left(-\frac{2t^2}{\sum_{i\in [n]}c_i^2}\right).\]
\end{lm}

\begin{lm}\label{lm:edgepartition}
	Let $0<\frac{1}{n}\ll \varepsilon\ll \frac{1}{k}\ll 1$.
	Let $G$ be a bipartite graph on vertex classes $A$ and $B$ of size $n$. 
	Suppose that $\Delta(G)\leq \varepsilon n$ and $e(G)\geq \frac{n}{2}$.
	Let $A_1\cup \dots\cup A_k$ be a random partition of $A$ such that, for each $i\in [k]$ and $v\in A$, $v\in A_i$ with probability $\frac{1}{k}$ independently of all other vertices.
	Similarly, let $B_1\cup \dots \cup B_k$ be a random partition of $B$ such that, for each $i\in [k]$ and $v\in B$, $v\in B_i$ with probability $\frac{1}{k}$ independently of all other vertices.
	Then, with probability at least $\frac{4}{5}$, we have $e_G(A_i, B_j)\geq \frac{e(G)}{2k^2}$ for all $i,j\in [k]$.
\end{lm}

\begin{proof}
	Denote $A=\{a_1, \dots, a_n\}$ and $B=\{b_1, \dots, b_n\}$.
	Let $i,j\in [k]$. For each $\ell\in [n]$, let
			\[X_\ell\coloneqq 
			\begin{cases}
				1 & \text{if }a_\ell\in A_i;\\
				0 & \text{otherwise;}
			\end{cases}
			\quad \text{and} \quad
			X_{2n+1-\ell}\coloneqq 
			\begin{cases}
				1 & \text{if }b_\ell\in B_j;\\
				0 & \text{otherwise.}
			\end{cases}\]
			Let $f(X_1, \dots, X_{2n})\coloneqq e_G(A_i,B_j)$. Then, $\mathbb{E}[f(X_1, \dots, X_{2n})]=\frac{e(G)}{k^2}$.
			Observe that, for each $\ell\in [n]$, we have 
			\[f(X_1, \dots, X_{\ell-1}, 1, X_{\ell+1}, \dots, X_{2n})-f(X_1, \dots, X_{\ell-1}, 0, X_{\ell+1}, \dots, X_{2n})\leq d_G(a_\ell)\]
			and
			\[f(X_1, \dots, X_{2n-\ell}, 1, X_{2n-\ell+2}, \dots, X_{2n})-f(X_1, \dots, X_{2n-\ell}, 0, X_{2n-\ell+2}, \dots, X_{2n})\leq d_G(b_\ell).\]
			Moreover, $\sum_{\ell\in [n]}((d_G(a_\ell))^2+(d_G(b_\ell))^2)\leq 2\frac{e(G)}{\Delta(G)}(\Delta(G))^2\leq 2e(G)\varepsilon n$%
			    \COMMENT{Sum maximised when as many vertices as possible have maximum degree.}.
			Thus, \cref{lm:McDiarmid} implies that
			\begin{align*}
				\mathbb{P}\left[e_G(A_i, B_j)<\frac{e(G)}{2k^2}\right]
				&\leq \mathbb{P}\left[|f(X_1, \dots, X_{2n})-\mathbb{E}[f(X_1, \dots, X_{2n})]|>\frac{e(G)}{2k^2}\right]\\
				&\leq 2\exp\left(-\frac{e(G)}{4k^4\varepsilon n}\right)\leq 2\exp\left(-\frac{1}{8\varepsilon k^4}\right).
			\end{align*}
			Therefore, a union bound implies that, with probability at least $1-2k^2\exp\left(-\frac{1}{\varepsilon k^5}\right)\geq \frac{4}{5}$, we have $e_G(A_i, B_j)\geq \frac{e(G)}{2k^2}$ for all $i,j\in [k]$.
\end{proof}

\onlyinsubfile{\bibliographystyle{abbrv}
\bibliography{Bibliography/Bibliography}}

\subsection{Matchings}

	\onlyinsubfile{
		\setcounter{section}{5}
		\setcounter{subsection}{3}
\section{Matchings}}

In this section, we collect tools for constructing and working with matchings. First, we need the following two \lcnamecrefs{prop:Koniglarge}, which follow from K\"{o}nig's theorem \cite{konig1916graphok} (see also \cite{konig1916graphen} for a German translation).

\begin{prop}\label{prop:Koniglarge}
	Let $G$ be a bipartite graph with maximum degree at most $\Delta$. Then, $G$ contains a matching of size $\frac{e(G)}{\Delta}$.
\end{prop}

\begin{prop}[see e.g.\ {\cite[Exercise 7.1.33]{west2001introduction}}]\label{prop:Konigsamesize}
	Let $G$ be a bipartite graph with maximum degree at most $\Delta$. Then, $G$ can be decomposed into edge-disjoint matchings $M_1, \dots, M_\Delta$ such that, for any $i,j\in[\Delta]$, $||M_i|-|M_j||\leq 1$.
\end{prop}

We will also need the following corollary of Hall's theorem \cite{hall1935representatives}.

\begin{prop}\label{prop:Hall}
	Let~$G$ be a bipartite graph on vertex classes~$A$ and~$B$ with~$|A|\leq |B|$. Suppose that, for each~$a\in A$, $d_G(a)\geq\frac{|B|}{2}$ and, for each $b\in B$, $d_G(b)\geq |A|-\frac{|B|}{2}$.
	Then,~$G$ contains a matching covering~$A$. 
\end{prop}

\COMMENT{\begin{proof}
		Let $A'\subseteq A$. If $|A'|\leq\frac{|B|}{2}$, then $|N_D(A')|\geq |A'|$. Otherwise, $N_G(A')=B$. Indeed, assume for a contradiction that $|A'|>\frac{|B|}{2}$ and $b\in B\setminus N_G(A')$. Then, $|A|-\frac{|B|}{2}\leq d_G(b)\leq |A|-|A'|< |A|-\frac{|B|}{2}$, a contradiction. 
\end{proof}}

\onlyinsubfile{\bibliographystyle{abbrv}
\bibliography{Bibliography/Bibliography}}

\subsection{Matching contractions}\label{sec:contractingM}

	\onlyinsubfile{
		\setcounter{section}{5}
		\setcounter{subsection}{4}
\subsection{Matching contractions}}

Note that the concepts introduced in this subsection will not be used formally until \cref{sec:cyclerobustdecomp}. However, we introduce them here as they will help us to explain the approximate decomposition strategy presented in \cref{sec:approxdecomp}.

As discussed in the proof overview, most of our Hamilton cycles will be formed by first constructing a perfect matching, which is then extended to a Hamilton cycle by constructing a Hamilton cycle in an auxiliary digraph which is, roughly speaking, obtained by contracting the edges of $M$. In this section, we give a formal definition of this auxiliary digraph and state its main properties.

\begin{definition}[\Gls*{contraction} and \gls*{expansion}]\label{def:contractexpand}
	Let $A$ and $B$ be disjoint vertex sets of equal size.
	Let $M$ be an auxiliary directed perfect matching from $B$ to $A$.
	\begin{enumerate}[label=\rm(\roman*)]
		\item Let $G$ be a bipartite graph on vertex classes $A$ and $B$. The \emph{$M$-contraction of $G$} is the digraph $G_M$ on vertex set $A$ defined as follows. Let $a,a'\in A$ be distinct and denote by $b$ the (unique) neighbour of $a$ in $M$. Then, $a'a\in E(G_M)$ if and only if $a'b\in E(G)$.\label{def:contract}
		\item Let $D$ be a digraph on vertex set $A$. The \emph{$M$-expansion of $D$} is the bipartite graph $D_M$ on vertex classes $A$ and $B$ defined as follows. Let $b\in B$ and let $a$ be the (unique) neighbour of $b$ in $M$. Then, for any $a'\in A$, $a'b\in E(D_M)$ if and only if $a'a\in E(D)$.\label{def:expand}%
			\COMMENT{Note that we automatically have $a\neq a'$ since $D$ is loopless.}
	\end{enumerate}
\end{definition}

(Recall that \cref{def:contractexpand} and all other main definitions are indexed in the glossary at the end of this paper.)

The condition that $a$ and $a'$ have to be distinct in \cref{def:contractexpand}\cref{def:contract} ensures that the resulting digraph $G_M$ does not contain any loop.
However, this implies that the edges lying along $M$ are lost in the process of contraction and expansion. 
(Of course, one could slightly change \cref{def:contractexpand}\cref{def:contract} to allow $M$-contractions to have loops. In this way, the $M$-expansion would the exact reverse operation of the $M$-contraction. But working with loops is impractical for our purposes.) 

\begin{fact}\label{fact:contractinguncontracting}
	Let $A$ and $B$ be disjoint vertex sets of equal size. Let $M$ be a directed perfect matching from $B$ to $A$.
	Let $G$ be a bipartite graph on vertex classes $A$ and $B$.
	Denote by $D$ the $M$-contraction of $G$ and by $G'$ the $M$-expansion of $D$.
	Then, $e(D)=e(G\setminus M[B,A])$ and $G'=G\setminus M[B,A]$.
\end{fact}

Let $A,B,M, G$, and $D$ be as in \cref{fact:contractinguncontracting}.
By \cref{def:contractexpand}\cref{def:contract}, the outneighbourhood of a vertex $a\in A$ in $D$ corresponds to the neighbourhood of $a$ in $G$, while the inneighbourhood of $a$ in $D$ corresponds to the neighbourhood of $N_M(a)$ in $G$.
We also observe for later use that there is a one-to-one correspondence between the connected components of $G\cup M[B, A]$ and $D$.

\begin{fact}\label{fact:Ncontract}
	Let $A,B,M, G$, and $D$ be as in \cref{fact:contractinguncontracting}.
	Then, the following hold.
	\begin{enumerate}
		\item Each $a\in A$ satisfies $N_D^+(a)=N_M(N_G(a))\setminus \{a\}$ and $N_D^-(a)=N_G(N_M(a))\setminus \{a\}$.\label{fact:Ncontract-N}
		\item Any $a,a'\in A$ belong to a common connected component of $D$ if and only if they belong to a common connected component of $G\cup M[B,A]$.\label{fact:Ncontract-components}
	\end{enumerate}
\end{fact}

Let $A,B$, and $M$ be as in \cref{fact:contractinguncontracting}.
Let $D$ be a digraph on $A$ and denote by $G$ the $M$-expansion of $D$.
By \cref{def:contractexpand}\cref{def:expand}, the neighbourhood of a vertex $a\in A$ in $G$ corresponds to the outneighbourhood of $a$ in $D$, while the neighbourhood of $N_M(a)$ in $G$ corresponds to the inneighbourhood of $a$ in $D$.

\begin{fact}\label{fact:Nuncontract}
	Let $A,B$, and $M$ be as in \cref{fact:contractinguncontracting}.
	Let $D$ be a digraph on $A$ and denote by $G$ the $M$-expansion of $D$. Then, the following hold.
	\begin{enumerate}
		\item Each $a\in A$ satisfies $N_G(a)=N_M(N_D^+(a))$.\label{fact:Nuncontract-A}
		\item Each $b\in B$ satisfies $N_G(b)=N_D^-(N_M(b))$.\label{fact:Nuncontract-B}
	\end{enumerate}
\end{fact}

We now state our key property of matching contractions and matching expansions: finding a Hamilton cycle in a bipartite digraph $D$ on vertex classes $A$ and $B$ is equivalent to finding a perfect matching $M$ from $B$ to $A$ in $D$ and then finding a Hamilton cycle in the $M$-contraction of $D[A, B]$.

\begin{fact}\label{fact:contractingHamcycle}
	Let $A$ and $B$ be disjoint vertex sets of equal size. Let $M$ be a directed perfect matching from $B$ to $A$.
	Let $H$ be a directed Hamilton cycle on $A$. Let $G$ be obtained by orienting from $A$ to $B$ all the edges in the $M$-expansion of $H$.
	Then, $G$ is a directed perfect matching from $A$ to $B$ and $G\cup M$ is a directed Hamilton cycle on $A\cup B$.
\end{fact}

To construct Hamilton cycles in matching contractions, we will use the following \lcnamecref{lm:contraction}, which states that almost regularity, superregularity, and robust outexpansion are preserved in contracted digraphs. Its proof follows easily from definitions and is therefore omitted. (Similar observations were also made and proved in \cite{kuhn2015robust}.)

\begin{prop}\label{lm:contraction}
	Let $0<\frac{1}{n}\ll \varepsilon\ll \nu\ll  \tau\ll \delta, d \leq 1$. Let $G$ be a bipartite graph on vertex classes $A$ and $B$ of size $n$ and $M$ be a directed perfect matching from $B$ to $A$.
	Then, the $M$-contraction $D$ of $G$ satisfies the following properties.
	\begin{enumerate}
		\item If $G$ is $(\delta, \varepsilon)$-almost regular, then $D$ is $(2\delta, 2\varepsilon)$-almost regular.\label{lm:contraction-regular}
		\item Let $A_1,A_2\subseteq A$ be disjoint. If $G[A_1, N_M(A_2)]$ is $[\varepsilon,d]$-superregular, then $D[A_1, A_2]$ is $[\varepsilon, d]$-superregular.\label{lm:contraction-supreg}
		\item If $G$ is a bipartite robust $(\nu, \tau)$-expander with bipartition $(A,B)$, then $D$ is a robust $(\frac{\nu}{2}, \tau)$-outexpander%
			\COMMENT{Parameters are left unchanged in \cite[Lemma 6.10]{kuhn2015robust} but there is mistake? We do not have $RN_{\nu, G^*}^+(S)\supseteq RN_{\nu, G}(S_A)$ on line 3 of the proof of \cite[Lemma 6.10]{kuhn2015robust} because loops are not allowed? (Suppose that $x\in RN_{\nu, G}(S_A)\cap S$ satisfies $|N_G(x)\cap S_A|=\nu n$. Then, $N_{G^*}^-(x)\cap S=(N_G(x)\cap S_A)\setminus \{x'\}$. I.e.\ $|N_{G^*}^-(x)\cap S|=\nu n-1$ and so $x\notin RN_{\nu, G^*}^-(S)$.)}.\label{lm:contraction-rob}
	\end{enumerate} 
\end{prop}

Note that \cref{lm:biprobHamcycle} follows from \cref{lm:rob} and \cref{lm:contraction}\cref{lm:contraction-rob}.

\begin{proof}[Proof of \cref{lm:biprobHamcycle}]
	First, we find a perfect matching from $A$ to $B$ as follows.
	Let $M$ be an arbitrary perfect matching from $B$ to $A$. Let $D_M$ be the $M$-contraction of $D[A,B]$. By \cref{lm:contraction}\cref{lm:contraction-rob}, $D_M$ is a robust $(\frac{\nu}{2},\tau)$-outexpander. Moreover, \cref{fact:Ncontract}\cref{fact:Ncontract-N} implies that $\delta^0(D_M)\geq \frac{\delta n}{2}$. Thus, \cref{lm:rob} implies that $D_M$ contains a Hamilton cycle $H$. Let $M'$ be obtained by orienting from $A$ to $B$ all the edges in the $M$-expansion of $H$. Then, \namecrefs{fact:Ncontract} \ref{fact:Ncontract}\cref{fact:Ncontract-N}, \ref{fact:Nuncontract}, and \ref{fact:contractingHamcycle} imply that $M'$ is a perfect matching of $D$ from $A$ to~$B$.
	
	We close $M'$ into a Hamilton cycle as follows. Let $D_{M'}$ be the $M'$-contraction of $D[B,A]$. By the same arguments as above, $D_{M'}$ contains a Hamilton cycle $H'$. Let $M''$ be obtained by orienting from $B$ to $A$ all the edges in the $M'$-expansion of $H'$. By the same arguments as above, $M''$ is a perfect matching of $D$ from $B$ to $A$. Moreover, \cref{fact:contractingHamcycle} implies that $M'\cup M''$ is a Hamilton cycle.
\end{proof}

Finally, observe that contracting a linear forest $F$ gives a linear forest with endpoints corresponding to those of $F$. This follows easily from \cref{fact:Ncontract} and so we omit the details.

\begin{prop}\label{fact:contractlinforest}
	Let $F$ be a balanced bipartite directed linear forest on vertex classes $A$ and $B$.
	Suppose that $F[B,A]$ is a perfect matching and let $M \coloneqq E_F(B,A)$.
	Denote by $D$ the $M$-contraction of $F[A,B]$. Then, $D$ is a linear forest satisfying
	\[V^+(D)=N_M(V^+(F)),\quad  V^-(D)=V^-(F),\quad \text{and}\quad V^0(D)=(V^0(F)\cap A)\setminus N_M(V^+(F)).\]%
		\COMMENT{(Recall that an isolated vertex is considered as a starting and ending point of a path.)}
\end{prop}

\COMMENT{\begin{proof}
	First, we show that $D$ is a linear forest. By \cref{fact:Ncontract}\cref{fact:Ncontract-N}, $\Delta^0(D)\leq 1$.
	Suppose for a contradiction that $C$ is a cycle in $D$. Since $\Delta^0(D)\leq 1$, $C$ is a connected component of $D$ and, since $M$ is a perfect matching from $B$ to $A$, we have $V^-(F)\subseteq A$.
	Thus, \cref{fact:Ncontract}\cref{fact:Ncontract-components} implies that there exists $a\in V(C)\cap V^-(F)$.
	Let $a'$ be the outneighbour of $a$ in $C$. By \cref{fact:Ncontract}\cref{fact:Ncontract-N}, $a'\in  N_M(N_{F[A,B]}(a))=N_M(N_F^+(a))=\emptyset$, a contradiction. Thus, $D$ is a linear forest, as desired.\\
	Let $a\in A$. Note that since $M$ is a perfect matching, $V^+(F)\subseteq B$ and $V^-(F)\subseteq A$. By \cref{fact:Ncontract}\cref{fact:Ncontract-N} and since $F$ is a linear forest containing $M$, we have
	\[N_D^-(a)=N_{F[A,B]}(N_M(a))\setminus \{a\}=N_F^-(N_M(a)).\]
	Thus, $a\in V^+(D)$ if and only if $N_F^-(N_M(a))=\emptyset$. That is, $V^+(D)=N_M(V^+(F)\cap B)=N_M(V^+(F))$, as desired.
	Similarly, \cref{fact:Ncontract}\cref{fact:Ncontract-N} implies that
	\[N_D^+(a)=N_M(N_{F[A,B]}(a))\setminus \{a\}=N_M(N_F^+(a))\]
	and so $a\in V^-(D)$ if and only if $N_F^+(a)=\emptyset$. That is, $V^-(D)=V^-(F)\cap A=V^-(F)$, as desired.
	Finally, \[V^0(D)=A\setminus (V^+(D)\cup V^-(D))=A\setminus (N_M(V^+(F))\cup V^-(F))=(V^0(F)\cap A)\setminus N_M(V^+(F)),\]
	so we are done.
\end{proof}}

\onlyinsubfile{\bibliographystyle{abbrv}
	\bibliography{Bibliography/Bibliography}}

\section{Main tools}\label{sec:maintools}

We now introduce our main tools for constructing approximate decompositions and decomposing leftovers. These will be used in the proofs of \cref{thm:biprobexp,thm:blowupC4}. 

\subsection{Approximate decomposition tools}\label{sec:approxdecomp}

	\onlyinsubfile{
		\setcounter{section}{6}
\subsection{Approximate decomposition tools}}

As mentioned in \cref{sec:sketch}, we adapt arguments of \cite{girao2020path} using the concept of matching contraction. We now discuss this in more detail. In \cite{girao2020path}, we showed that any dense almost regular robust outexpander $D$ can be approximately decomposed into Hamilton cycles. (Moreover, one can ensure that each Hamilton cycle contains a small set of prescribed edges.) The key idea behind the proof is to reserve a sparse random edge-slice $\Gamma\subseteq D$ and then construct, one by one, edge-disjoint Hamilton cycles which use very few edges of $\Gamma$. This ensures that robust outexpansion is preserved throughout the approximate decomposition.

\begin{thm}[{\cite[Theorem 14.2 and Lemma 14.3]{girao2020path}}]\label{thm:approxHamdecomp}
	Let $0<\frac{1}{n}\ll \tau\ll \delta \leq 1$ and  $0<\frac{1}{n}\ll\varepsilon\ll\eta,\nu\leq 1$. Let $\ell \leq (\delta-\eta)n$.	
	Let $D$ be a $(\delta,\varepsilon)$-almost regular robust~$(\nu,\tau)$-outexpander on~$n$ vertices. Suppose that $F_1, \dots, F_\ell$ are linear forests
	on~$V(D)$ satisfying the following properties.
	\begin{enumerate}
		\item For each $i\in [\ell]$, $e(F_i)\leq \varepsilon n$.
		\item For each~$v\in V(D)$, there exist at most~$\varepsilon n$ indices~$i\in[\ell]$ such that~$v\in V(F_i)$.
	\end{enumerate}
	Define a multidigraph~$\cF$ by $\cF\coloneqq \bigcup_{i\in [\ell]}F_i$.
	Then, the multidigraph $D\cup \cF$ contains edge-disjoint Hamilton cycles $C_1,\dots, C_\ell$ such that $F_i\subseteq C_i$ for each $i\in [\ell]$.
\end{thm}

Let $D$ be a bipartite digraph on vertex classes $A$ and $B$ and suppose that $D[A,B]$ is an almost regular bipartite robust expander. Let $M_1, \dots, M_\ell$ be edge-disjoint perfect matchings whose edges are all oriented from $B$ to $A$. Then, we can extend $M_1, \dots, M_\ell$ into edge-disjoint Hamilton cycles as follows. For each $i\in [\ell]$, denote by $D_i$ the $M_i$-contraction of $D$ and note that, by \cref{lm:contraction}, $D_i$ is an almost regular robust outexpander. 
By \cref{fact:contractingHamcycle}, it is enough to find, for each $i\in [\ell]$, a Hamilton cycle of $D_i$.
Since the $D_i$'s are distinct, we cannot apply \cref{thm:approxHamdecomp} directly. However, we can adapt the strategy discussed above as follows. We initially reserve a randomly chosen edge-slice $\Gamma\subseteq D[A,B]$ and denote, for each $i\in [\ell]$, by $\Gamma_i$ the corresponding random edge-slice of $D_i$. At each stage $i\in [\ell]$, we use the arguments of \cref{thm:approxHamdecomp} to construct a Hamilton cycle of $D_i$ which uses very few edges of $\Gamma_i$. This ensures that, overall, very few edges of $\Gamma$ are used and so $D[A,B]$ remains a bipartite robust expander throughout the approximate decomposition. By \cref{lm:contraction}, this implies that, at each stage $i\in [\ell]$, $D_i$ is still a robust outexpander and so the approximate decomposition can be completed.
\APPENDIX{(See \cref{app:approximatedecomp} for details.)}%
\NOAPPENDIX{(See the appendix of the arXiv version of this paper for details and a proof of \cite[Lemma 9.3]{girao2020path}, which was omitted in \cite{girao2020path} as it is straightforward.)}

Recall from \cref{sec:notation-regularity} that a balanced bipartite digraph $D$ on vertex classes of size $n$ is $(\delta, \varepsilon)$-regular if all its vertices have in- and outdegree both roughly equal to $\delta|V(D)|=2\delta n$. This justifies the factor of $2$ in the upper bound on $\ell$ in \cref{thm:biphalfapproxHamdecomp}. Moreover, recall from \cref{sec:notation-graphs} that the parallel edges of a multi(di)graph are considered to be distinct. Thus, we do not require the linear forests $F_1, \dots, F_\ell$ in \cref{thm:biphalfapproxHamdecomp} to be edge-disjoint and the theorem states that each edge of $D$ is covered by at most one of the resulting Hamilton cycles $C_1,\dots, C_\ell$ (while each linear forest $F_i$ is fully incorporated into its corresponding cycle $C_i$).

\begin{thm}[Extending an approximate perfect matching decomposition into an approximate Hamilton decomposition]\label{thm:biphalfapproxHamdecomp}
	Let $0<\frac{1}{n}\ll \tau\ll \delta \leq 1$ and  $0<\frac{1}{n}\ll\varepsilon\ll\eta,\nu\leq 1$. Let $\ell \leq 2(\delta-\eta)n$.
	Let $D$ be a balanced bipartite digraph on vertex classes $A$ and $B$ of size $n$.
	Suppose that~$D[A,B]$ is a $(\delta,\varepsilon)$-almost regular bipartite robust~$(\nu,\tau)$-expander with bipartition $(A,B)$. 
	Suppose that $F_1, \dots, F_\ell$ are bipartite directed linear forests on vertex classes $A$ and $B$ satisfying the following properties.
	\begin{enumerate}
		\item For each~$i\in[\ell]$, $e_{F_i}(B,A)=n$.\label{thm:biphalfapproxHamdecomp-BA}
		\item For each~$i\in[\ell]$, $e_{F_i}(A,B)\leq \varepsilon n$.\label{thm:biphalfapproxHamdecomp-AB}
		\item For each~$v\in V(D)$, there exist at most~$\varepsilon n$ indices~$i\in[\ell]$ such that~$d_{F_i[A,B]}(v)=1$.\label{thm:biphalfapproxHamdecomp-size}
	\end{enumerate}
	Define a multidigraph~$\cF$ by $\cF\coloneqq \bigcup_{i\in [\ell]}F_i$.
	Then, the multidigraph $D\cup \cF$ contains edge-disjoint Hamilton cycles $C_1,\dots, C_\ell$ such that $F_i\subseteq C_i$ for each $i\in [\ell]$.
	Moreover, $D[A,B]\setminus \bigcup_{i\in [\ell]} C_i$ is still a bipartite robust~$(\frac{\nu}{2},\tau)$-expander with bipartition $(A,B)$.
\end{thm}

If $D$ is a bipartite robust outexpander (i.e.\ if both $D[A,B]$ and $D[B,A]$ are bipartite robust expanders (recall \cref{fact:biprobexp})), then we can apply \cref{thm:biphalfapproxHamdecomp} twice in a row to construct an approximate Hamilton decomposition of $D$: first, we apply \cref{thm:biphalfapproxHamdecomp} with arbitrary perfect matchings from $B$ to $A$ to approximately decompose the edges of $D$ from $A$ to $B$ into edge-disjoint perfect matchings, and then we apply \cref{thm:biphalfapproxHamdecomp} a second time to extend these perfect matchings into edge-disjoint Hamilton cycles of $D$.

\begin{cor}[Approximate Hamilton decomposition]\label{cor:bipapproxHamdecomp}
	Let $0<\frac{1}{n}\ll \tau\ll \delta \leq 1$ and  $0<\frac{1}{n}\ll\varepsilon\ll\eta,\nu\leq 1$. Let $\ell \leq 2(\delta-\eta)n$.	
	Let~$D$ be a balanced bipartite digraph on vertex classes $A$ and $B$ of size~$n$. Suppose that $D$ is a $(\delta,\varepsilon)$-almost regular bipartite robust~$(\nu,\tau)$-outexpander with bipartition $(A,B)$.
	Suppose that $F_1, \dots, F_\ell$ are bipartite directed linear forests on vertex classes $A$ and $B$ satisfying the following properties.
	\begin{enumerate}
		\item For each $i\in[\ell]$, $e(F_i)\leq \varepsilon n$.\label{cor:bipapproxHamdecomp-size}
		\item For each~$v\in V(D)$, there exist at most~$\varepsilon n$ indices~$i\in[\ell]$ such that~$v\in V(F_i)$.\label{cor:bipapproxHamdecomp-deg}
	\end{enumerate}
	Define a multidigraph $\cF$ by $\cF\coloneqq \bigcup_{i\in [\ell]}F_i$.
	Then, the multidigraph $D\cup\cF$ contains edge-disjoint Hamilton cycles $C_1,\dots, C_\ell$ such that $F_i\subseteq C_i$ for each $i\in [\ell]$.
	Moreover, $D\setminus \bigcup_{i\in [\ell]} C_i$ is still a bipartite robust~$(\frac{\nu}{2},\tau)$-outexpander with bipartition $(A,B)$.
\end{cor}

\begin{proof}
	First, we extend $F_1, \dots, F_\ell$ to auxiliary linear forests which satisfy \cref{thm:biphalfapproxHamdecomp}\cref{thm:biphalfapproxHamdecomp-AB,thm:biphalfapproxHamdecomp-BA,thm:biphalfapproxHamdecomp-size}.
	\begin{claim}\label{claim:approxdecomp}
		For each $i\in [\ell]$, there exists a bipartite linear forest $F_i'$ on vertex classes $A$ and $B$ such that $F_i'[A,B]=F_i[A,B]$ and $F_i'[B,A]$ is a perfect matching containing $F_i[B,A]$.
	\end{claim}
	
	\begin{proofclaim}
		Let $i\in [\ell]$. 
		\begin{itemize}
			\item Denote by $a_1, \dots, a_q$ the vertices $v\in A$ satisfying $d_{F_i}^-(v)=0$ and $d_{F_i}^+(v)=1$.
			\item Denote by $a_{q+1}, \dots, a_{q+r}$ the vertices $v\in A$ satisfying $d_{F_i}^-(v)=0$ and $d_{F_i}^+(v)=0$.
			\item Denote by $b_1, \dots, b_s$ the vertices $v\in B$ satisfying $d_{F_i}^-(v)=1$ and $d_{F_i}^+(v)=0$.
			\item Denote by $b_{s+1}, \dots, b_{s+t}$ the vertices $v\in B$ satisfying $d_{F_i}^-(v)=0$ and $d_{F_i}^+(v)=0$.
		\end{itemize} 
		Observe that there exist exactly $r$ vertices in $A$ which have degree $0$ in $F_i$. Therefore,
		\begin{equation*}
			r\geq |A|-e(F_i)\stackrel{\text{\cref{cor:bipapproxHamdecomp-size}}}{\geq} (1-\varepsilon)n>0.
		\end{equation*}
		Note that $a_1, \dots, a_q$ is an enumeration of $V^+(F_i)\cap A$ and $b_1, \dots, b_s$ is an enumeration of $V^-(F_i)\cap B$.
		Since $F_i$ is a linear forest, we may therefore assume without loss of generality that, if $F_i$ contains an $(a_j, b_k)$-path for some $j\in [q]$ and $k\in [s]$, then $j=k$.
		Note that $e(F_i[B,A])=|A|-(q+r)=|B|-(s+t)$ and so $q+r=s+t$.		
		Let $F_i'\coloneqq F_i\cup \{b_ia_{i+1}\mid i\in [q+r]\}$ (where $a_{q+r+1}\coloneqq a_1$). Then, $F_i$ is a bipartite digraph on vertex classes $A$ and $B$ such that $F_i'[A,B]=F_i[A,B]$ and $F_i'[B,A]$ is a perfect matching which contains $F_i[B,A]$. It is easy to check that $F_i$ is a linear forest.%
			\COMMENT{Suppose for a contradiction that $C$ is a cycle in $F_i'$. For each $j\in [r]$, $d_{F_i'}(a_{q+j})=1$. Similarly, $d_{F_i'}(b_{s+j})=1$ for each $j\in [t]$. Moreover, $d_{F_i'}(a)=1$ for all $a\in A\cap V^-(F_i)$ and $d_{F_i'}(b)=1$ for all $b\in B\cap V^+(F_i')$.
			Therefore, $V(C)\subseteq V^0(F_i)\cup \{a_j\mid j\in [q]\}\cup \{b_j\mid j\in [s]\}$ and so, by assumption, there exist $j<k\in [\min\{q,s\}]$ such that $C=a_jP_jb_ja_{j+1}P_{j+1}b_{j+1}a_{j+2}\dots b_k$, where $P_j, \dots, P_k$ are components of $F_i$. But, $b_ka_j\in E(F_i')$ implies that $j=1$ and $k=q+r\leq q$. Thus, $r=0$, a contradiction. Therefore, $F_i'$ is a linear forest, as desired.}
	\end{proofclaim}
	
	Let $F_1', \dots, F_\ell'$ be the linear forests obtained by applying \cref{claim:approxdecomp}. Observe that \cref{thm:biphalfapproxHamdecomp}\cref{thm:biphalfapproxHamdecomp-BA,thm:biphalfapproxHamdecomp-AB,thm:biphalfapproxHamdecomp-size} are satisfied with $F_1', \dots, F_\ell'$ playing the roles of $F_1, \dots, F_\ell$.
	Define a multidigraph $\cF'$ by $\cF'\coloneqq \bigcup_{i\in [\ell]}F_i'$.
	By \cref{thm:biphalfapproxHamdecomp} (applied with $F_1', \dots, F_\ell'$ playing the roles of $F_1, \dots, F_\ell$), the multidigraph $D\cup \cF'$ contains edge-disjoint Hamilton cycles $C_1', \dots, C_\ell'$ such that $F_i'\subseteq C_i'$ for each $i\in [\ell]$.
	For each $i\in [\ell]$, let $F_i''\coloneqq C_i'[A,B]\cup F_i[B,A]\subsetneq C_i'$ and note that $F_i\subseteq F_i''\subseteq \cF\cup D[A,B]$. Moreover, \cref{thm:biphalfapproxHamdecomp}\cref{thm:biphalfapproxHamdecomp-BA,thm:biphalfapproxHamdecomp-AB,thm:biphalfapproxHamdecomp-size} are satisfied with $B,A$, and $F_1'', \dots, F_\ell''$ playing the roles of $A,B$, and $F_1, \dots, F_\ell$. Define a multidigraph $\cF''$ by $\cF''\coloneqq \bigcup_{i\in [\ell]} F_i''$. Let $D'\coloneqq D\setminus \cF''$. Note that $D'[B,A]=D[B,A]$ and, by the ``moreover part" of \cref{thm:biphalfapproxHamdecomp}, $D'[A,B]$ is a bipartite robust $(\frac{\nu}{2},\tau)$-expander with bipartition $(A,B)$.
	
	By \cref{thm:biphalfapproxHamdecomp} (applied with $D',B, A$, and $F_1'', \dots, F_\ell''$ playing the roles of $D, A, B$, and $F_1, \dots, F_\ell$), the multidigraph $D'\cup \cF''=D\cup \cF$ contains edge-disjoint Hamilton cycles $C_1, \dots, C_\ell$ such that $F_i\subseteq F_i''\subseteq C_i$ for each $i\in [\ell]$. 
	Let $D''\coloneqq D\setminus \bigcup_{i\in [\ell]}C_i$. 
	By the ``moreover part" of \cref{thm:biphalfapproxHamdecomp}, $D''[B,A]$ is a bipartite robust $(\frac{\nu}{2},\tau)$-expander with bipartition $(B,A)$.
	By construction, $D''[A,B]=D'[A,B]$ and so $D''[A,B]$ is also a bipartite robust $(\frac{\nu}{2},\tau)$-expander with bipartition $(A,B)$. Therefore,  \cref{fact:biprobexp} implies that $D''$ is a bipartite robust $(\frac{\nu}{2},\tau)$-outexpander with bipartition $(A,B)$.
\end{proof}

Note that the ``moreover part"  of \cref{cor:bipapproxHamdecomp} implies that we can prescribe some edges to most of the Hamilton cycles in the decomposition given by \cref{thm:biprobexp}.

Similarly, one can apply \cref{thm:biphalfapproxHamdecomp} with auxiliary perfect matchings from $B$ to $A$ to obtain an approximate decomposition of a bipartite robust expander into perfect matchings which extend given small matchings.

\begin{cor}[Approximate perfect matching decomposition]\label{cor:bipapproxmatchdecomp}
	Let $0<\frac{1}{n}\ll \tau\ll \delta \leq 1$ and  $0<\frac{1}{n}\ll\varepsilon\ll\eta,\nu\leq 1$. Let $\ell \leq 2(\delta-\eta)n$.	
	Let $G$ be a balanced bipartite graph on vertex classes $A$ and $B$ of size~$n$. Suppose that $G$ is a $(\delta,\varepsilon)$-almost regular bipartite robust~$(\nu,\tau)$-expander with bipartition $(A,B)$. Suppose that $F_1, \dots, F_\ell$ are bipartite matchings on vertex classes $A$ and $B$ satisfying the following properties.
	\begin{enumerate}
		\item For each~$i\in[\ell]$, $e(F_i)\leq \varepsilon n$.\label{cor:bipapproxmatchdecomp-size}
		\item For each~$v\in V(G)$, there exist at most~$\varepsilon n$ indices~$i\in[\ell]$ such that~$v\in V(F_i)$. \label{cor:bipapproxmatchdecomp-deg}
	\end{enumerate}
	Define a multigraph~$\cF$ by $\cF\coloneqq \bigcup_{i\in [\ell]}F_i$.
	Then, the multigraph $G\cup \cF$ contains edge-disjoint perfect matchings $M_1,\dots, M_\ell$ such that $F_i\subseteq M_i$ for each $i\in [\ell]$.
	Moreover, $G\setminus \bigcup_{i\in [\ell]} M_i$ is still a bipartite robust~$(\frac{\nu}{2},\tau)$-expander with bipartition $(A,B)$.
\end{cor}

\COMMENT{\begin{proof}
	Let $D$ be obtained from $G$ by orienting all edges from $A$ to $B$. Note that $D[A,B]=G$.
	Let $i\in [\ell]$. Let $a_1b_1, \dots, a_kb_k$ be an enumeration of $F_i$ where, for each $j\in [k]$, $a_j\in A$ and $b_j\in B$. Let $a_{k+1}, \dots, a_n$ and $b_{k+1}, \dots, b_n$ be enumerations of $A\setminus V(F_i)$ and $B\setminus V(F_i)$.
	Let $F_i'$ be the digraph on vertex set $A\cup B$ with $E_{F_i'}\coloneqq \{a_ib_i\mid i\in [k]\} \cup \{b_ia_{i+1} \mid i\in [n]\}$, where $a_{n+1}\coloneqq a_1$.
	Note that $F_i'[A,B]=F_i$ and, since $k\leq \varepsilon n$, $F_i'$ is a bipartite linear forest on vertex classes $A$ and $B$. Thus, \cref{thm:biphalfapproxHamdecomp}\cref{thm:biphalfapproxHamdecomp-AB,thm:biphalfapproxHamdecomp-BA,thm:biphalfapproxHamdecomp-size} are satisfied with $F_1', \dots, F_\ell'$ playing the roles of $F_1, \dots, F_\ell$.
	Define a multidigraph $\cF'$ by $\cF'\coloneqq \bigcup_{i\in [\ell]} F_i'$.
	Apply \cref{thm:biphalfapproxHamdecomp} with $F_1', \dots, F_\ell'$ playing the roles of $F_1, \dots, F_\ell$ to obtain edge-disjoint Hamilton cycles $C_1, \dots, C_\ell\subseteq D\cup \cF'$ such that, for each $i\in [\ell]$, $F_i'\subseteq C_i$.
	For each $i\in [\ell]$, let $M_i\coloneqq C_i[A,B]$. This completes the proof.
\end{proof}}

\onlyinsubfile{\bibliographystyle{abbrv}
\bibliography{Bibliography/Bibliography}}

\subsection{The robust decomposition lemma}\label{sec:robustdecomp}

	\onlyinsubfile{
		\setcounter{section}{6}
		\setcounter{subsection}{1}
\subsection{The robust decomposition lemma}}

In this section, we state (a modified version of) the robust decomposition lemma of \cite{kuhn2013hamilton}. Roughly speaking, this result guarantees the existence of a sparse absorber $D^{\rm rob}$ which can decompose any sparse leftover $H$ into Hamilton cycles. To state this lemma, we need some definitions. These are needed in order to describe the structure within which the sparse absorber $D^{\rm rob}$ can be found. Roughly speaking, the structure consists of a ``quasirandom" blow-up of a graph consisting of a cycle and a suitable set of chords on this cycle.

\subsubsection{Equivalent linear forests}\label{sec:equivalentP}
We start with a simple concept which will enable us to simplify some arguments.

\begin{definition}[\Gls*{equivalent linear forests}]\label{def:equivalentP}
	Two linear forests $F$ and $F'$ are \emph{equivalent} if $V(F) = V(F')$ and there exist enumerations $P_1, \dots, P_\ell$ and $P'_1, \dots, P'_\ell$ of the components of $F$ and $F'$ such for each $i \in [\ell]$, $P_i$ and $P_i'$ have the same starting and ending points.
\end{definition}

\begin{fact}\label{fact:equivalentP}
	Let $V$ be a vertex set and $D$ be a digraph with $V(D)\subseteq V$.
	Let $F$ and $F'$ be two equivalent linear forests. 
	Then, $D\cup F$ is a Hamilton cycle on $V$ if and only if $D\cup F'$ is a Hamilton cycle on $V$. 
\end{fact}

Roughly speaking, \cref{fact:equivalentP} states that, if $F$ is a linear forest that we want to extend into a Hamilton cycle, then the internal structure of $F$ is irrelevant.
This simple observation will enable us to simplify the statement and application of the robust decomposition lemma of \cite{kuhn2013hamilton}.

\subsubsection{Refinements}
Let $D$ be a digraph and $\cP$ be a partition of $V(D)$ into an exceptional set $V_0$ and $k$ clusters $V_1,\dots,V_k$ of size $m$. Let $\cP'$ be a partition of $V(D)$.
We say that $\cP'$ is an \emph{$\ell$-refinement of $\cP$} if $\cP'$ is obtained by splitting each cluster in $\cP$
into $\ell$ subclusters of size $\frac{m}{\ell}$. (Thus, $\cP'$ consists of the exceptional set $V_0$ and $\ell k$ clusters.)

\begin{definition}[\Gls*{uniform refinement}]\label{def:uniformref}
    Let $D$ be a digraph and $\cP$ be a partition of $V(D)$ into an exceptional set $V_0$ and $k$ clusters $V_1,\dots,V_k$ of size $m$. An $\ell$-refinement $\cP'$ of $\cP$ is \emph{$\eps$-uniform (with respect to $D$)} if the following condition holds, where for each $i\in [k]$, $V_{i,1}\cup \dots\cup V_{i,\ell}$ denotes the partition of $V_i$ induced by $\cP'$.
    \begin{enumerate}[label=(URef),longlabel]
    	\item Let $v\in V(D)$, $i\in [k]$, $j\in [\ell]$, and $\diamond\in \{+,-\}$. If $|N_D^\diamond(v)\cap V_i|\ge \varepsilon m$,
    	then $|N_D^\diamond(v)\cap V_{i,j}|=(1\pm \eps)\frac{|N_D^\diamond(v)\cap V_i|}{\ell}$.\label{def:URef}
    \end{enumerate}
\end{definition}

Given a partition $\cP$ and a random $\ell$-refinement $\cP'$ of $\cP$, one can use \cref{lm:Chernoff} to show that $\cP'$ is $\varepsilon$-uniform with high probability.

\begin{lm}[{\cite[Lemma 4.7]{kuhn2013hamilton}}]\label{lm:URefexistence}
	Let $0<\frac{1}{m} \ll \frac{1}{k},\varepsilon \ll \frac{1}{\ell} \le 1$ and suppose that $\frac{m}{\ell}\in\mathbb{N}$.
	Let $D$ be a digraph on $n\le 2km$ vertices and let $\cP$ be a partition of $V(D)$ into
	an exceptional set $V_0$ and $k$ clusters of size $m$. 
	Then, there exists an $\varepsilon$-uniform $\ell$-refinement of $\cP$.
\end{lm}

Using the definition of $\varepsilon$-regularity, one can easily verify that (super)regularity is preserved under taking uniform refinements.

\begin{lm}[{\cite[Lemma 4.7]{kuhn2013hamilton}}]\label{lm:URefreg}
	Let $0<\frac{1}{m} \ll \frac{1}{k},\eps \ll d,\frac{1}{\ell} \le 1$ and $\varepsilon\ll \varepsilon'\leq 1$. Suppose that $\frac{m}{\ell}\in\mathbb{N}$.
	Let $D$ be a digraph on $n\le 2km$ vertices and let $\cP$ be a partition of $V(D)$ into
	an exceptional set $V_0$ and $k$ clusters of size $m$. Let $\cP'$ is an $\varepsilon$-uniform $\ell$-refinement of $\cP$ and let $V,W\in\cP$ and $V',W'\in \cP'$ be distinct clusters satisfying $V'\subseteq V$ and
	$W'\subseteq W$.
	\begin{enumerate}
		\item If $D[V,W]$ is $(\varepsilon, \geq d)$-regular, then $D[V',W']$ is $(\varepsilon',\ge d-\varepsilon)$-regular.\label{lm:URef-reg}%
			\COMMENT{$d-\varepsilon$ here because our definition is slightly different than that of \cite[page 12]{kuhn2013hamilton}.}
		\item If $D[V,W]$ is $[\varepsilon,\geq d]$-superregular, then $D[V',W']$ is $[\varepsilon',\geq d]$-superregular.\label{lm:URef-supreg}
	\end{enumerate}
\end{lm}

Using \cref{lm:Chernoff}, one can easily verify that the uniformity of a refinement is preserved with high probability when considering edge-slices.

\begin{lm}\label{lm:URefrandom}
	Let $0<\frac{1}{m}\ll \frac{1}{k}, \varepsilon\ll \frac{1}{\ell},p\leq 1$. Let $D$ be a digraph on $n\leq 2km$ vertices and let $\cP$ be a partition of $V(D)$ into an exceptional set $V_0$ and $k$ clusters of size $m$. Let $\cP'$ be an $\varepsilon$-uniform $\ell$-refinement of $\cP$ with respect to $D$. Let $D'$ be obtained from $D$ by selecting each edge independently with probability $p$. Then, with high probability, $\cP'$ is $2\varepsilon$-uniform with respect to both $D'$ and $D\setminus D'$.
\end{lm}

Finally, observe that refinements are always uniform in digraphs of very high minimum degree.

\begin{lm}\label{lm:URefdense}
		Let $0<\frac{1}{m}\ll \varepsilon\ll \frac{1}{k}, \frac{1}{\ell}\leq 1$ and suppose that $\frac{m}{\ell}\in \mathbb{N}$. Let $D$ be a digraph on $n\leq 2km$ vertices and suppose that $\delta^0(D)\geq (1-\varepsilon)n$. Let $\cP$ be a partition of $V(D)$ into an exceptional set $V_0$ and $k$ clusters of size $m$. Then, any $\ell$-refinement of $\cP$ is $\sqrt{\varepsilon}$-uniform with respect to $D$.
\end{lm}

\begin{proof}
		Let $\cP'$ be an $\ell$-refinement of $\cP$. Let $v\in V(D)$ and fix clusters $V\in \cP$ and $W\in \cP'$ satisfying $W\subseteq V$. By assumption, both
		\begin{align*}
			|N^\pm(v)\cap W|\leq \frac{m}{\ell} \leq \frac{|N^\pm(v)\cap V|+\varepsilon n}{\ell}\leq (1+\sqrt{\varepsilon})\frac{|N^\pm(v)\cap V|}{\ell}
		\end{align*}
		and
		\begin{align*}
			|N^\pm(v)\cap W|\geq \frac{m}{\ell} -\varepsilon n\geq \frac{|N^\pm(v)\cap V|}{\ell}-\frac{\varepsilon\ell n}{\ell}\geq (1-\sqrt{\varepsilon})\frac{|N^\pm(v)\cap V|}{\ell}.
		\end{align*}
		Thus, \cref{def:URef} holds and we are done.
\end{proof}

\subsubsection{(Bi)-universal walks}
Let $R$ be a digraph whose vertices are $V_1,\dots,V_k$ and suppose that $C=V_1\dots V_k$ is a Hamilton cycle of $R$.
Let $i,j\in [k]$. A \emph{chord sequence $CS(V_i,V_j)$ from $V_i$ to $V_j$ in $R$} is an ordered sequence of edges of the form
\[ CS(V_i,V_j) = (V_{i_1-1} V_{i_2}, V_{i_2-1} V_{i_3},\dots, V_{i_t-1} V_{i_{t+1}}),\]
where $V_{i_1}\coloneqq V_i$, $V_{i_{t+1}} \coloneqq  V_j$ and, for each $s\in [t]$, $V_{i_s-1} V_{i_{s+1}}\in E(R)$.
Thus, the simplest example of a chord sequence $CS(V_i, V_j)$ is simply $(V_{i-1}V_j)$. Chord sequences are used in the proof of the robust decomposition lemma in \cite{kuhn2013hamilton} to extend arbitrary edges into cycles which meet each cluster $V_i$ the same number of times.

\begin{definition}[\Gls*{universal walk}]\label{def:U}
	Suppose that $R$ is a digraph whose vertices are $k$ clusters $V_1,\dots,V_k$ and that $C\coloneqq V_1\dots V_k$ is a Hamilton cycle of $R$.
	A closed walk $U$ in $R$ is a \emph{universal walk for $C$ with parameter $\ell'$} if the following conditions hold.
	\begin{enumerate}[label=\rm(U\arabic*),longlabel]
    	\item For every $i\in [k]$, $U$ contains a chord sequence $CS(V_i,V_{i+1})$ from $V_i$ to $V_{i+1}$ (where $V_{k+1}\coloneqq V_1$) such that \cref{def:U-size}, \cref{def:U-degree}, and the following hold.
    	All the remaining edges of $U$ lie on $C$.\label{def:U-edges} 
    	\item For each $i\in [k]$, $CS(V_i,V_{i+1})$ consists of at most $\frac{\sqrt{\ell'}}{2}$ edges.\label{def:U-size} 
    	\item For each $i\in [k]$, both $d_U^\pm(V_i)=\ell'$.\label{def:U-degree} 
	\end{enumerate}
\end{definition}

(Recall that \cref{def:U}, as well as all the core definitions and their main properties are indexed in the glossary at the end of this paper.)

\begin{lm}[{\cite[Lemma 2.9.1]{csaba2016proof}}]\label{lm:U}
	Let $R$ be a complete digraph and $C$ be a Hamilton cycle of $R$. For any $\ell'\geq 4$, $R$ contains a universal walk for $C$ with parameter $\ell'$.	
\end{lm}

We will also need the bipartite analogue of a universal walk.

\begin{definition}[\Gls*{bi-universal walk}]\label{def:BU}
	Suppose that $R$ is a digraph whose vertices are $k$ clusters $V_1,\dots,V_k$, where $k$ is even,
	and that $C\coloneqq V_1\dots V_k$ is a Hamilton cycle of $R$.
	A closed walk $U$ in $R$ is a \emph{bi-universal walk for $C$ with parameter $\ell'$} if the following conditions hold.
	\begin{enumerate}[label=\rm(BU\arabic*),longlabel]
		\item The edge set of $U$ has a partition into $U_{\rm odd}$ and $U_{\rm even}$ and, for every $i\in [k]$, $U$ contains a chord sequence $CS(V_i,V_{i+2})$
		from $V_i$ to $V_{i+2}$ (where $V_{k+1}\coloneqq V_1$ and $V_{k+2}\coloneqq V_2$) such that \cref{def:BU-size}, \cref{def:BU-degree}, and the following hold.
		All of the edges in the multiset $\bigcup \{CS(V_i,V_{i+2})\mid i\in [k] \text{ is odd}\}$ are contained in $U_{\rm odd}$, all of the edges in the multiset $\bigcup \{CS(V_i,V_{i+2})\mid i\in [k] \text{ is even}\}$ are contained in $U_{\rm even}$, and all the remaining edges of $U$ lie on~$C$.\label{def:BU-edges} 
		\item For each $i\in [k]$, $CS(V_i,V_{i+2})$ consists of at most $\frac{\sqrt{\ell'}}{2}$ edges.\label{def:BU-size} 
		\item For each $i\in [k]$, both $d_{U_{\rm odd}}^\pm(V_i)=\frac{\ell'}{2}$ and both $d_{U_{\rm even}}^\pm(V_i)=\frac{\ell'}{2}$\label{def:BU-degree} 
	\end{enumerate}
\end{definition}

\subsubsection{(Bi)-setups}\label{sec:ST}
We introduce the key structures required to construct the absorber in the robust decomposition lemma.

\begin{definition}[\Gls*{setup}]\label{def:ST}
	We say that $(D, \cP, \cP', \cP^*, R,  C, U,U')$ is an \emph{$(\ell', \ell^*, k, m, \varepsilon, d)$-setup} if the following properties are satisfied.
	\begin{enumerate}[label=\rm(ST\arabic*),longlabel]
		\item $D$ is a digraph. $\cP$ is a partition of $V(D)$ into an exceptional set $V_0$ of size $|V_0| \le \eps |V(D)|$ and $k$ clusters $V_1, \dots, V_k$ of size $m$.
		\label{def:ST-P}
		\item $R$ is a digraph on the clusters in $\cP$, that is, $V(R)=\{V_i\mid i\in [k]\}$. For each $VW\in E(R)$, the corresponding pair $D[V,W]$ is $(\eps,\ge d)$-regular.\label{def:ST-R}
		\item $C$ is a Hamilton cycle of $R$ and, for each $VW \in E(C)$, the corresponding pair $D[V,W]$ is $[\eps,\ge d]$-superregular.\label{def:ST-C}
		\item $U$ is a universal walk for $C$ in $R$ with parameter~$\ell'$.\label{def:ST-U}
		\item $\cP'$ is an $\eps$-uniform $\ell'$-refinement of $\cP$.\label{def:ST-P'}
		\item For each $i\in [k]$, let $V_{i,1},\dots,V_{i,\ell'}$ denote the subclusters of $V_i$ contained in $\cP'$. Then, $U'$ is a closed walk on the clusters in $\cP'$ which is obtained from $U$ as follows.
		For each $i\in [k]$ and $j\in [\ell']$, when $U$ visits $V_i$ for the $j^{\rm th}$ time, $U'$ visits the subcluster $V_{i,j}$.\label{def:ST-U'}
		\item For each $VW\in E(U')$, the corresponding pair $D[V,W]$ is $[\eps,\ge d]$-superregular.\label{def:ST-U'supreg}
		\item $\cP^*$ is an $\varepsilon$-uniform $\ell^*$-refinement of $\cP$.\label{def:ST-P*}
	\end{enumerate}
\end{definition}

We will also need the bipartite analogue of a setup. 

\begin{definition}[\Gls*{bi-setup}]\label{def:BST}
	We say that $(D,\cP,\cP', \cP^*,R,C,U,U')$ is an \emph{$(\ell',\ell^*,2k,m,\eps,d)$-bi-setup} if $k\in \mathbb{N}$ and the following properties are satisfied.
	\begin{enumerate}[label=\rm(BST\arabic*),longlabel]
		\item $D$ is a balanced bipartite digraph on vertex classes $A$ and $B$. $\mathcal{P}$ is a partition of $V(D)$ into an exceptional set $V_0$ which satisfies $|V_0\cap A|=|V_0\cap B|\leq \varepsilon|A|=\varepsilon|B|$, and
		$2k$ clusters $V_1, \dots, V_{2k}$ of size $m$. Let $\cA$ be the set of clusters $V\in \cP$ such that $V\subseteq A$. Define $\cB$ analogously.
		Then, $A\setminus V_0=\bigcup \cA$ and $B\setminus V_0=\bigcup \cB$.
		(In particular, each cluster $V\in \cP$ satisfies $V\subseteq A$ or $V\subseteq B$.)\label{def:BST-P}
		\item $R$ is a balanced bipartite digraph on vertex classes $\cA$ and $\cB$. For each $VW\in E(R)$, the corresponding pair $D[V,W]$ is $(\eps,\ge d)$-regular.\label{def:BST-R}
		\item $C$ is a Hamilton cycle of $R$ and for each $VW\in E(C)$ the corresponding pair $D[V,W]$ is $[\eps,\ge d]$-superregular.\label{def:BST-C}
		\item $U$ is a bi-universal walk for $C$ in $R$ with parameter~$\ell'$.\label{def:BST-U}
		\item $\cP'$ is an $\eps$-uniform $\ell'$-refinement of $\cP$.\label{def:BST-P'}
		\item For each $i\in [2k]$, let $V_{i,1},\dots,V_{i,\ell'}$ denote the subclusters of $V_i$ contained
		in $\cP'$. Then, $U'$ is a closed walk on the clusters in $\cP'$ which is obtained from $U$ as follows.
		For each $i\in [2k]$ and $j\in [\ell']$, when $U$ visits $V_i$ for the $j^{\rm th}$ time, $U'$ visits the subcluster $V_{i,j}$.\label{def:BST-U'}
		\item For each $VW\in E(U')$, the corresponding pair $D[V,W]$ is $[\eps,\ge d]$-superregular.\label{def:BST-U'supreg}
		\item $\cP^*$ is an $\varepsilon$-uniform $\ell^*$-refinement of $\cP$.\label{def:BST-P*}
	\end{enumerate}
\end{definition}

Note that these definitions of a setup and a bi-setup are slightly different to that of~\cite{kuhn2013hamilton}. 
The original definitions required the exceptional set $V_0$ to form an independent set in $D$. Here, we only need the definition of a (bi)-setup within the setting of the robust decomposition lemma (\cref{lm:newrobustdecomp} below), where $V_0$ is empty. The independent set condition is therefore redundant and we omit it.
In \cite{kuhn2013hamilton}, the refinement $\cP^*$ is added in the statement of the robust decomposition lemma directly. For convenience, we incorporate $\cP^*$ into the definition of a (bi)-setup and, for technical reasons, we also require that $\cP^*$ is $\varepsilon$-uniform.
Finally, the definition of a bi-setup in \cite{kuhn2013hamilton} did not require $D$ to be a bipartite digraph. We add this constraint here for convenience. For clarity, we also specify that the clusters in $\cP$ must be a subset of one of the vertex classes of $D$ (this actually follows from \cref{def:BST-C} and the fact that $D$ is bipartite).

By \cref{prop:epsremovingadding}, a (bi)-setup remains a (bi)-setup (with slightly worse parameters) if only a few edges are removed and added at each vertex. A similar observation was already made (and proved) in \cite[Lemma 9.2]{kuhn2013hamilton}, so we omit the details here.

\begin{prop}\label{prop:bisetupedgesremoval}
	Let $0<\frac{1}{m}\ll\frac{1}{k}, \varepsilon\leq \varepsilon' \ll d\ll \frac{1}{\ell'}\ll 1$ and $\varepsilon'\ll \frac{1}{\ell^*}$. 
	Let $D$ be a digraph and suppose that $D'$ is obtained from $D$ by removing and adding at most $\varepsilon' m$ inedges and at most $\varepsilon' m$ outedges incident to each vertex.
	If $(D,\cP,\cP', \cP^*,R,C,U,U')$ is an $(\ell',\ell^*,k,m,\eps,d)$-(bi)-setup, then $(D',\cP,\cP', \cP^*,R,C,U,U')$ is an $(\ell',\ell^*,k,m,(\varepsilon')^{\frac{1}{3}},\frac{d}{2})$-(bi)-setup.
\end{prop}

\COMMENT{\begin{proof}
	We show the bi-setup version. The setup version holds analogously.
	Note that \cref{def:BST-P,def:BST-U,def:BST-U',def:BST-P*} hold immediately. Since $D'$ is obtained from $D$ by removing at most $\varepsilon' m\leq \frac{(\varepsilon')^{\frac{4}{5}}m}{\ell'}$ inedges and at most $\varepsilon' m\leq \frac{(\varepsilon')^{\frac{4}{5}}m}{\ell'}$ outedges incident to each vertex \cref{def:BST-R,def:BST-C,def:BST-U'supreg} follow from \cref{prop:epsremovingadding}. For \cref{def:BST-P'}, let $v\in V(D)$, $i\in [2k]$, and let $W\subseteq V_i$ be a cluster in $\cP'$. If $|N_{D'}^\pm(v)\cap V_i|\geq (\varepsilon')^{\frac{1}{3}}m$, then \cref{def:BST-P'} and \cref{def:URef} imply that
	\begin{align*}
		|N_{D'}^\pm(v)\cap W|&\leq|N_D^\pm(v)\cap W|+ \varepsilon' m\leq (1+ \varepsilon)\frac{|N_D^\pm(v)\cap V_i|}{\ell'}+ (\varepsilon')^{\frac{2}{3}}\ell'\frac{|N_D^\pm(v)\cap V_i|}{\ell'}\\
		&\leq (1+(\varepsilon')^{\frac{1}{3}})\frac{|N_D^\pm(v)\cap V_i|}{\ell'}
	\end{align*}
	and
	\begin{align*}
		|N_{D'}^\pm(v)\cap W|&\geq|N_D^\pm(v)\cap W|- \varepsilon' m\geq (1- \varepsilon)\frac{|N_D^\pm(v)\cap V_i|}{\ell'}- \varepsilon'\ell'\frac{|N_D^\pm(v)\cap V_i|}{\ell'}\\
		&\geq (1-(\varepsilon')^{\frac{1}{3}})\frac{|N_D^\pm(v)\cap V_i|}{\ell'}.
	\end{align*}
	Thus, \cref{def:URef} holds and so \cref{def:BST-P'} is satisfied.
\end{proof}}

Note that any partition is an $\varepsilon$-uniform $1$-refinement of itself.

\begin{fact}\label{fact:bisetupP}
	Suppose that $(D,\cP,\cP', \cP^*,R,C,U,U')$ is an $(\ell',\ell^*,2k,m,\eps,d)$-(bi)-setup. Then, $(D,\cP,\cP', \cP,R,C,U,U')$ is an $(\ell',1,2k,m,\eps,d)$-(bi)-setup.
\end{fact}

By definition, one can delete the exceptional vertices of a (bi)-setup.

\begin{fact}\label{fact:bisetupV0}
	Let $(D,\cP,\cP', \cP^*,R,C,U,U')$ be an $(\ell',\ell^*,2k,m,\eps,d)$-(bi)-setup. Denote by $V_0$ the exceptional set contained in $\cP$. Let $\cP_\emptyset, \cP_\emptyset'$, and $\cP_\emptyset^*$ be obtained from $\cP, \cP'$, and $\cP^*$ by replacing the exceptional set $V_0$ by the empty set. Then, $(D-V_0,\cP_\emptyset,\cP_\emptyset', \cP_\emptyset^*,R,C,U,U')$ is an $(\ell',\ell^*,2k,m,\varepsilon,d)$-(bi)-setup.
\end{fact}

Finally, observe that if $D$ forms a (bi)-setup, then the edges of $D$ can be randomly partitioned to obtain, with high probability, two edge-disjoint digraphs which both form a (bi)-setup. Indeed, properties \cref{def:ST-P,def:ST-U,def:ST-U'} of a setup and properties \cref{def:BST-P,def:BST-U,def:BST-U'} of a bi-setup are automatically preserved. Moreover, \cref{lm:URefrandom} implies that properties \cref{def:ST-P',def:ST-P*} of a setup and properties \cref{def:BST-P',def:BST-P*} of a bi-setup hold with high probability.
Finally, \cref{lm:randomreg} implies that (super)regularity is preserved with high probability, as desired for properties \cref{def:ST-R,def:ST-C,def:ST-U'supreg} of a setup and properties \cref{def:BST-R,def:BST-C,def:BST-U'supreg} of a bi-setup.

\begin{lm}\label{lm:randomsetup}
	Let $0<\frac{1}{m}\ll \frac{1}{k}\ll \varepsilon\ll \varepsilon' \ll d \ll \frac{1}{\ell'}\ll 1$ and $\varepsilon\ll \frac{1}{\ell^*}$. Let $D$ be a digraph and suppose that $D'$ is obtained from $D$ by selecting each edge independently with probability $\frac{1}{2}$. If $(D,\cP,\cP', \cP^*,R,C, U,U')$ is an $(\ell',\ell^*,k,m,\eps,d)$-(bi)-setup, then, with high probability, both $(D',\cP,\cP', \cP^*,R,C, U,U')$ and $(D\setminus D',\cP,\cP', \cP^*,R,C, U,U')$ are $(\ell',\ell^*,k,m,\varepsilon',\frac{d}{2})$-(bi)-setups.
\end{lm}

\COMMENT{\begin{proof}
		Note that \cref{def:ST-P,def:ST-U,def:ST-U'} and \cref{def:BST-P,def:BST-U,def:BST-U'} are clearly satisfied. For \cref{def:ST-R,def:ST-C,def:ST-U'supreg} and \cref{def:BST-R,def:BST-C,def:BST-U'supreg}, use the arguments of \cite[Lemma 4.10(iv)]{kuhn2013hamilton}. By \cref{lm:URefexistence}, \cref{def:ST-P',def:ST-P*} and \cref{def:BST-P',def:BST-P*} still hold with high probability.
\end{proof}}

\subsubsection{Special path systems and special factors}\label{sec:SPS}

Roughly speaking, special path systems can be viewed as blocks of prescribed edges for our Hamilton cycles; in the robust decomposition lemma (\cref{lm:newrobustdecomp} below), each special path system will be extended to a distinct Hamilton cycle.
Special path systems are then organised into special factors to provide a convenient way of finding them and incorporating them into Hamilton cycles in a balanced way.

\begin{definition}[\Gls*{canonical interval partition}]\label{def:interval}
	Let $V$ be a vertex set and $\cP$ be a partition of $V$ into an exceptional set and $k$ clusters. Suppose that $C=V_1\dots V_k$ is a Hamilton cycle on the clusters in $\cP$. Suppose that $\frac{k}{f}\in \mathbb{N}$. 
	The \emph{canonical interval partition of $C$ into $f$ intervals} is $\cI=\{I_1, \dots, I_f\}$, where
	\[I_i=V_{(i-1)\frac{k}{f}+1}V_{(i-1)\frac{k}{f}+2}\dots V_{i\frac{k}{f}+1}\]
	for each $i\in [f]$. For each $i\in [f]$, the clusters $V_{(i-1)\frac{k}{f}+2},V_{(i-1)\frac{k}{f}+3},\dots, V_{i\frac{k}{f}}$ are called the \emph{internal clusters} of the interval $I_i$.
\end{definition}

\begin{definition}[\Gls*{special path system}]\label{def:SPS}
	Let $V$ be a vertex set and $\cP$ be a partition of $V$ into an exceptional set $V_0$ and $k$ clusters of size $m$. Suppose that $C=V_1\dots V_k$ is a Hamilton cycle on the clusters in $\cP$. Suppose that $\frac{k}{f}\in \mathbb{N}$.
	Suppose that $\cP^*$ is an $\ell^*$-refinement of $\cP$. For each $i\in [k]$, let $V_{i,1}, \dots, V_{i,\ell^*}$ be an enumeration of the subclusters of $V_i$ contained in $\cP^*$.
	For any $(h,j)\in [\ell^*]\times [f]$, an \emph{$(\ell^*, f,h,j)$-special path system $SPS$ with respect to $\cP^*$ and $C$} is a set of $\frac{m}{\ell^*}$ vertex-disjoint paths satisfying the following conditions.\COMMENT{Replaced ``style $h$ and spanning the $j^{\rm th}$ interval" by ``$(\ell^*,f,h,j)$" to get a more compact definition and be able to track the role of each variable more easily. This is particularly helpful when we use the blow-up cycle version, which has even more parameters (see \cref{sec:cyclerobustdecomp}).}
	\begin{enumerate}[label=\rm(SPS\arabic*),longlabel]
		\item $V^+(SPS)=V_{(j-1)\frac{k}{f}+1,h}$ and $V^-(SPS)=V_{j\frac{k}{f}+1,h}$.\label{def:SPS-V+-}
		\item $V^0(SPS)=V_{(j-1)\frac{k}{f}+2,h}\cup \dots \cup V_{j\frac{k}{f},h}$.\label{def:SPS-V0}
	\end{enumerate}
\end{definition}

Roughly speaking, an $(\ell^*, f,h,j)$-special path system with respect to $\cP^*$ and $C$ is a set of vertex-disjoint paths which lies along the ``$h^{\rm th}$ refinement" of the $j^{\rm th}$ interval in the canonical interval partition of $C$ into $f$ intervals (see also \cref{fig:SPS}).

\begin{definition}[\Gls*{special factor}]\label{def:SF}
	Let $V, \cP^*, C$, and $V_0$ be as in \cref{def:SPS}.
	An \emph{$(\ell^*, f)$-special factor $SF$ with respect to $\cP^*$ and $C$} is a $1$-regular digraph on $V\setminus V_0$ which has a decomposition $\{SPS_{h,j}\mid (h,j)\in [\ell^*]\times [f]\}$ where, for each $(h,j)\in [\ell^*]\times [f]$, $SPS_{h,j}$ induces an $(\ell^*,f,h,j)$-special path system in $D$.
\end{definition}

Observe that in the original definition of a special path system in \cite{kuhn2013hamilton}, most of the edges belonged to a host digraph $D$ and the other edges, called ``fictive edges", had to satisfy some additional properties. Thanks to \cref{fact:equivalentP}, we can omit these conditions. Indeed, suppose that we want to construct a Hamilton cycle which contains a special path system $SPS$, but $SPS$ does not satisfy all the desired internal conditions. Then, we can temporarily consider a suitable equivalent special path system $SPS'$ and construct a Hamilton cycle $H$ containing $SPS'$ instead of $SPS$ since, by \cref{fact:equivalentP}, $H$ induces a Hamilton cycle containing $SPS$, as desired.

\begin{figure}[htb]
	\centering
	\includegraphics[width=0.8\textwidth]{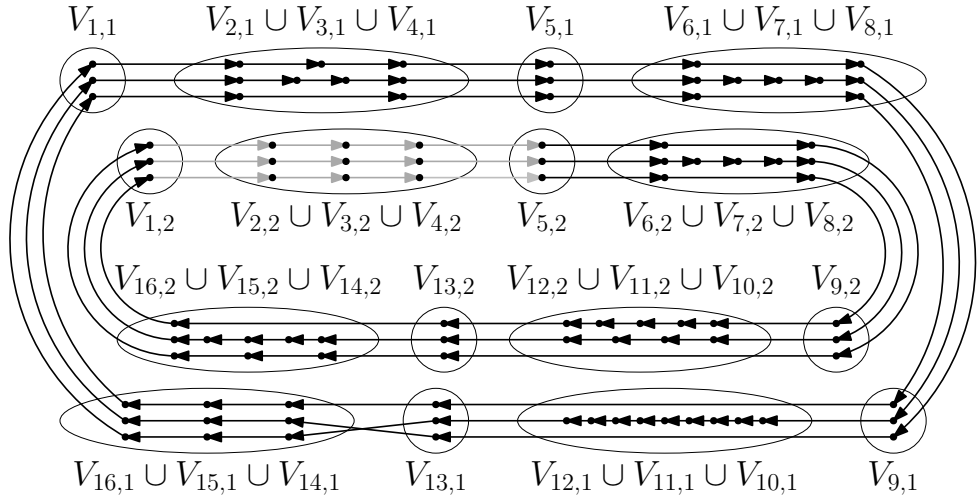}
	\caption{\small A $(2,4)$-special factor with respect to $\cP^*=\{V_0,V_{1,1}, V_{1,2}, V_{2,1}, \dots, V_{16,2}\}$ and $C=V_1\dots V_{16}$. The grey edges form a $(2,4,2,1)$-special path system with respect to $\cP^*$ and $C$.}\label{fig:SPS}
\end{figure}

\subsubsection{Statement of the robust decomposition lemma}\label{sec:newrobustdecomp}

We are now ready to state a modified version of the robust decomposition lemma of~\cite{kuhn2013hamilton}. We discuss the differences from the original version after the statement. \APPENDIX{(A formal derivation of \cref{lm:newrobustdecomp} is available in \cref{app:robustdecomp}.)}\NOAPPENDIX{(A formal derivation of \cref{lm:newrobustdecomp} is available in an appendix of the arXiv version of this paper.)}

Observe that $\mathcal{SF}$ and $\mathcal{SF}'$ may have edges with common starting and ending points in \cref{lm:newrobustdecomp}. Indeed, recall our convention that the edges of a multidigraph are all distinct (see \cref{sec:notation-graphs}) and that, in particular, a decomposition of a multidigraph covers each edge according to its multiplicity (see \cref{sec:notation-decomp}). Thus, \cref{lm:newrobustdecomp} simply states that each occurrence of an edge in each of $H, D^{\rm rob}, \mathcal{SF}$, and $\mathcal{SF}'$ is covered by precisely one of the Hamilton cycles in $\sC$.

\begin{lm}[{Modified robust decomposition lemma \cite{kuhn2013hamilton}}]\label{lm:newrobustdecomp}
	Let $0<\frac{1}{m}\ll \frac{1}{k}\ll \eps \ll \frac{1}{q} \ll \frac{1}{f} \ll \frac{r_1}{m}\ll d\ll \frac{1}{\ell'}, \frac{1}{g}\ll 1$ and suppose
	that $rk^2\le m$. Let
	\[r_2\coloneqq 96\ell'g^2kr, \quad  r_3\coloneqq \frac{rfk}{q}, \quad r^\diamond\coloneqq r_1+r_2+r-(q-1)r_3, \quad s'\coloneqq rfk+7r^\diamond,\]
	and suppose that $\frac{k}{14}, \frac{k}{f}, \frac{k}{g}, \frac{q}{f}, \frac{m}{4\ell'}, \frac{fm}{q}, \frac{2fk}{3g(g-1)} \in \mathbb{N}$.
	Suppose that $(D,\cP,\cP', \cP^*,R,C, U,U')$ is an $(\ell',\frac{q}{f},k,m,\eps,d)$-setup with empty exceptional set $V_0$.
	Let $\mathcal{SF}$ be a multidigraph which consists of the union of $r_3$ $(\frac{q}{f},f)$-special factors
	with respect to $\cP^*$ and $C$ and let $\mathcal{SF}'$ be a multidigraph which consists of the union of $r^\diamond$ $(1,7)$-special factors	with respect to $\cP$ and~$C$.
	
	Then, $D$ contains an $(r_1+r_2+5r^\diamond)$-regular spanning subdigraph $D^{\rm rob}$ for which the holds. 
	For any $r$-regular digraph $H$ on $V(D)$ which is edge-disjoint from $D^{\rm rob}$, the multidigraph $H \cup D^{\rm rob}\cup \mathcal{SF}\cup \mathcal{SF}'$ has a decomposition $\sC$ into $s'$ Hamilton cycles such that each cycle in $\sC$ contains precisely one of the special path systems in the multidigraph $\mathcal{SF}\cup \mathcal{SF}'$.
	
	The analogue holds if $(D,\cP,\cP', \cP^*,R,C,U,U')$ is an $(\ell',\frac{q}{f},k,m,\eps,d)$-bi-setup and $H$ is an $r$-regular bipartite digraph on the same vertex classes as $D$%
		\COMMENT{The vertex classes are $\bigcup\cV_{\rm even}$ and $\bigcup\cV_{\rm odd}$ in the original statement. But these are just $A$ and $B$ thanks to \cref{def:BST-P}.}.
\end{lm}

\Cref{lm:newrobustdecomp} differs from \cite[Lemma 12.1]{kuhn2013hamilton} in four minor points. \hypertarget{newrobustdecomplm-1}{(i)} As discussed in \cref{sec:SPS}, our definition of special factor is more general: there is no restriction on how many edges can lie outside $D$. 
\hypertarget{newrobustdecomplm-2}{(ii)} $\mathcal{SF}$ and $\mathcal{SF}'$ are multidigraphs rather than digraphs.
\hypertarget{newrobustdecomplm-3}{(iii)} The robustly decomposable digraph $D^{\rm rob}$ is now constructed in only one stage. In \cite[Lemma 12.1]{kuhn2013hamilton}, we first input a set of $r_3$ (edge-disjoint) $(\frac{q}{f},f)$-special factors to obtain a subdigraph $CA^\diamond(r)\subseteq D$ and then, in a second stage, we input a set of $r^\diamond$ (edge-disjoint) $(1,7)$-special factors to obtain a second subdigraph $PCA^\diamond(r)\subseteq D$. In \cref{lm:newrobustdecomp}, these two stages are condensed into one: $D^{\rm rob}$ from \cref{lm:newrobustdecomp} corresponds to $CA^\diamond(r)\cup PCA^\diamond(r)$ from \cite[Lemma 12.1]{kuhn2013hamilton}.
\hypertarget{newrobustdecomplm-4}{(iv)} $H$ only needs to be edge-disjoint from $D^{\rm rob}$.

Using the concept of equivalent sets of vertex-disjoint paths, modifications \hyperlink{newrobustdecomplm-1}{(i)}--\hyperlink{newrobustdecomplm-3}{(iii)} can be derived immediately from \cite[Lemma 12.1]{kuhn2013hamilton}. Indeed, since each special path system goes into a different Hamilton cycle, \cref{fact:equivalentP} implies that it is enough to apply \cite[Lemma 12.1]{kuhn2013hamilton} with special path systems which are equivalent to those in the multidigraph $\mathcal{SF}\cup \mathcal{SF}'$. In both stages of the application of \cite[Lemma 12.1]{kuhn2013hamilton}, one can use the superregular pairs in $D$ (which exist by \cref{def:ST-C} an \cref{def:BST-C}) to find suitable special path systems in $D$ which are equivalent to those contained in the multidigraph $\mathcal{SF}\cup \mathcal{SF}'$ and edge-disjoint from each other (as well as from $CA^\diamond(r)$ in the second stage).

Modification \hyperlink{newrobustdecomplm-4}{(iv)} cannot be derived immediately from the statement of \cite[Lemma 12.1]{kuhn2013hamilton} but follows easily from its proof and \cref{fact:equivalentP}. More precisely, using the equivalent special path system approach described above, \cite[Lemma 12.1]{kuhn2013hamilton} requires that $H$ is edge-disjoint from $D^{\rm rob}$ as well as the auxiliary special path systems we used to apply \cite[Lemma 12.1]{kuhn2013hamilton}. But, the proof of \cite[Lemma 12.1]{kuhn2013hamilton} implies that the relevant absorbing properties come from $CA^\diamond(r)$ and $PCA^\diamond(r)$. Thus, if $H$ is not edge-disjoint from the auxiliary special path systems used to apply \cite[Lemma 12.1]{kuhn2013hamilton}, then \cref{fact:equivalentP} implies that we can simply replace these special path systems by equivalent ones which are edge-disjoint from $H$.
\APPENDIX{(More details on how to obtain these modifications can be found in \cref{app:robustdecomp}.)}\NOAPPENDIX{(More details on how to obtain these modifications can be found in an appendix of the arXiv version of this paper.)}

\subsubsection{Incorporating the exceptional vertices}\label{sec:SC}
Recall that \cref{lm:newrobustdecomp} can only be applied with an empty exceptional set $V_0$. In general, $V_0$ will be non-empty and so we will have to apply \cref{lm:newrobustdecomp} with $D-V_0$ playing the role of $D$. As a result, the cycles obtained via \cref{lm:newrobustdecomp} will not be Hamilton cycles on $V(D)$, they will only span $V(D)\setminus V_0$. We will incorporate the exceptional vertices into these almost spanning cycles using the special path systems as follows. (Note that a special cover as defined below is a generalisation of an exceptional cover as defined in \cite{kuhn2013hamilton}, while a complete special sequence as defined below is the analogue of a complete exceptional sequence as defined in \cite{kuhn2013hamilton}.)

\begin{definition}[\Gls*{special cover}]\label{def:SC}
	Let $D$ be a digraph and $V_0\subseteq V(D)$ be an exceptional set. A \emph{special cover in $D$ (with respect to $V_0$)} is a linear forest $SC\subseteq D$ such that $V^0(SC)=V_0$.
\end{definition}

\begin{definition}[\Gls*{complete special sequence}]\label{def:MSC}
	Let $D$ be a digraph and $V_0\subseteq V(D)$ be an exceptional set. Suppose that $SC$ is a special cover in $D$. Let $P_1, \dots, P_\ell$ be an enumeration of the components of $SC$ which are not isolated vertices. For each $i\in [\ell]$, denote by $u_i$ and $v_i$ the starting and ending points of $P_i$. The \emph{complete special sequence associated to $SC$} is the directed matching $M_{SC}\coloneqq \{u_iv_i\mid i\in [\ell]\}$.
\end{definition}

Then, observe that the following holds.

\begin{fact}\label{fact:SC}
	Let $D$ be a digraph and $V_0\subseteq V(D)$ be an exceptional set. Let $SC$ be a special cover in $D$ with respect to $V_0$ and denote by $M_{SC}$ the complete special sequence associated to $SC$. Suppose that $C$ is a spanning cycle on $V(D)\setminus V_0$ satisfying $M_{SC}\subseteq C\subseteq D\cup M_{SC}$. Then, $(C\setminus M_{SC})\cup SC$ is a Hamilton cycle of $D$.
\end{fact}

Our strategy for incorporating the exceptional vertices into the cycles obtained via the robust decomposition lemma will thus be as follows. Before applying the robust decomposition lemma, we will reserve the edges from $s'$ special covers in $D$ with respect to $V_0$. Then, we will construct the special factors for \cref{lm:newrobustdecomp} in such a way that each of the $s'$ special path systems contains the complete special sequence associated to one of the reserved special covers. Since each cycle obtained via the robust decomposition lemma will contain precisely one of these special path systems, these cycles will contain precisely one of these complete special sequences. Using \cref{fact:SC} and the reserved special covers, we will thus be able to transform the cycles from the robust decomposition lemma into Hamilton cycles on $V(D)$.
We discuss how to find these special sequences in \cref{sec:applybirobdecomp}.

\onlyinsubfile{\bibliographystyle{abbrv}
\bibliography{Bibliography/Bibliography}}

\subsection{The preprocessing step}\label{sec:preprocessing}

	\onlyinsubfile{
		\setcounter{section}{6}
		\setcounter{subsection}{2}
		\subsection{The preprocessing lemma}}

Recall that in the robust decomposition lemma (\cref{lm:newrobustdecomp}), the exceptional set must be empty. In the proof of \cref{thm:biprobexp}, the partition $\cP$ in the bi-setup will be obtained via the regularity lemma (\cref{lm:bipreglm}) and so we will have a non-empty exceptional set $V_0$. Thus, we will only be able to apply the (bipartite) robust decomposition lemma in $D-V_0$ rather than $D$. This means that the absorber $D^{\rm rob}$ will only be able to decompose leftovers on $V(D)\setminus V_0$ and so we need an additional absorber to cover the leftover edges incident to $V_0$. This absorber will be constructed by adapting the preprocessing step of~\cite{kuhn2013hamilton} to the bipartite case. Roughly speaking, the preprocessing lemma (\cref{lm:preprocessing} below) says that given a bipartite digraph $D$ and an exceptional set $V_0$, one can find a sparse absorber $PG\subseteq D$ (called a preprocessing graph in \cite{kuhn2013hamilton}) such that for any very sparse leftover $H$ which is edge-disjoint from $PG$, the digraph $H\cup PG$ contains edge-disjoint Hamilton cycles which cover all the edges incident to $V_0$.

\subsubsection{Consistent bi-systems}
First, we need a bipartite analogue of a consistent system defined in \cite{kuhn2013hamilton}. This is the key structure required to construct the absorber in the preprocessing lemma. It is similar to a bi-setup.

\begin{definition}[\Gls*{consistent bi-system}]\label{def:CBSys}
	We say that $(D,\cP_0, R_0,C_0,\cP,R,C)$ is a \emph{consistent
		$(\ell^*,k,m,\eps,d,\nu,\tau,\delta,\theta)$-bi-system} if the following properties
	are satisfied.
	\begin{enumerate}[label=\rm(CBSys\arabic*),longlabel]
		\item $D$, $R_0$, and $R$ are balanced bipartite digraphs on vertex classes $A$ and $B$, $\cA_0$ and $\cB_0$, and $\cA$ and $\cB$, respectively.
		Moreover, $D$, $R_0$, and $R$ are bipartite robust $(\nu,\tau)$-outexpanders with bipartitions $(A,B)$, $(\cA_0,\cB_0)$, and $(\cA, \cB)$, respectively. Furthermore, $\delta^0(D)\ge \delta |D|$,
		$\delta^0(R_0)\ge \delta |R_0|$, and $\delta^0(R)\ge \delta |R|$.\label{def:CBSys-biproboutexp} 
		\item $\mathcal{P}_0$ is a partition of $V(D)$ into an exceptional set $V_0$ which satisfies $|V_0\cap A|=|V_0\cap B|\leq \varepsilon |A|=\varepsilon|B|$,
		and $\frac{k}{\ell^*}$ clusters of size $m\ell^*$. The vertex set of $R_0$ consists of these clusters. (Thus, $|R_0|=\frac{k}{\ell^*}$.)
		\label{def:CBSys-P0}
		\item $\mathcal{P}$ is an $\ell^*$-refinement of $\mathcal{P}_0$ (and so the clusters in $\cP$ have size $m$). The vertex set of $R$ consists of the clusters in $\cP$. (Thus, $|R|=k$.)\label{def:CBSys-P}%
			\COMMENT{This is stronger than in \cite{kuhn2013hamilton}.}
		\item For each $VW\in E(R)$, the corresponding pair $D[V,W]$ is $(\eps,\ge d)$-regular.\label{def:CBSys-R}
		\item $C_0$ is a Hamilton cycle in $R_0$ and $C$ is a Hamilton cycle in $R$. For each $VW\in E(C)$, the corresponding
		pair $D[V,W]$ is $[\eps,\ge d]$-superregular.\label{def:CBSys-C}
		\item Suppose that $W,W'\in V(R_0)$ and $V,V'\in E(R)$ satisfy $V\subseteq W$ and $V'\subseteq W'$.
		Then, $WW'\in E(R_0)$ if and only if $VV'\in E(R)$.\label{def:CBSys-R0}
		\item $C$ can be viewed as obtained by winding $\ell^*$ times around $C_0$, i.e.\ for every edge $WW'\in E(C_0)$,
		there are precisely $\ell^*$ edges $VV'\in E(C)$ such that $V\subseteq W$ and $V'\subseteq W'$.\label{def:CBSys-C0}
		\item Let $V$ be a cluster in $\mathcal{P}_0$ and $W\subseteq V$ be a cluster in $\cP$. Let $v\in V(D)$ and $\diamond\in \{+,-\}$.
		If $|N_D^\diamond(v)\cap V|\ge \tau |V|$, then $|N_D^\diamond(v)\cap W|\geq \frac{\theta |N_D^\diamond(v)\cap V|}{\ell^*}$.\label{def:CBSys-deg}
	\end{enumerate}
\end{definition}

Observe that the original definition of a consistent bi-system in \cite{kuhn2013hamilton} also required the exceptional set $V_0$ to be an independent set. However, we will see that this condition is not necessary for our purposes and so we omit it here for convenience.

By \cref{cor:verticesedgesremovalbiproboutexp,prop:epsremovingadding}, a consistent bi-system remains a consistent bi-system (with slightly worse parameters) if only a few edges are removed at each vertex. An analogous observation was made (and proved) in \cite[Lemma 7.1]{kuhn2013hamilton}, so we omit the details here. 

\begin{prop}\label{prop:CBSysedgesremoval}
	Let $0<\frac{1}{m}\ll \frac{1}{k}\ll \varepsilon\leq \varepsilon'\ll d \ll \nu\ll \tau \ll \delta, \theta\leq 1$. Let $D$ be a digraph and suppose that $D'$ is obtained from $D$ by removing at most $\varepsilon'm$ inedges and at most $\varepsilon'm$ outedges incident to each vertex.
	If $(D,\cP_0, R_0,C_0,\cP,R,C)$ is a consistent	$(\ell^*,k,m,\eps,d,\nu,\tau,\delta,\theta)$-bi-system, $(D',\cP_0, R_0,C_0,\cP,R,C)$ is still a consistent	$(\ell^*,k,m,3\sqrt{\varepsilon'},\frac{d}{2},\frac{\nu}{2},\tau,\frac{\delta}{2},\frac{\theta}{2})$-bi-system.%
		\COMMENT{Note that \cite[Lemma 7.1]{kuhn2013hamilton} has $d$ unchanged. This is because of the slight difference in the definition of $(\varepsilon, \geq d)$-regularity.}
\end{prop}

\COMMENT{\begin{proof}
	First, observe that \cref{def:CBSys-P0,def:CBSys-P,def:CBSys-R0,def:CBSys-C0} are immediately satisfied. Moreover, \cref{def:CBSys-biproboutexp} follows from \cref{cor:verticesedgesremovalbiproboutexp}, while \cref{def:CBSys-R,def:CBSys-C} follow from \cref{prop:epsremovingadding}. For \cref{def:CBSys-deg}, let $V$ be a cluster in $\cP_0$ and $W\subseteq V$ be a cluster in $\cP$. Suppose that $v\in V(D)$ satisfies $|N_{D'}^+(v)\cap V|\geq \tau |V|$. Then, \cref{def:CBSys-deg} implies that
	\begin{align*}
		|N_{D'}^+(v)\cap W|\geq |N_D^+(v)\cap W|-\varepsilon' m\geq \frac{\theta |N_D^+(v)\cap V|}{\ell^*}-\frac{\varepsilon' |N_{D'}^+(v)\cap V|}{\tau \ell^*}\geq \frac{\theta |N_{D'}^+(v)\cap V|}{2\ell^*}.
	\end{align*}
	Similarly, if $|N_{D'}^-(v)\cap V|\geq \tau |V|$, then $|N_{D'}^+(v)\cap W|\geq \frac{\theta |N_{D'}^+(v)\cap V|}{2\ell^*}$. Therefore, \cref{def:CBSys-deg} holds.
\end{proof}}

Finally, we observe that if $D$ forms a consistent bi-system, then the edges of $D$ can be randomly partitioned to obtain, with high probability, two edge-disjoint digraphs which both form a consistent bi-system. Indeed,  \cref{def:CBSys-P0,def:CBSys-P,def:CBSys-R0,def:CBSys-C0} are automatically preserved and, by \cref{lm:randomrob}, bipartite robust outexpansion is preserved with high probability, as desired for \cref{def:CBSys-biproboutexp}. Moreover, a simple application of \cref{lm:Chernoff} can guarantee suitable minimum degree conditions for \cref{def:CBSys-biproboutexp,def:CBSys-deg}. 
Finally, \cref{lm:randomreg} implies that (super)regularity is preserved with high probability, as desired for \cref{def:CBSys-R,def:CBSys-C}.

\begin{lm}\label{lm:randomsystem}
	Let $0<\frac{1}{m}\ll \frac{1}{k}\ll \varepsilon\ll \varepsilon' \ll d \ll\nu\ll \tau\ll\delta, \theta\leq 1$. Suppose that $(D,\cP_0, R_0, C_0, \cP, R,C)$ is a consistent $(\ell^*,k,m,\eps,d, \nu, \tau, \delta, \theta)$-bi-system. Let $D'$ be obtained from $D$ by selecting each edge independently with probability $\frac{1}{2}$. Then, with high probability, both $(D',\cP_0, R_0, C_0, \cP, R,C)$ and $(D\setminus D',\cP_0, R_0, C_0, \cP, R,C)$ are consistent $(\ell^*,k,m,\varepsilon',\frac{d}{2}, \frac{\nu}{4}, \tau, \frac{\delta}{3}, \frac{\theta}{3})$-bi-systems.
\end{lm}

\COMMENT{\begin{proof}
	\Cref{def:CBSys-P0,def:CBSys-P,def:CBSys-R0,def:CBSys-C0} hold automatically. For \cref{def:CBSys-R,def:CBSys-C}, use the arguments of \cite[Lemma 3.2(ii)]{kuhn2014hamilton}. By \cref{lm:randomrob}, both $D'$ and $D\setminus D'$ are bipartite robust $(\frac{\nu}{4}, \tau)$-outexpanders with high probability. We have $\mathbb{E}[\delta^0(D')]=\mathbb{E}[\delta^0(D\setminus D')]=\frac{\delta^0(D)}{2}\geq \frac{\delta n}{2}$. Thus, \cref{lm:Chernoff} implies that both $\delta^0(D'),\delta^0(D\setminus D')\geq \frac{\delta n}{3}$ with high probability. So \cref{def:CBSys-biproboutexp} holds. For \cref{def:CBSys-deg}, note that $\mathbb{E}[|N_{D'}^\diamond(v)\cap W|]=\mathbb{E}[|N_{D\setminus D'}^\diamond(v)\cap W|]=\frac{|N_D^\diamond(v)\cap W|}{2}\geq \frac{\theta |N_D^\diamond(v)\cap V|}{2\ell^*}$ and so \cref{lm:Chernoff} implies that both $|N_{D'}^\diamond(v)\cap W|,|N_{D\setminus D'}^\diamond(v)\cap W|\geq \frac{\theta |N_D^\diamond(v)\cap V|}{3\ell^*}$.
\end{proof}}

\subsubsection{Statement of the preprocessing lemma for bipartite digraphs}

The following \lcnamecref{lm:preprocessing} is a bipartite analogue of \cite[Corollary 8.5 and Lemma 8.6]{kuhn2013hamilton}.
Since it can be proved using very similar arguments as those of \cite{kuhn2013hamilton}, we omit its proof here. (\APPENDIX{A detailed explanation on how to derive \cref{lm:preprocessing} can be found in \cref{app:preprocessing}.}\NOAPPENDIX{A detailed explanation on how to derive \cref{lm:preprocessing} can be found in an appendix of the arXiv version of this paper.})
As mentioned at the beginning of \cref{sec:preprocessing}, \cref{lm:preprocessing} states that a consistent bi-system contains a sparse absorber $PG$ which can cover all the exceptional edges of a very sparse leftover $H$ with edge-disjoint Hamilton cycles.

\begin{lm}[Preprocessing lemma for bipartite digraphs {\cite{kuhn2013hamilton}}]\label{lm:preprocessing}
	Let $0<\frac{1}{m}\ll \frac{r}{m}\ll \frac{r'}{m}\ll \frac{1}{k}\ll \eps\ll \frac{1}{\ell^*}\ll d\ll \nu\ll \tau\ll \delta, \theta\le 1$. Denote $s\coloneqq \frac{10^7}{\nu^2}$ and suppose that $\frac{m}{50}, \frac{50\ell^*}{s-1}\in\mathbb{N}$.
	Let $(D,\cP_0, R_0,C_0,\cP,R,C)$ be a consistent $(\ell^*,k,m,\eps,d,\nu,\tau,\delta,\theta)$-bi-system.
	Then, there exists a spanning subdigraph $PG\subseteq D$ such that the following hold.
	\begin{enumerate}
		\item Each $v\in V_0$ satisfies $d_{PG}^\pm(v)=r(s-1)$ and each $w\in V(D)\setminus V_0$ satisfies $d_{PG}^\pm(w)=r'$. \label{lm:preprocessing-degrees}
		\item Let $H$ be an $r$-regular bipartite digraph on the same vertex classes as $D$. If $H$ is edge-disjoint from $PG$ and satisfies $e(H[V_0])=0$, then $H\cup PG$ contains $rs$ edge-disjoint Hamilton cycles $C_1, \dots, C_{rs}$ such that the following hold.\label{lm:preprocessing-Hamcycles}
		\begin{enumerate}[label=\rm(\alph*),ref=\rm(\roman{enumi}.\alph*)]
			\item $H\subseteq \bigcup_{i\in [rs]} C_i\subseteq H\cup PG$.\label{lm:preprocessing-Hamcycles-H}
			\item Let $PG'\coloneqq PG\setminus \bigcup_{i\in [rs]}C_i$. Then, each $v\in V_0$ satisfies $d_{PG'}^\pm(v)=0$ and each $w\in V(D)\setminus V_0$ satisfies $d_{PG'}^\pm(w)= r'-r(s-1)$.\label{lm:preprocessing-Hamcycles-degrees}
		\end{enumerate}
	\end{enumerate}
\end{lm}

Note that in \cite[Corollary 8.5]{kuhn2013hamilton}, $H$ must be a spanning subdigraph of the host digraph $D$. However, this condition is not necessary since, as $H$ is very sparse, its edges could be added to the host digraph without affecting the parameters of the consistent system significantly. This is why, in our bipartite version of the preprocessing lemma (\cref{lm:preprocessing}), we may omit the condition that $H$ is a spanning subdigraph of $D$.
Moreover, \cite[Corollary 8.5]{kuhn2013hamilton} requires the exceptional set $V_0$ to form an independent set in $D$. But, \cref{prop:CBSysedgesremoval} implies that $D\setminus D[V_0]$ also contains a consistent bi-system (with slightly worse parameters). Thus, this condition can be omitted in \cref{lm:preprocessing}.

Since $PG'$ is a regular digraph on $V(D)\setminus V_0$, we can use $D^{\rm rob}$ from the robust decomposition lemma (\cref{lm:newrobustdecomp}) to decompose it into Hamilton cycles on $V(D)\setminus V_0$. (The vertices in $V_0$ will later be incorporated into these cycles via the special path systems as discussed in \cref{sec:SC}.) Thus, we can combine $PG$ and $D^{\rm rob}$ to form an absorber $D^{\rm abs}$ which can decompose a sparse leftover digraph $H$ which have edges incident to $V_0$. This is made precise in the following \lcnamecref{cor:absorber}.

\begin{cor}\label{cor:absorber}
	Let $0<\frac{1}{m}\ll \frac{r}{m}\ll \frac{r'}{m}\ll \frac{1}{k}\ll \eps \ll \frac{1}{q} \ll \frac{1}{\ell^*},\frac{1}{f} \ll \frac{r_1}{m}\ll d\ll \nu,\frac{1}{\ell'}, \frac{1}{g}\ll \tau\ll \delta,\theta\leq 1$. Let 
	\[s\coloneqq \frac{10^7}{\nu^2}, \quad r^*\coloneqq r'-(s-1)r,\]
	\[r_2\coloneqq 96\ell'g^2kr^*, \quad  r_3\coloneqq \frac{r^*fk}{q}, \quad r^\diamond\coloneqq r_1+r_2+r^*-(q-1)r_3, \quad s'\coloneqq r^*fk+7r^\diamond.\]
	Suppose that $\frac{k}{14}, \frac{k}{f}, \frac{k}{g}, \frac{q}{f}, \frac{m}{50}, \frac{m}{4\ell'}, \frac{fm}{q}, \frac{2fk}{3g(g-1)}, \frac{50\ell^*}{s-1} \in \mathbb{N}$.
	Suppose that $(D,\cP_0, R_0,C_0, \cP, R, C)$ is a consistent $(\ell^*,k,m,\eps,d,\nu,\tau,\delta,\theta)$-bi-system and suppose that $(D,\cP,\cP', \cP^*,R,C, U,U')$ is an $(\ell',\frac{q}{f},k,m,\eps,d)$-bi-setup. Suppose that the exceptional set $V_0$ forms an independent set in~$D$.
	Let $\mathcal{SF}$ be a multidigraph which consists of the union of $r_3$ $(\frac{q}{f},f)$-special factors
	with respect to $\cP^*$ and $C$ and let $\mathcal{SF}'$ be a multidigraph which consists of the union of $r^\diamond$ $(1,7)$-special factors	with respect to $\cP$ and $C$.
	Then, $D$ contains a spanning subdigraph $D^{\rm abs}$ for which the following hold.
	\begin{enumerate}
		\item Each $v\in V_0$ satisfies $d_{D^{\rm abs}}^\pm(v)=r(s-1)$ and each $w\in V(D)\setminus V_0$ satisfies $d_{D^{\rm abs}}^\pm(w)=r'+r_1+r_2+5r^\diamond$.\label{cor:absorber-degrees}
		\item Let $H$ be an $r$-regular bipartite digraph on the same vertex classes as $D$. If $H$ is edge-disjoint from $D^{\rm abs}$ and satisfies $e(H[V_0])=0$, then the multidigraph $H \cup D^{\rm abs}\cup \mathcal{SF}\cup \mathcal{SF}'$ has a decomposition $\sC\cup \sC'$ into $rs+s'$ edge-disjoint cycles satisfying the following properties.\label{cor:absorber-decomp}
		\begin{enumerate}[label=\rm(\alph*),ref=\rm(\roman{enumi}.\alph*)]
			\item $\sC\subseteq H\cup D^{\rm abs}$ and consists of $rs$ edge-disjoint Hamilton cycles on $V(D)$.\label{cor:absorber-decomp-preprocessing}
			\item $\sC'$ consist of $s'$ edge-disjoint Hamilton cycles on $V(D)\setminus V_0$ such that each cycle in $\sC'$ contains precisely one of the special path systems in the multidigraph $\mathcal{SF}\cup \mathcal{SF}'$.\label{cor:absorber-decomp-robustdecomp}
		\end{enumerate}
	\end{enumerate}
\end{cor}

\begin{proof}
	First, let $PG$ be the spanning subdigraph of $D$ obtained by applying \cref{lm:preprocessing}. Define $D'\coloneqq D\setminus PG$ and let $\cP_\emptyset, \cP_\emptyset'$, and $\cP_\emptyset^*$ be obtained by replacing the exceptional set $V_0$ by the empty set in $\cP, \cP'$, and $\cP^*$, respectively. By \cref{lm:preprocessing}\cref{lm:preprocessing-degrees}, \cref{prop:bisetupedgesremoval}, and \cref{fact:bisetupV0}, $(D'-V_0, \cP_\emptyset, \cP_\emptyset', \cP_\emptyset^*, R,C,U,U')$ is an $(\ell',\frac{q}{f},k,m,\varepsilon^{\frac{1}{3}},\frac{d}{2})$-bi-setup with empty exceptional set. Let $D^{\rm rob}$ be the spanning subdigraph of $D'-V_0$ obtained by applying \cref{lm:newrobustdecomp} with $D'-V_0, \cP_\emptyset, \cP_\emptyset', \cP_\emptyset^*, \varepsilon^{\frac{1}{3}}, \frac{d}{2}$, and $r^*$ playing the roles of $D, \cP, \cP', \cP^*, \varepsilon, d$, and $r$.
	
	Define $D^{\rm abs}\coloneqq PG\cup D^{\rm rob}$. Then, \cref{cor:absorber-degrees} follows from \cref{lm:preprocessing}\cref{lm:preprocessing-degrees} and \cref{lm:newrobustdecomp}. For \cref{cor:absorber-decomp}, let $H$ be an $r$-regular bipartite digraph on the same vertex classes as $D$. Suppose that $H$ is edge-disjoint from $D^{\rm abs}$ and satisfies $e(H[V_0])=0$. Let $\sC$ be the set of $rs$ Hamilton cycles obtained by applying \cref{lm:preprocessing}\cref{lm:preprocessing-Hamcycles}. 
	In particular, \cref{cor:absorber-decomp-preprocessing} holds and \cref{lm:preprocessing}\cref{lm:preprocessing-Hamcycles-H} implies that $E(H)\subseteq E(\sC)\subseteq E(H)\cup E(PG)$.
	Let $PG'\coloneqq PG\setminus \sC$. Then, \cref{lm:preprocessing}\cref{lm:preprocessing-Hamcycles-degrees} implies that $E(PG')=E(PG'-V_0)$ and $PG'-V_0$ is $r^*$-regular. Let $\sC'$ be the set Hamilton cycles on $V(D)\setminus V_0$ obtained by applying \cref{lm:newrobustdecomp} with $PG'$ and $r^*$ playing the roles of $H$ and $r$. Then, \cref{cor:absorber-decomp-robustdecomp} holds and we are done.
\end{proof}

\onlyinsubfile{\bibliographystyle{abbrv}
	\bibliography{Bibliography/Bibliography}}

\section{The bipartite robust outexpander case: proof of Theorem \ref{thm:biprobexp}}\label{sec:biprobexp}

	\onlyinsubfile{
		\setcounter{section}{6}
\section{The bipartite robust outexpander case}}

In this section, we combine the approximate decomposition (\cref{cor:bipapproxHamdecomp}) and the absorption step (\cref{cor:absorber}) to derive \cref{thm:biprobexp}.

\subsection{Applying the robust decomposition lemma in a bipartite robust outexpander}\label{sec:applybirobdecomp}

In order to apply \cref{cor:absorber}, we will need to find a consistent bi-system, a bi-setup, and special factors. In this section, we discuss how these can be found in a bipartite robust outexpander.

First, we explain how to form the special factors required for \cref{cor:absorber}. As discussed in \cref{sec:SC}, their role is to incorporate the exceptional vertices into the cycles obtained via the robust decomposition lemma. By \cref{fact:SC}, we would like each special path system to consist of a complete special sequence (recall \cref{def:MSC}) and edges of $D$. (Note that such special path systems were called complete exceptional path systems in \cite{kuhn2013hamilton}.)
One can achieve this by adapting the arguments of \cite[Lemma 7.6]{kuhn2013hamilton} to the bipartite case (\APPENDIX{see \cref{app:preprocessing} for more details}\NOAPPENDIX{see the appendix in the arXiv version of this paper for more details}).

\begin{lm}[Constructing special covers and special factors]\label{lm:SF}
	Let $0< \frac{1}{m}\ll \frac{1}{k}\ll \varepsilon\ll d\ll \nu \ll \tau\ll \delta, \theta \leq 1$ and $\varepsilon\ll \frac{1}{\ell'}, \frac{1}{f}$ and $\frac{\ell' r}{m}\ll d$. Suppose that $\frac{\ell^*}{f}, \frac{m}{\ell'}\in \mathbb{N}$ and $\frac{f}{\ell^*}\ll 1$.
	Let $(D, \cP_0, R_0, C_0, \cP, R, C)$ be a consistent $(\ell^*, k,m,\varepsilon, d, \nu, \tau, \delta, \theta)$-bi-system with exceptional set $V_0$.
	Let $\cP'$ be an $\varepsilon$-uniform $\ell'$-refinement of $\cP$.	
	Then, there exist
	\begin{enumerate}
		\item a set $\mathcal{SC}=\{SC_{i,h,j}\mid (i,h,j)\in [r]\times[\ell']\times [f]\}$ of $r\ell'f$ edge-disjoint special covers in $D$ with respect to $V_0$ and
		\item $r$ $(\ell',f)$-special factors $SF_1, \dots, SF_r$ with respect to $\cP'$ and $C$
	\end{enumerate}	
	such that the following hold, where for each $(i,h,j)\in [r]\times [\ell']\times [f]$, $M_{i,h,j}$ denotes the complete special sequence associated to $SC_{i,h,j}$ and $SPS_{i,h,j}$ denotes the $(\ell',f,h,j)$-special path system contained in $SF_i$.
	\begin{enumerate}[resume]
		\item For each $(i,h,j)\in [r]\times [\ell']\times [f]$, we have $M_{i,h,j}\subseteq SPS_{i,h,j}\subseteq (D\setminus \mathcal{SC})\cup M_{i,h,j}$.\label{lm:SF-M}
		\item Let $(i,h,j),(i',h',j')\in [r]\times [\ell']\times [f]$ be distinct. Then, we have $(SPS_{i,h,j}\setminus M_{i,h,j})\cap (SPS_{i',h',j'}\setminus M_{i',h',j'})=\emptyset$.\label{lm:SF-disjoint}
	\end{enumerate} 
\end{lm}

Roughly speaking, \cref{lm:SF}\cref{lm:SF-M} means that each complete special sequence is incorporated into a distinct special path system, while \cref{lm:SF}\cref{lm:SF-disjoint} states that each edge of $D\setminus \mathcal{SC}$ is incorporated into at most one of the special path systems.
Note that in \cite[Lemma 7.6]{kuhn2013hamilton}, the exceptional set $V_0$ must form an independent set in $D$. But, \cref{prop:CBSysedgesremoval} implies that $D\setminus D[V_0]$ also contains a consistent bi-system (with slightly worse parameters). Thus, this condition can be omitted in \cref{lm:SF}.

Consistent bi-systems and bi-setups can be constructed from Szemer\'edi's regularity lemma (\cref{lm:bipreglm}) as follows.
Recall that, by \cref{lm:Rrob}, the reduced digraph $R_0$ obtained by applying the regularity lemma (\cref{lm:bipreglm}) to a bipartite robust outexpander $D$ is also a bipartite robust outexpander. By \cref{lm:biprobblowup2}, this property is also preserved when taking refinements. Thus, one can obtain a consistent bi-system by applying the regularity lemma to obtain a partition $\cP_0$ of $V(D)$, then selecting a uniform refinement $\cP$ of $\cP_0$ using \cref{lm:URefexistence}, and finally using \cref{lm:biprobHamcycle} to find the desired Hamilton cycles $C_0$ and $C$.

The additional refinements $\cP'$ and $\cP^*$ of $\cP$ required to form a bi-setup can also be obtained by applying \cref{lm:URefexistence} (the superregular pairs required for \cref{def:BST-U'supreg} can be obtained using \cref{lm:URefreg}). The next \lcnamecref{lm:U} gives the bi-universal walk required for \cref{def:BST-U}. (\Cref{lm:BU} is easily proved by adapting the arguments of \cite[Lemma 9.1]{kuhn2013hamilton} to the bipartite case,
so we defer its proof to \APPENDIX{\cref{app:preprocessing}}\NOAPPENDIX{an appendix of the arXiv version of this paper}.)

\begin{lm}\label{lm:BU}
	Let $0<\frac{1}{k}\ll \nu \ll \tau\ll \delta <1$ and let $\ell'$ be an even integer satisfying $\ell'\geq 36\nu^{-2}$. Suppose that $R$ a balanced bipartite robust $(\nu, \tau)$-outexpander with bipartition $(\cA, \cB)$, where $|\cA|=|\cB|=k$. Suppose that $\delta^0(R)\geq \delta k$. Let $C$ be a Hamilton cycle in $R$. Then, there exists a bi-universal walk $U$ for $C$ with parameter $\ell'$.
\end{lm}

Altogether, we obtain a consistent bi-system and a bi-setup (as defined in \cref{def:BST,def:CBSys}). Then, we apply \cref{lm:randomsetup,lm:randomsystem} to partition the edges of $D$ into two edge-disjoint subdigraphs $D_1$ and $D_2$ which each form a consistent bi-system and a bi-setup. We will use $D_1$ to construct the special factors (\cref{lm:SF}) and $D_2$ to construct the absorber (\cref{cor:absorber}). This is necessary because, after constructing all the special factors, the clusters in $\cP$ and $\cP'$ may no longer form suitable (super)regular pairs (that is, properties \cref{def:CBSys-R,def:CBSys-C} of a consistent bi-system and properties \cref{def:BST-R,def:BST-C,def:BST-U'supreg} of a bi-setup might not hold anymore), so we cannot apply \cref{cor:absorber} directly after \cref{lm:SF}.

\begin{lm}\label{lm:bisetup}
	Let $0<\frac{1}{M'}\ll \varepsilon$. Then, there exist $M'',n_0\in \mathbb{N}$ such that the following holds.
	Suppose that $\varepsilon\ll\frac{1}{q}\ll \frac{1}{f}, \frac{1}{\ell^*}\ll d \ll \nu \ll \tau \ll \delta, \theta \ll 1$ and $d\ll \frac{1}{g}\ll 1$.
	Moreover, let $\ell'\geq 324\nu^{-2}$ be even.
	Let $D$ be a balanced bipartite digraph on vertex classes $A$ and $B$ of size $n\geq n_0$.
	Suppose that $D$ is a bipartite robust $(\nu, \tau)$-outexpander with bipartition $(A, B)$ and that $\delta^0(D)\geq \delta n$.
	Then, there exist $m,k\in \mathbb{N}$ and edge-disjoint $D_1, D_2\subseteq D$ such that the following conditions are satisfied.
	\begin{enumerate}
		\item $M'\leq k\leq M''$ and $\frac{k}{7},\frac{k}{f}, \frac{k}{g},\frac{m}{50},\frac{m}{4\ell'},\frac{fm}{q},\frac{2fk}{3g(g-1)}\in \mathbb{N}$. \label{lm:bisetup-parameters}
		\item There exist $\cP_0, \cP, \cP', \cP^*, R_0, R, C_0, C, U$, and $U'$ which satisfy the following conditions for each $i\in [2]$.\label{lm:bisetup-bisetup}
		\begin{itemize}
			\item $(D_i,\cP_0, R_0, C_0, \cP, R, C)$ is a consistent $(\ell^*,2k, m, \varepsilon, d, \nu^4, 8\tau, \frac{\delta}{9}, \theta)$-bi-system.
			\item $(D_i, \cP, \cP',\cP^*, R, C, U ,U')$ is an $(\ell', \frac{q}{f}, 2k, m, \varepsilon, d)$-bi-setup.
		\end{itemize}
	\end{enumerate}
\end{lm}

Note that in order to apply \cref{cor:absorber}, we need a constant $g$ which divides $k$. We also require that this constant $g$ is sufficiently large but sufficiently small compared to $k$ and $f$. We would not be able to fix such a parameter if, for instance, $f$ is a prime number and $k=7f$%
	\COMMENT{\cref{lm:bisetup} already gives that $7$ and $f$ divide $k$.}.
It is therefore necessary to introduce $g$ in \cref{lm:bisetup} even though it does not explicitly appear in \cref{lm:bisetup}\cref{lm:bisetup-bisetup}.

For a formal proof of \cref{lm:bisetup}, see \APPENDIX{\cref{app:reglm}}\NOAPPENDIX{the appendix of the arXiv version of this paper}.

\subsection{Proof of Theorem \ref{thm:biprobexp}}

We are now ready to derive \cref{thm:biprobexp}. Our strategy is as follows. 
In \cref{step:bisetup}, we construct consistent bi-systems and bi-setups using \cref{lm:bisetup}.
In \cref{step:birobexp-A}, we construct an absorber $D^{\rm abs}$ using \cref{cor:absorber} (the required special factors are constructed with \cref{lm:SF}). In \cref{step:birobexp-approxdecomp}, we approximately decompose $D$ using \cref{cor:bipapproxHamdecomp}. In \cref{step:birobexp-leftovers}, we decompose the leftover using $D^{\rm abs}$.

\begin{proof}[Proof of \cref{thm:biprobexp}]
	We may assume without loss of generality that $\delta \ll 1$.
	Let $0<\frac{1}{n_0}\ll \tau \ll \delta$ and $\frac{1}{n_0}\ll \nu$.
	By \cref{fact:biprobexpparameters}, we may assume that $\nu \ll \tau$.
	Let $n\geq n_0$ and $r\geq \delta n$. Let $D$ be an $r$-regular bipartite digraph on vertex classes $A$ and $B$ of size $n$. 
	Suppose that $D$ is a bipartite robust $(\nu, \tau)$-outexpander with bipartition $(A, B)$. Let $s\coloneqq 10^7\nu^{-8}$.
	Fix additional parameters such that 
	\[0<\frac{1}{n}\ll \frac{1}{M''}\ll\frac{1}{M'}\ll \varepsilon \ll \frac{1}{q}\ll \frac{1}{f} \ll d \ll \nu , \frac{1}{g}\ll \tau \ll \delta, \theta\ll 1\]
	and $\frac{q}{f}, \frac{f^2}{7}, \frac{50f^2}{s-1}\in \mathbb{N}$.
	Let $\ell'$ be the smallest even integer satisfying $\ell'\geq 324\nu^{-2}$. Denote $\ell^*\coloneqq f^2$ and observe that $\frac{\ell^*}{7},\frac{50\ell^*}{s-1}\in \mathbb{N}$ and $\frac{f}{\ell^*}=\frac{1}{f}\ll 1$.
	
	\begin{steps}
		\item \textbf{Constructing consistent bi-systems and bi-setups.}\label{step:bisetup}
		Apply \cref{lm:bisetup} to obtain $m, k, D_1, D_2, \cP_0, \cP, \cP', \cP^*, R_0, R, C_0, C, U$, and $U'$ such that the following hold for each $i\in [2]$.
		\begin{enumerate}[label=(\roman*)]
			\item $M'\leq k\leq M''$ and $\frac{k}{7}, \frac{k}{f}, \frac{k}{g}, \frac{m}{50}, \frac{m}{4\ell'}, \frac{fm}{q}, \frac{2fk}{3g(g-1)}\in \mathbb{N}$.
			\newcounter{bisetup}
			\setcounter{bisetup}{\value{enumi}}
			\item $(D_i, \cP_0, R_0, C_0, \cP, R, C)$ is a consistent $(\ell^*,2k, m, \varepsilon, d, \nu^4, 8\tau, \frac{\delta}{9}, \theta)$-bi-system.\label{thm:birobexp-bisys}
			\item $(D_i, \cP, \cP', \cP^*, R, C, U, U')$ is an $(\ell', \frac{q}{f}, 2k, m, \varepsilon, d)$-bi-setup.\label{thm:birobexp-bisetup}
		\end{enumerate}

		Fix additional constants $\varepsilon',r_0,r_0'$, and $r_1$ such that
		$\frac{1}{n}\ll \varepsilon' \ll \frac{r_0}{m}\ll \frac{r_0'}{m}\ll \frac{1}{M''}$ and $\frac{1}{f}\ll \frac{r_1}{m} \ll d$. Let
		\[r^*\coloneqq r_0'-(s-1)r_0, \quad r_2\coloneqq 192\ell'g^2kr^*, \quad  r_3\coloneqq \frac{2r^*fk}{q},\]
		\[r^\diamond\coloneqq r_1+r_2+r^*-(q-1)r_3, \quad s'\coloneqq 2r^*fk+7r^\diamond.\]
		Let $V_0$ denote the exceptional set in $\cP$.

		\item \textbf{Constructing the absorber.}\label{step:birobexp-A}
		We will use \cref{cor:absorber}. First, we construct the required special factors in $D_1$.
		
		Note that \cref{thm:birobexp-bisetup} and \cref{def:BST-P*} imply that $\cP^*$ is an $\varepsilon$-uniform $\frac{q}{f}$-refinement of $\cP$ with respect to $D_1$.
		Apply \cref{lm:SF} with $D_1, r_3, 2k, \nu^4, 8\tau, \frac{\delta}{9}, \cP^*$, and $\frac{q}{f}$ playing the roles of $D, r, k, \nu, \tau, \delta, \cP'$, and $\ell'$ to obtain
		\begin{enumerate}[label=(\alph*)]
			\item a set $\mathcal{SC}=\{SC_{i,h,j}\mid (i,h,j)\in [r_3]\times [\frac{q}{f}]\times [f]\}$ of $qr_3$ edge-disjoint special covers in $D_1\subseteq D$ with respect to $V_0$ and\label{thm:birobexp-SC}
			\item $r_3$ $(\frac{q}{f},f)$-special factors $SF_1, \dots, SF_{r_3}$ with respect to $\cP^*$ and $C$\label{thm:birobexp-SF}
		\end{enumerate}
		for which the following hold, where for each $(i,h,j)\in [r_3]\times [\frac{q}{f}]\times [f]$, $M_{i,h,j}$ denotes the complete special sequence associated to $SC_{i,h,j}$ and $SPS_{i,h,j}$ denotes the $(\frac{q}{f},f,h,j)$-special path system contained in $SF_i$.
		\begin{enumerate}[resume,label=(\alph*)]
			\item For each $(i,h,j)\in [r_3]\times [\frac{q}{f}]\times [f]$, we have $M_{i,h,j}\subseteq SPS_{i,h,j}\subseteq (D_1\setminus \mathcal{SC})\cup M_{i,h,j}$.\label{thm:birobexp-M}
			\item Let $(i,h,j),(i',h',j')\in [r_3]\times [\frac{q}{f}]\times [f]$ be distinct. Then, we have $(SPS_{i,h,j}\setminus M_{i,h,j})\cap (SPS_{i',h',j'}\setminus M_{i',h',j'})=\emptyset$.\label{thm:birobexp-disjoint}
		\end{enumerate}
		Define a multiset $\mathcal{M}$ by $\mathcal{M}\coloneqq \{M_{i,h,j}\mid (i,h,j)\in [r_3]\times [\frac{q}{f}]\times [f]\}$ and a multidigraph $\mathcal{SF}$ by $\mathcal{SF}\coloneqq SF_1\cup \dots \cup SF_{r_3}$. Let $D_1'\coloneqq D_1\setminus (\mathcal{SC}\cup \mathcal{SF})$. Observe that $\mathcal{SC}$ consists of $r_3q$ linear forests and $\mathcal{SF}$ consists of $r_3$ digraphs of maximum semidegree $1$. Thus, \cref{prop:CBSysedgesremoval} and \cref{thm:birobexp-bisys} imply that $(D_1', \cP_0, R_0, C_0, \cP, R, C)$ is a consistent $(\ell^*,2k, m, 3\sqrt{\varepsilon}, \frac{d}{2}, \frac{\nu^4}{2}, 8\tau, \frac{\delta}{18}, \frac{\theta}{2})$-bi-system.
		Note that $\cP$ is a $3\sqrt{\varepsilon}$-refinement of itself with respect to $D_1'$.
		Apply \cref{lm:SF} with $D_1', r^\diamond, 2k, 3\sqrt{\varepsilon}, \frac{d}{2}, \frac{\nu^4}{2}, 8\tau, \frac{\delta}{18}, \frac{\theta}{2}, \cP, 1$, and $7$ playing the roles of $D, r, k, \varepsilon, d, \nu, \tau, \delta, \theta, \cP', \ell'$, and $f$ to obtain
		\begin{enumerate}[label=(\alph*$'$)]
			\item a set $\mathcal{SC}'=\{SC_{i,h,j}'\mid (i,h,j)\in [r^\diamond]\times [1]\times [7]\}$ of $7r^\diamond$ edge-disjoint special covers in $D_1'\subseteq D$ with respect to $V_0$ and\label{thm:birobexp-SC'}
			\item $r^\diamond$ $(1,7)$-special factors $SF_1', \dots, SF_{r^\diamond}'$ with respect to $\cP$ and $C$\label{thm:birobexp-SF'}
		\end{enumerate}
		for which the following hold, where for each $(i,h,j)\in [r^\diamond]\times [1]\times [7]$, $M_{i,h,j}'$ denotes the complete special sequence associated to $SC_{i,h,j}'$ and $SPS_{i,h,j}'$ denotes the $(1,7,h,j)$-special path system contained in $SF_i'$.
		\begin{enumerate}[resume,label=(\alph*$'$)]
			\item For each $(i,h,j)\in [r^\diamond]\times [1]\times [7]$, we have $M_{i,h,j}'\subseteq SPS_{i,h,j}'\subseteq (D_1'\setminus \mathcal{SC}')\cup M_{i,h,j}'$.\label{thm:birobexp-M'}
			\item Let $(i,h,j),(i',h',j')\in [r^\diamond]\times [1]\times [7]$ be distinct. Then, we have $(SPS_{i,h,j}'\setminus M_{i,h,j}')\cap (SPS_{i',h',j'}'\setminus M_{i',h',j'}')=\emptyset$.\label{thm:birobexp-disjoint'}
		\end{enumerate}
		Define a multiset $\mathcal{M}'$ by $\mathcal{M}'\coloneqq \{M_{i,h,j}'\mid (i,h,j)\in [r^\diamond]\times [1]\times [7]\}$ and a multidigraph $\mathcal{SF}'$ by $\mathcal{SF}'\coloneqq SF_1'\cup \dots \cup SF_{r^\diamond}'$.
		
		By \cref{thm:birobexp-bisys,thm:birobexp-bisetup}, we can let $D^{\rm abs}$ be the absorber obtained by applying \cref{cor:absorber} with $D_2, r_0, r_0', 2k, \nu^4, 8\tau$, and $\frac{\delta}{9}$ playing the roles of $D, r, r', k, \nu, \tau$, and $\delta$.
		
		\item \textbf{Approximate decomposition.}\label{step:birobexp-approxdecomp}
		In this step, we approximately decompose $D'\coloneqq D\setminus (\mathcal{SC}\cup \mathcal{SC}'\cup \mathcal{SF}\cup \mathcal{SF}'\cup D^{\rm abs})$ into Hamilton cycles, with a sparse leftover. Let $r'\coloneqq r-r_0(s-1)-s'$.
		
		\begin{claim}\label{claim:birobexp}
			$D'$ is an $r'$-regular bipartite robust $(\frac{\nu}{2},\tau)$-outexpander with bipartition $(A,B)$.
		\end{claim}
	
		\begin{proofclaim}
			By \cref{cor:verticesedgesremovalbiproboutexp}, it is enough to show that $D'$ is $r'$-regular.
			By \cref{cor:absorber}\cref{cor:absorber-degrees}, each $v\in V(D)$ satisfies
			\begin{equation*}
				d_{D^{\rm abs}}^\pm(v)=
				\begin{cases}
					r_0(s-1) & \text{if }v\in V_0;\\
					r_0'+r_1+r_2+5r^\diamond & \text{otherwise}.
				\end{cases}
			\end{equation*}
			By \cref{thm:birobexp-SC} and \cref{thm:birobexp-SC'}, $\mathcal{SC}$ and $\mathcal{SC}'$ are edge-disjoint sets of edge-disjoint special covers in $D$ with respect to $V_0$ and so \cref{def:SC,def:MSC} imply that each $v\in V(D)$ satisfies
			\begin{equation*}
				d_{\mathcal{SC} \cup\mathcal{SC}'}^\pm(v)=d_{\mathcal{SC}}^\pm(v)+d_{\mathcal{SC}'}^\pm(v)=
				\begin{cases}
					qr_3+7r^\diamond & \text{if }v\in V_0;\\
					d_\cM^\pm(v)+d_{\cM'}^\pm(v) & \text{otherwise}.
				\end{cases}
			\end{equation*}
			By \cref{def:SF}, \cref{thm:birobexp-SF}, and \cref{thm:birobexp-SF'}, each $v\in V(D)$ satisfies
			\begin{equation*}
				d_{\mathcal{SF}}^\pm(v)+d_{\mathcal{SF}'}^\pm(v)=
				\begin{cases}
					0 & \text{if }v\in V_0;\\
					r_3+r^\diamond & \text{otherwise}.
				\end{cases}
			\end{equation*}
			By construction, $D^{\rm abs}$ and $\mathcal{SC}\cup \mathcal{SC}'$ are edge-disjoint. By \cref{thm:birobexp-M}, \cref{thm:birobexp-disjoint}, \cref{thm:birobexp-M'}, and \cref{thm:birobexp-disjoint'}, $\mathcal{SF}\cap D$ and $\mathcal{SF}'\cap D$ are edge-disjoint digraphs (rather than multidigraphs) and are both subdigraphs of $D\setminus (\mathcal{SC} \cup \mathcal{SC}'\cup D^{\rm abs})$.
			Thus, each $v\in V_0$ satisfies
			\begin{align*}
				d_{D'}^\pm(v)&=d_D^\pm(v)-d_{D^{\rm abs}}^\pm(v)-d_{\mathcal{SC}\cup\mathcal{SC}'}^\pm(v)-d_{\mathcal{SF}\cap D}^\pm(v)-d_{\mathcal{SF}'\cap D}^\pm(v)\\
				&=r-r_0(s-1)-(qr_3+7r^\diamond)-0-0=r'.
			\end{align*}
			Moreover, \cref{thm:birobexp-M} and \cref{thm:birobexp-M'} imply that each $v\in V(D)\setminus V_0$ satisfies
			\begin{align*}
				d_{D'}^\pm(v)&=d_D^\pm(v)-d_{D^{\rm abs}}^\pm(v)-d_{\mathcal{SC}\cup \mathcal{SC}'}^\pm(v)-d_{\mathcal{SF}\cap D}^\pm(v)-d_{\mathcal{SF}'\cap D}^\pm(v)\\
				&=r-d_{D^{\rm abs}}^\pm(v)-d_{\mathcal{SC}\cup \mathcal{SC}'}^\pm(v)-(d_{\mathcal{SF}}^\pm(v)-d_\cM^\pm(v))-(d_{\mathcal{SF}'}^\pm(v)-d_{\cM'}^\pm(v))\\
				&=r-(r_0'+r_1+r_2+5r^\diamond)-(r_3+r^\diamond)=r'.
			\end{align*}
			Thus, $D'$ is $r'$-regular, as desired.
		\end{proofclaim}
	
		Note that by \cref{cor:absorber}\cref{cor:absorber-decomp}, $D^{\rm abs}$ cannot absorb any edges which entirely lie in $V_0$. Thus, we start by covering the edges of $D'[V_0]$ with a small number Hamilton cycles as follows. 
		By \cref{thm:birobexp-bisetup} and \cref{def:BST-P}, \[\ell_0\coloneqq |V_0\cap A|=|V_0\cap B|\leq \varepsilon n\leq r'-2\nu n.\]
		Apply K\"{o}nig's theorem to decompose $D'[V_0]$ into $\ell_0$
		edge-disjoint matchings $M_1, \dots, M_{\ell_0}$ and observe that \cref{claim:birobexp} and \cref{cor:verticesedgesremovalbiproboutexp} imply that $D'\setminus \bigcup_{i\in [\ell_0]}M_i$ is an $(\frac{r'}{2n}, \varepsilon)$-almost regular bipartite robust $(\frac{\nu}{3}, \tau)$-outexpander with bipartition $(A,B)$.
		Note that \cref{cor:bipapproxHamdecomp}\cref{cor:bipapproxHamdecomp-deg,cor:bipapproxHamdecomp-size} hold with $M_1, \dots, M_{\ell_0}$ playing the roles of $F_1, \dots, F_\ell$.		
		Apply \cref{cor:bipapproxHamdecomp} with $D'\setminus \bigcup_{i\in [\ell_0]}M_i, \ell_0, \frac{r'}{2n}, \nu, \frac{\nu}{3}$, and $M_1, \dots, M_{\ell_0}$ playing the roles of $D, \ell, \delta, \eta, \nu$, and $F_1, \dots, F_\ell$ to obtain a set $\sC_0$ of $\ell_0$ edge-disjoint Hamilton cycles of $D'$ such that $D'[V_0]\subseteq E(\sC_0)$ and $D''\coloneqq D'\setminus E(\sC_0)$ is a bipartite robust $(\frac{\nu}{6},\tau)$-outexpander with bipartition $(A, B)$.
		Note that $D''$ is regular of degree $r''\coloneqq r'-\ell_0$.
		
		We now approximately decompose $D''$ as follows.
		For each $i\in [r''-r_0]$, let $F_i$ be the empty digraph (so \cref{cor:bipapproxHamdecomp}\cref{cor:bipapproxHamdecomp-deg,cor:bipapproxHamdecomp-size} are satisfied).
		Apply \cref{cor:bipapproxHamdecomp} with $D'', r''-r_0, \frac{r''}{2n}, \varepsilon', \frac{r_0}{2n}$, and $\frac{\nu}{6}$ playing the roles of $D, \ell, \delta, \varepsilon, \eta$, and $\nu$ to obtain a set $\sC_{\rm approx}$ of $r''-r_0$ edge-disjoint Hamilton cycles of $D''$.

		\item \textbf{Absorbing the leftovers.}\label{step:birobexp-leftovers}
		Finally, we decompose $H\coloneqq D''\setminus \sC_{\rm approx}=D'\setminus (\sC_0\cup \sC_{\rm approx})$ using the absorber $D^{\rm abs}$ constructed in \cref{step:birobexp-A}. Note that $H$ is an $r_0$-regular subdigraph of $D'\setminus D'[V_0]$ (by \cref{claim:birobexp}, $H$ is obtained from the $r'$-regular digraph $D'$ by removing the edges of $r'-r_0$ edge-disjoint Hamilton cycles of $D'$).
		Define a multidigraph $D'''$ by $D'''\coloneqq H\cup D^{\rm abs}\cup \mathcal{SF}\cup \mathcal{SF}'$.
		
		\begin{claim}\label{claim:D5}
			$D'''\setminus (\cM\cup \cM')$ is a digraph (rather than a multidigraph) and satisfies $D'''\setminus (\cM\cup \cM')= D\setminus (\mathcal{SC}\cup \mathcal{SC}'\cup \sC_0\cup \sC_{\rm approx})$.
		\end{claim}
	
		\begin{proofclaim}
			By \cref{thm:birobexp-M}, \cref{thm:birobexp-disjoint}, \cref{thm:birobexp-M'}, and \cref{thm:birobexp-disjoint'}, $\mathcal{SF}\setminus \cM$ and $\mathcal{SF}'\setminus \cM'$ are digraphs rather than multidigraphs and are edge-disjoint subdigraphs of $D\setminus (\mathcal{SC}\cup \mathcal{SC}')$.
			By \cref{step:birobexp-A}, $D^{\rm abs}\subseteq D\setminus (\mathcal{SC}\cup \mathcal{SF}\cup \mathcal{SC}'\cup \mathcal{SF}')$ and, by definition,
			\begin{equation}\label{eq:birobexp-H}
				H=D''\setminus \sC_{\rm approx}=D\setminus (\mathcal{SC}\cup \mathcal{SC}'\cup \mathcal{SF}\cup \mathcal{SF}'\cup D^{\rm abs}\cup \sC_0\cup \sC_{\rm approx}).
			\end{equation}
			Thus, $\mathcal{SF}\setminus \cM$, $\mathcal{SF}'\setminus \cM'$, $D^{\rm abs}$, and $H$ are all pairwise edge-disjoint subdigraphs of $D$. Therefore, $D'''\setminus (\cM\cup \cM')$ is a digraph.
			Moreover, recall from \cref{step:birobexp-approxdecomp} that $\sC_0\cup \sC_{\rm approx}\subseteq D'\subseteq D\setminus (\mathcal{SF}\cup \mathcal{SF}'\cup D^{\rm abs})$. Thus, \cref{eq:birobexp-H} implies that $D'''\setminus (\cM\cup \cM')= D\setminus (\mathcal{SC}\cup \mathcal{SC}'\cup \sC_0\cup \sC_{\rm approx})$, as desired.
		\end{proofclaim}
		
		Let $\sC\cup \sC'$ be the decomposition of $D'''$ obtained by applying \cref{cor:absorber}\cref{cor:absorber-decomp} (with $r_0$ and $r_0'$ playing the roles of $r$ and $r'$).
		By \cref{cor:absorber}\cref{cor:absorber-decomp-robustdecomp}, $\sC'$ is a set of $s'$ Hamilton cycles on $V(D)\setminus V_0$, each containing precisely one of the special path systems in the multidigraph $\mathcal{SF}\cup \mathcal{SF}'$. That is, there exists an enumeration
		\[\{C_{i,h,j}\mid (i,h,j)\in [r_3]\times [\tfrac{q}{f}]\times [f]\}\cup \{C_{i,h,j}'\mid (i,h,j)\in [r^\diamond]\times [1]\times [7]\}\]
		of $\sC'$ such that $C_{i,h,j}\cap (\mathcal{SF}\cup \mathcal{SF}')=SPS_{i,h,j}$ for each $(i,h,j)\in [r_3]\times [\frac{q}{f}]\times [f]$ and $C_{i',h',j'}'\cap (\mathcal{SF}\cup \mathcal{SF}')=SPS_{i',h',j'}'$ for each $(i',h',j')\in [r^\diamond]\times [1]\times [7]$.
		
		For each $(i,h,j)\in [r_3]\times [\frac{q}{f}]\times [f]$, define $C_{i,h,j}^*\coloneqq (C_{i,h,j}\setminus M_{i,h,j})\cup SC_{i,h,j}$ and note that \cref{thm:birobexp-M} and \cref{fact:SC} imply that  $C_{i,h,j}^*$ is a Hamilton cycle of $D$.
		For each $(i,h,j)\in [r^\diamond]\times [1]\times [7]$, define $C_{i,h,j}''\coloneqq (C_{i,h,j}'\setminus M_{i,h,j}')\cup SC_{i,h,j}'$ and note that \cref{thm:birobexp-M'} and \cref{fact:SC} imply that $C_{i,h,j}''$ is a Hamilton cycle of $D$.
		Let \[\sC''\coloneqq \{C_{i,h,j}^*\mid (i,h,j)\in [r_3]\times [\tfrac{q}{f}]\times [f]\}\cup \{C_{i,h,j}''\mid (i,h,j)\in [r^\diamond]\times [1]\times [7]\}.\]
		By \cref{thm:birobexp-SC}, \cref{thm:birobexp-SC'}, and \cref{step:birobexp-approxdecomp}, $\mathcal{SC}, \mathcal{SC}'$, and $\sC_0\cup \sC_{\rm approx}$ are all pairwise edge-disjoint. Thus, \cref{claim:D5} implies that $\sC''$ is a Hamilton decomposition of \[(D'''\setminus (\cM\cup \cM'\cup \sC))\cup (\mathcal{SC}\cup \mathcal{SC}')=D\setminus (\sC_0\cup \sC_{\rm approx}\cup \sC).\] By \cref{step:birobexp-approxdecomp}, $\sC_0\cup \sC_{\rm approx}$ is a set of edge-disjoint Hamilton cycles of $D$ and, by \cref{cor:absorber}\cref{cor:absorber-decomp-preprocessing} and \cref{claim:D5}, $\sC$ is a set of edge-disjoint Hamilton cycles of $H\cup D^{\rm abs}\subseteq D\setminus (\sC_0\cup \sC_{\rm approx})$.		
		Thus, $\sC_0\cup \sC_{\rm approx}\cup \sC\cup \sC''$ is a Hamilton decomposition of $D$.		
		This completes the proof of \cref{thm:biprobexp}.\qedhere
	\end{steps}
\end{proof}

\onlyinsubfile{\bibliographystyle{abbrv}
\bibliography{Bibliography/Bibliography}}

\section{Blow-up cycles: definitions and proof of Lemma \ref{lm:twocases}}\label{sec:blowups}

	\onlyinsubfile{
		\setcounter{section}{3}
\section{Blow-up cycles: notation and proof of Lemma \ref{lm:twocases}}}

The remainder of this paper is devoted to the proofs of \cref{lm:twocases,thm:blowupC4}. In this section, we recall and expand on the definitions of the complete blow-up $C_4$ and a digraph which is $\varepsilon$-close to the complete blow-up $C_4$. We also state a few properties of blow-up cycles and prove \cref{lm:twocases}. 

\subsection{Blow-up cycles}\label{sec:blowupdef}
We now generalise the concept of complete blow-up $C_4$ introduced in \cref{sec:intro-blowupC4}.
Let $K\geq 3$. The \emph{complete blow-up $C_K$ on vertex classes of size $n$} is the $n$-fold blow-up of the consistently directed $C_K$. 
Any spanning subdigraph of the $n$-fold blow-up of the directed $C_K$ is called a \emph{blow-up $C_K$ on vertex classes of size $n$}.
Recall from \cref{sec:intro-blowupC4} that a digraph is \emph{$\varepsilon$-close to the complete blow-up $C_4$ on vertex classes of size $n$} if it can be obtained from the complete blow-up $C_4$ on vertex classes of size $n$ by flipping the direction of at most $4\varepsilon^2n$ edges.

Let $K\geq 3$. It will sometimes be convenient to label the vertex classes of the (complete) blow-up $C_K$. This motivates the following variants of the above definitions.
Let $U_1, \dots, U_K$ be disjoint vertex sets of size $n$ and let $\cU\coloneqq (U_1, \dots, U_K)$.
The \emph{complete blow-up $C_K$ with vertex partition $\cU$} is the digraph $D$ on $\bigcup\cU= \bigcup_{i\in [K]}U_i$ defined by $E(D)\coloneqq \{uv\mid u\in U_i, v\in U_{i+1}, i\in [K]\}$ (where $U_{K+1}\coloneqq U_1$). Note that the ordering of $U_1, \dots, U_K$ matters. The vertex sets $U_1, \dots, U_K$ are the \emph{vertex classes} of $D$. In informal discussions, we sometime refer to $(U_1,U_2), \dots, (U_K, U_1)$ as the \emph{pairs of the blow-up $C_K$}.
(In other words, $D$ is the $n$-fold blow-up of the directed $C_K$ whose vertex classes are denoted by $U_1, \dots, U_K$ and ordered according to the ordering of the vertices in the directed $C_K$.)
Let $D'$ be a digraph on $\bigcup\cU$.
We say that $D'$ is a \emph{blow-up $C_K$ with vertex partition $\cU$} if $D'$ is a spanning subdigraph of $D$.
If $K=4$, we say that $D'$ is \emph{$\varepsilon$-close to the complete blow-up $C_4$ with vertex partition $\cU$} if $D'$ can be obtained from $D$ by flipping the direction of at most $4\varepsilon n^2$ edges of $D$.

\begin{definition}[\gls*{epsilon partition}]\label{def:epspartition}
    Let $U_1, \dots, U_4$ be disjoint vertex sets of size $n$ and denote $\cU\coloneqq (U_1, \dots, U_4)$. Let $D$ be a digraph on $\bigcup\cU$. We say that $\cU$ is an \emph{$(\varepsilon,4)$-partition for $D$} if $D$ is $\varepsilon$-close to the complete blow-up $C_4$ with vertex partition $\cU$.
\end{definition}

Note that while the ordering of $U_1, \dots, U_K$ in the above definitions matters, this ordering can be shifted without effect. We also emphasise that, in the above definitions, we assume that the vertex classes $U_1, \dots, U_K$ are equally sized.

\begin{fact}\label{fact:partition}
	Let $D$ be a digraph on $4n$ vertices and suppose that $\cU=(U_1, \dots, U_4)$ is an $(\varepsilon,4)$-partition for $D$. Then, the following hold.
	\begin{enumerate}
		\item $|U_1|=\dots=|U_4|=n$.\label{fact:partition-size}
		\item For each $i\in [4]$, $(U_i, \dots, U_{i+3})$ is an $(\varepsilon,4)$-partition for $D$.\label{fact:partition-cycle}
	\end{enumerate}
\end{fact}

Throughout this paper, when we work with a vertex partition $\cU=(U_1, \dots, U_K)$, the subscripts in $U_1, \dots, U_K$ are always taken modulo $K$ (so $U_{K+1}\coloneqq U_1$ for example).

\subsection{Forward and backward edges}\label{sec:forwardbackwarddef}

Let $U_1, \dots, U_K$ be disjoint vertex sets (not necessarily of the same size) and denote $\cU\coloneqq (U_1, \dots, U_K)$.
Let $u,v\in \bigcup\cU$ be distinct.
We say that $uv$ is a \emph{forward edge (with respect to $\cU$)} if there exists $i\in [K]$ such that $u\in U_i$ and $v\in U_{i+1}$. 
We say that $uv$ is a \emph{backward edge (with respect to $\cU$)} if there exists $i\in [K]$ such that $u\in U_{i+1}$ and $v\in U_i$.
Let $D$ be a digraph on $\bigcup\cU$.
We denote by $\overrightarrow{D}_\cU$ the subdigraph of $D$ induced by the forward edges of $D$ with respect to $\cU$ and by $\overleftarrow{D}_\cU$ the subdigraph of $D$ induced by the backward edges of $D$ with respect to $\cU$.
Let $v\in V(D)$. 
The \emph{forward (in/out)degree of $v$ in $D$ (with respect to $\cU$)} is the (in/out)degree of $v$ in $\overrightarrow{D}_\cU$ and the \emph{backward (in/out)degree of $v$ in $D$ (with respect to $\cU$)} is the (in/out)degree of $v$ in $\overleftarrow{D}_\cU$. 
These are denoted by
\[\overrightarrow{d}_{D,\cU}(v)\coloneqq d_{\overrightarrow{D}_\cU}(v), \quad
\overrightarrow{d}_{D,\cU}^\pm(v)\coloneqq d_{\overrightarrow{D}_\cU}^\pm(v),\quad
\overleftarrow{d}_{D,\cU}(v)\coloneqq d_{\overleftarrow{D}_\cU}(v), \quad \overleftarrow{d}_{D,\cU}^\pm(v)\coloneqq d_{\overleftarrow{D}_\cU}^\pm(v).\]
Similarly, the \emph{forward (in/out)neighbourhood of $v$ in $D$ (with respect to $\cU$)} is the (in/out)neigh\-bourhood of $v$ in $\overrightarrow{D}_\cU$ and the \emph{backward (in/out)neighbourhood of $v$ in $D$ (with respect to $\cU$)} is the (in/out)neighbourhood of $v$ in $\overleftarrow{D}_\cU$.
These are denoted by
\[\overrightarrow{N}_{D,\cU}(v)\coloneqq N_{\overrightarrow{D}_\cU}(v), \quad
\overrightarrow{N}_{D,\cU}^\pm(v)\coloneqq N_{\overrightarrow{D}_\cU}^\pm(v), \quad  
\overleftarrow{N}_{D,\cU}(v)\coloneqq N_{\overleftarrow{D}_\cU}(v), \quad
\overleftarrow{N}_{D,\cU}^\pm(v)\coloneqq N_{\overleftarrow{D}_\cU}^\pm(v).\]

\subsection{Regular bipartite tournaments}\label{sec:regbiT}

We now make a few observations about regular bipartite tournaments. Let $U_1, \dots, U_4$ be disjoint vertex sets of size $n$ and denote $\cU\coloneqq (U_1, \dots, U_4)$. Let $T$ be a bipartite tournament and suppose that $\cU$ is an $(\varepsilon, 4)$-partition for $T$. Then, it is easy to see that $T$ is a bipartite tournament on vertex classes $U_1\cup U_3$ and $U_2\cup U_4$. Moreover, note that the complete blow-up $C_4$ with vertex partition $\cU$ is a regular digraph. Thus, $T$ must be obtained by changing, for each $v\in \bigcup\cU$, the direction of the same number of in- and outedges incident to $v$. 

\begin{fact}\label{fact:backwarddegree}
	Let $U_1, \dots, U_4$ be disjoint vertex sets and $\cU\coloneqq (U_1, \dots, U_4)$.
	Let $T$ be a regular bipartite tournament on vertex classes $U_1\cup U_3$ and $U_2\cup U_4$. Then, each $v\in V(T)$ satisfies \[\overrightarrow{d}_{T,\cU}^+(v)=\overrightarrow{d}_{T,\cU}^-(v) \quad \text{and} \quad \overleftarrow{d}_{T,\cU}^+(v)=\overleftarrow{d}_{T,\cU}^-(v).\]
\end{fact}

In particular, this implies that the number of forward/backward edges is the same in each pair of the blow-up $C_4$.

\begin{fact}\label{fact:backwardedges}
	Let $U_1, \dots, U_4$ be disjoint vertex sets and denote $\cU\coloneqq (U_1, \dots, U_4)$.
	Let $T$ be a regular bipartite tournament on vertex classes $U_1\cup U_3$ and $U_2\cup U_4$. Then,
	\[e_T(U_1,U_2)=e_T(U_2,U_3)=e_T(U_3,U_4)=e_T(U_4,U_1)\] and \[e_T(U_1, U_4)=e_T(U_4, U_3)=e_T(U_3, U_2)=e_T(U_2, U_1).\]
\end{fact}

Thus, we may use the following alternative definition of an $(\varepsilon,4)$-partition.

\begin{fact}\label{fact:epsilon4partition}
	Let $U_1, \dots, U_4$ be disjoint vertex sets and $\cU\coloneqq (U_1, \dots, U_4)$.
	Let $T$ be a regular bipartite tournament on vertex classes $U_1\cup U_3$ and $U_2\cup U_4$. Then, $\cU$ is an $(\varepsilon,4)$-partition for $T$ if and only if $e_T(U_i, U_{i-1})\leq \varepsilon n^2$ for each $i\in [4]$.
\end{fact}

\subsection{Proof of Lemma \ref{lm:twocases}}\label{sec:twocases}
We are now ready to prove \cref{lm:twocases}, which states that if a regular bipartite tournament $T$ is not a bipartite robust outexpander, then $T$ is close to the complete blow-up $C_4$.

\begin{proof}[Proof of \cref{lm:twocases}]
	Let $0<\frac{1}{n_0}\ll \nu'\leq \nu\ll \tau$ and let $T$ be a regular bipartite tournament on vertex classes $A$ and $B$ of size $2n\geq n_0$. Note that $T$ is $n$-regular.
	Suppose that $T$ is not a bipartite robust $(\nu',\tau)$-outexpander with bipartition $(A,B)$.  We show that $T$ is $\sqrt{\nu'}$-close to the complete blow-up $C_4$ on vertex classes of size $n$.
		
	We may assume without loss of generality that there exists $A'\subseteq A$ satisfying $2\tau n\leq |A'|\leq 2(1-\tau)n$ for which 
	\begin{equation}\label{eq:A'}
		|RN_{\nu', T}^+(A')|<|A'|+2\nu' n.
	\end{equation}
	Denote $B'\coloneqq RN_{\nu', T}^+(A')$.
	By definition of a bipartite robust outexpander, we have
	\begin{equation}\label{eq:A'BB'}
		e_T(A', B\setminus B')< 2\nu' n|B\setminus B'|\leq 4\nu' n^2.
	\end{equation}
	Thus,
	\begin{align*}
		|A'||B'|\geq e_T(A', B')&= e_T(A', B)-e_T(A', B\setminus B')
		\stackrel{\text{\cref{eq:A'BB'}}}{\geq} n|A'|-4\nu' n^2
		\geq \left(1-\frac{2\nu'}{\tau}\right)n|A'|.
	\end{align*}
	Therefore, 
	\begin{equation}\label{eq:B'>}
		|B'|\geq \left(1-\frac{2\nu'}{\tau}\right)n
	\end{equation}
	and so 
	\begin{equation}\label{eq:A'>}
		|A'|\stackrel{\text{\cref{eq:A'},\cref{eq:B'>}}}{\geq} \left(1-\frac{3\sqrt{\nu'}}{4}\right)n.
	\end{equation}
	Moreover,
	\begin{equation*}
		n|B\setminus B'|\geq e_T(B\setminus B',A')= |B\setminus B'||A'|-e_T(A', B\setminus B')\stackrel{\text{\cref{eq:A'BB'}}}{\geq} |B\setminus B'|(|A'|-2\nu' n).
	\end{equation*}
	Therefore, 
	\begin{equation}\label{eq:A'<}
		|A'|\leq (1+2\nu')n
	\end{equation}
	and so
	\begin{equation}\label{eq:B'<}
		|B'|\stackrel{\text{\cref{eq:A'},\cref{eq:A'<}}}{\leq} (1+4\nu')n.
	\end{equation}	
	Let $U_1\cup U_3$ be a partition of $A$ such that $|U_1|=n=|U_3|$ and $|U_1\bigtriangleup A'|=|n-|A'||$. Similarly, let $U_2\cup U_4$ be a partition of $B$ such that $|U_2|=n=|U_4|$ and $|U_2\bigtriangleup B'|=|n-|B'||$.
	Note that
	\begin{equation}\label{eq:U}
		|U_1\bigtriangleup A'|\stackrel{\text{\cref{eq:A'>},\cref{eq:A'<}}}{\leq} \frac{3\sqrt{\nu'}n}{4} \quad \text{and}\quad |U_2\bigtriangleup B'|\stackrel{\text{\cref{eq:B'>},\cref{eq:B'<}}}{\leq} \frac{3\sqrt{\nu'}n}{4}.
	\end{equation}
	By \cref{fact:backwardedges,fact:epsilon4partition}, it is enough to show that $e_T(U_1,U_4)\leq 4\sqrt{\nu'}n^2$.
	We have
	\begin{align*}
		e_T(U_1,U_4)&\stackrel{\text{\eqmakebox[twocases]{}}}{\leq} e_T(A'\cap U_1, U_4\setminus B')+e_T(U_1\setminus A',U_4)+e_T(U_1,B'\cap U_4)\\
		&\stackrel{\text{\eqmakebox[twocases]{\text{\cref{eq:A'BB'},\cref{eq:U}}}}}{\leq} 4\nu' n^2+ \frac{6\sqrt{\nu'}n^2}{4}+ \frac{6\sqrt{\nu'}n^2}{4}
		\leq 4\sqrt{\nu'} n^2,
	\end{align*}
	as desired.
\end{proof}

\section{A robust decomposition lemma for blow-up cycles}\label{sec:cyclerobustdecomp}

	\onlyinsubfile{
		\setcounter{section}{7}
		\section{A robust decomposition lemma for blow-up cycles}}
	
In this section, we introduce a robust decomposition lemma for blow-up cycles. This result will be used in the proof of \cref{thm:blowupC4} to decompose the edges leftover after the approximate decomposition. (See \cref{sec:sketch} for a proof overview of \cref{thm:blowupC4}.)

First, observe that the standard robust decomposition lemma (\cref{lm:newrobustdecomp}) cannot be directly applied when $T$ is ($\varepsilon$-close to) the complete blow-up $C_4$ because we cannot find a setup or a bi-setup: since (almost) all the edges lie along a blow-up cycle, we cannot find the necessary chord edges to form a universal or bi-universal walk.
Thus, we will need to derive an analogue of \cref{lm:newrobustdecomp} which holds for blow-up $C_4$'s.

\subsection{Aim and strategy}\label{sec:cyclerobustdecomp-sketch}

Let $D$ be a blow-up $C_4$ with vertex partition $\cU=(U_1, \dots, U_4)$. (In the proof of \cref{thm:blowupC4}, $D$ will be (a subdigraph of) $\overrightarrow{T}_\cU$, that is, $D$ will consist of (some of) the forward edges of $T$.) We want to find an absorber $D^{\rm rob}\subseteq D$ such that for any sparse regular leftover $H\subseteq D\setminus D^{\rm rob}$, the digraph $H\cup D^{\rm rob}$ has a Hamilton decomposition.

Roughly speaking, the overall strategy is as follows. Recall the notion of matching contractions from \cref{sec:contractingM}. For each $i\in [4]$ in turn, we apply the standard robust decomposition lemma in a suitable auxiliary ``contracted" digraph corresponding to the pair $(U_i, U_{i+1})$ of the blow-up $C_4$. This enables us to decompose small leftovers into suitable auxiliary ``contracted" Hamilton cycles spanning $U_i$.
These are then ``expanded" into full Hamilton cycles of $D$.
(See also \cref{sec:sketch-completeblowupC4} for an informal discussion about how to construct a Hamilton cycle in a blow-up $C_4$.)

We now explain this strategy in more detail. Note that it is enough to consider each pair $(U_i, U_{i+1})$ of the blow-up $C_4$ in turn. Indeed, suppose that for each $i\in [4]$, we have constructed an absorber $D_i^{\rm rob}\subseteq D$ such that for any sparse leftover $H_i\subseteq D(U_i,U_{i+1})\setminus D_i^{\rm rob}$, the digraph $H_i\cup D_i^{\rm rob}$ has a decomposition into Hamilton cycles of $D$. Let $D^{\rm rob}\coloneqq \bigcup_{i\in [4]}D_i^{\rm rob}$. Then, for any sparse leftover $H\subseteq D\setminus D^{\rm rob}$, we can use each $D_i^{\rm rob}$ in turn to decompose the edges of $H(U_i, U_{i+1})$. Altogether, this induces a Hamilton decomposition of $H\cup D^{\rm rob}$ (recall that $D$, and so $H$, only contains edges which lie in one of the pairs $(U_i, U_{i+1})$).

Let $i\in [4]$. We now explain our strategy for constructing the absorber $D_i^{\rm rob}$. 
First, as mentioned above, observe that the problem of constructing Hamilton cycles of $D$ can be reduced to constructing Hamilton cycles on $U_i\cup U_{i+1}$.
Then, the following holds.

\begin{fact}\label{fact:cyclerobustdecomp-sketch}
	Fix an auxiliary perfect matching $M_i$ from $U_{i+1}$ to $U_i$. Let $M$ be a perfect matching from $U_i$ to $U_{i+1}$ and suppose that $M\cup M_i$ forms a Hamilton cycle on $U_i\cup U_{i+1}$. Let $\sP$ be a spanning set of vertex-disjoint paths on $\bigcup\cU$ which consists of a $(u,v)$-path for each $uv\in M_i$. Then, $M\cup \sP$ forms a Hamilton cycle on~$\bigcup\cU$.
\end{fact}

In our robust decomposition lemma for blow-up cycles, we will input such spanning sets of vertex-disjoint paths (these will be incorporated into the special factors). Thus, we have reduced the original problem to that of finding an absorber $D_i^{\rm rob}$ such that the following holds: for any sparse leftover $H_i\subseteq D(U_i, U_{i+1})\setminus D_i^{\rm rob}$, the digraph $H_i\cup D_i^{\rm rob}$ has a decomposition into perfect matchings from $U_i$ to $U_{i+1}$, each of which forms a Hamilton cycle on $U_i\cup U_{i+1}$ with a fixed auxiliary matching $M_i$.

We now discuss the construction of $D_i^{\rm rob}$. We have already discussed (e.g.\ in \cref{sec:contractingM,sec:sketch-cycle,sec:approxdecomp}) that one can construct Hamilton cycles which contain a prescribed perfect matching by considering contracted digraphs. More precisely, fix an auxiliary perfect matching $M_i$ from $U_{i+1}$ to $U_i$ and let $\tD_i$ be the $M_i$-contraction of $D[U_i, U_{i+1}]$ (recall \cref{def:contractexpand}\cref{def:contract}). Then, as seen in \cref{fact:contractingHamcycle}, a Hamilton cycle in $\tD_i$ induces a perfect matching from $U_i$ to $U_{i+1}$ in $D$ which forms a Hamilton cycle on $U_i\cup U_{i+1}$ with $M_i$.
Thus, we can let $D_i^{\rm rob}$ be the $M_i$-expansion of the absorber $\tD_i^{\rm rob}$ obtained by applying the robust decomposition lemma in $\tD_i$.
Indeed, suppose that $H_i\subseteq D(U_i, U_{i+1})\setminus D_i^{\rm rob}$ is a sparse leftover. Denote by $\tH_i$ the $M_i$-contraction of $H_i$. Then, \cref{lm:newrobustdecomp} implies that $\tH_i\cup \tD_i^{\rm rob}$ has a decomposition into Hamilton cycles on $U_i$. By \cref{fact:contractingHamcycle}, this induces a decomposition of $H_i\cup D_i^{\rm rob}$ into perfect matchings from $U_i$ to $U_{i+1}$ which form Hamilton cycles on $U_i\cup U_{i+1}$ with $M_i$, as desired.

Note that since we consider each pair $(U_i, U_{i+1})$ of the blow-up $C_4$ in turn, our methods hold for more general blow-up cycles of any length. Thus, we write the rest of this section for general blow-up $C_K$'s for possible future use. In this paper, we will only need the case $K=4$ (to prove \cref{thm:blowupC4}).

\subsection{Definitions}\label{sec:cyclerobustdecomp-def}
First, we introduce the cycle analogues of setups, special path systems, and special factors (which were defined in \cref{sec:ST,sec:SPS}).

\subsubsection{Cycle-setups}
Let $D$ be a blow-up $C_K$ with vertex partition $\cU=(U_1, \dots, U_K)$. For each $i\in [K]$, let $M_i$ be an auxiliary perfect matching from $U_{i+1}$ to $U_i$ and let $\tD_i$ be the $M_i$-contraction of $D[U_i, U_{i+1}]$. 
As discussed in \cref{sec:cyclerobustdecomp-sketch}, we aim to apply the standard robust decomposition lemma (\cref{lm:newrobustdecomp}) in each $\tD_i$ in turn and so we will need a setup in each $\tD_i$. This motivates the next definition:
roughly speaking, a cycle-setup consists of $K$ setups, one in each $\tD_i$. 

\begin{definition}[\Gls*{cycle-setup}]\label{def:CST}
	$(D, \cU, \cP, \cP', \cP^*, \cR, \cC, \sU, \sU', \cM)$ is a \emph{$(K,\ell',\ell^*,k, m, \varepsilon,d)$-cycle-setup} if the following properties are satisfied.
	\begin{enumerate}[label=\rm(CST\arabic*),longlabel]
	\item $D$ is a blow-up $C_K$ with vertex partition $\cU=(U_1, \dots, U_K)$. \label{def:CST-D}
	\item $\cM=(M_1, \dots, M_K)$ where, for each $i\in [K]$, $M_i$ is an auxiliary directed perfect matching from $U_{i+1}$ to $U_i$.\label{def:CST-M}
	\item $\cP=(\cP_1, \dots, \cP_K)$, $\cP'=(\cP_1', \dots, \cP_K')$, $\cP^*=(\cP_1^*, \dots, \cP_K^*)$, $\cC=(C^1, \dots, C^K)$, $\cR=(R_1, \dots, R_K)$, $\sU=(U^1, \dots, U^K)$, and $\sU'=(U'^1, \dots, U'^K)$ are such that the following holds for each $i\in [K]$. Let $\tD_i$ be the $M_i$-contraction of $D[U_i, U_{i+1}]$. Then,
	$(\tD_i, \cP_i, \cP_i', \cP_i^*, R_i, C^i, U^i, U'^i)$ is an $(\ell', \ell^*, k, m, \varepsilon, d)$-setup with an empty exceptional set.\label{def:CST-ST}
	\end{enumerate}
\end{definition}

Whenever $(D, \cU, \cP, \cP', \cP^*, \cR, \cC, \sU, \sU', \cM)$ is a $(K,\ell',\ell^*,k, m, \varepsilon,d)$-cycle-setup, we implicitly use the notation $\cU=(U_1, \dots, U_K)$, $\cP=(\cP_1, \dots, \cP_K)$, $\cP'=(\cP_1', \dots, \cP_K')$, $\cP^*=(\cP_1^*, \dots, \cP_K^*)$, $\cR=(R_1, \dots, R_K)$, $\cC=(C^1, \dots, C^K)$, $\sU=(U^1, \dots, U^K)$, $\sU'=(U'^1, \dots, U'^K)$, and $\cM=(M_1, \dots, M_K)$.

\begin{definition}[\Gls*{cycle-framework}]\label{def:CF}
	To avoid repetitions, we say that $(\cU, \cP, \cP^*,\cC, \cM)$ is a \emph{$(K,\ell^*,k,n)$-cycle-framework} if $\cU= (U_1, \dots, U_K)$, $\cP=(\cP_1, \dots, \cP_K)$, $\cP^*=(\cP_1^*, \dots, \cP_K^*)$, $\cC= (C^1, \dots, C^K)$, and $\cM= (M_1, \dots, M_K)$ satisfy the following properties for each $i\in [K]$.
	\begin{enumerate}[label=\rm(CF\arabic*),longlabel]
		\item $U_i$ is a vertex set of size $n$ which is disjoint from the other sets in $\cU$.\label{def:CF-U}
		\item $\cP_i$ is a partition of $U_i$ into an empty exceptional set and $k$ clusters of size $\frac{n}{k}$.\label{def:CF-P}
		\item $\cP_i^*$ is an $\ell^*$-refinement of $\cP_i$.\label{def:CF-P*}
		\item $C^i$ is a Hamilton cycle on the clusters in $\cP_i$.\label{def:CF-C}
		\item $M_i$ is an auxiliary perfect matching from $U_{i+1}$ to $U_i$.\label{def:CF-M}
	\end{enumerate}
\end{definition}

Whenever we say that $(\cU, \cP, \cP^*,\cC, \cM)$ is a $(K,\ell^*,k,n)$-cycle-framework, we implicitly use the notation $\cU= (U_1, \dots, U_K)$, $\cP=(\cP_1, \dots, \cP_K)$, $\cP^*=(\cP_1^*, \dots, \cP_K^*)$, $\cC= (C^1, \dots, C^K)$, and $\cM= (M_1, \dots, M_K)$.

One can easily verify that a cycle-setup induces a cycle-framework.

\begin{fact}\label{fact:CSTCF}
	Let $(D, \cU, \cP, \cP', \cP^*, \cR, \cC, \sU, \sU', \cM)$ be a $(K,\ell',\ell^*,k, m, \varepsilon,d)$-cycle-setup. Then, $(\cU, \cP, \cP^*,\cC, \cM)$ is a $(K,\ell^*,k,n)$-cycle-framework where $n\coloneqq |U_1|$.
\end{fact}

\COMMENT{\begin{proof}
		First, \cref{def:CF-U,def:CF-M} follow from \cref{def:CST-D,def:CST-M}, respectively. By \cref{def:CST-ST}, \cref{def:CF-P,def:CF-P*} follow from \cref{def:ST-P,def:ST-P*}, respectively, while \cref{def:CF-C} follows from \cref{def:ST-C,def:ST-R}.
\end{proof}}

Any partition is a $1$-refinement of itself, so the following holds.

\begin{fact}\label{fact:CFP}
	Let $(\cU, \cP, \cP^*,\cC, \cM)$ be a $(K,\ell^*,k,n)$-cycle-framework. Then, $(\cU, \cP, \cP,\cC, \cM)$ is a $(K,1,k,n)$-cycle-framework.
\end{fact}

Recall from \cref{prop:bisetupedgesremoval} that a setup remains a setup (with slightly worse parameters) after removing a few edges incident to each vertex. Using similar arguments, one can show that the analogue holds for a cycle-setup.

\begin{prop}\label{prop:CST}
	Let $0<\frac{1}{m}\ll\frac{1}{k}, \varepsilon\leq \varepsilon'\ll d\ll \frac{1}{\ell'}\ll 1$ and $\varepsilon'\ll \frac{1}{\ell^*}$.
	Let $D$ be a digraph and let $D'$ be obtained from $D$ by removing at most $\varepsilon'm$ inedges and $\varepsilon'm$ outedges incident to each vertex.
	If $(D, \cU, \cP, \cP', \cP^*, \cR, \cC, \sU, \sU', \cM)$ is a $(K, \ell', \ell^*, k, m, \varepsilon, d)$-cycle-setup, then $(D', \cU, \cP, \cP', \cP^*, \cR, \cC, \sU, \sU', \cM)$ is a $(K, \ell', \ell^*, k, m, (\varepsilon')^{\frac{1}{3}}, \frac{d}{2})$-cycle-setup.
\end{prop}

\COMMENT{\begin{proof}
	Note that \cref{def:CST-D,def:CST-M} are still satisfied. For each $i\in [K]$, denote by $\tD_i$ and $\tD_i'$ the $M_i$-contractions of $D[U_i, U_{i+1}]$ and $D'[U_i, U_{i+1}]$. By \cref{fact:Ncontract}, $\tD_i'$ is obtained from $\tD_i$ by removing at most $\varepsilon'm+1$ inedges and $\varepsilon'm+1$ outedges incident to each vertex. Therefore, \cref{def:CST-ST} follows from the proof of \cref{prop:bisetupedgesremoval} (there is room to spare for the extra edges removed in the contraction step).
\end{proof}}

\subsubsection{Extended special path systems and extended special factors}

We will now introduce the concept of extended special path systems. Roughly speaking, these can be viewed as the analogues of the special path systems for blow-up $C_K$'s. As discussed in \cref{sec:SPS}, special path systems can viewed as building blocks for Hamilton cycles; in \cref{lm:newrobustdecomp}, each special path system that we input gives rise to a distinct Hamilton cycle. Analogously, extended special path systems (defined formally below) will be building blocks for constructing Hamilton cycles in a blow-up cycle; in the blow-up cycle version of the robust decomposition lemma (\cref{lm:cyclerobustdecomp} below), each extended special path system that we input will give rise to a distinct Hamilton cycle.

The structure of an extended special path system follows naturally from the proof idea described in \cref{sec:cyclerobustdecomp-sketch}.
More precisely, let $D$ be a blow-up $C_K$ with vertex partition $\cU=(U_1, \dots, U_K)$. For each $i\in [K]$, let $M_i$ be an auxiliary perfect matching from $U_{i+1}$ to $U_i$ and let $\tD_i$ be the $M_i$-contraction of $D[U_i, U_{i+1}]$.
As discussed in \cref{sec:cyclerobustdecomp-sketch}, the leftovers in each of the pairs $(U_i, U_{i+1})$ will be decomposed in two steps. First, we use the robust decomposition lemma in $\tD_i$ to decompose the leftovers into Hamilton cycles in the contracted pair $(U_i, U_{i+1})$.
Then, we expand each of these contracted Hamilton cycles using a spanning set of vertex-disjoint paths whose endpoints are prescribed by $M_i$ (see \cref{fact:cyclerobustdecomp-sketch}).
Thus, an extended special path system will consist of two parts: a special path system $SPS$ in the contracted pair $(U_i, U_{i+1})$ (which will be used to apply \cref{lm:newrobustdecomp} in $\tD_i$) and a spanning set of paths with prescribed endpoints (which will be used to expand the contracted Hamilton cycle containing $SPS$). (Recall that special path systems were introduced in \cref{def:SPS}.)

\begin{definition}[\Gls*{friendly extended special path system}]\label{def:FESPS}
	Let $(\cU, \cP, \cP^*,\cC, \cM)$ be a $(K,\ell^*,k,n)$-cycle-framework and suppose that $\frac{k}{f}\in \mathbb{N}$. For any $(h,i,j)\in [\ell^*]\times [K]\times [f]$, a \emph{friendly $(\ell^*,K, f, h,i,j)$-extended special path system with respect to $\cU, \cP^*, \cC$, and $\cM$} is a linear forest $FESPS$ for which the following hold.
	\begin{enumerate}[label=\rm(FESPS\arabic*),longlabel]
		\item The digraph obtained by deleting all the isolated vertices in the $M_i$-contraction of $FESPS[U_i, U_{i+1}]$ is an $(\ell^*,f,h,j)$-special path system with respect to $\cP_i^*$ and $C^i$.\label{def:FESPS-SPS}
		\item $FESPS\setminus E_{FESPS}(U_i, U_{i+1})$ is a spanning linear forest on $\bigcup\cU$ which consists of $n$ components, one $(u,v)$-path for each $uv\in M_i$.\label{def:FESPS-M}
	\end{enumerate}
\end{definition}

Recall that the main purpose of special path systems is to prescribe edges in our Hamilton decompositions. In particular, in the $\varepsilon$-close to the blow-up $C_4$ case (\cref{thm:blowupC4}), we will need to incorporate prescribed sets of backward edges. It turns out that the concept of friendly extended special path systems is very inconvenient for doing so. However, as discussed in \cref{sec:equivalentP}, if we want to incorporate a linear forest $F$ into a Hamilton cycle, then the internal structure of $F$ is not important; we can always consider an equivalent linear forest instead (recall \cref{def:equivalentP}). Thus, we can generalise the concept of friendly extended special path systems as follows.

\begin{definition}[\Gls*{extended special path system}]\label{def:ESPS}
	Let $(\cU, \cP, \cP^*,\cC, \cM)$ be a $(K,\ell^*,k,n)$-cycle-framework and suppose that $\frac{k}{f}\in \mathbb{N}$. For any $(h,i,j)\in [\ell^*]\times [K]\times [f]$, a linear forest is an \emph{$(\ell^*,K,f,h,i,j)$-extended special path system with respect to $\cU, \cP^*, \cC$, and $\cM$} if it is equivalent to a friendly $(\ell^*,K,f,h,i,j)$-extended special path system with respect to $\cU, \cP^*, \cC$, and $\cM$.
\end{definition}

Note that since a linear forest is equivalent to itself, a friendly extended special path system is indeed an extended special path system.

\begin{definition}[\Gls*{extended special factor}]\label{def:ESF}
	Let $(\cU, \cP, \cP^*,\cC, \cM)$ be a $(K,\ell^*,k,n)$-cycle-frame\-work and suppose that $\frac{k}{f}\in \mathbb{N}$.
	An \emph{$(\ell^*,K,f)$-extended special factor with respect to $\cU, \cP^*, \cC$, and $\cM$} is a multidigraph which has a decomposition $\{ESPS_{h,i,j}\mid (h,i,j)\in [\ell^*]\times [K]\times [f]\}$ where, for each $(h,i,j)\in [\ell^*]\times [K]\times [f]$, $ESPS_{h,i,j}$ is an $(\ell^*,K,f,h,i,j)$-extended special path system with respect to $\cU, \cP^*, \cC$, and $\cM$.
\end{definition}

\subsection{Statement of the robust decomposition lemma for blow-up cycles}

We are now ready to state a blow-up cycle version of the robust decomposition lemma.

\begin{lm}[Robust decomposition lemma for blow-up cycles]\label{lm:cyclerobustdecomp}
	Let $0<\frac{1}{m}\ll \frac{1}{k}\ll \eps \ll \frac{1}{q} \ll \frac{1}{f} \ll \frac{r_1}{m}\ll d\ll \frac{1}{\ell'}, \frac{1}{g}\ll 1$ and suppose
	that $rk^2\le m$. Let
	\[r_2\coloneqq 96\ell'g^2kr, \quad  r_3\coloneqq \frac{rfk}{q}, \quad r^\diamond\coloneqq r_1+r_2+r-(q-1)r_3, \quad s'\coloneqq rfk+7r^\diamond,\]
	and suppose that $\frac{k}{14}, \frac{k}{f}, \frac{k}{g}, \frac{q}{f}, \frac{m}{4\ell'}, \frac{fm}{q}, \frac{2fk}{3g(g-1)} \in \mathbb{N}$.
	Let $(D, \cU, \cP, \cP', \cP^*, \cR, \cC, \sU, \sU', \cM)$ be a $(K,\ell',\frac{q}{f},k, m, \varepsilon,d)$-cycle-setup.
	Let $\mathcal{ESF}$ be a multidigraph which consists of the union of $r_3$ $(\frac{q}{f},K,f)$-extended special factors with respect to $\cU, \cP^*, \cC$, and $\cM$ and let $\mathcal{ESF}'$ be a multidigraph which consists of the union of $r^\diamond$ $(1,K,7)$-extended special factors with respect to $\cU, \cP, \cC$, and $\cM$.
	Then, $D$ contains an $(r_1+r_2+5r^\diamond)$-regular spanning subdigraph $D^{\rm rob}$ for which the following holds.
	Let $H$ be an $r$-regular blow-up $C_K$ with vertex partition $\cU$. Suppose that $H$ is edge-disjoint from $D^{\rm rob}$ and that $E(H)\cap \{uv\mid vu\in \bigcup\cM\}=\emptyset$.
	Then, the multidigraph $H\cup D^{\rm rob}\cup \mathcal{ESF}\cup \mathcal{ESF}'$ has a decomposition $\sC$ into $Ks'$
	edge-disjoint Hamilton cycles such that each cycle in $\sC$ contains precisely one of the extended special path systems in the multidigraph $\mathcal{ESF}\cup \mathcal{ESF}'$.
\end{lm}

By \cref{fact:equivalentP}, we may assume without loss of generality that all extended special path systems contained in $\mathcal{ESF}\cup \mathcal{ESF}'$ are friendly. Thus, as discussed in \cref{sec:cyclerobustdecomp-sketch}, \cref{lm:cyclerobustdecomp} can be obtained by applying the original robust decomposition lemma (\cref{lm:newrobustdecomp}) to each contracted pair $(U_i,U_{i+1})$ of the blow-up cycle in turn. (\APPENDIX{A formal derivation is available in \cref{app:robustdecomp}.}\NOAPPENDIX{A formal derivation is available is available in an appendix of the arXiv version of this paper.})

\onlyinsubfile{\bibliographystyle{abbrv}
	\bibliography{Bibliography/Bibliography}}

\section{Applying the robust decomposition lemma in a very dense blow-up \texorpdfstring{$C_4$}{C4}}\label{sec:ESF}

	\onlyinsubfile{
		\setcounter{section}{7}
		\section{Applying the robust decomposition lemma in a very dense blow-up \texorpdfstring{$C_4$}{C4}}}

In this section, we discuss how to apply \cref{lm:cyclerobustdecomp} in the context of \cref{thm:blowupC4}. Let $T$ be a bipartite tournament which is $\varepsilon$-close to a blow-up $C_4$ with vertex partition $\cU$. Then, observe that $\overrightarrow{T}_\cU$ (that is, the subdigraph of $T$ induced by the forward edges of $T$ (see \cref{sec:forwardbackwarddef})) is a very dense blow-up $C_4$ with vertex partition $\cU$ (only at most an $\varepsilon$ proportion of the edges are missing for $\overrightarrow{T}_\cU$ to be the complete blow-up $C_4$). We will find our absorber $D^{\rm rob}$ by applying the robust decomposition lemma for blow-up cycles (\cref{lm:cyclerobustdecomp}) to (a subdigraph of) $\overrightarrow{T}_\cU$. In this section, we show how to construct, in a very dense blow-up $C_4$, the extended special factors and the cycle-setup required to apply \cref{lm:cyclerobustdecomp}.

\subsection{An alternative description of extended special path systems}
To construct extended special path systems, it will be convenient to consider the following alternative description of extended special path systems.

\begin{prop}\label{prop:ESPS}
	Let $(\cU, \cP, \cP^*,\cC, \cM)$ be a $(K,\ell^*,k,n)$-cycle-framework and suppose that $\frac{k}{f}\in \mathbb{N}$. Let $(h,i,j)\in [\ell^*]\times [K]\times [f]$ and denote $k'\coloneqq \frac{k}{f}+1$. Denote by $I=W_1\dots W_{k'}$ the $j^{\rm th}$ interval in the canonical interval partition of $C^i$ into $f$ intervals. Let $W_{1,h}, \dots, W_{k',h}$ denote the $h^{\rm th}$ subclusters of $W_1, \dots, W_{k'}$ contained in $\cP_i^*$, respectively. Then, a linear forest $ESPS$ is an $(\ell^*,K,f, h,i,j)$-extended special path system if and only if the following properties are satisfied.
	\begin{enumerate}
		\item $V(ESPS)=\bigcup\cU$.\label{prop:ESPS-V}
		\item $V^+(ESPS)= U_{i+1}\setminus N_{M_i}(W_{2,h}\cup \dots \cup W_{k',h})$.\label{prop:ESPS-V+}
		\item $V^-(ESPS)=U_i\setminus (W_{1,h}\cup \dots \cup W_{k'-1,h})$.\label{prop:ESPS-V-}
		\item If $uv\in M_i-(W_{1,h}\cup \dots\cup W_{k',h})$, then $ESPS$ has a component which is a $(u,v)$-path.\label{prop:ESPS-M} 
	\end{enumerate}
\end{prop}

We now give a brief overview of the idea behind \cref{prop:ESPS}.
Recall from \cref{def:SPS} that a special path system is a linear forest which covers a given interval of $C^i$. By \cref{fact:Nuncontract}, the $M_i$-expansion of a special path system is thus a matching which covers a given interval. Thus, \cref{def:FESPS} implies that a friendly extended special path system is a spanning linear forest whose components have endpoints which avoid a given interval and which are matched according to the auxiliary matching $M_i$.
By \cref{def:equivalentP}, these properties are shared by any linear forest which is equivalent to a friendly extended special path system, that is, by any extended special path system (recall \cref{def:ESPS}).
Thus, an extended special path system is simply a spanning linear forest with suitably prescribed endpoints.

\begin{proof}[Proof of \cref{prop:ESPS}]
	$(\Rightarrow)$ Firstly, suppose that $ESPS$ is an $(\ell^*,K,f, h,i,j)$-extended special path system. We need to show that \cref{prop:ESPS-M,prop:ESPS-V,prop:ESPS-V+,prop:ESPS-V-} are satisfied. By \cref{def:equivalentP}, we may assume without loss of generality that $ESPS$ is friendly. Denote $D_1\coloneqq ESPS\setminus ESPS(U_i, U_{i+1})$ and $D_2\coloneqq ESPS\setminus D_1$ (i.e.\ $E(D_2)= E_{ESPS}(U_i, U_{i+1})$). Then, \cref{def:FESPS-M} implies that each $v\in \bigcup\cU$ satisfies
	\begin{equation*}
		d_{D_1}^+(v)=
		\begin{cases}
			1 & \text{if }v\in \bigcup\cU\setminus U_i;\\
			0 & \text{if } v\in U_i;\\
		\end{cases}
		\quad \text{and} \quad
		d_{D_1}^-(v)=
		\begin{cases}
			1 & \text{if }v\in \bigcup\cU\setminus U_{i+1};\\
			0 & \text{if }v \in U_{i+1}.\\
		\end{cases}
	\end{equation*}
	In particular, \cref{prop:ESPS-V} holds.
	Let $SPS$ be the $M_i$-contraction of $ESPS[U_i,U_{i+1}]$. By \cref{def:FESPS-SPS} and \cref{def:SPS}, each $v\in U_i\cup U_{i+1}$ satisfies
	\begin{equation*}
		d_{SPS}^+(v)=
		\begin{cases}
			1 & \text{if }v\in \bigcup_{j'\in [k'-1]}W_{j',h};\\
			0 & \text{otherwise};\\
		\end{cases}
		\quad \text{and} \quad
		d_{SPS}^-(v)=
		\begin{cases}
			1 & \text{if }v\in \bigcup_{j'\in [k'-1]}W_{j'+1,h};\\
			0 & \text{otherwise}.\\
		\end{cases}
	\end{equation*}
	By \cref{def:FESPS-M}, $ESPS$ contains a $(u,v)$-path for each $uv\in M_i$. Since $ESPS$ is a linear forest, this implies that $E(ESPS)\cap \{uv\mid vu\in M_i\}=\emptyset$. Therefore, \cref{fact:contractinguncontracting} implies that $ESPS[U_i, U_{i+1}]$ is the $M_i$-expansion of $SPS$. Thus, \cref{fact:Nuncontract} implies that each $v\in \cU$ satisfies
	\begin{equation*}
		d_{D_2}^+(v)=
		\begin{cases}
			1 & \text{if }v\in \bigcup_{j'\in [k'-1]}W_{j',h};\\
			0 & \text{otherwise};\\
		\end{cases}
		\quad \text{and} \quad
		d_{D_2}^-(v)=
		\begin{cases}
			1 & \text{if }v\in \bigcup_{j'\in [k'-1]}N_{M_i}(W_{j'+1,h});\\
			0 & \text{otherwise}.\\
		\end{cases}
	\end{equation*}
	Therefore, \cref{prop:ESPS-V+,prop:ESPS-V-} are satisfied. For \cref{prop:ESPS-M}, suppose that $uv\in M_i-(W_{1,h}\cup \dots \cup W_{k',h})$. By \cref{def:FESPS-M}, $D_1$ has a component $P_{uv}$ which is a $(u,v)$-path. Moreover, $d_{D_2}(u)=0=d_{D_2}(v)$. Thus, $P_{uv}$ is also a component of $ESPS$ and so \cref{prop:ESPS-M} holds.
	
	$(\Leftarrow)$ Secondly, suppose that $ESPS$ is a linear forest which satisfies \cref{prop:ESPS-M,prop:ESPS-V,prop:ESPS-V+,prop:ESPS-V-}. We need to show that $ESPS$ is an $(\ell^*,K,f, h,i,j)$-extended special path system. By \cref{def:ESPS}, it is enough to construct a friendly $(\ell^*,K,f, h,i,j)$-extended special path system $FESPS$ which is equivalent to $ESPS$.
	
	In order to satisfy \cref{def:FESPS-M}, our friendly extended special path system will need to contain a spanning set of vertex-disjoint paths whose endpoints ``correspond" to the edges of $M_i$. We construct this set of paths as follows. For each $i'\in [K]$, let $u_{i',1}, \dots, u_{i',n}$ be an enumeration of $U_{i'}$. Suppose without loss of generality that $M_i=\{u_{i+1,j'}u_{i,j'}\mid j'\in [n]\}$. Let $\sP\coloneqq \{u_{i+1,j'}u_{i+2,j'}\dots u_{i+K-1,j'}u_{i,j'}\mid j'\in [n]\}$.
	Note that \cref{def:FESPS-M} holds with $\sP$ playing the role of $FESPS\setminus E_{FESPS}(U_i, U_{i+1})$.
	
	We now list the components of $ESPS$ and specify their endpoints. (This will enable to us to construct a friendly extended special path system which is equivalent to $ESPS$.)
	For each $uv\in M_i-(W_{1,h}\cup \dots\cup W_{k',h})$, denote by $P_{uv}$ the component of $ESPS$ which is a $(u,v)$-path ($P_{uv}$ exists by \cref{prop:ESPS-M} and is unique since $ESPS$ is a linear forest).
	Let $\sP_1\coloneqq \{P_{uv}\mid uv\in M_i-(W_{1,h}\cup \dots\cup W_{k',h})\}$. Let $\sP_2$ be the set of components of $ESPS\setminus \sP_1$. By \cref{prop:ESPS-V+,prop:ESPS-V-}, $\sP_2$ consists of $m'\coloneqq |W_{1,h}|$ paths, each of which starts in $N_{M_i}(W_{1,h})$ and ends in $W_{k',h}$.
	
	We are now ready to select the edges from $U_i$ to $U_{i+1}$.
	For each $j'\in [k']$, let $v_{j',1}, \dots, v_{j',m'}$ be an enumeration of $W_{j',h}$ and denote by $w_{j',1}, \dots, w_{j',m'}$ the (unique) neighbours of $v_{j',1}, \dots, v_{j',m'}$ in $M_i$, respectively. Suppose without loss of generality that $\sP_2$ consists of a $(w_{1,j'},v_{k',j'})$-path $P_{j'}$ for each $j'\in [m']$.
	For each $j'\in [k'-1]$, let $M_{j'}'\coloneqq \{v_{j',1}w_{j'+1,1}, \dots, v_{j',m'}w_{j'+1,m'}\}$. Let $M'\coloneqq M_1'\cup \dots \cup M_{k'-1}'$.
	
	Let $FESPS$ be the digraph on $\bigcup\cU$ defined by $E(FESPS)\coloneqq E(\sP) \cup M'$. Observe that $FESPS\setminus FESPS(U_i, U_{i+1})=\sP$ and so \cref{def:FESPS-M} holds.
	Thus, it remains to prove that $FESPS$ is a spanning linear forest which is equivalent to $ESPS$ and that \cref{def:FESPS-SPS} holds.
	
	\begin{claim}\label{claim:ESPS-linforest}
		$FESPS$ is a spanning linear forest satisfying both $V^\pm(FESPS)=V^\pm(ESPS)$.
	\end{claim}
	
	\begin{proofclaim}
		By construction of $\sP$, each $v\in \bigcup\cU$ satisfies
		\begin{equation*}
			d_{\sP}^+(v)=
			\begin{cases}
				1 & \text{if }v\in \bigcup\cU\setminus U_i;\\
				0 & \text{if }v\in U_i;\\
			\end{cases}
			\quad \text{and} \quad
			d_{\sP}^-(v)=
			\begin{cases}
				1 & \text{if }v\in \bigcup\cU\setminus U_{i+1};\\
				0 & \text{if }v\in U_{i+1}.\\
			\end{cases}
		\end{equation*}
		By definition of $M'$, each $v\in \bigcup\cU$ satisfies
		\begin{equation*}
			d_{M'}^+(v)=
			\begin{cases}
				1 & \text{if }v\in \bigcup_{j'\in [k'-1]}W_{j',h};\\
				0 & \text{otherwise};\\
			\end{cases}
			\quad \text{and} \quad
			d_{M'}^-(v)=
			\begin{cases}
				1 & \text{if }v\in \bigcup_{j'\in [k'-1]}N_{M_i}(W_{j'+1,h});\\
				0 & \text{otherwise}.\\
			\end{cases}
		\end{equation*}
		Thus, $FESPS$ is spanning and $\Delta^0(FESPS)=1$.
		Moreover, \cref{prop:ESPS-V+,prop:ESPS-V-} imply that both $V^\pm(FESPS)=V^\pm(ESPS)$. 
		
		Suppose for a contradiction that $FESPS$ contains a cycle $C$. Since $\sP$ is a linear forest, $E(C)\cap M'\neq \emptyset$. Let $j'\in [k'-1]$ be the largest index such that $E(C)\cap M_{j'}'\neq\emptyset$ and let $vw\in E(C)\cap M_{j'}'$. By construction of $M_{j'}'$, we have $w\in N_{M_i}(W_{j'+1,h})$. Let $w'$ be the (unique) neighbour of $w$ in $M_i$. Note that $w'\in W_{j'+1,h}\subseteq U_i$.		
		By definition of $\sP$, we have $w'\in V(C)$ and $d_{\sP}^+(w')=0$. Therefore, there exists $e\in E(C)\cap M'$ which starts at $w'$. By construction of $M'$, we have $j'< k'-1$ and $e\in M_{j'+1}'$. But this contradicts the maximality of $j'$ and so $FESPS$ does not contain a cycle.
	\end{proofclaim}
	
	\begin{claim}
		$ESPS$ and $FESPS$ are equivalent.
	\end{claim}
	
	\begin{proofclaim}
		Recall \cref{def:equivalentP}. By \cref{claim:ESPS-linforest} and \cref{prop:ESPS-V}, we have $V(ESPS)=\bigcup\cU=V(FESPS)$.
		Thus, it remains to find a bijection $\phi$ from the components of $ESPS$ to the components of $FESPS$ such that for each component $P$ of $ESPS$, the paths $P$ and $\phi(P)$ have the same starting and ending points.
		
		Recall that $\sP_1\cup \sP_2$ denotes the set of components of $ESPS$, where $\sP_1$ consists of a $(u,v)$-path $P_{uv}$ for each $uv\in M_i-(W_{1,h}\cup \dots\cup W_{k',h})$ and $\sP_2$ consists of a $(w_{1,j'},v_{k',j'})$-path $P_{j'}$ for each $j'\in [m']$.
		
		Let $uv\in M_i-(W_{1,h}\cup \dots\cup W_{k',h})$. 
		Let $P_{uv}'$ be the $(u,v)$-path contained in $\sP$.
		By construction of $M'$, both $u,v\notin V(M')$. Thus, $P_{uv}'$ is a component of $FESPS$ and so we can let $\phi(P_{uv})\coloneqq P_{uv}'$.
		
		Let $j'\in [m']$. By definition of $M'$, we have both $w_{1,j'},v_{k',j'}\notin V(M')$. Moreover, \[v_{1,j'}w_{2,j'}, v_{2,j'}w_{3,j'}, \dots, v_{k'-1,j'}w_{k',j'}\in M'.\] For each $i'\in [k']$, let $Q_{i'}$ be the $(w_{i',j'},v_{i',j'})$-path contained in $\sP$. Then, \[P_{j'}'\coloneqq w_{1,j'}Q_1v_{1,j'}w_{2,j'}Q_2v_{2,j'}w_{3,j'}\dots w_{k'-1,j'}Q_{k'-1}v_{k'-1,j'}w_{k',j'}\] is a component of $FESPS$ and so we can let $\phi(P_{j'})\coloneqq P_{j'}'$.
		
		By construction, $\phi$ is an injection from the components of $ESPS$ to the components of $FESPS$ such that for each component $P$ of $ESPS$, the paths $P$ and $\phi(P)$ have the same starting and ending points.
		Since $FESPS$ is a linear forest satisfying both $V^\pm(FESPS)=V^\pm(ESPS)$, $\phi$ is also a surjection. 
	\end{proofclaim}
	
	\begin{claim}
		\cref{def:FESPS-SPS} is satisfied.
	\end{claim}
	
	\begin{proofclaim}
		Let $D$ be the $M_i$-contraction of $FESPS[U_i, U_{i+1}]=M'[U_i,U_{i+1}]$ and let $SPS$ be obtained from $D$ by deleting all the isolated vertices. We need to show that $SPS$ is an $(\ell^*,f,h,j)$-special path system with respect to $\cP_i^*$ and $C^i$.
		By \cref{def:FESPS-M}, $F\coloneqq M'\cup M_i$ is 
		obtained from $FESPS$ by contracting each path in $\sP$ into an edge from its starting point to its ending point. Together with \cref{claim:ESPS-linforest}, \cref{prop:ESPS-V+}, and \cref{prop:ESPS-V-}, this implies that $F$ is	a linear forest satisfying the following properties.
		\begin{itemize}
			\item $V^+(F)= V^+(FESPS)=V^+(ESPS)= U_{i+1}\setminus N_{M_i}(W_{2,h}\cup \dots \cup W_{k',h})$.
			\item $V^-(F)=V^-(FESPS)=V^-(ESPS)=U_i\setminus (W_{1,h}\cup \dots \cup W_{k'-1,h})$.
			\item $V^0(F)\cap U_i= U_i\setminus (V^+(F)\cup V^-(F))=W_{1,h}\cup \dots \cup W_{k'-1,h}$.
		\end{itemize}
		Thus, \cref{fact:contractlinforest} (applied with $U_i, U_{i+1}$, and $M_i$ playing the roles of $A, B$, and $M$) implies that $D$ is a linear forest satisfying the following properties.
		\begin{itemize}
			\item $V^+(D)= N_{M_i}(V^+(F))=U_i\setminus(W_{2,h}\cup \dots \cup W_{k',h})$.
			\item $V^-(D)=V^-(F)=U_i\setminus (W_{1,h}\cup \dots \cup W_{k'-1,h})$.
			\item $V^0(D)= (V^0(F)\cap U_i)\setminus N_{M_i}(V^+(F))=W_{2,h}\cup \dots \cup W_{k'-1,h}$.
		\end{itemize}
		In particular, $SPS$ is a linear forest satisfying $V^0(SPS)=V^0(D)=W_{2,h}\cup \dots \cup W_{k'-1,h}$ and so \cref{def:SPS-V0} holds.
		Note that the set of isolated vertices in $D$ is precisely $V^+(D)\cap V^-(D)$. Thus, $V^+(SPS)=V^+(D)\setminus V^-(D)=W_{1,h}$ and $V^-(SPS)=V^-(D)\setminus V^+(D)=W_{k',h}$, so \cref{def:SPS-V+-} holds. Therefore, $SPS$ is an $(\ell^*,f,h,j)$-special path system and so \cref{def:FESPS-SPS} is satisfied.
	\end{proofclaim}
	This concludes the proof of \cref{prop:ESPS}.
\end{proof}

\begin{cor}\label{cor:ESFreg}
	Let $(\cU, \cP, \cP^*,\cC, \cM)$ be a $(K,\ell^*,k,n)$-cycle-framework and suppose that $\frac{k}{f}\in \mathbb{N}$. An $(\ell^*,K,f)$-extended special factor $ESF$ is a $(1+\ell^*(K-1)f)$-regular multidigraph.
\end{cor}

\begin{proof}
	Let $\{ESPS_{h,i,j}\mid (h,i,j)\in [\ell^*]\times [K]\times [f]\}$ be the decomposition of $ESF$ which witnesses that $ESF$ is an $(\ell^*,K,f)$-extended special factor.
	Let $i\in [K]$ and $v\in U_i$. 
	By \cref{prop:ESPS}, there is a pair $(h,j)\in [\ell^*]\times [f]$ such that $v\notin V^-(ESPS_{h,i,j})$ but $v\in V^-(ESPS_{h',i,j'})$ for all $(h',j')\in ([\ell^*]\times [f])\setminus \{(h,j)\}$.
	Moreover, $v\notin V^-(ESPS_{h',i',j'})$ for each $(h',i',j')\in [\ell^*]\times ([K]\setminus \{i\})\times [f]$. By \cref{prop:ESPS}\cref{prop:ESPS-V}, $ESPS$ is spanning linear forest on $\bigcup \cU$ and so
	\begin{align*}
		d_{ESF}^+(v)&=d_{ESPS_{h,i,j}}^+(v)+\sum_{(h',j')\in ([\ell^*]\times [f])\setminus \{(h,j)\}} d_{ESPS_{h',i,j'}}^+(v)\\ &\qquad\qquad+\sum_{(h',i',j')\in [\ell^*]\times ([K]\setminus \{i\})\times [f]}d_{ESPS_{h',i',j'}}^+(v)\\
		&=1+(\ell^*f-1)\cdot 0+\ell^*(K-1)f\cdot 1=1+\ell^*(K-1)f.
	\end{align*}
	Since $M_{i-1}$ is a perfect matching from $U_i$ to $U_{i-1}$, one can apply similar arguments to show that there are precisely $\ell^*f-1$ tuples $(h,i',j)\in [\ell^*]\times [K]\times [f]$ for which $v\in V^+(ESF)$ and so $d_{ESF}^-(v)=1+\ell^*(K-1)f$.
\end{proof}

\subsection{Constructing extended special factors}\label{sec:constructESF}

Recall that the blow-up cycle version of the robust decomposition lemma (\cref{lm:cyclerobustdecomp}) can only be applied when there are no exceptional vertices (see \cref{def:CST-ST}). In general, we will have a non-empty exceptional set $U^*\subseteq V(D)$ and so we will apply \cref{lm:cyclerobustdecomp} with $D-U^*$ playing the role of $D$. As a result, the cycles obtained via \cref{lm:cyclerobustdecomp} will not be Hamilton cycles on $V(D)$, but will only span $V(D)\setminus U^*$. We will incorporate the exceptional vertices into these cycles using the strategy described in \cref{sec:SC}: we will initially reserve $4s'$ special covers in $D$ (see \cref{def:SC}) and then construct the extended special factors for \cref{lm:cyclerobustdecomp} in such a way that each extended special path system contains the complete special sequence (see \cref{def:MSC}) associated to one of the reserved special covers.

However, as described in \cref{prop:ESPS}, an extended special path system is a linear forest whose components have prescribed endpoints. Thus, our special covers will need to satisfy certain constraints. More precisely, let $SC$ be a special cover and denote by $M_{SC}$ the associated complete special sequence. Suppose that we want to construct an extended special path system which contains $M_{SC}$.
Let $P$ be a component of $SC$ which is not an isolated vertex. By definition, $M_{SC}$ contains an edge from the starting point $u$ of $P$ to the ending point $v$ of $P$ and so for any linear forest $F\supseteq M_{SC}$, we have $u\notin V^-(F)$ and $v\notin V^+(F)$. Thus, \cref{prop:ESPS}\cref{prop:ESPS-V+,prop:ESPS-V-} imply that we require $u\notin U_i\setminus (W_{1,h}\cup \dots \cup W_{k'-1,h})$ and $v\notin U_{i+1}\setminus N_{M_i}(W_{2,h}\cup \dots \cup W_{k',h})$. 
Moreover, $u$ and $v$ will lie in a common connected component of $F$, so \cref{prop:ESPS}\cref{prop:ESPS-M} implies that we cannot have $u$ and $v$ lying in different edges of $M_i-(W_{1,h}\cup \dots\cup W_{k',h})$.
For convenience, we will require that $u$ and $v$ completely avoid the vertices of $M_i-(W_{1,h}\cup \dots\cup W_{k',h})$. Altogether, this motivates the following definition.

\begin{definition}[\Gls*{localised special cover}]\label{def:LSC}
	Let $(\cU, \cP, \cP^*,\cC, \cM)$ be a $(4,\ell^*,k,n)$-cycle-frame\-work. Let $D$ be a digraph with $V(D)\supseteq \bigcup \cU$ and denote by $U^*\coloneqq V(D)\setminus \bigcup \cU$ the exceptional set of $D$. Suppose that $\frac{k}{f}\in \mathbb{N}$ and  denote $k'\coloneqq \frac{k}{f}+1$. Let $(h,i,j)\in [\ell^*]\times [4]\times [f]$ and let $W_{1,h}, \dots, W_{k',h}$ be defined as in \cref{prop:ESPS}.
	A special cover $SC$ in $D$ with respect to $U^*$ is \emph{$(\ell^*,4,f,h,i,j)$-localised (with respect to $\cP^*, \cC$, and $\cM$)} if the following holds.
	\begin{equation*}
		(V^+(SC)\cup V^-(SC))\cap (U_i\cup U_{i+1}) \subseteq (W_{1,h}\cup \dots \cup W_{k'-1,h})\cup N_{M_i}(W_{2,h}\cup \dots \cup W_{k',h}).
	\end{equation*}
\end{definition}

The next \lcnamecref{fact:LSC} follows immediately from \cref{def:MSC}.

\begin{fact}\label{fact:LSC}
    Let $(\cU, \cP, \cP^*,\cC, \cM)$ be a $(4,\ell^*,k,n)$-cycle-framework. Let $D$ be a digraph with $V(D)\supseteq \bigcup \cU$ and denote by $U^*\coloneqq V(D)\setminus \bigcup \cU$ the exceptional set of $D$. Suppose that $\frac{k}{f}\in \mathbb{N}$ and  denote $k'\coloneqq \frac{k}{f}+1$. Let $(h,i,j)\in [\ell^*]\times [4]\times [f]$ and let $W_{1,h}, \dots, W_{k',h}$ be defined as in \cref{prop:ESPS}. Suppose that $SC$ is an $(\ell^*,4,f,h,i,j)$-localised special cover in $D$ with respect to $U^*$. Then, the complete special sequence $M_{SC}$ associated to $SC$ satisfies
    \begin{equation*}
		V(M_{SC})\cap (U_i\cup U_{i+1}) \subseteq (W_{1,h}\cup \dots \cup W_{k'-1,h})\cup N_{M_i}(W_{2,h}\cup \dots \cup W_{k',h}).
	\end{equation*}
\end{fact}

Our strategy for incorporating the complete special sequence $M_{SC}$ into an extended special path system will be to extend each edge of $M_{SC}$ into a longer path by ``winding around'' $\cU$. This can be done greedily, with room to spare, so it will be possible to ensure that these paths all start and end in given small sets of vertices $X$ and $Y$ and avoid the vertices of a small set $Z$.

\begin{lm}\label{lm:Meps}
	Let $0<\frac{1}{n}\ll \varepsilon\ll 1$. Let $D$ be a digraph and $U_1\cup \dots\cup U_4$ be a partition of $V(D)$ into vertex classes of size $n$. Suppose that $\delta(D[U_i, U_{i+1}])\geq (1-\varepsilon)n$ for each $i\in [4]$ (where $U_5\coloneqq U_1$). 
	Let $X\subseteq U_1$ and $Y\subseteq U_4$. Let $Z\subseteq V(D)\setminus (X\cup Y)$ satisfy $|Z|\leq \varepsilon n$.
	Let $M$ be a matching on $V(D)\setminus (X\cup Y\cup Z)$. Suppose that $|M|\leq |X|= |Y|\leq \varepsilon n$. For each $i\in [4]$, let $n_i^+$ and $n_i^-$ be the number of edges in $M$ which start and end in $U_i$, respectively. Suppose that $n_i^+=n_{i+1}^-$ for each $i\in [4]$ (where $n_5^-\coloneqq n_1^-$).
	Then, there exists a set $\sP$ of $|M|$ vertex-disjoint paths for which the following hold.
	\begin{enumerate}
		\item $M\subseteq \sP\subseteq D\cup M$.\label{lm:Meps-M}
		\item $V^+(\sP)\subseteq X$, $V^-(\sP)\subseteq Y$, and $V^0(\sP)\subseteq V(D)\setminus (X\cup Y\cup Z)$.\label{lm:Meps-V}
		\item $|V(\sP)\cap U_1|=\dots=|V(\sP)\cap U_4|\leq 4|M|$.\label{lm:Meps-balanced}
	\end{enumerate}
\end{lm}

\begin{proof}
	Let $u_1v_1, \dots, u_mv_m$ be an enumeration of $M$. Let $x_1, \dots, x_m\in X$ and $y_1, \dots, y_m\in Y$ be distinct.
	For each $j\in [m]$, we will construct a path $x_j P_j' u_jv_j Q_j' y_j$ such that $P_j'$ and $Q_j'$ are paths of length between $4$ and $8$ which ``wind around'' $\cU$.
	
	For each $j\in [m]$, let $j^+,j^-\in [4]$ be such that $u_j\in U_{j^+}$ and $v_j\in U_{j^-}$.
	For each $i\in [4]$, let $U_i'\coloneqq U_i\setminus (X\cup Y \cup Z\cup V(M))$. Note that each $i\in [4]$ and $v\in U_i$ satisfy
	\begin{equation}\label{eq:Meps-delta-}
		|N_D^-(v)\cap U_{i-1}'|\geq (1-\varepsilon)n-|X|-|Y|-|Z|-2|M|\geq (1-6\varepsilon)n> \frac{n}{2}+4m
	\end{equation}
	and, similarly,
	\begin{equation}\label{eq:Meps-delta+}
		|N_D^+(v)\cap U_{i+1}'|> \frac{n}{2}+4m.
	\end{equation}
	Thus, one can greedily construct vertex-disjoint paths $P_1, \dots, P_m, Q_1, \dots, Q_m\subseteq D$ such that for each $j\in [m]$, $P_j=u_{j,3}u_{j,4}u_{j,1}'\dots u_{j,j^+-1}'u_j$ for some $u_{j,3}\in U_3'$, $u_{j,4}\in U_4'$, $u_{j,1}'\in U_1', \dots$, $u_{j,j^+-1}'\in U_{j^+-1}'$ and, similarly, $Q_j=v_jv_{j,j^-+1}'\dots v_{j,4}'v_{j,1}v_{j,2}$ for some $v_{j,j^-+1}'\in U_{j^-+1}', \dots$, $v_{j,4}'\in U_4'$, $v_{j,1}\in U_1'$, $v_{j,2}\in U_2'$.
	Then, \cref{eq:Meps-delta-,eq:Meps-delta+} imply that there exist distinct $u_{1,2}, \dots, u_{m,2}\in U_2'\setminus \bigcup_{j\in [m]}V(P_m\cup Q_m)$ and $v_{1,3}, \dots, v_{m,3}\in U_3'\setminus \bigcup_{j\in [m]}V(P_m\cup Q_m)$ such that, for each $j\in [m]$, $u_{j,2}\in N_D^+(x_j)\cap N_D^-(u_{j,3})$ and $v_{j,3}\in N_D^+(v_{j,2})\cap N_D^-(y_j)$.
	For each $j\in [m]$, denote $P_j'\coloneqq x_ju_{j,2}u_{j,3}P_ju_j$ and $Q_j'\coloneqq v_jQ_jv_{j,2}v_{j,3}y_j$.
	Let $\sP\coloneqq \{x_jP_j'u_jv_jQ_j'y_j\mid j\in [m]\}$. By construction, $\sP$ is a set of vertex-disjoint paths satisfying \cref{lm:Meps-M,lm:Meps-V}. It remains to verify \cref{lm:Meps-balanced}. For each $i\in [4]$ and $j\in [m]$, we have
	\begin{equation}\label{eq:Meps-balanced}
		|V(P_j')\cap U_i|=
		\begin{cases}
			2 & \text{if }i\leq j^+;\\
			1 & \text{otherwise};\\
		\end{cases}
		\quad \text{and} \quad
		|V(Q_j')\cap U_i|=
		\begin{cases}
			2 & \text{if }i\geq j^-;\\
			1 & \text{otherwise}.\\
		\end{cases}
	\end{equation}
	Recall that for each $i\in [4]$, $n_i^+$ denotes the number of indices $j\in [m]$ for which $j^+=i$ and $n_i^-$ denotes the number of indices $j\in [m]$ for which $j^-=i$. Therefore, each $i\in [4]$ satisfies
	\begin{align*}
		|V(\sP)\cap U_i|&\stackrel{\text{\eqmakebox[Meps]{}}}{=}\sum_{j\in [m]}|V(P_j')\cap U_i|+\sum_{j\in [m]}|V(Q_j')\cap U_i|\\
		&\stackrel{\text{\eqmakebox[Meps]{\text{\cref{eq:Meps-balanced}}}}}{=}(m+n_i^++\dots +n_4^+)+(m+n_1^-+\dots+n_i^-)\\
		&\stackrel{\text{\eqmakebox[Meps]{}}}{=}2m+(n_{i+1}^-+\dots +n_5^-)+(n_1^-+\dots+n_i^-)=3m+n_1^-.
	\end{align*}
	Thus, \cref{lm:Meps-balanced} holds.
\end{proof}

Note that in the proof of \cref{lm:Meps} the conditions on the number of paths starting and ending in each vertex class was necessary to obtain a set $\sP$ of vertex-disjoint paths which covers the same number of vertices from each vertex class (see \cref{lm:Meps}\cref{lm:Meps-balanced}). Eventually, we will want to extend $\sP$ to a full extended special path system. The number of vertices covered by $\sP$ will thus be of particular importance since, by \cref{prop:ESPS}\cref{prop:ESPS-V}, an extended special path system needs to span all the vertex classes $U_1, \dots, U_4$, which are all of the same size (see \cref{def:CF-U}). 
This motivates the following \lcnamecref{def:BSC}.

\begin{definition}[\Gls*{balanced special cover}]\label{def:BSC}
	Let $(\cU, \cP, \cP^*,\cC, \cM)$ be a $(4,\ell^*,k,n)$-cycle-frame\-work. Let $D$ be a digraph with $V(D)\supseteq \bigcup \cU$ and denote by $U^*\coloneqq V(D)\setminus \bigcup \cU$ the exceptional set of $D$. Let $SC$ be a special cover in $D$ with respect to $U^*$. For each $i\in [4]$, let $n_i^+$ and $n_i^-$ be the number of components of $SC$ which are not isolated vertices and start and end in $U_i$, respectively. We say that $SC$ is \emph{$\cU$-balanced} if $n_i^+=n_{i+1}^-$ for each $i\in [4]$ (where $n_5^-\coloneqq n_1^-$).
\end{definition}

The next \lcnamecref{lm:ESF} states that given small special covers which are localised and balanced, one can incorporate the associated complete special sequences into extended special path systems. (Note that \cref{lm:ESF} below is the analogue of \cref{lm:SF} from the bipartite robust outexpander case. The only difference is that, in \cref{lm:SF}, we also constructed the special covers at the same time. In the context of \cref{thm:blowupC4}, constructing the special covers is much more difficult because of the backward edges and so this will be done separately at a later stage.)

\begin{lm}[Constructing extended special path systems and factors from special covers]\label{lm:ESF}
	Let $0<\frac{1}{n}\ll \varepsilon\ll \frac{1}{k}\ll \varepsilon'\ll 1$ and $\frac{1}{k}\ll \frac{1}{f}, \frac{1}{\ell^*}\leq 1$.
	Let $r$ be an integer satisfying $\frac{f(\ell^*)^2rk}{n} \ll 1$. Suppose that $\ell^*f\geq 2$ and $\frac{k}{f}\in \mathbb{N}$.
	Let $(\cU, \cP, \cP^*,\cC, \cM)$ be a $(4,\ell^*,k,n)$-cycle-framework. Let $D$ be a digraph with $V(D)\supseteq \bigcup \cU$ and denote by $U^*\coloneqq V(D)\setminus \bigcup \cU$ the exceptional set of $D$.
	Suppose that the following hold for each $i\in [4]$.
	\begin{enumerate}
		\item For any cluster $V\in \cP_i^*$, the set $N_{M_i}(V)$ is a cluster in $\cP_{i+1}^*$ (where $\cP_5^*\coloneqq \cP_1^*$).\label{lm:ESF-Mi}
		\item $D[V,W]$ is $[\varepsilon', \geq 1-\varepsilon']$-superregular whenever $V\subseteq U_i$ and $W\subseteq U_{i+1}$ are unions of clusters in $\cP_i^*$ and $\cP_{i+1}^*$, respectively.\label{lm:ESF-supreg}
	\end{enumerate}	
	Let $\mathcal{SC}=\{SC_{\ell,h,i,j}\mid (\ell,h,i,j)\in [r]\times [\ell^*]\times [4]\times [f]\}$ be a set of edge-disjoint special covers in $D$ with respect to $U^*$ such that the following hold for each $(\ell,h,i,j)\in [r]\times [\ell^*]\times [4]\times [f]$.
	\begin{enumerate}[resume]
		\item $|SC_{\ell,h,i,j}|\leq \varepsilon n$. In particular, \cref{def:MSC} implies that the complete special sequence $M_{\ell,h,i,j}$ associated to $SC_{\ell,h,i,j}$ satisfies $|M_{\ell,h,i,j}|\leq \varepsilon n$.\label{lm:ESF-small}
		\item $SC_{\ell,h,i,j}$ is $(\ell^*,4,f,h,i,j)$-localised.\label{lm:ESF-localised}
		\item $SC_{\ell,h,i,j}$ is $\cU$-balanced.\label{lm:ESF-balanced}
	\end{enumerate}
	Then, there exist $r$ $(\ell^*, 4,f)$-extended special factors $ESF_1, \dots, ESF_r$ with respect to $\cU, \cP^*, \cC$, and $\cM$
	such that the following hold, where for each $(\ell,h,i,j)\in [r]\times [\ell^*]\times [4]\times [f]$, $ESPS_{\ell,h,i,j}$ denotes the $(\ell^*,4,f,h,i,j)$-special path system contained in $ESF_\ell$.
	\begin{enumerate}[label=\rm(\alph*)]
		\item For each $(\ell,h,i,j)\in [r]\times [\ell^*]\times [4]\times [f]$, we have $M_{\ell,h,i,j}\subseteq ESPS_{\ell,h,i,j}\subseteq (D\setminus \mathcal{SC})\cup M_{\ell,h,i,j}$.\label{lm:ESF-M}
		\item Let $(\ell,h,i,j),(\ell',h',i',j')\in [r]\times [\ell^*]\times [4]\times [f]$ be distinct. Then, we have $(ESPS_{\ell,h,i,j}\setminus M_{\ell,h,i,j})\cap (ESPS_{\ell',h',i',j'}\setminus M_{\ell',h',i',j'})=\emptyset$.\label{lm:ESF-disjoint}
	\end{enumerate}
\end{lm}

Roughly speaking, \cref{lm:ESF}\cref{lm:ESF-M} means that each complete special sequence is incorporated into a distinct extended special path system, while \cref{lm:ESF}\cref{lm:ESF-disjoint} states that each edge of $D\setminus \mathcal{SC}$ is incorporated into at most one of the extended special path systems.

\begin{proof}[Proof of \cref{lm:ESF}]
	First, recall from \cref{def:CF-P*,def:CF-P} that for each $i\in [4]$, $U_i$ is the union of the clusters in $\cP_i^*$. Thus, \cref{lm:ESF-supreg} implies that each $i\in [4]$ satisfies
	\begin{equation}\label{eq:ESF-delta}
		\delta(D[U_i, U_{i+1}])\geq (1-2\varepsilon')n.
	\end{equation}
	
	Fix additional constants $\varepsilon_1$ and $\varepsilon_2$ such that $\frac{f(\ell^*)^2rk}{n}, \varepsilon' \ll \varepsilon_1\ll \varepsilon_2\ll 1$.
	Suppose inductively that for some $0\leq t\leq 4r\ell^*f$ we have constructed a set $T_t\subseteq [r]\times[\ell^*]\times [4]\times [f]$ of size $t$ and a set $\mathcal{ESPS}_t=\{ESPS_{\ell,h,i,j}\mid (\ell,h,i,j)\in T_t\}$ such that the following properties hold.
	\begin{enumerate}[label=\rm(\alph*$'$)]
		\item For each $(\ell,h,i,j)\in T_t$, we have $M_{\ell,h,i,j}\subseteq ESPS_{\ell,h,i,j} \subseteq (D\setminus \mathcal{SC})\cup M_{\ell,h,i,j}$.\label{lm:ESF-IH-M}
		\item Let $(\ell,h,i,j),(\ell',h',i',j')\in T_t$ be distinct. Then, we have $(ESPS_{\ell,h,i,j}\setminus M_{\ell,h,i,j})\cap (ESPS_{\ell',h',i',j'}\setminus M_{\ell',h',i',j'})=\emptyset$.\label{lm:ESF-IH-disjoint}
		\item For each $(\ell,h,i,j)\in T_t$, $ESPS_{\ell,h,i,j}$ is an $(\ell^*,4,f,h,i,j)$-extended special path system with respect to $\cU, \cP^*, \cC$, and $\cM$.\label{lm:ESF-IH-ESPS}
	\end{enumerate}
	First, suppose that $t= 4r\ell^*f$. Then, define $ESF_\ell\coloneqq \bigcup_{(h,i,j)\in [\ell^*]\times [4]\times [f]} ESPS_{\ell,h,i,j}$ for each $\ell\in [r]$.
	By \cref{lm:ESF-IH-ESPS}, $ESF_1, \dots, ESF_r$ are $(\ell^*, 4, f)$-extended special factors. Moreover, \cref{lm:ESF-M,lm:ESF-disjoint} follow from \cref{lm:ESF-IH-M,lm:ESF-IH-disjoint}.
	
	We may therefore assume that $t<4r\ell^*f$. Let $(\ell,h,i,j)\in ([r]\times[\ell^*]\times [4]\times [f])\setminus T_t$ and define $T_{t+1}\coloneqq T_t\cup \{(\ell,h,i,j)\}$.
	We will now construct $ESPS_{\ell,h,i,j}$ as follows. 
	Let $D'\coloneqq D\setminus (\mathcal{SC}\cup \mathcal{ESPS}_t)$. Since $\mathcal{SC}$ consists of $4r\ell^*f$ linear forests and $\mathcal{ESPS}_t$ consists of $t$ linear forests, we have
	\begin{equation}\label{eq:ESF-Delta}
		\Delta^0(D\setminus D')\leq 4r\ell^*f+t\leq \frac{\varepsilon_1 n}{k\ell^*}.
	\end{equation}
	
	By \cref{lm:ESF-small} and \cref{lm:adjustP} (applied with $2\varepsilon, U_{i-2}, U_{i-1}, U_i, U_{i+1}, \cP_{i-2}^*, \dots, \cP_{i+1}^*$, and $V(M_{\ell,h,i,j})$ playing the roles of $\varepsilon, U_1, \dots, U_4,\cP_1, \dots, \cP_4$, and $S$), we may assume without loss of generality that there exists, for each $i'\in [4]\setminus \{i,i+1\}$, a cluster $V_{i'}\in \cP_{i'}^*$ which satisfies 
	\begin{equation}\label{eq:ESF-Vi'}
		V(M_{\ell,h,i,j})\cap U_{i'}\subseteq V_{i'}.
	\end{equation}
	(Otherwise, we may simply apply the arguments below with the partitions $\cP_{i-2}', \dots, \cP_{i+1}'$ guaranteed by \cref{lm:adjustP} playing the roles of $\cP_{i-2}^*, \dots, \cP_{i+1}^*$. This is possible since these satisfy \cref{lm:ESF-supreg} up to a slightly worse $\varepsilon'$-parameter.)
	
	To ensure that \cref{lm:ESF-IH-ESPS} is satisfied, we will use \cref{prop:ESPS} and construct a linear forest which satisfies \cref{prop:ESPS}\cref{prop:ESPS-M,prop:ESPS-V,prop:ESPS-V+,prop:ESPS-V-}.
	Let $W_{1,h}, \dots, W_{k',h}$ be defined as in \cref{prop:ESPS} and denote $M_i'\coloneqq M_i-(W_{1,h}\cup \dots \cup W_{k',h})$. Let
	\begin{equation}\label{eq:ESF-defXi}
		X_i\coloneqq W_{k',h} \quad \text{and} \quad X_{i+1}\coloneqq N_{M_i}(W_{1,h}).
	\end{equation} 
	By \cref{lm:ESF-localised} and \cref{fact:LSC}, we have \begin{equation}\label{eq:ESF-Xi}
		V(M_{\ell,h,i,j})\cap (X_i\cup X_{i+1})=\emptyset.
	\end{equation}
	For each $i'\in [4]\setminus \{i,i+1\}$, fix a cluster
	\begin{equation}\label{eq:ESF-defXi'}
		X_{i'}\in \cP_{i'}^*\setminus \{V_{i'}\}.
	\end{equation} 
	By \cref{eq:ESF-Vi'}, we have
	\begin{equation}\label{eq:ESF-Xi'}
		V(M_{\ell,h,i,j})\cap X_{i'}=\emptyset.
	\end{equation} 
	By \cref{lm:ESF-supreg}, $D[X_{i'}, X_{i'+1}]$ is $[\varepsilon', \geq 1-\varepsilon']$-superregular for each $i'\in [4]$ (where $X_5\coloneqq X_1$). We will reserve these superregular pairs to finish off the construction of $ESPS_{\ell,h,i,j}$.
	
	\begin{steps}
		\item \textbf{Constructing the components for \cref{prop:ESPS}\cref{prop:ESPS-M}.}\label{step:ESF-Mi}
		In this step, we will use \cref{lm:regularitypaths} to construct a set $\sP_1$ of vertex-disjoint paths which consists of one $(u,v)$-path for each $uv\in M_i'$. The paths in $\sP_1$ will eventually be incorporated as components of $ESPS_{\ell,h,i,j}$ to ensure that \cref{prop:ESPS}\cref{prop:ESPS-M} is satisfied.
		
		Let
		\begin{equation}\label{eq:ESF-U'}
			U_i'\coloneqq U_i\cap V(M_i') \quad \text{and} \quad U_{i+1}'\coloneqq U_{i+1}\cap V(M_i').
		\end{equation}
		By \cref{def:CF-P}, \cref{def:CF-P*}, \cref{def:CF-M}, and \cref{lm:ESF-Mi}, $U_i'$ and $U_{i+1}'$ are unions of $k\ell^*-k'$ clusters in $\cP_i^*$ and $\cP_{i+1}^*$, respectively. Moreover, note for later that \cref{lm:ESF-localised} and \cref{fact:LSC} imply that 
		\begin{equation}\label{eq:ESF-U'M}
			V(M_{\ell,h,i,j})\cap (U_i'\cup U_{i+1}')=\emptyset.
		\end{equation}
		For each $i'\in [4]\setminus \{i,i+1\}$, let $U_{i'}'\subseteq U_i$ be the union of $k\ell^*-k'$ clusters in $\cP_{i'}^*\setminus \{V_{i'},X_{i'}\}$ and note for later that, by \cref{eq:ESF-Vi'}, we have
		\begin{equation}\label{eq:ESF-U'M2}
			V(M_{\ell,h,i,j})\cap U_{i'}'=\emptyset.
		\end{equation} 
		By \cref{def:CF-U,def:CF-P,def:CF-P*}, we have 
		\begin{equation}\label{eq:ESF-n'}
			n'\coloneqq |U_1'|=\dots= |U_4'|=n-\frac{k'n}{k\ell^*}.
		\end{equation} 
		Denote $\cU'\coloneqq (U_1', \dots, U_4')$.
		By \cref{lm:ESF-supreg}, \cref{eq:ESF-Delta}, and \cref{prop:epsremovingadding}\cref{prop:epsremovingadding-supreg},%
		    \COMMENT{Applied with $A=A'=U_{i'}'$ and $B=B'=U_{i'+1}'$.}
		$D'[U_{i'}', U_{i'+1}']$ is $[\varepsilon_2, \geq 1-\varepsilon_2]$-superregular for each $i'\in [4]$.
		Let $\sP_1$ be the set of vertex-disjoint paths obtained by applying \cref{lm:regularitypaths} with $D'[\bigcup\cU'], U_{i+1}', \dots, U_i', 4, n', \varepsilon_2, 1-\varepsilon_2$, and $M_i'$ playing the roles of $D, V_1, \dots, V_k, k, m, \varepsilon, d$, and $\{u_1v_1, \dots, u_mv_m\}$. Then, $V(\sP_1)=\bigcup\cU'$ and $\sP_1$ consists of a $(u,v)$-path for each $uv\in M_i'$, as desired.
		
		\item \textbf{Incorporating $M_{\ell,h,i,j}$.}\label{step:ESF-M}
		In order to satisfy \cref{lm:ESF-IH-M}, we will now use \cref{lm:Meps} to construct a set $\sP_2$ of vertex-disjoint paths which cover all the edges in $M_{\ell,h,i,j}$. 
		
		Recall from \cref{step:ESF-Mi}, \cref{eq:ESF-defXi}, and \cref{eq:ESF-defXi'} that, for each $i'\in [4]$, $U_{i'}'$ is the union of $k\ell^*-k'$ clusters in $\cP_{i'}^*$ and $X_{i'}\subseteq U_{i'}\setminus U_{i'}'$ is a cluster in $\cP_{i'}^*$. In particular, \cref{def:CF-P,def:CF-P*} imply that for each $i'\in [4]$, $U_{i'}\setminus U_{i'}'$ is the union of $k'$ clusters in $\cP_{i'}^*$, so
		\begin{equation}\label{eq:ESF-U-U'}
			|U_1\setminus U_1'|=\dots =|U_4\setminus U_4'|=\frac{k' n}{k\ell^*}
		\end{equation}
		and $Z\coloneqq \bigcup_{i'\in [4]\setminus \{i,i+1\}} X_{i'}$ satisfies
		\begin{equation}\label{eq:ESF-M}
			\frac{|M_{\ell,h,i,j}|}{\varepsilon_2}\stackrel{\text{\cref{lm:ESF-small}}}{\leq}\frac{n}{k\ell^*}=|X_1|=\dots=|X_4|\leq |Z|\leq \frac{\varepsilon_1k' n}{k\ell^*}.
		\end{equation}		
		By \cref{lm:ESF-supreg}, \cref{eq:ESF-Delta}, and \cref{eq:ESF-U-U'}, each $i'\in [4]$ satisfies
		\begin{equation*}
			\delta(D'[U_{i'}\setminus U_{i'}', U_{i'+1}\setminus U_{i'+1}'])\geq (1-2\varepsilon')\frac{k'n}{k\ell^*}-\frac{\varepsilon_1 n}{k\ell^*}\geq (1-\varepsilon_1)\frac{k'n}{k\ell^*}.
		\end{equation*}
		Moreover, \cref{eq:ESF-Xi,eq:ESF-Xi',eq:ESF-U'M,eq:ESF-U'M2} imply that 
		\[V(M_{\ell,h,i,j})\subseteq \bigcup \cU\setminus \left(\bigcup_{i'\in [4]}X_{i'}\cup \bigcup\cU'\right).\]
		For each $i'\in [4]$, let $n_{i'}^+$ and $n_{i'}^-$ be the number of edges of $M_{\ell,h,i,j}$ which start and end in $U_{i'}$, respectively. Since by \cref{lm:ESF-balanced} $SC_{\ell,h,i,j}$ is $\cU$-balanced, \cref{def:MSC,def:BSC} imply that $n_{i'}^+=n_{i'+1}^-$ for each $i'\in [4]$.
		Thus, we can let $\sP_2$ be the set of vertex-disjoint paths obtained by applying \cref{lm:Meps} with $D'-\bigcup\cU', U_{i+1}\setminus U_{i+1}', U_{i+2}\setminus U_{i+2}', \dots, U_{i-1}\setminus U_{i-1}', U_i\setminus U_i', X_{i+1}, X_i, \varepsilon_1, \frac{k'n}{k\ell^*}$, and $M_{\ell,h,i,j}$ playing the roles of $D, U_1, \dots, U_4, X, Y,\varepsilon, n$, and~$M$.
		
		\item \textbf{Covering the remaining vertices.}\label{step:ESF-spanning}
		In order to satisfy \cref{prop:ESPS}\cref{prop:ESPS-V}, we will now use \cref{lm:regularitypaths,cor:regularityHam} to construct a set $\sP_3$ of vertex-disjoint paths which cover all the vertices in $\bigcup\cU\setminus V(\sP_1\cup \sP_2)$.
		
		By \cref{eq:ESF-M} and \cref{lm:Meps}\cref{lm:Meps-V}, there exist $x\in X_{i+1}\setminus V(\sP_2)$ and $y\in X_i\setminus V(\sP_2)$. Denote $U_{i+1}''\coloneqq X_{i+1}\setminus (V(\sP_2)\cup \{x\})$ and $U_i''\coloneqq X_i\setminus (V(\sP_2)\cup \{y\})$. Let $U_{i+1}^*\coloneqq U_{i+1}\setminus (V(\sP_1\cup \sP_2)\cup U_{i+1}'')$ and $U_i^*\coloneqq U_i\setminus (V(\sP_1\cup \sP_2)\cup U_i'')$.
		By \cref{step:ESF-Mi} and \cref{lm:Meps}\cref{lm:Meps-V,lm:Meps-balanced}, we have
		\begin{align}\label{eq:ESF-n''}
			n''\coloneqq |U_{i+1}''|=|U_i''|= |X_i|-|M_{\ell,h,i,j}|-1\stackrel{\text{\cref{eq:ESF-M},\cref{lm:ESF-small}}}{\geq} \frac{n}{k\ell^*}-\varepsilon n-1\geq (1-\varepsilon_1)\frac{n}{k\ell^*}
		\end{align}
		and
		\begin{align}\label{eq:ESF-n*}
			n^*\coloneqq |U_{i+1}^*|=|U_i^*|=n-n'-|X_i|+1\stackrel{\text{\cref{eq:ESF-n'},\cref{eq:ESF-M}}}{\geq} (1-\varepsilon_1)\frac{k'n}{k\ell^*}.
		\end{align}
		For each $i'\in [4]\setminus \{i,i+1\}$, \cref{step:ESF-Mi} and \cref{lm:Meps}\cref{lm:Meps-V} imply that $X_{i'}\cap V(\sP_1\cup \sP_2)=\emptyset$ and so we can let $U_{i'}''\subseteq X_{i'}$ satisfy $|U_{i'}''|=n''$. For each $i'\in [4]\setminus \{i,i+1\}$, let $U_{i'}^*\coloneqq U_{i'}\setminus V(\sP_1\cup \sP_2\cup U_{i'}'')$ and observe that, by \cref{step:ESF-Mi} and \cref{lm:Meps}\cref{lm:Meps-balanced}, we have $|U_{i'}^*|=n^*$. Denote $\cU''\coloneqq (U_1'', \dots, U_4'')$ and $\cU^*\coloneqq (U_1^*, \dots, U_4^*)$.
		
		\begin{claim}
		    For each $i'\in[4]$, $D'[U_{i'}'', U_{i'+1}'']$ and $D'[U_{i'}^*, U_{i'+1}^*]$ are both $[\varepsilon_2, \geq 1-\varepsilon_2]$-superregular.
		\end{claim}
		
		\begin{proofclaim}
		    Let $i'\in[4]$. Recall that $X_{i'}$ and $X_{i'+1}$ are clusters of size $\frac{n}{k\ell^*}$ in $\cP_{i'}^*$ and $\cP_{i'+1}^*$, respectively. Thus, \cref{lm:ESF-supreg} implies that $D[X_{i'}, X_{i'+1}]$ is $[\varepsilon', \geq 1-\varepsilon']$-superregular. By \cref{eq:ESF-n''}, $U_{i'}''$ and $U_{i'+1}''$ are obtained from $X_{i'}$ and $X_{i'+1}$ by deleting at most $\varepsilon_1 |X_{i'}|=\varepsilon_1 |X_{i'+1}|$ vertices. Thus, \cref{eq:ESF-Delta} and \cref{prop:epsremovingadding}\cref{prop:epsremovingadding-supreg} imply that $D'[U_{i'}'', U_{i'+1}'']$ is still $[\varepsilon_2, \geq 1-\varepsilon_2]$-superregular.
		    
		    Similarly, recall from \cref{step:ESF-Mi} that $U_{i'}\setminus V(\sP_1)$ and $U_{i'+1}\setminus V(\sP_1)$ are the unions of $k'$ clusters of size $\frac{n}{k\ell^*}$ in $\cP_{i'}^*$ and $\cP_{i'+1}^*$, respectively. Thus, \cref{lm:ESF-supreg} implies that $D[U_{i'}\setminus V(\sP_1),U_{i'+1}\setminus V(\sP_1)]$ is $[\varepsilon', \geq 1-\varepsilon']$-superregular. By \cref{eq:ESF-n*}, $U_{i'}^*$ and $U_{i'+1}^*$ are obtained from $U_{i'}\setminus V(\sP_1)$ and $U_{i'+1}\setminus V(\sP_1)$ by deleting at most $\varepsilon_1 |U_{i'}\setminus V(\sP_1)|=\varepsilon_1 |U_{i'+1}\setminus V(\sP_1)|$ vertices. Thus, \cref{eq:ESF-Delta} and \cref{prop:epsremovingadding}\cref{prop:epsremovingadding-supreg} imply that $D'[U_{i'}^*, U_{i'+1}^*]$ is still $[\varepsilon_2, \geq 1-\varepsilon_2]$-superregular.
		\end{proofclaim}
		
		Let $u_1, \dots, u_{n''}$ and $v_1, \dots, v_{n''}$ be enumerations of $U_{i+1}''$ and $U_i''$.
		Let $\sP_3'$ be the set of vertex-disjoint paths obtained by applying \cref{lm:regularitypaths} with $D'[\bigcup\cU''], U_{i+1}'', \dots, U_i'', 4, n'', \varepsilon_2$, and $1-\varepsilon_2$ playing the roles of $D, V_1, \dots, V_k, k, m, \varepsilon$, and $d$.
		Apply \cref{cor:regularityHam} with $D'[\bigcup\cU^*], U_{i+1}^*, \dots$, $U_i^*, 4, n^*, \varepsilon_2, 1-\varepsilon_2, x$, and $y$ playing the roles of $D, V_1, \dots, V_k, k, m, \varepsilon, d, u$, and $v$ to obtain a Hamilton $(x,y)$-path $P$ of $D'[\bigcup \cU^*]$. Denote $\sP_3\coloneqq \sP_3'\cup \{P\}$. By \cref{lm:regularitypaths} and \cref{cor:regularityHam}, $\sP_3$ is a set of vertex-disjoint paths satisfying $V(\sP_3)=\bigcup\cU\setminus V(\sP_1\cup \sP_2)$, $V^+(\sP_3)=X_{i+1}\setminus V(\sP_2)$ and $V^-(\sP_3)=X_i\setminus V(\sP_2)$.
	\end{steps}
	
	Let $ESPS_{\ell,h,i,j}\coloneqq \sP_1\cup \sP_2\cup \sP_3$ and denote $\mathcal{ESPS}_{t+1}\coloneqq \mathcal{ESPS}_t\cup \{ESPS_{\ell,h,i,j}\}$. Then, \cref{lm:ESF-IH-M} holds by \cref{lm:Meps}\cref{lm:Meps-M} and definition of $D'$, while \cref{lm:ESF-IH-disjoint} holds by definition of $D'$.
	It remains to show that \cref{lm:ESF-IH-ESPS} holds. By construction, $ESPS_{\ell,h,i,j}$ is a linear forest and so \cref{prop:ESPS} implies that it is enough to verify that $ESPS_{\ell,h,i,j}$ satisfies \cref{prop:ESPS}\cref{prop:ESPS-M,prop:ESPS-V,prop:ESPS-V+,prop:ESPS-V-}.
	Note that \cref{prop:ESPS}\cref{prop:ESPS-V} follows from \cref{step:ESF-spanning} and \cref{prop:ESPS}\cref{prop:ESPS-M} follows from \cref{step:ESF-Mi}. 
	By construction, \cref{eq:ESF-defXi}, and \cref{eq:ESF-U'}, we have $V^+(ESPS_{\ell,h,i,j})= U_{i+1}'\cup X_{i+1}=U_{i+1}\setminus N_{M_i}(W_{2,h}\cup\dots \cup W_{k',h})$ and $V^-(ESPS_{\ell,h,i,j})= U_i'\cup X_i=U_i\setminus (W_{1,h}\cup\dots \cup W_{k'-1,h})$.
	Thus, \cref{prop:ESPS}\cref{prop:ESPS-V+,prop:ESPS-V-} hold and we are done.
\end{proof}

\subsection{Constructing a cycle-setup}\label{sec:CST}
Finally, we build the cycle-setup required for the robust decomposition lemma for blow-up cycles (\cref{lm:cyclerobustdecomp}). Given a very dense blow-up $C_4$, say $D$, one could of course construct a cycle-setup by first randomly partitioning the vertices of $D$ and then use \cref{lm:Chernoff} to exhibit the desired (super)regular pairs for \cref{def:CST-ST}. However, our cycle-setup will need to satisfy additional properties.
\begin{itemize}
	\item To construct the extended special factors, $\cP^*$ and $\cM$ need to satisfy \cref{lm:ESF}\cref{lm:ESF-supreg,lm:ESF-Mi}. (This motivates \cref{lm:CST}\cref{lm:CST-M,lm:CST-supreg} below.)
	\item After constructing the extended special factors, the clusters in $\bigcup\cP$ and $\bigcup\cP'$ may no longer form suitable (super)regular pairs%
	    \COMMENT{The $\varepsilon$-parameter gets too large for the robust decomposition lemma hierarchy.}.
	To solve this problem, we randomly partition the edges of $D$ into a dense digraph $D_1$ and a sparse digraph $D_2$. We will only use $D_1$ to construct the extended special factors and reserve $D_2$ for the application of \cref{lm:cyclerobustdecomp}. (This motivates \cref{lm:CST}\cref{lm:CST-CST} below.)
	(Recall that a similar strategy was used in the robust outexpander case, see \cref{sec:biprobexp}.)
	\item To be able to construct the localised and balanced special covers required for \cref{lm:ESF}, we will need the backward edges of our bipartite tournament $T$ to be well distributed across the clusters in $\cP$ and $\cP^*$. Since there may be relatively few backward edges, this cannot be guaranteed via a simple application of \cref{lm:Chernoff}. This explains why, in \cref{lm:CST}, we construct a cycle-setup with respect to given sets of partitions $\cP$ and $\cP^*$. (To help the reader gain intuition for this step, we will only detail the construction of $\cP$ and $\cP^*$ after we have discussed our strategy for decomposing backward edges.)
\end{itemize}
Note that in \cref{lm:CST}, we assume that the minimum semidegree of $D$ is very large. While $\overrightarrow{T}_\cU$ (that is, the digraph which consists of all the forward edges of $T$ (see \cref{sec:forwardbackwarddef})) is very dense, its minimum semidegree may be low (the backward edges may be concentrated on a few vertices). To solve this problem, we will assign the few vertices of large backward degree into the exceptional set $U^*$ and then apply \cref{lm:CST} with $\overrightarrow{T}_\cU-U^*$ playing the role of $D$. This will ensure that the minimum semidegree condition in \cref{lm:CST} is satisfied.

For technical reasons, we will need the matchings in $\cM$ to satisfy a stronger property than \cref{lm:ESF}\cref{lm:ESF-Mi}. Roughly speaking, \cref{lm:ESF}\cref{lm:ESF-Mi} states that, for each $i\in [4]$, $M_i$ matches the clusters in $\cP_i^*$ and $\cP_{i+1}^*$.  It will be convenient that each $M_i$ matches ``corresponding" clusters together.

\begin{definition}[\Gls*{consistent cycle-framework}]\label{def:CCF}
	We say that a $(4, \ell^*, k, n)$-cycle-framework $(\cU, \cP, \cP^*, \cC, \cM)$ is \emph{consistent} if the following holds for each $(h,i,j)\in [\ell^*]\times[4]\times [k]$. Let $V_{i,j}$ and $V_{i+1,j}$ denote the $j$\textsuperscript{th} clusters in $\cP_i$ and $\cP_{i+1}$, respectively. Let $V_{i,j,h}$ and $V_{i+1,j,h}$ denote the $h$\textsuperscript{th} subclusters of $V_{i,j}$ and $V_{i+1,j}$ contained in $\cP_i^*$ and $\cP_{i+1}^*$, respectively. Then, $N_{M_i}(V_{i,j,h})=V_{i+1,j,h}$.
\end{definition}

Recall from \cref{fact:CFP} that a cycle-framework remains a cycle-framework when $\cP^*$ is replaced by $\cP$. Observe that consistency is also preserved.

\begin{fact}\label{fact:CCFP}
	Suppose that $(\cU, \cP, \cP^*, \cC, \cM)$ is a consistent $(4, \ell^*, k, n)$-cycle-framework. Then, $(\cU, \cP, \cP, \cC, \cM)$ is also a consistent $(4, 1, k, n)$-cycle-framework.
\end{fact}

We are now ready to construct our cycle-setup.

\begin{lm}\label{lm:CST}
	Let $0<\frac{1}{n}\ll \varepsilon\ll \varepsilon'\ll\frac{1}{k}  \ll \frac{1}{\ell'}, \frac{1}{\ell^*},d\ll 1$ and denote $m\coloneqq \frac{n}{k}$. Suppose that $\frac{m}{\ell'}, \frac{m}{\ell^*}\in \mathbb{N}$.
	Let $U_1, \dots, U_4$ be disjoint vertex sets of size $n$ and denote $\cU\coloneqq (U_1, \dots, U_4)$. Let $D$ be a blow-up $C_4$ with vertex partition $\cU$. Suppose that $\delta^0(D)\geq (1-\varepsilon)n$. For each $i\in [4]$, let $\cP_i$ be a partition of $U_i$ into an empty exceptional set and $k$ clusters of size $m$, let $\cP_i^*$ be an $\ell^*$-refinement of $\cP_i$, and let $C^i$ be a Hamilton cycle on the clusters in $\cP_i$. Denote $\cP\coloneqq(\cP_1, \dots, \cP_4)$, $\cP^*\coloneqq(\cP_1^*, \dots, \cP_4^*)$, and $\cC\coloneqq(\cC^1, \dots, \cC^4)$. Then, there exist $D_1, \cP, \cP',\cR, \cC, \sU, \sU'$, and $\cM$ for which the following hold, where $D_2\coloneqq D\setminus D_1$.
	\begin{enumerate}
		\item $(\cU, \cP, \cP^*, \cC, \cM)$ is a consistent $(4, \ell^*, k, n)$-cycle-framework. In particular, the following hold. For any $i\in [4]$ and any cluster $V\in \cP_i$, the set $N_{M_i}(V)$ is a cluster in $\cP_{i+1}$ (where $\cP_5\coloneqq \cP_1$). The analogue holds for the partitions in $\cP^*$.\label{lm:CST-M}
		\newcounter{CST}
		\setcounter{CST}{\value{enumi}}
		\item For each $i\in [4]$, $D_1[V,W]$ is $[\varepsilon',\geq 1-3d]$-superregular whenever $V\subseteq U_i$ and $W\subseteq U_{i+1}$ are unions of clusters in $\cP_i^*$ and $\cP_{i+1}^*$, respectively. In particular, since $\cP_i^*$ is a refinement of $\cP_i$ for each $i\in [4]$, the analogue holds for the partitions in $\cP$.\label{lm:CST-supreg}
		\item $(D_2, \cU, \cP, \cP', \cP^*, \cR, \cC, \sU, \sU', \cM)$ is a $(4, \ell', \ell^*, k, m, \varepsilon', d)$-cycle-setup.\label{lm:CST-CST}
	\end{enumerate}
\end{lm}

\begin{proof}
    Fix additional constants $\varepsilon_1, \varepsilon_2$, and $\varepsilon_3$ satisfying $\varepsilon\ll \varepsilon_1\ll \varepsilon_2\ll \varepsilon_3 \ll \varepsilon'$.
    First, we construct the matchings in $\cM$. For each $i\in [4]$, let $V_{i,1}, \dots, V_{i,k}$ be an enumeration of the clusters in $\cP_i$ and, for each $(h,j)\in [\ell^*]\times [k]$, denote by $V_{i,j,h}$ the $h$\textsuperscript{th} subcluster of $V_{i,j}$ contained in $\cP_i^*$. For each $(h,i,j)\in [\frac{q}{f}]\times [4]\times [k]$, let $M_{h,i,j}$ be an auxiliary perfect matching from $V_{i+1,j,h}$ to $V_{i,j,h}$. For each $i\in [4]$, define $M_i\coloneqq \bigcup_{(h,j)\in [\frac{q}{f}]\times [k]}M_{h,i,j}$. Let $\cM\coloneqq (M_1, \dots, M_4)$ and observe that \cref{lm:CST-M} and \cref{def:CST-M} hold.
    
    Let $i\in [4]$. Denote by $\tD_i$ the $M_i$-contraction of $D[U_i, U_{i+1}]$. By \cref{fact:Ncontract}\cref{fact:Ncontract-N}, $\delta^0(\tD_i)\geq (1-2\varepsilon)n$ and so \cref{lm:URefdense} implies that $\cP_i^*$ is an $\sqrt{2\varepsilon}$-uniform refinement of $\cP_i$ with respect to $\tD_i$.
    
    Let $i\in [4]$ and let $V\subseteq U_i$ and $W\subseteq U_{i+1}$ be unions of clusters in $\cP_i^*$ and $\cP_{i+1}^*$, respectively. Then, observe that each $v\in V$ satisfies $|N_D^+(v)\cap W|\geq |W|-\varepsilon n\geq (1-\varepsilon_1)|W|$ and, similarly, each $w\in W$ satisfies $|N_D^-(v)\cap V|\geq |V|-\varepsilon n\geq (1-\varepsilon_1)|V|$. Thus, \cref{prop:almostcompleteeps} implies that $D[V,W]$ is $[\varepsilon_2, \geq 1-\varepsilon_2]$-superregular.
    
	Let $D_1$ be obtained from $D$ by selecting each edge independently with probability $1-2d$. Denote $D_2\coloneqq D\setminus D_1$.
	For each $i\in [4]$, denote by $\tD_i'$ the $M_i$-contraction of $D_2[U_i, U_{i+1}]$ and observe that, by definition, $\tD_i'$ is obtained from $\tD_i$ by selecting each edge independently with probability $2d$.
	Thus, \cref{lm:edgeslice,lm:URefrandom} imply that we may assume that the following hold.
	\begin{enumerate}[label=(\roman*$'$)]
		\setcounter{enumi}{\value{CST}}
		\item For each $i\in [4]$, $D_1[V,W]$ is $[\varepsilon_3,\geq 1-3d]$-superregular and $D_2[V,W]$ is $[\varepsilon_3,\geq d+\varepsilon_3]$-superregular whenever $V\subseteq U_i$ and $W\subseteq U_{i+1}$ are unions of clusters in $\cP_i^*$ and $\cP_{i+1}^*$, respectively.\label{lm:CST-supreg'}
		\item For each $i\in [4]$, $\cP_i^*$ is an $\varepsilon'$-uniform $\ell^*$-refinement of $\cP_i$ with respect to $\tD_i'$.\label{lm:CST-URef}
	\end{enumerate}
	In particular, \cref{lm:CST-supreg} is satisfied.
	
	We now construct the cycle-setup. First, observe that since $D$ is a blow-up $C_4$ with vertex partition $\cU$, $D_2$ satisfies \cref{def:CST-D}.
	Let $\cP_i'$ be the $\varepsilon$-uniform $\ell'$-refinement of $\cP_i$ obtained by applying \cref{lm:URefexistence} with $\tD_i', \cP_i$, and $\ell'$ playing the roles of $D, \cP$, and $\ell$.
	Let $R_i$ be the complete digraph on the clusters in $\cP_i$ and let $C^i\coloneqq V_{i,1} \dots V_{i,k}$. Note that $C^i$ is a Hamilton cycle of $R_i$. Let $U^i$ be the universal walk for $C^i$ with parameter $\ell'$ obtained by applying \cref{lm:U} with $R_i$ and $C^i$ playing the roles of $R$ and $C$. Let $U'^i$ be the closed walk on the clusters in $\cP_i'$ obtained from $U^i$ as described in \cref{def:ST-U'}.
	
	Denote $\cP'\coloneqq (\cP_1', \dots, \cP_4')$, $\cR\coloneqq (R_1, \dots, R_4)$, $\sU\coloneqq (U^1, \dots, U^4)$, and $\sU'\coloneqq (U'^1, \dots, U'^4)$. Let $i\in [4]$. 
	We need to show that $(\tD_i', \cP_i, \cP_i', \cP_i^*, R_i, C^i, U^i, U'^i)$ is an $(\ell',\ell^*, k, m, \varepsilon', d)$-setup. By construction and \cref{lm:CST-URef}, properties \cref{def:ST-P,def:ST-U,def:ST-U',def:ST-P*,def:ST-P'} are satisfied.
	Moreover, \cref{lm:CST-supreg'} and \cref{lm:contraction}\cref{lm:contraction-supreg} imply that $\tD_i'[V,W]$ is $[\varepsilon_3, \geq d+\varepsilon_3]$-superregular for any distinct clusters $V, W\in \cP_i$. Thus, \cref{def:ST-R,def:ST-C} are satisfied. 
	Moreover, \cref{lm:URefreg}\cref{lm:URef-supreg} implies that $\tD_i'[V,W]$ is $[\varepsilon', \geq d]$-superregular for any distinct clusters $V,W\in \cP_i'$. Therefore, \cref{def:ST-U'supreg} holds and so $(\tD_i', \cP_i, \cP_i', \cP_i^*, R_i, C^i, U^i, U'^i)$ is an $(\ell',\ell^*, k, m, \varepsilon', d)$-setup. By construction, the exceptional set in $\cP_i, \cP_i'$, and $\cP_i^*$ is empty and so \cref{def:CST-ST} holds.
	Thus, $(D_2, \cU, \cP, \cP', \cP^*, \cR, \cC, \sU, \sU', \cM)$ is a $(4, \ell', \ell^*,k, m, \varepsilon', d)$-cycle-setup and so \cref{lm:CST-CST} is satisfied.
\end{proof}

\onlyinsubfile{\bibliographystyle{abbrv}
	\bibliography{Bibliography/Bibliography}}

\section{Decomposing the backward edges}\label{sec:backward}

As briefly mentioned in the proof overview, the backward edges of $T$ will be decomposed separately at the beginning of the proof of \cref{thm:blowupC4}. We now discuss this in more detail.

\subsection{Feasible systems}\label{sec:feasible}

	\onlyinsubfile{
		\setcounter{section}{9}
		\setcounter{subsection}{1}
		\subsection{Feasible systems}}

The strategy for \cref{thm:blowupC4} will be to first decompose all the backward edges into $n$ edge-disjoint subdigraphs $\cF_1, \dots, \cF_n$ and then incorporate each $\cF_i$ into a distinct Hamilton cycle of the decomposition of $T$.
Each $\cF_i$ will need to have a very specific structure, which will be called a \emph{feasible system}; otherwise we would not be able to incorporate it into a Hamilton cycle. (See \cref{def:feasible} below for a formal definition.) To gain intuition, we start by giving some informal motivation.

First, each $\cF_i$ will have to be linear forest (since any proper subdigraph of a Hamilton cycle is a linear forest). This is property \cref{def:feasible-linforest} below. 
Moreover, we will show that for any Hamilton cycle of $T$, the set of backward edges satisfy the following ``balance property".
(To gain intuition behind \cref{cor:Hamcycle}, observe that in the proof of \cref{prop:tripartite}, we in fact showed that, in a tripartite tournament, one cannot construct a Hamilton cycle which contains a single backward edge.
The analogue holds in the blow-up $C_4$ case: a cycle which does not contain a ``balanced" number of backward edges will not cover all the vertex classes equitably.) 
Recall \cref{def:epspartition}.

\begin{prop}\label{cor:Hamcycle}
	Let $T$ be a bipartite tournament on $4n$ vertices and suppose that $\cU=(U_1, \dots, U_4)$ is an $(\varepsilon, 4)$-partition for $T$. Then, any Hamilton cycle $C$ of $T$ satisfies $e_C(U_1,U_4)=e_C(U_3,U_2)$ and $e_C(U_4,U_3)=e_C(U_2,U_1)$.
\end{prop}

\COMMENT{\begin{proof}
		By \cref{prop:cycle}, there exists $\ell\in \mathbb{Z}$ such that $\ell+ e_C(U_1,U_4)+e_C(U_4,U_3)=|U_4|=|U_3|=\ell+ e_C(U_4,U_3)+e_C(U_3,U_2)$ and $\ell+e_C(U_2,U_1)+e_C(U_1,U_4)=|U_1|=|U_4|=\ell+e_C(U_1,U_4)+e_C(U_4,U_3)$.
\end{proof}}

Thus, we will need to make sure that each $\cF_i$ contains the same number of backward edges in non-adjacent pairs of the blow-up $C_4$. This is property \cref{def:feasible-backward} below.
For convenience, we will also allow each $\cF_i$ to contain a few forward edges to ensure that all the exceptional vertices are covered. This is property \cref{def:feasible-exceptional} below. Roughly speaking, this means that we will decompose all the exceptional edges at the same time as the backward edges, which will enable us to ``ignore" the exceptional vertices when constructing our Hamilton cycles.

Altogether, this motivates the next definition.

\begin{definition}[\Gls*{feasible system}]\label{def:feasible}
	Let $U_1, \dots, U_4$ be disjoint vertex sets and denote $\cU\coloneqq(U_1, \dots, U_4)$. Let $U^*\subseteq \bigcup_{i\in [4]}U_i$ be an exceptional set.
	We say that $\cF$ is a \emph{feasible system (with respect to $\cU$ and $U^*$)} if the following hold.
	\begin{enumerate}[label=\rm(F\arabic*)]
		\item $e_\cF(U_1,U_4)=e_\cF(U_3,U_2)$ and $e_\cF(U_4,U_3)=e_\cF(U_2,U_1)$.\label{def:feasible-backward}
		\newcounter{feasible}
		\setcounter{feasible}{\value{enumi}}
		\item For each $v\in U^*$, $d_{\cF}^+(v)=1=d_{\cF}^-(v)$.\label{def:feasible-exceptional}
		\item $\cF$ is a linear forest.\label{def:feasible-linforest}
	\end{enumerate}
\end{definition}

Note that \cref{cor:Hamcycle} follows immediately from \cref{fact:partition}\cref{fact:partition-size} and the next result (which will also be used in the proof of \cref{fact:feasibleSC} below).

\begin{prop}\label{prop:cycle}
	Let $U_1, \dots, U_4$ be disjoint vertex sets (not necessarily of the same size) and let $C$ be a bipartite cycle on vertex classes $U_1\cup U_3$ and $U_2\cup U_4$.
	Then, there exists $\ell\in \mathbb{Z}$ such that, for each $i\in[4]$, $|U_i|=\ell + e_C(U_{i+1},U_i)+e_C(U_i,U_{i-1})$.
\end{prop}

\begin{proof}
	We proceed by induction on $|C|$.
	For the base case, suppose that $|C|=4$.
	If $|U_i|=1$ and $e_C(U_{i+1},U_i)=0$ for each $i\in [4]$, then we can let $\ell\coloneqq 1$ and we are done. If $|U_i|=1$ and $e_C(U_{i+1},U_i)=1$ for each $i\in [4]$, then we can let $\ell\coloneqq -1$ and we are done.
	Suppose that there exists $i\in [4]$ such that $C=v_1v_2v_3v_4$ for some $v_1\in U_i$, $v_2,v_4\in U_{i+1}$, and $v_3\in U_{i+2}$. Then, each $j\in [4]$ satisfies
	\begin{align*}
		|U_j|=\begin{cases}
			2 & \text{if } j=i+1;\\
			1 & \text{if } j\in \{i,i+2\};\\
			0 & \text{otherwise;}
		\end{cases}
		\quad\text{and}\quad
		e_C(U_{j+1},U_j)=\begin{cases}
			1 & \text{if } j\in  \{i, i+1\};\\
			0 & \text{otherwise}.
		\end{cases}
	\end{align*}
	Thus, we can let $\ell\coloneqq 0$ and we are done.
	We may therefore assume that there exists $i\in [4]$ such that $C=v_1v_2v_3v_4$ for some $v_1,v_3\in U_i$ and $v_2,v_4\in U_{i+1}$.
	Then, each $j\in [4]$ satisfies
	\begin{align*}
		|U_j|=\begin{cases}
			2 & \text{if } j\in \{i,i+1\};\\
			0 & \text{otherwise;}
		\end{cases}
		\quad\text{and}\quad
		e_C(U_{j+1},U_j)=\begin{cases}
			2 & \text{if } j=i;\\
			0 & \text{otherwise}.
		\end{cases}
	\end{align*}
	Thus, we can let $\ell\coloneqq 0$ and we are done.
	
	For the induction step, let $k>2$ and suppose that the \lcnamecref{prop:cycle} holds for any cycle of length $2(k-1)$. Assume that $|C|=2k$ and denote $C=v_1v_2\dots v_{2k}$. Suppose without loss of generality that $v_1\in U_1$. Then, observe that $v_{2k-2},v_{2k}\in U_2\cup U_4$ and $v_{2k-1}\in U_1\cup U_3$.
	For each $i\in [4]$, let $U_i'\coloneqq U_i\setminus \{v_{2k-1},v_{2k}\}$. Define a cycle $C'\coloneqq v_1\dots v_{2k-2}$. Then, $|C'|= 2(k-1)$ and $C'$ is a bipartite cycle on vertex classes $U_1'\cup U_3'$ and $U_2'\cup U_4'$. Thus, by the induction hypothesis, there exists $\ell'\in \mathbb{Z}$ such that $|U_i'|=\ell'+e_{C'}(U_{i+1}', U_i')+e_{C'}(U_i', U_{i-1}')$ for each $i\in [4]$.
	If $v_{2k-2}v_{2k-1}, v_{2k-1}v_{2k}$, and $v_{2k}v_1$ are all forward edges with respect to $\cU\coloneqq (U_1, \dots, U_4)$, then let $\ell\coloneqq \ell'+1$. If $v_{2k-2}v_{2k-1}, v_{2k-1}v_{2k}$, and $v_{2k}v_1$ are all backward edges with respect to $\cU$, then let $\ell\coloneqq \ell'-1$. Otherwise, let $\ell\coloneqq \ell'$.
	
	We now verify that $|U_i|=\ell+e_C(U_{i+1}, U_i)+e_C(U_i, U_{i-1})$ for each $i\in [4]$. We consider the case where $v_{2k-2}v_{2k-1}, v_{2k-1}v_{2k}$, and $v_{2k}v_1$ are all forward edges with respect to $\cU$ (the other cases can be verified with similar arguments).%
	\COMMENT{Case $v_{2k-2}v_{2k-1}, v_{2k-1}v_{2k}, v_{2k}v_1$ are all backward edges. Note that $v_{2k}\in U_2$, $v_{2k-1}\in U_3$, and $v_{2k-2}\in U_4$. Thus,
		\begin{equation*}
			|U_i'|=
			\begin{cases}
				|U_i| & \text{if }i\in \{1,4\};\\
				|U_i|-1 & \text{otherwise};
			\end{cases}
			\quad \text{and} \quad
			e_{C'}(U_{i+1}', U_i')=
			\begin{cases}
				e_C(U_{i+1}, U_i) & \text{if }i=4;\\
				e_C(U_{i+1}, U_i)-1 & \text{otherwise}.
			\end{cases}
		\end{equation*}
		Then, for each $i\in \{1,4\}$,
		\begin{align*}
			|U_i|&=|U_i'|=\ell' +e_{C'}(U_{i+1}', U_i')+e_{C'}(U_i', U_{i-1}')\\
			&=\ell+1+e_{C'}(U_{i+1}', U_i')+e_{C'}(U_i', U_{i-1}')
			=\ell +e_C(U_{i+1}, U_i)+e_C(U_i, U_{i-1}).
		\end{align*}
		Moreover, for each $i\in \{2,3\}$, 
		\begin{align*}
			|U_i|&=|U_i'|+1=\ell'+1 +e_{C'}(U_{i+1}', U_i')+e_{C'}(U_i', U_{i-1}')\\
			&=\ell+2 +e_{C'}(U_{i+1}', U_i')+e_{C'}(U_i', U_{i-1}')=\ell+e_C(U_{i+1}, U_i)+e_C(U_i, U_{i-1}),
		\end{align*}
		so we are done.}%
	\COMMENT{Case $v_{2k-2}v_{2k-1}$ and $v_{2k-1}v_{2k}$ are forward edges but $v_{2k}v_1$ is a backward edge.
		Then, note that $v_{2k}\in U_2$, $v_{2k-1}\in U_1$, and $v_{2k-2}\in U_4$. Thus,
		\begin{equation*}
			|U_i'|=
			\begin{cases}
				|U_i|-1 & \text{if }i\in [2];\\
				|U_i| & \text{otherwise};
			\end{cases}
			\quad \text{and} \quad
			e_{C'}(U_{i+1}', U_i')=
			\begin{cases}
				e_C(U_{i+1}, U_i)-1 & \text{if }i=1;\\
				e_C(U_{i+1}, U_i) & \text{otherwise}.
			\end{cases}
		\end{equation*}
		Then, for each $i\in [2]$,
		\begin{align*}
			|U_i|&=|U_i'|+1=\ell' +1+e_{C'}(U_{i+1}', U_i')+e_{C'}(U_i', U_{i-1}')\\
			&=\ell+e_C(U_{i+1}, U_i)+e_C(U_i, U_{i-1}).
		\end{align*}
		Moreover, for each $i\in [4]\setminus [2]$, 
		\begin{align*}
			|U_i|&=|U_i'|=\ell' +e_{C'}(U_{i+1}', U_i')+e_{C'}(U_i', U_{i-1}')\\
			&=\ell+e_C(U_{i+1}, U_i)+e_C(U_i, U_{i-1}),
		\end{align*}
		so we are done.}%
	\COMMENT{Case $v_{2k-2}v_{2k-1}$ and $v_{2k-1}v_{2k}$ are backward edges but $v_{2k}v_1$ is a forward edge. Note that $v_{2k}\in U_4$, $v_{2k-1}\in U_1$, and $v_{2k-2}\in U_2$. Thus,
		\begin{equation*}
			|U_i'|=
			\begin{cases}
				|U_i|-1 & \text{if }i\in \{1,4\};\\
				|U_i| & \text{otherwise};
			\end{cases}
			\quad \text{and} \quad
			e_{C'}(U_{i+1}', U_i')=
			\begin{cases}
				e_C(U_{i+1}, U_i)-1 & \text{if }i=4;\\
				e_C(U_{i+1}, U_i) & \text{otherwise}.
			\end{cases}
		\end{equation*}
		Then, for each $i\in \{1,4\}$,
		\begin{align*}
			|U_i|&=|U_i'|+1=\ell' +1+e_{C'}(U_{i+1}', U_i')+e_{C'}(U_i', U_{i-1}')\\
			&=\ell+e_C(U_{i+1}, U_i)+e_C(U_i, U_{i-1}).
		\end{align*}
		Moreover, for each $i\in \{2,3\}$, 
		\begin{align*}
			|U_i|&=|U_i'|=\ell' +e_{C'}(U_{i+1}', U_i')+e_{C'}(U_i', U_{i-1}')\\
			&=\ell+e_C(U_{i+1}, U_i)+e_C(U_i, U_{i-1}),
		\end{align*}
		so we are done.}%
	\COMMENT{Case $v_{2k-2}v_{2k-1}$ is a forward edge but $v_{2k-1}v_{2k}$ and $v_{2k}v_1$ are backward edges. Note that $v_{2k}\in U_2$, $v_{2k-1}\in U_3$, and $v_{2k-2}\in U_2$. Thus,
		\begin{equation*}
			|U_i'|=
			\begin{cases}
				|U_i|-1 & \text{if }i\in \{2,3\};\\
				|U_i| & \text{otherwise};
			\end{cases}
			\quad \text{and} \quad
			e_{C'}(U_{i+1}', U_i')=
			\begin{cases}
				e_C(U_{i+1}, U_i)-1 & \text{if }i=2;\\
				e_C(U_{i+1}, U_i) & \text{otherwise}.
			\end{cases}
		\end{equation*}
		Then, for each $i\in \{1,4\}$,
		\begin{align*}
			|U_i|&=|U_i'|=\ell' +e_{C'}(U_{i+1}', U_i')+e_{C'}(U_i', U_{i-1}')\\
			&=\ell+e_C(U_{i+1}, U_i)+e_C(U_i, U_{i-1}).
		\end{align*}
		Moreover, for each $i\in \{2,3\}$, 
		\begin{align*}
			|U_i|&=|U_i'|+1=\ell'+1 +e_{C'}(U_{i+1}', U_i')+e_{C'}(U_i', U_{i-1}')\\
			&=\ell+e_C(U_{i+1}, U_i)+e_C(U_i, U_{i-1}),
		\end{align*}
		so we are done.}
	Note that $v_{2k}\in U_4$, $v_{2k-1}\in U_3$, and $v_{2k-2}\in U_2$. Then,
	\begin{equation*}
		|U_i'|=
		\begin{cases}
			|U_i| & \text{if }i\in [2];\\
			|U_i|-1 & \text{otherwise};
		\end{cases}
		\quad \text{and} \quad
		e_{C'}(U_{i+1}', U_i')=
		\begin{cases}
			e_C(U_{i+1}, U_i)+1 & \text{if }i=1;\\
			e_C(U_{i+1}, U_i) & \text{otherwise}.
		\end{cases}
	\end{equation*}
	Then, for each $i\in [2]$,
	\begin{align*}
		|U_i|&=|U_i'|=\ell' +e_{C'}(U_{i+1}', U_i')+e_{C'}(U_i', U_{i-1}')\\
		&=\ell-1+e_{C'}(U_{i+1}', U_i')+e_{C'}(U_i', U_{i-1}')=\ell+e_C(U_{i+1}, U_i)+e_C(U_i, U_{i-1}).
	\end{align*}
	Moreover, for each $i\in [4]\setminus [2]$, 
	\begin{align*}
		|U_i|&=|U_i'|+1=\ell'+1 +e_{C'}(U_{i+1}', U_i')+e_{C'}(U_i', U_{i-1}')\\
		&=\ell+e_C(U_{i+1}, U_i)+e_C(U_i, U_{i-1}),
	\end{align*}
	so we are done.%
	\COMMENT{Case $v_{2k-2}v_{2k-1}$ is a backward edge but $v_{2k-1}v_{2k}$ and $v_{2k}v_1$ are forward edges. Note that $v_{2k}\in U_4$, $v_{2k-1}\in U_3$, and $v_{2k-2}\in U_4$. Thus,
		\begin{equation*}
			|U_i'|=
			\begin{cases}
				|U_i|-1 & \text{if }i\in \{3,4\};\\
				|U_i| & \text{otherwise};
			\end{cases}
			\quad \text{and} \quad
			e_{C'}(U_{i+1}', U_i')=
			\begin{cases}
				e_C(U_{i+1}, U_i)-1 & \text{if }i=3;\\
				e_C(U_{i+1}, U_i) & \text{otherwise}.
			\end{cases}
		\end{equation*}
		Then, for each $i\in [2]$,
		\begin{align*}
			|U_i|&=|U_i'|=\ell' +e_{C'}(U_{i+1}', U_i')+e_{C'}(U_i', U_{i-1}')\\
			&=\ell+e_C(U_{i+1}, U_i)+e_C(U_i, U_{i-1}).
		\end{align*}
		Moreover, for each $i\in[4]\setminus [2]$, 
		\begin{align*}
			|U_i|&=|U_i'|+1=\ell'+1 +e_{C'}(U_{i+1}', U_i')+e_{C'}(U_i', U_{i-1}')\\
			&=\ell+e_C(U_{i+1}, U_i)+e_C(U_i, U_{i-1}),
		\end{align*}
		so we are done.}%
	\COMMENT{Case $v_{2k-2}v_{2k-1}$ and $v_{2k}v_1$ are forward edges but $v_{2k-1}v_{2k}$ is a backward edge. Note that $v_{2k}\in U_4$, $v_{2k-1}\in U_1$, and $v_{2k-2}\in U_4$. Thus,
		\begin{equation*}
			|U_i'|=
			\begin{cases}
				|U_i|-1 & \text{if }i\in \{1,4\};\\
				|U_i| & \text{otherwise};
			\end{cases}
			\quad \text{and} \quad
			e_{C'}(U_{i+1}', U_i')=
			\begin{cases}
				e_C(U_{i+1}, U_i)-1 & \text{if }i=4;\\
				e_C(U_{i+1}, U_i) & \text{otherwise}.
			\end{cases}
		\end{equation*}
		Then, for each $i\in \{1,4\}$,
		\begin{align*}
			|U_i|&=|U_i'|+1=\ell'+1 +e_{C'}(U_{i+1}', U_i')+e_{C'}(U_i', U_{i-1}')\\
			&=\ell+e_C(U_{i+1}, U_i)+e_C(U_i, U_{i-1}).
		\end{align*}
		Moreover, for each $i\in \{2,3\}$, 
		\begin{align*}
			|U_i|&=|U_i'|=\ell'+e_{C'}(U_{i+1}', U_i')+e_{C'}(U_i', U_{i-1}')\\
			&=\ell+e_C(U_{i+1}, U_i)+e_C(U_i, U_{i-1}),
		\end{align*}
		so we are done.}%
	\COMMENT{Case $v_{2k-2}v_{2k-1}$ and $v_{2k}v_1$ are backward edges but $v_{2k-1}v_{2k}$ is a forward edge. Note that $v_{2k}\in U_2$, $v_{2k-1}\in U_1$, and $v_{2k-2}\in U_2$. Thus,
		\begin{equation*}
			|U_i'|=
			\begin{cases}
				|U_i|-1 & \text{if }i\in [2];\\
				|U_i| & \text{otherwise};
			\end{cases}
			\quad \text{and} \quad
			e_{C'}(U_{i+1}', U_i')=
			\begin{cases}
				e_C(U_{i+1}, U_i)-1 & \text{if }i=1;\\
				e_C(U_{i+1}, U_i) & \text{otherwise}.
			\end{cases}
		\end{equation*}
		Then, for each $i\in [2]$,
		\begin{align*}
			|U_i|&=|U_i'|+1=\ell'+1 +e_{C'}(U_{i+1}', U_i')+e_{C'}(U_i', U_{i-1}')\\
			&=\ell+e_C(U_{i+1}, U_i)+e_C(U_i, U_{i-1}).
		\end{align*}
		Moreover, for each $i\in [4]\setminus [2]$, 
		\begin{align*}
			|U_i|&=|U_i'|=\ell'+e_{C'}(U_{i+1}', U_i')+e_{C'}(U_i', U_{i-1}')\\
			&=\ell+e_C(U_{i+1}, U_i)+e_C(U_i, U_{i-1}),
		\end{align*}
		so we are done.}
\end{proof}

We now state a few useful properties of feasible systems. 
Observe that forward edges are only required to cover the exceptional set $U^*$, so any forward edge which is not incident to $U^*$ may be deleted or added from a feasible system.

\begin{fact}\label{fact:feasibleforward}
	Let $U_1, \dots, U_4$ be disjoint vertex sets. Denote $\cU\coloneqq(U_1, \dots, U_4)$ and let $U^*\subseteq \bigcup_{i\in [4]}U_i$.
	Let $\cF$ be a feasible system with respect to $\cU$ and $U^*$. Let $e$ be a forward edge with respect to $\cU$ which satisfies $V(e)\cap U^*=\emptyset$. Then, $\cF\setminus \{e\}$ is a feasible system and, if $\cF\cup \{e\}$ is a linear forest, then $\cF\cup \{e\}$ is also a feasible system.
\end{fact}

Note that isolated vertices play no role in a feasible system and so may be deleted.

\begin{fact}\label{fact:feasibleisolated}
	Let $\cU, U^*$, and $\cF$ be as in \cref{fact:feasibleforward}. Let $\cF'$ be obtained from $\cF$ by deleting all isolated vertices. Then, $\cF'$ is also a feasible system with respect to $\cU$ and $U^*$.
\end{fact}

As discussed above, we will decompose the backward and exceptional edges into $n$ feasible systems and then restrict ourselves to construct a Hamilton decomposition where each Hamilton cycle contains precisely one of the feasible systems.
The incorporation of feasible systems into the approximate decomposition is discussed in \cref{sec:blowupC4approxdecomp}. To decompose the leftovers, recall that we will be using the robust decomposition lemma for blow-up cycles (\cref{lm:cyclerobustdecomp}) and, as discussed in \cref{sec:constructESF}, all the cycles obtained via \cref{lm:cyclerobustdecomp} will be turned into Hamilton cycles of $T$ by incorporating a special cover.
Thus, we will require some of our feasible systems to form special covers. For \cref{lm:ESF}, these will also need to be balanced (recall \cref{def:BSC}).
In the next \lcnamecref{fact:feasibleSC}, we verify that feasible systems can induce balanced special covers.

\begin{lm}\label{fact:feasibleSC}
	Let $D$ be a digraph and $U_1, \dots, U_4$ be a partition of $V(D)$. Denote $\cU\coloneqq(U_1, \dots, U_4)$ and let $U^*\subseteq \bigcup_{i\in [4]}U_i$ be an exceptional set satisfying $|U^*\cap U_1|=\dots=|U^*\cap U_4|$.
	Let $\cF\subseteq D$ be a feasible system with respect to $\cU$ and $U^*$. If $V^0(\cF)=U^*$, then $\cF$ is a $\cU$-balanced special cover in $D$ with respect to $U^*$. In particular, $\cF$ is $\cU'$-balanced, where $\cU'\coloneqq (U_1\setminus U^*, \dots, U_4\setminus U^*)$.
\end{lm}

\begin{proof}
	Clearly, $\cF$ is a special cover in $D$ with respect to $U^*$ and the ``in particular part" holds since all the endpoints of the components of $\cF$ lie in $\bigcup\cU'$. We show that $\cF$ is $\cU$-balanced.
	By \cref{def:BSC,fact:feasibleisolated}, we may assume without loss of generality that $\cF$ does not contain any isolated vertex.
	For each $i\in [4]$, let $n_i^\pm\coloneqq |V^\pm(\cF)\cap U_i|$. By symmetry, it is enough to show that $n_1^+=n_2^-$.
	Using new vertices and edges, extend each component of $\cF$ to obtain a linear forest $\cF'\supseteq \cF$ and vertex sets $U_1'\supseteq U_1, \dots, U_4'\supseteq U_4$ such that the following hold, where $\cU''\coloneqq(U_1', \dots, U_4')$%
	\COMMENT{Greedily add forward edges at new vertices until \cref{cor:SCbalanced-endpoints} holds.}.
	\begin{enumerate}[label=(\alph*)]
		\item $\cF'$ is a bipartite linear forest on vertex classes $U_1'\cup U_3'$ and $U_2'\cup U_4'$.\label{cor:SCbalanced-bipartite}
		\item $E(\overleftarrow{\cF'}_{\cU''})=E(\overleftarrow{\cF}_\cU)$.\label{cor:SCbalanced-backward'}
		\item Each component of $\cF'$ is a path which starts in $U_1'$ and ends in $U_4'$.\label{cor:SCbalanced-endpoints}
		\item Each component of $\cF'\setminus \cF$ is a path of length at most $3$.\label{cor:SCbalanced-minimal}
	\end{enumerate}
	Let $\cF''$ be obtained from $\cF'$ by adding an edge from the ending point to the starting point of each component of $\cF'$.
	By \cref{cor:SCbalanced-endpoints,cor:SCbalanced-bipartite,cor:SCbalanced-backward'}, the following hold.
	\begin{enumerate}[label=(\alph*$'$)]
		\item $\cF''$ is a bipartite $1$-factor on vertex classes $U_1'\cup U_3'$ and $U_2'\cup U_4'$.\label{cor:SCbalanced-bipartite'}
		\item $E(\overleftarrow{\cF''}_{\cU''})=E(\overleftarrow{\cF}_\cU)$.\label{cor:SCbalanced-backward''}
	\end{enumerate}
	Let $\sC$ be the set of components of $\cF''$.
	For each $C\in \sC$, let $\ell_C\in \mathbb{Z}$ be the constant obtained by applying \cref{prop:cycle} with $V(C)\cap U_1', \dots, V(C)\cap U_4'$ playing the roles of $U_1, \dots, U_4$.	
	Then, 
	\begin{align}
		|U_1'|&
		\stackrel{\text{\eqmakebox[SC]{}}}{=}\sum_{C\in \sC}|V(C)\cap U_1|
		\stackrel{\text{\cref{prop:cycle}}}{=}\sum_{C\in \sC}(\ell_C+e_C(U_2',U_1')+e_C(U_1',U_4'))\nonumber\\
		&\stackrel{\text{\eqmakebox[SC]{}}}{=}\sum_{C\in \sC}\ell_C+e_{\cF''}(U_2',U_1')+e_{\cF''}(U_1',U_4')
		\stackrel{\text{\eqmakebox[SC]{\text{\cref{cor:SCbalanced-backward''}}}}}{=}\sum_{C\in \sC}\ell_C+e_{\cF''}(U_2',U_1')+e_\cF(U_1,U_4)\nonumber\\
		&\stackrel{\text{\cref{def:feasible-backward}}}{=}\sum_{C\in \sC}\ell_C+e_{\cF''}(U_2',U_1')+e_\cF(U_3,U_2)
		\stackrel{\text{\cref{cor:SCbalanced-backward''}}}{=}\sum_{C\in \sC}\ell_C+e_{\cF''}(U_2',U_1')+e_{\cF''}(U_3',U_2')\nonumber\\
		&\stackrel{\text{\eqmakebox[SC]{}}}{=}\sum_{C\in \sC}(\ell_C+e_C(U_2',U_1')+e_C(U_3',U_2'))
		\stackrel{\text{\cref{prop:cycle}}}{=}\sum_{C\in \sC}|V(C)\cap U_2'|\nonumber\\
		&\stackrel{\text{\eqmakebox[SC]{}}}{=}|U_2'|.\label{eq:SCbalanced}
	\end{align}
	Therefore,
	\begin{align*}
		(|V^0(\cF)\cap U_1|+n_1^++n_1^-)&+(n_2^++n_3^++n_4^+)\\
		&\stackrel{\text{\cref{cor:SCbalanced-backward'}--\cref{cor:SCbalanced-minimal}}}{=}|U_1|+|U_1'\setminus U_1|
		=|U_1'|
		\stackrel{\text{\cref{eq:SCbalanced}}}{=}|U_2'|
		=|U_2|+|U_2'\setminus U_2|\\
		&\stackrel{\text{\cref{cor:SCbalanced-backward'}--\cref{cor:SCbalanced-minimal}}}{=}(|V^0(\cF)\cap U_2|+n_2^++n_2^-)+(n_1^-+n_3^++n_4^+).
	\end{align*}
	Thus, $n_1^+=n_2^-$, as desired.
\end{proof}

\subsection{Optimal partitions}

	\onlyinsubfile{
		\setcounter{section}{9}
\subsection{Optimal partitions}}

We now introduce our main tool for decomposing the backward and exceptional edges into feasible systems. Suppose that $T$ is $\varepsilon$-close to the complete blow-up $C_4$ with vertex partition $\cU=(U_1, \dots, U_4)$. As observed in \cref{sec:regbiT}, $T$ is a bipartite tournament on vertex classes $U_1\cup U_3$ and $U_2\cup U_4$. Thus, the partition $\cU$ is not fixed (one may swap some vertices between $U_1$ and $U_3$, as well as some vertices between $U_2$ and $U_4$). We will consider a partition $\cU$ which minimises the number of backward edges.

\begin{definition}[\Gls*{optimal partition}]\label{def:optimal}
	Let $T$ be a bipartite tournament. An $(\varepsilon, 4)$-partition $\cU$ for $T$ is \emph{optimal} if it minimises the number of backward edges in $T$, that is, if
	\[|E(\overleftarrow{T}_\cU)|=\min\{|E(\overleftarrow{T}_{\cU'})| \mid \cU' \text{ is an $(\varepsilon, 4)$-partition for $T$}\}.\]
\end{definition}

Roughly speaking, an optimal $(\varepsilon,4)$-exceptional partition of $T$ will guarantee the existence of a subdigraph $H\subseteq T$ of small maximum degree which contains many backward edges. This will enable us to apply K\"{o}nig's theorem (\cref{prop:Koniglarge}) to find large matchings of backward edges. This is \cref{lm:optimalH} below. To state and prove this \lcnamecref{lm:optimalH}, we need some additional notation.

Let $U_1, \dots, U_4$ be disjoint vertex sets of size $n$ and denote $\cU\coloneqq (U_1, \dots, U_4)$. Let $T$ be a regular bipartite tournament on vertex classes $U_1\cup U_3$ and $U_2\cup U_4$. Recall from \cref{fact:backwarddegree} that $\overleftarrow{d}_{T,\cU}^+(v)=\overleftarrow{d}_{T,\cU}^-(v)$ for each $v\in V(T)$.
For each $i\in [4]$ and $0\leq \gamma < 1$, denote by
\begin{equation}\label{eq:Ugamma}
    U_i^{\gamma, \cU}(T)\coloneqq \left\{v\in U_i \mid \overleftarrow{d}_{T,\cU}^+(v)=\overleftarrow{d}_{T,\cU}^-(v)> \gamma n\right\}
\end{equation}
the set of vertices in $U_i$ whose backward out- and indegree is greater that $\gamma n$.
Define $U^{\gamma,\cU}(T)\coloneqq \bigcup_{i\in [4]}U_i^{\gamma,\cU}(T)$. 
In practice, $\cU$ will always be clear from the context and so we will omit the second superscript. That is, we will write $U_i^\gamma(T)$ and $U^\gamma(T)$ instead of $U_i^{\gamma,\cU}(T)$ and $U^{\gamma,\cU}(T)$.
Throughout the rest of this paper, the subscript $i$ in the above notation will always be taken modulo $4$, so $U_5^\gamma(T)\coloneqq U_1^\gamma(T)$ for example.

\begin{lm}\label{lm:optimalH}
	Let $0<\frac{1}{n}\ll \varepsilon\ll \gamma\leq \frac{1}{2}$. Let $T$ be a regular bipartite tournament on $4n$ vertices and let $\cU=(U_1, \dots, U_4)$ be an optimal $(\varepsilon,4)$-partition for $T$. Then, there exists $H\subseteq \overleftarrow{T}_\cU$ satisfying the following properties.
	\begin{enumerate}
		\item $\Delta^0(H)\leq \gamma n$.\label{lm:optimalH-Delta}
		\item For each $v\in U^{1-\gamma}(T)$, $d_H(v)=0$.\label{lm:optimalH-U**}
		\item For each $i\in [4]$, $e_{H-U^{1-\gamma}(T)}(U_i, U_{i-1})\geq (1-2\gamma)n|U_{i-2}^{1-\gamma}(T)\cup U_{i-3}^{1-\gamma}(T)|$.\label{lm:optimalH-e}
	\end{enumerate}
\end{lm}

Note that \cref{lm:optimalH}\cref{lm:optimalH-U**} implies that $H-U^{1-\gamma}(T)$ is simply $H$ in \cref{lm:optimalH}\cref{lm:optimalH-e}. However, we emphasise for later applications that none the edges are incident to $U^{1-\gamma}(T)$. Indeed, $H$ will be used in conjunction with other sets of edges incident to $U^{1-\gamma}(T)$ to construct feasible systems and the role of $H$ will be to balance out the number of backward edges chosen incident to $U^{1-\gamma}(T)$ (recall property \cref{def:feasible-backward} of a feasible system). At this stage, it will be crucial that there are many edges which are not incident to $U^{1-\gamma}(T)$.

To prove \cref{lm:optimalH}, we will need the next two results.

\begin{fact}\label{fact:U}
	Let $U_1, \dots, U_4$ be disjoint vertex sets of size $n$ and denote $\cU\coloneqq (U_1, \dots, U_4)$. Let $T$ be a regular bipartite tournament on vertex classes $U_1\cup U_3$ and $U_2\cup U_4$. For any $0\leq \gamma \leq \gamma'< 1$, $U^{\gamma'}(T)\subseteq U^\gamma(T)$.
\end{fact}

\begin{lm}\label{lm:gammaoptimal}
	Let $0<\varepsilon\leq 1$ and $0<\gamma \leq \frac{1}{2}$%
		\COMMENT{Only needed for the ``in particular part".}.
	Let $T$ be a bipartite tournament and let $\cU=(U_1, \dots, U_4)$ be an optimal $(\varepsilon, 4)$-partition for $T$. Then, for each $i\in [2]$, there exists $j_i\in \{i,i+2\}$ such that
	\[U_i^{1-\gamma}(T)=\emptyset=U_{i+2}^{1-\gamma}(T) \quad \text{or} \quad U_{j_i}^\gamma(T)=\emptyset.\]
	In particular,
	\[U_{j_1}^{1-\gamma}(T)=\emptyset=U_{j_2}^{1-\gamma}(T).\]
\end{lm}

\begin{proof}
	The ``in particular part" follows immediately from \cref{fact:U}.
	By symmetry, it suffices to show that the lemma holds for $i=1$. Let $j_1\in \{1, 3\}$ minimise $\max_{v\in U_{j_1}}\overleftarrow{d}_{T, \cU}(v)$. Suppose for a contradiction that there exist $u\in U_{j_1+2}^{1-\gamma}(T)$ and $v\in U_{j_1}^\gamma(T)$. We claim that swapping $u$ and $v$ decreases the number of backward edges and so $\cU$ is not optimal. Indeed, let $U_2'\coloneqq U_2$, $U_4'\coloneqq U_4$, $U_{j_1}'\coloneqq (U_{j_1}\setminus \{v\})\cup \{u\}$, and $U_{j_1+2}'\coloneqq (U_{j_1+2}\setminus \{u\})\cup \{v\}$. Define $\cU'\coloneqq (U_1', \dots, U_4')$. Then,
	\begin{align*}
		e_T(U_{j_1}', U_{j_1-1}')&=\overleftarrow{d}_{T, \cU'}^+(u)+\sum_{w\in U_{j_1}\setminus \{v\}} \overleftarrow{d}_{T, \cU'}^+(w)
		= \overrightarrow{d}_{T, \cU}^+(u)+\sum_{w\in U_{j_1}\setminus \{v\}} \overleftarrow{d}_{T, \cU}^+(w)\\
		&< \gamma n+\sum_{w\in U_{j_1}\setminus \{v\}} \overleftarrow{d}_{T, \cU}^+(w)
		< \overleftarrow{d}_{T, \cU}^+(v) +\sum_{w\in U_{j_1}\setminus \{v\}} \overleftarrow{d}_{T, \cU}^+(w) =e_T(U_{j_1}, U_{j_1-1}).
	\end{align*}
	Thus, \cref{fact:epsilon4partition,fact:backwardedges} imply that $\cU'$ is an $(\varepsilon,4)$-partition for $T$ which satisfies $|E(\overleftarrow{T}_{\cU'})|<|E(\overleftarrow{T}_\cU)|$. This contradicts the fact that $\cU$ is optimal and so
	\[U_{j_1+2}^{1-\gamma}(T)=\emptyset \quad \text{or} \quad U_{j_1}^\gamma(T)=\emptyset.\]
	Thus, it suffices to prove that if $U_{j_1+2}^{1-\gamma}(T)=\emptyset$, then $U_{j_1}^{1-\gamma}(T)=\emptyset$. Suppose not. Then, $\max_{w\in U_{j_1+2}}\overleftarrow{d}_{T, \cU}(w)\leq (1-\gamma)n < \max_{w\in U_{j_1}}\overleftarrow{d}_{T, \cU}(w)$, a contradiction to the definition of~$j_1$.
\end{proof}

\begin{proof}[Proof of \cref{lm:optimalH}]
	By \cref{fact:partition}\cref{fact:partition-cycle} and \cref{lm:gammaoptimal}, we may assume without loss of generality that, for each $i\in [2]$, 
	\begin{equation}\label{eq:optimalH2}
		U_i^{1-\gamma}(T)=\emptyset=U_{i+2}^{1-\gamma}(T) \quad \text{or} \quad U_{i+2}^\gamma(T)=\emptyset.
	\end{equation}
	In particular, we have
	\begin{equation}\label{eq:optimalH}
		U_3^{1-\gamma}(T)=\emptyset \quad \text{and} \quad U_4^{1-\gamma}(T)=\emptyset.
	\end{equation}
	
	For each $i\in \{1,4\}$ and $v\in U_i\setminus U^{1-\gamma}(T)$, let $E_v\subseteq E_T(\{v\}, U_{i-1})\subseteq E(\overleftarrow{T}_\cU)$ satisfy $|E_v|=|N_T^-(v)\cap U_{i+1}^{1-\gamma}(T)|$ (this is possible by \cref{fact:backwarddegree}). Similarly, for each $i\in \{2,3\}$ and $v\in U_i\setminus U^{1-\gamma}(T)$, let $E_v\subseteq E_T(U_{i+1}, \{v\})\subseteq E(\overleftarrow{T}_\cU)$ satisfy $|E_v|=|N_T^+(v)\cap U_{i-1}^{1-\gamma}(T)|$.
	Let $E_{14}\coloneqq \bigcup_{v\in U_1\setminus U^{1-\gamma}(T)}E_v$ and $E_{32}\coloneqq \bigcup_{v\in U_2\setminus U^{1-\gamma}(T)}E_v$.
	Define 
	\begin{equation*}
		E_{43}\coloneqq 
		\begin{cases}
			\emptyset& \text{if }U_1^{1-\gamma}(T)= \emptyset= U_2^{1-\gamma}(T);\\
			E_T(U_4, U_3)& \text{if } U_1^{1-\gamma}(T)\neq \emptyset\neq U_2^{1-\gamma}(T);\\
			\bigcup_{v\in U_4\setminus U^{1-\gamma}(T)}E_v& \text{if } U_1^{1-\gamma}\neq \emptyset=U_2^{1-\gamma}(T);\\
			\bigcup_{v\in U_3\setminus U^{1-\gamma}(T)}E_v& \text{if }U_1^{1-\gamma}= \emptyset\neq U_2^{1-\gamma}(T).
		\end{cases}
	\end{equation*}	
	Let $H$ be the digraph on $V(T)$ defined by $E(H)\coloneqq E_{14}\cup E_{32}\cup E_{43}$.
	
	By definition, $H\subseteq \overleftarrow{T}_\cU$.
	We verify that \cref{lm:optimalH-Delta,lm:optimalH-e,lm:optimalH-U**} are satisfied.
	One can easily verify that, for each $v\in V(T)\setminus U^{1-\gamma}(T)$ and $e\in E_v$, we have $V(e)\setminus \{v\}\subseteq U_3\cup U_4$%
		\COMMENT{$v\in U_1$ implies $E_v\subseteq E_T(\{v\},U_4)$.\\
		$v\in U_4$ implies that $E_v\subseteq E_T(\{v\},U_3)$.\\
		$v\in U_2$ implies that $E_v\subseteq E_T(U_3, \{v\})$.\\
		$v\in U_3$ implies that $E_v\subseteq E_T(U_4, \{v\})$.}.
	Together with \cref{eq:optimalH} and the fact that $E(H)\subseteq E_T(U_4, U_3)\cup \bigcup_{v\in V(T)\setminus U^{1-\gamma}(T)}E_v$, this implies that \cref{lm:optimalH-U**} holds.
	
	For \cref{lm:optimalH-Delta}, it only remains to check that $d_H^\pm(v)\leq \gamma n$ for each $v\in V(T)\setminus U^{1-\gamma}(T)$%
		\COMMENT{By \cref{lm:optimalH-U**}, we already have $d_H^\pm(v)=0\leq \gamma n$ for each $v\in U^{1-\gamma}(T)$.}.
	By \cref{fact:epsilon4partition}, the following holds for each $i\in [4]$.
	\begin{equation}\label{eq:optimalH-U*}
		|U_i^{1-\gamma}(T)|\leq \frac{\varepsilon n^2}{(1-\gamma)n}\leq \gamma n.
	\end{equation}
	Thus, for each $v\in U_1\setminus U^{1-\gamma}(T)$, we have
	\[d_H(v)=d_{E_{14}}(v)=|E_v|\leq |U_2^{1-\gamma}(T)|\stackrel{\text{\cref{eq:optimalH-U*}}}{\leq} \gamma n,\]
	as desired. 
	Similarly, each $v\in U_2\setminus U^{1-\gamma}(T)$ satisfies $d_H(v)\leq \gamma n$%
		\COMMENT{$d_H(v)=d_{E_{32}}(v)\leq |U_1^{1-\gamma}(T)|\leq \gamma n$.}.
	It remains to verify that $d_H^\pm(v)\leq \gamma n$ for each $v\in (U_3\cup U_4)\setminus U^{1-\gamma}(T)$.
	If $U_3^\gamma(T)=\emptyset=U_4^\gamma(T)$, then we have $d_H^\pm(v)\leq \overleftarrow{d}_{T,\cU}^\pm(v)\leq \gamma n$ for each $v\in U_3\cup U_4$, as desired. 
	Moreover, \cref{eq:optimalH2} implies that if $U_3^\gamma(T)\neq \emptyset\neq U_4^\gamma(T)$, then $U_i^{1-\gamma}(T)=\emptyset$ for each $i\in [4]$ and so $E(H)=\emptyset$. Thus, by symmetry, we may assume that $U_3^\gamma(T)=\emptyset\neq U_4^\gamma(T)$%
		\COMMENT{Case $U_3^\gamma(T)\neq \emptyset=U_4^\gamma(T)$:\\
		Each $v\in U_4$ satisfies $d_H^\pm(v)\leq \overleftarrow{d}_{T,\cU}^\pm(v)\leq \gamma n$, as desired.
		By \cref{eq:optimalH2}, we have $U_1^{1-\gamma}(T)=\emptyset$. Therefore, $E_{32}=\emptyset$ and so each $v\in U_3\setminus U^{1-\gamma}(T)$ satisfies $d_H^+(v)=d_{E_{32}}(v)=0$. Moreover, each $v\in U_3\setminus U^{1-\gamma}(T)$ satisfies
		\[d_H^-(v)=d_{E_{43}}(v)\leq|E_v|\leq |U_2^{1-\gamma}(T)|\stackrel{\text{\cref{eq:optimalH-U*}}}{\leq} \gamma n.\]}.
	Then, each $v\in U_3$ satisfies $d_H^\pm(v)\leq \overleftarrow{d}_{T,\cU}^\pm(v)\leq \gamma n$, as desired.
	By \cref{eq:optimalH2}, we have $U_2^{1-\gamma}(T)=\emptyset$. Therefore, $E_{14}=\emptyset$ and so each $v\in U_4\setminus U^{1-\gamma}(T)$ satisfies $d_H^-(v)=d_{E_{14}}(v)=0$. Moreover, each $v\in U_4\setminus U^{1-\gamma}(T)$ satisfies
	\[d_H^+(v)=d_{E_{43}}(v)\leq|E_v|\leq |U_1^{1-\gamma}(T)|\stackrel{\text{\cref{eq:optimalH-U*}}}{\leq} \gamma n.\]
	Thus, \cref{lm:optimalH-Delta} is satisfied.
	
	Finally, we verify \cref{lm:optimalH-e}. By \cref{eq:optimalH}, \cref{lm:optimalH-e} holds for $i=2$.
	Moreover,
	\begin{align*}
		e_{H-U^{1-\gamma}(T)}(U_1, U_4)
		&\stackrel{\text{\eqmakebox[optimalH]{}}}{=}|E_{14}|
		=\sum_{v\in U_1\setminus U^{1-\gamma}(T)}|N_T^-(v)\cap U_2^{1-\gamma}(T)|\\
		&\stackrel{\text{\eqmakebox[optimalH]{}}}{=}e_T(U_2^{1-\gamma}(T), U_1\setminus U^{1-\gamma}(T))
		\geq \sum_{v\in U_2^{1-\gamma}(T)}(\overleftarrow{d}_T^+(v)-|U_1^{1-\gamma}(T)|)\\
		&\stackrel{\text{\eqmakebox[optimalH]{}}}{\geq} |U_2^{1-\gamma}(T)|((1-\gamma)n-|U_1^{1-\gamma}(T)|)
		\stackrel{\text{\cref{eq:optimalH-U*}}}{\geq} (1-2\gamma)n|U_2^{1-\gamma}(T)|\\
		&\stackrel{\text{\eqmakebox[optimalH]{\text{\cref{eq:optimalH}}}}}{=} (1-2\gamma)n(|U_2^{1-\gamma}(T)|+|U_3^{1-\gamma}(T)|).
	\end{align*}
	Thus, \cref{lm:optimalH-e} holds for $i=1$. Similarly, \cref{lm:optimalH-e} holds for $i=3$%
		\COMMENT{\begin{align*}
				e_{H-U^{1-\gamma}(T)}(U_3, U_2)
				&\stackrel{\text{\eqmakebox[optimalH2]{}}}{=}|E_{32}|
				=\sum_{v\in U_2\setminus U^{1-\gamma}(T)}|N_T^+(v)\cap U_1^{1-\gamma}(T)|\\
				&\stackrel{\text{\eqmakebox[optimalH2]{}}}{=}e_T(U_2\setminus U^{1-\gamma}(T), U_1^{1-\gamma}(T))\\
				&\stackrel{\text{\eqmakebox[optimalH2]{}}}{\geq} \sum_{v\in U_1^{1-\gamma}(T)}(\overleftarrow{d}_T^-(v)-|U_2^{1-\gamma}(T)|)\\
				&\stackrel{\text{\eqmakebox[optimalH2]{}}}{\geq} |U_1^{1-\gamma}(T)|((1-\gamma)n-|U_2^{1-\gamma}(T)|)\\
				&\stackrel{\text{\eqmakebox[optimalH2]{\text{\cref{eq:optimalH-U*}}}}}{\geq} (1-2\gamma)n|U_1^{1-\gamma}(T)|\\
				&\stackrel{\text{\eqmakebox[optimalH2]{\text{\cref{eq:optimalH}}}}}{=} (1-2\gamma)n(|U_1^{1-\gamma}(T)|+|U_4^{1-\gamma}(T)|).
		\end{align*}}.
	If $U_1^{1-\gamma}(T)=\emptyset=U_2^{1-\gamma}(T)$, then \cref{lm:optimalH-e} is clearly satisfied for $i=4$.
	If $U_1^{1-\gamma}(T)\neq\emptyset\neq U_2^{1-\gamma}(T)$, then
	\begin{align*}
		e_{H-U^{1-\gamma}(T)}(U_4, U_3)
		&\stackrel{\text{\eqmakebox[optimalH3]{}}}{=}|E_{43}|
		=e_T(U_4, U_3)
		\stackrel{\text{\cref{fact:backwardedges}}}{=} e_T(U_2, U_1)\\
		&\stackrel{\text{\eqmakebox[optimalH3]{}}}{\geq} (|U_1^{1-\gamma}(T)|+|U_2^{1-\gamma}(T)|)(1-\gamma)n-|U_1^{1-\gamma}(T)||U_2^{1-\gamma}(T)|\\
		&\stackrel{\text{\eqmakebox[optimalH3]{\text{\cref{eq:optimalH-U*}}}}}{\geq}(1-2\gamma)n(|U_1^{1-\gamma}(T)|+|U_2^{1-\gamma}(T)|)
	\end{align*}
	and so \cref{lm:optimalH-e} holds for $i=4$.
	If $U_1^{1-\gamma}(T)\neq\emptyset=U_2^{1-\gamma}(T)$, then
	\begin{align*}
		e_{H-U^{1-\gamma}(T)}(U_4, U_3)
		&=|E_{43}|
		=\sum_{v\in U_4\setminus U^{1-\gamma}(T)}|N_T^-(v)\cap U_1^{1-\gamma}(T)|
		\stackrel{\text{\cref{eq:optimalH}}}{=}e_T(U_1^{1-\gamma}(T), U_4)\\
		&\geq (1-\gamma)n|U_1^{1-\gamma}(T)|
		=(1-\gamma)n(|U_1^{1-\gamma}(T)|+|U_2^{1-\gamma}(T)|)
	\end{align*}
	and so \cref{lm:optimalH-e} holds for $i=4$.
	Similarly, \cref{lm:optimalH-e} holds for $i=4$ if $U_1^{1-\gamma}(T)=\emptyset\neq U_2^{1-\gamma}(T)$%
		\COMMENT{\begin{align*}
				e_{H-U^{1-\gamma}(T)}(U_4, U_3)
				&\stackrel{\text{\eqmakebox[optimalH4]{}}}{=}|E_{43}|
				=\sum_{v\in U_3\setminus U^{1-\gamma}(T)}|N_T^+(v)\cap U_2^{1-\gamma}|\\
				&\stackrel{\text{\eqmakebox[optimalH4]{\text{\cref{eq:optimalH}}}}}{=}e_T(U_3, U_2^{1-\gamma}(T))
				\geq (1-\gamma)n|U_2^{1-\gamma}(T)|\\
				&\stackrel{\text{\eqmakebox[optimalH4]{}}}{=}(1-\gamma)n(|U_1^{1-\gamma}(T)|+|U_2^{1-\gamma}(T)|).
		\end{align*}}.
	Therefore, \cref{lm:optimalH-e} is satisfied.
\end{proof}

\subsection{Decomposing the backward and exceptional edges into feasible systems}\label{sec:main}

	\onlyinsubfile{
		\setcounter{section}{9}
		\setcounter{subsection}{1}
		\subsection{Decomposing the backward and exceptional edges into feasible systems}}
Finally, we state and motivate our main decomposition lemma. First, we need an additional definition.
As discussed in \cref{sec:CST}, the exceptional set $U^*$ will have to contain all the vertices of high backward degree (otherwise we would not be able to construct the cycle-setup required for the robust decomposition lemma). This motivates \cref{def:ES-backward} below. To facilitate the incorporation of the exceptional vertices into the Hamilton decomposition, we also require that $U^*$ is small and contain the same number of vertices from each vertex class. This is \cref{def:ES-size} below.  

\begin{definition}[\gls*{exceptional set}]\label{def:ES}
	Let $T$ be a regular bipartite tournament on $4n$ vertices. Suppose that $\cU=(U_1, \dots, U_4)$ is an $(\varepsilon, 4)$-partition for $T$. We say that $U^*$ is an \emph{$(\varepsilon', \cU)$-exceptional set for $T$} if the following hold, where $U_i^*\coloneqq U_i\cap U^*$ for each $i\in [4]$.
	\begin{enumerate}[label=\rm(ES\arabic*)]
		\item $U^{\varepsilon'}(T)\subseteq U^* \subseteq V(T)$.\label{def:ES-backward}
		\item $|U_1^*|=\dots=|U_4^*|\leq \varepsilon' n$.\label{def:ES-size}
	\end{enumerate}
\end{definition}

Let $T$ be a bipartite tournament on $4n$ vertices. Let $\cU$ be an optimal $(\varepsilon,4)$-partition for $T$ and let $U^*$ be an $(\varepsilon',\cU)$-exceptional set for $T$. Then, \cref{lm:main} states that $T$ contains $n$ edge-disjoint feasible systems $\cF_1, \dots, \cF_n$ which contain all the backward edges of $T$ (see \cref{lm:main}\cref{lm:main-backwardedges}). By \cref{lm:main}\cref{lm:main-size}, all these feasible systems are small, which will enable us to incorporate them into our Hamilton cycles. The first $t$ feasible systems will be those which will be incorporated into the Hamilton cycles given by the robust decomposition lemma (\cref{lm:cyclerobustdecomp}) and so we will require those to form special covers which are localised and balanced (see \cref{lm:ESF}). Together with \cref{fact:feasibleSC}, \cref{lm:main}\cref{lm:main-V0} will imply that $\cF_1, \dots, \cF_t$ are balanced special covers, as required for \cref{lm:ESF}. Additionally, \cref{lm:main}\cref{lm:main-H} ensures that $\cF_1, \dots, \cF_t$ are constructed out of prescribed sets $H_1, \dots, H_s$ of edges of $T$. These edges will be chosen in such a way that $\cF_1, \dots, \cF_t$ form localised special covers, as desired for \cref{lm:ESF}.
The last $n-t$ feasible systems will be incorporated into the approximate decomposition. For simplicity, we require that all of the components of these feasible systems are paths which start in $U_1$ and end in $U_4$ (see \cref{lm:main}\cref{lm:main-endpoints}).
Finally, \cref{lm:main}\cref{lm:main-backwardedges} will allow us to incorporate a small prescribed set $E$ of forward edges of $T$ into the feasible systems. In practice, $E$ will consists of all the edges of $T$ which cannot be decomposed via the robust decomposition lemma. (Recall from \cref{lm:cyclerobustdecomp} that the robustly decomposable digraph $D^{\rm rob}$ cannot decompose the edges which are lying along the auxiliary matchings in $\cM$.) This will ensure that they are not left over at the end of the approximate decomposition.

Roughly speaking, \cref{lm:main}\cref{lm:main-H-degree,lm:main-H-degreeV*,lm:main-H-degreeV**,lm:main-H-edges} ensure that $H_1, \dots, H_s$ contain many well distributed backward and exceptional edges. This is necessary, for otherwise we may not be able to construct the feasible systems satisfying \cref{lm:main}\cref{lm:main-H}. More precisely, \cref{lm:main}\cref{lm:main-H-degreeV*,lm:main-H-degreeV**} ensure that each exceptional vertex in $U^*$ has many in- and outneighbours in each $H_i$ (recall from \cref{def:feasible-exceptional} that a feasible system has to cover $U^*$). Additionally, \cref{lm:main}\cref{lm:main-H-degree,lm:main-H-edges} ensure that there are many backward edges and that these are evenly distributed across the non-exceptional vertices. This will enable us to use K\"{o}nig's theorem to find large matchings of backward edges, which will be convenient for adjusting the number of backward edges in each feasible system (recall from \cref{def:feasible-backward} that a feasible system must contain a balanced number of backward edges).

Note that $H_1, \dots, H_s$ will be constructed using \cref{lm:optimalH}.
(Compare the bounds in \cref{lm:main}\cref{lm:main-H-degree,lm:main-H-edges} to those in \cref{lm:optimalH}\cref{lm:optimalH-e,lm:optimalH-Delta}.)
This is a point where we make crucial use of the concept of optimal partitions.

Finally, observe that $H_1, \dots, H_s$ need not be edge-disjoint in \cref{lm:main}. The upper bound on $t$ will be sufficient to ensure that there are, overall, sufficiently many edges in $H_1, \dots, H_s$ to construct the edge-disjoint feasible systems $\cF_1, \dots, \cF_t$, each within its prescribed $H_i$.

\begin{lm}[Decomposing the backward and exceptional edges into feasible systems]\label{lm:main}
	Let $0<\frac{1}{n}\ll \varepsilon \ll \varepsilon'\ll \eta \ll\gamma \ll 1$ and $s\in \mathbb{N}$. Let $T$ be a regular bipartite tournament on $4n$ vertices. Let $\cU=(U_1, \dots, U_4)$ be an optimal $(\varepsilon, 4)$-partition for $T$ and $U^*$ be an $(\varepsilon, \cU)$-exceptional set for $T$.
	Suppose that, for each $j\in [s]$, $H_j\subseteq T$ satisfies the following.
	\begin{enumerate}
		\item For each $v\in U^{1-\gamma}(T)$, $\overleftarrow{d}_{H_j,\cU}^\pm(v)\geq 3\gamma n$.\label{lm:main-H-degreeV**}
		\item For each $v\in U^*\setminus U^{1-\gamma}(T)$, $\overrightarrow{d}_{H_j,\cU}^\pm(v)\geq \gamma^2 n$.\label{lm:main-H-degreeV*}
		\item For each $v\in V(T)\setminus U^{1-\gamma}(T)$, $\overleftarrow{d}_{H_j,\cU}^\pm(v)\leq 2\gamma n$.\label{lm:main-H-degree}
		\item For each $i\in [4]$, $e_{H_j-U^{1-\gamma}(T)}(U_i, U_{i-1})\geq 110\gamma n|U_{i-2}^{1-\gamma}(T)\cup U_{i-3}^{1-\gamma}(T)|$.\label{lm:main-H-edges}
	\end{enumerate}
	For each $i\in [s]$, let $s_i\in \mathbb{N}$ and $t_i\coloneqq \sum_{j\in [i-1]} s_j$. Let $t\coloneqq \sum_{i\in [s]} s_i$ and suppose that $t\leq \eta n$.
	Let $E\subseteq E(T)$ be such that the following hold.
	\begin{enumerate}[resume]
		\item $E\subseteq E(\overrightarrow{T}_\cU-U^*)$.\label{lm:main-E-forward}
		\item For each $v\in V(T)\setminus U^*$, $d_E^\pm(v)\leq 1$.\label{lm:main-E-deg}
	\end{enumerate}
	Then, there exist edge-disjoint feasible systems $\cF_1, \dots, \cF_n$ such that the following hold.
	\begin{enumerate}[label=\rm(\alph*),ref=(\alph*)]
		\item $E(\overleftarrow{T}_\cU)\cup E\subseteq \bigcup_{i\in [n]}E(\cF_i)\subseteq E(T)$.\label{lm:main-backwardedges}
		\item For each $i\in [n]$, $e(\cF_i)\leq \varepsilon'n$.\label{lm:main-size}
		\item For each $i\in [t]$, $V^0(\cF_i)=U^*$.\label{lm:main-V0}
		\item For each $i\in [s]$ and $j\in [s_i]$, $\cF_{t_i+j}\subseteq H_i\setminus E$\label{lm:main-H}.
		\item For each $i\in [n-t]$, we have $V^+(\cF_{t+i})\subseteq U_1$ and $V^-(\cF_{t+i})\subseteq U_4$.\label{lm:main-endpoints}
	\end{enumerate}
\end{lm}

To provide intuition into its formulation, we will first assume that \cref{lm:main} holds and derive \cref{thm:blowupC4}. The proof of \cref{lm:main} is spread over \cref{sec:pseudofeasible,sec:constructfeasible,sec:constructpseudofeasible,sec:specialpseudofeasible}. These \lcnamecrefs{sec:pseudofeasible} also include a detailed proof overview of \cref{lm:main}.

\section{The \texorpdfstring{$\varepsilon$}{epsilon}-close to the complete blow-up \texorpdfstring{$C_4$}{C4} case: proof of Theorem \ref{thm:blowupC4}}\label{sec:blowupC4}

	\onlyinsubfile{
		\setcounter{section}{9}
	\section{The complete blow-up \texorpdfstring{$C_4$}{C4} case}}

We will now prove \cref{thm:blowupC4}. First, we use our tools from \cref{sec:approxdecomp} to incorporate feasible systems into an approximate Hamilton decomposition (see \cref{lm:blowupC4approxdecomp} below). Then, we derive \cref{thm:blowupC4} from the robust decomposition lemma for blow-up cycles (\cref{lm:cyclerobustdecomp}), the decomposition lemma for backward and exceptional edges (\cref{lm:main}), and the approximate decomposition lemma (\cref{lm:blowupC4approxdecomp}).

\subsection{Approximate decomposition}\label{sec:blowupC4approxdecomp}

Let $T$ be a regular bipartite tournament on $4n$ vertices and suppose that $T$ is $\varepsilon$-close to the complete blow-up $C_4$ with vertex partition $\cU=(U_1, \dots, U_4)$. Our strategy for approximately decomposing $T$ is the following (see also \cref{fig:sketch-C4Ham}). First, we use \cref{cor:bipapproxmatchdecomp} to approximately decompose $T[U_1,U_2]$, $T[U_2, U_3]$, and $T[U_3, U_4]$ into perfect matchings. Combining a matching from each pair, we obtain an approximate decomposition of $T[U_1,U_2]\cup T[U_2, U_3] \cup T[U_3, U_4]$ into linear forests, each consisting of $n$ components which start in $U_1$ and end in $U_4$. Finally, using \cref{thm:biphalfapproxHamdecomp}, we close each of these linear forests into a Hamilton cycle by approximately decomposing $T[U_4, U_1]$ into ``suitable" perfect matchings.

Recall that in \cref{thm:biphalfapproxHamdecomp,cor:bipapproxmatchdecomp}, there is the flexibility of prescribing a few edges. This enables us to construct an approximate decomposition of $T$ which incorporates given feasible systems.

\begin{lm}[Incorporating feasible systems into an approximate Hamilton decomposition]\label{lm:blowupC4approxdecomp}
	Let $0<\frac{1}{n}\ll \tau\ll \delta \leq 1$ and  $0<\frac{1}{n}\ll\varepsilon\ll\eta,\nu\leq 1$. 
	Let $T$ be a regular bipartite tournament on $4n$ vertices. Let $\cU=(U_1, \dots, U_4)$ be an $(\varepsilon,4)$-partition for $T$ and $U^*$ be an $(\varepsilon,\cU)$-exceptional set for $T$.
	Let $\ell \leq 2(\delta-\eta)\left(n-\frac{|U^*|}{4}\right)$.
	Let $\cF_1, \dots, \cF_\ell\subseteq T$ be edge-disjoint feasible systems and $D\subseteq T\setminus \bigcup_{i\in [\ell]}\cF_i$.
	Suppose that the following hold.
	\begin{enumerate}
		\item For each $i\in [4]$, $D[U_i\setminus U^*, U_{i+1}\setminus U^*]$ is $(\delta,\varepsilon)$-almost regular.\label{lm:blowupC4approxdecomp-reg}
		\item For each $i\in [4]$, $D[U_i\setminus U^*, U_{i+1}\setminus U^*]$ is a bipartite robust $(\nu,\tau)$-expander with bipartition $(U_i\setminus U^*, U_{i+1}\setminus U^*)$.\label{lm:blowupC4approxdecomp-rob}
		\item For each $i\in [\ell]$, $e(\cF_i)\leq \varepsilon n$.\label{lm:blowupC4approxdecomp-size}
		\item For each $v\in V(T)\setminus U^*$, there exist at most $\varepsilon n$ indices $i\in [\ell]$ such that $v\in V(\cF_i)$.\label{lm:blowupC4approxdecomp-deg}
		\item For each $i\in [\ell]$, $V^+(\cF_i)\subseteq U_1$ and $V^-(\cF_i)\subseteq U_4$.\label{lm:blowupC4approxdecomp-endpoints}
	\end{enumerate}
	Then, there exist edge-disjoint Hamilton cycles $C_1, \dots, C_\ell$ of $T$ such that $\cF_i\subseteq C_i\subseteq D\cup \cF_i$ for each $i\in [\ell]$.	
\end{lm}

Given two digraphs $D$ and $D'$, we say that $D'$ is a \emph{subdivision} of $D$ if $D'$ can be obtained from $D$ by replacing some edges by internally vertex-disjoint paths.

\begin{proof}[Proof of \cref{lm:blowupC4approxdecomp}]
	For each $i\in [4]$, define $U_i^*\coloneqq U^*\cap U_i$. Let $A\coloneqq U_4\setminus U^*$ and $B\coloneqq U_1\setminus U^*$. By \cref{fact:partition}\cref{fact:partition-size} and \cref{def:ES-size}, we have 
	\begin{equation}\label{eq:blowupapproxdecomp-AB}
		|A|=|B|\geq (1-\varepsilon)n.
	\end{equation}
	
	\begin{steps}	
		\item \textbf{Approximately decomposing $D[U_1, U_2]\cup D[U_2,U_3]\cup D[U_3, U_4]$.}
		For each $i\in [3]$ in turn, we will use \cref{cor:bipapproxmatchdecomp} to approximately decompose $D[U_i, U_{i+1}]$ into matchings. We will then combine a matching from each pair to get an approximate decomposition of $D[U_1, U_2]\cup D[U_2,U_3]\cup D[U_3, U_4]$ into spanning linear forests.
		
		Let $i\in [3]$ and $j\in [\ell]$.
		Denote by $S_{i,j}^+$ the set of vertices $v\in U_i\setminus U^*$ such that $d_{\cF_j}^+(v)=1$ and let $S_{i+1,j}^-$ be the set of vertices $v\in U_{i+1}\setminus U^*$ such that $d_{\cF_j}^-(v)=1$. (Thus, $S_{i,j}^+\cup S_{i+1,j}^-$ is the set of vertices which are already covered by $\cF_j$ and so need to be avoided by the $j^{th}$ matching of $D[U_i, U_{i+1}]$.)
		Note that
		\begin{align}
			|S_{i,j}^+|
			&\stackrel{\text{\eqmakebox[blowupapprox]{\text{\cref{def:feasible-exceptional},\cref{def:feasible-linforest}}}}}{=} e_{\cF_j}(U_i,U_{i+1})+e_{\cF_j}(U_i, U_{i-1})-|U_i^*|\label{eq:Sij+}\\
			&\stackrel{\text{\eqmakebox[blowupapprox]{\text{\cref{def:feasible-backward},\cref{def:ES-size}}}}}{=}
			e_{\cF_j}(U_i,U_{i+1})+e_{\cF_j}(U_{i+2}, U_{i+1})-|U_{i+1}^*|
			\stackrel{\text{\cref{def:feasible-exceptional},\cref{def:feasible-linforest}}}{=}|S_{i+1,j}^-|.\nonumber
		\end{align}
		Let $F_{i,j}$ be an auxiliary perfect matching between $S_{i,j}^+$ and $S_{i+1,j}^-$. Then,
		\[e(F_{i,j})\stackrel{\text{\cref{eq:Sij+}}}{\leq} e(\cF_j)
		\stackrel{\text{\cref{lm:blowupC4approxdecomp-size}}}{\leq} \varepsilon n 
		\stackrel{\text{\cref{def:ES-size}}}{\leq}
		2\varepsilon (n-|U_i^*|)\]
		and so \cref{cor:bipapproxmatchdecomp}\cref{cor:bipapproxmatchdecomp-size} holds with $n- |U_i^*|, F_{i,j}$, and $2\varepsilon$ playing the roles of $n,F_i$, and $\varepsilon$.
		
		Let $i\in [3]$. For each $j\in [\ell]$, we have $V(F_{i,j})\subseteq V(\cF_j)$. Thus, \cref{lm:blowupC4approxdecomp-deg} implies that \cref{cor:bipapproxmatchdecomp}\cref{cor:bipapproxmatchdecomp-deg} holds with $D[U_i\setminus U^*, U_{i+1}\setminus U^*]$ and $F_{i,1}, \dots, F_{i,\ell}$ playing the roles of $G$ and $F_1, \dots, F_\ell$.
		Let $M_{i,1}, \dots, M_{i,\ell}$ be the matchings obtained by applying \cref{cor:bipapproxmatchdecomp} with $D[U_i\setminus U^*, U_{i+1}\setminus U^*], n-|U_i^*|, 2\varepsilon$, and $F_{i,1}, \dots, F_{i,\ell}$ playing the roles of $G, n, \varepsilon$, and $F_1, \dots, F_\ell$.
		For each $j\in [\ell]$, let $F_{i,j}'$ be obtained from $M_{i,j}\setminus F_{i,j}$ by orienting all the edges from $U_i$ to $U_{i+1}$ and observe that $F_{i,j}'\subseteq D(U_i\setminus U^*, U_{i+1}\setminus U^*)$.
		
		For each $j\in [\ell]$, let $\cF_j'\coloneqq \cF_j\cup \bigcup_{i\in [3]}F_{i,j}'$. We claim that $\cF_1', \dots, \cF_\ell'$ are edge-disjoint spanning linear forests whose components are paths which start in $B=U_1\setminus U^*$ and end in $A=U_4\setminus U^*$.
		
		\begin{claim}\label{claim:F'}
			$\cF_1', \dots, \cF_\ell'$ are edge-disjoint linear forests such that the following hold for each $i\in [\ell]$.
			\begin{enumerate}[label=\rm(\alph*),ref=(\alph*)]
				\item $E(\cF_i')\cap E_D(A,B)=\emptyset$.\label{claim:F'-D}
				\item $V(\cF_i')=V(T)$.\label{claim:F'-V}
				\item $|V^0(\cF_i')\cap (A\cup B)|\leq 3\varepsilon |A|$.\label{claim:F'-size}
				\item $V^+(\cF_i')\subseteq B$ and $V^-(\cF_i')\subseteq A$.\label{claim:F'-endpoints}
			\end{enumerate}		
		\end{claim}
		
		\begin{proofclaim}
			By assumption, $\cF_1, \dots, \cF_\ell$ are edge-disjoint. For each $i\in [3]$, \cref{cor:bipapproxmatchdecomp} implies that $F_{i,1}', \dots, F_{i,\ell}'$ are edge-disjoint matchings in $E_D(U_i, U_{i+1})$. Therefore, $\cF_1', \dots, \cF_\ell'$ are edge-disjoint, as desired.
			
			Let $j\in [\ell]$. Suppose for a contradiction that $\cF_j'$ is not a linear forest. By \cref{def:feasible-linforest} and construction, each $v\in V(T)$ satisfies both $d_{\cF_j'}^\pm(v)\leq 1$. Thus, $\cF_j'$ contains a cycle $C$. Clearly, $\bigcup_{i\in [3]}F_{i,j}'$ is a linear forest. Thus, there exists $e\in E(\cF_j)\cap E(C)$. Let $v$ be the starting point of the component of $\cF_j$ which contains $e$. 
			Then, $v\in V(C)$. Let $u$ be the inneighbour of $v$ in $C$. By assumption, $uv\notin E(\cF_j)$ (otherwise $v$ would not be the starting point of one of the components of $\cF_j$). Thus, $uv\in \bigcup_{i\in [3]}E(F_{i,j}')$ and so $v\in U_2\cup U_3\cup U_4$, which contradicts \cref{lm:blowupC4approxdecomp-endpoints}. Therefore, $\cF_j'$ is a linear forest, as desired.
			
			Let $j\in [\ell]$. We show that \cref{claim:F'-D,claim:F'-V,claim:F'-endpoints,claim:F'-size} are satisfied.
			By construction and since $E(\cF_j)\cap E(D)=\emptyset$, \cref{claim:F'-D} holds.
			By \cref{def:feasible-exceptional}, $U^*\subseteq V(\cF_j)\subseteq V(\cF_j')$ and, by construction, $(U_i\cup U_{i+1})\setminus (V(\cF_j)\cup U^*)\subseteq V(F_{i,j}')\subseteq V(\cF_j')$ for each $i\in [3]$. Thus, \cref{claim:F'-V} holds.
			By construction, $\bigcup_{i\in [3]}F_{i,j}$ does not contain any edge which starts in $A=U_4\setminus U^*$ or ends in $B=U_1\setminus U^*$. Thus, \[V^0(\cF_j')\cap A\subseteq (V^0(\cF_j)\cap A)\cup (V^+(\cF_j)\cap A)\] and \[V^0(\cF_j')\cap B\subseteq (V^0(\cF_j)\cap B)\cup (V^-(\cF_j)\cap B).\]
			Therefore,
			\begin{align*}
				|V^0(\cF_j')\cap (A\cup B)|\leq |V(\cF_j)|\stackrel{\text{\cref{lm:blowupC4approxdecomp-size}}}{\leq}2\varepsilon n\stackrel{\text{\cref{eq:blowupapproxdecomp-AB}}}{\leq} 3\varepsilon |A|
			\end{align*}
			and so \cref{claim:F'-size} holds.			
			Finally, we verify that $V^+(\cF_j')\subseteq B$ and $V^-(\cF_j')\subseteq A$. By \cref{def:feasible-exceptional}, each $v\in U^*$ satisfies $d_{\cF_j'}^+(v)=1=d_{\cF_j'}^-(v)$ and so $V^+(\cF_j')\cap U^*=\emptyset=V^-(\cF_j')\cap U^*$. 
			Let $i\in [3]$. Suppose that $v\in U_i\setminus U^*$. If $v\in S_{i,j}^+$, then $d_{\cF_j}^+(v)=1$; otherwise, $v\in V(F_{i,j}')$. Thus, $d_{\cF_j'}^+(v)=1$ and so $v\notin V^-(\cF_j')$. Therefore, $V^-(\cF_j')\subseteq U_4\setminus U^*=A$.
			Similarly, if $w\in U_{i+1}\setminus U^*$, then either $d_{\cF_j}^-(w)=1$ or $w\in V(F_{i,j}')$. Thus, $d_{\cF_j'}^-(w)=1$ for each $w\in U_{i+1}\setminus U^*$ and so $V^+(\cF_j')\subseteq U_1\setminus U^*=B$. Therefore, \cref{claim:F'-endpoints} is satisfied.
		\end{proofclaim}
		
		\item \textbf{Approximately decomposing $D[U_4,U_1]$.}
		In this step, we use \cref{thm:biphalfapproxHamdecomp} to approximately decompose	$D[U_4,U_1]$ into matchings which close $\cF_1', \dots, \cF_\ell'$ into Hamilton cycles of $T$.
		
		To apply \cref{thm:biphalfapproxHamdecomp}, we first need to contract $\cF_1', \dots, \cF_\ell'$ into auxiliary linear forests on $A\cup B$.
		For each $j\in [\ell]$, let $\widetilde{\cF}_j$ be the digraph on $A\cup B$ defined as follows. For any distinct $u,v\in A\cup B$, we let $uv\in E(\widetilde{\cF}_j)$ if and only if $\cF_j'$ contains a $(u,v)$-subpath $P$ which satisfies $V^0(P)\subseteq (U_2\cup U_3\cup U^*)$.
		
		\begin{claim}\label{claim:tF}
			Let $i\in [\ell]$. Then, $\cF_i'$ is a subdivision of $\widetilde{\cF}_i$. In particular, $\widetilde{\cF}_i$ is a linear forest satisfying the following properties.
			\begin{enumerate}[label=\rm(\greek*),ref=(\greek*)]
				\item $V^0(\widetilde{\cF}_i)=V^0(\cF_i')\cap (A\cup B)=(A\setminus V^-(\cF_i'))\cup (B\setminus V^+(\cF_i'))$.\label{claim:tF-V0}
				\item $V^+(\widetilde{\cF}_i)=V^+(\cF_i')\subseteq B$ and $V^-(\widetilde{\cF}_i)=V^-(\cF_i')\subseteq A$.\label{claim:tF-endpoints}
			\end{enumerate}
		\end{claim}
		
		\begin{proofclaim}
			Let $i\in [\ell]$. Using \cref{claim:F'}, it is easy to check that $\cF_i'$ is a subdivision of $\widetilde{\cF}_i$.%
			\COMMENT{For each $e\in E(\widetilde{\cF}_i)$, let $P_e\subseteq\cF_i'$ be the subpath which witnesses that $e\in E(\widetilde{\cF}_i)$ ($P_e$ is well defined since $(A\cup B)\cap (U_2\cup U_3\cup U^*)=\emptyset$).
				It suffices to show that any distinct $e,e'\in E(\widetilde{\cF}_i)$ satisfy \[V(P_e)\cap V(P_{e'})\subseteq (V^+(P_e)\cap V^-(P_{e'}))\cup (V^-(P_e)\cap V^+(P_{e'})).\]
				Let $e=uv,e'=u'v'\in E(\widetilde{\cF}_i)$ be distinct.
				Since $(A\cup B)\cap (U_2\cup U_3\cup U^*)=\emptyset$, both $V^\pm(P_e)\cap V^0(P_{e'})=\emptyset=V^\pm(P_{e'})\cap V^0(P_e)$.
				Thus, it suffices to show that $u\neq u'$, $v\neq v'$, and $V^0(P_e)\cap V^0(P_{e'})=\emptyset$.\\
				Suppose for a contradiction that $u=u'$. Since $e\neq e'$, we have $v\neq v'$. Therefore, there exists $w\in (V(P_e)\cap V(P_{e'}))\setminus \{v,v'\}$ such that $N_{P_e}^+(w)\neq N_{P_{e'}}^+(w)$. Then, $d_{\cF_i}^+(w)\geq 2$, which contradicts \cref{claim:F'}. Therefore, $u\neq u'$ and, by symmetry, $v\neq v'$.\\
				Suppose for a contradiction that $V^0(P_e)\cap V^0(P_{e'})\neq \emptyset$. Since $e\neq e'$, we may assume without loss of generality that $u\neq u'$. Then, there exists $w'\in V^0(P_e)\cap V^0(P_{e'})$ such that $N_{P_e}^-(w')\neq N_{P_{e'}}^-(w')$. Then, $d_{\cF_i}^-(w)\geq 2$, which contradicts \cref{claim:F'}.}
			Then, each $v\in V(\widetilde{\cF}_i)$ satisfies both $d_{\widetilde{\cF}_i}^\pm(v)=d_{\cF_i'}^\pm(v)$. Moreover, each cycle in $\widetilde{\cF}_i$ would induce a cycle in $\cF_i'$. Recall from \cref{claim:F'} that $\cF_i'$ is a linear forest. Thus, $\widetilde{\cF}_i$ is also a linear forest and \cref{claim:tF-V0,claim:tF-endpoints} follow from \cref{claim:F'-V,claim:F'-endpoints}.
		\end{proofclaim}
		
		Since $\widetilde{\cF}_1, \dots, \widetilde{\cF}_\ell$ may not be bipartite on vertex classes $A$ and $B$, we cannot apply \cref{thm:biphalfapproxHamdecomp} directly and need to consider equivalent linear forests (recall \cref{def:equivalentP}).
		
		\begin{claim}\label{claim:tF'}
			There exist bipartite linear forests $\widetilde{\cF}_1', \dots, \widetilde{\cF}_\ell'$ on vertex classes $A$ and $B$ such that $\widetilde{\cF}_i$ and $\widetilde{\cF}_i'$ are equivalent for each $i\in [\ell]$.
			In particular, the following hold for each $i\in [\ell]$.
			\begin{enumerate}[label=\rm(\greek*$'$)]
				\item $V^0(\widetilde{\cF}_i')=V^0(\widetilde{\cF}_i)=V^0(\cF_i')\cap (A\cup B)=(A\setminus V^-(\cF_i'))\cup (B\setminus V^+(\cF_i'))$.\label{claim:tF'-V0}
				\item $V^+(\widetilde{\cF}_i')=V^+(\widetilde{\cF}_i)=V^+(\cF_i')\subseteq B$ and $V^-(\widetilde{\cF}_i')=V^-(\widetilde{\cF}_i)=V^-(\cF_i')\subseteq A$.\label{claim:tF'-endpoints}
			\end{enumerate}
		\end{claim}
		
		\begin{proofclaim}
			Let $i\in [\ell]$. By \cref{claim:tF}, $\widetilde{\cF}_i$ is a linear forest which spans $A\cup B$ and whose components are all paths which start in $B$ and end in $A$. Recall from \cref{eq:blowupapproxdecomp-AB} that $|A|=|B|$. Thus, one can easily construct an auxiliary bipartite linear forest $\widetilde{\cF}_i'$ which is equivalent to $\widetilde{\cF}_i$. Then, \cref{claim:tF'-V0,claim:tF'-endpoints} follow from \cref{claim:tF-V0,claim:tF-endpoints}.
		\end{proofclaim}
		
		Let $\widetilde{\cF}_1', \dots, \widetilde{\cF}_\ell'$ be the linear forests obtained by applying \cref{claim:tF'}.
		We now verify that \cref{thm:biphalfapproxHamdecomp}\cref{thm:biphalfapproxHamdecomp-AB,thm:biphalfapproxHamdecomp-BA,thm:biphalfapproxHamdecomp-size} hold for $\widetilde{\cF}_1', \dots, \widetilde{\cF}_\ell'$. Let $i\in [\ell]$. 
		Since $\widetilde{\cF}_i'$ is a spanning bipartite linear forest on vertex classes $A$ and $B$, we have
		\begin{align*}
			e(\widetilde{\cF}_i'[B,A])=\sum_{v\in B}d_{\widetilde{\cF}_i'}^+(v)=
			|B\setminus V^-(\widetilde{\cF}_i')|
			\stackrel{\text{\cref{claim:tF'-endpoints}}}=|B|\stackrel{\text{\cref{eq:blowupapproxdecomp-AB}}}{=}|A|
		\end{align*}
		and
		\begin{align*}
			e(\widetilde{\cF}_i'[A,B])=\sum_{v\in A}d_{\widetilde{\cF}_i'}^+(v)
			=|A\setminus V^-(\widetilde{\cF}_i')|
			\stackrel{\text{\cref{claim:tF'-endpoints}}}
			=|V^0(\widetilde{\cF}_i)\cap A|
			\stackrel{\text{\cref{claim:F'-size},\cref{claim:tF'-V0}}}{\leq} 3\varepsilon |A|.
		\end{align*} 
		Thus, \cref{thm:biphalfapproxHamdecomp}\cref{thm:biphalfapproxHamdecomp-AB,thm:biphalfapproxHamdecomp-BA} hold with $\widetilde{\cF}_i', |A|$, and $3\varepsilon$ playing the roles of $F_i, n$, and $\varepsilon$.
		Moreover, each $v\in A$ satisfies
		\begin{align*}
			d_{\widetilde{\cF}_i'[A,B]}(v)=d_{\widetilde{\cF}_i'}^+(v) \stackrel{\text{\cref{claim:tF'-V0},\cref{claim:tF'-endpoints}}}{=}d_{\cF_i'}^+(v)\stackrel{\text{\cref{claim:F'-D}}}{=} d_{\cF_i}^+(v)=d_{\cF_i[A,B]}(v)
		\end{align*} 
		and, similarly, each $w\in B$ satisfies $d_{\widetilde{\cF}_i'[A,B]}(w)=d_{\cF_i[A,B]}(w)$. Thus, \cref{lm:blowupC4approxdecomp-deg,eq:blowupapproxdecomp-AB} imply that \cref{thm:biphalfapproxHamdecomp}\cref{thm:biphalfapproxHamdecomp-size} holds with $\widetilde{\cF}_i', |A|$, and $3\varepsilon$ playing the roles of $F_i, n$, and $\varepsilon$.
		
		Apply \cref{thm:biphalfapproxHamdecomp} with $D[A\cup B], |A|, 3\varepsilon$, and $\widetilde{\cF}_1', \dots, \widetilde{\cF}_\ell'$ playing the roles of $D, n,\varepsilon$, and $F_1, \dots, F_\ell$ to obtain edge-disjoint cycles $\tC_1, \dots, \tC_\ell$ such that, for each $i\in [\ell]$, $V(\tC_i)=A\cup B$ and $\widetilde{\cF}_i'\subseteq \tC_i\subseteq D(A,B)\cup \widetilde{\cF}_i'$. 		
		For each $i\in [\ell]$, let $\tC_i'\coloneqq (\tC_i\setminus \widetilde{\cF}_i') \cup \widetilde{\cF}_i$ and $C_i\coloneqq (\tC_i'\setminus \widetilde{\cF}_i) \cup \cF_i'=(\tC_i\setminus \widetilde{\cF}_i') \cup \cF_i'$.
		
		\item \textbf{Verifying the conclusions of the \lcnamecref{lm:blowupC4approxdecomp}.}
		We now verify that $C_1, \dots, C_\ell$ are edge-disjoint Hamilton cycles of $T$ such that, for each $i\in [\ell]$, $\cF_i\subseteq C_i\subseteq D\cup \cF_i$.
		
		Recall from \cref{claim:F'} that $\cF_1', \dots, \cF_\ell'$ are edge-disjoint. Thus, \cref{claim:F'-D} and \cref{thm:biphalfapproxHamdecomp} imply that $C_1, \dots, C_\ell$ are edge-disjoint.
		Let $i\in [\ell]$. By construction, $\cF_i'\subseteq D\cup \cF_i$ and $\tC_i\setminus \widetilde{\cF}_i'\subseteq D$. Thus, $C_i=(\tC_i\setminus \widetilde{\cF}_i') \cup \cF_i'\subseteq D\cup \cF_i$. Moreover, $\cF_i\subseteq \cF_i'\subseteq C_i$, as desired.
		
		Let $i\in [\ell]$ and recall from \cref{claim:tF'} that $\widetilde{\cF}_i$ and $\widetilde{\cF}_i'$ are equivalent. Thus, \cref{fact:equivalentP} implies that $\tC_i'$ is also a Hamilton cycle on $A\cup B$. By \cref{claim:tF}, $\cF_i'$ is a subdivision of $\widetilde{\cF}_i$ and so $C_i$ is a subdivision of $\tC_i'$. In particular, $C_i$ is a cycle satisfying
		\begin{align*}
			V(T)\supseteq V(C_i)=V(\tC_i')\cup V(\cF_i')
			\stackrel{\text{\cref{claim:F'-V}}}{\supseteq}V(T).
		\end{align*}
		That is, $C_i$ is a Hamilton cycle of $T$.\qedhere
	\end{steps}
\end{proof}

\subsection{Proof of Theorem \ref{thm:blowupC4}}\label{sec:blowupC4proof}

We are now ready to derive \cref{thm:blowupC4}. Our strategy is as follows. In \cref{step:CST}, we use \cref{lm:CST} to construct a cycle-setup for the robust decomposition lemma (\cref{lm:cyclerobustdecomp}). In \cref{step:backwardedges}, we decompose the backward edges and exceptional edges into feasible systems using \cref{lm:main}. In \cref{step:robustdecomplm}, we apply the robust decomposition lemma (\cref{lm:cyclerobustdecomp}) to obtain an absorber $D^{\rm rob}$ (the required extended special factors are constructed using \cref{lm:ESF}). In \cref{step:approxdecomp}, we construct an approximate Hamilton decomposition using \cref{lm:blowupC4approxdecomp}. In \cref{step:leftovers}, we decompose the leftovers using $D^{\rm rob}$.

\begin{proof}[Proof of \cref{thm:blowupC4}]
	Fix additional constants such that
	\[0<\frac{1}{n_0}\ll\varepsilon\ll \varepsilon_1\ll \varepsilon_2\ll \eta\ll \frac{1}{k}\ll \varepsilon_3\ll \gamma \ll \frac{1}{q}\ll\frac{1}{f}\ll d \ll \frac{1}{\ell'}, \frac{1}{g}, \nu\ll \tau\ll 1\]
	and $\frac{k}{14}, \frac{k}{f}, \frac{k}{g}, \frac{q}{f}, \frac{2fk}{3g(g-1)}\in \mathbb{N}$.
	Let $m_0\in \mathbb{N}$ be such that $\varepsilon_1^2 n\leq m_0\leq \varepsilon_1 n$ and $m\coloneqq \frac{n-m_0}{k}, \frac{m}{4\ell'}, \frac{fm}{q}\in \mathbb{N}$.
	Fix additional constants such that $\frac{1}{f}\ll \frac{r_1}{m}\ll d$ and $\eta\ll \frac{r}{m}\ll \frac{1}{k}$. Define
	\begin{equation}\label{eq:blowupC4-constants}
		r_2\coloneqq 96\ell'g^2kr, \quad r_3\coloneqq \frac{rfk}{q}, \quad r^\diamond\coloneqq r_1+r_2+r-(q-1)r_3, \quad s'\coloneqq rfk+7r^\diamond.
	\end{equation}
	For simplicity, we denote
	\begin{equation}\label{eq:blowupC4-Q}
		Q\coloneqq [r_3]\times [\tfrac{q}{f}]\times [4]\times[f] \quad \text{and} \quad Q'\coloneqq [r^\diamond]\times [1]\times [4]\times[7].
	\end{equation}
	Let $T$ be a regular bipartite tournament which is $\varepsilon$-close to the complete blow-up $C_4$ on vertex classes of size $n\geq n_0$. Let $\cU=(U_1, \dots, U_4)$ be an optimal $(\varepsilon, 4)$-partition for $T$ and denote by $H\subseteq \overleftarrow{T}_\cU$ the digraph obtained by applying \cref{lm:optimalH}.
	
	\begin{steps}
		\item \textbf{Constructing a cycle-setup.}\label{step:CST}
		We will use \cref{lm:CST}. We first construct the partitions in $\cP^*$ randomly, to ensure that the edges of $H$ are well distributed across the clusters.
		More precisely, for each $(h,i,j)\in [\frac{q}{f}]\times [4]\times [k]$, let $V_{i,j,h}\subseteq U_i$ be obtained by including each $v\in U_i$ with probability $\frac{f}{kq}$ independently of all other vertices. For each $(i,j)\in [4]\times [k]$, denote $V_{i,j}\coloneqq \bigcup_{h\in [\frac{q}{f}]} V_{i,j,h}$. For each $i\in [4]$, let $\tC^i\coloneqq V_{i,1}\dots V_{i,k}$ and let $\cI_i\coloneqq \{I_1, \dots, I_f\}$ denote the canonical interval partition of $\tC^i$ into $f$ interval, that is,
		\[I_j\coloneqq V_{i,(j-1)\frac{k}{f}+1}V_{i,(j-1)\frac{k}{f}+2}\dots V_{i,j\frac{k}{f}+1}\]
		for each $j\in [f]$ (see \cref{def:interval}).
		For each $(h,i,j)\in [\frac{q}{f}]\times [4]\times [f]$, define 
		\begin{equation}
			S_{h,i,j}\coloneqq V_{i,(j-1)\frac{k}{f}+2,h}\cup V_{i,(j-1)\frac{k}{f}+3,h}\cup\dots \cup V_{i,j\frac{k}{f},h},
		\end{equation}
		that is, $S_{h,i,j}$ is the union of the $h$\textsuperscript{th} subclusters of the internal clusters in the $j$\textsuperscript{th} interval in the canonical interval partition of $\tC^i$ into $f$ intervals.
		
		\begin{claim}
			With positive probability, all of the following hold.
			\begin{enumerate}[label=\rm(\roman*),ref=(\roman*)]
				\item For each $(h,i,j)\in [\frac{q}{f}]\times [4]\times [k]$, we have $|V_{i,j,h}|\geq \frac{(1- \varepsilon)fn}{kq}$.\label{thm:blowupC4-partition-size}
				\item For each $(h,i,j)\in [\frac{q}{f}]\times [4]\times [k]$ and $v\in V(T)$, we have $|N_T^\pm (v)\cap  V_{i,j,h}|\geq \frac{f|N_T^\pm(v)\cap U_i|}{kq}-\varepsilon n$.\label{thm:blowupC4-partition-d}
				\item For each $h,h'\in [\frac{q}{f}]$, $i\in [4]$, and $j,j'\in [f]$, we have\label{thm:blowupC4-partition-eH}
				\[e_{H-U^{1-\gamma}(T)}(S_{h,i,j}, S_{h',i-1,j'})\geq 110\gamma n |U_{i-2}^{1-\gamma}(T)\cup U_{i-3}^{1-\gamma}(T)|.\]
			\end{enumerate}
		\end{claim}
		
		\begin{proofclaim}
			By \cref{lm:Chernoff} and a union bound, \cref{thm:blowupC4-partition-size,thm:blowupC4-partition-d} hold with probability at least $1-\frac{1}{n}$%
			\COMMENT{Let $i\in [4]$ and $j\in [\frac{kq}{f}]$. Then, $\mathbb{E}[|V_{i,j}|]=\frac{fn}{kq}$ and so \cref{lm:Chernoff} implies that 
				\[\mathbb{P}[|V_{i,j}|<(1-\varepsilon)\mathbb{E}[|V_{i,j}|]]\leq \exp\left(-\frac{\varepsilon^2fn}{3kq}\right).\]
				Let $v\in V(T)$ and $\diamond\in \{+,-\}$. If $|N_T^\diamond(v)\cap U_i|\leq \varepsilon n$, then we are done. We may therefore assume that $|N_T^\diamond(v)\cap U_i|\geq \varepsilon n$.
				Then, $\mathbb{E}[|N_T^\diamond(v)\cap V_{i,j}|]=\frac{f|N_T^\diamond(v)\cap U_i|}{kq}\geq \frac{\varepsilon fn}{kq}$ and so \cref{lm:Chernoff} implies that
				\[\mathbb{P}\left[|N_T^\diamond(v)\cap V_{i,j}|<\frac{f|N_T^\diamond(v)\cap U_i|}{kq}-\varepsilon n\right]\leq \mathbb{P}[|N_T^\diamond(v)\cap V_{i,j}|<(1-\varepsilon)\mathbb{E}[|N_T^\diamond(v)\cap V_{i,j}|]]\leq \exp\left(-\frac{\varepsilon^3fn}{3kq}\right).\]}.
			
			Let $i\in [4]$. 
			If $|U_{i-2}^{1-\gamma}(T)\cup U_{i-3}^{1-\gamma}(T)|=0$, then \cref{thm:blowupC4-partition-eH} holds for $i$ with probability $1$. Suppose that $|U_{i-2}^{1-\gamma}(T)\cup U_{i-3}^{1-\gamma}(T)|\geq 1$.
			Denote $G\coloneqq H[U_i, U_{i-1}]$. By \cref{lm:optimalH}\cref{lm:optimalH-Delta,lm:optimalH-e}, $\Delta(G)\leq \gamma n$ and $e(G)\geq \frac{n}{2}$.
			Also observe that, for each $i'\in [4]$, the $q$ sets $S_{1,i',1}, \dots, S_{\frac{q}{f},i',1}, S_{1,i',2}, \dots, S_{\frac{q}{f},i',f}$ randomly partition $U_{i'}$.
			Thus, \cref{lm:edgepartition} (applied with $q,\gamma, U_i, U_{i-1},S_{1,i,1}, \dots, S_{\frac{q}{f},i,f}$, and $S_{1,i-1,1}, \dots, S_{\frac{q}{f},i-1,f}$ playing the roles of $k, \varepsilon, A, B$, $A_1, \dots, A_k$, and $B_1, \dots, B_k$) implies that, with probability at least $\frac{4}{5}$, all $h,h'\in [\frac{q}{f}]$ and $j,j'\in [f]$ satisfy
			\[e_H(S_{h,i,j}, S_{h',i-1,j'})\geq \frac{e_H(U_i, U_{i-1})}{2q^2}\stackrel{\text{\cref{lm:optimalH}\cref{lm:optimalH-e}}}{\geq} 110\gamma n |U_{i-2}^{1-\gamma}(T)\cup U_{i-3}^{1-\gamma}(T)|.\]
			Thus, \cref{lm:optimalH}\cref{lm:optimalH-U**} implies that \cref{thm:blowupC4-partition-eH} holds for $i$ with probability at least $\frac{4}{5}$.
			
			Therefore, a union bound over all $i\in [4]$ implies that \cref{thm:blowupC4-partition-eH} holds with probability at least $\frac{1}{5}$. Then, a union bound implies that, with positive probability, \cref{thm:blowupC4-partition-eH,thm:blowupC4-partition-size,thm:blowupC4-partition-d} are all satisfied.
		\end{proofclaim}
		
		We may therefore assume that \cref{thm:blowupC4-partition-d,thm:blowupC4-partition-eH,thm:blowupC4-partition-size} are all satisfied.
		We now equalise the partition classes to achieve \cref{thm:blowupC4-partition-size'} below, without affecting the bounds in \cref{thm:blowupC4-partition-d,thm:blowupC4-partition-eH} too much. For this, note that 
		\[|U_i^{\varepsilon_1}(T)|\stackrel{\text{\cref{eq:Ugamma}}}{\leq} \frac{e_T(U_i, U_{i-1})}{\varepsilon_1 n}\stackrel{\text{\cref{fact:epsilon4partition}}}{\leq} \frac{\varepsilon n^2}{\varepsilon_1n}\leq \varepsilon_1^3n\]
		for each $i\in [4]$. Moreover, \cref{thm:blowupC4-partition-size} implies that 
		\begin{align*}
			|V_{i,j,h}|&\geq \frac{fm}{q}+\frac{f(m_0-\varepsilon n)}{kq}\geq \frac{fm}{q}+\frac{(\varepsilon_1^2-\varepsilon)fn}{kq}
			\geq \frac{fm}{q}+|U_i^{\varepsilon_1}(T)|
		\end{align*} 
		for each $(h,i,j)\in [\frac{q}{f}]\times [4]\times [k]$.
		Thus, for each $i\in [4]$, we can let $U_i^*$ be obtained from $U_i^{\varepsilon_1}(T)$ by adding, for each 
		$(h,j)\in [\frac{q}{f}]\times [k]$, precisely $|V_{i,j,h}\setminus U_i^{\varepsilon_1}(T)|-\frac{fm}{q}$ vertices of $V_{i,j,h}\setminus U_i^{\varepsilon_1}(T)$. For each $(h,i,j)\in [\frac{q}{f}]\times [4]\times [k]$, let $V_{i,j,h}'\coloneqq V_{i,j,h}\setminus U_i^*$ and $V_{i,j}'\coloneqq V_{i,j}\setminus U_i^*$.
		For each $i\in [4]$, define $C^i\coloneqq V_{i,1}'\dots V_{i,k}'$. For each $(h,i,j)\in [\frac{q}{f}]\times [4]\times [f]$, let $S_{h,i,j}'\coloneqq S_{h,i,j}\setminus U_i^*$ and observe that $S_{h,i,j}'$ is the union of the $h$\textsuperscript{th} subclusters of the internal clusters in the $j$\textsuperscript{th} interval in the canonical interval partition of $C^i$ into $f$ intervals.
		Then, \cref{thm:blowupC4-partition-size,thm:blowupC4-partition-d,thm:blowupC4-partition-eH} imply that the following properties are satisfied.
		\begin{enumerate}[label=(\roman*$'$)]
			\item For each $(h,i,j)\in [\frac{q}{f}]\times [4]\times [k]$, we have $|V_{i,j,h}'|=\frac{fm}{q}$ and $|U_i^*|=m_0$.\label{thm:blowupC4-partition-size'}
			\item For each $(h,i,j)\in [\frac{q}{f}]\times [4]\times [f]$ and $v\in V(T)$, we have\label{thm:blowupC4-partition-d'}
			\begin{align*}
				\left|N_T^\pm(v)\cap  \left(U_i^*\cup S_{h,i,j}'\right)\right|&\geq \left(\frac{k}{f}-1\right)\cdot\left(\frac{f|N_T^\pm(v)\cap U_i|}{kq}-\varepsilon n\right)\\
				&\geq \frac{|N_T^\pm(v)\cap U_i|}{2q}-\varepsilon_1 n.
			\end{align*}
			\item For each $h,h'\in [\frac{q}{f}]$, $i\in [4]$, and $j,j'\in [f]$, we have\label{thm:blowupC4-partition-eH'}
			\[e_{H-U^{1-\gamma}(T)}\left(U_i^*\cup S_{h,i,j}', U_{i-1}^*\cup S_{h',i-1,j'}'\right)\geq 110\gamma n|U_{i-2}^{1-\gamma}(T)\cup U_{i-3}^{1-\gamma}(T)|.\]
		\end{enumerate}
		By \cref{thm:blowupC4-partition-size'} and construction, $U^*\coloneqq \bigcup_{i\in [4]}U_i^*$ is an $(\varepsilon_1,\cU)$-exceptional set for $T$ (see \cref{def:ES}). Thus, \cref{fact:partition}\cref{fact:partition-size} and \cref{def:ES-size} imply that 
		\begin{equation}\label{eq:blowupC4-n'}
			n'\coloneqq |U_1\setminus U^*|=\dots=|U_4\setminus U^*|=n-\frac{|U^*|}{4}\geq (1-\varepsilon_1)n.
		\end{equation}
		Therefore, each $i\in [4]$ satisfies
		\begin{align*}
			\delta(T[U_i\setminus U^*, U_{i+1}\setminus U^*])\geq n'-\varepsilon_1 n\geq (1-\sqrt{\varepsilon_1})n'.
		\end{align*}
		Define \[\cU'\coloneqq (U_1\setminus U^*, \dots, U_4\setminus U^*).\]
		For each $i\in [4]$, let $\cP_i$ be the partition of $U_i\setminus U^*$ into an empty exceptional set and the $k$ clusters $V_{i,1}', \dots, V_{i,k}'$ and let $\cP_i^*$ be the partition of $U_i\setminus U^*$ into an empty exceptional set and the $\frac{kq}{f}$ clusters $V_{i,1,1}', \dots, V_{i,1,\frac{q}{f}}', V_{i,2,1}', \dots, V_{i,k,\frac{q}{f}}'$. Denote $\cP\coloneqq (\cP_1, \dots, \cP_4)$, $\cP^*\coloneqq (\cP_1^*, \dots, \cP_4^*)$, and $\cC\coloneqq (C^1, \dots, C^4)$. For each $i\in [4]$, note that $\cP_i^*$ is a $\frac{q}{f}$-refinement of $\cP_i$ and $C^i$ is a Hamilton cycle on the clusters in $\cP_i$. 
		Let $D_1, D_2, \cP, \cP',\cR, \cC, \sU, \sU'$, and $\cM$ be obtained by applying \cref{lm:CST} with $\overrightarrow{T}_\cU-U^*,\cU', n', \frac{q}{f}, \sqrt{\varepsilon_1}$, and $\varepsilon_2$ playing the roles of $D, \cU, n, \ell^*, \varepsilon$, and $\varepsilon'$. Let $D_1'\coloneqq T\setminus D_2$ and observe that $D_1'$ is obtained from $D_1$ by adding backward and exceptional edges only. Thus, \cref{lm:CST}\cref{lm:CST-CST,lm:CST-M,lm:CST-supreg} are still satisfied with $D_1'$ playing the role of $D_1$. That is, the following hold.
		\begin{enumerate}[resume,label=(\roman*)]
			\item $(\cU', \cP, \cP^*, \cC, \cM)$ is a consistent $(4, \ell^*, k, n')$-cycle-framework. In particular, the following hold.\label{thm:blowupC4-M}
			\begin{itemize}
				\item By \cref{fact:CCFP}, $(\cU', \cP, \cP, \cC, \cM)$ is a consistent $(4, \ell^*, k, n')$-cycle-framework.
				\item For any $i\in [4]$ and any cluster $V\in \cP_i$, the set $N_{M_i}(V)$ is a cluster in $\cP_{i+1}$ (where $\cP_5\coloneqq \cP_1$).
				\item The analogue holds for the partitions in $\cP^*$.
			\end{itemize} 
			\item For any $i\in [4]$, $D_1'[V,W]$ is $[\varepsilon_2, \geq 1-3d]$-superregular whenever $V\subseteq U_i\setminus U^*$ and $W\subseteq U_{i+1}\setminus U^*$ are unions of clusters in $\cP_i^*$ and $\cP_{i+1}^*$, respectively. In particular, since $\cP_i^*$ is a refinement of $\cP_i$ for each $i\in [4]$, the analogue holds for the partitions in $\cP$.\label{thm:blowupC4-supreg}
			\item $(D_2, \cU', \cP, \cP', \cP^*, \cR, \cC, \sU, \sU', \cM)$ is a $(4, \ell', \frac{q}{f}, k, m, \varepsilon_2, d)$-cycle-setup. In particular, \cref{fact:CSTCF,fact:CFP} imply that $(\cU',\cP,\cP^*,\cC, \cM)$ is a $(4,\frac{q}{f},k,n')$-cycle-framework and $(\cU',\cP,\cP,\cC, \cM)$ is a $(4,1,k,n')$-cycle-framework.\label{thm:blowupC4-CST}
		\end{enumerate}
		As discussed in \cref{sec:CST}, we will use $D_1'$ to construct the required extended special factors (via \cref{lm:ESF}), while $D_2$ will be reserved for the application of the robust decomposition lemma for blow-up cycles (\cref{lm:cyclerobustdecomp}).
		
		\item \textbf{Decomposing the backward and exceptional edges.}\label{step:backwardedges}
		We will use \cref{lm:main}, so we start by building digraphs $H_1, \dots, H_s$ which satisfy \cref{lm:main}\cref{lm:main-H-degree,lm:main-H-degreeV**,lm:main-H-degreeV*,lm:main-H-edges}.
		
		First, observe that if we apply \cref{lm:optimalH}, then $H$ satisfies \cref{lm:main}\cref{lm:main-H-degree,lm:main-H-edges}, but not \cref{lm:main}\cref{lm:main-H-degreeV**,lm:main-H-degreeV*}. To achieve the latter, we add edges as follows.
		Let $\tH$ be the spanning subdigraph of $\overleftarrow{T}_\cU$ which consists of all the edges incident to $U^{1-\gamma}(T)$. Let $\hH$ be the spanning subdigraph of $\overrightarrow{T}_\cU$ which consists of all the edges incident to $U^*\setminus U^{1-\gamma}(T)$.
		By \cref{eq:Ugamma}, we see that $H\cup \tH\cup \hH$ now satisfies all the bounds in \cref{lm:main}\cref{lm:main-H-degree,lm:main-H-degreeV**,lm:main-H-degreeV*,lm:main-H-edges}.
		
		However, as discussed in \cref{sec:main}, the feasible systems $\cF_1, \dots, \cF_t$ constructed within $H_1, \dots, H_s$ will form the special covers required for \cref{lm:ESF}. By \namecrefs{fact:feasibleSC}~\ref{fact:feasibleSC} and \ref{lm:main}\cref{lm:main-V0}, these feasible systems $\cF_1, \dots, \cF_t$ will automatically form balanced special covers. Additionally, \cref{lm:ESF} requires $\cF_1, \dots, \cF_t$ to be localised (recall \cref{def:LSC}). To achieve this, each $H_\ell$ will be associated to one set of ``locality parameters" (that is, a choice of $(h,i,j)$ in \cref{def:LSC}) and then obtained from $H\cup \tH\cup \hH$ by removing all the edges which are forbidden with respect to this set of parameters.
		In \cref{claim:H} below, we will verify that not too many edges are removed and so $H_1, \dots, H_s$ still satisfy \cref{lm:main}\cref{lm:main-H-degree,lm:main-H-degreeV**,lm:main-H-degreeV*,lm:main-H-edges}. Moreover, each of the $\cF_i\subseteq H_\ell$ will automatically be localised.
		
		Recall from \cref{lm:cyclerobustdecomp} that we need two types of extended special factors: some $(\frac{q}{f}, 4, f)$-extended special factors with respect to $\cU', \cP^*, \cC$, and $\cM$ and some $(1,4,7)$-extended special factors with respect to $\cU', \cP, \cC$, and $\cM$. These will be constructed separately by applying \cref{lm:ESF} successively.
		First, we construct the $H_\ell$'s for the first application of \cref{lm:ESF}, that is, for the construction of the $(\frac{q}{f}, 4, f)$-extended special factors with respect to $\cU', \cP^*, \cC$, and $\cM$.
		
		More precisely, let $(h,i,j)\in [\frac{q}{f}]\times[4]\times [f]$. 
		Let $k'\coloneqq \frac{k}{f}+1$ and denote by $W_1 \dots W_{k'}$ the $j^{\rm th}$ interval in the canonical interval partition of $C^i$ into $f$ intervals.
		Denote by $W_{1,h}, \dots, W_{k',h}$ the $h^{\rm th}$ subclusters of $W_1, \dots, W_{k'}$ contained in $\cP_i^*$.
		Let $E_{h,i,j}$ be the set of edges $e\in E(T)$ such that 		
		\begin{equation}\label{eq:blowupC4-E}
			V(e)\cap (U_i\cup U_{i+1})\not\subseteq U^*\cup (W_{1,h}\cup\dots \cup W_{k'-1,h})\cup N_{M_i}(W_{2,h}\cup\dots \cup W_{k',h}).
		\end{equation}
		(Roughly speaking, the set $E_{h,i,j}$ consists of all the edges of $T$ which cannot be included in a $(\frac{q}{f},4,f,h,i,j)$-localised special cover. See the proof of \cref{claim:H}\cref{claim:H-localised} below for details.)
		Let $H_{h,i,j}\coloneqq (H\cup \tH\cup \hH)\setminus E_{h,i,j}$.
		We now claim that any special cover in $H_{h,i,j}$ is $(\frac{q}{f},4,f,h,i,j)$-localised with respect to $\cP^*, \cC$, and $\cM$, and that $H_{h,i,j}$ satisfies \cref{lm:main}\cref{lm:main-H-degree,lm:main-H-degreeV**,lm:main-H-degreeV*,lm:main-H-edges}.
		
		\begin{claim}\label{claim:H}
			For each $(h,i,j)\in [\frac{q}{f}]\times[4]\times[f]$, the following properties are satisfied.
			\begin{enumerate}[label=\rm(\alph*)]
				\item Let $SC$ be a special cover in $T$ with respect to $U^*$. If $SC\subseteq H_{h,i,j}$, then $SC$ is $(\frac{q}{f},4,f, h,i,j)$-localised with respect to $\cP^*, \cC$, and $\cM$.\label{claim:H-localised}
				\item For each $v\in U^{1-\gamma}(T)$, $\overleftarrow{d}_{H_{h,i,j},\cU}^\pm(v)\geq 3\gamma n$.\label{claim:H-backwardU**}
				\item For each $v\in U^*\setminus U^{1-\gamma}(T)$, $\overrightarrow{d}_{H_{h,i,j},\cU}^\pm(v) \geq \gamma^2 n$.\label{claim:H-forwardU*}
				\item For each $v\in V(T)\setminus U^{1-\gamma}(T)$, $\overleftarrow{d}_{H_{h,i,j},\cU}^\pm(v)\leq 2\gamma n$.\label{claim:H-backwardU'}
				\item For each $i'\in [4]$, $e_{H_{h,i,j}-U^{1-\gamma}(T)}(U_{i'}, U_{i'-1})\geq 110\gamma n|U_{i'-2}^{1-\gamma}(T)\cup U_{i'-3}^{1-\gamma}(T)|$.\label{claim:H-backwardeU'}
			\end{enumerate}
		\end{claim}
		
		\begin{proofclaim}
			Let $(h,i,j)\in [\frac{q}{f}]\times[4]\times [f]$.
			Denote by $W_1 \dots W_{k'}$ the $j^{\rm th}$ interval in the canonical interval partition of $C^i$ into $f$ intervals.
			Denote by $W_{1,h}, \dots, W_{k',h}$ the $h^{\rm th}$ subclusters of $W_1, \dots, W_{k'}$ contained in $\cP_i^*$.
			Note that $S_i\coloneqq S_{h,i,j}'=W_{2,h}\cup \dots \cup W_{k'-1,h}$ and, by \cref{thm:blowupC4-M}, $S_{i+1}\coloneqq S_{h,i+1,j}'=N_{M_i}(W_{2,h}\cup \dots\cup W_{k'-1,h})$.
			For each $i'\in [4]\setminus \{i,i+1\}$, let $S_{i'}\coloneqq S_{1,i',1}'$.
			
			Let $SC\subseteq H_{h,i,j}$ be a special cover in $T$ with respect to $U^*$. By \cref{def:SC}, $V^+(SC)\cup V^-(SC)\subseteq V(T)\setminus U^*$ and so \cref{eq:blowupC4-E} implies that 
			\[(V^+(SC)\cup V^-(SC))\cap (U_i\cup U_{i+1})\subseteq (W_{1,h}\cup \dots\cup W_{k'-1,h})\cup N_{M_i}(W_{2,h}\cup \dots \cup W_{k',h}).\] 
			Thus,			
			$SC$ is $(\frac{q}{f},4,f,h,i,j)$-localised with respect to $\cP^*, \cC$, and $\cM$, and so \cref{claim:H-localised} holds.
			
			Let $i'\in [4]$ and $v\in U_{i'}^{1-\gamma}(T)$.
			By \cref{fact:U} and \cref{def:ES-backward}, $v\in U^*$ and so \cref{eq:blowupC4-E} implies that $E_{h,i,j}$ does not contain any edge from $v$ to $U_{i'-1}^*\cup S_{i'-1}$.
			Thus,
			\begin{align*}
				\overleftarrow{d}_{H_{h,i,j},\cU}^+(v)
				&\geq d_{\tH\setminus E_{h,i,j}}^+(v)
				\geq |N_T^+(v)\cap (U_{i'-1}^*\cup S_{i'-1})|
				\stackrel{\text{\cref{thm:blowupC4-partition-d'}}}{\geq}\frac{(1-\gamma)n}{2q}-\varepsilon_1 n
				\geq 3\gamma n.
			\end{align*}
			Similarly, $\overleftarrow{d}_{H_{h,i,j},\cU}^-(v)\geq 3\gamma n$ and so \cref{claim:H-backwardU**} is satisfied.
			
			Let $i'\in [4]$ and $v\in U_{i'}^*\setminus U^{1-\gamma}(T)$.
			Since $v\in U^*$, \cref{eq:blowupC4-E} implies that $E_{h,i,j}$ does not contain any edge from $U_{i'-1}^*\cup S_{i'-1}$ to $v$. Thus,
			\begin{align*}
				\overrightarrow{d}_{H_{h,i,j},\cU}^-(v)
				&\geq d_{\hH\setminus E_{h,i,j}}^-(v)
				\geq |N_T^-(v)\cap (U_{i'-1}^*\cup S_{i'-1})|
				\stackrel{\text{\cref{thm:blowupC4-partition-d'}}}{\geq}\frac{\gamma n}{2q}-\varepsilon_1 n\geq \gamma^2n.
			\end{align*}
			Similarly, $\overrightarrow{d}_{H_{h,i,j},\cU}^+(v)\geq \gamma^2n$ and so \cref{claim:H-forwardU*} holds.
			
			For any $v\in V(T)\setminus U^{1-\gamma}(T)$, we have
			\begin{align*}
				\overleftarrow{d}_{H_{h,i,j}, \cU}^\pm(v)
				&\stackrel{\text{\eqmakebox[blowupC4backwardU']{}}}{\leq} d_H^\pm(v)+ d_{\tH}^\pm(v)
				\stackrel{\text{\cref{lm:optimalH}\cref{lm:optimalH-Delta}}}{\leq}\gamma n+|U^{1-\gamma}(T)|\\
				&\stackrel{\text{\eqmakebox[blowupC4backwardU']{\text{\cref{fact:U},\cref{def:ES}}}}}{\leq} \gamma n+ 4\varepsilon_1 n\leq 2\gamma n
			\end{align*}
			and so \cref{claim:H-backwardU'} holds.
			
			Let $i'\in [4]$. By \cref{eq:blowupC4-E}, $E_{h,i,j}$ does not contain any edge from $U_{i'}^*\cup S_{i'}$ to $U_{i'-1}^*\cup S_{i'-1}$.
			Thus,
			\begin{align*}
				e_{H_{h,i,j}-U^{1-\gamma}(T)}(U_{i'}, U_{i'-1})
				&\stackrel{\text{\eqmakebox[Hhij]{}}}{\geq} e_{(H\setminus E_{h,i,j})-U^{1-\gamma}(T)}(U_{i'}, U_{i'-1})\\
				&\stackrel{\text{\eqmakebox[Hhij]{}}}{\geq} e_{H-U^{1-\gamma}(T)}(U_{i'}^*\cup S_{i'},U_{i'-1}^*\cup S_{i'-1})\\
				&\stackrel{\text{\eqmakebox[Hhij]{\text{\cref{thm:blowupC4-partition-eH'}}}}}{\geq} 110\gamma n|U_{i'-2}^{1-\gamma}(T)\cup U_{i'-3}^{1-\gamma}(T)|
			\end{align*}
			and so \cref{claim:H-backwardeU'} is satisfied.
		\end{proofclaim}
		
		The $H_\ell$'s for the second application of \cref{lm:ESF}, that is, for the construction of the $(1,4,7)$-extended special factors with respect to $\cU', \cP, \cC$, and $\cM$, can be constructed analogously.
		More precisely, let $(h,i,j)\in [1]\times[4]\times [7]$. Let $k''\coloneqq \frac{k}{7}+1$ and denote by $W_1'\dots W_{k''}'$ the $j^{\rm th}$ interval in the canonical interval partition of $C^i$ into $7$ intervals.
		Let $E_{h,i,j}'$ be the set of edges $e\in E(T)$ such that 
		\[V(e)\cap (U_i\cup U_{i+1})\not\subseteq U^*\cup (W_1'\cup\dots\cup W_{k''-1}')\cup N_{M_i}(W_2'\cup\dots\cup W_{k''}').\]
		Let $H_{h,i,j}'\coloneqq (H\cup \tH\cup \hH)\setminus E_{h,i,j}'$.
		Since $f> 14$, note that every interval in the canonical interval partition of $C^i$ into $7$ intervals contains, as a subinterval, an interval from the canonical interval partition of $C^i$ into $f$ intervals. That is, there exists $(h',j')\in [\frac{q}{f}]\times [f]$ such that $S_{h',i,j'}'\subseteq W_2'\cup\dots\cup W_{k''-1}'$. Thus, we can apply similar arguments as in \cref{claim:H}, to show that the following hold.
		\begin{enumerate}[label=(\alph*$'$)]
			\item Let $SC$ be a special cover in $T$ with respect to $U^*$. If $SC\subseteq H_{h,i,j}'$, then $SC$ is $(1,4,7, h,i,j)$-localised with respect to $\cP, \cC$, and $\cM$.\label{claim:H'-localised'}
			\item For each $v\in U^{1-\gamma}(T)$, $\overleftarrow{d}_{H_{h,i,j}',\cU}^\pm(v)\geq 3\gamma n$.\label{claim:H'-backwardU**}
			\item For each $v\in U^*\setminus U^{1-\gamma}(T)$, $\overrightarrow{d}_{H_{h,i,j}',\cU}^\pm(v) \geq \gamma^2 n$.\label{claim:H'-forwardU*}
			\item For each $v\in V(T)\setminus U^{1-\gamma}(T)$, $\overleftarrow{d}_{H_{h,i,j}',\cU}^\pm(v)\leq 2\gamma n$.\label{claim:H'-backwardU'}
			\item For each $i'\in [4]$, $e_{H_{h,i,j}'-U^{1-\gamma}(T)}(U_{i'}, U_{i'-1})\geq 110\gamma n|U_{i'-2}^{1-\gamma}(T)\cup U_{i'-3}^{1-\gamma}(T)|$.\label{claim:H'-backwardeU'}
		\end{enumerate}

		Denote $\cH\coloneqq \{H_{h,i,j}\mid (h,i,j)\in [\frac{q}{f}]\times [4]\times [f]\}$ and $\cH'\coloneqq \{H_{h,i,j}'\mid (h,i,j)\in [1]\times[4]\times [7]\}$. Let $s\coloneqq 4q+28$ and let $H_1, \dots, H_s$ be an enumeration of $\cH\cup \cH'$.
		Recall from \cref{lm:cyclerobustdecomp} that we need to construct $r_3$ $(\frac{q}{f},4,f)$-extended special factors with respect to $\cU', \cP^*, \cC$, and $\cM$, as well as $r^\diamond$ $(1,4,7)$-extended special factors with respect to $\cU', \cP, \cC$, and $\cM$.
		Also recall that in \cref{lm:main}, $s_1, \dots, s_s$ denote the number of feasible systems constructed within $H_1, \dots, H_s$, respectively.
		Thus, for each $i\in [s]$, define
		\[s_i\coloneqq
		\begin{cases}
			r_3 & \text{if }H_i\in \cH;\\
			r^\diamond & \text{if }H_i\in \cH'.
		\end{cases}\]
		Note that $\sum_{i\in [s]}s_i=4qr_3+28r^\diamond=4s'\leq \varepsilon_3 n$ (see \cref{eq:blowupC4-constants} for the definition of $s'$).
		By \cref{claim:H-backwardU**,claim:H-backwardU',claim:H-backwardeU',claim:H-forwardU*} and \cref{claim:H'-backwardU**,claim:H'-backwardU',claim:H'-backwardeU',claim:H'-forwardU*}, \cref{lm:main}\cref{lm:main-H-degree,lm:main-H-degreeV**,lm:main-H-edges,lm:main-H-degreeV*} hold.
		
		Let $E\coloneqq \{uv\in E(T)\mid vu\in \bigcup\cM\}$ (this is precisely the set of edges which cannot be decomposed via \cref{lm:cyclerobustdecomp}). By \cref{thm:blowupC4-CST} and \cref{def:CF-M}, \cref{lm:main}\cref{lm:main-E-forward,lm:main-E-deg} hold. Recall the notation introduced in \cref{eq:blowupC4-Q}.
		By \cref{lm:main} (applied with $\varepsilon_1, \varepsilon_2$, and $\varepsilon_3$ playing the roles of $\varepsilon, \varepsilon'$, and $\eta$),
		there exist disjoint sets
		\[S=\{\cF_{\ell, h,i,j}\mid (\ell,h,i,j)\in Q\}, \enspace S'=\{\cF_{\ell,h,i,j}'\mid (\ell,h,i,j)\in Q'\},\enspace S''=\{\cF_i\mid i\in [n-4s']\}\]
		such that 
		$S^*\coloneqq S\cup S'\cup S''$ is a set of $n$ edge-disjoint feasible systems which satisfy the following properties.
		\begin{enumerate}[label=(\greek*)]
			\item $E(\overleftarrow{T}_\cU)\cup E\subseteq E(S^*)\subseteq E(T)$.\label{thm:blowupC4-backward-T}
			\item For each $\cF\in S^*$, $e(\cF)\leq \varepsilon_2n$.\label{thm:blowupC4-backward-size}
			\newcounter{backward}
			\setcounter{backward}{\value{enumi}}
			\item For each $\cF\in S\cup S'$, $V^0(\cF)=U^*$.\label{thm:blowupC4-backward-SC}
			\item For each $(\ell,h,i,j)\in Q$, we have $\cF_{\ell, h, i,j}\subseteq H_{h,i,j}\setminus E$.\label{thm:blowupC4-backward-LSC}
			\item For each $(\ell,h,i,j)\in Q'$, we have $\cF_{\ell, h, i,j}'\subseteq H_{h,i,j}'\setminus E$.\label{thm:blowupC4-backward-LSC'}
			\item For each $\cF\in S''$, we have $V^+(\cF)\subseteq U_1$ and $V^-(\cF)\subseteq U_4$.\label{thm:blowupC4-backward-endpoints}
		\end{enumerate}
		By \cref{fact:feasibleforward}, we may assume without loss of generality that all the forward edges in $E(S^*)$ are either edges of $E$ or incident to $U^*$. Thus, \cref{lm:main}\cref{lm:main-E-deg} implies that the following holds.
		\begin{enumerate}[resume,label=(\greek*)]
			\item For each $v\in V(T)\setminus U^*$, we have $|\overrightarrow{N}_{S^*}^\pm(v)\setminus U^*|\leq 1$.\label{thm:blowupC4-backward-forward}
		\end{enumerate}
		We next observe that the feasible systems in $S\cup S'$ are localised and balanced special covers.
		
		\begin{claim}\label{claim:SC}
			The following hold.
			\begin{enumerate}[label=\rm(\greek*$'$)]
				\setcounter{enumi}{\value{backward}}
				\item Each $\cF\in S\cup S'$ is a $\cU'$-balanced special cover in $D_1'$ with respect to $U^*$.\label{claim:SC-balanced}
				\item For each $(\ell,h,i,j)\in Q$, $\cF_{\ell,h,i,j}$ is $(\frac{q}{f},4,f,h,i,j)$-localised with respect to $\cP^*, \cC$, and~$\cM$.\label{claim:SC-localised}
				\item For each $(\ell,h,i,j)\in Q'$, $\cF_{\ell,h,i,j}'$ is $(1,4,7,h,i,j)$-localised with respect to $\cP, \cC$, and~$\cM$.\label{claim:SC-localised'}
			\end{enumerate}
		\end{claim}
		
		\begin{proofclaim}
			By assumption, \cref{thm:blowupC4-backward-LSC}, and \cref{thm:blowupC4-backward-LSC'}, all the forward edges in $E(S\cup S')$ are incident to $U^*$. Since by \cref{step:CST} $D_2=T\setminus D_1'\subseteq \overrightarrow{T}_\cU-U^*$, this implies that $E(S\cup S')\subseteq E(D_1')$.
			Moreover, \cref{def:ES-size} implies that $|U_1^*|=\dots=|U_4^*|$ and so \cref{fact:feasibleSC} and \cref{thm:blowupC4-backward-SC} imply that each $\cF\in S\cup S'$ is a $\cU'$-balanced special cover in $D_1'$ with respect to $U^*$. 
			In particular, \cref{claim:SC-balanced} holds.			
			Then, \cref{claim:SC-localised} follows from \cref{claim:H-localised} and \cref{thm:blowupC4-backward-LSC} and, similarly, \cref{claim:SC-localised'} holds by \cref{claim:H'-localised'} and \cref{thm:blowupC4-backward-LSC'}.
		\end{proofclaim}
		
		\item \textbf{Applying the robust decomposition lemma.}\label{step:robustdecomplm}
		In this step, we will apply the robust decomposition lemma (\cref{lm:cyclerobustdecomp}) to obtain a robustly decomposable digraph $D^{\rm rob}$ which will enable us to decompose the leftovers after the approximate decomposition. First, we use \cref{lm:ESF} to construct the required extended special factors.
		
		Recall from \cref{claim:SC} that $S\cup S'$ consists of special covers in $D_1'$ with respect to $U^*$. For each $(\ell,h,i,j)\in Q$, let $M_{\ell,h,i,j}$ denote the complete special sequence associated to $\cF_{\ell, h, i,j}\in S$ (see \cref{def:MSC}). Define a multiset $\sM$ by $\sM\coloneqq \{M_{\ell,h,i,j}\mid (\ell,h,i,j)\in Q\}$.
		For each $(\ell,h,i,j)\in Q'$, let $M_{\ell,h,i,j}'$ denote the complete special sequence associated to $\cF_{\ell, h, i,j}'\in S'$. Define a multiset $\sM'$ by $\sM'\coloneqq \{M_{\ell,h,i,j}\mid (\ell,h,i,j)\in Q'\}$.
		
		First, we construct the $(\frac{q}{f},4,f)$-extended special factors with respect to $\cU', \cP^*, \cC$, and $\cM$ required for \cref{lm:cyclerobustdecomp}. By \cref{thm:blowupC4-backward-size}, \cref{claim:SC-balanced}, and \cref{claim:SC-localised}, \cref{lm:ESF}\cref{lm:ESF-small,lm:ESF-localised,lm:ESF-balanced} hold with $\cF_{\ell,h,i,j}, \frac{q}{f}, \varepsilon_2$, and $\cU'$ playing the roles of $SC_{\ell,h,i,j}, \ell^*, \varepsilon$, and $\cU$.  Moreover, \cref{lm:ESF}\cref{lm:ESF-Mi} follows immediately from \cref{thm:blowupC4-M}.
		Let $D_1''\coloneqq D_1'\setminus (S'\cup S'')$. By \cref{prop:epsremovingadding}\cref{prop:epsremovingadding-supreg}, \cref{thm:blowupC4-supreg}, and \cref{thm:blowupC4-backward-forward}, \cref{lm:ESF}\cref{lm:ESF-supreg} holds with $D_1''$ and $3d$ playing the roles of $D$ and $\varepsilon'$.
		Apply \cref{lm:ESF} with $D_1'', \cU', r_3, \frac{q}{f}, n', \varepsilon_2, 3d$, and $S$ playing the roles of $D, \cU, r, \ell^*,n, \varepsilon, \varepsilon'$, and $SC$ to obtain $r_3$ $(\frac{q}{f},4,f)$-extended special factors $ESF_1, \dots, ESF_{r_3}$ with respect to $\cU', \cP^*, \cC$, and $\cM$ which satisfy the following properties, where for each $(\ell,h,i,j)\in Q$, $ESPS_{\ell,h,i,j}$ denotes the $(\frac{q}{f},4,f,h,i,j)$-extended special path system contained in $ESF_\ell$.
		\begin{enumerate}[label=(\Roman*)]
			\item For each $(\ell,h,i,j)\in Q$, we have $M_{\ell,h,i,j}\subseteq ESPS_{\ell,h,i,j}\subseteq (D_1''\setminus S)\cup M_{\ell,h,i,j}$.\label{thm:blowupC4-ESF-M}
			\item Let $(\ell,h,i,j),(\ell',h',i',j')\in Q$ be distinct. Then, we have $(ESPS_{\ell,h,i,j}\setminus M_{\ell,h,i,j})\cap (ESPS_{\ell',h',i',j'}\setminus M_{\ell',h',i',j'})=\emptyset$.\label{thm:blowupC4-ESF-disjoint}
		\end{enumerate}
		Define a multidigraph $\mathcal{ESF}$ by $\mathcal{ESF}\coloneqq ESF_1\cup \dots\cup ESF_{r_3}$.
		
		Next, we construct the $(1,4,7)$-extended special factors with respect to $\cU', \cP, \cC$, and $\cM$ required for \cref{lm:cyclerobustdecomp}. By \cref{thm:blowupC4-backward-size}, \cref{claim:SC-balanced}, and \cref{claim:SC-localised'}, \cref{lm:ESF}\cref{lm:ESF-small,lm:ESF-localised,lm:ESF-balanced} hold with $\cF_{\ell,h,i,j}', 1,7, \varepsilon_2$, and $\cU'$ playing the roles of $SC_{\ell,h,i,j},\ell^*, f, \varepsilon$, and $\cU$. By \cref{thm:blowupC4-M}, \cref{lm:ESF}\cref{lm:ESF-Mi} holds with $\cP$ playing the role of $\cP^*$.
		Let $D_1'''\coloneqq D_1'\setminus (S\cup S''\cup\mathcal{ESF})$. By \cref{prop:epsremovingadding}\cref{prop:epsremovingadding-supreg}, \cref{thm:blowupC4-supreg}, \cref{thm:blowupC4-backward-forward}, and \cref{cor:ESFreg}, \cref{lm:ESF}\cref{lm:ESF-supreg} holds with $D_1''', \cP$, and $3d$ playing the roles of $D, \cP^*$, and $\varepsilon'$.
		Apply \cref{lm:ESF} with $D_1''',\cU', \cP, r^\diamond, 1, 7, n', \varepsilon_2, 3d$, and $S'$ playing the roles of $D,\cU, \cP^*, r, \ell^*, f, n, \varepsilon, \varepsilon'$, and $SC$ to obtain $r^\diamond$ $(1,4,7)$-extended special factors $ESF_1', \dots, ESF_{r^\diamond}'$ with respect to $\cU',\cP, \cC$, and $\cM$ which satisfy the following properties, where for each $(\ell,h,i,j)\in Q'$, $ESPS_{\ell,h,i,j}'$ denotes the $(1,4,7,h,i,j)$-extended special path system contained in $ESF_\ell'$.
		\begin{enumerate}[label=(\Roman*$'$)]
			\item For each $(\ell,h,i,j)\in Q'$, we have $M_{\ell,h,i,j}'\subseteq ESPS_{\ell,h,i,j}'\subseteq (D_1'''\setminus S')\cup M_{\ell,h,i,j}'$.\label{thm:blowupC4-ESF'-M}
			\item Let $(\ell,h,i,j),(\ell',h',i',j')\in Q'$ be distinct. Then, we have $(ESPS_{\ell,h,i,j}'\setminus M_{\ell,h,i,j}')\cap (ESPS_{\ell',h',i',j'}'\setminus M_{\ell',h',i',j'}')=\emptyset$.\label{thm:blowupC4-ESF'-disjoint}
		\end{enumerate}
		Define a multidigraph $\mathcal{ESF}'$ by $\mathcal{ESF}'\coloneqq ESF_1'\cup \dots\cup ESF_{r^\diamond}'$.
		
		We are now ready to apply the robust decomposition lemma. Let $D_2'\coloneqq D_2\setminus S^*$ and recall from \cref{step:CST} that $D_2\subseteq \overrightarrow{T}_\cU-U^*$.
		By \cref{thm:blowupC4-CST}, \cref{thm:blowupC4-backward-forward}, and \cref{prop:CST}, $(D_2', \cU', \cP, \cP', \cP^*, \cR, \cC, \sU, \sU', \cM)$ is a $(4, \ell', \frac{q}{f}, k,m, \varepsilon_3, \frac{d}{2})$-cycle-setup. Let $D^{\rm rob}$ be the robustly decomposable digraph obtained by applying \cref{lm:cyclerobustdecomp} with $D_2', \cU', 4, \varepsilon_3$, and $\frac{d}{2}$ playing the roles of $D, \cU, K, \varepsilon$, and $d$.
		
		\item \textbf{Approximate decomposition.}\label{step:approxdecomp}
		Let $D\coloneqq T\setminus (S^*\cup \mathcal{ESF}\cup \mathcal{ESF}'\cup D^{\rm rob})$.
		In this step, we will approximately decompose $D\cup S''$ using \cref{lm:blowupC4approxdecomp}.
		By \cref{thm:blowupC4-backward-size,thm:blowupC4-backward-endpoints}, \cref{lm:blowupC4approxdecomp}\cref{lm:blowupC4approxdecomp-size,lm:blowupC4approxdecomp-endpoints} hold with $S'', n-t$, and $\varepsilon_2$ playing the roles of $\{\cF_1, \dots, \cF_\ell\}, \ell$, and $\varepsilon$.
		Moreover, \cref{thm:blowupC4-backward-forward} implies that each $v\in V(T)\setminus U^*$ satisfies
		\begin{equation}\label{eq:blowupC4-S3}
			d_{S''}(v)\leq \overleftarrow{d}_{T, \cU}(v)+|U^*|+2\stackrel{\text{\cref{def:ES}}}{\leq} 2\varepsilon_1 n+4\varepsilon_1 n+2\leq \varepsilon_2n
		\end{equation}
		and so \cref{lm:blowupC4approxdecomp}\cref{lm:blowupC4approxdecomp-deg} holds with $S'', n-t$, and $\varepsilon_2$ playing the roles of $\{\cF_1, \dots, \cF_\ell\}, \ell$, and $\varepsilon$.
		It remains to show that $D[U_i\setminus U^*, U_{i+1}\setminus U^*]$ is an almost regular bipartite robust expander for each $i\in [4]$.
		
		\begin{claim}\label{claim:D}
			Each $v\in V(T)$ satisfies 
			\begin{equation}
				d_D^\pm(v)=
				\begin{cases}
					0 & \text{if }v\in U^*;\\
					n-(4s'-r)-d_{S''}^\pm(v) & \text{otherwise}.
				\end{cases}
			\end{equation}
		\end{claim}
		
		\begin{proofclaim}
			By \cref{def:feasible-exceptional} and since the feasible systems in $S^*$ are edge-disjoint, each $v\in U^*$ satisfies $d_{S^*}^\pm(v)=n$ and so $d_D^\pm(v)=0$, as desired. Let $v\in V(T)\setminus U^*$.
			First, note that
			\begin{equation}\label{eq:blowupC4-ESF}
				|N_{\mathcal{ESF}}^\pm(v)\setminus N_T^\pm(v)|\stackrel{\text{\cref{thm:blowupC4-ESF-M}}}{=}d_{\sM}^\pm(v)\stackrel{\text{\cref{def:SC,def:MSC}}}{=}d_S^\pm(v)
			\end{equation}
			and, similarly, 
			\begin{equation}\label{eq:blowupC4-ESF'}
				|N_{\mathcal{ESF}'}^\pm(v)\setminus N_T^\pm(v)|\stackrel{\text{\cref{thm:blowupC4-ESF'-M}}}{=}d_{\sM'}^\pm(v)\stackrel{\text{\cref{def:SC,def:MSC}}}{=}d_{S'}^\pm(v).
			\end{equation}
			Moreover, \cref{cor:ESFreg} implies that
			\begin{equation}\label{eq:blowupC4-ESF''}
				d_{\mathcal{ESF}}^\pm(v)=(1+3q)r_3 \quad \text{and} \quad d_{\mathcal{ESF}'}^\pm(v)=22r^\diamond.
			\end{equation}
			By \cref{thm:blowupC4-ESF-M}, \cref{thm:blowupC4-ESF-disjoint}, \cref{thm:blowupC4-ESF'-M}, and \cref{thm:blowupC4-ESF'-disjoint}, $\mathcal{ESF}\setminus \sM$ and $\mathcal{ESF}'\setminus \sM'$ are edge-disjoint subdigraphs of $D_1'\setminus S^*\subseteq T\setminus S^*$. 			
			Moreover, \cref{step:robustdecomplm} and \cref{lm:cyclerobustdecomp} imply that $D^{\rm rob}\subseteq D_2'\subseteq T\setminus (S^*\cup \mathcal{ESF}\cup \mathcal{ESF}')$.
			Therefore, \cref{lm:cyclerobustdecomp} implies that
			\begin{align*}
				d_D^\pm(v)&\stackrel{\text{\eqmakebox[Drob]{}}}{=}d_T^\pm(v)-d_{D^{\rm rob}}^\pm(v)-d_{\mathcal{ESF}\setminus \sM}^\pm(v)-d_{\mathcal{ESF}'\setminus \sM'}^\pm(v)-d_{S^*}^\pm(v)\\
				&\stackrel{\text{\eqmakebox[Drob]{\text{\cref{eq:blowupC4-ESF},\cref{eq:blowupC4-ESF'}}}}}{=}n-d_{D^{\rm rob}}^\pm(v)-d_{\mathcal{ESF}}^\pm(v)-d_{\mathcal{ESF}'}^\pm(v)-d_{S''}^\pm(v)\\
				&\stackrel{\text{\eqmakebox[Drob]{\text{\cref{eq:blowupC4-ESF''}}}}}{=}n-(r_1+r_2+5r^\diamond)-(1+3q)r_3-22r^\diamond-d_{S''}^\pm(v)\\ &\stackrel{\text{\eqmakebox[Drob]{\text{\cref{eq:blowupC4-constants}}}}}{=}n-(4s'-r)-d_{S''}^\pm(v),
			\end{align*}
			as desired.
		\end{proofclaim}
		
		Let $\delta\coloneqq \frac{1}{2}\left(1-\frac{4s'-r}{n}\right)$. By \cref{claim:D}, \cref{eq:blowupC4-n'}, and \cref{eq:blowupC4-S3}, each $v\in V(T)\setminus U^*$ satisfies $d_D^\pm(v)=2\delta n\pm \varepsilon_2 n=2(\delta\pm \varepsilon_2)n'$. Since by \cref{thm:blowupC4-backward-T} $D$ only contains forward edges, this implies that $D[U_i\setminus U^*, U_{i+1}\setminus U^*]$ is $(\delta, \varepsilon_2)$-almost regular for each $i\in [4]$. Thus, \cref{lm:blowupC4approxdecomp}\cref{lm:blowupC4approxdecomp-reg} holds with $\varepsilon_2$ playing the role of $\varepsilon$. Moreover, \cref{prop:almostcompleterob} (applied with $\frac{4s'-r}{n}+\varepsilon_2$ playing the role of $\varepsilon$) implies that \cref{lm:blowupC4approxdecomp}\cref{lm:blowupC4approxdecomp-rob} is satisfied.
		By \cref{eq:blowupC4-n'}, we have
		\[n-4s'=2\delta n-r\leq 2\delta n'+2\varepsilon_1n'-r\leq 2(\delta-\eta)n'.\]
		Finally, recall that $U^*$ is an $(\varepsilon_1, \cU)$-exceptional set for $T$. Thus, \cref{fact:U} implies that $U^*$ is also an $(\varepsilon_2, \cU)$-exceptional set for $T$.
		Let $\sC_{\rm approx}$ be the set of $n-4s'$ Hamilton cycles of $T$ obtained by applying \cref{lm:blowupC4approxdecomp} with $n-4s', \varepsilon_2$, and $S''$ playing the roles of $\ell, \varepsilon$, and $\{\cF_1, \dots, \cF_\ell\}$.

		\item \textbf{Absorbing the leftovers.}\label{step:leftovers}
		Finally, we decompose $H\coloneqq D\setminus \sC_{\rm approx}$
		using the robustly decomposable digraph $D^{\rm rob}$.
		Define a multidigraph $D'$ by $D'\coloneqq (H-U^*)\cup D^{\rm rob}\cup \mathcal{ESF}\cup \mathcal{ESF}'$.
		
		\begin{claim}\label{claim:D'}
			$D'\setminus (\sM\cup \sM')$ is a digraph (rather than a multidigraph) and satisfies $E(D'\setminus (\sM\cup \sM'))= E(T\setminus (S\cup S'\cup \sC_{\rm approx}))$.
		\end{claim}
		
		\begin{proofclaim}
			By \cref{thm:blowupC4-ESF-M}, \cref{thm:blowupC4-ESF-disjoint}, \cref{thm:blowupC4-ESF'-M}, and \cref{thm:blowupC4-ESF'-disjoint}, $\mathcal{ESF}\setminus \sM$ and $\mathcal{ESF}'\setminus \sM'$ are digraphs (rather than multidigraphs) and are edge-disjoint subdigraphs of $D_1'\setminus S^*\subseteq D_1'\setminus (S\cup S')$.
			By \cref{step:robustdecomplm}, $D^{\rm rob}\subseteq D_2\setminus S^*\subseteq T\setminus (D_1'\cup S\cup S')$. 
			By definition,
			\begin{equation}\label{eq:blowupC4-H}
				H= D\setminus \sC_{\rm approx}= T\setminus (S^*\cup \mathcal{ESF} \cup \mathcal{ESF}'\cup D^{\rm rob}\cup\sC_{\rm approx}).
			\end{equation}
			Thus, $\mathcal{ESF}\setminus \sM$, $\mathcal{ESF}'\setminus \sM'$, $D^{\rm rob}$, and $H-U^*$ are pairwise edge-disjoint subdigraphs of $T$. Therefore, $D'\setminus (\cM\cup \cM')$ is a digraph. Moreover, \cref{lm:blowupC4approxdecomp} implies that $S''\subseteq \sC_{\rm approx}\subseteq D\cup S''\subseteq T\setminus (S\cup S'\cup \mathcal{ESF} \cup \mathcal{ESF}'\cup D^{\rm rob})$. Thus, \cref{eq:blowupC4-H} implies that $D'\setminus (\sM\cup \sM')=T\setminus (S\cup S'\cup \sC_{\rm approx})$, as desired.
		\end{proofclaim}
		
		We are now ready to decompose $D'$.
		By \cref{lm:blowupC4approxdecomp}, $\sC_{\rm approx}$ is a set of $n-4s'$ edge-disjoint Hamilton cycles of $D\cup S''$ which altogether cover $S''$ and so \cref{claim:D} implies that each $v\in V(T)$ satisfies
		\begin{equation*}
			d_H^\pm(v)=
			\begin{cases}
				0 & \text{if }v\in U^*;\\
				r & \text{otherwise}.
			\end{cases}
		\end{equation*}
		In particular, $H-U^*$ is $r$-regular.
		Moreover, \cref{thm:blowupC4-backward-T} implies that $H$ only consists of forward edges and avoids the edges in $E=\{uv\in E(T)\mid vu\in \bigcup\cM\}$. 
		In particular, $H-U^*$ is a blow-up $C_4$ with vertex partition $\cU'$. By definition, $H\subseteq D\subseteq T\setminus D^{\rm rob}$.
		Thus, \cref{lm:cyclerobustdecomp} implies that the multidigraph $D'$ has a decomposition $\sC_{\rm rob}$ into $4s'$ edge-disjoint Hamilton cycles on $V(T)\setminus U^*$ such that each cycle in $\sC_{\rm rob}$ contains precisely one of the extended special path systems in the multidigraph $\mathcal{ESF}\cup \mathcal{ESF}'$.
		That is, there is an enumeration
		\[\{C_{\ell,h,i,j}\mid (\ell,h,i,j)\in Q\}\cup \{C_{\ell,h,i,j}'\mid (\ell,h,i,j)\in Q'\}\]
		of $\sC_{\rm rob}$ such that
		$C_{\ell,h,i,j}\cap (\mathcal{ESF}\cup \mathcal{ESF}')=ESPS_{\ell,h,i,j}$ for each $(\ell,h,i,j)\in Q$ and
		$C_{\ell',h',i',j'}'\cap (\mathcal{ESF}\cup \mathcal{ESF}')=ESPS_{\ell',h',i',j'}'$ for each $(\ell',h',i',j')\in Q'$.
		
		Recall from \cref{claim:SC-balanced} that $S=\{\cF_{\ell,h,i,j}\mid (\ell,h,i,j)\in Q\}$ and $S'=\{\cF_{\ell,h,i,j}'\mid (\ell,h,i,j)\in Q'\}$ are edge-disjoint sets of edge-disjoint special covers in $D_1'$ with respect to $U^*$. By \cref{step:robustdecomplm}, $\cM$ and $\cM'$ are multisets which consist of the complete special sequences associated to the special covers in $S$ and $S'$, respectively.
		
		For each $(\ell,h,i,j)\in Q$, define $C_{\ell,h,i,j}^*\coloneqq (C_{\ell,h,i,j}\setminus M_{\ell,h,i,j})\cup \cF_{\ell,h,i,j}$ and observe that, by \cref{thm:blowupC4-ESF-M} and \cref{fact:SC}, $C_{\ell,h,i,j}^*$ is a Hamilton cycle of $T$. For each $(\ell,h,i,j)\in Q'$, define $C_{\ell,h,i,j}''\coloneqq (C_{\ell,h,i,j}'\setminus M_{\ell,h,i,j}')\cup \cF_{\ell,h,i,j}'$ and observe that, by \cref{thm:blowupC4-ESF'-M} and \cref{fact:SC}, $C_{\ell,h,i,j}''$ is a Hamilton cycle of $T$. Let
		\[\sC_{\rm rob}'\coloneqq \{C_{\ell,h,i,j}^*\mid (\ell,h,i,j)\in Q\}\cup \{C_{\ell,h,i,j}''\mid (\ell,h,i,j)\in Q'\}.\]
		By \cref{step:backwardedges,step:approxdecomp}, $S,S'$, and $\cC_{\rm approx}$ are pairwise edge-disjoint. Thus, \cref{claim:D'} implies that $\sC_{\rm rob}'$ is a Hamilton decomposition of $(D'\setminus (\sM\cup \sM'))\cup (S\cup S')= T\setminus \sC_{\rm approx}$.
		Recall from \cref{step:approxdecomp} that $\sC_{\rm approx}$ is a set of edge-disjoint Hamilton cycles of $T$.
		Therefore, $\sC_{\rm approx}\cup \sC_{\rm rob}'$ is a Hamilton decomposition of $T$.
		This completes the proof of \cref{thm:blowupC4}.
		\qedhere
	\end{steps}
\end{proof}

\onlyinsubfile{\bibliographystyle{abbrv}
	\bibliography{Bibliography/Bibliography}}

\section{Pseudo-feasible systems}\label{sec:pseudofeasible}

	\onlyinsubfile{
		\setcounter{section}{9}
		\section{Pseudo-feasible systems}}

It remains to prove \cref{lm:main}. First, we observe that one can initially decompose the backward and exceptional edges into structures (called pseudo-feasible systems) which are slightly more general than feasible systems.

\subsection{Definitions}

We now define the concept of a \emph{placeholder} (defined formally below). Rough\-ly speaking, an edge $e$ with precisely one endpoint $v\in U^*$ is called a \emph{placeholder} if $T$ contains many edges of the same type: if $e$ is a forward in/outedge at $v$, then $e$ is a placeholder if $T$ has many forward in/outedges at $v$; similarly, if $e$ is a backward in/outedge at $v$, then $e$ is a placeholder if $T$ has many backward in/outedges at $v$. A placeholder will be used to hold the place for an edge $e'$ of the same type as $e$. A suitable $e'$ will exist since, by definition of a placeholder, there exist many edges which are of the same type as $e$.

\begin{definition}[\Gls*{placeholder}]\label{def:placeholder}
	Let $T$ be a regular bipartite tournament on $4n$ vertices. Let $\cU=(U_1, \dots, U_4)$ be an $(\varepsilon,4)$-partition for $T$ and let $U^*$ be an $(\varepsilon,\cU)$-exceptional set for $T$. For each $i\in [4]$, denote $U_i^*\coloneqq U_i\cap U^*$.
	Let $uv\in E(T)$ and denote by $i,j\in [4]$ the unique indices such that $u\in U_i$ and $v\in U_j$. We say that $uv$ is a \emph{$(\gamma,T)$-placeholder (with respect to $\cU$ and $U^*$)} if one of the following holds.
	\begin{itemize}
		\item $u\in U_i^*$, $v\in U_j\setminus U_j^*$, and $|N_T^+(u)\cap U_j|> \gamma n$.
		\item $u\in U_i\setminus U_i^*$, $v\in U_j^*$, and $|N_T^-(v)\cap U_i|> \gamma n$.
	\end{itemize}
\end{definition}

\begin{fact}\label{fact:placeholderforward}
	Let $0\leq \varepsilon\ll \gamma\leq 1$. Let $T$ be a regular bipartite tournament. Let $\cU=(U_1, \dots, U_4)$ be an $(\varepsilon,4)$-partition for $T$ and $U^*$ be an $(\varepsilon,\cU)$-exceptional set for $T$. Suppose that $e\in E(\overrightarrow{T})$ is a $(\gamma,T)$-placeholder. Then, $V(e)\cap U^{1-\gamma}(T)=\emptyset$.
\end{fact}

Recall from \cref{def:feasible} that a feasible system is a linear forest which contains an appropriate number of backward and exceptional edges. Roughly speaking, we say that $\cF$ is a \emph{pseudo-feasible system} (defined formally below) if the only obstructions to $\cF$ being a feasible system are caused by placeholders. More precisely, a pseudo-feasible system $\cF$ may not form a linear forest (that is, $\cF$ may not satisfy property \cref{def:feasible-linforest} of a feasible system), but all of the cycles in $\cF$ contain a placeholder (see \cref{def:feasible-cycle} below) and all of the excess degree can be accounted for by placeholders (see \cref{def:feasible-degV'} below). Additionally, a pseudo-feasible system $\cF$ may not cover all of the vertices in $U^*$ (that is, $\cF$ may not satisfy property \cref{def:feasible-exceptional}), but the uncovered vertices in $U^*$ have high forward degree and so the missing edges at $U^*$ are forward placeholders (see \cref{def:feasible-degV*} below).

\begin{definition}[\Gls*{pseudo-feasible system}]\label{def:pseudo}
	We say that $\cF$ is a \emph{$(\gamma, T)$-pseudo-feasible system (with respect to $\cU$ and $U^*$)} if $\cF\subseteq T$, \cref{def:feasible-backward} is satisfied, and the following hold.
	\begin{enumerate}[label=\rm(F\arabic*$'$)]
		\setcounter{enumi}{1}
		\item For each $v\in U^*$, both $d_\cF^\pm(v)\leq 1$ and, if $v\in U^{1-\gamma}(T)$, then both $d_\cF^\pm(v)=1$.\label{def:feasible-degV*}
		\item Let $v\in V(T)\setminus U^*$. Then, $\cF$ contains at most one edge which starts at $v$ and is not a $(\gamma,T)$-placeholder. Similarly, $\cF$ contains at most one edge which ends at $v$ and is not a $(\gamma,T)$-placeholder.\label{def:feasible-degV'}
		\item Each cycle in $\cF$ contains a $(\gamma,T)$-placeholder.\label{def:feasible-cycle}
	\end{enumerate}
\end{definition}

As for feasible systems (recall \cref{fact:feasibleforward}), forward edges which are not incident to $U^*$ play no role in a pseudo-feasible system. Additionally, \cref{fact:placeholderforward} implies that all of the forward placeholders can be deleted.

\begin{fact}\label{fact:pseudofeasibleforward}
	Let $0\leq \varepsilon\ll \gamma\leq 1$. Let $T$ be a regular bipartite tournament. Let $\cU=(U_1, \dots, U_4)$ be an $(\varepsilon,4)$-partition for $T$ and $U^*$ be an $(\varepsilon,\cU)$-exceptional set for $T$.
	Let $\cF$ be a $(\gamma,T)$-pseudo-feasible system and $e\in E(\overrightarrow{\cF}_\cU)$.
	If $V(e)\cap U^*=\emptyset$ or $e$ is a $(\gamma, T)$-placeholder,
	then $\cF\setminus \{e\}$ is a $(\gamma,T)$-pseudo-feasible system%
		\COMMENT{$\varepsilon\ll \gamma$ implies that $U^{1-\gamma}(T)\subseteq U^*$. Together with \cref{fact:placeholderforward}, this implies that $V(e)\cap U^{1-\gamma}(T)=\emptyset$.}.
\end{fact}

\subsection{Transforming pseudo-feasible systems into feasible systems: proof overview}\label{sec:constructfeasiblesketch}

The next \lcnamecref{lm:feasible} states that pseudo-feasible systems can be transformed into feasible systems.
The idea behind the proof of \cref{lm:feasible} is to replace the placeholders with edges of the same type to form linear forests and then add forward edges to cover $U^*\setminus U^{1-\gamma}(T)$.

More precisely, let $T$ be a regular bipartite tournament on $4n$ vertices. Let $\cU$ be an $(\varepsilon,4)$-partition for $T$ and let $U^*$ be an $(\varepsilon,\cU)$-exceptional set for $T$.
Suppose that $\cF$ is a $(\gamma, T)$-pseudo-feasible system. We can transform $\cF$ into a feasible system as follows.
Suppose that $\cF$ contains a cycle $C$. By \cref{def:feasible-cycle}, $C$ contains a $(\gamma, T)$-placeholder $e$, say $e$ is a backward edge from $u\in U^*$ to $v\notin U^*$ for instance. 
Then, by definition of a placeholder, $T$ contains many backward edges which start at $u$ and end in $V(T)\setminus U^*$. 
Therefore, we can find a backward edge $e'=uv'$ with $v'\in V(T)\setminus (U^*\cup V(\cF))$. Then, replacing $e$ by $e'$ in $\cF$ breaks the cycle $C$ without affecting \cref{def:feasible-backward} and \cref{def:feasible-degV*,def:feasible-degV',def:feasible-cycle}. Repeating this argument, we can eventually remove all cycles in $\cF$. By \cref{def:feasible-degV',def:feasible-degV*}, we can use similar arguments to ensure that $\Delta^0(\cF)\leq 1$. Then, $\cF$ satisfies \cref{def:feasible-linforest}.
Finally, we add forward edges to ensure that \cref{def:feasible-exceptional} holds as follows. 
Suppose that $x\in U^*$ satisfies $d_\cF^+(x)=0$. Then, \cref{def:feasible-degV*} implies that $x\notin U^{1-\gamma}(T)$ and so $T$ contains many forward outedges at $x$. Thus, we can find a forward edge $e''=xy$ with $y\in V(T)\setminus V(\cF)$. Then, adding $e''$ to $\cF$ does not affect \cref{def:feasible-backward,def:feasible-linforest}. Repeating this argument, $\cF$ eventually satisfies \cref{def:feasible-exceptional} and so $\cF$ becomes a feasible system.

Additionally, we will add forward edges to incorporate a given suitable set of edges $E$ (see \cref{lm:main}\cref{lm:main-backwardedges}) and to ensure that all the components of the feasible systems start in $U_1$ and end in $U_4$ (see \cref{lm:main}\cref{lm:main-endpoints}).

\begin{lm}[Transforming pseudo-feasible systems into feasible ones]\label{lm:feasible}
	Let $0<\frac{1}{n}\ll \varepsilon\ll \eta \ll \gamma\ll 1$ and $(1-\eta)n\leq r\leq n$. Let $T$ be a regular bipartite tournament on $4n$ vertices. Let $\cU=(U_1, \dots, U_4)$ be an $(\varepsilon, 4)$-partition for $T$ and $U^*$ be an $(\varepsilon, \cU)$-exceptional set for $T$.
	Suppose that $D\subseteq T$ satisfies $\delta^0(D)\geq r$.
	Let $\cF_1, \dots, \cF_r$ be edge-disjoint $(\gamma, T)$-pseudo-feasible systems satisfying the following properties.
	\begin{enumerate}
		\item $E(\overleftarrow{D}_\cU)\cup E(D[U^*])\subseteq \bigcup_{i\in [r]}E(\cF_i)\subseteq E(D)$.\label{lm:feasible-F-E}
		\item For each $i\in [r]$, $e(\cF_i)\leq \varepsilon n$.\label{lm:feasible-F-e}
	\end{enumerate}
	Let $E\subseteq E(D)$ be such that the following hold.
	\begin{enumerate}[resume]
		\item $E\subseteq E(\overrightarrow{D}_\cU-U^*)$.\label{lm:feasible-E-forward}
		\item For each $v\in V(T)\setminus U^*$, $d_E^\pm(v)\leq 1$.\label{lm:feasible-E-deg}
	\end{enumerate}
	Then, there exist edge-disjoint feasible systems $\cF_1', \dots, \cF_r'$ such that the following hold.
	\begin{enumerate}[label=\rm(\alph*)]
		\item $E(\overleftarrow{D}_\cU)\cup E\subseteq \bigcup_{i\in [r]}E(\cF_i')\subseteq E(D)$.\label{lm:feasible-E}
		\item For each $i\in [r]$, $e(\cF_i')\leq \varepsilon^{\frac{1}{3}}n$.\label{lm:feasible-e}
		\item For each $v\in V(T)\setminus U^*$, there exist at most $\varepsilon^{\frac{1}{3}}n$ indices $i\in [r]$ such that $v\in V(\cF_i')$.\label{lm:feasible-deg}
		\item For each $i\in [r]$, $V^+(\cF_i')\subseteq U_1$ and $V^-(\cF_i')\subseteq U_4$.\label{lm:feasible-endpoints}
	\end{enumerate}
\end{lm}

Note that $E(D[U^*])$ appears in \cref{lm:feasible}\cref{lm:feasible-F-E} for technical reasons (this ensures that all the edges available for transforming the pseudo-feasible systems into feasible ones do not entirely lie in the exceptional set $U^*$). On the other hand, $E(D[U^*])$ does not need to explicitly appear in \cref{lm:feasible}\cref{lm:feasible-E} since these edges will automatically be covered by definition of a feasible system. Indeed, recall from \cref{lm:main} that we aim to construct $n$ edge-disjoint feasible systems. But property \cref{def:feasible-exceptional} of a feasible system states that each exceptional vertex is covered by both an in- and an outedge, so any set of $n$ edge-disjoint feasible systems automatically covers all the edges incident to $U^*$.

In practice, the above argument needs to be carried out to all of the pseudo-feasible systems in parallel. To gain intuition, we first derive \cref{lm:main} and defer the proof of \cref{lm:feasible} to \cref{sec:constructfeasible}.

\subsection{Proof of Lemma \ref{lm:main}}

By \cref{lm:feasible}, it is enough to decompose the backward edges into pseudo-feasible systems (rather than feasible ones).
However, recall from \cref{lm:main}\cref{lm:main-H} that we require a few of the feasible systems to be constructed out of prescribed sets $H_1, \dots, H_s$ of edges of $T$. It is therefore more convenient to construct these feasible systems straight away. Thus, it is most convenient to prove the following pseudo-feasible system analogue of \cref{lm:main}.
(The proof of \cref{lm:backwardedges}, as well as a detailed proof overview, can be found in \cref{sec:constructpseudofeasible}.)

\begin{lm}[Decomposing the backward and exceptional edges into pseudo-feasible systems]\label{lm:backwardedges}
	Let $0<\frac{1}{n}\ll \varepsilon\ll \eta \ll\gamma \ll 1$ and $s\in \mathbb{N}$. Let $T$ be a regular bipartite tournament on $4n$ vertices. Let $\cU=(U_1, \dots, U_4)$ be an optimal $(\varepsilon, 4)$-partition for $T$ and $U^*$ be an $(\varepsilon, \cU)$-exceptional set for $T$.
	Suppose that, for each $i\in [s]$, $H_i\subseteq T$ satisfies \cref{lm:main}\cref{lm:main-H-degree,lm:main-H-degreeV**,lm:main-H-degreeV*,lm:main-H-edges}.
	For each $i\in [s]$, let $s_i\in \mathbb{N}$ and $t_i\coloneqq \sum_{j\in [i-1]} s_j$. Let $t\coloneqq \sum_{i\in [s]} s_i$ and suppose that $t\leq \eta n$.	
	Then, there exist edge-disjoint $(\gamma,T)$-pseudo-feasible systems $\cF_1, \dots, \cF_n$ for which the following hold.
	\begin{enumerate}[label=\rm(\alph*),ref=(\alph*)]
		\item $E(\overleftarrow{T}_\cU)\cup E(T[U^*])\subseteq \bigcup_{i\in [n]}E(\cF_i)\subseteq E(T)$.\label{lm:backwardedges-backwardedges}
		\item For each $i\in [n]$, $e(\cF_i)\leq \sqrt{\varepsilon}n$.\label{lm:backwardedges-size}
		\item For each $i\in [t]$, $\cF_i$ is a feasible system with $V^0(\cF_i)=U^*$.\label{lm:backwardedges-feasible}
		\item For each $i\in [s]$ and $j\in [s_i]$, $\cF_{t_i+j}\subseteq H_i$\label{lm:backwardedges-H}.
	\end{enumerate}
\end{lm}

\begin{proof}[Proof of \cref{lm:main}]
	Let $\cF_1^*, \dots, \cF_n^*$ be the $(\gamma, T)$-pseudo-feasible systems obtained by applying \cref{lm:backwardedges}. For each $i\in [t]$, let $\cF_i\coloneqq \cF_i^*\setminus E$. By \cref{lm:main-E-forward}, \cref{fact:feasibleforward}, and \cref{lm:backwardedges}\cref{lm:backwardedges-feasible,lm:backwardedges-H}, $\cF_1, \dots, \cF_t$ are feasible systems which satisfy \cref{lm:main-V0,lm:main-H}.
	
	We transform $\cF_{t+1}^*, \dots, \cF_{n}^*$ into feasible systems using \cref{lm:feasible} as follows.
	Let $r\coloneqq n-t$ and $D\coloneqq T\setminus \bigcup_{i\in [t]}\cF_i$. By \cref{def:feasible-linforest}, $\delta^0(D)\geq r$ and, by \cref{lm:backwardedges}\cref{lm:backwardedges-backwardedges,lm:backwardedges-size}, \cref{lm:feasible}\cref{lm:feasible-F-E,lm:feasible-F-e} hold with $\cF_{t+1}^*, \dots, \cF_{n}^*$, and $\sqrt{\varepsilon}$ playing the roles of $\cF_1, \dots, \cF_r$, and $\varepsilon$.
	By construction, $E\subseteq E(D)$ and so \cref{lm:feasible}\cref{lm:feasible-E-deg,lm:feasible-E-forward} follow from \cref{lm:main-E-forward,lm:main-E-deg}. 
	Let $\cF_{t+1}, \dots, \cF_n$ be the feasible systems obtained by applying \cref{lm:feasible} with $\cF_{t+1}^*, \dots, \cF_n^*$, and $\sqrt{\varepsilon}$ playing the roles of $\cF_1, \dots, \cF_r$, and $\varepsilon$. Then, \cref{lm:main-backwardedges} follows from \cref{lm:feasible}\cref{lm:feasible-E} and \cref{lm:backwardedges}\cref{lm:backwardedges-backwardedges}, while \cref{lm:main-size} follows from \cref{lm:feasible}\cref{lm:feasible-e} and \cref{lm:backwardedges}\cref{lm:backwardedges-size}. Finally, \cref{lm:main-endpoints} holds by \cref{lm:feasible}\cref{lm:feasible-endpoints}.
\end{proof}

\section{Transforming pseudo-feasible systems into feasible systems: proof of Lemma~\ref{lm:feasible}}\label{sec:constructfeasible}

	\onlyinsubfile{
		\setcounter{section}{10}
\section{Proof of Lemma \ref{lm:feasible}}}

We proceed as described in \cref{sec:constructfeasiblesketch}. First, we redistribute all the placeholders contained in $\cF_1, \dots, \cF_r$ to break all the cycles and reduce the maximum semidegree to~$1$ (\cref{lm:redistribute}). Then, we add some forward edges to cover $U^*$ and thus form feasible systems (\cref{lm:coverU*}). Next, we incorporate the set $E$ of prescribed edges (\cref{lm:incorporateE}). Finally, we add some additional forward edges to ensure that each component of the feasible systems have its endpoints in the desired vertex classes (\cref{cor:endpoints}).

\subsection{Extending linear forests}

As mentioned above, the feasible systems will be constructed in stages. At each stage, we will consider linear forests and need to extend them in a prescribed way (e.g.\ in \cref{lm:coverU*} we will need to cover precisely the uncovered vertices in $U^*$). Most of the time, this will be done using the next \lcnamecref{lm:extendlinforest}.

Roughly speaking, \cref{lm:extendlinforest} states that a sufficiently dense bipartite digraph $D$ on vertex classes $A$ and $B$ contains edge-disjoint linear forests $\cQ_1, \dots, \cQ_\ell$, where each $\cQ_i$ covers a prescribed set $S_i^+\subseteq A$ with outedges, covers a prescribed set $S_i^-\subseteq A$ with inedges, and avoids a prescribed set $T_i\subseteq B$ (see \cref{lm:extendlinforest}\cref{lm:extendlinforest-cover} below). Moreover, these linear forests can be constructed in such a way that every vertex of $B$ is not covered by too many of the linear forests (see \cref{lm:extendlinforest}\cref{lm:extendlinforest-N}) and is adjacent to at most one edge in each linear forest (see \cref{lm:extendlinforest}\cref{lm:extendlinforest-degree1}).

In general, given linear forests $\cF_1, \dots, \cF_\ell$ that we want to extend, we will apply \cref{lm:extendlinforest} with $T_i=V(E(\cF_i))\cap B$ for each $i\in [\ell]$. This will ensure that the linear forests $\cQ_1, \dots, \cQ_\ell$ guaranteed by \cref{lm:extendlinforest} can be incorporated into $\cF_1, \dots, \cF_\ell$ to form larger linear forests. For each $i\in [\ell]$, the sets $S_i^+$ and $S_i^-$ will correspond to the sets of vertices that need to be covered in $\cF_i$ with out- and inedges (e.g.\ in the proof of \cref{lm:coverU*} we will apply this to the vertices of $U^*$ which are not yet covered with out- and inedges by $\cF_i$). 

Note that \cref{lm:extendlinforest}\cref{lm:extendlinforest-degree1} will not be used until \cref{sec:specialpseudofeasible}.

\begin{lm}\label{lm:extendlinforest}
	Let $D$ be a bipartite digraph on vertex classes $A$ and $B$.
	For each $i\in [\ell]$, let $S_i^+, S_i^-\subseteq A$ and $T_i\subseteq B$.
	For each $v\in B$, let $n_v$ denote the number of indices $i\in [\ell]$ such that $v\in T_i$.
	Let $1\leq N\leq 2|A|$.
	For each $i\in [\ell]$, $\diamond\in \{+,-\}$, and $v\in S_i^\diamond$,
	denote
	\[c_{i,\diamond, v}\coloneqq \max\{|\{i'\in [\ell]\mid v\in S_{i'}^\diamond\}|, 2(|S_i^+|+|S_i^-|+|T_i|), 2(\max_{w\in B}n_w+N)\}\]		
	and suppose that
	\begin{equation}\label{eq:extendlinforest}
		d_D^\diamond(v)\geq
		\begin{cases}
			c_{i,\diamond, v} & \text{if }N=2|A|;\\
			\frac{1}{\lfloor N\rfloor}\sum_{i'\in [\ell]}(|S_{i'}^+|+|S_{i'}^-|)+c_{i,\diamond, v} & \text{if }N<2|A|.\\
		\end{cases}
	\end{equation}
	Then, $D$ contains edge-disjoint linear forests $\cQ_1, \dots, \cQ_\ell$ such that the following hold.
	\begin{enumerate}[label=\rm(\alph*)]
		\item For each $i\in [\ell]$, $\cQ_i$ consists of a matching of $D(B\setminus T_i, S_i^-)$ of size $|S_i^-|$ and a matching of $D(S_i^+,B\setminus T_i)$ of size $|S_i^+|$.\label{lm:extendlinforest-cover}
		\item For each $v\in B$, there exists at most $N$ indices $i\in [\ell]$ such that $v\in V(\cQ_i)$.\label{lm:extendlinforest-N}
		\item For each $i\in [\ell]$ and $v\in B$, we have $d_{\cQ_i}(v)\leq 1$ (i.e.\ the two matchings in \cref{lm:extendlinforest-cover} do not intersect in $B$).\label{lm:extendlinforest-degree1}
	\end{enumerate}
	In particular, if $\cF_1, \dots, \cF_\ell$ are linear forests which are edge-disjoint from each other and from $D$ such that both 
	\[S_i^\pm \cap V(\cF_i)\subseteq V^\mp(\cF_i) \quad \text{and} \quad V(E(\cF_i))\cap B\subseteq T_i\]
	for each $i\in [\ell]$, then $\cF_1\cup \cQ_1, \dots, \cF_\ell\cup \cQ_\ell$ are edge-disjoint linear forests.
\end{lm}

\begin{proof}
	Note that the ``in particular part" follows immediately from \cref{lm:extendlinforest-cover,lm:extendlinforest-degree1}. Thus, it suffices to construct edge-disjoint linear forests $\cQ_1, \dots, \cQ_\ell$ which satisfy \cref{lm:extendlinforest-N,lm:extendlinforest-cover,lm:extendlinforest-degree1}. 
	Let $S^\pm \coloneqq \bigcup_{i\in [\ell]}S_i^\pm$ and denote $S\coloneqq \{(+,v)\mid v\in S^+\}\cup \{(-,v)\mid v\in S^-\}$. We will consider each tuple $(\diamond, v)\in S$ in turn and, at each stage, choose all the edges corresponding to the current tuple $(\diamond, v)\in S$ (that is, all the outedges at $v$ if $\diamond=+$ and all the inedges at $v$ if $\diamond=-$).
	
	Suppose inductively that, for some $0\leq k\leq |S|$, there exist $S^k\subseteq S$ of size $k$ and edge-disjoint linear forests $\cQ_1^k, \dots, \cQ_\ell^k$ such that the following hold, where $S^{\pm, k}\coloneqq \{v \mid (\pm, v)\in S^k\}$.
	\begin{enumerate}[label=(\greek*)]
		\item For each $i\in [\ell]$, $\cQ_i^k$ consists of a matching of $D(B\setminus T_i, S_i^-\cap S^{-,k})$ of size $|S_i^-\cap S^{-,k}|$ and a matching of $D(S_i^+\cap S^{+,k},B\setminus T_i)$ of size $|S_i^+\cap S^{+,k}|$.\label{lm:extendlinforest-IH-cover}
		\item For each $v\in B$, there exist at most $N$ indices $i\in [\ell]$ such that $v\in V(\cQ_i^k)$.\label{lm:extendlinforest-IH-N}
		\item For each $i\in [\ell]$ and $v\in B$, we have $d_{\cQ_i^k}(v)\leq 1$ (i.e.\ the two matchings in \cref{lm:extendlinforest-IH-cover} do not intersect in $B$).\label{lm:extendlinforest-IH-degree1}
	\end{enumerate}
	First, suppose that $k=|S|$. Let $\cQ_i\coloneqq \cQ_i^k$ for each $i\in [\ell]$. Then, \cref{lm:extendlinforest-N,lm:extendlinforest-cover,lm:extendlinforest-degree1} follow from \cref{lm:extendlinforest-IH-N,lm:extendlinforest-IH-cover,lm:extendlinforest-IH-degree1}.
	
	We may therefore assume that $k<|S|$. Let $(\diamond, v)\in S\setminus S^k$ and define $S^{k+1}\coloneqq S^k\cup \{(\diamond,v)\}$.
	Let $X$ be the set of vertices $w\in B$ such that there exist $\lfloor N\rfloor $ indices $i\in [\ell]$ such that $w\in V(\cQ_i^k)$ (so $X$ is the set of vertices of $B$ that cannot be used anymore). Let $Y\coloneqq \{i\in [\ell]\mid v\in S_i^\diamond\}$ (so $Y$ lists the $\cQ_i^k$'s to which we need to add an edge incident to $v$ in this step). 
	
	\begin{claim}
		For each $i\in Y$, the following hold.
		\begin{enumerate}[label=\rm(\Roman*)]
			\item $d_D^\diamond(v)\geq |X|+|Y|$.\label{cor:extendlinforest-delta1}
			\item $d_D^\diamond(v)\geq |X|+2(|S_i^+|+|S_i^-|+|T_i|)$.\label{cor:extendlinforest-delta2}
			\item $d_D^\diamond(v)\geq |X|+2(\max_{w\in B}n_w+N)$.\label{cor:extendlinforest-delta3}
		\end{enumerate}
	\end{claim}
	
	\begin{proofclaim}
		If $N= 2|A|$, then $N\geq|S|>k$ and so $X=\emptyset$. Thus, \cref{cor:extendlinforest-delta1,cor:extendlinforest-delta2,cor:extendlinforest-delta3} follow immediately from \cref{eq:extendlinforest}.
		We may therefore assume that $N<2|A|$.
		Since $D$ is a bipartite graph on vertex classes $A$ and $B$, we have
		\begin{equation*}\label{eq:extendlinforest-X}
			|X|\leq\frac{\sum_{i\in [\ell]}e(\cQ_i^k)}{\lfloor N\rfloor}\stackrel{\text{\cref{lm:extendlinforest-IH-cover}}}{\leq}\frac{1}{\lfloor N\rfloor}\sum_{i'\in [\ell]}(|S_{i'}^+|+|S_{i'}^-|).
		\end{equation*}
		Therefore, \cref{cor:extendlinforest-delta1,cor:extendlinforest-delta2,cor:extendlinforest-delta3} follow from \cref{eq:extendlinforest}.
	\end{proofclaim}
	
	If $\diamond=+$, then let $Z\coloneqq \{vw\in E(D)\mid w\notin X\}$; otherwise, let $Z\coloneqq \{wv\in E(D)\mid w\notin X\}$ (so $Z$ consists of the edges of $D$ that we may use to extend $\cQ_1^k, \dots, \cQ_\ell^k$ in this step). 
	Let $G$ be the auxiliary bipartite graph on vertex classes $Y$ and $Z$ defined as follows.
	For each $i\in Y$ and $e\in Z$, $ie\in E(G)$ if and only if $V(e)\cap B\cap (V(\cQ_i^k)\cup T_i)=\emptyset$.
	Note that
	\begin{align}
		|Z|&\geq d_D^\diamond(v)-|X|\label{eq:cor:extendlinforest-Z}\\
		&\stackrel{\text{\cref{cor:extendlinforest-delta1}}}{\geq} |Y|.\nonumber
	\end{align} 
	Then, each $i\in Y$ satisfies
	\begin{align*}
		d_G(i)\geq |Z|-|V(\cQ_i^k)\cap B|-|T_i|
		\stackrel{\text{\cref{lm:extendlinforest-IH-cover}}}{\geq}|Z|-(|S_i^+|+|S_i^-|)-|T_i|
		\stackrel{\text{\cref{cor:extendlinforest-delta2}}}{\geq}|Z|-\frac{d_D^\diamond(v)-|X|}{2}\stackrel{\text{\cref{eq:cor:extendlinforest-Z}}}{\geq}\frac{|Z|}{2}.
	\end{align*}
	Let $e\in Z$ and denote by $w$ the (unique) vertex $w\in V(e)\cap B$. Let $n_w'$ be the number of indices $i\in [\ell]$ such that $w\in V(\cQ_i^k)\cup T_i$. Then,
	\begin{align*}
		d_G(e)&\geq |Y|-n_w'\stackrel{\text{\cref{lm:extendlinforest-IH-N}}}{\geq}|Y|-(n_w+N)
		\stackrel{\text{\cref{cor:extendlinforest-delta3}}}{\geq} |Y|-\frac{d_D^\diamond(v)-|X|}{2}\stackrel{\text{\cref{eq:cor:extendlinforest-Z}}}{\geq}|Y|-\frac{|Z|}{2}.
	\end{align*}
	Apply \cref{prop:Hall} with $Y$ and $Z$ playing the roles of $A$ and $B$ to obtain a matching $M$ of $G$ which covers $Y$.
	For each $i\in Y$, let $e_i$ denote the (unique) neighbour of $i$ in $M$ and let $\cQ_i^{k+1}\coloneqq \cQ_i^k\cup \{e_i\}$. For each $i\in [\ell]\setminus Y$, let $\cQ_i^{k+1}\coloneqq \cQ_i^k$.
	
	Since $M$ is a matching, $\cQ_1^{k+1}, \dots, \cQ_\ell^{k+1}$ are pairwise edge-disjoint. 
	Moreover, the definition of $G$ and the induction hypothesis imply that \cref{lm:extendlinforest-IH-degree1} holds with $k+1$ playing the role of $k$ and $\cQ_i^{k+1}$ is a linear forest for each $i\in [\ell]$.
	By definition of $Y$ and $G$ and the induction hypothesis, \cref{lm:extendlinforest-IH-cover} holds with $k+1$ playing the role of $k$. Finally, 
	\cref{lm:extendlinforest-IH-N} holds with $k+1$ playing the role of $k$ by definition of $X$ and $Z$ and the induction hypothesis.
\end{proof}

\subsection{Proof of Lemma~\ref{lm:feasible}}

We are now ready to prove \cref{lm:feasible}, which states that edge-disjoint pseudo-feasible systems can be transformed into edge-disjoint feasible systems.
As discussed at the start of \cref{sec:constructfeasible}, we spilt the proof into several lemmas. First, we redistribute placeholders to break all the cycles and reduce the maximum semidegree to $1$.

\begin{lm}[Redistributing placeholders]\label{lm:redistribute}
	Let $0<\frac{1}{n}\ll \varepsilon\ll \eta \ll \gamma\ll 1$ and $(1-\eta)n\leq r\leq n$. Let $T$ be a regular bipartite tournament on $4n$ vertices. Let $\cU=(U_1, \dots, U_4)$ be an $(\varepsilon, 4)$-partition for $T$ and $U^*$ be an $(\varepsilon, \cU)$-exceptional set for $T$.
	Suppose that $D\subseteq T$ satisfies $\delta^0(D)\geq r$.
	Let $\cF_1, \dots, \cF_r$ be edge-disjoint $(\gamma, T)$-pseudo-feasible systems which satisfy the following properties.
	\begin{enumerate}
		\item $E(\overleftarrow{D}_\cU)\cup E(D[U^*])\subseteq \bigcup_{i\in [r]}E(\cF_i)\subseteq E(D)$.\label{lm:redistribute-E}
		\item For each $i\in [r]$, $e(\cF_i)\leq \varepsilon n$.\label{lm:redistribute-e}
		\item Suppose that $e\in \bigcup_{i\in [r]}E(\cF_i)$ is a forward edge. Then, $V(e)\cap U^*\neq \emptyset$ and $e$ is not a $(\gamma,T)$-placeholder.\label{lm:redistribute-forward}
	\end{enumerate}
	Then, there exist edge-disjoint $(\gamma, T)$-pseudo-feasible systems $\cF_1', \dots, \cF_r'$ such that the following hold.
	\begin{enumerate}[label=\rm(\alph*)]
		\item $\bigcup_{i\in [r]}E(\cF_i')=\bigcup_{i\in [r]}E(\cF_i)$. In particular, $E(\overleftarrow{D}_\cU)\cup E(D[U^*])\subseteq \bigcup_{i\in [r]}E(\cF_i')\subseteq E(D)$.\label{lm:redistribute-redistribute}
		\item For each $i\in [r]$, $e(\cF_i')= e(\cF_i)\leq \varepsilon n$.\label{lm:redistribute-size}
		\item For each $i\in [r]$, $\cF_i'$ is a linear forest.\label{lm:redistribute-linforest}
	\end{enumerate}
\end{lm}

\begin{proof}
	Let $E$ be the set of $(\gamma,T)$-placeholders contained in $\bigcup_{i\in [\ell]}E(\cF_i)$.
	For each $i\in [r]$,
	let $\widetilde{\cF}_i$ be obtained from $\cF_i\setminus E$ by removing all isolated vertices%
		\COMMENT{Why needed?}	
	and denote $E_i\coloneqq E\cap E(\cF_i)$. Note that, by \cref{def:feasible-cycle,def:feasible-degV',def:feasible-degV*}, $\widetilde{\cF}_1, \dots, \widetilde{\cF}_r$ are linear forests. Thus, we may assume without loss of generality that $E\neq \emptyset$ and so, by \cref{def:placeholder}, $U^*\neq \emptyset$%
		\COMMENT{Needed for $N\geq 1$.}.
	
	We will redistribute the placeholders in $E$ into the linear forests $\tF_1, \dots, \tF_r$ using \cref{lm:extendlinforest}. More precisely, we will add, for each $v\in U^*$ and $i\in [r]$, an in/outedge at $v$ from $E$ to $\tF_i$ if and only if $\cF_i$ contains a placeholder which is an in/outedge at $v$.
		
	Let $A\coloneqq U^*$ and $B\coloneqq V(T)\setminus U^*$.
	Let $D'$ be the digraph on $V(T)$ defined by $E(D')\coloneqq E$. By \cref{def:placeholder}, $D'$ is a bipartite digraph on vertex classes $A$ and $B$.
	For each $i\in [r]$, let $S_i^+,S_i^-\subseteq U^*$ be the sets of vertices which are incident to an out/inedge in $E_i$, respectively (so $S_i^+$ and $S_i^-$ list the vertices in $U^*$ which we need to cover with an out/inedge from $E$) and define $T_i\coloneqq V(E(\widetilde{\cF}_i))\cap B$.
	Note for later that \cref{def:feasible-degV*} implies that both
	\begin{equation}\label{eq:feasible-ST}
		S_i^\pm\cap V(\widetilde{\cF}_i)\subseteq  V^\mp(\widetilde{\cF}_i) \quad \text{and} \quad V(E(\widetilde{\cF}_i))\cap B\subseteq T_i
	\end{equation}
	for each $i\in [r]$.
	Define $N\coloneqq 2|U^*|$.
	For each $v\in B$, let $n_v$ denote the number of indices $i\in [r]$ such that $v\in T_i$.
	
	We verify that \cref{eq:extendlinforest} holds with $D'$ and $r$ playing the roles of $D$ and $\ell$.
	By \cref{def:placeholder}, each $i\in [r]$ satisfies
	\begin{align}
		|S_i^+|+|S_i^-|&\stackrel{\text{\eqmakebox[feasibleS]{}}}{=}|E_i|\label{eq:feasible-Siexact}\\
		&\stackrel{\text{\eqmakebox[feasibleS]{\text{\cref{def:feasible-degV*}}}}}{\leq} 2|U^*|\stackrel{\text{\cref{def:ES-size}}}{\leq} 8\varepsilon n\label{eq:feasible-Si}
	\end{align}
	and
	\begin{equation}\label{eq:feasible-Ti}
		|T_i|\leq |V(E(\cF_i))|\stackrel{\text{\cref{lm:redistribute-e}}}{\leq} 2\varepsilon n.
	\end{equation}
	Moreover, each $v\in B=V(T)\setminus U^*$ satisfies
	\begin{align}\label{eq:feasible-nv}
		n_v\leq \sum_{i\in [r]}d_{\cF_i}(v)\stackrel{\text{\cref{lm:redistribute-E},\cref{lm:redistribute-forward}}}{\leq} \overleftarrow{d}_{D,\cU}(v)+|\overrightarrow{N}_{D,\cU}(v)\cap U^*|\stackrel{\text{\cref{def:ES-backward}}}{\leq}2\varepsilon n+|U^*|\stackrel{\text{\cref{def:ES-size}}}{\leq} 6\varepsilon n.
	\end{align}
	Therefore, each $v\in U^\gamma(T)\subseteq U^*$%
		\COMMENT{This holds by \cref{def:ES-backward}.}
	satisfies 
	\begin{align}
		d_{D'}^\pm(v)&\stackrel{\text{\eqmakebox[feasibledeg]{}}}{=}d_E^\pm(v)=|\{i\in [r]\mid v\in S_i^\pm\}|\label{eq:feasible-redistribute}\\
		&\stackrel{\text{\eqmakebox[feasibledeg]{\text{\cref{def:placeholder},\cref{lm:redistribute-E}}}}}{=} \overleftarrow{d}_{D,\cU}^\pm(v)-|\overleftarrow{N}_{D,\cU}^\pm(v)\cap U^*|\nonumber\\
		&\stackrel{\text{\eqmakebox[feasibledeg]{}}}{\geq}
		\overleftarrow{d}_{T,\cU}^\pm(v)-\Delta^0(T\setminus D)-|\overleftarrow{N}_{T,\cU}^\pm(v)\cap U^*|\nonumber\\
		&\stackrel{\text{\eqmakebox[feasibledeg]{}}}{\geq} \gamma n-(n-r)-|U^*|\nonumber\\
		&\stackrel{\text{\eqmakebox[feasibledeg]{\text{\cref{def:ES-size},\cref{eq:feasible-Si,eq:feasible-Ti,eq:feasible-nv}}}}}{\geq} 2\max_{i\in [r]}(|S_i^+|+|S_i^-|+|T_i|)+2(\max_{w\in B}n_w+N).\label{eq:feasible-deg}
	\end{align}
	By \cref{lm:redistribute-forward}, all the edges in $E$ are backward edges and so \cref{def:placeholder} implies that each edge in $E$ is incident to a vertex in $U^\gamma(T)\subseteq U^*$.
	Thus, $S_i^+\cup S_i^-\subseteq U^\gamma(T)$ for each $i\in [r]$ and so \cref{eq:extendlinforest} follows from \cref{eq:feasible-deg,eq:feasible-redistribute}.
	
	Let $\cQ_1, \dots, \cQ_r$ be the edge-disjoint linear forests obtained by applying \cref{lm:extendlinforest} with $D'$ and $r$ playing the roles of $D$ and $\ell$.
	For each $i\in [r]$, denote $\cF_i'\coloneqq \widetilde{\cF}_i\cup \cQ_i$. We claim that $\cF_1', \dots, \cF_r'$ are edge-disjoint $(\gamma,T)$-pseudo-feasible systems which satisfy \cref{lm:redistribute-redistribute,lm:redistribute-size,lm:redistribute-linforest}.
	By construction, $\widetilde{\cF}_1, \dots, \widetilde{\cF}_r$ are edge-disjoint from each other and from $D'$. Thus, \cref{eq:feasible-ST} and the ``in particular part" of \cref{lm:extendlinforest} imply that $\cF_1', \dots, \cF_r'$ are edge-disjoint linear forests. In particular, \cref{lm:redistribute-linforest}, \cref{def:feasible-degV'}, and \cref{def:feasible-cycle} are satisfied.
	
	By \cref{lm:extendlinforest}\cref{lm:extendlinforest-cover}, $\bigcup_{i\in [r]}E(\cF_i')\subseteq \bigcup_{i\in [r]}E(\widetilde{\cF}_i)\cup E=\bigcup_{i\in [r]}E(\cF_i)$.
	Moreover, 
	\[\sum_{i\in [r]}e(\cF_i'\setminus \widetilde{\cF}_i)\stackrel{\text{\cref{lm:extendlinforest}\cref{lm:extendlinforest-cover}}}{=}\sum_{i\in [r]}(|S_i^+|+|S_i^-|)\stackrel{\text{\cref{eq:feasible-Siexact}}}{=}\sum_{i\in [r]}e(\cF_i\setminus \widetilde{\cF}_i).\]
	Thus, \cref{lm:redistribute-redistribute} is satisfied. 
	For each $i\in [r]$,
	\[e(\cF_i')\stackrel{\text{\cref{lm:extendlinforest}\cref{lm:extendlinforest-cover}}}{=}e(\widetilde{\cF}_i)+|S_i^+|+|S_i^-|\stackrel{\text{\cref{eq:feasible-Siexact}}}{=}e(\cF_i)\stackrel{\text{\cref{lm:redistribute-e}}}{\leq} \varepsilon n,\]
	so \cref{lm:redistribute-size} is satisfied.	
	Let $j\in [r]$. By \cref{lm:extendlinforest}\cref{lm:extendlinforest-cover} and definition of $S_j^\pm$, each $v\in U^*$ satisfies
	\begin{equation}\label{eq:redistribute-placeholder}
		\overleftarrow{d}_{\cF_j',\cU}^\pm(v)=\overleftarrow{d}_{\cF_j,\cU}^\pm(v).
	\end{equation}	
	Thus, \cref{def:feasible-degV*} follows from the fact that $\cF_j$ is a $(\gamma,T)$-pseudo-feasible system. Recall that $E(\cF_j'\setminus\widetilde{\cF}_j)\cup E(\cF_j\setminus \widetilde{\cF}_j)\subseteq E$ and so, by \cref{def:placeholder} and \cref{lm:redistribute-forward},  $E(\cF_j'\setminus\widetilde{\cF}_j)\cup E(\cF_j\setminus \widetilde{\cF}_j)$ is a set of backward edges which have one endpoint in $U^\gamma(T)\subseteq U^*$ and one endpoint in $V(T)\setminus U^*$.
	Thus, the following holds for each $i\in [4]$.
	\begin{align*}
		e_{\cF_j'}(U_i, U_{i-1})
		&\stackrel{\text{\eqmakebox[feasiblebackward]{}}}{=}e_{\widetilde{\cF}_j}(U_i, U_{i-1})+\sum_{v\in U_i^\gamma(T)}\left(\overleftarrow{d}_{\cF_j',\cU}^+(v)-\overleftarrow{d}_{\widetilde{\cF}_j,\cU}^+(v)\right)\\
		&\qquad\qquad+\sum_{v\in U_{i-1}^\gamma(T)}\left(\overleftarrow{d}_{\cF_j',\cU}^-(v)-\overleftarrow{d}_{\widetilde{\cF}_j,\cU}^-(v)\right)\\
		&\stackrel{\text{\eqmakebox[feasiblebackward]{\text{\cref{eq:redistribute-placeholder}}}}}{=}e_{\widetilde{\cF}_j}(U_i, U_{i-1})+\sum_{v\in U_i^\gamma(T)}\left(\overleftarrow{d}_{\cF_j,\cU}^+(v)-\overleftarrow{d}_{\widetilde{\cF}_j,\cU}^+(v)\right)\\
		&\qquad\qquad+\sum_{v\in U_{i-1}^\gamma(T)}\left(\overleftarrow{d}_{\cF_j,\cU}^-(v)-\overleftarrow{d}_{\widetilde{\cF}_j,\cU}^-(v)\right)\\
		&\stackrel{\text{\eqmakebox[feasiblebackward]{}}}{=}e_{\cF_j}(U_i, U_{i-1}).
	\end{align*}
	Thus, \cref{def:feasible-backward} follows from the fact that $\cF_j$ is a $(\gamma, T)$-pseudo-feasible system. Therefore, $\cF_j'$ is a $(\gamma, T)$-pseudo-feasible system, as desired.
\end{proof}

\begin{lm}[Covering $U^*$]\label{lm:coverU*}
	Let $0<\frac{1}{n}\ll \varepsilon\ll \eta \ll \gamma\ll 1$ and $(1-\eta)n\leq r\leq n$. Let $T$ be a regular bipartite tournament on $4n$ vertices. Let $\cU=(U_1, \dots, U_4)$ be an $(\varepsilon, 4)$-partition for $T$ and $U^*$ be an $(\varepsilon, \cU)$-exceptional set for $T$.
	Suppose that $D\subseteq T$ satisfies $\delta^0(D)\geq r$.
	Let $\cF_1, \dots, \cF_r$ be edge-disjoint $(\gamma, T)$-pseudo-feasible systems which satisfy the following.
	\begin{enumerate}
		\item $E(\overleftarrow{D}_\cU)\cup E(D[U^*])\subseteq \bigcup_{i\in [r]}E(\cF_i)\subseteq E(D)$.\label{lm:coverU*-E'}
		\item For each $i\in [r]$, $e(\cF_i)\leq \varepsilon n$.\label{lm:coverU*-e}
		\item Suppose that $e\in \bigcup_{i\in [r]}E(\cF_i)$ is a forward edge. Then, $V(e)\cap U^*\neq \emptyset$ and $e$ is not a $(\gamma,T)$-placeholder.\label{lm:coverU*-forward}
		\item For each $i\in [r]$, $\cF_i$ is a linear forest.\label{lm:coverU*-linforest}
	\end{enumerate}
	Then, there exist edge-disjoint feasible systems $\cF_1', \dots, \cF_r'$ such that the following hold.
	\begin{enumerate}[label=\rm(\alph*)]
		\item $\bigcup_{i\in [r]}E(\cF_i)\subseteq \bigcup_{i\in [r]}E(\cF_i')\subseteq E(D)$. In particular, $E(\overleftarrow{D}_\cU)\cup E(D[U^*])\subseteq \bigcup_{i\in [r]}E(\cF_i')$ $\subseteq E(D)$.\label{lm:coverU*-E}
		\item For each $i\in [r]$, $e(\cF_i')\leq 9\varepsilon n$.\label{lm:coverU*-size}
		\item For each $v\in V(T)\setminus U^*$, there exist at most $6\varepsilon n$ indices $i\in [r]$ such that $v\in V(\cF_i')$.\label{lm:coverU*-deg}
	\end{enumerate}
\end{lm}

\begin{proof}
	First, note that we may assume without loss of generality that $U^*\neq \emptyset$. Indeed, if $U^*=\emptyset$, then \cref{def:feasible-exceptional} holds automatically and so $\cF_1, \dots, \cF_r$ are already feasible systems.
	
	We extend the linear forests $\cF_1, \dots, \cF_r$ into larger linear forests which cover $U^*$ (and so satisfy \cref{def:feasible-exceptional}) using \cref{lm:extendlinforest}.
	Let $A\coloneqq U^*$ and $B\coloneqq V(T)\setminus U^*$. Let $D'$ be the bipartite digraph on vertex classes $A$ and $B$ induced by $\overrightarrow{D}_\cU-U^{1-\gamma}(T)$%
		\COMMENT{I.e.\ $E(D')$ consists of all the edges of $\overrightarrow{D}_\cU-U^{1-\gamma}(T)$ which have one endpoint in $A$ and one endpoint in $B$.}.
	Note for later that $E(D')$ is a set of $(\gamma, T)$-placeholders, so \cref{lm:coverU*-forward} implies that
	\begin{equation}\label{eq:coverU*-D'}
		E(D')\cap \bigcup_{i\in [r]}E(\cF_i)=\emptyset.
	\end{equation}
	For each $i\in [r]$, let $S_i^\pm$ be the set of vertices $v\in U^*$ which satisfy $d_{\cF_i}^\pm(v)=0$ (so $S_i^+$ and $S_i^-$ list the vertices in $U^*$ which are not yet covered with an out/inedge in $\cF_i$) and define $T_i\coloneqq V(E(\cF_i))\cap B$. 
	Note for later that both
	\begin{equation}\label{eq:coverU*-ST}
		S_i^\pm\cap V(\cF_i)\subseteq V^\mp(\cF_i) \quad \text{and} \quad V(E(\cF_i))\cap B\subseteq T_i
	\end{equation}
	for each $i\in [r]$.
	Define $N\coloneqq 2|U^*|$.
	For each $v\in B$, let $n_v$ denote the number of indices $i\in [\ell]$ such that $v\in T_i$.
	
	We verify that \cref{eq:extendlinforest} holds with $D'$ and $r$ playing the roles of $D$ and $\ell$.
	For each $i\in [r]$, we have
	\begin{equation}\label{eq:coverU*-S}
		|S_i^+|+|S_i^-|\leq 2|U^*|\stackrel{\text{\cref{def:ES-size}}}{\leq} 8\varepsilon n
	\end{equation}
	and
	\begin{equation}\label{eq:coverU*-T}
		|T_i|\leq |V(E(\cF_i))|\stackrel{\text{\cref{lm:coverU*-e}}}{\leq}2\varepsilon n.
	\end{equation}
	Moreover, each $v\in B=V(T)\setminus U^*$ satisfies
	\begin{align}\label{eq:coverU*-nv}
		n_v\leq \sum_{i\in [r]}d_{\cF_i}(v)\stackrel{\text{\cref{lm:coverU*-E'},\cref{lm:coverU*-forward}}}{\leq} \overleftarrow{d}_{D,\cU}(v)+|\overrightarrow{N}_{D,\cU}(v)\cap U^*|\stackrel{\text{\cref{def:ES-backward}}}{\leq}2\varepsilon n+|U^*|\stackrel{\text{\cref{def:ES-size}}}{\leq} 6\varepsilon n.
	\end{align}
	Therefore, each $v\in U^*\setminus U^{1-\gamma}(T)$ satisfies%
	    \COMMENT{First equality holds since $D'$ is bipartite on vertex classes $A$ and $B$ and since $U^{1-\gamma}(T)\subseteq U^*$.}
	\begin{align}
		d_{D'}^\pm(v)&\stackrel{\text{\eqmakebox[coverU*deg]{}}}{=}|\overrightarrow{N}_{D, \cU}^\pm(v)\setminus U^*|
		\geq \gamma n-|U^*|\nonumber\\
		&\stackrel{\text{\eqmakebox[coverU*deg]{\text{\cref{def:ES-size},\cref{eq:coverU*-S,eq:coverU*-nv,eq:coverU*-T}}}}}{\geq} 2\max_{i\in [r]}(|S_i^+|+|S_i^-|+|T_i|)+2(\max_{w\in B}n_w+N)\label{eq:coverU*-delta1}
	\end{align}
	and
	\begin{align}
		d_{D'}^\pm(v)&\stackrel{\text{\eqmakebox[coverU*deg2]{}}}{=}|\overrightarrow{N}_{D, \cU}^\pm(v)\setminus U^*|
		\stackrel{\delta^0(D)\geq r}{\geq} r-\overleftarrow{d}_{D, \cU}^\pm(v)-|\overrightarrow{N}_{D, \cU}^\pm(v)\cap U^*|
		\stackrel{\text{\cref{lm:coverU*-E'}}}{\geq}r-\sum_{i\in [r]}d_{\cF_i}^\pm(v)\nonumber\\
		&\stackrel{\text{\eqmakebox[coverU*deg2]{\text{\cref{lm:coverU*-linforest}}}}}{=}|\{i\in [r]\mid v\in S_i^\pm\}|.\label{eq:coverU*-delta2}
	\end{align}
	By \cref{def:feasible-degV*}, the vertices in $U^{1-\gamma}(T)$ are already covered with both an in- and outedge in each of the pseudo-feasible systems $\cF_1, \dots, \cF_r$, so $S_i^+\cup S_i^-\subseteq U^*\setminus U^{1-\gamma}(T)$ for each $i\in [r]$.
	Thus, \cref{eq:extendlinforest} follows from \cref{eq:coverU*-delta1,eq:coverU*-delta2}.
	
	Let $\cQ_1, \dots, \cQ_r$ be the edge-disjoint linear forests obtained by applying \cref{lm:extendlinforest} with $D'$ and $r$ playing the roles of $D$ and $\ell$.
	For each $i\in [r]$, denote $\cF_i'\coloneqq \cF_i\cup \cQ_i$. We claim that $\cF_1', \dots, \cF_r'$ are edge-disjoint feasible systems which satisfy \cref{lm:coverU*-E,lm:coverU*-size,lm:coverU*-deg}.
	By assumption, $\cF_1, \dots, \cF_r$ are edge-disjoint. Thus, \cref{eq:coverU*-D'}, \cref{eq:coverU*-ST}, and the ``in particular part" of \cref{lm:extendlinforest} imply that $\cF_1', \dots, \cF_r'$ are edge-disjoint linear forests. In particular, \cref{def:feasible-linforest} is satisfied.
	By construction and \cref{lm:extendlinforest}\cref{lm:extendlinforest-cover}, each $i\in [r]$ satisfies
	\begin{equation}\label{eq:coverU*-forward}
		E(\cF_i)\subseteq E(\cF_i')\subseteq E(\cF_i)\cup \{e\in E(\overrightarrow{D}_\cU)\mid V(e)\cap U^*\neq \emptyset\}.
	\end{equation}
	In particular, \cref{lm:coverU*-E} follows from \cref{lm:coverU*-E'}, while \cref{lm:coverU*-size} follows from \cref{lm:coverU*-e}, \cref{eq:coverU*-S}, and \cref{lm:extendlinforest}\cref{lm:extendlinforest-cover}.
	For each $v\in V(T)\setminus U^*$, we have
	\begin{align*}
		\sum_{i\in [r]}d_{\cF_i'}(v)\stackrel{\text{\cref{lm:coverU*-E'},\cref{lm:coverU*-forward},\cref{eq:coverU*-forward}}}{\leq} \overleftarrow{d}_{D,\cU}(v)+|\overrightarrow{N}_{D, \cU}(v)\cap U^*|\stackrel{\text{\cref{def:ES-backward}}}{\leq}2\varepsilon n+|U^*|\stackrel{\text{\cref{def:ES-size}}}{\leq} 6\varepsilon n.
	\end{align*}
	Thus, \cref{lm:coverU*-deg} holds.
	Let $i\in [r]$.
	By \cref{eq:coverU*-forward}, $\cF_i'$ is obtained from $\cF_i$ by adding forward edges, so \cref{def:feasible-backward} follows from the fact that $\cF_i$ is a $(\gamma, T)$-pseudo-feasible system. By definition of $S_i^+$ and $S_i^-$, \cref{lm:extendlinforest}\cref{lm:extendlinforest-cover} implies that \cref{def:feasible-exceptional} is satisfied. Therefore, $\cF_i'$ is a feasible system, as desired.
\end{proof}

\begin{lm}[Incorporating $E$]\label{lm:incorporateE}
	Let $0<\frac{1}{n}\ll \varepsilon\ll \eta \ll \gamma\ll 1$ and $(1-\eta)n\leq r\leq n$. Let $T$ be a regular bipartite tournament on $4n$ vertices. Let $\cU=(U_1, \dots, U_4)$ be an $(\varepsilon, 4)$-partition for $T$ and $U^*$ be an $(\varepsilon, \cU)$-exceptional set for $T$.
	Suppose that $D\subseteq T$ satisfies $\delta^0(D)\geq r$.
	Let $\cF_1, \dots, \cF_r$ be edge-disjoint feasible systems which satisfy the following.
	\begin{enumerate}
		\item $E(\overleftarrow{D}_\cU)\cup E(D[U^*])\subseteq \bigcup_{i\in [r]}E(\cF_i)\subseteq E(D)$.\label{lm:incorporateE-E'}
		\item For each $i\in [r]$, $e(\cF_i)\leq \varepsilon n$.\label{lm:incorporateE-e}
		\item For each $v\in V(T)\setminus U^*$, there exist at most $\varepsilon n$ indices $i\in [r]$ such that $v\in V(\cF_i)$.\label{lm:incorporateE-deg}
	\end{enumerate}
	Let $E\subseteq E(D)$ satisfy the following properties.
	\begin{enumerate}[resume]
		\item $E\subseteq E(\overrightarrow{D}_\cU-U^*)$.\label{lm:incorporateE-forward}
		\item For each $v\in V(T)\setminus U^*$, $d_E^\pm(v)\leq 1$.\label{lm:incorporateE-Edeg}
	\end{enumerate}
	Then, there exist edge-disjoint feasible systems $\cF_1', \dots, \cF_r'$ such that the following hold.
	\begin{enumerate}[label=\rm(\alph*)]
		\item $\bigcup_{i\in [r]}E(\cF_i')= \bigcup_{i\in [r]}E(\cF_i)\cup E$. In particular, $E(\overleftarrow{D}_\cU)\cup E(D[U^*])\cup E\subseteq \bigcup_{i\in [r]}E(\cF_i')\subseteq E(D)$.\label{lm:incorporateE-E}
		\item For each $i\in [r]$, $e(\cF_i')\leq e(\cF_i)+5\leq 2\varepsilon n$.\label{lm:incorporateE-size}
	\end{enumerate}
\end{lm}

\begin{proof}
	To ensure that we do not create any cycle when adding the edges in $E$, we will separate the feasible systems $\cF_1, \dots, \cF_r$ into four groups. For each $i\in [4]$, the edges of $E$ from $U_i$ to $U_{i+1}$ will be distributed among the feasible systems from the $i^{\rm th}$ group. In this way, each $\cF_j'$ will be obtained from $\cF_j$ by adding a matching of forward edges. The edges of $E$ will be distributed using Hall's theorem (\cref{prop:Hall}).
	
	For each $i\in [4]$, let $A_i\coloneqq E_E(U_i, U_{i+1})\setminus \bigcup_{j\in [r]}E(\cF_j)$.
	Let $B_1\cup \dots \cup B_4$ be a partition of $[r]$ such that $|B_i|\geq \left\lfloor\frac{r}{4}\right\rfloor\eqqcolon r'$ for each $i\in [4]$. For each $i\in [4]$, let $B_i'$ be the multiset which consists of~$5$ copies of each $j\in B_i$ and let $G_i$ be the auxiliary bipartite graph on vertex classes $A_i$ and $B_i'$ defined as follows. For each $e\in A_i$ and each (copy of) $j\in B_i'$, $ej\in E(G)$ if and only if $V(e)\cap V(E(\cF_j))=\emptyset$.
	
	Let $i\in [4]$. By \cref{lm:incorporateE-Edeg} and \cref{fact:partition}\cref{fact:partition-size}, $|A_i|\leq n\leq 5r'\leq |B_i'|$.
	For each $e\in A_i$, we have 
	\[d_{G}(e)\stackrel{\text{\cref{lm:incorporateE-deg}}}{\geq} 5(r'-2\varepsilon n)\geq \frac{|B_i'|}{2}.\]
	Moreover, each (copy of) $j\in B_i'$ satisfies
	\[d_{G_i}(j)\stackrel{\text{\cref{lm:incorporateE-Edeg}}}{\geq} |A_i|-|V(E(\cF_j))\geq |A_i|-2e(\cF_j)\stackrel{\text{\cref{lm:incorporateE-e}}}{\geq} |A_i|-\frac{|B_i'|}{2}.\]			
	Apply \cref{prop:Hall} to obtain a matching $M_i$ of $G_i$ which covers $A_i$.
	
	Denote $A\coloneqq \bigcup_{i\in [4]}A_i$ and $M\coloneqq \bigcup_{i\in [4]}M_i$.
	For each $j\in [r]$, let $\cF_j'$ be obtained from $\cF_j$ by adding all the edges $e\in A$ such that $e$ is adjacent to a copy of $j$ in $M$.
	We now verify that $\cF_1', \dots, \cF_r'$ are edge-disjoint feasible systems for which \cref{lm:incorporateE-E,lm:incorporateE-size} are satisfied. 
	By construction, $M$ is a matching covering $A$ and \cref{lm:incorporateE-forward} implies that $A=E\setminus \bigcup_{j\in [r]}E(\cF_j)$.
	Therefore, $\cF_1', \dots, \cF_r'$ are edge-disjoint and \cref{lm:incorporateE-E} holds. Moreover, \cref{lm:incorporateE-size} holds by \cref{lm:incorporateE-e} and definition of $B_1', \dots, B_4'$.
	
	Let $j\in [r]$. Recall that $\cF_j$ is a feasible system. In particular, \cref{def:feasible-linforest} implies that $\cF_j$ is a linear forest.
	By definition of $G_1, \dots, G_4$, we have $V(E(\cF_j'\setminus \cF_j))\cap V(E(\cF_j))=\emptyset$. Moreover, \cref{lm:incorporateE-Edeg} implies that $A_1, \dots, A_4$ are all matchings, so, by construction, $E(\cF_j'\setminus \cF_j)$ is a matching. Thus, $\cF_j'$ is also a linear forest and so \cref{fact:feasibleforward} and \cref{lm:incorporateE-forward} imply that $\cF_j'$ is also a feasible system.
\end{proof}

In the following \lcnamecref{lm:endpoints}, we add forward edges to ensure that all the components of each feasible system have their ending point in $U_4$.

\begin{lm}[Extending the ending points of feasible systems into $U_4$]\label{lm:endpoints}
	Let $0<\frac{1}{n}\ll \varepsilon\ll \eta \ll \gamma\ll 1$ and $(1-\eta)n\leq r\leq n$. Let $T$ be a regular bipartite tournament on $4n$ vertices. Let $\cU=(U_1, \dots, U_4)$ be an $(\varepsilon, 4)$-partition for $T$ and $U^*$ be an $(\varepsilon, \cU)$-exceptional set for $T$.
	Suppose that $D\subseteq T$ satisfies $\delta^0(D)\geq r$.
	Let $\cF_1, \dots, \cF_r$ be edge-disjoint feasible systems which satisfy the following.
	\begin{enumerate}
		\item $E(\overleftarrow{D}_\cU)\cup E(D[U^*])\subseteq \bigcup_{i\in [r]}E(\cF_i)\subseteq E(D)$.\label{lm:endpoints-E'}
		\item For each $i\in [r]$, $e(\cF_i)\leq \varepsilon n$.\label{lm:endpoints-e}
		\item For each $v\in V(T)\setminus U^*$, there exist at most $\varepsilon n$ indices $i\in [r]$ such that $v\in V(\cF_i)$.\label{lm:endpoints-deg}
	\end{enumerate}
	Then, there exist edge-disjoint feasible systems $\cF_1', \dots, \cF_r'$ such that the following hold.
	\begin{enumerate}[label=\rm(\alph*)]
		\item $\bigcup_{i\in [r]}E(\cF_i)\subseteq \bigcup_{i\in [r]}E(\cF_i')\subseteq E(D)$. In particular, $E(\overleftarrow{D}_\cU)\cup E(D[U^*])\subseteq \bigcup_{i\in [r]}E(\cF_i')$ $\subseteq E(D)$.\label{lm:endpoints-E}
		\item For each $i\in [r]$, $e(\cF_i')\leq 4e(\cF_i)\leq 4\varepsilon n$.\label{lm:endpoints-size}
		\item For each $v\in V(T)\setminus U^*$, there exist at most $4\sqrt{\varepsilon}n$ indices $i\in [r]$ such that $v\in V(\cF_i')$.\label{lm:endpoints-deg'}
		\item For each $i\in [r]$, $V^+(\cF_i')=V^+(\cF_i)$ and $V^-(\cF_i')\subseteq U_4$. \label{lm:endpoints-endpoints}
	\end{enumerate}
\end{lm}

\begin{proof}
	We extend the components of the feasible systems in three stages as follows. At each stage $i\in [3]$, we use edges of $D(U_i,U_{i+1})$ to extend the components of the feasible systems which currently end in $U_i$ into components which end in $U_{i+1}$.
	This is achieved via \cref{lm:extendlinforest}. 
	
	By \cref{fact:feasibleisolated}, we may assume without loss of generality that $\cF_1, \dots, \cF_r$ do not contain any isolated vertices. For each $i\in [r]$, let $\cF_i^0\coloneqq \cF_i$. 
	Suppose inductively that, for some $0\leq i\leq 3$, we have constructed edge-disjoint feasible systems $\cF_1^i, \dots, \cF_r^i$ such that the following hold.
	\begin{enumerate}[label=\rm(\alph*$'$)]
		\item For each $j\in [r]$, $E(\cF_j)\subseteq E(\cF_j^i)\subseteq E(D)$.\label{lm:endpoints-IH-E}
		\item For each $j\in [r]$, $e(\cF_j^i)\leq (i+1)e(\cF_j)\leq (i+1)\varepsilon n$.\label{lm:endpoints-IH-size}
		\item For each $v\in V(T)\setminus U^*$, there exist at most $(i+1)\sqrt{\varepsilon}n$ indices $j\in [r]$ such that $v\in V(\cF_j^i)$.\label{lm:endpoints-IH-deg'}
		\item For each $j\in [r]$, $V^+(\cF_j^i)=V^+(\cF_j)$.\label{lm:endpoints-IH-starting}
		\item For each $j\in [r]$, $V^-(\cF_j^i)\subseteq V(T)\setminus \bigcup_{i'\in [i]}U_{i'}$.\label{lm:endpoints-IH-ending}
		\item For each $j\in [r]$, $\cF_j^i$ does not contain any isolated vertex.\label{lm:endpoints-IH-isolated}
	\end{enumerate}
	First, assume that $i=3$. For each $j\in [r]$, let $\cF_j'\coloneqq \cF_j^i$. Then, \cref{lm:endpoints-E,lm:endpoints-deg',lm:endpoints-endpoints,lm:endpoints-size} follow from \cref{lm:endpoints-IH-E,lm:endpoints-IH-deg',lm:endpoints-IH-ending,lm:endpoints-IH-size,lm:endpoints-IH-starting}.
	
	We may therefore assume that $i<3$. We construct $\cF_1^{i+1}, \dots, \cF_r^{i+1}$ using \cref{lm:extendlinforest} as follows.
	Let $A\coloneqq U_{i+1}\setminus U^*$ and $B\coloneqq U_{i+2}\setminus U^*$. Let $D'$ be the bipartite digraph on vertex classes $A$ and $B$ which is induced by $(\overrightarrow{D}_\cU\setminus \bigcup_{j\in [r]}\cF_j^i)-U^*$%
		\COMMENT{I.e.\ $E(D')$ consists of all the edges of $(\overrightarrow{D}\setminus \bigcup_{j\in [r]}\cF_j^i)-U^*$ with an endpoint in $A$ and an endpoint in $B$.}.
	For each $j\in [r]$, let $S_j^-\coloneqq \emptyset$, let $S_j^+\coloneqq V^-(\cF_j^i)\cap A$ (so $S_j^+$ lists the ending points of the components that currently end in $U_{i+1}\setminus U^*$ and which we want to extend in this step), and define $T_j\coloneqq V(E(\cF_j^i))\cap B$.
	Note for later that both
	\begin{equation}\label{eq:endpoints-ST}
		S_j^\pm\cap V(\cF_j^i)\subseteq V^\mp(\cF_j^i) \quad \text{and} \quad V(E(\cF_j^i))\cap B\subseteq T_j
	\end{equation}
	for each $j\in [r]$.
	Define $N\coloneqq \sqrt{\varepsilon}n$. For each $v\in A\cup B$, let $n_v$ denote the number of indices $j\in [r]$ such that $v\in V(\cF_j^i)$ and observe that%
		\COMMENT{For the lower bound: if $v\in B$ then  $|\{j\in [r]\mid v\in S_j^+\}|=0$; otherwise, $|\{j\in [r]\mid v\in S_j^+\}|$ is the number of indices $j\in [r]$ such that $v\in V^-(\cF_j^i)$.}
	\begin{equation}\label{eq:endpoints-nv}
		|\{j\in [r]\mid v\in S_j^+\}|\leq n_v\stackrel{\text{\cref{lm:endpoints-IH-deg'}}}{\leq}4\sqrt{\varepsilon} n.
	\end{equation}
	
	We verify that \cref{eq:extendlinforest} holds with $D'$ and $r$ playing the roles of $D$ and $\ell$. Recall that $\cF_1, \dots, \cF_r$ do not contain any isolated vertices.
	Thus, each $j\in [r]$ satisfies
	\begin{align}
		|S_j^+|+|S_j^-|&\stackrel{\text{\eqmakebox[endpoints]{}}}{\leq} |V^-(\cF_j^i)|=|V^+(\cF_j^i)|\stackrel{\text{\cref{lm:endpoints-IH-starting}}}{=}|V^+(\cF_j)|\leq e(\cF_j)\label{eq:endpoints-Sjexact}\\ &\stackrel{\text{\eqmakebox[endpoints]{\text{\cref{lm:endpoints-e}}}}}{\leq} \varepsilon n\label{eq:endpoints-Sj}
	\end{align}
	and
	\begin{equation}\label{eq:edpoints-Tj}
	    |T_j|\leq |V(E(\cF_j^i))|\stackrel{\text{\cref{lm:endpoints-IH-size}}}{\leq} 8\varepsilon n.
	\end{equation}
	Then,
	\begin{equation}\label{eq:endpoints-S}
		\frac{1}{\lfloor N\rfloor}\sum_{j\in [r]}(|S_j^+|+|S_j^-|)\stackrel{\text{\cref{eq:endpoints-Sj}}}{\leq} \frac{\varepsilon n r}{\lfloor\sqrt{\varepsilon}n\rfloor}\leq 2\sqrt{\varepsilon}n.
	\end{equation}
	For each $v\in U_{i+1}\setminus U^*=A$, we have
	\begin{align*}
		d_{D'}^+(v)&\stackrel{\text{\eqmakebox[endpointsdeg]{}}}{\geq} |N_D^+(v)\cap U_{i+2}|- |N_D^+(v)\cap U^*|-\sum_{j\in [r]}d_{\cF_j^i}^+(v)\\
		&\stackrel{\text{\eqmakebox[endpointsdeg]{\text{\cref{def:feasible-linforest}}}}}{\geq} \overrightarrow{d}_{D,\cU}^+(v)-|U^*|-n_v\\
		&\stackrel{\text{\eqmakebox[endpointsdeg]{}}}{\geq} \left(\overrightarrow{d}_{T,\cU}^+(v)-(n-r)\right)-|U^*|-n_v\\
		&\stackrel{\text{\eqmakebox[endpointsdeg]{\text{\cref{def:ES},\cref{eq:endpoints-nv}}}}}{\geq} (1-2\eta)n-4\varepsilon n-4\sqrt{\varepsilon}n\\
		&\stackrel{\text{\eqmakebox[endpointsdeg]{\text{\cref{eq:endpoints-nv},\cref{eq:endpoints-Sj,eq:endpoints-S,eq:edpoints-Tj}}}}}{\geq} \frac{1}{\lfloor N\rfloor}\sum_{j\in [r]}(|S_j^+|+|S_j^-|)+|\{j\in [r]\mid v\in S_j^+\}|\\
		&\eqmakebox[endpointsdeg]{}\qquad\qquad +2\max_{j\in [r]}(|S_j^+|+|S_j^-|+|T_j|)+2(\max_{w\in B}n_w+N).
	\end{align*}
	Therefore, \cref{eq:extendlinforest} is satisfied.
	
	Let $\cQ_1, \dots, \cQ_r$ be the edge-disjoint linear forests obtained by applying \cref{lm:extendlinforest} with $D'$ and $r$ playing the roles of $D$ and $\ell$.
	For each $j\in [r]$, denote $\cF_j^{i+1}\coloneqq \cF_j^i\cup \cQ_j$.
	We claim that $\cF_1^{i+1}, \dots, \cF_r^{i+1}$ are edge-disjoint feasible systems such that \cref{lm:endpoints-IH-E,lm:endpoints-IH-deg',lm:endpoints-IH-ending,lm:endpoints-IH-size,lm:endpoints-IH-starting,lm:endpoints-IH-isolated} are satisfied with $i+1$ playing the role of $i$.
	By assumption and definition of $D'$, $\cF_1^i, \dots, \cF_r^i$ are edge-disjoint from each other and from $D'$. Thus, \cref{eq:endpoints-ST} and the ``in particular part" of \cref{lm:extendlinforest} imply that $\cF_1^{i+1}, \dots, \cF_r^{i+1}$ are edge-disjoint linear forests.
	Moreover, \cref{lm:endpoints-IH-isolated} follows from \cref{lm:extendlinforest}\cref{lm:extendlinforest-cover} and the induction hypothesis, while \cref{lm:endpoints-IH-deg'} holds by \cref{lm:extendlinforest}\cref{lm:extendlinforest-N} and the induction hypothesis.
	Furthermore, \cref{lm:endpoints-IH-size} follows from \cref{lm:endpoints-e}, \cref{eq:endpoints-Sjexact}, \cref{lm:extendlinforest}\cref{lm:extendlinforest-cover}, and the induction hypothesis.
	By \cref{lm:extendlinforest}\cref{lm:extendlinforest-cover}, each $j\in [r]$ satisfies
	\[E(\cF_j^i)\subseteq E(\cF_j^{i+1})\subseteq E(\cF_j^i)\cup E(\overrightarrow{D}_\cU-U^*).\]
	Therefore, \cref{lm:endpoints-IH-E} follows from the induction hypothesis and, by \cref{fact:feasibleforward}, $\cF_j^{i+1}$ is still a feasible system for each $j\in [r]$.
	For each $j\in [r]$, the definition of $S_j^+$ and $S_j^-$ and \cref{lm:extendlinforest}\cref{lm:extendlinforest-cover} imply that all the edges of $\cF_j^{i+1}\setminus \cF_j^i=\cQ_j$ start at a vertex in $V^-(\cF_j^i)$. Thus, \cref{lm:endpoints-IH-starting} holds. 
	By \cref{lm:extendlinforest}\cref{lm:extendlinforest-cover}, each $j\in[r]$ satisfies
	\begin{align*}
		V^-(\cF_j^{i+1})&\stackrel{\text{\eqmakebox[endpointsV-]{}}}{\subseteq} (V^-(\cF_j^i)\setminus A)\cup B
		\subseteq (V^-(\cF_j^i)\setminus U_{i+1})\cup (V^-(\cF_j^i)\cap U^*)\cup U_{i+2}\\
		&\stackrel{\text{\eqmakebox[endpointsV-]{\text{\cref{def:feasible-exceptional}}}}}{=}(V^-(\cF_j^i)\setminus U_{i+1})\cup U_{i+2}
		\stackrel{\text{\cref{lm:endpoints-IH-ending}}}{=}\left(V(T)\setminus \left(\bigcup_{i'\in [i]}U_{i'}\cup U_{i+1}\right)\right)\cup U_{i+2}\\
		&\stackrel{\text{\eqmakebox[endpointsV-]{}}}{=}V(T)\setminus \bigcup_{i'\in [i+1]}U_{i'}.
	\end{align*} 
	Therefore, \cref{lm:endpoints-IH-ending} holds.
\end{proof}

By symmetry, we can proceed analogously to ensure that the starting point of each component of the feasible systems also lies in the correct vertex class.

\begin{lm}[Extending the starting points of feasible systems into $U_1$]\label{cor:endpoints}
	Under the conditions of \cref{lm:endpoints}, there exist edge-disjoint feasible systems $\cF_1', \dots, \cF_r'$ such that the following hold.
	\begin{enumerate}[label=\rm(\alph*)]
		\item $\bigcup_{i\in [r]}E(\cF_i)\subseteq \bigcup_{i\in [r]}E(\cF_i')\subseteq E(D)$. In particular, $E(\overleftarrow{D}_\cU)\cup E(D[U^*])\subseteq \bigcup_{i\in [r]}E(\cF_i')$ $\subseteq E(D)$.\label{cor:endpoints-E}
		\item For each $i\in [r]$, $e(\cF_i')\leq 7e(\cF_i)\leq 7\varepsilon n$.\label{cor:endpoints-size}
		\item For each $v\in V(T)\setminus U^*$, there exist at most $7\sqrt{\varepsilon}n$ indices $i\in [r]$ such that $v\in V(\cF_i')$.\label{cor:endpoints-deg'}
		\item For each $i\in [r]$, $V^+(\cF_i')\subseteq U_1$ and $V^-(\cF_i')\subseteq U_4$.\label{cor:endpoints-endpoints}
	\end{enumerate}
\end{lm}

We are now ready to derive \cref{lm:feasible}.

\begin{proof}[Proof of \cref{lm:feasible}]
	Let $i\in [r]$. Suppose that $e\in E(\cF_i)$ is a forward edge such that $V(e)\cap U^*=\emptyset$ or $e$ is a $(\gamma, T)$-placeholder. Then, \cref{lm:feasible-F-E,lm:feasible-F-e,lm:feasible-E-deg,lm:feasible-E-forward} are still satisfied if we replace $\cF_i$ by $\cF_i\setminus \{e\}$. Moreover, \cref{fact:pseudofeasibleforward} implies that $\cF_i\setminus \{e\}$ is still a $(\gamma, T)$-pseudo-feasible system. Thus, by deleting edges if necessary, we may assume that \cref{lm:redistribute}\cref{lm:redistribute-forward} is satisfied. Moreover, \cref{lm:redistribute}\cref{lm:redistribute-E,lm:redistribute-e} follow from \cref{lm:feasible-F-E,lm:feasible-F-e}.
	Apply \cref{lm:redistribute} to obtain edge-disjoint $(\gamma, T)$-pseudo-feasible systems $\cF_1^1, \dots, \cF_r^1$ satisfying \cref{lm:redistribute}\cref{lm:redistribute-redistribute,lm:redistribute-linforest,lm:redistribute-size}.
	
	By \cref{lm:redistribute}\cref{lm:redistribute-redistribute,lm:redistribute-size,lm:redistribute-linforest}, \cref{lm:coverU*}\cref{lm:coverU*-E',lm:coverU*-e,lm:coverU*-forward,lm:coverU*-linforest} are satisfied with $\cF_1^1, \dots, \cF_r^1$ playing the roles of $\cF_1, \dots, \cF_r$%
		\COMMENT{\cref{lm:coverU*}\cref{lm:coverU*-E'} follows from \cref{lm:redistribute}\cref{lm:redistribute-redistribute}.\\
		\cref{lm:coverU*}\cref{lm:coverU*-e} follows from \cref{lm:redistribute}\cref{lm:redistribute-size}.\\
		\cref{lm:coverU*}\cref{lm:coverU*-forward} follows from \cref{lm:redistribute}\cref{lm:redistribute-redistribute}.\\
		\cref{lm:coverU*}\cref{lm:coverU*-linforest} follows from \cref{lm:redistribute}\cref{lm:redistribute-linforest}.}. 
	Apply \cref{lm:coverU*} with $\cF_1^1, \dots, \cF_r^1$ playing the roles of $\cF_1, \dots, \cF_r$ to obtain edge-disjoint feasible systems $\cF_1^2, \dots, \cF_r^2$ satisfying \cref{lm:coverU*}\cref{lm:coverU*-E,lm:coverU*-deg,lm:coverU*-size}.

	By \cref{lm:coverU*}\cref{lm:coverU*-E,lm:coverU*-deg,lm:coverU*-size}, \cref{lm:incorporateE}\cref{lm:incorporateE-deg,lm:incorporateE-E',lm:incorporateE-e} are satisfied with $\cF_1^2, \dots, \cF_r^2$, and $9\varepsilon$ playing the roles of $\cF_1, \dots, \cF_r$, and $\varepsilon$%
		\COMMENT{\cref{lm:incorporateE}\cref{lm:incorporateE-E'} follows from \cref{lm:coverU*}\cref{lm:coverU*-E}.\\
		\cref{lm:incorporateE}\cref{lm:incorporateE-e} follows from \cref{lm:coverU*}\cref{lm:coverU*-size}.\\
		\cref{lm:incorporateE}\cref{lm:incorporateE-deg} follows from \cref{lm:coverU*}\cref{lm:coverU*-deg}.}.
	By \cref{lm:feasible-E-forward,lm:feasible-E-deg}, \cref{lm:incorporateE}\cref{lm:incorporateE-Edeg,lm:incorporateE-forward} are satisfied.
	Apply \cref{lm:incorporateE} with $\cF_1^2, \dots, \cF_r^2$, and $9\varepsilon$ playing the roles of $\cF_1, \dots, \cF_r$, and $\varepsilon$ to obtain edge-disjoint feasible systems $\cF_1^3, \dots, \cF_r^3$ for which \cref{lm:incorporateE}\cref{lm:incorporateE-E,lm:incorporateE-size} are satisfied (with $9\varepsilon$ playing the role of $\varepsilon$).
	
	By \cref{lm:feasible-E-deg}, \cref{lm:coverU*}\cref{lm:coverU*-deg}, and \cref{lm:incorporateE}\cref{lm:incorporateE-E} and \cref{lm:incorporateE-size}, \cref{lm:endpoints}\cref{lm:endpoints-E',lm:endpoints-e,lm:endpoints-deg} are satisfied with $\cF_1^3, \dots, \cF_r^3$, and $18\varepsilon$ playing the roles of $\cF_1, \dots, \cF_r$, and $\varepsilon$%
		\COMMENT{\cref{lm:endpoints}\cref{lm:endpoints-E'} follows from \cref{lm:incorporateE}\cref{lm:incorporateE-E}.\\
		\cref{lm:endpoints}\cref{lm:endpoints-e} follows from \cref{lm:incorporateE}\cref{lm:incorporateE-size}.\\
		\cref{lm:endpoints}\cref{lm:endpoints-deg} follows from \cref{lm:feasible-E-deg}, \cref{lm:coverU*}\cref{lm:coverU*-deg}, and \cref{lm:incorporateE}\cref{lm:incorporateE-E}.\\}.
	Apply \cref{cor:endpoints} with $\cF_1^3, \dots, \cF_r^3$, and $18\varepsilon$ playing the roles of $\cF_1, \dots, \cF_r$, and $\varepsilon$ to obtain edge-disjoint feasible systems $\cF_1', \dots, \cF_r'$ satisfying \cref{cor:endpoints}\cref{cor:endpoints-E,cor:endpoints-deg',cor:endpoints-endpoints,cor:endpoints-size} (with $18\varepsilon$ playing the role of $\varepsilon$).	
	Then, \cref{lm:feasible-E} follows from \cref{lm:incorporateE}\cref{lm:incorporateE-E} and \cref{cor:endpoints}\cref{cor:endpoints-E}.
	Moreover, \cref{lm:feasible-e,lm:feasible-deg,lm:feasible-endpoints} follow from \cref{cor:endpoints}\cref{cor:endpoints-size,cor:endpoints-deg',cor:endpoints-endpoints}, respectively.
\end{proof}

\section{Constructing pseudo-feasible systems: proof of Lemma~\ref{lm:backwardedges}}\label{sec:constructpseudofeasible}

	\onlyinsubfile{
		\setcounter{section}{11}
		\section{Proof of Lemma \ref{lm:backwardedges}}}
	
In this section, we prove \cref{lm:backwardedges}, which states that the backward and exceptional edges of a regular bipartite tournament can be decomposed into pseudo-feasible systems.
	
\subsection{Proof overview}\label{sec:sketchbackward}

Let $T$ be a bipartite tournament on $4n$ vertices. Let $\cU=(U_1, \dots, U_4)$ be an $(\varepsilon,4)$-partition for $T$ and $U^*$ be an $(\varepsilon, \cU)$-exceptional set for $T$. Suppose that we want to decompose the backward edges of $T$ into $n$ pseudo-feasible systems. The main difficulty is to construct linear forests (up to placeholders) which have a balanced number of backward edges (recall \cref{def:feasible-backward}).

\subsubsection{Simplified argument}

For simplicity, first assume that $\Delta^0(\overleftarrow{T}_\cU)\leq \frac{n}{2}$. The idea is to decompose each pair $E_T(U_i, U_{i-1})$ into $\frac{n}{2}$ matchings (which is possible by K\"{o}nig's theorem) and then construct pseudo-feasible systems from these matchings as follows. First, we pair each of the $\frac{n}{2}$ matchings from $E_T(U_1, U_4)$ with a distinct matching from $E_T(U_3, U_2)$. Overall, we obtain a decomposition of $E_D(U_1, U_4)\cup E_D(U_3, U_2)$ into $\frac{n}{2}$ matchings $\cF_1, \dots, \cF_{\frac{n}{2}}$.
Similarly, we pair each of the $\frac{n}{2}$ matchings from $E_T(U_4, U_3)$ with a distinct matching from $E_T(U_2, U_1)$ to obtain a decomposition of $E_T(U_4, U_3)\cup E_T(U_2, U_1)$ into $\frac{n}{2}$ matchings $\cF_{\frac{n}{2}+1}, \dots, \cF_n$. 
In particular, $\cF_1, \dots, \cF_n$ are linear forests and so satisfy \cref{def:feasible-cycle,def:feasible-degV'}. By assumption, $T$ does not contain any vertex of very high backward degree and so \cref{def:feasible-degV*} also holds.
Thus, for $\cF_1, \dots, \cF_n$ to be pseudo-feasible systems, we only need them to contain a balanced number of backward edges (see \cref{def:feasible-backward}). More precisely, each of $\cF_1, \dots, \cF_{\frac{n}{2}}$ must contain the same number of edges from $E_T(U_1, U_4)$ and $E_T(U_3, U_2)$ and each of $\cF_{\frac{n}{2}+1}, \dots, \cF_n$ must contain the same number of edges from $E_T(U_4, U_3)$ and $E_T(U_2, U_1)$.
Since $T$ contains the same number of backward edges in each pair of the blow-up $C_4$ (recall \cref{fact:backwardedges}), this can be easily achieved by using \cref{prop:Konigsamesize} to initially decompose each $E_T(U_i, U_{i-1})$ into $\frac{n}{2}$ matchings of (almost) the same size.

\subsubsection{General argument}\label{sec:sketchbackward-general}

In general, $\Delta^0(\overleftarrow{T}_\cU)$ may be larger than $\frac{n}{2}$ and so the above strategy does not work (we cannot decompose each pair into $\frac{n}{2}$ matchings of backward edges). However, we adapt the above argument by constructing $\frac{n}{2}$ pseudo-feasible systems which mostly consist of edges from $E_T(U_1, U_4)\cup E_T(U_3, U_2)$ and $\frac{n}{2}$ pseudo-feasible systems which mostly consist of edges from $E_T(U_4, U_3)\cup E_T(U_2, U_1)$. 

To do so, we consider an auxiliary digraph $D$ obtained from $\overleftarrow{T}_\cU$ as follows. For each $v\in U^{\frac{1}{2}}(T)$ (that is, for each $v\in V(T)$ satisfying $\overleftarrow{d}_{T, \cU}^+(v)=\overleftarrow{d}_{T, \cU}^-(v)>\frac{n}{2}$ (recall \cref{fact:backwarddegree,eq:Ugamma})), we replace $v$ by two copies $v$ and replace all the edges incident to $v$ by an edge incident to one of the copies of $v$. By splitting neighbourhoods evenly between the two copies of each vertex in $U^{\frac{1}{2}}(T)$, we can ensure that $\Delta^0(D)\leq \frac{n}{2}$.
Then, we can proceed as above to partition $E_D(U_1, U_4)\cup E_D(U_3, U_2)$ into $\frac{n}{2}$ pseudo-feasible systems $\cF_1, \dots, \cF_{\frac{n}{2}}$ and partition $E_D(U_4, U_3)\cup E_D(U_2, U_1)$ into $\frac{n}{2}$ pseudo-feasible systems $\cF_{\frac{n}{2}+1}, \dots, \cF_n$.

Denote by $\cF_1', \dots, \cF_n'$ the decomposition of $\overleftarrow{T}_\cU$ induced by $\cF_1, \dots, \cF_n$. Then, $\cF_1', \dots, \cF_n'$ each contain a balanced number of backward edges but may contain up to two edges (of the same direction) incident to the vertices in $U^{\frac{1}{2}}(T)$ (all the other vertices have degree at most one). We solve this problem as follows. For each $i\in [n]$, we move an edge of $\cF_i'$ at each vertex of degree two to $\cF_{\frac{n}{2}+i}'$ (where the index $\frac{n}{2}+i$ is taken modulo $n$). By construction, $\cF_i'$ and $\cF_{\frac{n}{2}+i}'$ are constructed from different pairs of the blow-up $C_4$. Thus, we do not create additional vertices of semidegree at least two and so we now have $\Delta^0(\cF_i')=1$ for each $i\in [n]$. In particular, \cref{def:feasible-degV'} is now satisfied. Of course, some cycles may be created in the process. However, each cycle will contain one of the moved edges and so will contain a backward edge incident to vertex of very high backward degree, that is, a placeholder. Thus, \cref{def:feasible-cycle} is also satisfied. We may have unbalanced the number of backward edges in each $\cF_i'$ in the process but, by moving a few additional edges, we will be able to satisfy \cref{def:feasible-backward} without affecting \cref{def:feasible-degV',def:feasible-cycle}. This latter step will be achieved using \cref{lm:movedegree2} below. The overall argument corresponds to \cref{lm:backwardall} below.

\subsubsection{Limitations}
To sum up, we have so far decomposed the edges of $\overleftarrow{T}_\cU$ into digraphs $\cF_1', \dots, \cF_n'$ satisfying \cref{def:feasible-backward}, \cref{def:feasible-degV'}, and \cref{def:feasible-cycle}. Moreover, $\Delta^0(\cF_i')=1$ for each $i\in [n]$. Thus, $\cF_1', \dots, \cF_n'$ are almost pseudo-feasible systems and it only remains to cover $U^{1-\gamma}(T)$ to ensure that \cref{def:feasible-degV*} is satisfied. Unfortunately, this may not be possible. Indeed, suppose that $v\in U^{1-\gamma}(T)$ satisfies $\overleftarrow{d}_{T, \cU}^+(v)=n-1$. Then, there is precisely one $\cF_i'$ which does not contain an outedge at $v$ and so, to turn this $\cF_i'$ into a pseudo-feasible system, we would have to add the unique forward outedge at $v$ in $T$ to $\cF_i'$ (so that we satisfy \cref{def:feasible-degV*}). However, this edge is not a placeholder and so we may break \cref{def:feasible-degV'} and/or \cref{def:feasible-cycle} in the process. More generally, we may not be able to cover the vertices in $U^{1-\gamma}(T)$ which are incident to at least one forward edge. (The vertices $v\in U^{1-\gamma}(T)$ with no forward edges are not a problem because the $n$ backward in- and outedges at $v$ are already entirely covering $v$ in each $\cF_i'$.)

This why, in \cref{lm:backwardall}, we will assume that none of the vertices in $U^{1-\gamma}(T)$ are adjacent to a forward edge (see \cref{lm:backwardall}\cref{lm:backwardall-V**,lm:backwardall-Delta}).
Before applying \cref{lm:backwardall}, we will thus have to construct a few pseudo-feasible systems which cover all the forward edges incident to $U^{1-\gamma}(T)$. This is achieved in \cref{lm:forwardUstar} below.

In \cref{lm:forwardUstar}, we also cover all the forward edges which entirely lie in $U^*$. 
Recall from \cref{lm:backwardedges}\cref{lm:backwardedges-backwardedges} that we will have to incorporate these edges into our pseudo-feasible systems. But, it may not be possible to incorporate them into the pseudo-feasible systems constructed with the above arguments. Indeed, the edges which lie entirely in $U^*$ are not placeholders and so incorporating them may break \cref{def:feasible-degV'} and/or \cref{def:feasible-cycle}. Thus, we will cover them separately with a few pseudo-feasible systems in \cref{lm:forwardUstar}.

Finally, observe that, with the above arguments, we are only constructing pseudo-feasible systems (rather than feasible systems) and we have no control on which edges are used in each of the pseudo-feasible system. Thus, we will have to construct the $t$ feasible systems satisfying \cref{lm:backwardedges}\cref{lm:backwardedges-H,lm:backwardedges-feasible} separately. This is achieved in \cref{lm:exceptionalfeasible} below.

\subsection{Proof of Lemma \ref{lm:backwardedges}}
	
First, we build the $t$ feasible systems which satisfy \cref{lm:backwardedges}\cref{lm:backwardedges-feasible,lm:backwardedges-H}. 
	
\begin{lm}[Constructing the few feasible systems]\label{lm:exceptionalfeasible}
	Let $0<\frac{1}{n}\ll \varepsilon\ll \eta \ll\gamma \ll 1$ and $s\in \mathbb{N}$. Let $T$ be a regular bipartite tournament on $4n$ vertices. Let $\cU=(U_1, \dots, U_4)$ be an optimal $(\varepsilon, 4)$-partition for $T$ and $U^*$ be an $(\varepsilon, \cU)$-exceptional set for $T$.
	Suppose that, for each $i\in [s]$, $H_i\subseteq T$ satisfies \cref{lm:main}\cref{lm:main-H-degree,lm:main-H-degreeV**,lm:main-H-edges,lm:main-H-degreeV*}.
	For each $i\in [s]$, let $s_i\in \mathbb{N}$ and $t_i\coloneqq \sum_{j\in [i-1]} s_j$. Let $t\coloneqq \sum_{i\in [s]} s_i$ and suppose that $t\leq \eta n$.
	Then, there exist edge-disjoint feasible systems $\cF_1, \dots, \cF_t\subseteq T$ for which the following hold, where $\cF\coloneqq \bigcup_{i\in [t]}\cF_i$.
	\begin{enumerate}[label=\rm(\alph*)]
		\item For each $i\in [4]$, we have $e_{\cF-U^{1-\gamma}(T)}(U_i, U_{i-1})=t|U_{i-2}^{1-\gamma}(T)\cup U_{i-3}^{1-\gamma}(T)|$.\label{lm:exceptionalfeasible-U'}
		\item For each $i\in [t]$, we have $e(\cF_i)\leq \sqrt{\varepsilon}n$.\label{lm:exceptionalfeasible-size}
        \item For each $i\in [t]$, we have $V^0(\cF_i)=U^*$.\label{lm:exceptionalfeasible-feasible}
        \item For each $i\in [s]$ and $j\in [s_i]$, we have $\cF_{t_i+j}\subseteq H_i$.\label{lm:exceptionalfeasible-H}
    \end{enumerate}
\end{lm}

Note that \cref{lm:exceptionalfeasible}\cref{lm:exceptionalfeasible-feasible,lm:exceptionalfeasible-H} will automatically imply \cref{lm:backwardedges}\cref{lm:backwardedges-feasible,lm:backwardedges-H}; while \cref{lm:exceptionalfeasible}\cref{lm:exceptionalfeasible-size} corresponds to \cref{lm:backwardedges}\cref{lm:backwardedges-size}. 

Then, we construct a few pseudo-feasible systems which cover all the forward edges in $U^*$ and all the forward edges incident to $U^{1-\gamma}(T)$. 

\begin{lm}[Covering the forward edges in $U^*$ and incident to $U^{1-\gamma}(T)$]\label{lm:forwardUstar}
	Let $0<\frac{1}{n}\ll \varepsilon\ll \eta \ll \gamma\ll 1$. Let $t\leq \eta n$ and $t'\in \{\lfloor\gamma n\rfloor,\lfloor\gamma n\rfloor+1\}$.
	Let $T$ be a regular bipartite tournament on $4n$ vertices. Let $\cU$ be an optimal $(\varepsilon,4)$-partition for $T$ and $U^*$ be an $(\varepsilon, \cU)$-exceptional set for $T$.
	Let $D\subseteq T$ and suppose that the following hold.
	\begin{enumerate}
		\item $\Delta^0(T\setminus D)\leq t$.\label{lm:forwardUstar-D-Delta0}
		\item For each $i\in [4]$, we have $e_{(T\setminus D)-U^{1-\gamma}(T)}(U_i, U_{i-1})\leq t|U_{i-2}^{1-\gamma}(T)\cup U_{i-3}^{1-\gamma}(T)|$.\label{lm:forwardUstar-D-U'}
	\end{enumerate}
	Then, there exist edge-disjoint $(\gamma,T)$-pseudo-feasible systems $\cF_1, \dots, \cF_{t'}$ such that the following hold, where $\cF\coloneqq \bigcup_{i\in [t']} \cF_i$.
	\begin{enumerate}[label=\rm(\alph*)]
		\item $E(\overrightarrow{D}_\cU[U^*])\subseteq E(\cF)\subseteq E(D)$.\label{lm:forwardUstar-forwardV*}
		\item For each $i\in [t']$, we have $e(\cF_i)\leq \sqrt{\varepsilon}n$.\label{lm:forwardUstar-size}
		\item For each $v\in U^{1-\gamma}(T)$, we have $\overrightarrow{d}_{\cF,\cU}^\pm(v)=\overrightarrow{d}_{D,\cU}^\pm(v)$.\label{lm:forwardUstar-forwardV**}
		\item For each $v\in U^{1-2\gamma}(T)\setminus U^{1-\gamma}(T)$, we have $\overleftarrow{d}_{\cF,\cU}^\pm(v)\geq t'-4\varepsilon n$.\label{lm:forwardUstar-forwarddegreeV*}
	\end{enumerate}
\end{lm}

Finally, we use the arguments outlined in \cref{sec:sketchbackward-general} to construct the remaining pseudo-feasible systems using all the leftover backward edges of $T$.

\begin{lm}[Decomposing the backward edges]\label{lm:backwardall}
	Let $0<\frac{1}{n}\ll\varepsilon\ll\gamma \ll 1$ and $\gamma n< r\leq \frac{n}{2}$. Let $T$ be a regular bipartite tournament on $4n$ vertices. Let $\cU$ be an $(\varepsilon, 4)$-partition for $T$ and $U^*$ be an $(\varepsilon, \cU)$-exceptional set for $T$.
	Let $D\subseteq \overleftarrow{T}_{\cU}$ satisfy the following.
	\begin{enumerate}
		\item $e_D(U_1,U_4)=e_D(U_3, U_2)$ and $e_D(U_4,U_3)=e_D(U_2,U_1)$.\label{lm:backwardall-backward}
		\item $\Delta^0(D)\leq 2r$.\label{lm:backwardall-Delta}
		\item For each $v\in U^{1-\gamma}(T)$, $d_D^\pm(v)=2r$.\label{lm:backwardall-V**}
	\end{enumerate}
	Then, there exist edge-disjoint $(\gamma,T)$-pseudo-feasible systems $\cF_1, \dots, \cF_{2r}$ such that $D=\bigcup_{i\in [2r]}\cF_i$ and $e(\cF_i)\leq \sqrt{\varepsilon}n$ for each $i\in [2r]$.
\end{lm}
	
We first assume that \cref{lm:exceptionalfeasible,lm:forwardUstar,lm:backwardall} hold and derive \cref{lm:backwardedges} as follows. \Cref{lm:exceptionalfeasible,lm:forwardUstar} will be proved in \cref{sec:specialpseudofeasible}.
	
\begin{proof}[Proof of \cref{lm:backwardedges}]		
	Apply \cref{lm:exceptionalfeasible} to obtain edge-disjoint feasible systems $\cF_1, \dots, \cF_t\subseteq T$ which satisfy \cref{lm:exceptionalfeasible}\cref{lm:exceptionalfeasible-U',lm:exceptionalfeasible-size,lm:exceptionalfeasible-feasible,lm:exceptionalfeasible-H}. In particular,  \cref{lm:backwardedges-feasible,lm:backwardedges-H} hold.
	
	Let $D\coloneqq T\setminus \bigcup_{i\in [t]}\cF_i$. Then, \cref{lm:forwardUstar}\cref{lm:forwardUstar-D-Delta0} follows from \cref{def:feasible-linforest} and \cref{lm:forwardUstar}\cref{lm:forwardUstar-D-U'} follows from \cref{lm:exceptionalfeasible}\cref{lm:exceptionalfeasible-U'}.
	Let $t'\in \{\lfloor\gamma n\rfloor, \lfloor\gamma n\rfloor+1\}$ be such that $n-t-t'$ is even.
	Let $\cF_{t+1}, \dots, \cF_{t+t'}$ be the edge-disjoint $(\gamma,T)$-pseudo-feasible systems obtained by applying \cref{lm:forwardUstar}. 
	
	Let $D'\coloneqq D\setminus \bigcup_{i\in [t']}\cF_{t+i}=T\setminus \bigcup_{i\in [t+t']}\cF_i$. Let $r\coloneqq \frac{n-t-t'}{2}$ and note that $r\in \mathbb{N}$.
	We claim that \cref{lm:backwardall}\cref{lm:backwardall-backward,lm:backwardall-Delta,lm:backwardall-V**,lm:backwardall-V**} are satisfied with $\overleftarrow{D'}_\cU$ playing the role of $D$.
	Indeed, \cref{lm:backwardall}\cref{lm:backwardall-backward} holds by \cref{fact:backwardedges} and \cref{def:feasible-backward}.
	By \cref{eq:Ugamma}, each $v\in V(T)\setminus U^{1-2\gamma}(T)$ satisfies $\overleftarrow{d}_{D',\cU}^\pm(v)\leq \overleftarrow{d}_{T, \cU}^\pm(v)\leq (1-2\gamma)n\leq 2r$. 
	Moreover, each $v\in U^{1-2\gamma}(T)\setminus U^{1-\gamma}(T)$ satisfies \[\overleftarrow{d}_{D',\cU}^\pm(v)=\overleftarrow{d}_{T, \cU}^\pm(v)-\sum_{i\in [t+t']}\overleftarrow{d}_{\cF_i,\cU}^\pm (v)\stackrel{\text{\cref{lm:forwardUstar}\cref{lm:forwardUstar-forwarddegreeV*}}}{\leq}(1-\gamma)n-(t'-4\varepsilon n)\leq 2r.\] 
	Similarly, each $v\in U^{1-\gamma}(T)$ satisfies
	\[\overleftarrow{d}_{D',\cU}^\pm(v)=\overleftarrow{d}_{T, \cU}^\pm(v)-\sum_{i\in [t+t']}\overleftarrow{d}_{\cF_i,\cU}^\pm (v)\stackrel{\text{\cref{def:feasible-exceptional},\cref{def:feasible-degV*},\cref{lm:forwardUstar}\cref{lm:forwardUstar-forwardV**}}}{=} \overleftarrow{d}_{T, \cU}^\pm(v)-(t+t'-\overrightarrow{d}_{T, \cU}^\pm(v))=2r.\]
	Thus, \cref{lm:backwardall}\cref{lm:backwardall-Delta,lm:backwardall-V**} are satisfied.			
	Apply \cref{lm:backwardall} with $\overleftarrow{D'}_\cU$ playing the role of $D$ to obtain $2r$ edge-disjoint $(\gamma,T)$-pseudo-feasible systems $\cF_{t+t'+1}, \dots, \cF_n$ such that $\overleftarrow{D'}_\cU=\bigcup_{i\in [2r]}\cF_{t+t'+i}$ and $e(\cF_{t+t'+i})\leq \sqrt{\varepsilon}n$ for each $i\in [2r]$.
	
	Then, \cref{lm:backwardedges-backwardedges} follows from \cref{lm:exceptionalfeasible}, \cref{lm:forwardUstar}\cref{lm:forwardUstar-forwardV*}, and \cref{lm:backwardall}.
	Moreover, \cref{lm:backwardedges-size} holds by \cref{lm:exceptionalfeasible}\cref{lm:exceptionalfeasible-size}, \cref{lm:forwardUstar}\cref{lm:forwardUstar-size}, and \cref{lm:backwardall}.
	Finally, \cref{lm:backwardedges-feasible,lm:backwardedges-H} follow from \cref{lm:exceptionalfeasible}\cref{lm:exceptionalfeasible-feasible,lm:exceptionalfeasible-H}.
\end{proof}

\subsection{Proof of Lemma \ref{lm:backwardall}}

We use the arguments presented in \cref{sec:sketchbackward-general}. To keep the number of backward edges balanced in each pseudo-feasible system, we will use the following lemma. Roughly speaking, \cref{lm:movedegree2} states that if we have two equal sized matchings $M_1$ and $M_2$ in an auxiliary digraph where a set $W$ of vertices is replaced by two copies $W^1$ and $W^2$ of $W$, then there exist equal sized $M_1'\subseteq M_1$ and $M_2'\subseteq M_2$ (see \cref{lm:movedegree2}\cref{lm:movedegree2-size2}) such that each of $M_1', M_1\setminus M_1', M_2'$, and $M_2\setminus M_2'$ cover at most one copy of each vertex in $W$ (see \cref{lm:movedegree2}\cref{lm:movedegree2-degree2}) and all edges in $M_1'$ or $M_2'$ are incident to $W^2$ (see \cref{lm:movedegree2}\cref{lm:movedegree2-A2}). In the proof of \cref{lm:backwardall}, we will move the matchings $M_1'$ and $M_2'$ from $\cF_i'$ to $\cF_{\frac{n}{2}+i}'$ (as discussed in \cref{sec:sketchbackward-general}) and so \cref{lm:movedegree2}\cref{lm:movedegree2-degree2} will ensure that the maximum semidegree is reduced to $1$. Moreover, we will construct our auxiliary digraph (see \cref{sec:sketchbackward-general}) in such a way that all the edges incident to the second copy $W^2$ of $W\coloneqq U^{\frac{1}{2}}(T)$ correspond to placeholders. Thus, \cref{lm:movedegree2}\cref{lm:movedegree2-A2} will ensure that no problematic cycle is created (i.e.\ every cycle will contain a placeholder). Since $|M_1'|=|M_2'|$ and $|M_1\setminus M_1'|=|M_2\setminus M_2'|$, the number of backward edges will remain balanced.

\begin{lm}\label{lm:movedegree2}
	Let $W$ and $V'$ be disjoint vertex sets and suppose that $W^1$ and $W^2$ are two copies of $W$. 
	Let $M_1$ and $M_2$ be undirected matchings satisfying the following properties.
	\begin{enumerate}
		\item $|M_1|=|M_2|$.\label{lm:movedegree2-size}
		\item $V(M_1\cup M_2)\subseteq V'\cup W^1\cup W^2$.\label{lm:movedegree2-V}
		\item $e_{M_1\cup M_2}(W^2,W^1\cup W^2)=0$.\label{lm:movedegree2-nonexceptional}
	\end{enumerate}
	Then, there exist $M_1'\subseteq M_1$ and $M_2'\subseteq M_2$ such that the following hold.
	\begin{enumerate}[label=\rm(\alph*),ref=(\alph*)]
		\item $|M_1'|=|M_2'|$.\label{lm:movedegree2-size2}
		\item There exists $i\in [2]$ such that all the edges in $M_i'$ are incident to $W^2$.\label{lm:movedegree2-A2}
		\item For each $i\in [2]$ and $w\in W$, both $V(M_i')$ and $V(M_i\setminus M_i')$ contain at most one copy of~$w$.\label{lm:movedegree2-degree2}
	\end{enumerate}
\end{lm}

\begin{proof}
	By induction on $m\coloneqq|M_1|=|M_2|$. If $m\in \{0,1\}$, then we let $M_1'\coloneqq \emptyset\eqqcolon M_2'$ and we are done. For the induction step, suppose that $m\geq 2$ and that the \lcnamecref{lm:movedegree2} holds for any matchings of size less than $m$ which satisfy \cref{lm:movedegree2-V,lm:movedegree2-nonexceptional,lm:movedegree2-size}.
	
	For any $w\in W$, denote by $w^1\in W^1$ and $w^2\in W^2$ the copies of $w$. 	
	For each $i\in [2]$, denote by $X_i^1\coloneqq \{w^1\in W^1\mid w^1,w^2\in V(M_i)\}$ the set of vertices $w^1\in W^1$ such that $M_i$ covers both $w^1$ and its corresponding vertex $w^2\in W^2$.
	Note that if $X_1^1=\emptyset=X_2^1$, then we may simply set $M_1'=\emptyset=M_2'$ and we are done. Thus, we may view $X_1^1$ and $X_2^1$ as the set of (first copies of the) problematic vertices in $M_1$ and $M_2$. For each $i\in [2]$, let $Y_i^1\coloneqq \{w^1\in X_i^1\mid N_{M_i}(w^1)\subseteq X_i^1\}$, define $Z_i^1\coloneqq X_i^1\setminus Y_i^1$, and note that $M_i[Y_i^1]$ is a matching of size $\frac{|Y_i^1|}{2}$.
	
	\begin{case}		
		\item \textbf{$\min \{|Y_1^1|,|Y_2^1|\}\neq 0$.}
		Then, there exist $v_1^1w_1^1\in M_1[Y_1^1]$ and $v_2^1w_2^1\in M_2[Y_2^1]$. For each $i\in [2]$, denote by $e_i$ and $e_i'$ the edges of $M_i$ which are incident to $v_i^2$ and $w_i^2$, respectively ($e_i$ and $e_i'$ exist by definition of $X_i^1\supseteq Y_i^1$). Note that \cref{lm:movedegree2-nonexceptional} implies that $e_i\neq e_i'$ for each $i\in [2]$.
		Then, observe that $\tM_1\coloneqq M_1\setminus \{v_1^1w_1^1,e_1,e_1'\}$ and $M_2\setminus \{v_2^1w_2^1,e_2,e_2'\}$ are matchings of size $m-3$ which still satisfy \cref{lm:movedegree2-V,lm:movedegree2-nonexceptional,lm:movedegree2-size}. Thus, the induction hypothesis implies that there exist $\tM_1'\subseteq \tM_1$ and $\tM_2'\subseteq \tM_2$ such that \cref{lm:movedegree2-A2,lm:movedegree2-degree2,lm:movedegree2-size2} hold with $\tM_1, \tM_1', \tM_2$, and $\tM_2'$ playing the roles of $M_1, M_1', M_2$, and $M_2'$. Let $M_1'\coloneqq \tM_1'\cup \{e_1, e_1'\}$ and $M_2'\coloneqq \tM_2'\cup \{e_2, e_2'\}$. Since \cref{lm:movedegree2-A2,lm:movedegree2-size2} hold with $\tM_1, \tM_1', \tM_2$, and $\tM_2'$ playing the roles of $M_1, M_1', M_2$, and $M_2'$ and since $e_1,e_1',e_2$, and $e_2'$ are all incident to $W^2$, \cref{lm:movedegree2-A2,lm:movedegree2-size2} hold. For \cref{lm:movedegree2-degree2}, let $i\in [2]$. By construction, \cref{lm:movedegree2-degree2} holds for $v_i$ and $w_i$. Let $w\in W\setminus \{v_i,w_i\}$. 
		By definition, $v_i^1w_i^1$ does not cover a copy of $w$ and, by \cref{lm:movedegree2-nonexceptional}, neither $e_i$ nor $e_i'$ covers a copy of $w$.
		Therefore, since \cref{lm:movedegree2-degree2} holds with $\tM_i$ and $\tM_i'$ playing the roles of $M_i$ and $M_i'$, we have
		\[|V(M_i\setminus M_i')\cap \{w^1, w^2\}|=|V(\tM_i\setminus \tM_i')\cap \{w^1, w^2\}|\leq 1\]
		and
		\[|V(M_i')\cap \{w^1, w^2\}|=|V(\tM_i')\cap \{w^1,w^2\}|\leq 1.\]
		Thus, \cref{lm:movedegree2-degree2} holds and we are done.
		
		\item \textbf{$\min \{|Z_1^1|,|Z_2^1|\}\neq 0$.}
		Then, there exist $w_1^1\in Z_1^1$ and $w_2^1\in Z_2^1$. For each $i\in [2]$, denote by $e_i^1$ and $e_i^2$ the edges of $M_i$ which are incident to $w_i^1$ and $w_i^2$, respectively ($e_i^1$ and $e_i^2$ exist by definition of $X_i^1\supseteq Z_i^1$). Note that, by \cref{lm:movedegree2-nonexceptional}, $e_i^1$ and $e_i^2$ are distinct. Moreover, $\hM_1\coloneqq M_1\setminus \{e_1^1,e_1^2\}$ and $\hM_2\coloneqq M_2\setminus \{e_2^1,e_2^2\}$ are matchings of size $m-2$ and still satisfy \cref{lm:movedegree2-V,lm:movedegree2-nonexceptional,lm:movedegree2-size}. Thus, the induction hypothesis implies that there exist $\hM_1'\subseteq \hM_1$ and $\hM_2'\subseteq \hM_2$ such that \cref{lm:movedegree2-A2,lm:movedegree2-degree2,lm:movedegree2-size2} hold with $\hM_1, \hM_1', \hM_2$, and $\hM_2'$ playing the roles of $M_1, M_1', M_2$, and $M_2'$. Let $M_1'\coloneqq \hM_1'\cup \{e_1^2\}$ and $M_2'\coloneqq \hM_2'\cup \{e_2^2\}$. Since \cref{lm:movedegree2-A2,lm:movedegree2-size2} hold with $\hM_1, \hM_1', \hM_2$, and $\hM_2'$ playing the roles of $M_1, M_1', M_2$, and $M_2'$ and since $e_1^2$ and $e_2^2$ are both incident to $W^2$, \cref{lm:movedegree2-size2,lm:movedegree2-A2} hold.
		For \cref{lm:movedegree2-degree2}, let $i\in [2]$. By construction, \cref{lm:movedegree2-degree2} holds for $w_i$. Suppose that $w\in W\setminus \{w_i\}$.
		If $w^1\notin X_i^1$, then $V(M_i)$ contains at most one copy of $w$ and so \cref{lm:movedegree2-degree2} holds for $w$. We may therefore assume that $w^1\in X_i^1$. 
		By \cref{lm:movedegree2-nonexceptional}, $e_i^2$ does not cover a copy of $w$ and, by definition of $Z_i^1$, neither does $e_i^1$.
		Therefore, since \cref{lm:movedegree2-degree2} holds with $\hM_i$ and $\hM_i'$ playing the roles of $M_i$ and $M_i'$, we have
		\[|V(M_i\setminus M_i')\cap \{w^1, w^2\}|=|V(\hM_i\setminus \hM_i')\cap \{w^1, w^2\}|\leq 1\]
		and
		\[|V(M_i')\cap \{w^1, w^2\}|=|V(\hM_i')\cap \{w^1,w^2\}|\leq 1.\] 
		Thus, \cref{lm:movedegree2-degree2} holds and we are done.
		
		\item \textbf{$\min \{|Y_1^1|,|Y_2^1|\}=0=\min \{|Z_1^1|,|Z_2^1|\}$.} Define $y\coloneqq \max \{|Y_1^1|,|Y_2^1|\}$ and $z\coloneqq \max \{|Z_1^1|,|Z_2^1|\}$.
		For each $i\in [2]$, denote $X_i^2\coloneqq \{w^2\mid w^1\in X_i^1\}$ and recall that $M_i[Y_i^1]$ is a matching of size $\frac{|Y_i^1|}{2}$. Thus, \cref{lm:movedegree2-nonexceptional} implies that each $i\in [2]$ and $S^2\subseteq X_i^2$ satisfies
    	\begin{align}
    		|\{e\in M_i\mid V(e)\cap S^2\neq \emptyset\}|
    		&= |S^2|\label{eq:movedegree2-Xi'}
    	\end{align}
    	and
    	\begin{align}
    		|\{e\in M_i\mid V(e)\cap (X_i^1\cup X_i^2)\neq \emptyset\}|
    	    &\stackrel{\text{\eqmakebox[movedegree2]{}}}{=} |\{e\in M_i\mid V(e)\cap Y_i^1\neq \emptyset\}|\nonumber\\
    	    &\eqmakebox[movedegree2]{}\qquad+ |\{e\in M_i\mid V(e)\cap Z_i^1\neq \emptyset\}|\nonumber\\
    		&\eqmakebox[movedegree2]{}\qquad+|\{e\in M_i\mid V(e)\cap X_i^2\neq \emptyset\}|\nonumber\\
    		&\stackrel{\text{\eqmakebox[movedegree2]{\text{\cref{eq:movedegree2-Xi'}}}}}{=} \frac{|Y_i^1|}{2}+|Z_i^1|+|X_i^2|=\frac{3|Y_i^1|}{2}+2|Z_i^1|.\label{eq:movedegree2}
    	\end{align}
    	We proceed as follows.
    	\begin{enumerate}[label=\textbf{Case 3.\arabic*:},ref=3.\arabic*,wide,parsep=0pt,itemsep=10pt,topsep=10pt]
    		\item \textbf{There exists $i\in [2]$ such that $|Y_i^1|=0=|Z_i^1|$.}\label{case:sameside}
    		Suppose without loss of generality that $|Y_1^1|=0=|Z_1^1|$, $|Y_2^1|=y$, and $|Z_2^1|=z$. Then,
    		\[|M_1|\stackrel{\text{\cref{lm:movedegree2-size}}}{=}|M_2| \stackrel{\text{\cref{eq:movedegree2}}}{\geq}\frac{3y}{2}+2z\geq y+z.\]
    		Let $M_1'\subseteq M_1$ satisfy $|M_1'|=y+z$ and let $M_2'$ consists of all the edges of $M_2$ which are incident to $X_2^2$. Then, all the edges of $M_2'$ are incident to $W^2$ and so \cref{lm:movedegree2-A2} holds. Moreover,
    		\[|M_2'|\stackrel{\text{\cref{eq:movedegree2-Xi'}}}{=}y+z=|M_1'|.\]
    		Thus, \cref{lm:movedegree2-size2} holds. Recall that $X_1^1=\emptyset=X_1^2$. By construction, $X_2^2\subseteq V(M_2')$ and, by \cref{lm:movedegree2-nonexceptional}, $V(M_2')\cap X_2^1=\emptyset$. Therefore, \cref{lm:movedegree2-degree2} holds and we are done.
    		
    		\item \textbf{$z \geq y$.}
    		By Case \ref{case:sameside}, we may assume without loss of generality that $|Y_1^1|=y$, $|Z_1^1|=0=|Y_2^1|$, and $|Z_2^1|=z$.
    		Then,
    		\begin{align*}
    			|M_1-(X_1^1\cup X_1^2)|\stackrel{\text{\cref{eq:movedegree2}}}{=} |M_1|-\frac{3y}{2} \stackrel{\text{\cref{lm:movedegree2-size}}}{=}|M_2|-\frac{3y}{2}\stackrel{\text{\cref{eq:movedegree2}}}{\geq} 2z-\frac{3y}{2}\geq z-y.
    		\end{align*}
    		Let $M_1'$ consist of all the edges of $M_1$ which are incident to $X_1^2$ plus $z-y$ edges of $M_1-(X_1^1\cup X_1^2)$. Let $M_2'$ consist of all the edges of $M_2$ which are incident to $X_2^2$.
    		Then, all the edges in $M_2'$ are incident to $W^2$ and so \cref{lm:movedegree2-A2} holds. Moreover,
    		\begin{align*}
    			|M_1'|\stackrel{\text{\cref{eq:movedegree2-Xi'}}}{=} y+(z-y)= z \stackrel{\text{\cref{eq:movedegree2-Xi'}}}{=} |M_2'|.
    		\end{align*}
    		Thus, \cref{lm:movedegree2-size2} holds. 
    		Let $i\in [2]$. By construction, $X_i^2\subseteq V(M_i')$ and, by \cref{lm:movedegree2-nonexceptional}, $V(M_i')\cap X_i^1=\emptyset$. Thus, \cref{lm:movedegree2-degree2} holds and we are done.
    		
    		\item \textbf{$y \geq 2z$.}
    		By Case \ref{case:sameside}, we may assume without loss of generality that $|Y_1^1|=y$, $|Z_1^1|=0=|Y_2^1|$, and $|Z_2^1|=z$.
    		Then,
    		\begin{align*}
    			|M_2-(X_2^1\cup X_2^2)|\stackrel{\text{\cref{eq:movedegree2}}}{=} |M_2|-2z \stackrel{\text{\cref{lm:movedegree2-size}}}{=}|M_1|-2z\stackrel{\text{\cref{eq:movedegree2}}}{\geq} \frac{3y}{2}-2z\geq y-z.
    		\end{align*}
    		Let $M_1'$ consist of all the edges of $M_1$ which are incident to $X_1^2$. Let $M_2'$ consist of all the edges of $M_2$ which are incident to $X_2^2$  plus $y-z$ edges of $M_2-(X_2^1\cup X_2^2)$.
    		Then, all the edges in $M_1'$ are incident to $W^2$ and so \cref{lm:movedegree2-A2} holds. Moreover,
    		\begin{align*}
    			|M_1'|\stackrel{\text{\cref{eq:movedegree2-Xi'}}}{=} y= z+(y-z) \stackrel{\text{\cref{eq:movedegree2-Xi'}}}{=} |M_2'|.
    		\end{align*}
    		Thus, \cref{lm:movedegree2-size2} holds. 
    		Let $i\in [2]$. By construction, $X_i^2\subseteq V(M_i')$ and, by \cref{lm:movedegree2-nonexceptional}, $V(M_i')\cap X_i^1=\emptyset$. Thus, \cref{lm:movedegree2-degree2} holds and we are done.
    		
    		\item \textbf{$2z\geq y \geq z$.}
    		By Case \ref{case:sameside}, we may assume without loss of generality that $|Y_1^1|=y$, $|Z_1^1|=0=|Y_2^1|$, and $|Z_2^1|=z$.
    		Then, $M_1[Y_1^1]$ is a matching of size $\frac{y}{2}\geq y-z$.
    		Let $S^1\subseteq M_1[Y_1^1]$ satisfy $|S^1|=y-z$ and define $S^2\coloneqq \{w^2\mid w^1\in V(S^1)\}$.
    		Let $M_1'$ be obtained from $S^1$ by adding all the edges $M_1$ which are incident to $X_1^2\setminus S^2$. Let $M_2'$ consist of all the edges of $M_2$ which are incident to $X_2^2$.
    		Then, all the edges in $M_2'$ are incident to $W^2$ and so \cref{lm:movedegree2-A2} holds.
    		Note that
    		\begin{align*}
    			|M_1'|\stackrel{\text{\cref{eq:movedegree2-Xi'}}}{=}|S^1|+ (y-2|S^1|)=z\stackrel{\text{\cref{eq:movedegree2-Xi'}}}{=} |M_2'|.
    		\end{align*}
    		Thus, \cref{lm:movedegree2-size2} holds. 
    		By construction, $X_2^2\subseteq V(M_2')$ and, by \cref{lm:movedegree2-nonexceptional}, $V(M_2')\cap X_2^1=\emptyset$. Moreover, \cref{lm:movedegree2-nonexceptional} implies that $V(M_1')\cap X_1^2=X_1^2\setminus S^2$ and $V(M_1')\cap X_1^1=V(S^1)$. Thus, \cref{lm:movedegree2-degree2} holds and we are done.\qedhere
    	\end{enumerate}
    	\end{case}	
\end{proof}

\begin{proof}[Proof of \cref{lm:backwardall}]
	We proceed as follows.
	\begin{steps}
		\item \textbf{Constructing auxiliary graphs.}
		For each $i\in [4]$, we construct an auxiliary (undirected) graph $H_i$ as follows.
		Let $i\in [4]$.
		Let $W_i$ be the set of vertices $w\in U_i\cup U_{i-1}$ such that $d_{D[U_i, U_{i-1}]}(w)\geq r$. Observe that
		\begin{equation}\label{eq:backwardall-W}
			W_i\subseteq U^\gamma(T)\stackrel{\text{\cref{def:ES-backward},\cref{fact:U}}}{\subseteq} U^*.
		\end{equation}
		Let $W_i^1$ and $W_i^2$ be two copies of $W_i$. For each $w\in W_i$, we denote by $w^1\in W_i^1$ and $w^2\in W_i^2$ the copies of $w$.
		For each $w\in W_i$, let $N_i^1(w)\cup N_i^2(w)$ be a partition of $N_{D[U_i,U_{i-1}]}(w)$ satisfying $|N_i^j(w)|\leq r$ for each $j\in [2]$ (this is possible by \cref{lm:backwardall-Delta}). By \cref{def:ES-size}, $|U^*|\leq r$ and so we may assume that
		\begin{equation}\label{eq:backwardall-U*}
			U^*\cap N_i^2(w)=\emptyset
		\end{equation}
		for each $w\in W_i$.
		
		For each $i\in [4]$, let $H_i$ be the (undirected) graph on $((U_i\cup U_{i-1})\setminus W_i)\cup (W_i^1\cup W_i^2)$ which contains all of the following edges, and no other edges.
		\begin{itemize}
			\item If $uv\in E(D[U_i\setminus W_i, U_{i-1}\setminus W_i])$, then $uv\in E(H_i)$.
			\item If $uv\in E(D[U_i\cap W_i, U_{i-1}\cap W_i])$, then $u^1v^1\in E(H_i)$.
			\item Suppose that $uv\in E(D[U_i\cap W_i, U_{i-1}\setminus W_i])$. If $v\in N_i^1(u)$, then $u^1v\in E(H_i)$. Otherwise, $v\in N_i^2(u)$ and $u^2v\in E(H_i)$.
			\item Suppose that $uv\in E(D[U_i\setminus W_i, U_{i-1}\cap W_i])$. If $u\in N_i^1(v)$, then $uv^1\in E(H_i)$. Otherwise, $u\in N_i^2(v)$ and $uv^2\in E(H_i)$.
		\end{itemize}
		
		\begin{claim}
			For each $i\in [4]$, $H_i$ is a bipartite graph which satisfies the following properties.
			\begin{enumerate}[label=\rm(\alph*)]
				\item $V(H_i)= ((U_i\cup U_{i-1})\setminus W_i)\cup (W_i^1\cup W_i^2)$.\label{claim:backwardallH-V}
				\item $e(H_i)=e_D(U_i, U_{i-1})$.\label{claim:backwardallH-e}
				\item $\Delta(H_i)\leq r$.\label{claim:backwardallH-Delta}
				\item $e_{H_i}(W_i^2, W_i^1\cup W_i^2)=0$.\label{claim:backwardallH-nonexceptional}
			\end{enumerate}
		\end{claim}
	
		\begin{proofclaim}
			Let $i\in [4]$. One can easily verify that $H_i$ is a bipartite graph on vertex classes $(U_i\setminus W_i)\cup \{w^j\mid j\in [2], w\in W_i\cap U_i\}$ and $(U_{i-1}\setminus W_i)\cup \{w^j\mid j\in [2], w\in W_i\cap U_{i-1}\}$. In particular, \cref{claim:backwardallH-V} holds.
			Moreover, there is a one-to-one correspondence between the edges of $D[U_i, U_{i-1}]$ and $H_i$ and so \cref{claim:backwardallH-e} holds. For \cref{claim:backwardallH-Delta},			
			observe that each $v\in (U_i\cup U_{i-1})\setminus W_i$ satisfies $d_{H_i}(v)=d_{D[U_i, U_{i-1}]}(v)\leq r$, while each $w^1\in W_i^1$ satisfies
			\[d_{H_i}(w^1)=|N_{D[U_i, U_{i-1}]}(w)\cap W_i|+|N_i^1(w)\setminus W_i|\stackrel{\text{\cref{eq:backwardall-W},\cref{eq:backwardall-U*}}}{=}|N_i^1(w)|\leq r.\]
			Moreover, each $w^2\in W_i^2$ satisfies $N_{H_i}(w^2)=N_i^2(w)\setminus W_i$ and so $d_{H_i}(w^2)\leq r$. Thus, \cref{claim:backwardallH-Delta,claim:backwardallH-nonexceptional} hold.
		\end{proofclaim}
		
		\item \textbf{Decomposing the auxiliary graphs.}		
		For each $i\in [4]$, apply \cref{prop:Konigsamesize} to decompose $H_i$ into $r$ edge-disjoint (undirected) matchings $M_{i,1}, \dots, M_{i,r}$ such that, for any $j,j'\in [r]$, $||M_{i,j}|-|M_{i,j'}||\leq 1$ (this is possible by \cref{claim:backwardallH-Delta}).
		By \cref{lm:backwardall-backward} and \cref{claim:backwardallH-e}, we may assume without loss of generality that, for each $i\in [r]$, we have
		\begin{equation}\label{eq:backwardall-size}
			|M_{1,i}|=|M_{3,i}| \quad \text{and} \quad |M_{4,i}|=|M_{2,i}|.
		\end{equation}
		Moreover, each $i\in [4]$ and $j\in [r]$ satisfy
		\begin{align}\label{eq:backwardall-M}
			|M_{i,j}|\stackrel{\text{\cref{fact:epsilon4partition}}}{\leq} \left\lceil\frac{\varepsilon n^2}{r}\right\rceil \leq \frac{\sqrt{\varepsilon} n}{4}.
		\end{align}
	
		\item \textbf{Decomposing $D$.}
		Let $i\in [2]$ and $j\in [r]$.
		Define $W\coloneqq W_i\cup W_{i+2}$, $W^1\coloneqq W_i^1\cup W_{i+2}^1$, and $W^2\coloneqq W_i^2\cup W_{i+2}^2$. Let $V'\coloneqq V(T)\setminus W$.
		Note that \cref{lm:movedegree2}\cref{lm:movedegree2-V,lm:movedegree2-nonexceptional,lm:movedegree2-size} are satisfied with $M_{i,j}$ and $M_{i+2,j}$ playing the roles of $M_1$ and $M_2$.
		Indeed, \cref{lm:movedegree2}\cref{lm:movedegree2-size} follows from \cref{eq:backwardall-size}, while \cref{lm:movedegree2}\cref{lm:movedegree2-V} holds by \cref{claim:backwardallH-V}. Moreover, \cref{lm:movedegree2}\cref{lm:movedegree2-nonexceptional} follows from \cref{claim:backwardallH-V,claim:backwardallH-nonexceptional}%
			\COMMENT{Need \cref{claim:backwardallH-V} for $e_{H_i\cup H_{i+2}}(W_i^2,W_{i+2}^1\cup W_{i+2}^2)=0=e_{H_i\cup H_{i+2}}(W_{i+2}^2,W_i^1\cup W_i^2)$.}.
		Apply \cref{lm:movedegree2} with $M_{i,j}$ and $M_{i+2,j}$ playing the roles of $M_1$ and $M_2$ to obtain $M_{i,j}'\subseteq M_{i,j}$ and $M_{i+2,j}'\subseteq M_{i+2,j}$ satisfying \cref{lm:movedegree2}\cref{lm:movedegree2-size2,lm:movedegree2-A2,lm:movedegree2-degree2}.
		For each $i'\in \{i,i+2\}$, let $\tM_{i',j}$ and $\tM_{i',j}'$ be obtained from $M_{i',j}$ and $M_{i',j}'$ by replacing, for each $j\in [2]$, each $w^j\in W_{i'}^j$ by $w$, and then orienting all the edges from $U_{i'}$ to $U_{i'-1}$.
		By definition of $H_i$ and $H_{i+2}$, we have $\tM_{i,j}'\subseteq \tM_{i,j}\subseteq E_D(U_i, U_{i-1})$ and $\tM_{i+2,j}'\subseteq \tM_{i+2,j}\subseteq E_D(U_{i+2}, U_{i+1})$.
		
		For each $i\in [r]$, let 
		\begin{itemize}
		    \item $\cF_i\coloneqq (\tM_{1,i}\setminus \tM_{1,i}')\cup \tM_{2,i}'\cup (\tM_{3,i}\setminus \tM_{3,i}')\cup \tM_{4,i}'$ and
		    \item $\cF_{r+i}\coloneqq \tM_{1,i}'\cup(\tM_{2,i}\setminus \tM_{2,i}')\cup \tM_{3,i}'\cup(\tM_{4,i}\setminus \tM_{4,i}')$.
		\end{itemize} 
		By definition, there is a one-to-one correspondence between the edges of $H_1\cup \dots\cup H_4$ and $D$. Thus, $\cF_1, \dots, \cF_{2r}$ are edge-disjoint and $D=\bigcup_{i\in [2r]}\cF_i$. By \cref{eq:backwardall-M}, we have $e(\cF_i)\leq \sqrt{\varepsilon}n$ for each $i\in [2r]$.
		
		\item \textbf{Verifying \cref{def:feasible-backward} and \cref{def:feasible-degV*,def:feasible-degV',def:feasible-cycle}.} Finally, we verify that $\cF_1, \dots, \cF_{2r}$ are $(\gamma, T)$-pseudo-feasible systems.
		Let $j\in [r]$.
		First, observe that
		\begin{align*}
			e_{\cF_j}(U_1, U_4)=|\tM_{1,j}|- |\tM_{1,j}'|\stackrel{\text{\cref{eq:backwardall-size},\cref{lm:movedegree2}\cref{lm:movedegree2-size2}}}{=}|\tM_{3,j}|- |\tM_{3,j}'|=e_{\cF_j}(U_3,U_2)
		\end{align*}
		and
		\begin{align*}
			e_{\cF_j}(U_4,U_3)=|\tM_{4,j}'|\stackrel{\text{\cref{lm:movedegree2}\cref{lm:movedegree2-size2}}}{=} |\tM_{2,j}'|=e_{\cF_j}(U_2, U_1).
		\end{align*}
		Thus, \cref{def:feasible-backward} holds.
		By \cref{lm:movedegree2}\cref{lm:movedegree2-degree2}, $\tM_{i,j'}'$ and $\tM_{i,j'}\setminus \tM_{i,j'}'$ are matchings for each $i\in [4]$ and $j'\in [r]$. 
		Therefore, $\Delta^0(\cF_{j'})\leq 1$ for each $j'\in [2r]$. In particular, \cref{def:feasible-degV'} holds for $\cF_j$.
		Moreover, \cref{lm:backwardall-V**} implies that each $v\in U^{1-\gamma}(T)$ satisfies $d_{\cF_{j'}}^+(v)=1=d_{\cF_{j'}}^-(v)$ for each $j'\in [2r]$. In particular, \cref{def:feasible-degV*} holds for $\cF_j$. For \cref{def:feasible-cycle}, suppose that $C$ is a cycle in $\cF_j$. By \cref{lm:movedegree2}\cref{lm:movedegree2-A2}, there exists $i\in \{2,4\}$ such that all the edges in $M_{i,j}'$ are incident to $W_i^2$. Since $C\subseteq D\subseteq \overleftarrow{T}_\cU$, there exists $e\in E_C(U_i,U_{i-1})\subseteq \tM_{i,j}'$.
		Denote by $e'$ the edge of $M_{i,j}'$ which witnesses that $e\in \tM_{i,j}'$. By assumption, $e'$ is incident to $W_i^2$, say $e'=w^2v$ for some $w^2\in W_i^2$ (similar arguments hold if the ending point of $e'$ is in $W_i^2$). By \cref{claim:backwardallH-nonexceptional}, we have $v\in U_{i-1}\setminus W_i$ and so, by construction, $e=wv$ with $w\in W_i\cap U_i$ and $v\in N_i^2(w)$.
		Therefore, \cref{eq:backwardall-W,eq:backwardall-U*} imply that $e$ is a backward edge from $U^{\gamma}(T)\cap U^*$ to $V(T)\setminus U^*$. Thus, $e$ is a $(\gamma, T)$-placeholder and so \cref{def:feasible-cycle} holds. Therefore, $\cF_j$ is a $(\gamma, T)$-pseudo-feasible system and, by similar arguments, $\cF_{r+j}$ is also a $(\gamma, T)$-pseudo-feasible system.\qedhere
	\end{steps}
\end{proof}

\section{Constructing a few special (pseudo)-feasible systems: proofs of Lemmas~\ref{lm:exceptionalfeasible}~and~\ref{lm:forwardUstar}}\label{sec:specialpseudofeasible}

	\onlyinsubfile{
		\setcounter{section}{11}
		\section{Constructing a few special (pseudo)-feasible systems: proof of Lemmas \ref{lm:exceptionalfeasible} and \ref{lm:forwardUstar}}}

Finally, we show how to construct a few (pseudo)-feasible systems which satisfy some special properties: in \cref{lm:exceptionalfeasible}, we construct a few feasible systems out of prescribed sets of edges and, in \cref{lm:forwardUstar}, we construct a few pseudo-feasible systems which incorporate a given set of exceptional edges.

\subsection{Proof overview}\label{sec:sketch-specialpseudofeasible}
Each (pseudo)-feasible system $\cF$ in \cref{lm:exceptionalfeasible,lm:forwardUstar} will be constructed using the following approach.
Let $T$ be a regular bipartite tournament. Let $\cU=(U_1, \dots, U_4)$ be an optimal $(\varepsilon,4)$-partition for $T$ and let $U^*$ be an $(\varepsilon, \cU)$-exceptional set for $T$.

\begin{steps}
	\item \textbf{Selecting the forward edges which are incident to $U^{1-\gamma}(T)$.}\label{step:sketch-forwardU**}
	First, we fix the forward edges incident to $U^{1-\gamma}(T)$ that we want $\cF$ to cover. We make sure that these edges form a linear forest $\cF^1$.
	(Note that this step will be void in the proof of \cref{lm:exceptionalfeasible}, all these forward edges will be covered in \cref{lm:forwardUstar}.)
	
	\item \textbf{Covering $U^{1-\gamma}(T)$.}\label{step:sketch-coverU**}
	Then, to ensure that \cref{def:feasible-degV*} is satisfied, we add backward edges to $\cF^1$ to cover all the uncovered vertices in $U^{1-\gamma}(T)$. 
	Since the vertices in $U^{1-\gamma}(T)$ have, by definition, very high backward degree, we can do so greedily and in such a way that we still have a linear forest $\cF^2$.
	
	\item \textbf{Balancing the number of backward edges.}\label{step:sketch-backward} 
	Next, we construct a linear forest $\cF^3$ by adding to $\cF^2$ precisely $m^{i\downarrow}\coloneqq e_{\cF^2}(U_{i-2},U_{i-3})$ additional backward edges of $T(U_i,U_{i-1})$ for each $i\in [4]$ (where the superscript $i$ is taken modulo $4$). Then, $e_{\cF^3}(U_i, U_{i-1})=m^{(i+2)\downarrow}+m^{i\downarrow}$ for each $i\in [4]$ and so \cref{def:feasible-backward} holds.
	
	These new edges will be selected using K\"{o}nig's theorem as follows. First, we find a subdigraph $H\subseteq \overleftarrow{T}_\cU-U^{1-\gamma}(T)$ which contains many edges but has small maximum degree. (In practice, $H$ is already given in \cref{lm:exceptionalfeasible} (see \cref{lm:main}\cref{lm:main-H-degree,lm:main-H-edges}). For \cref{lm:forwardUstar}, $H$ will be constructed using \cref{lm:optimalH}.)
	Then, we can use \cref{prop:Koniglarge} to find, for each $i\in [4]$, a matching of size $m^{i\downarrow}$ in $H(U_i, U_{i-1})$ which avoids all the vertices in $\cF^2$.
	
	\item \textbf{Adding forward edges incident to $U^*\setminus U^{1-\gamma}(T)$.}\label{step:sketch-forwardU*}
	Finally, we add to $\cF^3$ a few extra forward edges incident to $U^*\setminus U^{1-\gamma}(T)$. We do this in such a way that the resulting digraph $\cF$ is still a linear forest.
	
	For \cref{lm:exceptionalfeasible}, this is necessary because we want $\cF$ to be a feasible system and so $U^*$ needs to be entirely covered by in- and outedges (see \cref{def:feasible-exceptional}). This can be done greedily since the vertices in $U^*\setminus U^{1-\gamma}(T)$ have high forward degree (see \cref{lm:main}\cref{lm:main-H-degreeV*}).
	
	For \cref{lm:forwardUstar}, we only need a pseudo-feasible system and so $U^*$ need not be entirely covered (see \cref{def:feasible-degV*}). However, recall that we will need to incorporate all the forward edges of $T$ which lie inside $U^*$ (see \cref{lm:forwardUstar}\cref{lm:forwardUstar-forwardV*}). We will thus add a few of these to $\cF^3$ at this stage (and distribute the remaining such edges to the other pseudo-feasible systems).
\end{steps}

In practice, all the (pseudo)-feasible systems in \cref{lm:exceptionalfeasible,lm:forwardUstar} will be constructed in parallel. In particular, for \cref{lm:forwardUstar}, we will decompose all the forward edges of $T$ which are incident to $U^{1-\gamma}(T)$ at the start of the proof (this corresponds to \cref{step:sketch-forwardU**} above) and the other exceptional forward edges of $T$ will be distributed greedily at the end of the proof (this corresponds to \cref{step:sketch-forwardU*}).

\subsection{Selecting backward edges}

\Cref{step:sketch-backward,step:sketch-coverU**} from \cref{sec:sketch-specialpseudofeasible} are combined into the following lemma.
Let $T$ be a regular bipartite tournament. Let $\cU=(U_1, \dots, U_4)$ be an optimal $(\varepsilon,4)$-partition for $T$ and let $U^*$ be an $(\varepsilon, \cU)$-exceptional set for $T$.
Let $H\subseteq \overleftarrow{T}_\cU$ contain many well distributed backward edges. 
Roughly speaking, \cref{lm:backwardmatchings} states that $H$ contains edge-disjoint linear forests $\cF_1, \dots, \cF_\ell$, where each $\cF_j(U_i,U_{i-1})$ covers all vertices of $U_i^{1-\gamma}(T)\cup U_{i-1}^{1-\gamma}(T)$ apart from those in a given prescribed set $S_j^{i\downarrow}$ of vertices to avoid and contains a prescribed number $m_j^{i\downarrow}$ of additional edges (see \cref{lm:backwardmatchings}\cref{lm:backwardmatchings-matchings}). 
Moreover, these linear forests can be constructed in such a way that every vertex of $V(T)\setminus U^{1-\gamma}(T)$ is not covered by too many of the linear forests (see \cref{lm:backwardmatchings}\cref{lm:backwardmatchings-degree}) and is adjacent to at most one edge in each linear forest (see \cref{lm:backwardmatchings}\cref{lm:backwardmatchings-degree1}).

In our applications, the sets $S_j^{i\downarrow}$ of vertices to avoid will consist of the vertices which are already covered at the end of \cref{step:sketch-forwardU**} from \cref{sec:sketch-specialpseudofeasible} and the constants $m_j^{i\downarrow}$ will be chosen as described in \cref{step:sketch-backward} from \cref{sec:sketch-specialpseudofeasible} in order to balance the number of backward edges in each pair of the blow-up $C_4$.

As explained in \cref{sec:sketch-specialpseudofeasible}, the backward edges will be chosen in two stages, depending on whether they are incident to $U^{1-\gamma}(T)$ or not. However, we will swap the order of \cref{step:sketch-coverU**,step:sketch-backward}. That is, we will select the backward edges incident to $U^{1-\gamma}(T)$ only after the other backward edges have been selected. This is because it is much easier to select backward edges incident to $U^{1-\gamma}(T)$ (the vertices in $U^{1-\gamma}(T)$ have high backward degree and so can be covered greedily).

\begin{lm}[Selecting backward edges]\label{lm:backwardmatchings}
	Let $0<\frac{1}{n}\ll \varepsilon \ll \gamma\ll 1$ and $\ell \leq \gamma n$. Let $T$ be a regular bipartite tournament on $4n$ vertices and $\cU=(U_1, \dots, U_4)$ be an $(\varepsilon,4)$-partition for $T$.
	Suppose that $H\subseteq \overleftarrow{T}$ satisfies the following.
	\begin{enumerate}
		\item For each $v\in U^{1-\gamma}(T)$, $d_H^\pm(v)\geq 2\gamma n$.\label{lm:backwardmatchings-H-degreeV**}
		\item For each $v\in V(T)\setminus U^{1-\gamma}(T)$, $d_H^\pm(v)\leq 2\gamma n$.\label{lm:backwardmatchings-H-degree}
		\item For each $i\in [4]$, $e_{H-U^{1-\gamma}(T)}(U_i, U_{i-1})\geq 109\gamma n|U_{i-2}^{1-\gamma}(T)\cup U_{i-3}^{1-\gamma}(T)|$.\label{lm:backwardmatchings-H-edges}
	\end{enumerate}
	For each $i\in [4]$ and $j\in [\ell]$, let $S_j^{i\downarrow}\subseteq U_i\cup U_{i-1}$  and $m_j^{i\downarrow}\in \mathbb{N}$ satisfy the following.
	\begin{enumerate}[resume]
		\item $|S_j^{i\downarrow}\setminus U^{1-\gamma}(T)|\leq |U_{i-2}^{1-\gamma}(T)\cup U_{i-3}^{1-\gamma}(T)|$.\label{lm:backwardmatchings-sizeS}
		\item $m_j^{i\downarrow}\leq |U_{i-2}^{1-\gamma}(T)\cup U_{i-3}^{1-\gamma}(T)|$.\label{lm:backwardmatchings-m}
	\end{enumerate}
	Then, $H$ contains edge-disjoint linear forests $\cF_1, \dots, \cF_\ell$ such that the following hold, where $\cF\coloneqq \bigcup_{i\in [\ell]}\cF_i$.
	\begin{enumerate}[label=\rm(\alph*)]
		\item For each $j\in [\ell]$, $\cF_j$ consists of\label{lm:backwardmatchings-matchings}
		\begin{enumerate}[label=\rm(\greek*), ref=\rm(\alph{enumi}.\greek*)]
			\item a matching of $H_j^{i\downarrow}\coloneqq H(U_i\setminus (U^{1-\gamma}(T)\cup S_j^{i\downarrow}), U_{i-1}\setminus (U^{1-\gamma}(T)\cup S_j^{i\downarrow}))$ of size $m_j^{i\downarrow}$ for each $i\in [4]$;\label{lm:backwardmatchings-matchings-V'}
			\item a matching of $H(U_i^{1-\gamma}(T)\setminus S_j^{i\downarrow}, U_{i-1}\setminus (U^{1-\gamma}(T)\cup S_j^{i\downarrow}))$ of size $|U_i^{1-\gamma}(T)\setminus S_j^{i\downarrow}|$ for each $i\in [4]$; and\label{lm:backwardmatchings-matchings-V*}
			\item a matching of $H(U_i\setminus (U^{1-\gamma}(T)\cup S_j^{i\downarrow}),U_{i-1}^{1-\gamma}(T)\setminus S_j^{i\downarrow})$ of size $|U_{i-1}^{1-\gamma}(T)\setminus S_j^{i\downarrow}|$ for each $i\in [4]$.\label{lm:backwardmatchings-matchings-V*'}
		\end{enumerate}
		\item For each $v\in V(T)\setminus U^{1-\gamma}(T)$, we have $d_{\cF}(v)\leq \frac{\gamma n}{6}$.\label{lm:backwardmatchings-degree}
		\item For each $i\in [\ell]$ and $v\in V(T)\setminus U^{1-\gamma}(T)$, we have $d_{\cF_i}(v)\leq 1$. \label{lm:backwardmatchings-degree1}
	\end{enumerate}
\end{lm}

\begin{proof}
	We will first select the edges which are not incident to $U^{1-\gamma}(T)$ (i.e.\ those in \cref{lm:backwardmatchings-matchings-V'}) as follows.
	
	\begin{claim}\label{claim:backwardmatchings}
		$H$ contains edge-disjoint linear forests $\cQ_1, \dots, \cQ_\ell$ such that the following hold, where $\cQ\coloneqq \bigcup_{i\in [\ell]}\cQ_i$.
		\begin{enumerate}[label=\rm(\alph*$'$)]
			\item For each $j\in [\ell]$, $\cQ_j$ consists of
			a matching of $H_j^{i\downarrow}$ of size $m_j^{i\downarrow}$ for each $i\in [4]$.\label{lm:backwardmatchings-matchings'}
			\newcounter{backwardmatchings}
			\setcounter{backwardmatchings}{\value{enumi}}
			\item For each $v\in V(T)\setminus U^{1-\gamma}(T)$, we have $d_\cQ(v)\leq \frac{\gamma n}{6}$.\label{lm:backwardmatchings-degree'}
			\item For each $i\in [\ell]$ and $v\in V(T)\setminus U^{1-\gamma}(T)$, we have $d_{\cQ_i}(v)\leq 1$.\label{lm:backwardmatchings-degree1'}
		\end{enumerate}
	\end{claim}
	
	First, we assume that \cref{claim:backwardmatchings} holds and select the edges incident to $U^{1-\gamma}(T)$ (i.e.\ those in \cref{lm:backwardmatchings-matchings-V*,lm:backwardmatchings-matchings-V*'}) using \cref{lm:extendlinforest} as follows. Let $\cQ_1, \dots, \cQ_\ell$ be the edge-disjoint linear forests obtained by applying \cref{claim:backwardmatchings} and denote $\cQ\coloneqq \bigcup_{i\in [\ell]}\cQ_i$.
	If $U^{1-\gamma}(T)=\emptyset$, then let $\cF_i\coloneqq \cQ_i$ for each $i\in [\ell]$ and observe that \cref{lm:backwardmatchings-matchings,lm:backwardmatchings-degree,lm:backwardmatchings-degree1} follow from \cref{lm:backwardmatchings-matchings',lm:backwardmatchings-degree',lm:backwardmatchings-degree1'}. We may therefore assume that $U^{1-\gamma}(T)\neq \emptyset$%
		\COMMENT{Needed for $N\geq 1$.}.
	
	First, note that each $i\in [\ell]$ satisfies%
		\COMMENT{$e(\cQ_i)\leq 2|U^{1-\gamma}(T)|$ because each $U_j^{1-\gamma}(T)$ appears twice in \cref{lm:backwardmatchings-m}.}
	\begin{equation}\label{eq:backwardmatchings-Q}
		|\cQ_i| \leq 2e(\cQ_i)\stackrel{\text{\cref{lm:backwardmatchings-m},\cref{lm:backwardmatchings-matchings'}}}{\leq}4|U^{1-\gamma}(T)|\stackrel{\text{\cref{fact:U},\cref{def:ES}}}{\leq}16\varepsilon n.
	\end{equation}
	Let $X$ be the set of vertices $v\in V(T)\setminus U^{1-\gamma}(T)$ such that $d_{\cQ}(v)\geq \frac{\gamma n}{7}$ (so $X$ is the set of vertices which are already covered by many of the linear forests $\cQ_1, \dots, \cQ_\ell$).
	Note that
	\begin{equation}\label{eq:backwardmatchings-X}
		|X|\stackrel{\text{\cref{eq:backwardmatchings-Q},\cref{lm:backwardmatchings-degree1'}}}{\leq} \frac{7\cdot 16\varepsilon n\ell}{\gamma n}\leq 112\varepsilon n.
	\end{equation}
	For each $j\in [\ell]$, denote by
	\begin{equation*}
		S_j\coloneqq \bigcup_{i\in [4]}(S_j^{i\downarrow}\setminus U^{1-\gamma}(T))
	\end{equation*}
	the set of vertices outside $U^{1-\gamma}(T)$ that need to be avoided by $\cF_j$
	and observe that
	\begin{align}
		|S_j|&\leq \sum_{i\in [4]}|S_j^{i\downarrow}\setminus U^{1-\gamma}(T)| \stackrel{\text{\cref{lm:backwardmatchings-sizeS}}}{\leq} 2|U^{1-\gamma}(T)|
		\stackrel{\text{\cref{fact:U},\cref{def:ES}}}{\leq} 8\varepsilon n.\label{eq:backwardmatchings-S}
	\end{align}
	Let $Y$ be the set of vertices $v\in V(T)\setminus U^{1-\gamma}(T)$ for which there exist at least $\frac{\gamma n}{7}$ indices $i\in [\ell]$ such that $v\in S_i$ (so $Y$ is the set vertices which need to be avoided by many of the linear forests $\cF_1, \dots, \cF_\ell$).
	Note that 
	\begin{equation}\label{eq:backwardmatchings-Y}
		|Y|\stackrel{\text{\cref{eq:backwardmatchings-S}}}{\leq}\frac{7\cdot 8\varepsilon n \ell}{\gamma n}\leq 56\varepsilon n.
	\end{equation}
	Let $A\coloneqq U^{1-\gamma}(T)$ and $B\coloneqq V(T)\setminus (A\cup X\cup Y)$. Let $D$ be the bipartite digraph on vertex classes $A$ and $B$ induced by $H$%
		\COMMENT{I.e.\ $E(D)$ consists of all the edges of $H$ with one endpoint in $A$ and one endpoint in $B$.}.
	For each $j\in [\ell]$, let $S_j'\coloneqq V(E(\cQ_j))\cup S_j$ and define 
	\begin{equation*}
		T_j^+\coloneqq \bigcup_{i\in [4]} U_i^{1-\gamma}(T)\setminus S_j^{i\downarrow} \quad \text{and}\quad T_j^-\coloneqq \bigcup_{i\in [4]} U_{i-1}^{1-\gamma}(T)\setminus S_j^{i\downarrow}.
	\end{equation*}
	By \cref{lm:backwardmatchings-matchings'}, $V(\cQ_i)\cap  U^{1-\gamma}(T)=\emptyset$ for each $i\in [\ell]$. Thus, note for later that both
	\begin{equation}\label{eq:backwardmatchings-ST}
		T_i^\pm\cap V(\cQ_i)=\emptyset \quad \text{and} \quad V(E(\cQ_i))\cap B\subseteq S_i'
	\end{equation}
	for each $i\in [\ell]$.
	Define $N\coloneqq 2|U^{1-\gamma}(T)|=2|A|$.
	For each $v\in B$, let $n_v$ denote the number of indices $i\in [\ell]$ such that $v\in S_i'$.
	
	We verify that \cref{eq:extendlinforest} holds with $T_1^+, \dots, T_\ell^+, T_1^-, \dots, T_\ell^-$, and $S_1', \dots, S_\ell'$ playing the roles of $S_1^+, \dots, S_\ell^+, S_1^-, \dots, S_\ell^-$, and $T_1, \dots, T_\ell$.
	For each $i\in [\ell]$, we have
	\begin{equation}\label{eq:backwardmatchings-T}
		|T_i^+|+|T_i^-|\leq 2|U^{1-\gamma}(T)|\stackrel{\text{\cref{fact:U},\cref{def:ES}}}{\leq} 8\varepsilon n
	\end{equation}
	and
	\begin{equation}\label{eq:backwardmatchings-S'}
	    |S_i'|\leq |\cQ_i|+|S_i|\stackrel{\text{\cref{eq:backwardmatchings-Q},\cref{eq:backwardmatchings-S}}}{\leq} 24\varepsilon n.
	\end{equation}
	By definition of $X$ and $Y$, each $v\in B\subseteq V(T)\setminus (U^{1-\gamma}(T)\cup X\cup Y)$ satisfies \begin{equation}\label{eq:backwardmatchings-nv}
		n_v\leq d_\cQ(v)+|\{i\in [\ell]\mid v\in S_i\}|\leq \frac{\gamma n}{7}+\frac{\gamma n}{7}=\frac{2\gamma n}{7}.
	\end{equation}
	Thus, \cref{eq:backwardmatchings-X,eq:backwardmatchings-Y,eq:backwardmatchings-nv,eq:backwardmatchings-T,eq:backwardmatchings-S'} imply that each $v\in U^{1-\gamma}(T)$ satisfies
	\begin{align*}\label{eq:backwardmatchings-deg}
		d_D^\pm(v)
		&\stackrel{\text{\eqmakebox[backwardmatchings]{\text{\cref{lm:backwardmatchings-H-degreeV**}}}}}{\geq} 2\gamma n-|U^{1-\gamma}(T)|-|X|-|Y|\\
		&\stackrel{\text{\eqmakebox[backwardmatchings]{\text{\cref{fact:U},\cref{def:ES}}}}}{\geq} \ell+2\max_{j\in [\ell]}(|T_j^+|+|T_j^-|+|S_j'|)
		+2(\max_{w\in B}n_w+N).
	\end{align*}
	Thus, \cref{eq:extendlinforest} holds with $T_1^+, \dots, T_\ell^+, T_1^-, \dots, T_\ell^-$, and $S_1', \dots, S_\ell'$ playing the roles of $S_1^+, \dots, S_\ell^+$, $S_1^-, \dots, S_\ell^-$, and $T_1, \dots, T_\ell$.
	
	Let $\cQ_1', \dots, \cQ_\ell'$ be the edge-disjoint linear forests obtained by applying \cref{lm:extendlinforest} with $T_1^+,\dots, T_\ell^+,T_1^-, \dots, T_\ell^-$, and $S_1', \dots, S_\ell'$ playing the roles of $S_1^+,\dots, S_\ell^+,S_1^-, \dots, S_\ell^-$, and $T_1, \dots$, $T_\ell$. For each $i\in [\ell]$, denote $\cF_i\coloneqq \cQ_i\cup \cQ_i'$. Let $\cF\coloneqq \bigcup_{i\in [\ell]}\cF_i$. We claim that $\cF_1,\dots, \cF_\ell$ are edge-disjoint linear forests which satisfy \cref{lm:backwardmatchings-matchings,lm:backwardmatchings-degree,lm:backwardmatchings-degree1}. By \cref{claim:backwardmatchings}, $\cQ_1, \dots, \cQ_\ell$ are edge-disjoint linear forests. By \cref{lm:backwardmatchings-matchings'}, their edges are not adjacent to $U^{1-\gamma}(T)$ and so $\cQ_1, \dots, \cQ_\ell$ are edge-disjoint from $D$. Therefore, \cref{eq:backwardmatchings-ST} and the ``in particular part'' of \cref{lm:extendlinforest} implies that $\cF_1, \dots, \cF_\ell$ are edge-disjoint linear forests. Moreover, \cref{lm:backwardmatchings-matchings} holds by \cref{lm:backwardmatchings-matchings'} and \cref{lm:extendlinforest}\cref{lm:extendlinforest-cover}. For \cref{lm:backwardmatchings-degree}, let $v\in V(T)\setminus U^{1-\gamma}(T)$. If $v\in X$, then recall that $X\cap V(D)=\emptyset$ and so \cref{lm:backwardmatchings-degree'} implies that $d_\cF(v)=d_\cQ(v)\leq\frac{\gamma n}{6}$, as desired. Otherwise, the definition of $X$ implies that
	\[d_\cF(v)\leq d_\cQ(v)+d_D(v)\leq \frac{\gamma n}{7}+2|U^{1-\gamma}(T)|\stackrel{\text{\cref{fact:U},\cref{def:ES}}}{\leq} \frac{\gamma n}{7}+8\varepsilon n\leq \frac{\gamma n}{6},\]
	so \cref{lm:backwardmatchings-degree} holds. For \cref{lm:backwardmatchings-degree1}, let $i\in [\ell]$ and $v\in V(T)\setminus U^{1-\gamma}(T)$. If $d_{\cQ_i}(v)=0$, then \cref{lm:extendlinforest}\cref{lm:extendlinforest-degree1} implies that $d_{\cF_i}(v)\leq 1$. Otherwise, \cref{lm:extendlinforest}\cref{lm:extendlinforest-cover} and the definition of $S_i'$ imply that $d_{\cQ_i'}(v)=0$ and so \cref{lm:backwardmatchings-degree1'} implies that $d_{\cF_i}(v)\leq 1$. Thus, \cref{lm:backwardmatchings-degree1} holds.
	
	\begin{proofclaim}[Proof of \cref{claim:backwardmatchings}]
		We will construct, for each $i\in [4]$ and $j\in [\ell]$, a matching $M_j^{i\downarrow}\subseteq H_j^{i\downarrow}$ of size $m_j^{i\downarrow}$. Then, we will let each $\cQ_j$ consist of the union $\bigcup_{i\in [4]}M_j^{i\downarrow}$. In this way, \cref{lm:backwardmatchings-matchings'} will be automatically satisfied. 
		
		These matchings will be constructed one by one using \cref{prop:Koniglarge} as follows.
		Suppose that we want to construct $M_j^{i\downarrow}$ for some $i\in [4]$ and $j\in [\ell]$, and suppose furthermore that we have already constructed $M_j^{i'\downarrow}$ for some $i'\in [4]\setminus \{i\}$.
		Then, in order to satisfy \cref{lm:backwardmatchings-degree1'}, we need to avoid the vertices in $V(M_j^{i'\downarrow})$. Thus, in order to minimise the number of vertices we have to avoid at each stage (and thus maximise the number of available edges for each matching), we will construct the matchings in ascending size order.
		
		Let $\emptyset \eqqcolon X_0\subsetneq X_1\subsetneq \dots \subsetneq X_{4\ell}\coloneqq [4]\times [\ell]$ be such that, for each $k\in [4\ell]$, $(i,j)\in X_k\setminus X_{k-1}$, and $(i',j')\in X_{k-1}$, we have $m_{j'}^{i'\downarrow}\leq m_j^{i\downarrow}$. Note that $|X_k\setminus X_{k-1}|=1$ for each $k\in [4\ell]$%
		\COMMENT{This holds since $|X_0|=0$, $|X_{4\ell}|=4\ell$, and $|X_{k'}|>|X_{k'-1}|$ for each $k'\in [4\ell]$.}.
		
		Suppose inductively that, for some $0\leq k\leq 4\ell$, we have constructed a set $\cM_k=\{M_j^{i\downarrow}\mid (i,j)\in X_k\}$ of edge-disjoint matchings such that the following hold.
		\begin{enumerate}[label=(\alph*$''$)]
			\item For each $(i,j)\in X_k$, $M_j^{i\downarrow}$ is a matching of $H_j^{i\downarrow}$ of size $m_j^{i\downarrow}$.\label{lm:backwardmatchings-matchings''}
			\item For each $v\in V(T)$, we have $\sum_{(i,j)\in X_k}d_{M_j^{i\downarrow}}(v)\leq \frac{\gamma n}{6}$.\label{lm:backwardmatchings-degree''}
			\item For each $j\in [\ell]$ and $v\in V(T)\setminus U^{1-\gamma}(T)$, we have $\sum_{i\colon(i,j)\in X_k} d_{M_j^{i\downarrow}}(v)\leq 1$.\label{lm:backwardmatchings-degree1''}
		\end{enumerate}
		First, suppose that $k=4\ell$.
		For each $j\in [\ell]$, let $\cQ_j\coloneqq \bigcup_{i\in [4]}M_j^{i\downarrow}$.
		Then, \cref{lm:backwardmatchings-matchings',lm:backwardmatchings-degree',lm:backwardmatchings-degree1'} follow from \cref{lm:backwardmatchings-matchings'',lm:backwardmatchings-degree'',lm:backwardmatchings-degree1''}.
		
		We may therefore assume that $k<4\ell$. Let $(i,j)\in X_{k+1}\setminus X_k$ and let $H'$ be obtained from $H(U_i\setminus U^{1-\gamma}(T), U_{i-1}\setminus U^{1-\gamma}(T))$ by deleting all the edges in $\bigcup \cM_k$.
		The matching $M_j^{i\downarrow}$ will be constructed in $H'$ using \cref{prop:Koniglarge}.
		By definition of $X_0, \dots, X_{4\ell}$, we have
		\begin{align}
			e(H')&\stackrel{\text{\eqmakebox[backwardmatchingsH']{\text{\cref{lm:backwardmatchings-matchings''}}}}}{=} e_{H-U^{1-\gamma}(T)}(U_i, U_{i-1})-\sum_{j'\colon (i,j')\in X_k}|M_{j'}^{i\downarrow}|\nonumber\\
			&\stackrel{\text{\eqmakebox[backwardmatchingsH']{\text{\cref{lm:backwardmatchings-matchings''}}}}}{=}
			e_{H-U^{1-\gamma}(T)}(U_i, U_{i-1})-\sum_{j'\colon (i,j')\in X_k}m_{j'}^{i\downarrow}\nonumber
			\geq e_{H-U^{1-\gamma}(T)}(U_i, U_{i-1})-\ell m_j^{i\downarrow}\nonumber\\
			&\stackrel{\text{\eqmakebox[backwardmatchingsH']{\text{\cref{lm:backwardmatchings-H-edges},\cref{lm:backwardmatchings-m}}}}}{\geq} 108\gamma n|U_{i-2}^{1-\gamma}(T)\cup U_{i-3}^{1-\gamma}(T)|.\label{eq:backwardmatchings-H'}
		\end{align}
		We will now list and count all the vertices that $M_j^{i\downarrow}$ needs to avoid.
		For \cref{lm:backwardmatchings-matchings''}, we will need to avoid the vertices in $S_j^{i\downarrow}\setminus U^{1-\gamma}(T)$, where
		\begin{equation}\label{eq:backwardmatchings-1}
			|S_j^{i\downarrow}\setminus U^{1-\gamma}(T)|\stackrel{\text{\cref{lm:backwardmatchings-sizeS}}}{\leq} |U_{i-2}^{1-\gamma}(T)\cup U_{i-3}^{1-\gamma}(T)|.
		\end{equation}
		(The vertices in $S_j^{i\downarrow}\cap U^{1-\gamma}(T)$ will automatically be avoided since $V(H')\cap U^{1-\gamma}(T)=\emptyset$.)
		For \cref{lm:backwardmatchings-degree''}, we will need to avoid the set of vertices $Y_j^{i\downarrow}$ of vertices $v\in V(H')$ for which $\sum_{(i',j')\in X_k}d_{M_{j'}^{i'\downarrow}}(v)=\left\lfloor\frac{\gamma n}{6}\right\rfloor$ where, by definition of $X_0, \dots, X_{4\ell}$,
		\begin{align}
			|Y_j^{i\downarrow}|&\stackrel{\text{\eqmakebox[backwardmatchingsY]{}}}{\leq} \frac{\sum_{(i',j')\in X_k}|V(M_{j'}^{i'\downarrow})|}{\left\lfloor\frac{\gamma n}{6}\right\rfloor}\stackrel{\text{\cref{lm:backwardmatchings-matchings''}}}{=} \frac{2\sum_{(i',j')\in X_k} m_{j'}^{i'\downarrow}}{\left\lfloor\frac{\gamma n}{6}\right\rfloor}
			\leq \frac{2\cdot 4\ell\cdot m_j^{i\downarrow}}{\left\lfloor\frac{\gamma n}{6}\right\rfloor}\nonumber\\ 
			&\stackrel{\text{\eqmakebox[backwardmatchingsY]{\text{\cref{lm:backwardmatchings-m}}}}}{\leq} 49 |U_{i-2}^{1-\gamma}(T)\cup U_{i-3}^{1-\gamma}(T)|.\label{eq:backwardmatchings-3}
		\end{align}
		Finally, for \cref{lm:backwardmatchings-degree1''}, we will need to avoid all the vertices in $Z_j^{i\downarrow}\coloneqq \bigcup_{i'\colon (i',j)\in X_k}(V(M_j^{i'\downarrow})\cap V(H'))$ where, by definition of $X_0, \dots, X_{4\ell}$,
		\begin{align}
		    |Z_j^{i\downarrow}|&\stackrel{\text{\eqmakebox[backwardmatchingsZ]{}}}{\leq} \sum_{i'\colon (i',j)\in X_k}	\left|V(M_j^{i'\downarrow})\cap (U_i\cup U_{i-1})\right|\stackrel{\text{\cref{lm:backwardmatchings-matchings''}}}{\leq} \sum_{i'\colon (i',j)\in X_k} m_j^{i'\downarrow}
			\leq 3m_j^{i\downarrow}\nonumber \\
			&\stackrel{\text{\eqmakebox[backwardmatchingsZ]{\text{\cref{lm:backwardmatchings-m}}}}}{\leq} 3|U_{i-2}^{1-\gamma}(T)\cup U_{i-3}^{1-\gamma}(T)|.\label{eq:backwardmatchings-2}
		\end{align}
		Let $\tS_j^{i\downarrow}$ consist of all the above mentioned vertices, that is, let
		\[\tS_j^{i\downarrow}\coloneqq (S_j^{i\downarrow}\setminus U^{1-\gamma}(T))\cup Y_j^{i\downarrow}\cup Z_j^{i\downarrow}.\]
		Let $\tH_j^{i\downarrow}\coloneqq H'-\tS_j^{i\downarrow}\subseteq H_j^{i\downarrow}$.
		By \cref{eq:backwardmatchings-1,eq:backwardmatchings-2,eq:backwardmatchings-3}, we have $|\tS_j^{i\downarrow}|\leq 53|U_{i-2}^{1-\gamma}(T)\cup U_{i-3}^{1-\gamma}(T)|$ and so
		\begin{align*}
			e(\tH_j^{i\downarrow})\geq e(H')-\Delta(H')|\tS_j^{i\downarrow}|\stackrel{\text{\cref{lm:backwardmatchings-H-degree},\cref{eq:backwardmatchings-H'}}}{\geq} 2\gamma n|U_{i-2}^{1-\gamma}(T)\cup U_{i-3}^{1-\gamma}(T)|
			\stackrel{\text{\cref{lm:backwardmatchings-H-degree},\cref{lm:backwardmatchings-m}}}{\geq} \Delta(\tH_j^{i\downarrow})m_j^{i\downarrow}.
		\end{align*}
		Thus, \cref{prop:Koniglarge} (applied with the undirected graph underlying $\tH_j^{i\downarrow}$ playing the role of $G$) implies that $\tH_j^{i\downarrow}$ contains a matching $M_j^{i\downarrow}$ of size $m_j^{i\downarrow}$.
		One can easily verify that \cref{lm:backwardmatchings-matchings'',lm:backwardmatchings-degree'',lm:backwardmatchings-degree1''} hold with $k+1$ playing the role of $k$. 
	\end{proofclaim}
	This completes the proof of \cref{lm:backwardmatchings}.
\end{proof}

\subsection{Proofs of Lemmas \ref{lm:exceptionalfeasible} and \ref{lm:forwardUstar}}

We are now ready to prove \cref{lm:forwardUstar,lm:exceptionalfeasible}, using the arguments described in \cref{sec:sketch-specialpseudofeasible}.

\begin{proof}[Proof of \cref{lm:exceptionalfeasible}]
    Denote $t_0\coloneqq 0\eqqcolon s_0$.
	Suppose inductively that, for some $0\leq k\leq s$, we have constructed
	edge-disjoint feasible systems $\cF_1, \dots, \cF_{t_k+s_k}$ such that the following hold.
	\begin{enumerate}[label=(\greek*)]
		\item For each $i\in [4]$ and $j\in [t_k+s_k]$, we have $e_{\cF_j-U^{1-\gamma}(T)}(U_i, U_{i-1}) = |U_{i-2}^{1-\gamma}(T)\cup U_{i-3}^{1-\gamma}(T)|$.\label{lm:exceptionalfeasible-IH-edgesV'}
		\item For each $i\in [t_k+s_k]$, $e(\cF_i)\leq \sqrt{\varepsilon}n$.\label{lm:exceptionalfeasible-IH-size}
		\item For each $i\in [t_k+s_k]$ and $v\in V(T)\setminus U^*$, $d_{\cF_i}(v)\leq 1$. \label{lm:exceptionalfeasible-IH-V'}
		\item For each $i\in [k]$ and $j\in [s_i]$, $\cF_{t_i+j}\subseteq H_i$.\label{lm:exceptionalfeasible-IH-H}
	\end{enumerate}
	First, suppose that $k=s$.
	Then, \cref{lm:exceptionalfeasible-U',lm:exceptionalfeasible-size,lm:exceptionalfeasible-H} follow from \cref{lm:exceptionalfeasible-IH-edgesV',lm:exceptionalfeasible-IH-H,lm:exceptionalfeasible-IH-size}, respectively.
	Let $i\in [t]$. By \cref{def:feasible-exceptional}, $U^*\subseteq V^0(\cF_i)$ and, by \cref{lm:exceptionalfeasible-IH-V'}, $V^0(\cF_i)\subseteq U^*$. Therefore, \cref{lm:exceptionalfeasible-feasible} holds.
	
	We may therefore assume that $k<s$. Note that $t_k+s_k=t_{k+1}$. We will now construct $\cF_{t_{k+1}+1}, \dots, \cF_{t_{k+1}+s_{k+1}}$.			
	Let $H\coloneqq H_{k+1}\setminus (H_{k+1}[U^{1-\gamma}(T)]\cup \bigcup_{i\in [t_{k+1}]}\cF_i)$.
	
	\begin{steps}
		\item \textbf{Selecting backward edges.}
		First, we show that \cref{lm:backwardmatchings}\cref{lm:backwardmatchings-H-degree,lm:backwardmatchings-H-degreeV**,lm:backwardmatchings-H-edges} hold with $\overleftarrow{H}_\cU$ playing the role of $H$.
		Note that \cref{lm:backwardmatchings}\cref{lm:backwardmatchings-H-degree} follows immediately from \cref{lm:main}\cref{lm:main-H-degree}.
		For each $v\in U^{1-\gamma}(T)$, we have
		\begin{align*}
			\overleftarrow{d}_{H,\cU}^\pm(v)&\stackrel{\text{\eqmakebox[exceptionalfeasible]{\text{\cref{def:feasible-linforest}}}}}{\geq} \overleftarrow{d}_{H_{k+1},\cU}^\pm(v)-|\overleftarrow{N}_{H_{k+1},\cU}^\pm(v)\cap U^{1-\gamma}(T)|-t_{k+1}\\
			&\stackrel{\text{\eqmakebox[exceptionalfeasible]{\text{\cref{lm:main}\cref{lm:main-H-degreeV**}}}}}{\geq}3\gamma n-|U^{1-\gamma}(T)|-\eta n\stackrel{\text{\cref{fact:U},\cref{def:ES}}}{\geq} 2\gamma n.
		\end{align*}
		Thus, \cref{lm:backwardmatchings}\cref{lm:backwardmatchings-H-degreeV**} holds.
		For each $i\in [4]$, we have
		\begin{align*}
			e_{H-U^{1-\gamma}(T)}(U_i,U_{i-1})
			&\stackrel{\text{\eqmakebox[exceptionalfeasible2]{}}}{=} e_{H_{k+1}-U^{1-\gamma}(T)}(U_i,U_{i-1})\\
			& \eqmakebox[exceptionalfeasible2]{} \qquad -\sum_{j\in [t_{k+1}]}
			e_{\cF_j-U^{1-\gamma}(T)}(U_i,U_{i-1})\\
			&\stackrel{\text{\eqmakebox[exceptionalfeasible2]{\text{\cref{lm:exceptionalfeasible-IH-edgesV'},\cref{lm:main}\cref{lm:main-H-edges}}}}}{\geq}109\gamma n|U_{i-2}^{1-\gamma}(T)\cup U_{i-3}^{1-\gamma}(T)|.
		\end{align*}
		Thus, \cref{lm:backwardmatchings}\cref{lm:backwardmatchings-H-edges} is satisfied.
		
		For each $i\in [4]$ and $j\in [s_{k+1}]$, let $S_j^{i\downarrow}\coloneqq \emptyset$ and 
		\begin{equation}\label{eq:exceptionalfeasible-m}
			m_j^{i\downarrow}\coloneqq |U_{i-2}^{1-\gamma}(T)\cup U_{i-3}^{1-\gamma}(T)|\stackrel{\text{\cref{fact:U},\cref{def:ES}}}{\leq}2\varepsilon n.
		\end{equation}
		Then, \cref{lm:backwardmatchings}\cref{lm:backwardmatchings-m,lm:backwardmatchings-sizeS} hold with $s_{k+1}$ playing the role of $\ell$.
		Let $\cF_{t_{k+1}+1}', \dots, \cF_{t_{k+1}+s_{k+1}}'$ be the edge-disjoint linear forests obtained by applying \cref{lm:backwardmatchings} with $\overleftarrow{H}_\cU$ and $s_{k+1}$ playing the roles of $H$ and $\ell$.
		
		\item \textbf{Covering $U^*$.}\label{step:exceptionalfeasible-coverU*}
		We now add forward edges to $\cF_{t_{k+1}+1}', \dots, \cF_{t_{k+1}+s_{k+1}}'$ to ensure that \cref{def:feasible-exceptional} is satisfied. We will use \cref{lm:extendlinforest} as follows.
		Let $A\coloneqq U^*$ and $B\coloneqq V(T)\setminus U^*$. Let $D$ be the bipartite digraph on vertex classes $A$ and $B$ induced by $\overrightarrow{H}_{\cU}$%
		    	\COMMENT{I.e.\ $E(D)$ consists of all the edges of $\overrightarrow{H}_\cU$ which have one endpoint in $A$ and one endpoint in $B$.}.
		For each $i\in [s_{k+1}]$, let $S_i^\pm$ be the set of vertices $v\in U^*$ which satisfy $d_{\cF_{t_{k+1}+i}'}^\pm(v)=0$ (so $S_i^+$ and $S_i^-$ list the vertices in $U^*$ which are not yet covered with an out/inedge in $\cF_{t_{k+1}+i}'$) and define $T_i\coloneqq V(E(\cF_{t_{k+1}+i}'))\cap B$. Note for later that both 
		\begin{equation}\label{eq:exceptionalfeasible-ST}
		    S_i^\pm\cap V(\cF_{t_{k+1}+i}')\subseteq V^\mp(\cF_{t_{k+1}+i}') \quad \text{and} \quad V(E(\cF_{t_{k+1}+i}')\cap B\subseteq T_i
		\end{equation}
		for each $i\in [s_{k+1}]$.
		Define $N\coloneqq 2|U^*|=2|A|$. For each $v\in B$, let $n_v$ denote the number of indices $i\in [s_{k+1}]$ such that $v\in T_i$.
		
		We verify that \cref{eq:extendlinforest} holds with $s_{k+1}$ playing the role of $\ell$. For each $i\in [s_{k+1}]$, we have
		\begin{equation}\label{eq:exceptionalfeasible-S}
		    |S_i^+|+|S_i^-|\leq 2|U^*|\stackrel{\text{\cref{def:ES-size}}}{\leq} 8\varepsilon n
		\end{equation}
		and
		\begin{align}\label{eq:exceptionalfeasible-T}
		    |T_i|\leq |V(E(\cF_{t_{k+1}+i}'))|\stackrel{\text{\cref{lm:backwardmatchings}\cref{lm:backwardmatchings-matchings},\cref{eq:exceptionalfeasible-m}}}{\leq} 4|U^{1-\gamma}(T)|\stackrel{\text{\cref{fact:U},\cref{def:ES}}}{\leq} 16\varepsilon n.
		\end{align}
		Moreover, each $v\in B=V(T)\setminus U^*$ satisfies
		\begin{equation}\label{eq:exceptionalfeasible-nv}
		    n_v\leq \sum_{i\in [s_{k+1}]}d_{\cF_{t_{k+1}+i}'}(v)\stackrel{\text{\cref{lm:backwardmatchings}\cref{lm:backwardmatchings-matchings}}}{\leq} \overleftarrow{d}_{H, \cU}(v)\stackrel{\text{\cref{def:ES-backward}}}{\leq} 2\varepsilon n.
		\end{equation}
		Therefore, each $v\in U^*\setminus U^{1-\gamma}(T)$ satisfies
		\begin{align}
			d_D^\pm(v)&\stackrel{\text{\eqmakebox[exceptionalfeasibledeg]{}}}{=}\overrightarrow{d}_{H_{k+1},\cU}^\pm(v)-|\overrightarrow{N}_{H_{k+1},\cU}^\pm(v)\cap U^*|-
			\sum_{j\in [t_{k+1}]}\overrightarrow{d}_{\cF_j}^\pm(v)\nonumber\\
			&\stackrel{\text{\eqmakebox[exceptionalfeasibledeg]{\text{\cref{lm:main}\cref{lm:main-H-degreeV*},\cref{def:feasible-linforest}}}}}{\geq} \gamma^2n-|U^*|-t_{k+1}
			\stackrel{\text{\cref{def:ES-size}}}{\geq} \gamma^3 n\nonumber\\
			&\stackrel{\text{\eqmakebox[exceptionalfeasibledeg]{\text{\cref{def:ES-size},\cref{eq:exceptionalfeasible-S,eq:exceptionalfeasible-T,eq:exceptionalfeasible-nv}}}}}{\geq} s_{k+1}+2\max_{i\in [s_{k+1}]}(|S_i^+|+|S_i^-|+|T_i|)\nonumber\\
			&\eqmakebox[exceptionalfeasibledeg]{} \qquad\qquad\qquad+2(\max_{w\in B}n_w+N).\label{eq:exceptionalfeasible-deg}
		\end{align}
		By \cref{lm:backwardmatchings}\cref{lm:backwardmatchings-matchings-V*,lm:backwardmatchings-matchings-V*'}, each vertex in $U^{1-\gamma}(T)$ is already covered with both an in- and an outedge in each of the linear forests $\cF_{t_{k+1}+1}', \dots, \cF_{t_{k+1}+s_{k+1}}'$, so $S_i^+\cup S_i^-\subseteq U^*\setminus U^{1-\gamma}(T)$ for each $i\in [s_{k+1}]$. Thus, \cref{eq:extendlinforest} follows from \cref{eq:exceptionalfeasible-deg}.
		
		Let $\cQ_1, \dots, \cQ_{s_{k+1}}$ be the edge-disjoint linear forests obtained by applying \cref{lm:extendlinforest} with $s_{k+1}$ playing the role of $\ell$. For each $i\in [s_{k+1}]$, denote $\cF_{t_{k+1}+i}\coloneqq \cF_{t_{k+1}+i}'\cup \cQ_i$.
		
		\item \textbf{Verifying \cref{lm:exceptionalfeasible-IH-H,lm:exceptionalfeasible-IH-V',lm:exceptionalfeasible-IH-edgesV',lm:exceptionalfeasible-IH-size}.}
		We claim that $\cF_{t_{k+1}+1}, \dots, \cF_{t_{k+1}+s_{k+1}}$ are edge-disjoint feasible systems and that \cref{lm:exceptionalfeasible-IH-H,lm:exceptionalfeasible-IH-V',lm:exceptionalfeasible-IH-edgesV',lm:exceptionalfeasible-IH-size} hold with $k+1$ playing the role of $k$.
		By \cref{lm:backwardmatchings}, $\cF_{t_{k+1}+1}', \dots$, $\cF_{t_{k+1}+s_{k+1}}'$ are edge-disjoint linear forests and, by \cref{lm:backwardmatchings}\cref{lm:backwardmatchings-matchings}, $\cF_{t_{k+1}+1}', \dots, \cF_{t_{k+1}+s_{k+1}}'$ consist of backward edges and so are edge-disjoint from $D$. Thus, \cref{eq:exceptionalfeasible-ST} and the ``in particular part" of \cref{lm:extendlinforest} imply that $\cF_{t_{k+1}+1}, \dots, \cF_{t_{k+1}+s_{k+1}}$ are edge-disjoint linear forests.
		Moreover, \cref{lm:exceptionalfeasible-IH-edgesV'} holds by \cref{lm:backwardmatchings}\cref{lm:backwardmatchings-matchings}, \cref{eq:exceptionalfeasible-m}, and the fact that \cref{step:exceptionalfeasible-coverU*} involves only forward edges.
		Note that each $j\in [s_{k+1}]$ satisfies
		\begin{align*}
			e(\cF_{t_{k+1}+j})
			&\stackrel{\text{\eqmakebox[exceptionalfeasiblesize]{}}}{\leq}e(\cF_{t_{k+1}+j}')+e(\cQ_j)\\
			&\stackrel{\text{\eqmakebox[exceptionalfeasiblesize]{\text{\cref{lm:backwardmatchings}\cref{lm:backwardmatchings-matchings},\cref{lm:extendlinforest}\cref{lm:extendlinforest-cover}}}}}{\leq} \left(\sum_{i\in [4]}m_j^{i\downarrow}+2|U^{1-\gamma}(T)|\right)+(|S_j^+|+|S_j^-|)\\
			&\stackrel{\text{\eqmakebox[exceptionalfeasiblesize]{\text{\cref{eq:exceptionalfeasible-m},\cref{eq:exceptionalfeasible-S}}}}}{\leq} 4|U^{1-\gamma}(T)|+2|U^*|
			\stackrel{\text{\cref{fact:U},\cref{def:ES}}}{\leq} \sqrt{\varepsilon}n,
		\end{align*}
		so \cref{lm:exceptionalfeasible-IH-size} holds.
		Recall from \cref{fact:U,def:ES} that $U^{1-\gamma}(T)\subseteq U^*$. Thus, \namecrefs{lm:backwardmatchings}~\ref{lm:backwardmatchings}\ref{lm:backwardmatchings-degree1} and \ref{lm:extendlinforest}\cref{lm:extendlinforest-degree1} imply that each $v\in V(T)\setminus U^*$ and $i\in [s_{k+1}]$ satisfy both $d_{\cF_{t_{k+1}+i}}(v)\leq 1$ and $d_{\cQ_i}(v)\leq 1$. Therefore, \cref{lm:extendlinforest}\cref{lm:extendlinforest-cover} and the definition of $T_1, \dots, T_{s_{k+1}}$ imply that \cref{lm:exceptionalfeasible-IH-V'} holds.
		Moreover, \cref{lm:exceptionalfeasible-IH-H} follows from \cref{lm:backwardmatchings}\cref{lm:backwardmatchings-matchings} and \cref{lm:extendlinforest}\cref{lm:extendlinforest-cover}.
		
		Let $j\in [s_{k+1}]$. We now check that $\cF_{t_{k+1}+j}$ is a feasible system. We have already verified that \cref{def:feasible-linforest} holds.
		By \cref{lm:backwardmatchings}\cref{lm:backwardmatchings-matchings} and since \cref{step:exceptionalfeasible-coverU*} only involves forward edges, each $i\in [4]$ satisfies 
		\begin{align*}
		    e_{\cF_{t_{k+1}+j}}(U_i, U_{i-1})=e_{\cF_{t_{k+1}+j}'}(U_i, U_{i-1})= m_j^{i\downarrow}+|U_i^{1-\gamma}(T)|+|U_{i-1}^{1-\gamma}(T)|\stackrel{\text{\cref{eq:exceptionalfeasible-m}}}{=}|U^{1-\gamma}(T)|.
		\end{align*}
		Thus, \cref{def:feasible-backward} is satisfied.
		By \cref{lm:extendlinforest}\cref{lm:extendlinforest-cover}, \cref{def:feasible-exceptional} is satisfied.  Therefore, $\cF_{t_{k+1}+j}$ is a feasible system, as desired.\qedhere
	\end{steps}
\end{proof}

Let $D$ be a digraph. For each $e\in E(D)$, let $L(e)$ be a list of colours. A \emph{proper list edge-colouring of $D$} is a colouring of the edges of $D$ such that each edge $e\in E(D)$ is coloured with one of the colours in its list $L(e)$ and no two adjacent edges receive the same colour. The next \lcnamecref{prop:listcolouring} states that if the lists are all sufficiently large, then a proper list edge-colouring exists. Its proof follows from a simple greedy colouring argument and is therefore omitted.

\begin{prop}\label{prop:listcolouring}
    Let $D$ be a digraph. For each $e=uv\in E(D)$, let $L(e)$ be a list of colours satisfying $|L(e)|\geq d_D(u)+d_D(v)+1$. Then, $D$ has a proper list edge-colouring.
\end{prop}

\COMMENT{\begin{proof}
    Suppose inductively that we have already coloured $0\leq k\leq e(D)$ edges. Let $e=uv$ be an uncoloured edge. At most $d_D(u)+d_D(v)$ colours are already used to colour edges adjacent to $e$, so there is a free colour in $L(e)$ for $e$.
\end{proof}}

\begin{proof}[Proof of \cref{lm:forwardUstar}]
	For each $i\in [4]$, denote $U_i^*\coloneqq U^*\cap U_i$.
	By \cref{fact:partition}\cref{fact:partition-cycle} and the ``in particular part" of \cref{lm:gammaoptimal} (applied with $2\gamma$ playing the role of $\gamma$), we may assume without loss of generality that
	\begin{align}\label{eq:forwardV*-optimal}
		U_3^{1-2\gamma}(T)=\emptyset =U_4^{1-2\gamma}(T).
	\end{align}		
	\begin{steps}
		\item \textbf{Decomposing the forward edges of $D$ which are incident to $U^{1-\gamma}(T)$ and the forward edges of $D[U^*]$ which are incident to $U^{1-2\gamma}(T)$.}\label{step:forwardV**}
		By \cref{eq:Ugamma}, each $v\in U^{1-\gamma}(T)$ satisfies \[\overrightarrow{d}_{D,\cU}^\pm(v)\leq \overrightarrow{d}_{T,\cU}^\pm(v)\leq\lfloor\gamma n\rfloor\leq t'.\]
		Moreover, each $v\in U^{1-2\gamma}(T)$ satisfies
		\[|\overrightarrow{N}_{D, \cU}^\pm(v)\cap U^*|\leq |U^*|\stackrel{\text{\cref{def:ES-size}}}{\leq} t'.\]
		Also recall from \cref{fact:U,def:ES} that $|U^{1-\gamma}(T)|\leq |U^{1-2\gamma}(T)|\leq t'$.
		For each $i\in [4]$, \cref{prop:Konigsamesize} (applied with the corresponding underlying undirected graph playing the role of $G$) implies that the digraph
		\[D(U_i^{1-\gamma}(T), U_{i+1})\cup D(U_i, U_{i+1}^{1-\gamma}(T))\cup D(U_i^{1-2\gamma}(T), U_{i+1}^*)\cup D(U_i^*, U_{i+1}^{1-2\gamma}(T))\]
		can be decomposed into $t'$ edge-disjoint matchings $M_1^{i\uparrow}, \dots, M_{t'}^{i\uparrow}$.
		Observe that the following hold.
		\begin{enumerate}[label=(\greek*)]
			\item For each $i\in [4]$ and $j\in [t']$, $M_j^{i\uparrow}\subseteq E_D(U_i, U_{i+1})$.\label{lm:forwardUstar-forwardV**-forward}
			\item $\bigcup_{(i,j)\in [4]\times [t']}M_j^{i\uparrow}= \{e\in E(\overrightarrow{D}_\cU)\mid V(e)\cap U^{1-\gamma}(T)\neq \emptyset\}\cup \{e\in E(\overrightarrow{D}_\cU[U^*])\mid V(e)\cap U^{1-2\gamma}(T)\neq \emptyset\}$. In particular, \cref{fact:U} implies that each edge in $\bigcup_{(i,j)\in [4]\times [t']}M_j^{i\uparrow}$ is incident to $U^{1-2\gamma}(T)$.\label{lm:forwardUstar-forwardV**-V**}
		\end{enumerate}
		
		\item \textbf{Selecting backward edges.}\label{step:backward}
		In this step, we use backward edges to ensure that each vertex in $U^{1-2\gamma}(T)$ is covered by both an in- and an outedge in each of the feasible system. Covering $U^{1-\gamma}(T)$ is necessary for \cref{def:feasible-degV*} to be satisfied, while covering $U^{1-2\gamma}(T)\setminus U^{1-\gamma}(T)$ will ensure that \cref{lm:forwardUstar-forwarddegreeV*} is satisfied. We will also balance the number of backward edges using edges which are not incident to $U^{1-2\gamma}(T)$. This corresponds to \cref{step:sketch-coverU**,step:sketch-backward} of the proof overview presented in \cref{sec:sketch-specialpseudofeasible} and will be carried out using \cref{lm:backwardmatchings}.
		
		Let $H\subseteq \overleftarrow{T}_\cU$ be the digraph obtained by applying \cref{lm:optimalH} with $2\gamma$ playing the role of $\gamma$. 
		Note that the following hold.
		\begin{enumerate}[label=\rm(\Roman*)]
		    \item $\Delta^0(H)\leq 2\gamma n$.\label{lm:forwardU*-Delta}
		    \item For each $v\in U^{1-2\gamma}(T)$, $d_H(v)=0$.\label{lm:forwardU*-U**}
		    \item For each $i\in [4]$, $e_{H-U^{1-2\gamma}(T)}(U_i, U_{i-1})\geq (1-4\gamma)n|U_{i-2}^{1-2\gamma}(T)\cup U_{i-3}^{1-2\gamma}(T)|$.\label{lm:forwardU*-e}
		\end{enumerate}
		Let $H'$ be obtained from $H\cap D$ by adding all the edges of $\overleftarrow{D}_\cU$ which have precisely one endpoint in $U^{1-2\gamma}(T)$ and precisely one endpoint in $V(T)\setminus U^*$. 
		Note that
		\begin{equation}\label{eq:forwardV*-H}
			e_{H'}(U^{1-2\gamma}(T),U^*)+e_{H'}(U^*,U^{1-2\gamma}(T))\stackrel{\text{\cref{lm:forwardU*-U**}}}{=}0.
		\end{equation}		
		We now verify that \cref{lm:backwardmatchings}\cref{lm:backwardmatchings-H-degree,lm:backwardmatchings-H-degreeV**,lm:backwardmatchings-H-edges} hold with $H'$ and $2\gamma$ playing the roles of $H$ and $\gamma$.
		For each $v\in U^{1-2\gamma}(T)$,
		\begin{align*}
			d_{H'}^\pm(v)&\stackrel{\text{\eqmakebox[forwardU*H']{\text{\cref{lm:forwardU*-U**}}}}}{=}|\overleftarrow{N}_{D,\cU}^\pm(v)\setminus U^*|
			\geq \overleftarrow{d}_{T,\cU}^\pm (v)-\Delta^0(T\setminus D)-|U^*|\\
			&\stackrel{\text{\eqmakebox[forwardU*H']{\text{\cref{lm:forwardUstar-D-Delta0}}}}}{\geq} (1-2\gamma)n-t-|U^*|\stackrel{\text{\cref{def:ES-size}}}{\geq} (1-3\gamma) n.
		\end{align*}
		Therefore, \cref{lm:backwardmatchings}\cref{lm:backwardmatchings-H-degreeV**} is satisfied, with room to spare.
		For each $v\in V(T)\setminus U^{1-2\gamma}(T)$, 
		\[d_{H'}^\pm(v)\leq d_H^\pm(v)+|U^{1-2\gamma}(T)|\stackrel{\text{\cref{lm:forwardU*-Delta},\cref{fact:U},\cref{def:ES}}}{\leq} 3\gamma n.\]
		Thus, \cref{lm:backwardmatchings}\cref{lm:backwardmatchings-H-degree} is satisfied, with room to spare.
		For each $i\in [4]$,
		\begin{align*}
			e_{H'-U^{1-2\gamma}(T)}(U_i, U_{i-1})\stackrel{\text{\cref{lm:forwardUstar-D-U'},\cref{lm:forwardU*-e}}}{\geq} (1-5\gamma)n|U_{i-2}^{1-2\gamma}(T)\cup U_{i-3}^{1-2\gamma}(T)|.
		\end{align*}
		Therefore, \cref{lm:backwardmatchings}\cref{lm:backwardmatchings-H-edges} holds, with room to spare.
		
		Let $i\in [4]$ and $j\in [t']$. Recall from \cref{lm:backwardmatchings}\cref{lm:backwardmatchings-matchings} that, when applying \cref{lm:backwardmatchings}, $S_j^{i\downarrow}$ denotes the set of vertices that need to be avoided by the edges from $U_i$ to $U_{i-1}$. Here, we need to avoid the vertices in $U_i$ which are already covered with an outedge and the vertices in $U_{i-1}$ which are already covered with an inedge. By \cref{lm:forwardUstar-forwardV**-forward}, these are precisely the vertices in
		\begin{equation}\label{eq:forwardV*-backward-S}
			S_j^{i\downarrow}\coloneqq (V(M_j^{i\uparrow})\cup V(M_j^{(i-2)\uparrow}))\cap (U_i\cup U_{i-1}).
		\end{equation}
		Then,	
		\begin{align*}
			|S_j^{i\downarrow}\setminus U^{1-2\gamma}(T)|
			&\stackrel{\text{\eqmakebox[forwardU*S]{\text{\cref{lm:forwardUstar-forwardV**-forward}}}}}{=}|V(M_j^{i\uparrow})\cap (U_i\setminus U^{1-2\gamma}(T))|+|V(M_j^{(i-2)\uparrow})\cap (U_{i-1}\setminus U^{1-2\gamma}(T))|\\
			&\stackrel{\text{\eqmakebox[forwardU*S]{\text{\cref{lm:forwardUstar-forwardV**-forward},\cref{lm:forwardUstar-forwardV**-V**}}}}}{\leq} |V(M_j^{i\uparrow})\cap U_{i+1}^{1-2\gamma}(T)|+|V(M_j^{(i-2)\uparrow})\cap U_{i-2}^{1-2\gamma}(T)|\\
			&\stackrel{\text{\eqmakebox[forwardU*S]{}}}{\leq} |U_{i+1}^{1-2\gamma}(T)|+ |U_{i-2}^{1-2\gamma}(T)|.
		\end{align*}
		Therefore, \cref{lm:backwardmatchings}\cref{lm:backwardmatchings-sizeS} is satisfied with $2\gamma$ playing the role of $\gamma$.
		Let 
		\begin{equation}\label{eq:forwardV*-backward-m}
			m_j^{i\downarrow}\coloneqq \left|\left(U_{i-2}^{1-2\gamma}(T)\cup U_{i-3}^{1-2\gamma}(T)\right)\setminus S_j^{(i-2)\downarrow}\right|.
		\end{equation}
		Then, \cref{lm:backwardmatchings}\cref{lm:backwardmatchings-m} holds with $2\gamma$ playing the role of $\gamma$.
		
		Let $\cF_1', \dots, \cF_{t'}'$ be the edge-disjoint linear forests obtained by applying \cref{lm:backwardmatchings} with $H',t'$, and $2\gamma$ playing the roles of $H,\ell$, and $\gamma$.
		For each $j\in [t']$, note that
		\begin{equation}\label{eq:forwardV*-F'}
		    \cF_j'\subseteq \overleftarrow{D}_\cU \subseteq \overleftarrow{T}_\cU
		\end{equation}
		and let $\cF_j''\coloneqq \bigcup_{i\in [4]} M_j^{i\uparrow}\cup \cF_j'$. Denote $\cF''\coloneqq \bigcup_{j\in [t']}\cF_j''$.
		
		\begin{claim}\label{claim:forwardV*}
			$\cF_1'', \dots, \cF_{t'}''$ are edge-disjoint and satisfy the following properties.
			\begin{enumerate}[label=\rm(\Alph*)]
				\item For each $j\in [t']$, $\cF_j''$ is a $(\gamma,T)$-pseudo-feasible system.\label{lm:forwardUstar-F'-feasible}
				\item For each $j\in [t']$, $e(\cF_j'')\leq 24\varepsilon n$.\label{lm:forwardUstar-F'-size}
				\item For each $v\in U^{1-\gamma}(T)$, $\overrightarrow{d}_{\cF'', \cU}^\pm(v)=\overrightarrow{d}_{D,\cU}^\pm(v)$.\label{lm:forwardUstar-F'-V**}
				\item For each $v\in U^{1-2\gamma}(T)\setminus U^{1-\gamma}(T)$, $\overleftarrow{d}_{\cF'', \cU}^\pm(v)\geq t'-4\varepsilon n$.\label{lm:forwardUstar-F'-forwarddegreeV*}
				\item For each $v\in V(T)\setminus U^{1-2\gamma}(T)$, $d_{\cF''}(v)\leq \frac{2\gamma n}{5}$.\label{lm:forwardUstar-F'-V'}
				\item $E(\overrightarrow{\cF''}_\cU)=\{e\in E(\overrightarrow{D}_\cU)\mid V(e)\cap U^{1-\gamma}(T)\neq \emptyset\}\cup \{e\in E(\overrightarrow{D}_\cU[U^*])\mid V(e)\cap U^{1-2\gamma}(T)\neq \emptyset\}$. In particular, $E(\overrightarrow{\cF''}_\cU[U^*\setminus U^{1-2\gamma}(T)])=\emptyset$.\label{lm:forwardUstar-F'-forward}
			\end{enumerate}
		\end{claim}
		
		\begin{proofclaim}
			By \cref{lm:forwardUstar-forwardV**-forward} and \cref{eq:forwardV*-F'}, $E(\cF_j')\cap M_{j'}^{i\uparrow}=\emptyset$ for any $i\in [4]$ and $j,j'\in [t']$. Thus, $\cF_1'', \dots, \cF_{t'}''$ are edge-disjoint.	
			Moreover, \cref{eq:forwardV*-F'} implies that $\overrightarrow{\cF''}_\cU=\bigcup_{(i,j)\in [4]\times [t']}M_j^{i\uparrow}$. Thus, \cref{lm:forwardUstar-F'-V**,lm:forwardUstar-F'-forward} follow from \cref{lm:forwardUstar-forwardV**-V**}. Moreover, note for later that each $v\in U^{1-2\gamma}(T)\setminus U^{1-\gamma}(T)$ satisfies
			\begin{equation}\label{eq:forwardV*-U**}
				\overrightarrow{d}_{\cF'',\cU}^\pm(v)\stackrel{\text{\cref{lm:forwardUstar-forwardV**-V**}}}{=}|\overrightarrow{N}_{D,\cU}^\pm(v)\cap U^*|\stackrel{\text{\cref{def:ES-size}}}{\leq}4\varepsilon n.
			\end{equation}
			For each $j\in [t']$,
			\begin{align*}
				e(\cF_j'')&\stackrel{\text{\eqmakebox[forwardU*e]{}}}{=}\sum_{i\in [4]}|M_j^{i\uparrow}|+e(\cF_j')\\
				&\stackrel{\text{\eqmakebox[forwardU*e]{\text{\cref{lm:forwardUstar-forwardV**-forward},\cref{lm:forwardUstar-forwardV**-V**},\cref{lm:backwardmatchings}\cref{lm:backwardmatchings-matchings}}}}}{\leq} 2|U^{1-2\gamma}(T)|+\left(\sum_{i\in [4]}m_j^{i\downarrow}+2|U^{1-2\gamma}(T)|\right)\\
				&\stackrel{\text{\eqmakebox[forwardU*e]{\text{\cref{eq:forwardV*-backward-m}}}}}{\leq} 6|U^{1-2\gamma}(T)|
				\stackrel{\text{\cref{fact:U},\cref{def:ES}}}{\leq} 24\varepsilon n, 
			\end{align*}
			so \cref{lm:forwardUstar-F'-size} holds. For each $i\in [4]$ and $v\in U_i\setminus U^{1-2\gamma}(T)$, we have%
				\COMMENT{Recall that we applied \cref{lm:backwardmatchings} with $2\gamma$ playing the role of $\gamma$.}
			\begin{align*}
				d_{\cF''}(v)&\stackrel{\text{\eqmakebox[forwardU*deg]{\text{\cref{lm:forwardUstar-forwardV**-forward}}}}}{=}\sum_{j\in[t']}\left(\left(d_{M_j^{i\uparrow}}(v)+d_{M_j^{(i-1)\uparrow}}(v)\right)+d_{\cF_j'}(v)\right)\\
				&\stackrel{\text{\eqmakebox[forwardU*deg]{\text{\cref{lm:forwardUstar-forwardV**-V**},\cref{lm:backwardmatchings}\cref{lm:backwardmatchings-degree}}}}}{\leq} 2|U^{1-2\gamma}(T)|+\frac{\gamma n}{3}
				\stackrel{\text{\cref{fact:U},\cref{def:ES}}}{\leq} \frac{2\gamma n}{5},
			\end{align*}
			so \cref{lm:forwardUstar-F'-V'} holds.
			Let $j\in [t']$. We show that $\cF_j''$ is a $(\gamma, T)$-pseudo-feasible system.
			First, note that
			\begin{align*}
				e_{\cF_j''}(U_1, U_4)&
				\stackrel{\text{\eqmakebox[forwardU*F1]{\text{\cref{lm:forwardUstar-forwardV**-forward}}}}}{=}e_{\cF_j'}(U_1, U_4)
				\stackrel{\text{\cref{lm:backwardmatchings}\cref{lm:backwardmatchings-matchings}}}{=}m_j^{1\downarrow}+|(U_1^{1-2\gamma}(T)\cup U_4^{1-2\gamma}(T))\setminus S_j^{1\downarrow}|\\
				&\stackrel{\text{\eqmakebox[forwardU*F1]{\text{\cref{eq:forwardV*-backward-m}}}}}{=} |(U_3^{1-2\gamma}(T)\cup U_2^{1-2\gamma}(T))\setminus S_j^{3\downarrow}|+m_j^{3\downarrow}\\
				&\stackrel{\text{\eqmakebox[forwardU*F1]{\text{\cref{lm:backwardmatchings}\cref{lm:backwardmatchings-matchings}}}}}{=} e_{\cF_j'}(U_3,U_2)
				\stackrel{\text{\cref{lm:forwardUstar-forwardV**-forward}}}{=}e_{\cF_j''}(U_3, U_2).
			\end{align*} 
			Similarly, $e_{\cF_j''}(U_4, U_3)=e_{\cF_j''}(U_2, U_1)$%
			\COMMENT{\begin{align*}
					e_{\cF_j''}(U_2, U_1)&
					\stackrel{\text{\eqmakebox[forwardU*F1']{\text{\cref{lm:forwardUstar-forwardV**-forward}}}}}{=}e_{\cF_j'}(U_2, U_1)
					\stackrel{\text{\cref{lm:backwardmatchings}\cref{lm:backwardmatchings-matchings}}}{=}m_j^{2\downarrow}+|(U_2^{1-2\gamma}(T)\cup U_1^{1-2\gamma})\setminus S_j^{2\downarrow}|\\
					&\stackrel{\text{\eqmakebox[forwardU*F1']{\text{\cref{eq:forwardV*-backward-m}}}}}{=}|(U_4^{1-2\gamma}(T)\cup U_3^{1-2\gamma}(T))\setminus S_j^{4\downarrow}|+|(U_2^{1-2\gamma}(T)\cup U_1^{1-2\gamma})\setminus S_j^{2\downarrow}|\\
					&\stackrel{\text{\eqmakebox[forwardU*F1']{\text{\cref{eq:forwardV*-backward-m}}}}}{=} |(U_4^{1-2\gamma}(T)\cup U_3^{1-2\gamma}(T))\setminus S_j^{4\downarrow}|+m_j^{4\downarrow}\\
					&\stackrel{\text{\eqmakebox[forwardU*F1']{\text{\cref{lm:backwardmatchings}\cref{lm:backwardmatchings-matchings}}}}}{=} e_{\cF_j'}(U_4,U_3)
					\stackrel{\text{\cref{lm:forwardUstar-forwardV**-forward}}}{=}e_{\cF_j''}(U_4, U_3).
			\end{align*}}
			and so \cref{def:feasible-backward} is satisfied.
			For each $i\in [4]$ and $v\in U_i$,
			\begin{align*}
				d_{\cF_j''}^+(v)&\stackrel{\text{\cref{lm:forwardUstar-forwardV**-forward},\cref{eq:forwardV*-F'}}}{=}d_{M_j^{i\uparrow}}(v)+d_{\cF_j'(U_i,U_{i-1})}(v)
				\stackrel{\text{\cref{eq:forwardV*-backward-S},\cref{lm:backwardmatchings}\cref{lm:backwardmatchings-matchings}}} \leq 1.
			\end{align*}
			Similarly, each $v\in V(T)$ satisfies $d_{\cF_j''}^-(v)\leq 1$%
			\COMMENT{Suppose that $v\in U_i$. Then, 
				\begin{align*}
					d_{\cF_j''}^-(v)&\stackrel{\text{\cref{lm:forwardUstar-forwardV**-forward},\cref{eq:forwardV*-F'}}}{=}d_{M_j^{(i-1)\uparrow}}(v)+d_{\cF_j'(U_{i+1},U_i)}(v)
					\stackrel{\text{\cref{eq:forwardV*-backward-S},\cref{lm:backwardmatchings}\cref{lm:backwardmatchings-matchings}}} \leq 1.
			\end{align*}}. 
			Thus, \cref{def:feasible-degV'} holds.
			Moreover, each $i\in [4]$ and $v\in U_i^{1-2\gamma}(T)$ satisfy
			\begin{align*}
				d_{\cF_j''}^+(v)&\stackrel{\text{\cref{lm:forwardUstar-forwardV**-forward},\cref{eq:forwardV*-F'}}}{=}d_{M_j^{i\uparrow}}(v)+d_{\cF_j'(U_i,U_{i-1})}(v)
				\stackrel{\text{\cref{eq:forwardV*-backward-S},\cref{lm:backwardmatchings}\cref{lm:backwardmatchings-matchings}}} = 1.
			\end{align*}
			Similarly, each $v\in U^{1-2\gamma}(T)$ satisfies $d_{\cF_j''}^-(v)=1$%
			\COMMENT{Suppose that $v\in U_i^{1-2\gamma}(T)$. Then,
				\begin{align*}
					d_{\cF_j''}^-(v)&\stackrel{\text{\cref{lm:forwardUstar-forwardV**-forward},\cref{eq:forwardV*-F'}}}{=}d_{M_j^{(i-1)\uparrow}}(v)+d_{\cF_j'(U_{i+1},U_i)}(v)
					\stackrel{\text{\cref{eq:forwardV*-backward-S},\cref{lm:backwardmatchings}\cref{lm:backwardmatchings-matchings}}} = 1.
			\end{align*}}.
			Thus, \cref{def:feasible-degV*} holds and \cref{lm:forwardUstar-F'-forwarddegreeV*} follows from \cref{eq:forwardV*-U**}.
			
			Finally, to verify \cref{def:feasible-cycle}, suppose that $C$ is a cycle in $\cF_j''$. We show that $C$ contains a $(\gamma,T)$-placeholder.
			By \cref{eq:forwardV*-optimal}, \cref{lm:forwardUstar-forwardV**-V**}, and \cref{eq:forwardV*-F'},
			\[e_{\cF_j''}(U_3,U_4)+e_{\cF_j''}(U_4^{1-2\gamma}(T),U_3)+e_{\cF_j''}(U_4,U_3^{1-2\gamma}(T))=0.\]
			Therefore, \cref{lm:backwardmatchings}\cref{lm:backwardmatchings-degree1} implies that each edge in $E_{\cF_j''}(U_4, U_3)$ forms a component in $\cF_j''$. Altogether, we have
			\begin{equation}\label{lm:forwardUstar-C}
				e_C(U_3,U_4)+e_C(U_4,U_3)=0.
			\end{equation}
		
			Suppose that there exists $e\in E_C(U_2,U_1)\subseteq E_{\cF_j''}(U_2,U_1)$.
			By \cref{lm:forwardUstar-forwardV**-forward}, $e\in E(\cF_j')$ and, by \cref{eq:forwardV*-optimal,eq:forwardV*-backward-m}, $m_j^{2\downarrow}=0$.
			Thus, \cref{lm:backwardmatchings}\cref{lm:backwardmatchings-matchings} implies that $V(e)\cap U^{1-2\gamma}(T)\neq \emptyset$.
			By \cref{lm:backwardmatchings}\cref{lm:backwardmatchings-matchings}, \cref{eq:forwardV*-H}, \cref{fact:U}, and \cref{def:ES-backward}, $e$ is a backward edge with precisely one endpoint in $U^{1-2\gamma}(T)\subseteq U^*$ and one endpoint in $V(T)\setminus U^*$. Thus, $e$ is a $(\gamma,T)$-placeholder, as desired.
			
			We may therefore assume that $e_C(U_2,U_1)=0$.
			Then, \cref{lm:forwardUstar-forwardV**-forward}, \cref{eq:forwardV*-F'}, and \cref{lm:forwardUstar-C} imply that either $V(C)\subseteq U_4\cup U_1$ or $V(C)\subseteq U_3\cup U_2$%
			    \COMMENT{By \cref{lm:backwardmatchings}\cref{lm:backwardmatchings-matchings} and \cref{lm:forwardUstar-forwardV**-forward} all the edges lie in the pairs of the blow-up $C_4$.}.
			Suppose the former (similar arguments hold in the other case). By \cref{lm:backwardmatchings}, $\cF_j'$ is a linear forest and so \cref{lm:forwardUstar-forwardV**-forward} implies that there exists $uv\in M_j^{4\uparrow}\cap E(C)$. By \cref{eq:forwardV*-optimal}, \cref{lm:forwardUstar-forwardV**-forward}, and \cref{lm:forwardUstar-forwardV**-V**}, $v\in U_1^{1-2\gamma}(T)$. 
			Let $w$ denote the outneighbour of $v$ on $C$. By assumption, \cref{lm:forwardUstar-forwardV**-forward}, and \cref{eq:forwardV*-F'}, $w\in U_4$. In particular, $vw$ is a backward edge and so \cref{lm:forwardUstar-forwardV**-forward} and \cref{lm:backwardmatchings}\cref{lm:backwardmatchings-matchings} imply that
			$vw\in E(H')$. 
			Thus, \cref{eq:forwardV*-H} implies that $w\in V(T)\setminus U^*$ and so $vw$ is a $(\gamma,T)$-placeholder, as desired.
			Therefore, \cref{def:feasible-cycle} holds and so $\cF_j''$ is a $(\gamma,T)$-pseudo-feasible system. Thus, \cref{lm:forwardUstar-F'-feasible} holds.
		\end{proofclaim}
		
		\item \textbf{Covering the edges of $\overrightarrow{D}[U^*\setminus U^{1-2\gamma}(T)]$.}\label{step:forwardV*}
		We will use \cref{prop:listcolouring} as follows. For each $e\in E(\overrightarrow{D}_\cU[U^*\setminus U^{1-2\gamma}(T)])$, let $L(e)$ be the set of colours $i\in [t']$ such that $V(e)\cap V(E(\cF_i''))=\emptyset$.
		For each $e=uv\in E(\overrightarrow{D}_\cU[U^*\setminus U^{1-2\gamma}(T)])$, we have
		\[|L(e)|\geq t' -d_{\cF''}(u)-d_{\cF''}(v) \stackrel{\text{\cref{lm:forwardUstar-F'-V'}}}{\geq}\frac{\gamma n}{6}\stackrel{\text{\cref{def:ES-size}}}{\geq}|\overrightarrow{N}_{D,\cU}(u)\cap U^*|+ |\overrightarrow{N}_{D,\cU}(v)\cap U^*|+1.\]
		Thus, \cref{prop:listcolouring} implies that $\overrightarrow{D}_\cU[U^*\setminus U^{1-2\gamma}(T)]$ has a proper list edge-colouring $\phi: E(\overrightarrow{D}_\cU[U^*\setminus U^{1-2\gamma}(T)])\longrightarrow [t']$. 
		For each $i\in [t']$, let $\cF_i\coloneqq \cF_i'' \cup \phi^{-1}(i)$. Since $\phi$ is an edge-colouring of $\overrightarrow{D}_\cU[U^*\setminus U^{1-2\gamma}(T)]$, any distinct $i,i'\in [t']$ satisfy
		\begin{equation}\label{eq:forwardU*-Fdisjoint}
		    E(\cF_i\setminus \cF_i'')\cap E(\cF_{i'}\setminus \cF_{i'}'')=\emptyset.
		\end{equation}
		Denote $\cF\coloneqq \bigcup_{i\in [t']}\cF_i$ and note that
		\begin{equation}\label{eq:forwardU*-Fforward}
		    E(\cF\setminus \cF'')\stackrel{\text{\cref{lm:forwardUstar-F'-forward}}}{=} E(\overrightarrow{D}_\cU[U^*\setminus U^{1-2\gamma}(T)]).
		\end{equation}

		\item \textbf{Verifying \cref{lm:forwardUstar-size,lm:forwardUstar-forwardV*,lm:forwardUstar-forwardV**,lm:forwardUstar-forwarddegreeV*}.}
		We claim that $\cF_1, \dots, \cF_{t'}$ are edge-disjoint $(\gamma,T)$-pseudo-feas\-ible systems satisfying \cref{lm:forwardUstar-forwardV*,lm:forwardUstar-forwardV**,lm:forwardUstar-size,lm:forwardUstar-forwarddegreeV*}.
		Recall from \cref{claim:forwardV*} that $E(\cF_i'')\cap E(\cF_{i'}'')=\emptyset$ for any distinct $i,i'\in [t']$. Thus, \cref{eq:forwardU*-Fdisjoint}, \cref{eq:forwardU*-Fforward}, and the ``in particular part" of \cref{lm:forwardUstar-F'-forward} imply that $\cF_1, \dots, \cF_{t'}$ are edge-disjoint.
		
		Let $i\in [t']$. We now show that $\cF_i$ is a $(\gamma, T)$-pseudo-feasible system.
		By construction, $E(\cF_i)\setminus E(\cF_i'')\subseteq E(\overrightarrow{T}_\cU)$. Thus, \cref{def:feasible-backward} follows from \cref{lm:forwardUstar-F'-feasible}.
		By construction of the lists of colours, we have $V(E(\cF_i\setminus \cF_i''))\cap V(E(\cF_i''))=\emptyset$ and, since $\phi$ is proper, $\phi^{-1}(i)$ is a matching for each $i\in [t']$. Thus, \cref{def:feasible-degV*,def:feasible-degV'} follow from \cref{lm:forwardUstar-F'-feasible}. Moreover, each cycle in $\cF_i$ is a cycle in $\cF_i''$. Thus, \cref{def:feasible-cycle} also follows from \cref{lm:forwardUstar-F'-feasible}. Therefore, $\cF_i$ is a $(\gamma,T)$-pseudo-feasible system, as desired.
		
		Moreover, \cref{lm:forwardUstar-forwardV*} follows from \cref{lm:forwardUstar-F'-forward} and \cref{eq:forwardU*-Fforward},
		\cref{lm:forwardUstar-forwardV**} follows from \cref{lm:forwardUstar-F'-V**} and \cref{eq:forwardU*-Fforward}, and \cref{lm:forwardUstar-forwarddegreeV*} follows from \cref{lm:forwardUstar-F'-forwarddegreeV*}.
		Finally, the fact that $\phi^{-1}(i)$ is a matching implies that
		\[e(\cF_i)=e(\cF_i'')+|\phi^{-1}(i)|\stackrel{\text{\cref{lm:forwardUstar-F'-size}}}{\leq}24\varepsilon n+|U^*\setminus U^{1-2\gamma}(T)|\stackrel{\text{\cref{def:ES-size}}}{\leq} \sqrt{\varepsilon}n\]
		and so \cref{lm:forwardUstar-size} holds.
		\qedhere	
	\end{steps}
\end{proof}

\section*{Acknowledgements}
We thank Ant\'{o}nio Gir\~{a}o, Daniela K\"{u}hn, Allan Lo, and Deryk Osthus for helpful discussions and advice throughout this project. In particular, the author is very grateful to Allan Lo for sharing ideas and to Daniela K\"{u}hn and Deryk Osthus for supplying guidance and feedback.

Thank you to Eoin Long for helpful comments on an earlier version of this paper.

\bibliographystyle{abbrv}
\bibliography{Bibliography/Bibliography}

\APPENDIX{\appendix
    \section{Optimal packings of Hamilton cycles: proof of Corollary \ref{thm:packings}}\label{app:packings}
    
	\onlyinsubfile{
		\setcounter{section}{13}
		\section{Packings}}

In this appendix, we prove \cref{thm:packings}. First, we will need the following properties of $(\varepsilon,d)$-regular bipartite graphs. (Recall that those were defined in \cref{sec:regularity}.)
The next two \lcnamecrefs{lm:epsvertexslice} hold by definition (and so their proofs are omitted).

\begin{lm}\label{lm:epsvertexslice}
	Let $0<\frac{1}{m}\ll\varepsilon\ll d< 1$ and $\varepsilon\leq \eta \ll 1$.
	Let $G$ be a bipartite graph on vertex classes $A$ and $B$ of size $m$. Suppose that $G$ is $(\varepsilon, d)$-regular. Let $A'\subseteq A$ and $B'\subseteq B$ satisfy $|A'|, |B'|\geq \eta m$. 
	Then, $G[A', B']$ is $(\frac{\varepsilon}{\eta}, \geq d-\varepsilon)$-regular.
\end{lm}

\COMMENT{\begin{proof}
	Let $A''\subseteq A'$ and $B''\subseteqq B'$ satisfy $|A''|\geq \frac{\varepsilon}{\eta}|A'|\geq \varepsilon m$ and $|B''|\geq \frac{\varepsilon}{\eta}|B'|\geq \varepsilon m$. Then, $d_G(A'',B'')=d\pm \varepsilon=d_G(A',B')\pm 2\varepsilon$.
\end{proof}}

\begin{lm}[{\cite[Proposition 4.2]{kuhn2013hamilton}}]\label{lm:epsdeg}
	Let $0<\varepsilon\leq d\leq 1$.
	Let $G$ be an  $(\varepsilon, d)$-regular bipartite graph on vertex classes $A$ and $B$. Then, fewer than $\varepsilon |A|$ vertices $a\in A$ satisfy $d_G(a)\geq (d+\varepsilon)|B|$ and fewer than $\varepsilon |A|$ vertices $a\in A$ satisfy $d_G(a)\leq (d-\varepsilon)|B|$.
\end{lm}

One can easily deduce that $\varepsilon$-regular bipartite graphs of linear minimum degree are also bipartite robust expanders.

\begin{lm}\label{lm:epsrob}
    Let $0<\frac{1}{n}\ll \varepsilon\ll \nu\ll \tau \ll \delta \leq 1$. Let $G$ be a bipartite graph on vertex classes $A$ and $B$ of size $n$. Suppose that $G$ is $\varepsilon$-regular and $\delta(G)\geq \delta n$. Then, $G$ is a bipartite robust $(\nu, \tau)$-expander with bipartition $(A,B)$.
\end{lm}

\begin{proof}
    Note that $G$ is $(\varepsilon, \geq \delta)$-regular.
    Let $S\subseteq A$ satisfy $\tau n\leq |S|\leq (1-\tau)n$. If suffices to show that $|RN_{\nu, G}(S)|\geq |S|+\nu n$. By \cref{lm:epsvertexslice}, $G[S, B]$ is $(\sqrt{\varepsilon}, \geq \delta-\varepsilon)$-regular and so \cref{lm:epsdeg} implies that all but at most $\sqrt{\varepsilon}n$ vertices $v\in B$ satisfy $|N_G(v)\cap S|\geq (\delta-\varepsilon-\sqrt{\varepsilon})|S|\geq \frac{\delta \tau n}{2}\geq \nu n$. Thus, $|RN_{\nu, G}(S)|\geq (1-\sqrt{\varepsilon})n\geq |S|+\nu n$, as desired.  
\end{proof}

Using the max-flow min-cut theorem, Frieze and Krivelevich \cite{frieze2005packing} showed that $\varepsilon$-regular bipartite graphs of linear minimum degree contain a dense regular spanning subgraph.

\begin{lm}[{\cite{frieze2005packing}}]\label{lm:epsregeven}
    Let $0<\frac{1}{n}\ll \varepsilon\ll \delta \leq 1$. Let $G$ be an $\varepsilon$-regular bipartite graph on vertex classes of size $n$. Suppose that $\delta(G)\geq \delta n$. Then, $\reg_{\rm even}(G)\geq (\delta-2\varepsilon)n$.
\end{lm}

By \cref{prop:almostcompleteeps}, the complete bipartite graph $K_{n,n}$ is $[\varepsilon, 1]$-superregular. Therefore, \cref{lm:randomreg} implies that $G_{n,n,p}$ is also superregular with high probability.

\begin{cor}\label{lm:randomepsG}
    Let $0<\frac{1}{n}\ll \varepsilon\ll p\leq 1$. With high probability, $G_{n,n,p}$ is $\varepsilon$-regular and $\delta(G_{n,n,p})\geq (p-\varepsilon)n$.
\end{cor}

\COMMENT{\begin{proof} 
    Fix an additional constant such that $\varepsilon'\ll \varepsilon \ll p$.
    By \cref{prop:almostcompleteeps}, the complete bipartite graph $K_{n,n}$ is $[\varepsilon', 1]$-regular and so \cref{lm:randomreg} implies that $G_{n,n,p}$ is $\varepsilon$-regular and $\delta^0(G_{n,n,p})\geq (p-\varepsilon)n$ with high probability.
\end{proof}}

We are now ready to prove \cref{thm:packings}.

\begin{proof}[Proof of \cref{thm:packings}]
	Let $0<p\leq 1$.
    Fix additional constants such that $0<\frac{1}{n_0}\ll\varepsilon\ll \varepsilon_1\ll \varepsilon_2\ll \varepsilon_3\ll \nu\ll \tau\ll p$.
    By \cref{lm:randomepsG}, \cref{thm:packings-Gnnp} follows immediately from \cref{thm:packings-epsG} (with $p-\varepsilon$ playing the role of $p$). 
    
    For \cref{thm:packings-Dnnp}, denote by $A$ and $B$ the vertex classes of $D_{n,n,p}$. Observe that $D_{n,n,p}[A,B]\sim G_{n,n,p}$ and $D_{n,n,p}[B,A]\sim G_{n,n,p}$. Thus, \cref{lm:randomepsG} implies that $D_{n,n,p}$ is $\varepsilon$-regular of minimum semidegree $\delta^0(D_{n,n,p})\geq (p-\varepsilon)n$ with high probability and so \cref{thm:packings-Dnnp} follows from~\cref{thm:packings-epsD} (with $p-\varepsilon$ playing the role of $p$).
    
    For \cref{thm:packings-T}, let $T$ be chosen uniformly at random among the bipartite tournaments on vertex classes $A$ and $B$ of size $n$. Observe that $T[A,B]\sim G_{n,n,\frac{1}{2}}$ and $T[B,A]\sim G_{n,n,\frac{1}{2}}$. Thus, \cref{lm:randomepsG} implies that $T$ is $\varepsilon$-regular of minimum semidegree $\delta^0(T)\geq (\frac{1}{2}-\varepsilon)n$ with high probability and so \cref{thm:packings-T} also follows from \cref{thm:packings-epsD} (with $\frac{1}{2}-\varepsilon$ playing the role of $p$).

    For \cref{thm:packings-epsG}, let $G$ be an $\varepsilon$-regular bipartite graph on vertex classes $A$ and $B$ of size $n$ and suppose that $\delta(G)\geq pn$. Note that $G$ is $(\varepsilon, d)$-regular for some $d\geq p$. Let $S$ be the set of vertices $v\in V(G)$ which satisfy $d_G(v)\geq (d+\varepsilon)n$. By \cref{lm:epsdeg}, $|S|\leq \varepsilon n$ and so \cref{prop:epsremovingadding}\cref{prop:epsremovingadding-reg} implies that $G-S$ is still $\varepsilon_1$-regular.
    Let $G'$ be a spanning $\reg_{\rm even}(G)$-regular subgraph of $G$. By \cref{lm:epsregeven}, $G'-S$ is obtained from $G-S$ by deleting at most $2\varepsilon n$ edges incident to each vertex and so \cref{prop:epsremovingadding}\cref{prop:epsremovingadding-reg} implies that $G'-S$ is $\varepsilon_2$-regular. Another application of \cref{prop:epsremovingadding}\cref{prop:epsremovingadding-reg} implies that $G'$ is $\varepsilon_3$-regular. Thus, \cref{lm:epsrob} implies that $G'$ is a bipartite robust $(\nu, \tau)$-expander with bipartition $(A, B)$, as well as with bipartition $(B, A)$. Apply \cref{cor:undirected} to decompose $G'$ into $\frac{\reg_{\rm even}(G)}{2}$ edge-disjoint Hamilton cycles.
    
    For \cref{thm:packings-epsD}, let $D$ be an $\varepsilon$-regular bipartite digraph on vertex classes $A$ and $B$ of size $n$ and suppose that $\delta^0(D)\geq pn$. Let $D'$ be a $\reg(D)$-regular spanning subdigraph of $D$. Observe that \cref{lm:epsregeven} implies that $\reg(D)\geq (p-2\varepsilon)n$ and so, by similar arguments as above%
    	\COMMENT{By definition, $D[A,B]$ and $D[B,A]$ are both $\varepsilon$-regular of minimum degree at least $pn$. Thus, the above arguments imply that $D'[A,B]$ and $D'[B,A]$ are both of minimum degree at least $(p-2\varepsilon)n$ and are bipartite robust $(\nu, \tau)$-expanders with bipartition $(A,B)$ and $(B,A)$, respectively. Thus, $D'$ is a bipartite robust $(\nu, \tau)$-outexpander with bipartition $(A,B)$.},
    $D'$ is a bipartite robust $(\nu, \tau)$-outexpander with bipartition $(A,B)$. Apply \cref{thm:biprobexp} to decompose $D'$ into $\reg(D)$ edge-disjoint Hamilton cycles.
\end{proof}

\onlyinsubfile{\bibliographystyle{abbrv}
	\bibliography{Bibliography/Bibliography}}

	\section{Approximate decomposition: proof of Theorem \ref{thm:biphalfapproxHamdecomp}}\label{app:approximatedecomp}
	
	\onlyinsubfile{
		\appendix
		\section{Approximate decomposition}}

In this appendix, we adapt the arguments of \cite{girao2020path} to prove \cref{thm:biphalfapproxHamdecomp}.

\subsection{Preliminaries}

We will need some additional preliminary results.

\COMMENT{\begin{lm*}[{\cite[Lemma 4.19]{girao2020path}}]\label{cor:Hallreg}
		Let $0<\frac{1}{n}\ll\varepsilon \ll \delta\leq 1$. Let $G$ be a bipartite graph on vertex classes $A$ and $B$ such that $|A|,|B|=(1\pm \varepsilon)n$. Suppose that, for each $v\in V(G)$, $d_G(v)=(\delta\pm \varepsilon) n$. Then, $G$ contains a matching of size at least $\left(1-\frac{3\varepsilon}{\delta}\right)n$. 
\end{lm*}}

\subsubsection{Regularity}
Recall the definition of an $(\varepsilon,d)$-regular bipartite graph from \cref{sec:regularity}.
We need the (non-bipartite version of the) degree form regularity lemma for digraphs.

\begin{lm}[Degree form regularity lemma for digraphs]\label{lm:reglm}
	For all $\varepsilon>0$ and $M'\in \mathbb{N}$, there exist $M, n_0\in \mathbb{N}$ such that if $D$ is a digraph on $n\geq n_0$ vertices and $d\in [0,1]$, then there exist a spanning subdigraph $D'\subseteq D$ and a partition of $V(D)$ into an \emph{exceptional set} $V_0$ and $k$ \emph{clusters} $V_1,\dots, V_k$ such that the following hold.
	\begin{enumerate}
		\item $M'\leq k\leq M$.
		\item $|V_0|\leq \varepsilon n$.\label{lm:reglm-V0}
		\item $|V_1|=\dots=|V_k|\eqqcolon m$.
		\item For each $v\in V(D)$, both $d_{D'}^\pm(v)>d_D^\pm(v)-(d+\varepsilon)n$.
		\item For each $i\in [k]$, $D'[V_i]$ is empty.
		\item Let $i,j\in [k]$ be distinct. Then, $D'[V_i,V_j]$ is either empty or $(\varepsilon, \geq d)$-regular. Moreover, if $D'[V_i, V_j]$ is non-empty, then $D'[V_i, V_j]=D[V_i, V_j]$.\label{lm:reglm-reg}
	\end{enumerate}
\end{lm}

Let $\varepsilon>0$, $M'\in \mathbb{N}$, and $d\in [0,1]$.
Let $D$ be a digraph.
The \emph{pure digraph of $D$ with parameters $\varepsilon, d$, and $M'$} is the digraph $D'\subseteq D$ obtained by applying \cref{lm:reglm} with these parameters.
The \emph{reduced digraph of $D$ with parameters $\varepsilon, d$, and $M'$} is the digraph $R$ defined as follows. Let $V_0, V_1, \dots, V_k$ be the partition obtained by applying \cref{lm:reglm} with parameters $\varepsilon, d$, and $M'$. Denote by $D'$ the pure digraph of $D$ with parameters $\varepsilon, d$, and $M'$.
Then, $V(R)\coloneqq \{V_i\mid i\in [k]\}$ and, for any distinct $U,V\in V(R)$, $UV\in E(R)$ if and only if $D'[U,V]$ is non-empty. Note that \cref{lm:reglm}\cref{lm:reglm-reg} implies that $D'[U,V]=D[U,V]$ is $(\varepsilon, \geq d)$-regular for any $UV\in E(R)$.

The following result states that robust outexpansion is inherited by the reduced digraph.

\begin{lm}[{\cite[Lemma 14]{kuhn2010hamiltonian}}]\label{lm:Rrob2}
	Let $0<\frac{1}{n}\ll \varepsilon \ll d\ll \nu, \tau, \delta \leq 1$ and $\frac{M'}{n}\ll 1$. Let $D$ be a robust $(\nu, \tau)$-outexpander on $n$ vertices. Suppose that $\delta^0(D)\geq \delta n$.
	Let $R$ be the reduced digraph of $D$ with parameters $\varepsilon, d$, and $M'$.
	Then, $\delta^0(R)\geq \frac{\delta |R|}{2}$ and $R$ is a robust $(\frac{\nu}{2}, 2\tau)$-outexpander. 
\end{lm}

\subsubsection{Robust outexpanders}

By definition of a bipartite robust outexpander, the $\tau$-parameter can be made arbitrarily small. An analogous observation for the non-bipartite setting was made (and proved) in \cite[Lemma 4.3]{girao2020path}, so we omit the details here.

\begin{lm}\label{lm:robparameters}
	Let $0<\frac{1}{n}\ll \nu\ll \tau\leq \frac{\delta}{2}\leq 1$. Let $D$ be a bipartite digraph on vertex classes $A$ and $B$ of size $n$. Suppose that $D$ is a bipartite robust $(\nu, \frac{\delta}{2})$-outexpander with bipartition $(A,B)$. Suppose furthermore that $\delta^0(D)\geq \delta n$. Then, $D$ is a bipartite robust $(\nu, \tau)$-outexpander with bipartition $(A,B)$. 
\end{lm}

\COMMENT{\begin{proof}
		Let $S\subseteq A$ satisfy $\tau n\leq |S|\leq (1-\tau)n$ and denote $T\coloneqq RN_{\nu, D}^+(S)$. If $\frac{\delta n}{2}\leq |S|\leq (1-\frac{\delta}{2})n$, then, by assumption, $|RN_{\nu, D}^+(S)|\geq |S|+\nu n$.\\
		Assume that $|S|<\frac{\delta n}{2}$. Let $m\coloneqq |\{e\in E(D)\mid V^+(e)\subseteq S\}|$. Then, $m\geq |S|\delta n$. Moreover, $m\leq |T||S|+(n-|T|)\nu n\leq |T||S|+\nu n^2$. Therefore, $|T|\geq \frac{|S|\delta n-\nu n^2}{|S|}\geq \delta n-\frac{\nu n}{\tau}\geq \frac{\delta n}{2}+\nu n\geq |S|+\nu n$,
		as desired.\\
		Assume that $|S|>(1-\frac{\delta}{2})n$. Then, for each $v\in V(D)$, $|N_D^\pm(v)\cap S|\geq \delta n- |A\setminus S|\geq \frac{\delta n}{2}\geq \nu n$. Therefore, $T=B$ and $|T|=n\geq |S|+\nu n$. This completes the proof.
\end{proof}}

\begin{lm}[{\cite[Lemma 4.2]{girao2020path}}]\label{lm:verticesedgesremovalrobout}
	Let $0<\varepsilon\leq \nu\ll \tau\leq 1$. Let $D$ be a robust $(\nu, \tau)$-outexpander on $n$ vertices.
	\begin{enumerate}
		\item If $D'$ is obtained from $D$ by removing at most $\varepsilon n$ inedges and at most $\varepsilon n$ outedges at each vertex, then $D'$ is a robust $(\nu-\varepsilon,\tau)$-outexpander.\label{lm:verticesedgesremovalrobout-edges}
		\item If $D'$ is obtained from $D$ by adding or removing at most $\varepsilon n$ vertices, then $D'$ is a robust $(\nu-\varepsilon, 2\tau)$-outexpander.\label{lm:verticesedgesremovalrobout-vertices}
	\end{enumerate}
	
\end{lm}

We will need \cite[Lemma 14.3]{girao2020path}, which states that robust outexpansion is inherited by random vertex subsets with high probability. Note that we only gave a brief proof overview of this lemma in \cite{girao2020path} as it was only used to sketch a new shorter proof of the main result of \cite{osthus2013approximate}. Thus, for completeness, we include its proof here.

\begin{lm}[{\cite[Lemma 14.3]{girao2020path}}]\label{lm:epsilonrob}
	Let $0<\frac{1}{n}\ll \varepsilon\ll \nu'\ll \delta, \nu, \tau\ll 1$. Fix a positive integer $n'\geq \varepsilon n$. Suppose that $D$ is a robust $(\nu,\tau)$-outexpander on $n$ vertices satisfying $\delta^0(D)\geq \delta n$. Suppose that $V$ is chosen uniformly at random among the subsets of $V(D)$ of size $n'$. Then, $D[V]$ is a robust $(\nu',4\tau)$-outexpander with probability at least $1-n^{-2}$.
\end{lm}

\begin{proof}
	Fix additional constants such that $\frac{1}{n}\ll \varepsilon ' \ll \varepsilon \ll\nu'\ll d\ll \nu,\tau$ and $\frac{M'}{n}\ll 1$. 
	Let $V_0, V_1, \dots, V_k$ be the partition of $V(D)$ obtained by applying \cref{lm:reglm} with $\varepsilon'$ playing the role of $\varepsilon$ and define $m\coloneqq |V_1|=\dots =|V_k|$.	
	Denote by $R$ the reduced digraph of $D$ with parameters $\varepsilon', d$, and $M'$.
	By \cref{lm:Rrob2}, $R$ is a robust $(\frac{\nu}{2},2\tau)$-outexpander.
	Let $n'\geq \varepsilon n$ and suppose that $V$ is chosen uniformly at random among the subsets of $V(D)$ of size $n'$. We show that $D'\coloneqq D[V]$ is a robust $(\nu',4\tau)$-outexpander with probability at least $1-n^{-2}$.
	
	For each $i\in [k]$, denote $V_i'\coloneqq V_i\cap V(D')$. Let $i\in [k]$. Then, $\mathbb{E}[|V_i'|]=\frac{n' m}{n}\eqqcolon m'\geq \varepsilon m$. Then, \cref{lm:Chernoff} implies that
	\[\mathbb{P}[|V_i'|\neq (1\pm \varepsilon)m']\leq 2\exp\left(-\frac{\varepsilon^3m}{3}\right).\]
	Thus, a union bound implies that $|V_i'|=(1\pm\varepsilon)m'$ for each $i\in [k]$ with probability at least $1-n^{-2}$.
	Therefore, we assume that $|V_i'|=(1\pm\varepsilon)m'$ for each $i\in [k]$ and show that $D'$ is a robust $(\nu',4\tau)$-outexpander.
	
	Note that \cref{lm:epsvertexslice} (with $\varepsilon'$ and $\varepsilon^2$ playing the roles of $\varepsilon$ and $\eta$) implies that, for each $V_iV_j\in E(R)$ and $S\subseteq V_i'$ satisfying $|S|\geq \varepsilon m'\geq \varepsilon^2 m$, the pair $D'[S,V_j']$ is still $(\varepsilon, \geq d-\varepsilon')$-regular. 
	Let $S\subseteq V(D')$ satisfy $4\tau n' \leq |S|\leq (1-4\tau)n'$.
	We need to show that $|RN_{\nu', D'}^+(S)|\geq |S|+\nu'n'$.
	Let $S'\coloneqq \{V_i\mid i\in [k], |S\cap V_i|=|S\cap V_i'|\geq dm'\}$. Then, 
	\begin{align}
		|S'|&\geq \frac{|S|-dm'k}{m'}\label{eq:epsilonrob-S'}\\
		&\geq \frac{4\tau n'}{m'}-d k= \frac{4\tau n}{m}-d k\geq 2\tau k.\nonumber
	\end{align}
	If $|S'|\leq (1-2\tau)k$, then let $S''\coloneqq S'$; otherwise, choose $S''\subseteq S'$ of size $(1-2\tau)k$.
	Then, $|RN_{\frac{\nu}{2},R}^+(S'')|\geq |S''|+\frac{\nu k}{2}$.
	
	Let $V_i\in S''$ and $S_i\coloneqq V_i'\cap S=V_i\cap S$. By definition of $S'$, we have $|S_i|\geq dm'$ and so $D'[S_i,V_j']$ is $(\varepsilon, \geq d-\varepsilon')$-regular for each $V_iV_j\in E(R)$. Then, \cref{lm:epsdeg} implies that, for each $V_iV_j\in E(R)$, all but at most $\varepsilon(1+\varepsilon)m'\leq 2\varepsilon m'$ vertices $v\in V_j'$ satisfy $d_{D'[S_i,V_j']}(v)\geq(d-\varepsilon'-\varepsilon)|S_i|\geq (d-2\varepsilon)dm'$.
	
	Thus, for each $V_i\in RN_{\frac{\nu}{2}, R}^+(S'')$, all but at most $\frac{2\varepsilon m' \cdot k}{2\sqrt{\varepsilon}k}=\sqrt{\varepsilon}m'$ vertices
	$v\in V_i'$ satisfy%
	\COMMENT{Count vertices of $V_i'$ which have good degree in at least $(\frac{\nu}{2}-2\sqrt{\varepsilon})k$ pairs of $R$.}
	\begin{align*}
		|N_{D'}^-(v)\cap S|\geq \left(\frac{\nu }{2}-2\sqrt{\varepsilon}\right)k\cdot  (d-2\varepsilon)dm'\geq \left(\frac{\nu d^2}{2}-2\sqrt{\varepsilon}d^2-\nu\varepsilon d\right)km'\stackrel{\text{\cref{lm:reglm}\cref{lm:reglm-V0}}}{\geq} \nu' n'.
	\end{align*}
	Therefore, 
	\begin{align*}
		|RN_{\nu',D'}^+(S)|&\stackrel{\text{\eqmakebox[epsilonrob]{}}}{\geq} |RN_{\frac{\nu}{2}, R}^+(S'')|\left(1-\varepsilon-\sqrt{\varepsilon}\right)m'\geq \left(|S''|+\frac{\nu k}{2}\right)(1-2\sqrt{\varepsilon})m'\\
		&\stackrel{\text{\eqmakebox[epsilonrob]{\text{\cref{lm:reglm}\cref{lm:reglm-V0}}}}}{\geq} |S''|m'+\frac{\nu n'}{3}.
	\end{align*}	
	If $|S''|=(1-2\tau)k$, then \cref{lm:reglm}\cref{lm:reglm-V0} implies that $|S''|m'\geq(1-4\tau)n'\geq |S|$ and so $|RN_{\nu',D'}^+(S)|\geq |S|+\nu'n'$, as desired. We may therefore assume that $S''=S'$. Then, \cref{eq:epsilonrob-S'} implies that $|S''|m'\geq |S|-d n'$ and so $|RN_{\nu',D'}^+(S)|\geq |S|+\nu'n'$, as desired.
\end{proof}

\subsubsection{Probabilistic estimates}

The following \lcnamecrefs{lm:randomslice} are easy consequences of \cref{lm:Chernoff}.

\begin{lm}\label{lm:randomslice}
	Let $0<\frac{1}{n}\ll \varepsilon\ll \nu \ll \tau \ll \gamma \ll \delta \leq 1$. Let $G$ be a balanced bipartite graph on vertex classes $A$ and $B$ of size $n$. Suppose that $G$ is a $(\delta,\varepsilon)$-almost regular bipartite robust $(\nu, \tau)$-expander with bipartition $(A,B)$.
	Let $\Gamma$ be obtained from $G$ by taking each edge independently with probability $\frac{\gamma}{\delta}$. Then, with positive probability, all of the following hold.
	\begin{enumerate}
		\item $G\setminus \Gamma$ is $(\delta-\gamma, \varepsilon)$-almost regular.\label{lm:randomslice-reg}
		\item $\Gamma$ is $(\gamma, \varepsilon)$-almost regular.\label{lm:randomslice-reg2}
		\item $\Gamma$ is a bipartite robust $(\frac{\gamma\nu}{2\delta}, \tau)$-expander with bipartition $(A,B)$.\label{lm:randomslice-rob}
	\end{enumerate}
\end{lm}

\COMMENT{For a proof, see comments in \cite{girao2020path} (Lemmas 4.12 and 4.13).}

\begin{lm}\label{lm:randompartition}
	Let $0<\frac{1}{n}\ll \varepsilon \leq 1$ and fix positive integers $k,\ell \geq \varepsilon n$.
	Let $A$ and $B$ be disjoint vertex sets of size $n$. Suppose that $M_1, \dots, M_\ell$ are bipartite perfect matchings on vertex classes $A$ and $B$.
	Suppose that $A_1, \dots, A_\ell$ are chosen independently and uniformly at random among subsets of $A$ of size $k$.
	Then, with probability at least $1-n^{-1}$, both of the following hold.
	\begin{enumerate}
		\item For each $v\in A$, there exist at most $\frac{(1+\varepsilon)\ell k}{n}$ indices $i\in [\ell]$ such that $v\in A_i$.\label{lm:randompartition-A}
		\item For each $v\in B$, there exist at most $\frac{(1+\varepsilon)\ell k}{n}$ indices $i\in [\ell]$ such that $v\in N_{M_i}(A_i)$.\label{lm:randompartition-B}
	\end{enumerate}
\end{lm}

\COMMENT{\begin{proof}
		Let $v\in A$. Let $X$ be the number of indices $i\in [\ell]$ such that $v\in A_i$. Then, $\mathbb{E}[X]=\frac{\ell k}{n}$ and so, by \cref{lm:Chernoff},
		\[\mathbb{P}\left[X\geq (1+\varepsilon)\frac{\ell k}{n}\right]\leq \exp\left(-\frac{\varepsilon^4 n}{3}\right)\leq \frac{1}{n^3}.\]
		Thus, by a union bound, \cref{lm:randompartition-A}, fails with probability at most $n^{-2}$.
		Similarly, \cref{lm:randompartition-B} fails with probability at most $n^{-2}$. The \lcnamecref{lm:randompartition} holds by a union bound.
\end{proof}}

\begin{lm}\label{lm:randomsubset}
	Let $0<\frac{1}{n}\ll \varepsilon \ll \delta \leq 1$ and fix a positive integer $k\geq \varepsilon n$.
	Let $D$ be a $(\delta, \varepsilon)$-almost regular digraph on $n$ vertices. Let $A$ be chosen uniformly at random among subsets of $V(D)$ of size $k$. Then, the following hold with probability at least $1-n^{-2}$.
	\begin{enumerate}
		\item $D[A]$ and $D-A$ are $(\delta, 3\varepsilon)$-almost regular.\label{lm:randomsubset-regular}
		\item Each $v\in V(D)\setminus A$ satisfies $|N_D^\pm(v)\cap A|\geq (\delta-3\varepsilon)|A|$.\label{lm:randomsubset-degree}
	\end{enumerate} 
\end{lm}

\subsubsection{Matching contractions}

Let $G$ be a bipartite graph on vertex classes $A$ and $B$ and let $D$ be a digraph on $A$. Let $M$ be an auxiliary perfect matching from $B$ to $A$.
Recall \cref{def:contractexpand} and note that there is a one-to-one correspondence between the edges of $G\setminus M[B,A]$ and the $M$-contraction of $G$, as well as between the edges of $D$ and the $M$-expansion of $D$. Thus, edge-disjointness and sub(di)graph relations are preserved when considering $M$-contractions and $M$-expansions.

\begin{fact}\label{fact:subgraphcontract}
	Let $A$ and $B$ be disjoint vertex sets of equal size. Let $M$ be a directed perfect matching from $B$ to $A$.
	Let $G$ and $G'$ be bipartite graphs on vertex classes $A$ and $B$ and denote by $G_M$ and $G_M'$ the $M$-contractions of $G$ and $G'$, respectively. Let $D$ and $D'$ be digraphs on $A$ and let $D_M$ and $D_M'$ be the $M$-expansions of $D$ and $D'$, respectively.
	Then, the following hold.
	\begin{enumerate}
		\item If $G'\subseteq G$, then $G_M'\subseteq G_M$.\label{fact:subgraphcontract-subgraphG}
		\item If $G$ and $G'$ are edge-disjoint, then $G_M$ and $G_M'$ are edge-disjoint.\label{fact:subgraphcontract-disjointG}
		\item If $D'\subseteq D$, then $D_M'\subseteq D_M$.\label{fact:subgraphcontract-subgraphD}
		\item If $D$ and $D'$ are edge-disjoint, then $D_M$ and $D_M'$ are edge-disjoint.\label{fact:subgraphcontract-disjointD}
	\end{enumerate}
\end{fact}

\subsection{Proof of Theorem \ref{thm:biphalfapproxHamdecomp}}
We need a (simplified) bipartite analogue of \cite[Lemma 7.3]{girao2020path}.

\begin{lm}\label{lm:approxdecomp}
	Let $0<\frac{1}{n}\ll \varepsilon \ll \nu' \ll \nu \ll \tau\ll \gamma \ll \eta, \delta \leq 1$ and $\ell \leq 2(\delta-\eta)n$.
	Let $D$ and $\Gamma$ be edge-disjoint balanced bipartite digraphs on common vertex classes $A$ and $B$ of size $n$. 
	Suppose that $D[A,B]$ is $(\delta,\varepsilon)$-almost regular
	and $\Gamma[A,B]$ is $(\gamma, \varepsilon)$-almost regular.
	Suppose that $\Gamma[A,B]$ is a bipartite robust $(\nu,\tau)$-expander with bipartition $(A,B)$. 
	Suppose that, for each $i\in[\ell]$, $F_i$ is a bipartite directed linear forest on vertex classes $A$ and $B$ such that the following hold.
	\begin{enumerate}
		\item For each $i\in [\ell]$, $e_{F_i}(B,A)=n$.\label{lm:approxdecomp-perfectmatching}
		\item For each $i\in [\ell]$, $e_{F_i}(A,B)\leq \varepsilon^4 n$.\label{lm:approxdecomp-smallmatching}
		\item For each $v\in V(D)$, there exist at most $\varepsilon^3 n$ indices $i\in[\ell]$ such that $d_{F_i[A,B]}(v)>0$.\label{lm:approxdecomp-degreev}
	\end{enumerate}
	Define a multidigraph $\cF$ by $\cF\coloneqq \bigcup_{i\in [\ell]}F_i$.
	Then, the multidigraph $D\cup \Gamma\cup \cF$ contains edge-disjoint Hamilton cycles $C_1,\dots, C_\ell $ such that $F_i\subseteq C_i$ for each $i\in [\ell]$ and the following hold, where $D'\coloneqq D\setminus \bigcup_{i\in [\ell]}C_i$ and $\Gamma'\coloneqq \Gamma\setminus \bigcup_{i\in [\ell]}C_i$.
	\begin{enumerate}[label=\rm(\alph*)]
		\item If $\ell\leq \varepsilon^2 n$, then $\Gamma'[A,B]$ is obtained from $\Gamma[A,B]$ by removing at most $3\varepsilon^3(\nu')^{-4}n$ edges incident to each vertex.\label{lm:approxdecomp-smallchunks}
		\item If $\ell\leq (\nu')^5n$, then $D'[A,B]$ is $(\delta-\frac{\ell}{2n}, 2\varepsilon)$-almost regular and $\Gamma'[A,B]$ is $(\gamma, 2\varepsilon)$-almost regular. Moreover, $\Gamma'[A,B]$ is a bipartite robust $(\nu-\varepsilon, \tau)$-expander with bipartition $(A,B)$.\label{lm:approxdecomp-bigchunks}
		\item $D'[A,B]\cup \Gamma'[A,B]$ is a bipartite robust $(\frac{\nu}{2}, \tau)$-expander with bipartition $(A,B)$.\label{lm:approxdecomp-all}
	\end{enumerate}
\end{lm}

We first derive \cref{thm:biphalfapproxHamdecomp} from \cref{lm:approxdecomp}\cref{lm:approxdecomp-all}.

\begin{proof}[Proof of \cref{thm:biphalfapproxHamdecomp}]
	By \cref{fact:biprobexpparameters,lm:robparameters}, we may assume without loss of generality that $\varepsilon\ll \nu \ll \tau\ll \eta, \delta$. Define additional constants such that $\varepsilon\ll \nu'\ll \nu$ and $\tau\ll \gamma \ll \eta, \delta$.
	By \cref{lm:randomslice}, there exists $\Gamma\subseteq D$ such that $(D\setminus \Gamma)[A, B]$ is $(\delta-\gamma, \varepsilon)$-almost regular and $\Gamma[A,B]$ is a $(\gamma, \varepsilon)$-almost regular bipartite robust $(\frac{\gamma \nu}{2\delta}, \tau)$-expander with bipartition $(A,B)$.
	Apply \cref{lm:approxdecomp}\cref{lm:approxdecomp-all} with $D\setminus \Gamma, \delta-\gamma, \frac{\gamma \nu}{2\delta}$, and $\varepsilon^{\frac{1}{4}}$ playing the roles of $D, \delta, \nu$, and $\varepsilon$.
	This completes the proof of \cref{thm:biphalfapproxHamdecomp}.	
\end{proof}

We now prove \cref{lm:approxdecomp}.
First, \cref{lm:approxdecomp}\cref{lm:approxdecomp-bigchunks} follows by repeated applications of \cref{lm:approxdecomp}\cref{lm:approxdecomp-smallchunks} and \cref{lm:approxdecomp}\cref{lm:approxdecomp-all} follows by repeated applications of \cref{lm:approxdecomp}\cref{lm:approxdecomp-bigchunks}. The arguments are the same as in \cite{girao2020path}, so we omit these proofs here.%
It remains to prove \cref{lm:approxdecomp}\cref{lm:approxdecomp-smallchunks}.

\begin{proof}[Proof of \cref{lm:approxdecomp}\cref{lm:approxdecomp-smallchunks}]
	Let $i\in [\ell]$. Define $M_i\coloneqq E_{F_i}(B,A)$ and note that \cref{lm:approxdecomp-perfectmatching} implies that $M_i$ is a perfect matching from $B$ to $A$.
	Let $\tD_i$, $\tGamma_i$, and $\tF_i$ be the $M_i$-contractions of $D[A,B]$, $\Gamma[A,B]$, and $F_i[A, B]$, respectively.
	(In the rest of the proof, tildes will be used to denote ``contracted'' digraphs on vertex set $A$.)
	By \cref{lm:contraction}\cref{lm:contraction-regular,lm:contraction-rob}, $\tD_i$ is $(2\delta,2\varepsilon)$-regular and $\tGamma_i$ is a $(2\gamma, 2\varepsilon)$-regular robust $(\frac{\nu}{2}, \tau)$-outexpander.
	By \cref{fact:contractlinforest}, $\tF_i$ is a linear forest. 
	Let $\tP_{i,1}, \dots, \tP_{i,k_i}$ be an enumeration of the non-trivial components of $\tF_i$. For each $j\in [k_i]$, denote by $x_{i,j}^+$ and $x_{i,j}^-$ the starting and ending points of $\tP_{i,j}$. Let $S_i$ be the set of vertices $v\in A$ which are not isolated in $\tF_i$. Note that
	\begin{equation}\label{eq:approxdecomp-ki}
		k_i\leq e(\tF_i)\stackrel{\text{\cref{fact:contractinguncontracting}}}{\leq} e(F_i[A,B])\stackrel{\text{\cref{lm:approxdecomp-smallmatching}}}{\leq} \varepsilon^4 n
	\end{equation}
	and
	\begin{equation}\label{eq:approxdecomp-Si}
		|S_i|\leq 2e(\tF_i)\stackrel{\text{\cref{fact:contractinguncontracting}}}{\leq} 2e(F_i[A,B])\stackrel{\text{\cref{lm:approxdecomp-smallmatching}}}{\leq} 2\varepsilon^4 n.
	\end{equation}
	
	By \cref{lm:randompartition,lm:randomsubset,lm:epsilonrob}, there exist $A_1, \dots, A_\ell\subseteq A$ such that the following hold%
		\COMMENT{Apply \cref{lm:randompartition,lm:randomsubset,lm:epsilonrob} with $k=\lfloor\varepsilon(\nu')^{-4}n\rfloor$.
		\cref{lm:approxdecomp-Ai-A,lm:approxdecomp-Ai-B} hold from \cref{lm:randompartition} with probability at least $1-n^{-1}$.
		\cref{lm:approxdecomp-Ai-Dreg,lm:approxdecomp-Ai-Ddegree} hold from \cref{lm:randomsubset} with probability at least $1-\ell n^{-2}$.
		\cref{lm:approxdecomp-Ai-Gammareg} holds from \cref{lm:randomsubset} with probability at least $1-\ell n^{-2}$.
		\cref{lm:approxdecomp-Ai-Gammarob} holds from \cref{lm:epsilonrob} with probability at least $1-\ell n^{-2}$.}.
		\begin{enumerate}[label=(\greek*)]
			\item For each $v\in A$, there exist at most $2\varepsilon^3(\nu')^{-4}n$ indices $i\in [\ell]$ such that $v\in A_i$.\label{lm:approxdecomp-Ai-A}
			\item For each $v\in B$, there exist at most $2\varepsilon^3(\nu')^{-4}n$ indices $i\in [\ell]$ such that $v\in N_{F_i[B,A]}(A_i)$.\label{lm:approxdecomp-Ai-B}
			\item For each $i\in [\ell]$, $|A_i|=\lfloor\varepsilon(\nu')^{-4}n\rfloor$.\label{lm:approxdecomp-Ai-size}
			\item For each $i\in [\ell]$, $\tD_i-A_i$ is $(2\delta, 6\varepsilon)$-almost regular.\label{lm:approxdecomp-Ai-Dreg}
			\item For each $i\in [\ell]$ and $v\in A\setminus A_i$, $|N_{\tD_i}^\pm(v)\cap A_i|\geq \frac{\varepsilon \delta n}{(\nu')^4}$.\label{lm:approxdecomp-Ai-Ddegree}
			\item For each $i\in [\ell]$, $\tGamma_i[A_i]$ and $\tGamma_i-A_i$ are $(2\gamma, 6\varepsilon)$-almost regular.\label{lm:approxdecomp-Ai-Gammareg}
			\item For each $i\in [\ell]$, $\tGamma_i[A_i]$ is a robust $(\nu', 4\tau)$-outexpander.\label{lm:approxdecomp-Ai-Gammarob}
		\end{enumerate}
	For each $i\in [\ell]$, let $A_i'\coloneqq A_i\setminus S_i$. 
	By \cref{eq:approxdecomp-Si}, we have $|A_i'|\geq |A_i|-2\varepsilon^4n$.
	Therefore, \cref{lm:verticesedgesremovalrobout}\cref{lm:verticesedgesremovalrobout-vertices} and \cref{lm:approxdecomp-Ai-A,lm:approxdecomp-Ai-B,lm:approxdecomp-Ai-Ddegree,lm:approxdecomp-Ai-Dreg,lm:approxdecomp-Ai-Gammareg,lm:approxdecomp-Ai-Gammarob,lm:approxdecomp-Ai-size} imply that the following hold for each $i\in [\ell]$.
	\begin{enumerate}[label=(\greek*$'$)]
		\item For each $v\in A$, there exist at most $2\varepsilon^3(\nu')^{-4}n$ indices $i\in [\ell]$ such that $v\in A_i'$.\label{lm:approxdecomp-Ai'-A}
		\item For each $v\in B$, there exist at most $2\varepsilon^3(\nu')^{-4}n$ indices $i\in [\ell]$ such that $v\in N_{F_i[B,A]}(A_i')$.\label{lm:approxdecomp-Ai'-B}
		\item $|A_i'|=\varepsilon((\nu')^{-4}\pm 1)n$.\label{lm:approxdecomp-Ai'-size}
		\item $\tD_i-A_i'$ is $(2\delta, 7\varepsilon)$-almost regular\label{lm:approxdecomp-Ai'-Dreg}. 
		\item For each $v\in A\setminus A_i'$, $|N_{\tD_i}^\pm(v)\cap A_i'|\geq \frac{\varepsilon\delta n}{2(\nu')^4}$\label{lm:approxdecomp-Ai'-Ddegree}.
		\item $\tGamma_i[A_i']$ and $\tGamma_i-A_i'$ are both $(2\gamma, 7\varepsilon)$-almost regular.\label{lm:approxdecomp-Ai'-Gammareg}
		\item $\tGamma_i[A_i']$ and $\tGamma_i-A_i'$ are both robust $(\frac{\nu'}{2}, 8\tau)$-outexpanders.\label{lm:approxdecomp-Ai'-Gammarob}
	\end{enumerate}

	Assume inductively that for some $0\leq m\leq \ell$ we have constructed, for each $i\in [m]$, a set $\widetilde{\cQ}_i= \{\tQ_{i,j}\mid j\in[k_i]\}$ of paths in $\tD_i\cup \tGamma_i$ such that the following hold, where, for each $i\in [m]$, $\cQ_i$ is obtained from the $M_i$-expansion of $\widetilde{\cQ}_i$ by orienting all the edges from $A$ to~$B$.
	\begin{enumerate}[label=(\Alph*)]
		\item $\cQ_1, \dots, \cQ_m$ are edge-disjoint.\label{lm:approxdecomp-IH-disjoint}
		\item Let $i\in [m]$. For each $j\in [k_i]$, $\tQ_{i,j}$ is an $(x_{i,j}^-,x_{i,j+1}^+)$-path (where $x_{i,k_i+1}^+\coloneqq x_{i,1}^+$). Moreover, the paths in $\widetilde{\cQ}_i\cup \{\tP_{i,1}, \dots, \tP_{i,k_i}\}$ are internally vertex-disjoint and span~$A$. In particular, $\tC_i\coloneqq \widetilde{\cQ}_i\cup \tF_i$ is a Hamilton cycle on $A$.\label{lm:approxdecomp-IH-shape}
		\item For each $i\in [m]$ and $j\in [k_i-1]$, $V(\tQ_{i,j})\cap A_i'=\emptyset$ and $e(\tQ_{i,j})\leq 9(\nu')^{-1}$. Moreover, for each $v\in A$, there exist at most $\varepsilon^3 n$ indices $i\in [m]$ such that $v\in V(E(\widetilde{\cQ}_i\setminus \{\tQ_{i,k_i}\})\cap E(\tGamma_i))$ and for each $v\in B$, there exist at most $\varepsilon^3n$ indices $i\in[m]$ such that $v\in N_{F_i[B,A]}(V(E(\widetilde{\cQ}_i\setminus \{\tQ_{i,k_i}\})\cap E(\tGamma_i)))$.\label{lm:approxdecomp-IH-shortpaths}
		\item For each $i\in [m]$, $E(\tQ_{i, k_i})\cap E(\tGamma_i)\subseteq E(\tGamma_i[A_i'])$.\label{lm:approxdecomp-IH-spanningpath}
	\end{enumerate}
	Denote $D_{m+1}\coloneqq D\setminus \bigcup_{i\in [m]} \cQ_i$ and $\Gamma_{m+1}\coloneqq \Gamma\setminus \bigcup_{i\in [m]} \cQ_i$.
	
	\begin{claim}\label{claim:GammaD}
		$\Gamma_{m+1}[A,B]$ is obtained from $\Gamma[A,B]$ by removing at most $3\varepsilon^3(\nu')^{-4}n$ edges incident to each vertex and $D_{m+1}[A,B]$ is obtained from $D[A,B]$ by removing at most $m$ and at least $m-4\varepsilon^3(\nu')^{-4}n$ edges incident to each vertex.
	\end{claim}

	\begin{proofclaim}
		Let $v\in A$.
		By \cref{lm:approxdecomp-IH-shape}, $d_{\widetilde{\cQ}_i}^+(v)\leq 1$ for each $i\in [m]$ and so
		\begin{align}
			|N_{D[A,B]}(v)\setminus N_{D_{m+1}[A,B]}(v)|
			&\stackrel{\text{\eqmakebox[DAB]{}}}{=}\sum_{i\in [m]}|N_{D[A,B]}(v)\cap N_{\cQ_i[A,B]}(v)|\nonumber\\
			&\stackrel{\text{\eqmakebox[DAB]{\text{\cref{fact:Nuncontract}\cref{fact:Nuncontract-A}}}}}{=}\sum_{i\in[m]}|N_{\tD_i}^+(v)\cap N_{\widetilde{\cQ}_i}^+(v)|\label{eq:approxdecomp-DAB}\\
			&\stackrel{\text{\eqmakebox[DAB]{}}}{\leq} m.\nonumber
		\end{align}
		Moreover, \cref{lm:approxdecomp-IH-shortpaths} implies that there are at most $\varepsilon^3 n$ indices $i\in [m]$ such that $v\in V(E(\widetilde{\cQ}_i\setminus \{\tQ_{i,k_i}\})\cap E(\tGamma_i))$ and, by \cref{lm:approxdecomp-Ai'-A} and \cref{lm:approxdecomp-IH-spanningpath}, there are at most $2\varepsilon^3(\nu')^{-4}n$ indices $i\in [m]$ such that $v\in V(E(\tQ_{i,k_i})\cap E(\tGamma_i))$.
		Thus,
		\begin{align}
			|N_{\Gamma[A,B]}(v)\setminus N_{\Gamma_{m+1}[A,B]}(v)|
			&\stackrel{\text{\eqmakebox[GammaAB]{}}}{=}\sum_{i\in [m]}|N_{\Gamma[A,B]}(v)\cap N_{\cQ_i[A,B]}(v)|\nonumber\\
			&\stackrel{\text{\eqmakebox[GammaAB]{\text{\cref{fact:Nuncontract}\cref{fact:Nuncontract-A}}}}}{=}\sum_{i\in[m]}|N_{\tGamma_i}^+(v)\cap N_{\widetilde{\cQ}_i}^+(v)|\nonumber\\
			&\stackrel{\text{\eqmakebox[GammaAB]{}}}{\leq} \varepsilon^3n+2\varepsilon^3(\nu')^{-4}n\nonumber\\
			&\stackrel{\text{\eqmakebox[GammaAB]{}}}{\leq} 3\varepsilon^3(\nu')^{-4}n.\label{eq:approxdecomp-GammaABa2}
		\end{align}
		Moreover, \cref{fact:contractlinforest} implies that, for each $i\in [m]$, we have $d_{\tF_i}^+(v)> 0$ if and only if $d_{F_i[A,B]}^+(v)> 0$. Thus, \cref{lm:approxdecomp-degreev} implies that there are at most $\varepsilon^3 n$ indices $i\in [m]$ such that $d_{\tF_i}^+(v)> 0$ and so
		\begin{align*}
			|N_{D[A,B]}(v)\setminus N_{D_{m+1}[A,B]}(v)|&\stackrel{\text{\eqmakebox[DAB2]{\text{\cref{eq:approxdecomp-DAB}}}}}{=}\sum_{i\in[m]}|N_{\tD_i}^+(v)\cap 	N_{\widetilde{\cQ}_i}^+(v)|\\
			&\stackrel{\text{\eqmakebox[DAB2]{\text{\cref{lm:approxdecomp-IH-shape}}}}}{=} m-\sum_{i\in [m]}|N_{\tGamma_i}^+(v)\cap N_{\widetilde{\cQ}_i}^+(v)|-\sum_{i\in [m]}|N_{\tF_i}^+(v)|\\
			&\stackrel{\text{\eqmakebox[DAB2]{\text{\cref{eq:approxdecomp-GammaABa2}}}}}{\geq} m-3\varepsilon^3(\nu')^{-4}n-\varepsilon^3n\geq m-4\varepsilon^3(\nu')^{-4}n.
		\end{align*}
		
		Similarly, let $v\in B$. For each $i\in [m]$, denote by $v_i$ the unique vertex of $A$ such that $vv_i\in F_i[B,A]$ ($v_i$ exists by \cref{lm:approxdecomp-perfectmatching}).
		By \cref{lm:approxdecomp-IH-shape}, $d_{\widetilde{\cQ}_i}^-(v_i)\leq 1$ for each $i\in [m]$ and so
		\begin{align}
			|N_{D[A,B]}(v)\setminus N_{D_{m+1}[A,B]}(v)|
			&\stackrel{\text{\eqmakebox[DAB3]{}}}{=}\sum_{i\in [m]}|N_{D[A,B]}(v)\cap N_{\cQ_i[A,B]}(v)|\nonumber\\
			&\stackrel{\text{\eqmakebox[DAB3]{\text{\cref{fact:Nuncontract}\cref{fact:Nuncontract-B}}}}}{=}\sum_{i\in[m]}|N_{\tD_i}^-(v_i)\cap N_{\widetilde{\cQ}_i}^-(v_i)|\label{eq:approxdecomp-DAB2}\\
			&\stackrel{\text{\eqmakebox[DAB3]{}}}{\leq} m.\nonumber
		\end{align}		
		Moreover, \cref{lm:approxdecomp-IH-shortpaths} implies that there are at most $\varepsilon^3 n$ indices $i\in [m]$ such that $v_i\in V(E(\widetilde{\cQ}_i\setminus \{\tQ_{i,k_i}\})\cap E(\tGamma_i))$ and, by \cref{lm:approxdecomp-Ai'-B} and \cref{lm:approxdecomp-IH-spanningpath}, there are at most $2\varepsilon^3(\nu')^{-4}n$ indices $i\in [m]$ such that $v_i\in V(E(\tQ_{i,k_i})\cap E(\tGamma_i))$.
		Thus,
		\begin{align}
			|N_{\Gamma[A,B]}(v)\setminus N_{\Gamma_{m+1}[A,B]}(v)|
			&\stackrel{\text{\eqmakebox[GammaAB2]{}}}{=}\sum_{i\in [m]}|N_{\Gamma[A,B]}(v)\cap N_{Q_i[A,B]}(v)|\nonumber\\
			&\stackrel{\text{\eqmakebox[GammaAB2]{\text{\cref{fact:Nuncontract}\cref{fact:Nuncontract-B}}}}}{=}\sum_{i\in[m]}|N_{\tGamma_i}^-(v_i)\cap N_{\widetilde{\cQ}_i}^-(v_i)|\nonumber\\
			&\stackrel{\text{\eqmakebox[GammaAB2]{}}}{\leq} \varepsilon^3n+2\varepsilon^3(\nu')^{-4}n\nonumber\\
			&\stackrel{\text{\eqmakebox[GammaAB2]{}}}{\leq} 3\varepsilon^3(\nu')^{-4}n.\label{eq:approxdecomp-GammaABb2}
		\end{align}
		Moreover, \cref{fact:contractlinforest} implies that, for each $i\in [m]$, we have $d_{\tF_i}^-(v_i)> 0$ if and only if $d_{F_i[A,B]}^-(v)> 0$. Thus, \cref{lm:approxdecomp-degreev} implies that there are at most $\varepsilon^3 n$ indices $i\in [m]$ such that $d_{\tF_i}^-(v_i)> 0$ and so
		\begin{align*}
			|N_{D[A,B]}(v)\setminus N_{D_{m+1}[A,B]}(v)|&\stackrel{\text{\eqmakebox[DAB4]{\text{\cref{eq:approxdecomp-DAB2}}}}}{=}\sum_{i\in[m]}|N_{\tD_i}^-(v_i)\cap N_{\widetilde{\cQ}_i}^-(v_i)|\\
			&\stackrel{\text{\eqmakebox[DAB4]{\text{\cref{lm:approxdecomp-IH-shape}}}}}{=} m-\sum_{i\in [m]}|N_{\tGamma_i}^-(v_i)\cap N_{\widetilde{\cQ}_i}^-(v_i)|-\sum_{i\in [m]}|N_{\tF_i}^-(v_i)|\\
			&\stackrel{\text{\eqmakebox[DAB4]{\text{\cref{eq:approxdecomp-GammaABb2}}}}}{\geq} m-3\varepsilon^3(\nu')^{-4}n-\varepsilon^3n\geq m-4\varepsilon^3(\nu')^{-4}n.
		\end{align*}
		This completes the proof of \cref{claim:GammaD}.
	\end{proofclaim}
	
	Suppose that $m=\ell$. For each $i\in [\ell]$, let $C_i\coloneqq \cQ_i\cup F_i$.
	By \cref{lm:approxdecomp-IH-disjoint}, $C_1, \dots, C_\ell$ are edge-disjoint.
	Let $i\in [\ell]$. Note that \cref{fact:contractinguncontracting} implies that $F_i[A,B]$ is the $M_i$-expansion of $\tF_i$. Thus, \cref{lm:approxdecomp-IH-shape} implies that $C_i$ is obtained from the $M_i$-expansion of $\tC_i$ by orienting all the edges from $A$ to $B$, and then adding the edges in $E_{F_i}(B,A)=M_i$.
	In particular, \cref{fact:contractingHamcycle,lm:approxdecomp-IH-shape} imply that $C_i$ is a Hamilton cycle on $V(D)$.  Moreover, \cref{lm:approxdecomp-smallchunks} follows from \cref{claim:GammaD}.
	It remains to show that $F_i\subseteq C_i\subseteq D\cup \Gamma \cup F_i$.
	By~\cref{lm:approxdecomp-IH-shape}, $\tF_i\subseteq \tC_i\subseteq \tD_i\cup \tGamma_i\cup \tF_i$ and, by \cref{fact:contractinguncontracting}, the $M_i$-expansions of $\tD_i$ and $\tGamma_i$ are subgraphs of $D[A,B]$ and $\Gamma[A,B]$. Thus, \cref{fact:subgraphcontract}\cref{fact:subgraphcontract-subgraphD} implies that $F_i\subseteq C_i\subseteq D\cup \Gamma \cup F_i$ and so we are done.
	
	Assume that $m<\ell$. Let $\tD_{m+1}'$ and $\tGamma_{m+1}'$ be the $M_{m+1}$-contractions of $D_{m+1}$ and $\Gamma_{m+1}$, respectively. By \cref{fact:subgraphcontract}\cref{fact:subgraphcontract-subgraphG}, $\tD_{m+1}'\subseteq \tD_{m+1}$ and $\tGamma_{m+1}'\subseteq \tGamma_{m+1}$. By \cref{fact:Ncontract}\cref{fact:Ncontract-N} and \cref{claim:GammaD}, $\tGamma_{m+1}'$ is obtained from $\tGamma_{m+1}$ by removing at most $3\varepsilon^3(\nu')^{-4}n$ inedges incident to each vertex and at most $3\varepsilon^3(\nu')^{-4}n$ outedges incident to each vertex.
	Similarly, $\tD_{m+1}'$ is obtained from $\tD_{m+1}$ by removing at most $m$ and at least $m-4\varepsilon^3(\nu')^{-4}n$ inedges incident to each vertex, as well as at most $m$ and at least $m-4\varepsilon^3(\nu')^{-4}n$ outedges incident to each vertex.
	Thus, using \cref{lm:approxdecomp-Ai'-Gammareg,lm:approxdecomp-Ai'-Gammarob,lm:approxdecomp-Ai'-size,lm:approxdecomp-Ai'-Dreg,lm:approxdecomp-Ai'-Ddegree} and \cref{lm:verticesedgesremovalrobout}\cref{lm:verticesedgesremovalrobout-edges}, it is easy to check that the following hold%
		\COMMENT{\cref{lm:approxdecomp-IS-Dreg} follows from \cref{lm:approxdecomp-Ai'-Dreg,lm:approxdecomp-Ai'-size}.\\
		\cref{lm:approxdecomp-IS-Ddegree} follows from \cref{lm:approxdecomp-Ai'-Ddegree}.\\
		\cref{lm:approxdecomp-IS-Gammareg} follows from \cref{lm:approxdecomp-Ai'-size,lm:approxdecomp-Ai'-Gammareg}.\\
		\cref{lm:approxdecomp-IS-Gammarob} follows from \cref{lm:approxdecomp-Ai'-size}, \cref{lm:approxdecomp-Ai'-Gammarob}, and \cref{lm:verticesedgesremovalrobout}\cref{lm:verticesedgesremovalrobout-edges}.}.
	\begin{enumerate}[label=(\Roman*)]
		\item $\tD_{m+1}'-A_{m+1}'$ is $(2\delta-\frac{m}{n}, 8\varepsilon)$-almost regular.\label{lm:approxdecomp-IS-Dreg}
		\item For each $v\in A\setminus A_{m+1}'$, $|N_{\tD_{m+1}'}^\pm(v)\cap A_{m+1}'|\geq \frac{\varepsilon\delta n}{3(\nu')^4}$. \label{lm:approxdecomp-IS-Ddegree}
		\item $\tGamma_{m+1}'[A_{m+1}']$ and $\tGamma_{m+1}'-A_{m+1}'$ are both $(2\gamma, 8\varepsilon)$-almost regular.\label{lm:approxdecomp-IS-Gammareg}
		\item $\tGamma_{m+1}'[A_{m+1}']$ and $\tGamma_{m+1}'-A_{m+1}'$ are robust $(\frac{\nu'}{4}, 8\tau)$-outexpanders.\label{lm:approxdecomp-IS-Gammarob}
	\end{enumerate}
	
	Let $S_A$ be the set of vertices $v\in A$ for which there exist $\lfloor\varepsilon^3 n\rfloor$ indices $i\in [m]$ such that $v\in V(E(\widetilde{\cQ}_i\setminus \{\tQ_{i, k_i}\})\cap E(\tGamma_i))$. Observe that, by \cref{eq:approxdecomp-ki}, \cref{lm:approxdecomp-Ai'-size}, and \cref{lm:approxdecomp-IH-shortpaths}, $|S_A|\leq \frac{9(\nu')^{-1} \cdot \varepsilon^4 n\cdot m}{\lfloor\varepsilon^3 n\rfloor}\leq \varepsilon |A\setminus A_{m+1}'|$%
		\COMMENT{$|S_A|\leq \frac{9(\nu')^{-1} \cdot \varepsilon^4n\cdot \ell }{\lfloor\varepsilon^3 n\rfloor}\leq \frac{10(\nu')^{-1}\varepsilon^6n^2}{\varepsilon^3n}\leq 10(\nu')^{-1}\varepsilon^3n\leq 10\varepsilon^2n\leq \varepsilon |A\setminus A_{m+1}'|$}.
	Let $S_B$ be the set of vertices $v\in B$ such that there exist $\lfloor \varepsilon^3n\rfloor$ indices $i\in [m]$ such that $v\in N_{F_i[B,A]}(V(E(\widetilde{\cQ}_i\setminus \{\tQ_{i, k_i}\})\cap E(\tGamma_i)))$.
	Let $S_B'\coloneqq N_{F_{m+1}[B,A]}(S_B)$.
	Then, by similar arguments as above, $|S_B'|=|S_B|\leq \varepsilon |A\setminus A_{m+1}'|$%
		\COMMENT{$|S_B'|=|S_B|\leq \frac{\varepsilon^4 n\cdot 9(\nu')^{-1}\cdot \ell}{\lfloor \varepsilon^3n\rfloor}\leq \frac{10(\nu')^{-1}\varepsilon^6n^2}{\varepsilon^3n}\leq 10(\nu')^{-1}\varepsilon^3n\leq 10\varepsilon^2n\leq \varepsilon |A\setminus A_{m+1}'|$.}.
	Recall that $S_{m+1}$ denotes the set of vertices $v\in A$ which are not isolated in $\tF_{m+1}$ (and so $\{x_{m+1,i}^+,x_{m+1,i}^-\mid i\in [k_{m+1}]\}\subseteq S_{m+1}$). Moreover, $S_{m+1}\cap A_{m+1}'=\emptyset$. By \cref{eq:approxdecomp-Si}, \cref{lm:approxdecomp-Ai'-size}, and the above, $|S_A\cup S_B'\cup S_{m+1}|\leq 3\varepsilon|A\setminus A_{m+1}'|$. By \cref{eq:approxdecomp-ki} and \cref{lm:approxdecomp-Ai'-size}, $k_i\leq \varepsilon^4n\leq \varepsilon^3|A\setminus A_{m+1}'|$.
	Thus, \cref{lm:approxdecomp-IS-Dreg} implies that there exist distinct $y_{m+1,1}^-,\dots, y_{m+1, k_{m+1}-1}^-,y_{m+1,2}^+,\dots, y_{m+1, k_{m+1}}^+\in A\setminus (A_{m+1}'\cup S_A\cup S_B'\cup S_{m+1})$ such that $y_{m+1,i}^-\in N_{\tD_{m+1}'}^+(x_{m+1,i}^-)$ and $y_{m+1,i+1}^+\in N_{\tD_{m+1}'}^-(x_{m+1,i+1}^+)$ for each $i\in [k_{m+1}-1]$ and we can apply the arguments of~\cite{girao2020path} to construct vertex-disjoint paths $\tQ_{m+1,1}', \dots, \tQ_{m+1,k_{m+1}-1}'\subseteq \tGamma_{m+1}-A_{m+1}'$ such that, for each $i\in[k_{m+1}-1]$,
	\begin{itemize}
		\item $\tQ_{m+1,i}'$ is a $(y_{m+1,i}^-,y_{m+1,i+1}^+)$-path of length at most $8(\nu')^{-1}$; and
		\item $V(\tQ_{m+1,i}')\subseteq A\setminus (A_{m+1}'\cup S_A\cup S_B'\cup S_{m+1})$.
	\end{itemize}%
		\COMMENT{Apply \cref{cor:robshortpaths} with $\tGamma_{m+1}'-A_{m+1}', \frac{\nu'}{4}, 8\tau, 2\gamma-8\varepsilon, 3\varepsilon, k_{m+1}-1,S_A\cup S_B'\cup S_{m+1}, y_{m+1,1}^-, \dots, y_{m+1,k_{m+1}-1}^-$, and $y_{m+1,2}^+, \dots, y_{m+1,k_{m+1}}^+$ playing the roles of $D, \nu, \tau, \delta, \varepsilon, k, S, x_1, \dots, x_k$, and $x_1', \dots, x_k'$ to obtain internally vertex-disjoint paths $\tQ_{m+1,1}',\dots, \tQ_{m+1,k_{m+1}-1}'\subseteq \tGamma_{m+1}'-A_{m+1}'$ such that, for each $i\in [k_{m+1}-1]$, $\tQ_{m+1, i}'$ is a $(y_{m+1,i}^-,y_{m+1,i+1}^+)$-path of length at most $8(\nu')^{-1}$ with $V(\tQ_{m+1, i}')\subseteq A\setminus (A_{m+1}'\cup S_A\cup S_B'\cup S_{m+1})$.}%
	(Roughly speaking, the paths $\tQ_{m+1,1}', \dots, \tQ_{m+1,k_{m+1}-1}'$ are constructed greedily by applying the definition of robust outexpansion.)
	For each $i\in [k_{m+1}-1]$, define 
	\[\tQ_{m+1,i}\coloneqq x_{m+1,i}^-y_{m+1,i}^-\tQ_{m+1,i}'y_{m+1,i+1}^+x_{m+1,i+1}^+.\]
	Let $\widetilde{\cQ}_{m+1}'\coloneqq \{\tQ_{m+1,i}\mid i\in [k_{m+1}-1]\}$ and proceed as in \cite{girao2020path} to build an $(x_{m+1,k_{m+1}}^-,x_{m+1,1}^+)$-path $\tQ_{m+1,k_{m+1}}$ satisfying the following.
	\begin{itemize}
		\item $V^0(\tQ_{m+1,k_{m+1}})=A\setminus (V(\widetilde{\cQ}_{m+1}')\cup S_{m+1})$.
		\item $\tQ_{m+1,k_{m+1}}[A_{m+1}']\subseteq \tGamma_{m+1}'$.
		\item $\tQ_{m+1,k_{m+1}}\setminus \tQ_{m+1,k_{m+1}}[A_{m+1}']\subseteq \tD_{m+1}'$.
	\end{itemize}%
	(Roughly speaking, $\tQ_{m+1,k_{m+1}}$ is constructed by applying \cref{lm:rob} to find a Hamilton cycle in a suitable auxiliary digraph.)
	Let $\widetilde{\cQ}_{m+1}\coloneqq \widetilde{\cQ}_{m+1}'\cup \{\tQ_{m+1,k_{m+1}}\}$.
	Then, \cref{lm:approxdecomp-IH-shape,lm:approxdecomp-IH-shortpaths,lm:approxdecomp-IH-spanningpath,lm:approxdecomp-IH-disjoint} hold with $m+1$ playing the role of~$m$.
\end{proof}

\onlyinsubfile{\bibliographystyle{abbrv}
\bibliography{Bibliography/Bibliography}}

	\section{The robust decomposition lemmas: proofs of Lemmas \ref{lm:newrobustdecomp} and \ref{lm:cyclerobustdecomp}}\label{app:robustdecomp}
	
	\onlyinsubfile{
		\appendix
		\section{The robust decomposition lemma}}
	
In this appendix, we derive the modified robust decomposition lemma (\cref{lm:newrobustdecomp}) and the robust decomposition lemma for blow-up cycles (\cref{lm:cyclerobustdecomp}).

\subsection{Proof of Lemma \ref{lm:newrobustdecomp}}

We need the following version of the robust decomposition lemma, which follows immediately from the proof of \cite[Lemma 12.1]{kuhn2013hamilton} and the definition of a ``chord absorber" in \cite{kuhn2013hamilton}. (Note that \cite[Lemma 12.1]{kuhn2013hamilton} is not explicitly proven because its proof is identical to that of \cite[Lemma 11.2]{kuhn2013hamilton}. See the paragraph before \cite[Lemma 12.1]{kuhn2013hamilton} for more details.)

\begin{lm}[{Robust decomposition lemma \cite{kuhn2013hamilton}}]\label{lm:robustdecomp}
	Let $0<\frac{1}{m}\ll \frac{1}{k}\ll \eps \ll \frac{1}{q} \ll \frac{1}{f} \ll \frac{r_1}{m}\ll d\ll \frac{1}{\ell'}, \frac{1}{g}\ll 1$ and suppose
	that $rk^2\le m$. Let
	\[r_2\coloneqq 96\ell'g^2kr, \quad  r_3\coloneqq \frac{rfk}{q}, \quad r^\diamond\coloneqq r_1+r_2+r-(q-1)r_3,\]
	and suppose that $\frac{k}{14}, \frac{k}{f}, \frac{k}{g}, \frac{q}{f}, \frac{m}{4\ell'}, \frac{fm}{q}, \frac{2fk}{3g(g-1)} \in \mathbb{N}$.
	Suppose that $(D,\cP,\cP', \cP^*,R,C, U,U')$ is an $(\ell',\frac{q}{f},k,m,\eps,d)$-setup with empty exceptional set $V_0$.
	Suppose that $\mathcal{SF}\subseteq D$ consists of $r_3$ edge-disjoint $(\frac{q}{f},f)$-special factors
	with respect to $\cP^*$ and $C$.
	Then, there exists a spanning subdigraph $CA^\diamond(r)\subseteq D$ for which the following hold.
	\begin{enumerate}
		\item $CA^\diamond(r)$ is an $(r_1+r_2)$-regular spanning subdigraph of $D$ which is edge-disjoint from $\mathcal{SF}$.\label{lm:robustdecomp-CA}
		\item Let $H$ be an $r$-regular digraph on $V(D)$ which is edge-disjoint from $CA^\diamond(r)$.
		Suppose that $\mathcal{SF}^*$ consists of $r_3$ edge-disjoint $(\frac{q}{f},f)$-special factors
		with respect to $\cP^*$ and $C$ which are edge-disjoint from $CA^\diamond(r)\cup H$. (Note that $\mathcal{SF}^*$ is not necessarily a subdigraph of $D$ here.) Then, there exists a set $\sC_1$ of $rfk$ edge-disjoint Hamilton cycles such that $E(H)\cup E(\mathcal{SF}^*)\subseteq E(\sC_1)\subseteq E(CA^\diamond(r))\cup E(H)\cup E(\mathcal{SF}^*)$ and each cycle in $\sC_1$ contains precisely one of the special path systems contained in $\mathcal{SF}^*$. Denote $H'\coloneqq CA^\diamond(r)\setminus E(\sC_1)$.\label{lm:robustdecomp-CAH}
		\item Suppose that $\mathcal{SF}'\subseteq D\setminus (CA^\diamond(r)\cup \mathcal{SF})$ consists of $r^\diamond$ edge-disjoint $(1,7)$-special factors
		with respect to $\cP$ and $C$.
		Then, there exists a spanning subdigraph $PCA^\diamond(r)\subseteq D$ for which the following hold.\label{lm:robustdecomp-PCA}
		\begin{enumerate}[label=\rm(\alph*),ref=\rm(\roman{enumi}.\alph*)]
			\item $PCA^\diamond(r)$ is a $5r^\diamond$-regular spanning subdigraph of $D$ which
			is edge-disjoint from $CA^\diamond(r)\cup \mathcal{SF}\cup \mathcal{SF}'$.\label{lm:robustdecomp-PCAreg}
			\item $H'\cup PCA^\diamond(r)\cup \mathcal{SF}'$ has a decomposition $\sC_2$ into $7r^\diamond$ edge-disjoint Hamilton cycles such that each cycle in $\sC_2$ contains precisely one of the special path systems in $\mathcal{SF}'$.\label{lm:robustdecomp-Hamcycles}
		\end{enumerate}
	\end{enumerate}
	The analogue holds if $(D,\cP,\cP', \cP^*,R,C,U,U')$ is an $(\ell',\frac{q}{f},k,m,\eps,d)$-bi-setup and $H$ is an $r$-regular bipartite digraph on the same vertex classes as $D$.
\end{lm}

Note that \cref{lm:robustdecomp-CA} and \cref{lm:robustdecomp-PCAreg} correspond to \cite[Lemma 12.1(i) and (ii.a)]{kuhn2013hamilton}. To see \cref{lm:robustdecomp-CAH}, observe that by the proof of \cite[Lemma 12.1]{kuhn2013hamilton} (see the proof of \cite[Lemma 11.2]{kuhn2013hamilton}), $CA^\diamond(r)\cup \mathcal{SF}$ is a ``chord absorber". By definition, $CA^\diamond(r)\cup \mathcal{SF}^*$ is also a ``chord absorber" in $D\cup \mathcal{SF}^*$. Moreover, since $r_3$ is very small, \cref{prop:bisetupedgesremoval} implies that $(D\cup \mathcal{SF}^*, \cP, \cP', \cP^*, R, C, U, U')$ also forms a (bi)-setup (with slightly worse parameters).
Thus, $\cC_1$ in \cref{lm:robustdecomp-CAH} can be obtained by applying the arguments of \cite[Lemma 11.2]{kuhn2013hamilton} with $D\cup \mathcal{SF}^*$ and $CA^\diamond(r)\cup \mathcal{SF}^*$ playing the role of $G$ and $CA^\diamond(r)\cup \mathcal{SF}$. Then, $\cC_2$ in \cref{lm:robustdecomp-Hamcycles} is the set of Hamilton cycles obtained by applying the arguments of \cite[Lemma 11.2]{kuhn2013hamilton} with $H'$ playing the role of $H_1$.

Here, \cref{lm:robustdecomp} is stated in terms of our simplified definition of a special factors (see \cref{sec:SPS}). The fact that $\mathcal{SF}\cup \mathcal{SF}'\subseteq D$ ensures that the special factors in $\mathcal{SF}\cup \mathcal{SF}'$ are indeed special factors in the stronger sense of \cite{kuhn2013hamilton} (the set of fictive edges is empty). As discussed above, the special factors in $\mathcal{SF}^*$ will be considered within the digraph $D\cup \mathcal{SF}^*$, so they also satisfy the original definition of special factors \cite{kuhn2013hamilton} with $D\cup \mathcal{SF}^*$ playing the role of $G$ (once again, with an empty set of fictive edges).

As discussed in \cref{sec:newrobustdecomp}, the key idea for deriving \cref{lm:newrobustdecomp} from \cref{lm:robustdecomp} is to consider equivalent special path systems. These will be constructed using the superregular pairs of the (bi)-setup.

\begin{lm}\label{cor:equivalentSPS}
	Let $0<\frac{1}{m}\ll \frac{1}{k}\ll \varepsilon\ll d\leq 1$ and $\frac{1}{k}\ll \frac{1}{f}, \frac{1}{\ell^*}\leq 1$ and $\frac{r\ell^*}{m}\ll d$. Let $(D,\cP,\cP', \cP^*,R,C,U,U')$ be an $(\ell',\ell^*,k,m,\eps,d)$-(bi)-setup and suppose that $\cP^*$ is an $\varepsilon$-uniform $\ell^*$-refinement of $\cP$. 
	Let $(h,j)\in [\ell^*]\times [f]$ and suppose that $SPS_1, \dots, SPS_r$ are $(\ell^*, f,h,j)$-special path systems with respect to $\cP^*$ and $C$. 
	Then, there exist edge-disjoint $(\ell^*,f,h,j)$-special path systems $SPS_1', \dots, SPS_r'\subseteq D$ with respect to $\cP^*$ and $C$ such that $SPS_i$ and $SPS_i'$ are equivalent for each $i\in [r]$.
\end{lm}

\begin{proof}
	Fix additional constants such that $\frac{r\ell^*}{m}, \varepsilon\ll \varepsilon_1\ll \varepsilon_2 \ll d$.
	Denote by $I=V_1\dots V_{k'}$ the $j^{th}$ interval in the canonical interval partition of $C$ into $f$ intervals. For each $i\in [k']$, denote by $V_{i,h}$ the $h^{th}$ subcluster of $V_i$ contained in $\cP^*$ and observe that \cref{def:ST-P}, \cref{def:ST-P*}, \cref{def:BST-P}, and \cref{def:BST-P*} imply that $|V_{i,h}|=\frac{m}{\ell^*}$.
	Suppose inductively that, for some $\ell\in [r]$, we have already constructed edge-disjoint $(\ell^*,f,h,j)$-special path systems $SPS_1', \dots, SPS_{\ell-1}'\subseteq D$ with respect to $\cP^*$ and $C$ such that $SPS_i$ and $SPS_i'$ are equivalent for each $i\in [\ell-1]$. We construct $SPS_\ell'$ using \cref{lm:regularitypaths} as follows.
	
	Let $D'\coloneqq D\setminus \bigcup_{i\in [\ell-1]}SPS_i$. We claim that $D'[V_{i,h}, V_{i+1, h}]$ is $[3\sqrt{\varepsilon_2}, \geq d]$-superregular for each $i\in [k'-1]$. Indeed, \cref{def:ST-C}, \cref{def:BST-C}, and \cref{lm:URefreg}\cref{lm:URef-supreg} imply that $D[V_{i,h}, V_{i+1,h}]$ is $[\varepsilon_1, \geq d]$-superregular for each $i\in [k'-1]$. Since $D'$ is obtained from $D$ by removing at most $r\leq \frac{\varepsilon_2m}{\ell^*}$ in- and outedges incident to each vertex, \cref{prop:epsremovingadding}\cref{prop:epsremovingadding-supreg} implies that $D'[V_{i,h}, V_{i+1, h}]$ is $[3\sqrt{\varepsilon_2}, \geq d]$-superregular for each $i\in [k'-1]$.
	Let $u_1, \dots, u_{\frac{m}{\ell^*}}$ be an enumeration of $V^+(SPS_\ell)$. For each $i\in [\frac{m}{\ell^*}]$, let $v_i$ denote the ending point of the path in $SPS_\ell$ which starts at $u_i$. By \cref{def:SPS-V+-}, $u_1, \dots, u_{\frac{m}{\ell^*}}$ and $v_1, \dots, v_{\frac{m}{\ell^*}}$ are enumerations of $V_{1,h}$ and $V_{k',h}$, respectively.
	Let $SPS_\ell'$ be the set of paths obtained by applying \cref{lm:regularitypaths} with $D'[\bigcup_{i\in [k']}V_{i,h}], \frac{m}{\ell^*},k', 3\sqrt{\varepsilon_2}$, and $V_{1,h}, \dots, V_{k', h}$ playing the roles of $D, m, k, \varepsilon$, and $V_1, \dots, V_k$. Then, $SPS_\ell'$ is an $(\ell^*,f,h,j)$-special path system with respect to $\cP^*$ and $C$ and, by construction, it is equivalent to $SPS_\ell$, as desired.
\end{proof}

\begin{proof}[Proof of \cref{lm:newrobustdecomp}]
	We prove the setup and bi-setup versions of \cref{lm:newrobustdecomp} in parallel.
	Fix additional constants such that $\varepsilon\ll \varepsilon_1\ll \varepsilon_2\ll \frac{1}{q}$ and $\frac{r_1}{m}\ll \varepsilon_3\ll d$.
	By \cref{lm:randomsetup,lm:URefrandom}, there exist edge-disjoint $D_1,D_2\subseteq D$ such that both $(D_1,\cP,\cP', \cP^*,R,C, U,U')$ and $(D_2,\cP,\cP', \cP^*,R,C, U,U')$ are $(\ell',\frac{q}{f},k,m,\varepsilon_1,\frac{d}{2})$-(bi)-setups.
	
	Since $\mathcal{SF}$ consists of special factors which are not necessarily edge-disjoint, we need to construct auxiliary special path systems which are edge-disjoint from each other and equivalent to those in $\mathcal{SF}$.
	For each $(h,j)\in [\frac{q}{f}]\times [f]$, denote by $SPS_{1,h,j}, \dots, SPS_{r_3,h,j}$ the $(\frac{q}{f},f,h,j)$-special path systems contained in $\mathcal{SF}$.
	
	\begin{claim}\label{claim:newSPS}
		For each $(h,j)\in [\frac{q}{f}]\times [f]$, there exist edge-disjoint $(\frac{q}{f},f,h,j)$-special path systems $SPS_{1,h,j}^\diamond, \dots, SPS_{r_3, h,j}^\diamond$ with respect to $\cP^*$ and $C$ such that $SPS_{i,h,j}$ and $SPS_{i,h,j}^\diamond$ are equivalent for each $i\in [r_3]$. In particular, $SF_i^\diamond \coloneqq \bigcup_{(h,j)\in [\frac{q}{f}]\times [f]}SPS_{i,h,j}^\diamond$ is a $(\frac{q}{f},f)$-special factor with respect to $\cP^*$ and $C$ for each $i\in [r_3]$.
	\end{claim}

	\begin{proofclaim}
		For the setup version of \cref{lm:newrobustdecomp}, let $K$ be the complete graph on $V(D)$; for the bi-setup version, let $K$ be the complete bipartite graph on the vertex classes of $D$. One can easily verify that $(K,\cP, \cP', \cP^*, R,C, U,U')$ is an $(\ell',\frac{q}{f},k,m,\varepsilon,d)$-(bi)-setup ((super)regular pairs exist by \cref{prop:almostcompleteeps}).
		For each $(h,j)\in [\frac{q}{f}]\times [f]$, let $SPS_{1,h,j}^\diamond, \dots, SPS_{r_3, h,j}^\diamond$ be the special path systems obtained by applying \cref{cor:equivalentSPS} with $K, \frac{q}{f},r_3$, and $SPS_{1,h,j}, \dots, SPS_{r_3,h,j}$ playing the roles of $D, \ell^*,r$, and $SPS_1, \dots, SPS_r$.
	\end{proofclaim}
	
	Let $\mathcal{SF}^\diamond$ be the union of the $r_3$ edge-disjoint $(\frac{q}{f},f)$-special factors with respect to $\cP^*$ and $C$ obtained by applying \cref{claim:newSPS}.
	Let $D_1'\coloneqq D_1\cup \mathcal{SF}^\diamond$ and observe that $D_1'$ is obtained from $D_1$ by adding at most $r_3$ in- and outedges incident to each vertex (recall from \cref{def:SF} that special factors are digraphs of maximum semidegree $1$). Thus, \cref{prop:bisetupedgesremoval} implies that $(D_1',\cP,\cP', \cP^*,R,C,U,U')$ is still an $(\ell',\frac{q}{f},k,m,\varepsilon_2,\frac{d}{4})$-(bi)-setup.
	Let $CA^\diamond(r)$ be the spanning subdigraph of $D_1'$ obtained by applying \cref{lm:robustdecomp} with $D_1', \mathcal{SF}^\diamond, \varepsilon_2$, and $\frac{d}{4}$ playing the roles of $D, \mathcal{SF}, \varepsilon$, and $d$. Note that \cref{lm:robustdecomp}\cref{lm:robustdecomp-CA} implies that $CA^\diamond(r)\subseteq D_1\subseteq D$.
	
	Since the special factors in $\mathcal{SF}'$ may not be edge-disjoint subdigraphs of $D_1''\coloneqq D_1'\setminus (CA^\diamond(r)\cup \mathcal{SF}^\diamond)=D_1\setminus (CA^\diamond(r)\cup \mathcal{SF}^\diamond)$, we need to construct auxiliary special path systems which are equivalent to those in $\mathcal{SF}'$. For each $(h,j)\in [1]\times [7]$, let $SPS_{1,h,j}', \dots, SPS_{r^\diamond,h,j}'$ denote the $(1,7,h,j)$-special path systems contained in $\mathcal{SF}'$.
	
	\begin{claim}\label{claim:newSPS'}
		For each $(h,j)\in [1]\times [7]$, there exist edge-disjoint $(1,7,h,j)$-special path systems $SPS_{1,h,j}'', \dots, SPS_{r^\diamond, h,j}''\subseteq D_1''$ with respect to $\cP$ and $C$ such that $SPS_{i,h,j}'$ and $SPS_{i,h,j}''$ are equivalent for each $i\in [r^\diamond]$. In particular, $SF_i'' \coloneqq \bigcup_{(h,j)\in [1]\times [7]}SPS_{i,h,j}$ is a $(1,7)$-special factor with respect to $\cP$ and $C$ for each $i\in [r^\diamond]$.
	\end{claim}
	
	\begin{proofclaim}
		By \cref{fact:bisetupP}, $(D_1,\cP, \cP', \cP, R,C,U,U')$ is an $(\ell', 1, k,m, \varepsilon_1, \frac{d}{2})$-(bi)-setup.
		By \cref{lm:robustdecomp}\cref{lm:robustdecomp-CA} and \cref{def:SF}, $D_1''$ is obtained from $D_1$ removing at most $r_1+r_2+r_3$ in- and outedges at each vertex, so \cref{prop:bisetupedgesremoval} implies that $(D_1'',\cP, \cP', \cP, R,C,U,U')$ is still an $(\ell', 1, k,m, \varepsilon_3, \frac{d}{4})$-(bi)-setup.
		For each $(h,j)\in [1]\times [7]$, let $SPS_{1,h,j}'', \dots, SPS_{r^\diamond, h,j}''$ be the special path systems obtained by applying \cref{cor:equivalentSPS} with $D_1'', \cP, \varepsilon_3, \frac{d}{4}, 1, 7, r^\diamond$, and $SPS_{1,1,j}', \dots, SPS_{r^\diamond, 1,j}'$ playing the roles of $D, \cP^*, \varepsilon, d,  \ell^*, f,r$, and $SPS_1, \dots, SPS_r$.
	\end{proofclaim}

	Let $\mathcal{SF}''$ be the union of the $r^\diamond$ edge-disjoint $(1,7)$-special factors with respect to $\cP$ and $C$ obtained by applying \cref{claim:newSPS'}. Let $PCA^\diamond(r)$ be the spanning subdigraph of $D_1$ obtained by applying \cref{lm:robustdecomp}\cref{lm:robustdecomp-PCA} with $\mathcal{SF}''$ playing the role of $\mathcal{SF}'$. By \cref{lm:robustdecomp}\cref{lm:robustdecomp-PCAreg}, $PCA^\diamond(r)\subseteq D_1''\setminus \mathcal{SF}''\subseteq D_1\subseteq D$.
	Define $D^{\rm rob}\coloneqq CA^\diamond(r)\cup PCA^\diamond(r) \subseteq D$ and observe that \cref{lm:robustdecomp}\cref{lm:robustdecomp-CA,lm:robustdecomp-PCAreg} imply that $D^{\rm rob}$ is $(r_1+r_2+5r^\diamond)$-regular, as desired.
	
	Let $H$ be an $r$-regular digraph on $V(D)$. For the bi-setup version of \cref{lm:newrobustdecomp}, suppose furthermore that $H$ is a bipartite digraph on the same vertex classes as $D$.	
	It remains to show that the multidigraph $H\cup D^{\rm rob}\cup \mathcal{SF}\cup \mathcal{SF}'$ has a decomposition $\sC$ into $s'$ edge-disjoint Hamilton cycles such that each Hamilton cycle in $\sC$ contains precisely one of the special path systems in the multidigraph $\mathcal{SF}\cup \mathcal{SF}'$.
	We will use \cref{lm:robustdecomp}\cref{lm:robustdecomp-CAH,lm:robustdecomp-Hamcycles}.
	
	\begin{claim}\label{claim:newSPS*}
		For each $(h,j)\in [\frac{q}{f}]\times [f]$, there exist edge-disjoint $(\frac{q}{f},f,h,j)$-special path systems $SPS_{1,h,j}^*, \dots, SPS_{r_3, h,j}^*$ with respect to $\cP^*$ and $C$ which are edge-disjoint from $CA^\diamond(r)\cup H$ and such that $SPS_{i,h,j}$ and $SPS_{i,h,j}^*$ are equivalent for each $i\in [r_3]$. In particular, $SF_i^* \coloneqq \bigcup_{(h,j)\in [\frac{q}{f}]\times [f]}SPS_{i,h,j}^*$ is a $(\frac{q}{f},f)$-special factor with respect to $\cP^*$ and $C$ for each $i\in [r_3]$.
	\end{claim}

	\begin{proofclaim}
		Let $D_2'\coloneqq D_2\setminus H$. Since $CA^\diamond(r)\subseteq D_1$, note that $D_2'$ is edge-disjoint from $CA^\diamond(r)\cup H$.  By \cref{prop:bisetupedgesremoval}, $(D_2',\cP, \cP', \cP^*, R,C, U,U')$ is an $(\ell',\frac{q}{f},k,m,\varepsilon_2,\frac{d}{4})$-(bi)-setup.
		For each $(h,j)\in [\frac{q}{f}]\times [f]$, let $SPS_{1,h,j}^*, \dots, SPS_{r_3, h,j}^*$ be the $(\frac{q}{f},f)$-special path systems with respect to $\cP^*$ and $C$ obtained by applying \cref{cor:equivalentSPS} with $D_2', \varepsilon_2, \frac{d}{4}, \frac{q}{f}$, and $r_3$ playing the roles of $D, \varepsilon, d, \ell^*$, and $r$.
	\end{proofclaim}
	
	Let $\mathcal{SF}^*$ be the union of the $r_3$ edge-disjoint $(\frac{q}{f},f)$-special factors with respect to $\cP^*$ and $C$ obtained by applying \cref{claim:newSPS*}. 
	By \cref{lm:robustdecomp}\cref{lm:robustdecomp-CAH}, there exists a set $\sC_1$ of $rfk$ edge-disjoint Hamilton cycles such that $E(H)\cup E(\mathcal{SF}^*)\subseteq E(\sC_1)\subseteq  E(CA^\diamond(r))\cup E(H)\cup E(\mathcal{SF}^*)$ and each cycle in $\sC_1$ contains precisely one of the special path systems contained in $\mathcal{SF}^*$. Denote $H'\coloneqq CA^\diamond(r)\setminus E(\sC_1)$.
	By \cref{lm:robustdecomp}\cref{lm:robustdecomp-Hamcycles}, $H'\cup PCA^\diamond(r)\cup \mathcal{SF}''$ has a decomposition $\sC_2$ into $7r^\diamond$ edge-disjoint Hamilton cycles such that each cycle in $\sC_2$ contains precisely one of the special path systems in $\mathcal{SF}''$.
	Altogether, $\sC_1\cup \sC_2$ forms a decomposition of the multidigraph $H\cup D^{\rm rob}\cup \mathcal{SF}^*\cup \mathcal{SF}''$ into $s'$ Hamilton cycles such that each cycle in $\sC_1\cup \sC_2$ contains precisely one of the special path systems in $\mathcal{SF}^*\cup \mathcal{SF}''$.
	By \cref{claim:newSPS*,claim:newSPS'} and \cref{fact:equivalentP}, $\sC_1\cup \sC_2$ induces a decomposition $\sC$ of the multidigraph $H\cup D^{\rm rob}\cup \mathcal{SF}\cup \mathcal{SF}'$ into $s'$ edge-disjoint Hamilton cycles such that each Hamilton cycle in $\sC$ contains precisely one of the special path systems contained in the multidigraph $\mathcal{SF}\cup \mathcal{SF}'$.
\end{proof}

\subsection{Proof of Lemma \ref{lm:cyclerobustdecomp}}

We know derive the blow-up cycle version of the robust decomposition lemma using the strategy presented in \cref{sec:cyclerobustdecomp-sketch}.

\begin{proof}[Proof of \cref{lm:cyclerobustdecomp}]
	By \cref{fact:equivalentP} and \cref{def:ESPS}, we may assume without loss of generality that all extended special path systems contained in the multidigraph $\mathcal{ESF}\cup \mathcal{ESF}'$ are friendly.
	For each $(h,i,j)\in [\frac{q}{f}]\times [K]\times [f]$, denote by $ESPS_{1,h,i,j}, \dots, ESPS_{r_3,h,i,j}$ the $r_3$ friendly $(\frac{q}{f},K,f,h,i,j)$-extended special path systems contained in $\mathcal{ESF}$.
	For each $(h,i,j)\in [1]\times [K]\times [7]$, denote by $ESPS_{1,h,i,j}', \dots, ESPS_{r^\diamond,h,i,j}'$ the $r^\diamond$ friendly $(\frac{q}{f},K,f,h,i,j)$-extended special path systems contained in $\mathcal{ESF}'$.
	For each $i\in [K]$, define
	\begin{equation*}
		\mathcal{ESF}_i\coloneqq \bigcup_{(\ell,h,j)\in [r_3]\times[\frac{q}{f}]\times [f]}ESPS_{\ell,h,i,j} \quad \text{and} \quad \mathcal{ESF}_i'\coloneqq \bigcup_{(\ell,h,j)\in [r^\diamond]\times[1]\times [7]}ESPS_{\ell,h,i,j}'.
	\end{equation*}
	
	\begin{steps}
		\item \textbf{Applying the robust decomposition lemma in each contracted pair.}
		Let $i\in [K]$. Denote by $\tD_i$ the $M_i$-contraction of $D[U_i, U_{i+1}]$.
		By \cref{def:CST-ST}, $(\tD_i, \cP_i, \cP_i', \cP_i^*, R_i, C^i, U^i, U'^i)$ is an $(\ell', \frac{q}{f}, k,m, \varepsilon,d)$-setup with an empty exceptional set. We construct the required special factors for applying the robust decomposition lemma in $\tD_i$ as follows.
		For each $(h,j)\in [\frac{q}{f}]\times [f]$, let $SPS_{1,h,i,j}, \dots, SPS_{r_3, h,i,j}$ be obtained from the $M_i$-contractions of $ESPS_{1,h,i,j}[U_i,U_{i+1}], \dots$, $ESPS_{r_3,h,i,j}[U_i,U_{i+1}]$ by deleting all isolated vertices. For each $(h,j)\in [\frac{q}{f}]\times [f]$, \cref{def:FESPS-SPS} implies that $SPS_{1,h,i,j}, \dots, SPS_{r_3, h,i,j}$ are $(\frac{q}{f},f,h,j)$-special path systems with respect to $\cP_i^*$ and $C^i$.
		For each $\ell\in [r_3]$, let $SF_{\ell,i}\coloneqq \bigcup_{(h,j)\in [\frac{q}{f}]\times [f]}SPS_{\ell,h,i,j}$ and observe that $SF_{\ell,i}$ is a $(\frac{q}{f},f)$-special factor with respect to $\cP_i^*$ and $C^i$. Define a multidigraph $\mathcal{SF}_i$ by $\mathcal{SF}_i\coloneqq \bigcup_{\ell\in [r_3]}SF_{\ell,i}$. Define $\mathcal{SF}_i'$ analogously.
		Let $\tD_i^{\rm rob}$ be the spanning subdigraph of $\tD_i$ obtained by applying \cref{lm:newrobustdecomp} with $\tD_i, \cP_i, \cP_i', \cP_i^*, R_i, C^i, U^i, U'^i, \mathcal{SF}_i$, and $\mathcal{SF}_i'$ playing the roles of $D, \cP, \cP', \cP^*, R, C, U, U', \mathcal{SF}$, and $\mathcal{SF}'$.
		
		\item \textbf{Constructing the robustly decomposable digraph.}
		For each $i\in [K]$, let $D_i^{\rm rob}$ be obtained from the $M_i$-expansion of $\tD_i^{\rm rob}$ by orienting all the edges from $U_i$ to $U_{i+1}$. Define $D^{\rm rob}\coloneqq \bigcup_{i\in [K]}D_i^{\rm rob}$. First, observe that $D^{\rm rob}$ is a regular spanning subdigraph of~$D$.
		
		\begin{claim}
			$D^{\rm rob}$ is an $(r_1+r_2+5r^\diamond)$-regular spanning subdigraph of $D$.
		\end{claim}
		
		\begin{proofclaim}
			By definition, $V(D_i^{\rm rob})=U_i\cup U_{i+1}$ for each $i\in [K]$ and so $D^{\rm rob}$ is spanning. For each $i\in [K]$, $\tD_i^{\rm rob}\subseteq \tD_i$ and \cref{fact:contractinguncontracting} implies that the $M_i$-expansion of $\tD_i$ is a subdigraph of $D[U_i, U_{i+1}]$. Thus, \cref{fact:subgraphcontract}\cref{fact:subgraphcontract-subgraphD} implies that $D_i^{\rm rob}[U_i, U_{i+1}]\subseteq D[U_i, U_{i+1}]$ for each $i\in [K]$ and so $D^{\rm rob}\subseteq D$.
			Let $i\in [K]$ and $v\in U_i$. By construction, 
			\begin{align*}
				d_{D^{\rm rob}}^+(v)=d_{D_i^{\rm rob}}^+(v)\stackrel{\text{\cref{fact:Nuncontract}\cref{fact:Nuncontract-A}}}{=}d_{\tD_i^{\rm rob}}^+(v)\stackrel{\text{\cref{lm:newrobustdecomp}}}{=}r_1+r_2+5r^\diamond.
			\end{align*}
			Let $v'$ be the unique neighbour of $v$ in $M_{i-1}$. Then,
			\begin{align*}
				d_{D^{\rm rob}}^-(v)=d_{D_{i-1}^{\rm rob}}^-(v)\stackrel{\text{\cref{fact:Nuncontract}\cref{fact:Nuncontract-B}}}{=}d_{\tD_{i-1}^{\rm rob}}^-(v')\stackrel{\text{\cref{lm:newrobustdecomp}}}{=}r_1+r_2+5r^\diamond.
			\end{align*}
			Thus, $D^{\rm rob}$ is $(r_1+r_2+5r^\diamond)$-regular.
		\end{proofclaim}
		
		Moreover, observe that each $D_i^{\rm rob}$ can decompose a sparse digraph in the pair $(U_i, U_{i+1})$ into perfect matchings.
		
		\begin{claim}\label{claim:Mdecomp}
			Let $i\in [K]$. Let $H_i$ be a bipartite digraph on vertex classes $U_i$ and $U_{i+1}$ which is edge-disjoint from $D_i^{\rm rob}$ and such that $E(H_i)\cap \{uv\mid vu\in M_i\}=\emptyset$. Suppose that $H_i[U_i, U_{i+1}]$ is $r$-regular and $H_i[U_{i+1}, U_i]$ is empty.
			Denote by $\tH_i$ the $M_i$-contraction of $H_i[U_i, U_{i+1}]$. 
			Define a multidigraph $\cH_i$ by $\cH_i\coloneqq H_i\cup D_i^{\rm rob}\cup \mathcal{ESF}_i(U_i, U_{i+1})\cup \mathcal{ESF}_i'(U_i, U_{i+1})$ and define a multidigraph $\widetilde{\cH}_i$ by $\widetilde{\cH}_i\coloneqq \tH_i\cup \tD_i^{\rm rob}\cup \mathcal{SF}_i\cup \mathcal{SF}_i'$.
			Then, the following hold.
			\begin{enumerate}
				\item The multidigraph $\cH_i$ can be obtained from the $M_i$-expansion of the multidigraph $\cH_i'$ by orienting all the edges from $U_i$ to $U_{i+1}$.\label{claim:Mdecomp-expansion}
				\item The multidigraph $\cH_i$ has a decomposition $\sM_i$ into $s'$ perfect matchings from $U_i$ to $U_{i+1}$ such that the following hold for each $M\in \sM_i$.\label{claim:Mdecomp-decomp}
				\begin{enumerate}[label=\rm(\alph*),ref=\rm(\roman{enumi}.\alph*)]
					\item $M\cup M_i$ forms a Hamilton cycle on $U_i\cup U_{i+1}$.\label{claim:Mdecomp-Ham}
					\item There exists an extended special path system $ESPS_M$ in the multidigraph $\mathcal{ESF}_i\cup \mathcal{ESF}_i'$ such that $M\cap (E(\mathcal{ESF})\cup E(\mathcal{ESF}'))= E_{ESPS_M}(U_i,U_{i+1})$. (I.e.\ $M$ contains precisely one of the ``special path system parts" in $\mathcal{ESF}\cup \mathcal{ESF}'$.)\label{claim:Mdecomp-SF}
				\end{enumerate}
			\end{enumerate}
		\end{claim}
	
		\begin{proofclaim}
			First, we show \cref{claim:Mdecomp-expansion}. By \cref{fact:contractinguncontracting} and assumption, $H_i[U_i, U_{i+1}]$ is the $M_i$-expansion of $\tH_i$. Moreover, recall that $H[U_{i+1}, U_i]$ is empty, so $H_i$ can be obtained from the $M_i$-expansion of $\tH_i$ by orienting all the edges from $U_i$ to $U_{i+1}$. By definition, $D_i^{\rm rob}$ is obtained from the $M_i$-expansion of $\tD_i^{\rm rob}$ by orienting all the edges from $U_i$ to $U_{i+1}$.
			By \cref{def:FESPS}, each extended special path system $ESPS$ in the multidigraph $\mathcal{ESF}_i\cup \mathcal{ESF}_i'$ satisfies $ESPS[U_i, U_{i+1}]\cap M_i[U_{i+1}, U_i]=\emptyset$ (otherwise, \cref{def:FESPS-M} would imply that $ESPS$ is not a linear forest). Thus, \cref{fact:contractinguncontracting} implies that the $M_i$-expansions of $\mathcal{SF}_i$ and $\mathcal{SF}_i'$ are $\mathcal{ESF}_i[U_i, U_{i+1}]$ and $\mathcal{ESF}_i'[U_i, U_{i+1}]$. Altogether, this implies that the multidigraph $\cH_i$ can indeed be obtained from the $M_i$-expansion of the multidigraph $\widetilde{\cH}_i$ by orienting all the edges from $U_i$ to $U_{i+1}$, as desired.
			
			For \cref{claim:Mdecomp-decomp}, we decompose the multidigraph $\cH_i$ as follows. By \cref{fact:Ncontract}\cref{fact:Ncontract-N}, $\tH_i$ is an $r$-regular digraph on $U_i$ and, by \cref{fact:subgraphcontract}\cref{fact:subgraphcontract-disjointG}, $\tH_i$ is edge-disjoint from $\tD_i^{\rm rob}$. Thus, \cref{lm:newrobustdecomp} implies that the multidigraph $\widetilde{\cH}_i$ has a decomposition $\widetilde{\sC}_i$ into $s'$ Hamilton cycles on $U_i$ such that each cycle in $\widetilde{\sC}_i$ contains precisely one of the special path systems contained in the multidigraph $\mathcal{SF}_i\cup \mathcal{SF}_i'$.
			Let $\sM_i$ consist of the digraphs obtained by orienting all the edges from $U_i$ to $U_{i+1}$ in the $M_i$-expansions of the cycles in $\widetilde{\sC}_i$.
			
			Then, \cref{fact:contractingHamcycle} implies that \cref{claim:Mdecomp-Ham} holds and $\sM_i$ is a set of $s'$ perfect matchings from $U_i$ to $U_{i+1}$. By \cref{fact:subgraphcontract}\cref{fact:subgraphcontract-disjointD}, the matchings in $\sM_i$ are edge-disjoint.
			Thus, \cref{claim:Mdecomp-expansion} implies that $\cM_i$ is a decomposition of $\cH_i$, as desired.
			
			For \cref{claim:Mdecomp-SF}, let $M\in \sM_i$ and denote by $C\in \widetilde{\sC}_i$ its corresponding cycle. By \cref{lm:newrobustdecomp}, the multidigraph $\mathcal{SF}_i\cup \mathcal{SF}_i'$ contains a special path system $SPS$ such that $E(\widetilde{\sC}_i)\cap (E(\mathcal{SF}_i)\cup E(\mathcal{SF}_i'))=E(SPS)$. Let $SPS'$ be obtained from the $M_i$-expansion of $SPS$ by orienting all the edges from $U_i$ to $U_{i+1}$. By \cref{fact:subgraphcontract}\cref{fact:subgraphcontract-disjointD}, $M\cap (E(\mathcal{ESF})\cup E(\mathcal{ESF}'))=E(SPS')$. By construction, there exists an extended special path system $ESPS$ in the multidigraph $\mathcal{ESF}\cup \mathcal{ESF}'$ such that $SPS$ is obtained from the $M_i$-contraction of $ESPS[U_i, U_{i+1}]$ by deleting isolated vertices. By \cref{def:FESPS}, $ESPS[U_i, U_{i+1}]\cap M_i[U_{i+1}, U_i]=\emptyset$ and so \cref{fact:contractinguncontracting} implies that $E(SPS')= E_{ESPS}(U_i, U_{i+1})$ and so \cref{claim:Mdecomp-SF} holds.
		\end{proofclaim}
	
		\item \textbf{Decomposing $H\cup D^{\rm rob}\cup \mathcal{ESF}\cup \mathcal{ESF}'$.}
		For each $i\in [K]$, let $H_i$ be digraph on $U_i\cup U_{i+1}$ defined by $E(H_i)\coloneqq E_H(U_i, U_{i+1})$. Since $H$ is a blow-up $C_K$ with vertex partition $\cU$, we have $H=\bigcup_{i\in [K]}H_i$ and so it is enough to show that, for each $i\in [K]$, the multidigraph $H_i\cup D_i^{\rm rob}\cup \mathcal{ESF}_i\cup \mathcal{ESF}_i'$ has a decomposition $\sC_i$ into $s'$ Hamilton cycles such that each cycle in $\sC_i$ contains precisely one of the extended special path systems in the multidigraph $\mathcal{ESF}_i\cup \mathcal{ESF}_i'$.
		
		Let $i\in [K]$. Denote by $\sM_i$ the decomposition of the multidigraph $H_i\cup D_i^{\rm rob}\cup \mathcal{ESF}_i(U_i,U_{i+1})\cup \mathcal{ESF}_i'(U_i,U_{i+1})$ obtained by applying \cref{claim:Mdecomp}\cref{claim:Mdecomp-decomp}.
		Let $\sC_i$ be obtained from $\sM_i$ by replacing each $M\in \sM_i$ by the digraph $M\cup ESPS_M$ (recall that $ESPS_M$ was defined in \cref{claim:Mdecomp}\cref{claim:Mdecomp-SF}). 
		Then, \cref{fact:cyclerobustdecomp-sketch}, \cref{def:FESPS-M}, and \cref{claim:Mdecomp}\cref{claim:Mdecomp-Ham} imply that $\sC_i$ is a Hamilton decomposition of the multidigraph $H_i\cup D_i^{\rm rob}\cup \mathcal{ESF}_i\cup \mathcal{ESF}_i'$. By \cref{claim:Mdecomp}\cref{claim:Mdecomp-SF}, each cycle in $\sC_i$ contains precisely one of the extended special path systems in the multidigraph $\mathcal{ESF}_i\cup \mathcal{ESF}_i'$.\qedhere
	\end{steps}
\end{proof}

\onlyinsubfile{\bibliographystyle{abbrv}
	\bibliography{Bibliography/Bibliography}}

	\section{The preprocessing step: proof of Lemma \ref{lm:preprocessing}}\label{app:preprocessing}
	
	\onlyinsubfile{
		\appendix
		\setcounter{section}{2}
\section{The Preprocessing step}}

In this appendix, we discuss how to derive the preprocessing lemma for bipartite digraphs (\cref{lm:preprocessing}). First, note that \cref{lm:preprocessing} is a direct corollary of the bipartite versions of \cite[Lemma 8.6]{kuhn2013hamilton} (which guarantees the existence of $PG$) and \cite[Corollary 8.5]{kuhn2013hamilton} (which verifies the properties of $PG$). \Cref{fig:preprocessing} illustrates the overall structure of the proofs of \cite[Corollary 8.5 and Lemma 8.6]{kuhn2013hamilton}.
In this appendix, we will discuss the bipartite versions of the dark grey lemmas from \cref{fig:preprocessing}. 
The white lemmas from \cref{fig:preprocessing} can be used in their original versions. The statements of the light grey lemmas from \cref{fig:preprocessing} can be adapted simply by replacing a consistent system by a consistent bi-system whose exceptional set forms an independent set, while their proofs follow immediately from the white lemmas from \cref{fig:preprocessing} and the bipartite versions of the grey lemmas from \cref{fig:preprocessing}.

\tikzstyle{blue} = [rectangle, rounded corners, draw, fill=gray!80]
\tikzstyle{purple} = [rectangle, rounded corners, draw, fill=gray!40]
\tikzstyle{orange} = [rectangle, rounded corners, draw, fill=white]
\pgfdeclarelayer{background}    
\pgfsetlayers{background,main}

\begin{figure}[htb]
    \centering
    \begin{tikzpicture}[scale=0.5]
        \node [blue] (53) {Lemma 5.3};
        \node [orange, below of=53, node distance=1cm] (47) {Lemma 4.7};
        \node [blue, right of=47, node distance=3.5cm] (75) {Lemma 7.5};
        \node [orange, above of=75, node distance=2cm] (43) {Proposition 4.3};
        \node [purple, right of=75, node distance=3.5cm] (76) {Lemma 7.6};
        \node [purple, above of=76, node distance=1cm] (71) {Lemma 7.1};
        \node [orange, above of=71, node distance=1cm] (82) {Lemma 8.2};
        \node [blue, above of=82, node distance=1cm] (52) {Lemma 5.2};
        \node [purple, right of=82, node distance=3.5cm] (86) {Lemma 8.6};
        \node [orange, right of=76, node distance=3.5cm] (414) {Proposition 4.14};
        \node [orange, below of=414, node distance=1cm] (65) {Lemma 6.5};
        \node [blue, below of=47, node distance=2cm] (72) {Lemma 7.2};
        \node [orange, right of=72, node distance=3.5cm] (73) {Proposition 7.3};
        \node [orange, below of=72, node distance=2cm] (74) {Observation 7.4};
        \node [orange, below of=74, node distance=1cm] (412) {Lemma 4.12};
        \node [orange, below of=412, node distance=1cm] (64) {Lemma 6.4};
        \node [blue, right of=74, node distance=3.5cm] (83) {Lemma 8.3};
        \node [orange, below of=83, node distance=1cm] (81) {Lemma 8.1};
        \node [purple, right of=83, node distance=3.5cm] (84) {Lemma 8.4};
        \node [purple, right of=84, node distance=3.5cm] (85) {Corollary 8.5};
        \node [orange, below of=84, node distance=2cm] (48) {Proposition 4.8};
        \node [orange, below of=81, node distance=1cm] (61) {Proposition 6.1};
        \begin{pgfonlayer}{background}
            \draw [->,thick] (53) -- (75);
            \draw [->,thick] (47) -- (75);
            \draw [->,thick] (75) -- (76);
            \draw [->,thick] (43) -- (75);
            \draw [->,thick] (43) -- (71);
            \draw [->,thick] (43) -- (76);
            \draw [->,thick] (76) -- (86);
            \draw [->,thick] (71) -- (86);
            \draw [->,thick] (82) -- (86);
            \draw [->,thick] (52) -- (86);
            \draw [->,thick] (414) -- (76);
            \draw [->,thick] (65) -- (76);
            \draw [->,thick] (72) -- (75);
            \draw [->,thick] (72) -- (83);
            \draw [->,thick] (73) -- (76);
            \draw [->,thick] (73) -- (83);
            \draw [->,thick] (74) -- (83);
            \draw [->,thick] (412) -- (83);
            \draw [->,thick] (64) -- (83);
            \draw [->,thick] (83) -- (84);
            \draw [->,thick] (81) -- (84);
            \draw [->,thick] (84) -- (85);
            \draw [->,thick] (48) -- (84);
            \draw [->,thick] (61) -- (84);
        \end{pgfonlayer}
    \end{tikzpicture}
    \caption{The structure of the proofs of \cite[Corollary 8.5 and Lemma 8.6]{kuhn2013hamilton}.\label{fig:preprocessing}}
\end{figure}
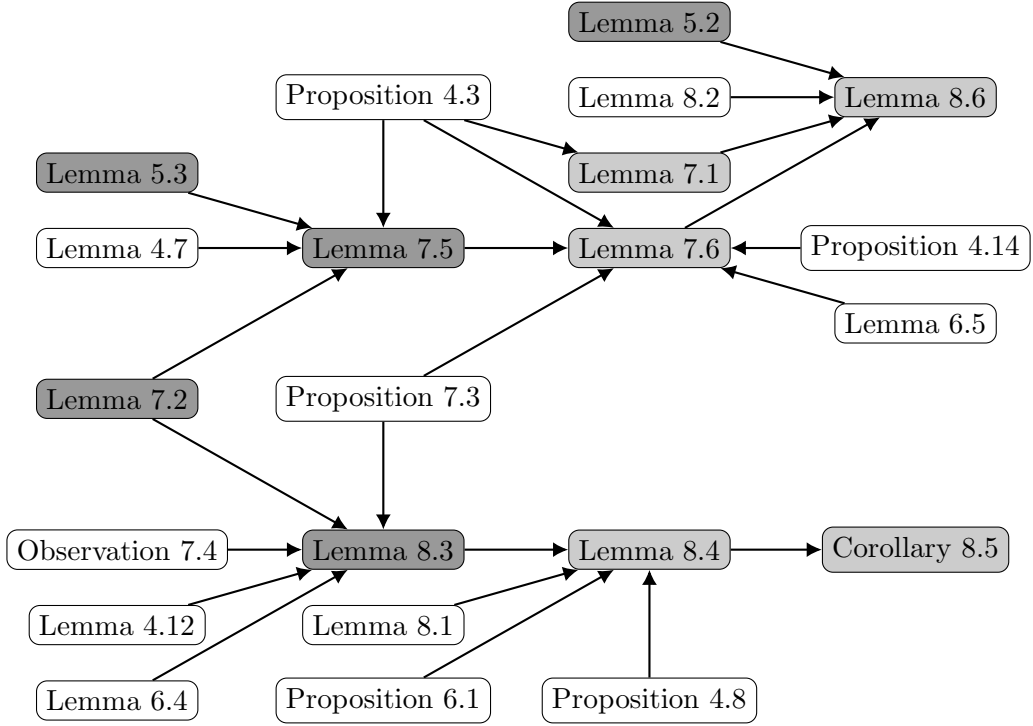

Note that \cref{lm:SF} corresponds to the bipartite version of \cite[Lemma 7.6]{kuhn2013hamilton} and \cref{lm:biprobblowup2} is the bipartite version of \cite[Lemma 5.3]{kuhn2013hamilton}. As already mentioned, \cref{lm:biprobblowup2} can be proven using the same arguments as in the proof of \cite[Lemma 5.3]{kuhn2013hamilton}, so we omit the details.
Finally, note that we will use the bipartite version of \cite[Lemma 7.2]{kuhn2013hamilton} (\cref{buildchord} below) to derive \cref{lm:BU} (that is, the bipartite version of \cite[Lemma 9.1]{kuhn2013hamilton}) at the end of this appendix.

\begin{lm}[{Bipartite version of \cite[Lemma 5.2]{kuhn2013hamilton}}]\label{regrobust}
	Let $0<\frac{1}{n}\ll\varepsilon\ll \nu \leq \tau \ll \delta<1$ and $\frac{1}{n} \ll \xi\leq \frac{\nu^2}{3}$. Let~$D$ be a balanced bipartite digraph on vertex classes $A$ and $B$ of size~$n$. Suppose that $\delta^0(D)\geq \delta n$ and that $D$ is a bipartite robust $(\nu,\tau)$-outexpander with bipartition $(A,B)$. For each $v\in V(D)$, let $n_v^+,n_v^-\in\mathbb{N}$
	be such that $(1-\varepsilon)\xi n\le n_v^+, n_v^-\leq (1+\varepsilon)\xi n$ and such that both $\sum_{v\in A} n_v^\pm=\sum_{v\in B} n_v^\mp$.
	Then, $D$ contains a spanning subdigraph $D'$ such that $d_{D'}^+(v)=n_v^+$ and $d_{D'}^-(v)=n_v^-$ for each $v\in V(D)$.
\end{lm}

\begin{proof}
	By symmetry, it is enough to find a spanning subgraph $H\subseteq D[A, B]$ satisfying $d_D(a)=n_a^+$ for each $a\in A$ and $d_H(b)=n_b^-$ for each $b\in B$.
	Let $N$ be the flow network obtained from $D(A,B)$ by giving each edge of $D(A,B)$ capacity $1$, by adding a source $s$ which is
	joined to every vertex $a\in A$ with an outedge of capacity $n_a^+$, and by adding a sink $t$ which is
	joined to every vertex $b\in B$ with an inedge of capacity $n_b^-$. Let $r\coloneqq \sum_{v\in A} n_v^+=\sum_{v\in B} n_v^-$.
	Using similar arguments as in \cite[Lemma 5.2]{kuhn2013hamilton}, one can show that any $s-t$ cut in $N$ has capacity at least $r$%
		\COMMENT{Let $\cC$ be a minimal cut. Let $S$ be the set of vertices $a\in A$ for which $sa\notin \mathcal{C}$.
		Similarly, let $T$ be the set of vertices $b\in B$ for which $bt\notin \mathcal{C}$.
		Let $S'\coloneqq A\setminus S$ and $T'\coloneqq B\setminus T$.
		Then, the capacity of $\cC$ is
		\[c\coloneqq \sum_{v\in S'} n_v^++e_D(S,T)+ \sum_{v\in T'} n_v^-.\]
		If $|T'|\geq |S|+\frac{\nu n}{2}$, then
		\[c\geq \sum_{v\in S'} n_v^++\sum_{v\in T'} n_v^-
		\geq (1-\varepsilon)\xi n (|S'|+|T'|)\geq (1-\varepsilon)\xi n (n+\frac{\nu n}{2})\geq (1+\varepsilon)\xi n^2\geq r,\]
		as required.\\		
		If $|T'|\le |S|+\frac{\nu n}{2}$ and $\tau n<|S|<(1-\tau)n$, then $|T\cap RN^+_{\nu,D}(S)|\geq \frac{\nu n}{2}$ and so
		\[c\geq  e_D(S,T)\geq \frac{\nu^2 n^2}{2}\geq (1+\varepsilon)\xi n^2\geq r,\] as required.\\		
		If $|T'|\leq |S|+\frac{\nu n}{2}$ and $|S|\leq \tau n$, then
		\[c\geq \sum_{v\in S'} n_v^++|S|(\delta n-|T'|)\geq \sum_{v\in S'} n_v^+ +|S|(1+\varepsilon)\xi n\geq r,\]
		as required.\\		
		Finally, if $|S|\geq (1-\tau )n$, then
		\[c\geq \sum_{v\in T'} n_v^-+|T|(\delta n-|S'|)\geq \sum_{v\in T'} n_v^- +|T|(1+\varepsilon)\xi n\geq r,\]
		as required.}.	
	Thus, the max-flow min-cut theorem implies that $N$ has an $s-t$ flow of value $r$. This flow corresponds to the desired spanning subgraph $H\subseteq D[A,B]$.
\end{proof}

By definition, a bipartite digraph $R$ can only contain a chord sequence between clusters which belong to a common vertex class of $R$. Such chord sequences can be constructed using the same arguments as in \cite[Lemma 7.2]{kuhn2013hamilton}, so we omit the details.

\begin{lm}[Bipartite version of {\cite[Lemma 7.2]{kuhn2013hamilton}}] \label{buildchord}
	Let $0<\frac{1}{k}\ll \nu \ll \tau \ll \delta <1$.
	Let $R$ be a balanced bipartite digraph on vertex classes $\cA$ and $\cB$ of size $k$. Suppose that $R$ is a bipartite robust $(\nu, \tau)$-outexpander with bipartition $(A,B)$ and that $\delta^0(R)\geq \delta k$. Let $C$ be a Hamilton cycle in $R$.
	Let $\cV\subseteq V(R)$ satisfy $|\cV| \leq \frac{\nu k}{4}$.
	Suppose that $A_1,A_2\in \cA$. Then, there exists a chord sequence $CS(A_1,A_2)\subseteq E_R(\cB, \cA)$ containing at most $3\nu^{-1}$ edges and such that $V(CS(A_1,A_2)) \cap \cV \subseteq \{A_1^-,A_2\}$%
		\COMMENT{We may assume without loss of generality that $A_1^-$ and $A_2$ occur only one, at the start and end, respectively. So the statement matches \cite[Lemma 7.2]{kuhn2013hamilton}.}, 
	where $A_1^-$ denote the predecessor of $A_1$ on $C$.
\end{lm}

Since we can no longer construct chord sequences between any pair of clusters, we will need to be more careful and adapt the proofs of the lemmas which use \cite[Lemma 7.2]{kuhn2013hamilton} (that is, \cite[Lemmas 7.5 and 8.3]{kuhn2013hamilton}).

Define a \emph{path system extender $PE$ for $C$ and $R$ with parameters $(\eps,d,d',\zeta)$} as in \cite{kuhn2013hamilton}. (The precise definition is not relevant for our purposes and so we omit it here.)

\begin{lm}[{Bipartite version of \cite[Lemma 8.3]{kuhn2013hamilton}}]\label{findham}
	Let $0<\frac{1}{n}\ll d'\ll \frac{1}{k}\ll \varepsilon\ll \frac{1}{\ell^*}\ll d\ll \nu\ll \tau\ll \delta, \theta\leq 1$ and $d\ll \zeta\leq \frac{1}{2}$. Suppose that $\frac{m}{50}\in \mathbb{N}$.
	Let $(D,\cP_0, R_0,C_0,\cP,R,C)$ be a consistent $(\ell^*,k,m,\eps,d,\nu,\tau,\delta,\theta)$-bi-system
	with $|D|=n$ and exceptional set $V_0$. Let $\cP'$ be a $(d')^2$-uniform $50$-refinement of $\cP$.
	Let $PE$ be a path system extender with parameters $(\eps,d,d',\zeta)$ for $C$ and $R$.
	Let $s\coloneqq \frac{10^7}{\nu^2}$ and suppose that $Q$ is a set of vertex-disjoint paths of $D$ such that the following hold.
	\begin{enumerate}
		\item $Q$ and $PE$ are edge-disjoint.
		\item $Q$ contains a special cover $SC$ in $D$ with respect to $V_0$ such that each component of $SC$ is a path of length $2$.\label{lm:findham-SC}
		\item There exists a set $\cV$ of five clusters of $\cP'$ such that each $e\in E(Q)$ satisfies $V(e)\subseteq V_0\cup \bigcup\cV$.
		\item $|E(Q)| \leq \frac{1230n}{s}$.
	\end{enumerate}
	Then, $D$ contains a Hamilton cycle $H$ such that $Q\subseteq H\subseteq PE \cup Q$. 
\end{lm}

\begin{proof}
	Let $M_{SC}$ be the complete special sequence associated to $SC$ and denote by $E$ the edge set of $(Q\setminus SC)\cup M_{SC}$. (Note that $E$ is precisely $E(Q^{\rm basic})$ with respect to the notation of \cite{kuhn2013hamilton}).
	For any cluster $V\in \cP$, denote by $V^-$ and $V^+$ the predecessor and successor of $V$ on $C$. For any vertex $v\in V(D)\setminus V_0$, denote by $V_v$ the cluster in $\cP$ which contains $v$.

	\begin{claim}\label{claim:findham}
		There exist chord sequences $CS(W_1, \tW_1^+), \dots, CS(W_{|E|},\tW_{|E|}^+)$ in $R$ for which the following hold.
		\begin{enumerate}[label=\rm(\alph*)]
			\item There exists an enumeration $u_1, \dots, u_{|E|}$ of the starting points of the edges in $E$ such that $\tW_i=V_{u_i}$ for each $i\in [|E|]$.\label{claim:findham-end}
			\item There exists an enumeration $v_1, \dots, v_{|E|}$ of the ending points of the edges in $E$ such that $W_i=V_{u_i}$ for each $i\in [|E|]$.\label{claim:findham-start}
			\item Altogether, $CS(W_1, \tW_1^+), \dots, CS(W_{|E|},\tW_{|E|}^+)$ contain at most $\frac{21m}{100}$ edges incident to each cluster in $\cP$ and at most $\frac{m}{50}$ occurrences of every edge of $R$.\label{claim:findham-interior}
		\end{enumerate}
	\end{claim}

	If \cref{claim:findham} holds, one can conclude the proof of \cref{findham} using the arguments of \cite[Lemma 8.3]{kuhn2013hamilton}. Thus, it suffices to prove \cref{claim:findham}.
	
	\begin{proofclaim}[Proof of \cref{claim:findham}]
		In the proof of \cite[Lemma 8.3]{kuhn2013hamilton}, \cite[Lemma 7.2]{kuhn2013hamilton} is used to construct, for each edge $uv\in E$, a chord sequence $CS(V_v, V_u^+)$ in $R$. This is not possible here because \cref{def:CBSys-biproboutexp} and \cref{lm:findham-SC} imply that, for any $uv\in M_{SC}\subseteq E$, $u$ and $v$ belong to a common vertex class of $D$ and so $V_v$ and $V_u^+$ belong to distinct vertex classes of $C\subseteq R$.
		
		We circumvent this problem as follows. Recall from \cref{def:CBSys-biproboutexp} that $D$ is a balanced bipartite digraph. Denote by $A$ and $B$ the vertex classes of $D$. Let $M_{SC,A}\coloneqq \{e\in V(M_{SC})\mid V(e)\subseteq A\}$ and $M_{SC,B}\coloneqq \{e\in V(M_{SC})\mid V(e)\subseteq B\}$. By \cref{lm:findham-SC} and since $D$ is bipartite, $M_{SC,A}$ and $M_{SC,B}$ partition $M_{SC}$. By \cref{def:CBSys-P0}, we have $|V_0\cap A|=|V_0\cap B|$ and so $|M_{SC,A}|=|M_{SC,B}|$.
		Fix a bijection $\phi\colon M_{SC,A}\longrightarrow M_{SC,B}$ and define $E'\coloneqq (Q\setminus SC)\cup \{ab',ba'\mid aa'\in M_{SC,A}, bb'=\phi(aa')\}$.
		By construction, each edge in $E'$ has precisely one endpoint in $A$ and one endpoint in $B$.
		Using \cref{buildchord} instead of its non-bipartite analogue \cite[Lemma 7.2]{kuhn2013hamilton}, one can apply the arguments of \cite[Lemma 8.3]{kuhn2013hamilton} to construct, for each $uv\in E'$, a chord sequence $CS(V_v, V_u^+)$ in $R$ such that, altogether, \cref{claim:findham-interior} holds. Then, \cref{claim:findham-end,claim:findham-start} hold by definition of $E'$.
	\end{proofclaim}
	
	This completes the proof of \cref{findham}.
\end{proof}

A similar problem arises when adapting \cite[Lemma 7.5]{kuhn2013hamilton} to the bipartite case. Roughly speaking, \cite[Lemma 7.5]{kuhn2013hamilton} guarantees the existence of a special cover $SC$ whose components are paths of length $2$, and chord sequences from $V_v$ to $V_u^+$ for each edge $uv$ of the complete special sequence $M_{SC}$ associated to $SC$ (where $V_v$ denotes the cluster in $\cP$ which contains $v$ and $V_u^+$ denotes the successor on $C$ of the cluster in $\cP$ which contains $u$). As discussed in the proof of \cref{findham}, such chord sequences do not exist. We can only guarantee that the number of times a cluster is at the start/end of a chord sequence equals the number of edges in $M_{SC}$ which start/end in that cluster. This is sufficient for proving the bipartite version of \cite[Lemma 7.6]{kuhn2013hamilton} (that is, \cref{lm:SF}).

\begin{lm}[{Bipartite version of \cite[Lemma 7.5]{kuhn2013hamilton}}]\label{prelimfactor}
	Let $0<\frac{1}{n}\ll \frac{1}{k}\ll \varepsilon\ll \varepsilon'\ll d\ll \nu\ll \tau\ll \delta,\theta\le 1$ and $\frac{f}{\ell^*}\ll 1$ and $\varepsilon\ll \frac{1}{\ell'}, \frac{1}{f}$. Suppose that $\frac{\ell^*}{f}, \frac{m}{\ell'}\in\mathbb{N}$.
	Let $(D,\cP_0, R_0,C_0,\cP,R,C)$ be a consistent $(\ell^*,k,m,\eps,d,\nu,\tau,\delta,\theta)$-bi-system
	with $|D|=n$ and exceptional set $V_0$. Suppose that $D'$ is a spanning subdigraph of $D$ and that $\cP'$ is a partition of $V(D)$ such that the following
	conditions are satisfied.
	\begin{enumerate}
		\item $\cP'$ is an $\varepsilon$-uniform $\ell'$-refinement of $\cP$.
		\item For any $v\in V_0$, we have $d^\pm_D(x)- d^\pm_{D'}(x)\leq \varepsilon n$.
		\item For any $v\in V(D)\setminus V_0$, we have $d^\pm_D(v)- d^\pm_{D'}(v)\leq \frac{(\varepsilon')^3m}{\ell'}$.
	\end{enumerate}
	For any cluster $V\in \cP$, denote by $V^-$ and $V^+$ the predecessor and successor of $V$ on $C$. For any vertex $v\in V(D)\setminus V_0$, denote by $V_v$ the cluster in $\cP$ which contains $v$.
	Let $(h,j)\in [\ell']\times [f]$ and let $I=W_1 \dots W_{k'}$ denote the $j^{\rm th}$ interval in the canonical interval partition of $C$ into $f$ intervals. For any cluster $V\in \cP$, denote by $V^h$ the $h^{\rm th}$ subcluster of $V$ in $\cP'$. In particular, let $W_1^h, \dots, W_{k'}^h$ denote the $h^{\rm th}$ subclusters of $W_1, \dots, W_{k'}$ in $\cP'$. Then, the following hold.
	\begin{enumerate}[label=\rm(\alph*)]
		\item There exists a special cover $SC$ in $D'$ with respect to $V_0$ which satisfies the following properties.
		\begin{itemize}
			\item Each component of $SC$ is a path of length $2$.
			\item $V(SC)\subseteq V_0\cup W_3^h\cup \dots W_{k'-2}^h$.
			\item Each cluster $V\in \cP$ satisfies $|V\cap V(SC)|\leq \varepsilon^{\frac{1}{4}} m$.
		\end{itemize}
		\item There exist chord sequences $CS(\tW_1, \hW_1^+), \dots, CS(\tW_{|V_0|}, \hW_{|V_0|}^+)$ in $R$ for which the following hold.\label{lm:prelimfactor-CS}
		\begin{itemize}
			\item There exists an enumeration $u_1, \dots, u_{|V_0|}$ of the starting points of the components of $SC$ such that $\hW_i=V_{u_i}$ for each $i\in [|V_0|]$.
			\item There exists an enumeration $v_1, \dots, v_{|V_0|}$ of the ending points of the components of $SC$ such that $\tW_i=V_{v_i}$ for each $i\in [|V_0|]$.
			\item For each $i\in [|V_0|]$, $CS(\tW_i, \hW_i^+)$ contains at most $3\nu^{-3}$ edges and all its vertices lie in $W_2\cup \dots\cup W_{k'-1}$.
			\item Altogether, $CS(\tW_1, \hW_1^+), \dots, CS(\tW_{|V_0|}, \hW_{|V_0|}^+)$ contain at most $4\varepsilon^{\frac{1}{4}}m$ edges incident to each cluster in $\cP$.
		\end{itemize}
		\item $D'$ contains a matching $M$ which satisfies the following properties.
		\begin{itemize}
			\item $M$ can be obtained by replacing, for each $i\in [|V_0|]$, each edge $UV$ of $CS(\tW_i, \hW_i^+)$ by an edge of $D'(U^h, V^h)$.
			\item $V(M)\cap V(SC)=\emptyset$.
		\end{itemize}
		\item For each $UV\in E(C)$, the pair $D'[U^h,V^h]$ is $[\varepsilon',\geq d]$-superregular.
	\end{enumerate}
\end{lm}

\Cref{prelimfactor} can be proven using the same arguments as in \cite[Lemma 7.5]{kuhn2013hamilton}, with \cref{lm:biprobblowup2} and \cref{buildchord} playing the roles of \cite[Lemmas 5.3 and 7.2]{kuhn2013hamilton} and using the arguments of \cref{claim:findham} of the proof of \cref{findham} to choose the endpoints of the chord sequences in \cref{prelimfactor}\cref{lm:prelimfactor-CS}. Therefore, we omit the details here.

Finally, we use \cref{buildchord} and adapt the arguments of \cite[Lemma 9.1]{kuhn2013hamilton} to derive \cref{lm:BU}.

\begin{proof}[Proof of \cref{lm:BU}]
	Denote $C=V_1\dots V_{2k}$. Since $R$ is bipartite, we may assume without loss of generality that $\cA= \{V_i\mid i\in [2k] \text{ is odd}\}$ and $\cB= \{V_i\mid i\in [2k] \text{ is even}\}$. Denote $V_{2k+1}\coloneqq V_1$ and $V_{2k+2}\coloneqq V_2$.
	For simplicity, split \cref{def:BU-edges} into two parts as follows.
	\begin{enumerate}[label=(BU1\alph*)]
		\item The edge set of $U$ has a partition into $U_{\rm odd}$ and $U_{\rm even}$ and, for every $i\in [2k]$, $U$ contains a chord sequence $CS(V_i,V_{i+2})$
		from $V_i$ to $V_{i+2}$ such that \cref{def:BU-size}, \cref{def:BU-degree}, and the following hold.
		All of the edges in the multiset $\bigcup \{CS(V_i,V_{i+2})\mid i\in [2k] \text{ is odd}\}$ are contained in $U_{\rm odd}$, all of the edges in the multiset $\bigcup \{CS(V_i,V_{i+2})\mid i\in [2k] \text{ is even}\}$ are contained in $U_{\rm even}$, and \label{def:BU-edges1} 
		\item all the remaining edges of $U$ lie on $C$.\label{def:BU-edges2} 
	\end{enumerate}
	
	Apply the arguments of \cite[Lemma 9.1]{kuhn2013hamilton} with \cref{buildchord} playing the role of \cite[Lemma 7.2]{kuhn2013hamilton} to obtain chord sequences $CS(V_1, V_3), CS(V_2,V_4), \dots, CS(V_{2k}, V_{2k+2})$ which satisfy the following properties,
	where $U_{\rm odd}'$ denotes the multiset of edges defined by
	\[U_{\rm odd}'\coloneqq E(CS(V_1,V_3))\cup E(CS(V_3,V_5))\cup \dots \cup E(CS(V_{2k-1},V_{2k+1}))\]
	and $U_{\rm even}'$ denotes the multiset of edges defined by
	\[U_{\rm even}'\coloneqq E(CS(V_2,V_4))\cup E(CS(V_4,V_6))\cup \dots \cup E(CS(V_{2k},V_{2k+2})).\]%
		\COMMENT{Suppose inductively that, for some $0\leq \ell < k$, we have constructed chord sequences $CS(V_1,V_3), CS(V_3,V_5), \dots, CS(V_{2\ell-1}, V_{2\ell+1})$ satisfying \cref{lm:BU-size,lm:BU-full}. Let $\cV$ be the set of clusters which are already visited at least $\frac{\ell'}{3}$ times. Since each chord sequence contains at most $3\nu^{-1}$ edges and since $R$ is bipartite, we have $|\cV|\leq \frac{3\cdot 3\nu^{-1}\cdot k}{\ell'}=\frac{\nu k}{4}$. Let $CS(V_{2\ell+1}, V_{2\ell+3})$ be the chord sequence obtained by applying \cref{buildchord}.\\
		Construct  $CS(V_2,V_4), CS(V_4,V_6), \dots, CS(V_{2k}, V_{2k+2})$ similarly.}
	\begin{enumerate}[label=(\roman*)]
		\item For each $i\in [2k]$, $CS(V_i, V_{i+2})$ contains at most $3\nu^{-1}\leq \frac{\sqrt{\ell'}}{2}$ edges.\label{lm:BU-size}
		\item For each $i\in [2k]$, we have $d_{U_{\rm odd}'}(V_i)\leq \frac{2\ell'}{5}$ and $d_{U_{\rm even}'}(V_i)\leq \frac{2\ell'}{5}$.\label{lm:BU-full}
	\end{enumerate}
	Let $U'$ be the multidigraph on $V(R)$ whose multiset of edges is defined by $E(U')\coloneqq U_{\rm odd}'\cup U_{\rm even}'=\bigcup_{i\in [2k]} E(CS(V_i, V_{i+2}))$.
	By \cref{lm:BU-size} and construction, $U'$ satisfies \cref{def:BU-edges1} and \cref{def:BU-size}.
	
	For each $i\in [2k]$, let $n_{i,\rm odd}^\pm\coloneqq d_{U_{\rm odd}'}^\pm(V_i)$ and $n_{i,\rm even}^\pm\coloneqq d_{U_{\rm even}'}^\pm(V_i)$.
	By similar arguments as in the proof of \cite[Lemma 9.1]{kuhn2013hamilton}, we have $n_{i+1,\rm odd}^-=n_{i,\rm odd}^+$ and $n_{i+1,\rm even}^-=n_{i,\rm even}^+$ for each $i\in [2k]$ (where $n_{2k+1, \rm odd}^-\coloneqq n_{1, \rm odd}^-$ and $n_{2k+1, \rm even}^-\coloneqq n_{1, \rm even}^-$).%
	For each $i\in [2k]$, let $\ell_{i, \rm odd}\coloneqq \frac{\ell'}{2}-n_{i,\rm odd}^-$ and $\ell_{i, \rm even}\coloneqq \frac{\ell'}{2}-n_{i,\rm even}^-$.
	Let $U$ be obtained from $U'$ by adding, for each $i\in [2k]$, exactly $\ell_{i,\rm odd}+\ell_{i, \rm even}$ copies of the edge $V_{i-1}V_i$. Let $U_{\rm odd}$ be obtained from $U_{\rm odd}'$ by adding exactly $\ell_{i, \rm odd}$ copies of the edge $V_{i-1}V_i$ and let $U_{\rm even}$ be obtained from $U_{\rm even}'$ by adding exactly $\ell_{i, \rm even}$ copies of the edge $V_{i-1}V_i$. Note that $U_{\rm odd}$ and $U_{\rm even}$ partition the edges of $U$.
	Since $U'$ satisfies \cref{def:BU-edges1} and \cref{def:BU-size}, $U$ also satisfies \cref{def:BU-edges1} and \cref{def:BU-size}. By construction, \cref{def:BU-edges2} also holds.
		
	For each $i\in [2k]$, we have $d_{U_{\rm odd}}^-(V_i)=n_{i, \rm odd}^-+\ell_{i,\rm odd}=\frac{\ell'}{2}$ and \[d_{U_{\rm odd}}^+(V_i)=n_{i, \rm odd}^++\ell_{i+1,\rm odd}=n_{i+1, \rm odd}^-+\ell_{i+1,\rm odd}=\frac{\ell'}{2}.\]
	Similarly, both $d_{U_{\rm even}}^\pm(V_i)=\frac{\ell'}{2}$ for each $i\in[2k]$ and so \cref{def:BU-degree} holds.
	One can show that $U$ forms a closed walk in $R$ using similar arguments as in \cite[Lemma 9.1]{kuhn2013hamilton}.%
		\COMMENT{By \cref{lm:BU-full}, $U$ contains at least one copy of each edge of $C$. Let $U''$ be obtained from $U$ by deleting one copy of each such edge.
		By \cite[Proposition 6.1]{kuhn2013hamilton}, there exists a decomposition $\cF$ of $U''$ into $\ell'-1$ $1$-factors.
		Let $C_{1,1}, \dots, C_{1,n_1}, C_{2,1}, \dots,C_{2,n_2}, C_{3,1}, \dots, C_{2k,n_{2k}}$ be an enumeration of all the cycles contained in the $1$-factors in $\cF$ such that, for each $i\in [2k]$ and $j\in [n_i]$, $i$ is the smallest index such that $V_i\in V(C_{i,j})$. Then, the edges of $U$ form the closed walk $V_1C_{1,1}V_1\dots V_1C_{1,n_1}V_1V_2C_{2,1}V_2\dots V_2C_{2,n_2}V_2V_3\dots V_{2k}C_{2k,n_{2k}}V_{2k}V_1$.}
\end{proof}

\onlyinsubfile{\bibliographystyle{abbrv}
	\bibliography{Bibliography/Bibliography}}

	\section{Applying the regularity lemma: proof of Lemma \ref{lm:bisetup}}\label{app:reglm}
	
	\onlyinsubfile{
		\appendix
		\setcounter{section}{1}
		\section{Applying the regularity lemma: proof of Lemma \ref{lm:bisetup}}}

In this appendix, we will prove \cref{lm:bisetup}, which guarantees the existence of consistent bi-systems and bi-setups in a bipartite robust outexpander. We will need the bipartite analogue of \cref{lm:verticesedgesremovalrobout}\cref{lm:verticesedgesremovalrobout-vertices}. The proof follows easily from the definition of a bipartite robust outexpander and is therefore omitted.

\begin{lm}\label{cor:verticesedgesremovalbiproboutexp2}
	Let $0<\frac{1}{n}\ll\varepsilon\leq \nu\ll \tau \leq 1$. 
	Let $D$ be a bipartite digraph on vertex classes $A$ and $B$ of size $n$ and suppose that $D$ is a bipartite robust $(\nu, \tau)$-outexpander with bipartition $(A,B)$. Let $A'\subseteq A$ and $B'\subseteq B$ satisfy $|A'|=|B'|\geq (1-\varepsilon)n$.
	Then, $D(A',B')$ is a bipartite robust $(\nu-\varepsilon, 2\tau)$-outexpander with bipartition $(A', B')$.
\end{lm}

\COMMENT{\begin{proof}
		Denote $n'\coloneqq |A'|$.
		Let $S\subseteq A'$ satisfy $2\tau n'\leq |S|\leq (1-2\tau)n'$. Then,
		\[\tau n\leq 2\tau (1-\varepsilon)n\leq |S|\leq (1-2\tau)n\leq (1-\tau)n.\]
		Thus, $|RN_{\nu, D}^+(S)\cap B'|\geq |S|+\nu n-\varepsilon n\geq |S|+(\nu-\varepsilon)n'$. Since $RN_{\nu-\varepsilon, D(A',B')}^+(S)\supseteq RN_{\nu, D}^+(S)\cap B'$, we are done.
	\end{proof}}

\begin{proof}[Proof of \cref{lm:bisetup}]
	Let $0<\frac{1}{M'}\ll \varepsilon$. Fix additional constants such that $\frac{1}{M'}\ll\varepsilon_1\ll \varepsilon_2\ll \varepsilon_3\ll \varepsilon_4 \ll \varepsilon$.
	
	\begin{steps}
		\item \textbf{Applying the regularity lemma.}\label{step:bisetup-reg}	
		Let $M$ and $n_0$ be the constants obtained by applying \cref{lm:bipreglm} with $\varepsilon_1$ playing the role of $\varepsilon$.
		By \cref{lm:bipreglm}\cref{lm:bipreglm-k}, we may assume without loss of generality that $0<\frac{1}{n_0}\ll \frac{1}{M}\leq \frac{1}{M'}\ll \varepsilon_1$.
		Fix additional constants such that
		$\varepsilon\ll\frac{1}{q}\ll \frac{1}{f}, \frac{1}{\ell^*}\ll d \ll \nu \ll \tau \ll \delta, \theta \ll 1$ and $d\ll \frac{1}{g}\ll 1$. Moreover, let $\ell'\geq 324\nu^{-2}$ be even.
		Let $D$ be a balanced bipartite on vertex classes $A$ and $B$ of size $n\geq n_0$.
		Suppose that $D$ is a bipartite robust $(\nu, \tau)$-outexpander with bipartition $(A, B)$ and that $\delta^0(D)\geq \delta n$.
	
		Apply \cref{lm:bipreglm} with $\varepsilon_1$ and $4d$ playing the roles of $\varepsilon$ and $d$ to obtain a spanning subdigraph $D'\subseteq D$ and a partition $\widetilde{\cP}=\{\tV_0, \tV_1, \dots, \tV_{2\tk}\}$ of $V(D)$ such that	
		\cref{lm:bipreglm}\cref{lm:bipreglm-k,lm:bipreglm-V0,lm:bipreglm-m,lm:bipreglm-deg,lm:bipreglm-empty,lm:bipreglm-reg,lm:bipreglm-AB} hold with $2\tk, \tm, \tV_0, \tV_1, \dots, \tV_{2\tk}$ playing the roles of $k, m, V_0, V_1, \dots, V_{2k}$.
		Denote $\widetilde{\cA}\coloneqq \{\tV_i\mid i\in [2\tk], \tV_i\subseteq A\}$ and $\widetilde{\cB}\coloneqq \{\tV_i\mid i\in [2\tk], \tV_i\subseteq B\}$.
		We may assume without loss of generality that $\widetilde{\cA}\coloneqq \{V_{2i-1}\mid i\in [\tk]\}$ and $\widetilde{\cB}\coloneqq \{\tV_{2i}\mid i\in [\tk]\}$.
		Let $\tR$ be the bipartite reduced digraph of $D$ with parameters $\varepsilon_1, 4d$, and $M'$.
		By \cref{lm:Rrob}, $\delta^0(\tR)\geq \frac{\delta \tk}{2}$ and $\tR$ is a bipartite robust $(\frac{\nu}{2}, 2\tau)$-outexpander with bipartition $(\widetilde{\cA}, \widetilde{\cB})$.
		Observe that, by \cref{lm:bipreglm}\cref{lm:bipreglm-reg}, $D[U,V]$ is $(\varepsilon_1, \geq 4d)$-regular for each $UV\in E(\tR)$.
	
		\item \textbf{Ensuring the desired divisibility conditions.}\label{step:bisetup-k}
		Let $\hk$ be the largest integer satisfying $\hk\leq \tk$ and $\frac{\hk}{21fg(g-1)}\in \mathbb{N}$.
		Let $\hV_0\coloneqq \tV_0\cup \bigcup_{i\in [2\tk-2\hk]}\tV_{2\hk+i}$.
		By \cref{lm:bipreglm}\cref{lm:bipreglm-V0,lm:bipreglm-m,lm:bipreglm-k}, 
		\begin{equation}\label{eq:bisetup-V0}
			|\hV_0\cap A|=|\hV_0\cap B|\leq \varepsilon_1n+21fg(g-1)\tm\leq 2\varepsilon_1 n.
		\end{equation}
		Let $\widehat{\cP}\coloneqq \{\hV_0, \tV_1, \dots, \tV_{2\hk}\}$ and let $\hV_1, \dots, \hV_{2\hk}$ be a relabelling of $\tV_1, \dots, \tV_{2\hk}$.
		Let $\hR\coloneqq \tR-\{\tV_{2\hk+i} \mid i\in [2\tk-2\hk]\}$.
		Denote $\widehat{\cA}\coloneqq \widetilde{\cA} \setminus \{\tV_{2\hk+i}\mid i\in [2\tk-2\hk]\}$ and $\widehat{\cB}\coloneqq \widetilde{\cB} \setminus \{\tV_{2\hk+i} \mid i\in [2\tk-2\hk]\}$.
		Then, $\delta^0(\hR)\geq \frac{\delta \tk}{2}-21fg(g-1)\geq \frac{\delta \hk}{3}$ and, by \cref{cor:verticesedgesremovalbiproboutexp2}, $\hR$ is a bipartite robust $(\frac{\nu}{3}, 4\tau)$-outexpander with bipartition $(\widehat{\cA}, \widehat{\cB})$.
		
		\item \textbf{Finding a Hamilton cycle and a bi-universal walk in the reduced graph.}
		Apply \cref{lm:biprobHamcycle} with $\hR, \widehat{\cA}, \widehat{\cB}, \hk, \frac{\nu}{3}, 4\tau$, and $\frac{\delta}{3}$ playing the roles of $D, A, B, n, \nu, \tau$, and $\delta$ to obtain a Hamilton cycle $\hC$ of $\hR$.
		We may assume without loss of generality that $\hC=\hV_1\dots \hV_{2\hk}$.
		Apply \cref{lm:BU} with $\hR, \hC, \hk, \frac{\nu}{3}, 4\tau$, and $\frac{\delta}{3}$ playing the roles of $R, C, k, \nu, \tau$, and $\delta$ to obtain a universal walk $\hU$ for $\hC$ in $\hR$ with parameter $\ell'$. Denote $\hU=\hV_{i_1}\dots \hV_{i_{2\ell'\hk}}$.
		
		\item \textbf{Forming superregular pairs.}\label{step:bisetup-supreg}
		Let $E\coloneqq \{e\in E(\hC)\cup E(\hU)\}$. We adjust the partition $\widehat{\cP}$ to ensure that each edge in $E$ corresponds to a superregular pair in $D$.
		By \cref{def:BU-degree}, each $V\in V(\hR)$ satisfies $d_E^\pm(V)\leq \ell'+1$.
		For each $e=UV\in E$, denote $d_e\coloneqq d_D(U,V)$ and observe that, by \cref{step:bisetup-reg}, $d_e\geq 4d$.
		Fix an integer $m_0$ such that $(1-2\sqrt{\varepsilon_1})\tm\leq m_0\leq (1-\sqrt{\varepsilon_1})\tm$ and $\frac{fm_0}{200q\ell'\ell^*}\in \mathbb{N}$.	
		Let $i\in [2\hk]$. Let $\hV_i'$ consist of all the vertices $v\in \hV_i$ such that there exists $j\in [2\hk]\setminus \{i\}$ such that $e=\hV_i\hV_j\in E$ but $|N_D^+(v)\cap \hV_j|\neq (d_e\pm \varepsilon_1)\tm$ or, $e'=\hV_j\hV_i\in E$ but $|N_D^-(v)\cap \hV_j|\neq (d_{e'}\pm \varepsilon_1)\tm$.
		By \cref{lm:epsdeg}, $|\hV_i'|\leq 2\varepsilon_1\tm d_E(\hV_i)\leq 4\varepsilon_1\tm(\ell'+1)\leq \sqrt{\varepsilon_1} \tm$.
		Let $\hV_i'\subseteq \hV_i''\subseteq \hV_i$ satisfy $|\hV_i''|=\tm-m_0$.
		Let $V_0\coloneqq \hV_0\cup \bigcup_{i\in [2\hk]}\hV_i''$.
		By \cref{eq:bisetup-V0}, 
		\[|V_0\cap A|=|V_0\cap B|\leq 2\varepsilon_1n+2\sqrt{\varepsilon_1}\tm\cdot 2\hk\leq \varepsilon n.\]
		For each $i\in [2\hk]$, let $V_i^0\coloneqq \hV_i\setminus \hV_i''$.
		By construction, $|V_1^0|=\dots=|V_{2\hk}^0|=m_0$.
		Let $\cP_0\coloneqq \{V_0, V_1^0, \dots, V_{2\hk}^0\}$.
		Denote $\cA_0\coloneqq \{V_i^0\in \cP_0\mid \hV_i\in \widehat{\cA}\}$ and $\cB_0\coloneqq \{V_i^0\in \cP_0\mid \hV_i\in \widehat{\cB}\}$.
		Let $R_0$ be the digraph on $\cA^0\cup \cB^0$ which is induced by $\hR$, i.e.\ defined as follows. For any $i,j\in [2\hk]$, $V_i^0V_j^0\in E(R_0)$ if and only if $\hV_i\hV_j\in E(\hR)$.
		By \cref{lm:epsvertexslice}, $D[U,V]$ is $(\varepsilon_2, \geq 3d)$-regular for each $UV\in E(R_0)$.
		Moreover, $\delta^0(R_0)\geq \frac{\delta \hk}{3}$ and $R_0$ is a bipartite robust $(\frac{\nu}{3}, 4\tau)$-outexpander with bipartition $(\cA_0,\cB_0)$.
		Let $C_0\coloneqq V_1^0\dots V_{2\hk}^0$ and $U_0=V_{i_1}^0\dots V_{i_{2\ell'\hk}}^0$.
		By construction, $C_0$ is Hamilton cycle of $R_0$ and $U_0$ is a universal walk for $C_0$ in $R_0$ with parameter $\ell'$.
		Moreover, $D[U,V]$ is $[\varepsilon_2, \geq 3d]$-superregular for each $UV\in E(C_0)\cup E(U_0)$%
			\COMMENT{Each remaining vertex started with degree $(d_e\pm \varepsilon_1)\tm$. Delete at most $2\sqrt{\varepsilon_1}\tm$ vertices per cluster. So now each remaining vertex has degree $(d_e\pm 3\sqrt{\varepsilon_1})\tm=(d_e\pm \varepsilon_2)\tm$. Recall that $d_e\geq 3d$.}.
		
		\item \textbf{Finding the refinements.}\label{step:bisetup-refinements}
		Apply \cref{lm:URefexistence} with $2n, m_0, 2\hk, \cP_0, \varepsilon_2$, and $\ell^*$ playing the roles of $n, m, k, \cP, \varepsilon$, and $\ell$ to obtain an $\varepsilon_2$-uniform $\ell^*$-refinement $\cP$ of $\cP_0$.
		Let $\cA$ be the set of clusters $V\in \cP$ such that $V\subseteq W$ for some $W\in \cA_0$. Let $\cB$ be the set of clusters in $\cP\setminus \cA$.
		Let $k\coloneqq \ell^*\hk$ and $m\coloneqq \frac{m_0}{\ell^*}$.
		Let $R$ be the $\ell^*$-fold blow-up of $R_0$ induced by $\cP$. Then, $\delta^0(R)=\ell^*\delta^0(R_0)\geq \frac{\delta k}{3}$. By \cref{lm:biprobblowup2}, $R$ is a bipartite robust $(4\nu^4, 8\tau)$-outexpander with bipartition $(\cA, \cB)$.
		By \cref{lm:URefreg}\cref{lm:URef-reg} and \cref{step:bisetup-supreg}, $D[U,V]$ is $(\varepsilon_3, \geq 2d)$-regular for each $UV\in E(R)$.
		For each $i\in [2\hk]$, denote by $V_{i,1}^0, \dots, V_{i, \ell^*}^0$ the subclusters of $V_i^0$ contained in $\cP$.
		Let $C\coloneqq V_{1,1}^0V_{2,1}^0\dots V_{2\hk,1}^0V_{1,2}^0\dots V_{2\hk, \ell^*}^0$ and $U\coloneqq V_{i_1,1}^0V_{i_2,1}^0\dots V_{i_{2\ell'\hk},1}^0V_{1,2}^0\dots V_{i_{2\ell'\hk}, \ell^*}^0$.
		Then, $C$ is a Hamilton cycle of $R$ and $U$ is a bi-universal walk for $C$ in $R$ with parameter $\ell'$.
		Moreover, by \cref{lm:URefreg}\cref{lm:URef-supreg} and \cref{step:bisetup-supreg}, $D[U,V]$ is $[\varepsilon_3, \geq 2d]$-superregular for each $UV\in E(C)\cup E(U)$.
	
		Apply \cref{lm:URefexistence} with $2n, \varepsilon_3$, and $\ell'$ playing the roles of $n, \varepsilon$, and $\ell$ to obtain an $\epsilon_3$-uniform $\ell'$-refinement $\cP'$ of $\cP$.
		Let $V_1, \dots, V_{2k}$ be a relabelling of the clusters in $\cP$ such that $C=V_1\dots V_{2k}$ and let $i_1', \dots, i_{2\ell'k}'$ be such that $U=V_{i_1'}\dots V_{i_{2\ell'k}'}$.
		For each $i\in [2k]$, denote by $V_{i,1}, \dots, V_{i, \ell'}$ the subclusters of $V_i$ contained in $\cP'$.
		Let $U'\coloneqq V_{i_1',1}V_{i_2',1}\dots V_{i_{2\ell'k}',1}V_{i_1',2}\dots V_{i_{2\ell'k}',\ell'}$.
		By \cref{lm:URefreg}\cref{lm:URef-supreg}, $D[U,V]$ is $[\varepsilon_4, \geq 2d]$-superregular for each $UV\in E(U')$.
		
		Apply \cref{lm:URefexistence} with $2n, \varepsilon_3$, and $\frac{q}{f}$ playing the roles of $n, \varepsilon$, and $\ell$ to obtain an $\varepsilon_3$-uniform $\frac{q}{f}$-refinement $\cP^*$ of $\cP$.
		
		\item \textbf{Verifying \cref{lm:bisetup-parameters,lm:bisetup-bisetup}.}	
		Let $M''\coloneqq \ell^*M$. By our choice of $\hk$ in \cref{step:bisetup-k} and definition of $k$ in \cref{step:bisetup-refinements}, we have $\frac{\ell^*\tk}{2}\leq \ell^*\hk =k\leq \ell^*\tk$.
		Then, \cref{lm:bipreglm}\cref{lm:bipreglm-k} (with $\tk$ playing the role of $k$) implies that $M'\leq k\leq M''$. Moreover, \cref{step:bisetup-k} implies that $\frac{k}{7}, \frac{k}{f}, \frac{k}{g}, \frac{2fk}{3g(g-1)}\in \mathbb{N}$.
		By our choice of $m_0$ in \cref{step:bisetup-supreg} and definition of $m$ in \cref{step:bisetup-refinements}, we have $\frac{m}{50}, \frac{m}{4\ell'}, \frac{fm}{q}\in \mathbb{N}$. Thus, \cref{lm:bisetup-parameters} is satisfied.
				
		Let $D_1$ be obtained from $D$ by taking each edge independently with probability $\frac{1}{2}$. Define $D_2\coloneqq D\setminus D_1$. We need to show that \cref{lm:bisetup-bisetup} holds with positive probability. By \cref{lm:randomsetup,lm:randomsystem}, it suffices to show that the following properties are satisfied.
		\begin{enumerate}[label=\rm(\alph*)]
			\item $(D,\cP_0, R_0, C_0, \cP, R, C)$ is a consistent $(\ell^*,2k, m, \varepsilon_4, 2d, 4\nu^4, 8\tau, \frac{\delta}{3}, 3\theta)$-bi-system.\label{lm:bisetup-bisetup-bisystem}
			\item $(D, \cP, \cP',\cP^*, R, C, U ,U')$ is an $(\ell', \frac{q}{f}, 2k, m, \varepsilon_4, 2d)$-bi-setup.\label{lm:bisetup-bisetup-bisetup}
		\end{enumerate}
		By \cref{lm:bipreglm}\cref{lm:bipreglm-AB} and \cref{step:bisetup-refinements,step:bisetup-supreg}, \cref{def:BST-P} holds. Moreover, \cref{def:BST-R,def:BST-C,def:BST-U,def:BST-P',def:BST-U',def:BST-U'supreg,def:BST-P*} follow from \cref{step:bisetup-refinements}. Thus, \cref{lm:bisetup-bisetup-bisetup} holds.
		By \cref{step:bisetup-supreg} and definition of $k$ and $m$ in \cref{step:bisetup-refinements}, \cref{def:CBSys-P0} holds. By \cref{step:bisetup-refinements}, \cref{def:CBSys-P,def:CBSys-R,def:CBSys-R0} are satisfied and \cref{def:CBSys-deg} follows from \cref{def:URef}. Moreover, \cref{def:CBSys-C,def:CBSys-biproboutexp,def:CBSys-C0} follow from \cref{step:bisetup-refinements,step:bisetup-supreg}. Therefore \cref{lm:bisetup-bisetup-bisystem} holds.		
		\qedhere
	\end{steps}
\end{proof}

\onlyinsubfile{\bibliographystyle{abbrv}
	\bibliography{Bibliography/Bibliography}}

	}

\printglossary
\end{document}